\documentclass[oldfontcommands,6x9]{book}
\usepackage{amsmath,amsfonts,amsthm,amssymb,url,graphicx,makeidx,epigraph}
\usepackage{times,tikz}
\usetikzlibrary{arrows,automata}
\usepackage[english,polutonikogreek]{babel}
\usepackage[LGR]{fontenc}
\addto{\extrasgreek}{\greektext}
\DeclareFontFamily{OT1}{rsfs}{}
\DeclareFontShape{OT1}{rsfs}{n}{it}{<-> rsfs10}{}
\DeclareMathAlphabet{\mathscr}{OT1}{rsfs}{n}{it}
\DeclareMathOperator{\sgn}{sgn}

\DeclareMathOperator{\Res}{Res}

\DeclareMathOperator{\supp}{supp}
\DeclareMathOperator{\mo}{\,mod}

\DeclareMathOperator{\hyp}{hyp}

\DeclareMathOperator{\err}{err}
\DeclareMathOperator{\erf}{erf}

\DeclareMathOperator{\Si}{Si}

\newcommand{\sume}{\mathop{\sum\nolimits^{*\mkern-12mu}}\limits}
\newtheorem{prop}{Proposition}[section]
\newtheorem{theorem}[prop]{Theorem}
\newtheorem{corollary}[prop]{Corollary}
\newtheorem{lemma}[prop]{Lemma}

\newenvironment{remark}{{\bf Remark.}}{}
\begin{document}
\frontmatter                            
\title{The ternary Goldbach problem}
\author{Harald Andr\'es Helfgott}
\date{}
    \selectlanguage{english}
\maketitle                              

\tableofcontents                        

\chapter{Preface}
\epigraph{
  \selectlanguage{greek}
  >egg`uc d'' ~>hn t'eleoc; <`o d`e t>o tr'iton <~hke q[am~aze;\\
    s`un t~wi d'' >ex'efugen j'anaton ka`i k~h[ra m'elainan
      \selectlanguage{english}
      }{Hesiod (?), {\em Ehoiai}, fr. 76.21--2 Merkelbach and West}

The ternary Goldbach conjecture (or {\em three-prime conjecture}) 
states that every odd number $n$ greater than $5$ can be written as
the sum of three primes. The purpose of this book is to give the first full
proof of this conjecture. 

The proof builds on the great advances made in the early 20th century by
Hardy and Littlewood (1922) and Vinogradov (1937). Progress since then has 
been more gradual. In some ways, it was necessary to clear the board and
start work using only the main existing ideas towards the problem, together
with techniques developed elsewhere.

Part of the aim has been to keep the exposition as accessible as possible, with
an emphasis on qualitative improvements and new technical ideas that should
be of use elsewhere. The main strategy was to give an analytic approach
that is efficient, relatively clean, and, as it must be for this problem,
explicit; the focus does not lie in optimizing explicit constants, or in 
performing calculations, necessary as these tasks are.

\textbf{Organization.} 
In the introduction, after a summary of the history of the problem, we will
go over a detailed outline of the proof.
The rest of the book is divided in three parts, structured so that they
can be read independently: the first two parts do not refer to each other,
and the third part uses only the main results (clearly marked) of the first
two parts.

As is the case in most proofs involving the circle method, the problem
is reduced to showing that a certain integral over the ``circle''
$\mathbb{R}/\mathbb{Z}$ is non-zero.
The circle is divided into major arcs
and minor arcs. 
In Part \ref{part:min} -- in some ways the technical heart of the proof --
we will see how to give upper bounds on the integrand when $\alpha$ is in the minor arcs.
Part \ref{part:maj} will provide rather
precise estimates for the integrand when the variable $\alpha$ is in the major 
arcs. Lastly, Part \ref{part:concl} shows how to use these inputs as well as
possible to estimate the integral.

Each part and each chapter starts with a general discussion of the strategy
and the main ideas involved. Some of the more
technical bounds and computations are relegated to the appendices.

\vfill
\text{}
\section*{Dependencies between the chapters} \vfill

\begin{tikzpicture}[auto,node distance=2.7cm,thick]

\tikzstyle{every state}=[fill=none,draw=black,text=black]
\tikzstyle{information text}=[rounded corners,fill=gray!20,inner
sep=1ex]

\node[state] (1)                    {1};
\node[state] (2) [right of=1]       {2};
\node[state] (3) [below of=1]       {3};
\node[state] (7) [right of=3]       {7};
\node[state] (10) [right of=7]       {10};
\node[state] (4) [below of=3]       {4};
\node[state] (8) [right  of=4]       {8};
\node[state] (11) [right  of=8]       {11};
\node[state] (5) [below of=4]       {5};
\node[state] (9) [right of=5]       {9};
\node[state] (12)[right of=9]       {12};
\node[state] (6) [below of=5]       {6};
\node[state] (13)[below of=12] {13};
\node[state] (14)[below of=13] {14}; 

\begin{scope}[->,>=stealth',shorten >=1pt,] \path (2) edge  node {} (3);
  \path (2) edge  node {} (7); \path (2) edge  node {} (10);
  \path (3) edge  [bend left] node {} (5);
  \path (3) edge  node {} (4);
  \path (4) edge  [bend right] node {} (6);
  \path (7) edge  node {} (8); \path (2) edge
node {} (11); \path (4) edge  node {} (5); \path (8) edge  node {} (9);
\path(2) edge  node {} (12);
\path(11) edge  [bend left] node {} (13);
\path(11) edge  [bend right] node {} (14);
\path(10) edge  [bend left] node {} (14);
\path(5) edge  node {} (6);
\path (12) edge  node {} (13);
\path (6) edge  [bend left] node {} (11);
\path(13) edge  node {} (14);
\path(9) edge [bend right] node {} (14);
\end{scope}

\begin{scope}[color=gray, rounded corners=8pt, fill opacity=0.4] \fill
  (1) +(-0.8,-0.8) rectangle +(0.8,0.8); \fill (2) +(-0.8,-0.8) rectangle
  +(0.8,0.8);
  \fill (3) +(-0.8,1) rectangle +(0.8,-9);
  \fill (7) +(-0.8,1) rectangle +(0.8,-6.5);
  \fill (10) +(-0.8,1) rectangle +(0.8,-9);
\fill (14) +(-1,1) rectangle +(1,-1);
\end{scope}

\draw
(1) +(-1.8,1.6)node[right,text width=2cm,style=information text]
{Introduction};

\draw
(2) +(0,1.4)node[right,text width=2.5cm,style=information text] {Notation and preliminaries};

\draw
(3) +(-3,0.7)node[right,text width=2cm,style=information text] {Minor arcs: introduction};

\draw
(4) +(-3,0.7)node[right,text width=2cm,style=information text] {Type I sums};

\draw
(5) +(-3,0.7)node[right,text width=2cm,style=information text] {Type II sums};

\draw
(6) +(-3,0.7)node[right,text width=1.7cm,style=information text] {Minor-arc totals};

\draw
(7) +(-1.5,1.2)node[right,text width=1.8cm,style=information text]
    {Major arcs: overview};

\draw
(8) +(-1.5,1.2)node[right,text width=3cm,style=information text]
    {Mellin transform of twisted Gaussian};

    \draw
(9) +(-1.5,1.2)node[right,text width=2.5cm,style=information text]
{Explicit formulas};

\draw
(10) +(0.3,1)node[right,text width=2.5cm,style=information text] {The integral over the major arcs};

\draw
(11) +(0.3,1)node[right,text width=2.7cm,style=information text] {Smoothing functions and their use};

\draw
(12) +(0.3,1)node[right,text width=2.5cm,style=information text] {The $\ell_2$ norm and the large sieve};

\draw
(13) +(0.3,1.2)node[right,text width=2.5cm,style=information text] {The integral over the minor arcs};

\draw
(14) +(-1.2,-1)node[right,text width=2cm,style=information text]
{Conclusion};

\end{tikzpicture}

\chapter{Acknowledgements}
The author is very thankful to D. Platt, who, working in close coordination
with him, provided GRH verifications in the necessary ranges, and
also helped him with the usage of interval arithmetic. He is also deeply
grateful to O. Ramar\'e, who, in reply to his requests, prepared and sent for
publication several auxiliary results, and who otherwise provided much-needed
feedback.

The author is also much 
indebted to A. Booker, B. Green,  R. Heath-Brown,
H. Kadiri, D. Platt,
 T. Tao and M. Watkins for many discussions on Goldbach's problem and 
related issues. 
Several historical questions became clearer due to the 
help of J. Brandes, K. Gong, R. Heath-Brown,
Z. Silagadze, R. Vaughan and T. Wooley.
Additional 
references were graciously provided by R. Bryant, S. Huntsman
and I. Rezvyakova.   
Thanks are also due to 
B. Bukh, A. Granville and P. Sarnak for their valuable advice.

The introduction is largely based on the author's 
article for the Proceedings of the 2014 ICM \cite{HelfICM}.
That article, in turn, is based in part on the informal note \cite{Helblog},
which was published in Spanish translation
(\cite{Gaceta}, translated by M. A. Morales and the author,
and revised with the help of J. Cilleruelo and M. Helfgott) and
in a French version 
(\cite{MR3201598}, translated by M. Bilu and revised by the author).  
The proof first appeared as a series of preprints:
 \cite{Helf}, \cite{HelfMaj}, \cite{HelfTern}. 

Travel and other expenses were funded in part by 
the Adams Prize and the Philip Leverhulme Prize.
The author's work on the problem started at the
Universit\'e de Montr\'eal (CRM) in 2006; he is grateful to both the 
Universit\'e
de Montr\'eal and the \'Ecole Normale Sup\'erieure for providing pleasant
working environments. During the last stages of the work, travel was partly
covered by ANR Project Caesar No. ANR-12-BS01-0011.

The present work would most likely not have been possible without free and 
publicly available
software: SAGE, PARI, Maxima, gnuplot, VNODE-LP, PROFIL / BIAS, and, of course, \LaTeX, Emacs,
the gcc compiler and GNU/Linux in general. Some exploratory work was done
in SAGE and Mathematica. Rigorous calculations used either D. Platt's
interval-arithmetic package (based in part on Crlibm) or the PROFIL/BIAS interval arithmetic package
underlying VNODE-LP.

The calculations contained in this paper used a nearly trivial amount of
resources; they were all carried out on the author's desktop computers at
home and work. However, D. Platt's computations \cite{Plattfresh}
used a significant amount of resources, kindly donated to D. Platt and the
author by several institutions. This crucial help was provided by MesoPSL
(affiliated with the Observatoire de Paris and Paris Sciences et Lettres),
Universit\'e de Paris VI/VII (UPMC - DSI - P\^{o}le Calcul), University of Warwick (thanks
to Bill Hart), University of Bristol, France Grilles (French National Grid
Infrastructure, DIRAC national instance),
Universit\'e de Lyon 1 and Universit\'e de Bordeaux 1. Both D. Platt and the
author would like to thank the donating organizations, their technical
staff, and all those who helped to make these resources available to them.

\mainmatter                             
\chapter{Introduction}

The question we will discuss, or one similar to it, seems to have been
first posed by Descartes, in a manuscript published only centuries after
his death \cite[p. 298]{zbMATH02639585}. Descartes states:
``Sed \& omnis numerus par fit ex uno vel duobus vel tribus primis''
(``But also every even number is made out of one, two or three prime numbers.''\footnote{Thanks are due to J. Brandes and R. Vaughan for a discussion on a possible
 ambiguity in the Latin wording. Descartes' statement is 
mentioned (with a translation much like the one given here) in Dickson's
{\em History} \cite[Ch. XVIII]{MR0245499}.}.)
This statement comes in the middle of a discussion of sums of 
polygonal numbers, such as the squares.

Statements on sums of primes and sums of values of polynomials 
(polygonal numbers, powers $n^k$, etc.) 
have since shown themselves to be much more
than mere curiosities -- and not just because they are often very difficult
to prove. Whereas the study of sums of powers can rely on
 their algebraic structure, the study of sums of primes leads to the 
realization that, from several perspectives, the set of primes behaves much 
like the set of integers, or like a random set of integers. (It also
leads to the realization that this is very hard to prove.)

If, instead of the primes, we had a random set of odd integers $S$ whose density
-- an intuitive concept that can be made precise -- equaled that of the 
primes, then we would expect to be able to write every odd number as a sum
of three elements of $S$, and every even number as the sum of two elements
of $S$. We would have to check by hand whether this is true for small odd
and even numbers, but it is relatively easy to show that,
after a long enough check, it would be very unlikely
that there would be any exceptions left among the infinitely many cases left
to check.

The question, then, is in what sense we need the primes to be like a random
set of integers;
in other words, we need to know what we can prove about the regularities of
the distribution of the primes. This is one of the main questions of analytic
number theory; progress on it has been very slow and difficult. 

Fourier analysis expresses information on the distribution of a sequence
in terms of frequencies. In the case of the primes, what may be
called the main frequencies -- those in the {\em major arcs} -- correspond
to the same kind of large-scale distribution that is encoded by
{\em $L$-functions}, the family of functions to which the Riemann zeta
function belongs. On some of the crucial questions on $L$-functions, the
limits of our
knowledge have barely budged in the last century. There is something 
relatively new now,
namely, rigorous numerical data of non-negligible scope; still, such data is,
by definition, finite, and, as a consequence, its range of applicability
is very narrow. Thus, the real question in the major-arc regime
is how to use well the limited information we do have on the large-scale
distribution of the primes. As we will see, this requires delicate work
on explicit asymptotic analysis and smoothing functions.

Outside the main frequencies -- that is, in what are called the {\em
minor arcs} -- estimates based on $L$-functions no longer apply, and
what is remarkable is that one can say anything meaningful on the
distribution of the primes. Vinogradov
was the first to give unconditional, non-trivial bounds, showing
that there are no great irregularities in the minor arcs; this is what makes
them ``minor''.  Here the task is to give sharper bounds than Vinogradov.
It is in this regime that we can genuinely say that we learn a little more about
the distribution of the primes, based on what is essentially an elementary
and highly optimized analytic-combinatorial analysis of exponential sums,
i.e., Fourier coefficients given by series (supported on the primes, in our
case).

The {\em circle method} reduces an additive problem -- that is, a problems on
sums, such as sums of primes, powers, etc. -- to the estimation of an integral
on the space of frequencies (the ``circle'' $\mathbb{R}/\mathbb{Z}$).
In the case of the primes, as we have just discussed,
 we have precise estimates on the integrand on
part of the circle (the major arcs), and upper 
bounds on the rest of the circle (the minor arcs). 
Putting them together efficiently
to give an estimate on the integral is a delicate matter;
we leave it for the last part, as it is really what is particular to our
problem, as opposed to being of immediate
general relevance to the study of the primes. As we shall see, estimating
the integral well does involve using -- and improving -- general estimates
on the variance of irregularities in the distribution of the primes, as
given by the {\em large sieve}. 

In fact, one of the main general
lessons of the proof is that there is a very close relationship between
the circle method and the large sieve; we will use the large sieve not 
just as a tool -- which we shall, incidentally, sharpen in certain contexts --
but as a source for ideas on how to apply the circle method more effectively.

This has been an attempt at a first look from above. Let us now undertake
a more leisurely and detailed overview of the problem and its solution.

\section{History and new developments}

The history of the conjecture starts properly with Euler and 
his close friend, Christian Goldbach, both of whom lived and worked in Russia
at the time of their correspondence -- about a century after Descartes'
isolated statement. 
Goldbach, a man of many interests,
is usually classed as a serious amateur; he seems to have awakened Euler's 
passion for number theory, which would lead to the 
beginning of the modern era of the subject \cite[Ch. 3, \S IV]{MR734177}.
In a letter dated June 7, 1742,
Goldbach made a conjectural statement on prime numbers, and Euler rapidly
reduced it to the following conjecture, which, he said, Goldbach had already
posed to him: every positive integer can be written as the sum of at most
three prime numbers.

We would now say ``every integer greater than $1$'', since we no long
consider $1$ to be a prime number. Moreover, the conjecture is nowadays split 
into two: 
\begin{itemize}
\item the {\em weak}, or ternary, Goldbach conjecture states that every odd integer greater than $5$ can be written as the sum of three primes; 
\item the {\em strong}, or binary, Goldbach conjecture states that every even integer greater than $2$ can be written as the sum of two primes.
\end{itemize}
As their names indicate, the strong conjecture implies the weak one 
(easily: subtract $3$ from your odd number $n$, then express $n-3$ as the sum 
of two primes).

The strong conjecture remains out of reach. A short while ago -- the first
complete version appeared on May 13, 2013 -- the author
 proved the weak Goldbach conjecture.

\begin{theorem}
Every odd integer
greater than $5$ can be written as the sum of three primes.
\end{theorem}

In 1937, I. M. Vinogradov proved \cite{zbMATH02522879} that the conjecture is true for all odd numbers $n$ 
larger than some constant $C$. 
(Hardy and Littlewood had proved the same statement 
under the assumption of the 
Generalized Riemann Hypothesis, which we shall have the chance to discuss
later.) 

It is clear that a computation can verify the conjecture only for $n\leq c$,
$c$ a constant: computations have to be finite. What can make a result
coming from analytic number theory be valid only for $n\geq C$?

An analytic proof, generally speaking, gives us more than just existence. In this kind of problem, it gives us more than the possibility of doing
something (here, writing an integer $n$ as the sum of three
primes). It gives us a rigorous estimate for the number of ways in 
which this {\em something}
is possible; that is, it shows us that this number of ways equals
\begin{equation}\label{eq:huta}
\text{main term} + \text{error term},\end{equation}
where the main term is a precise quantity $f(n)$, and the error term is  
something whose absolute value is at most another precise quantity $g(n)$.
If $f(n)>g(n)$, then (\ref{eq:huta}) is non-zero, i.e., we will have
shown the existence of a way to write our number as the sum of
three primes.

(Since what we truly care about is existence, we are free to weigh different
ways of writing $n$ as the sum of three primes however we wish -- that is,
we can decide that some primes ``count'' twice or thrice as much as others,
and that some do not count at all.)

Typically, after much work, we succeed in obtaining (\ref{eq:huta})
with $f(n)$ and $g(n)$ such that $f(n)>g(n)$ asymptotically, that is,
for $n$ large enough. To give a highly simplified example:
if, say, $f(n) = n^2$ and $g(n) = 100 n^{3/2}$, then $f(n)>g(n)$ for
$n>C$, where $C=10^4$, and so the number of ways
(\ref{eq:huta}) is positive for $n>C$.

We want a moderate value of $C$, that is, a $C$ small enough that all cases
$n\leq C$ can be checked computationally. To ensure this, we must make
the error term bound $g(n)$ as small as possible. This is our main
task. A secondary (and sometimes neglected)
 possibility is to rig the weights so as to make the main term
$f(n)$ larger in comparison to $g(n)$; this can generally be done only up to a
certain point, but is nonetheless very helpful.

As we said, the first unconditional proof that odd numbers $n\geq C$
can be written as the sum of three primes is due to Vinogradov.
Analytic bounds fall into several categories, or stages; quite often, successive
versions of the same theorem will go through successive stages.
\begin{enumerate}
\item An {\em ineffective} result shows that a statement is true for some
constant $C$, but gives no way to determine what the constant $C$ might be.
Vinogradov's first proof of his theorem (in \cite{zbMATH02522879}) is like this:
it shows that there exists a constant $C$ such that
 every odd number $n>C$ is the sum of three primes, yet give us no hope
of finding out what the constant $C$ might be.\footnote{Here, as is often the 
case in ineffective results in analytic number theory, the underlying 
issue is that of {\em Siegel zeros}, which are believed not to exist, but
have not been shown not to; the strongest bounds on (i.e., against) such zeros
are ineffective, and so are all of the many results using such estimates.}
Many proofs of Vinogradov's result in textbooks are also of this type.
\item An {\em effective}, but not explicit, result shows that a statement
is true for some unspecified constant $C$ in a way that makes it clear that
a constant $C$ could in principle be determined following and reworking
the proof with great care. Vinogradov's later proof 
(\cite{zbMATH03063033}, translated in \cite{MR0062183}) is of this nature.
As Chudakov \cite[\S IV.2]{MR0031961} pointed out, the improvement
on \cite{zbMATH02522879} given
by Mardzhanishvili \cite{Mardzh} already had the effect of making
the result effective.\footnote{The proof in \cite{Mardzh} 
combined the bounds in \cite{zbMATH02522879} with a more careful accounting of
the effect of the single possible Siegel zero within range.}
\item An {\em explicit} result gives a value of $C$. 
According to \cite[p. 201]{MR0031961}, the first explicit
  version of Vinogradov's result was given by Borozdkin in his unpublished
doctoral dissertation, written under the direction of Vinogradov 
(1939): $C = \exp(\exp(\exp(41.96)))$. Such a result is, by definition,
also effective. Borodzkin later
  \cite{Borodzkin} 
gave the value $C = e^{e^{16.038}}$, though he does not seem to have published
the proof. The best -- that is, smallest -- value of $C$ known before the 
present work was that of Liu and 
Wang \cite{MR1932763}: $C = 2\cdot 10^{1346}$.
\item What we may call an {\em efficient} proof gives a reasonable value
for $C$ -- in our case, a value small enough that checking all cases up to
$C$ is feasible.
\end{enumerate}

How far were we from an efficient proof? That is, what sort of computation 
could ever be feasible? The situation was paradoxical: the conjecture was
known above an explicit $C$, but $C=2 \cdot 10^{1346}$ is
 so large that it could not be said that the problem could be
attacked by any foreseeable computational means within our physical universe.
(A truly brute-force verification up to $C$ takes at least $C$ steps;
a cleverer verification takes well over $\sqrt{C}$ steps.
The number of picoseconds since the beginning of the universe is less than 
$10^{30}$, whereas the number of protons in the observable universe is currently
estimated at $\sim 10^{80}$ \cite{EB}; this limits the number of steps that
can be taken in any currently imaginable computer, even if it were to do 
parallel processing
on an astronomical scale.) Thus, the only way forward was a series of drastic 
improvements in the mathematical, rather than computational, side.

I gave a proof with $C=10^{29}$ in May 2013. Since D. Platt and I had verified
the conjecture for all odd numbers up to $n\leq 8.8\cdot 10^{30}$
by computer \cite{MR3171101}, this established the conjecture for all
odd numbers $n$.

(In December 2013, I reduced $C$ 
 to $10^{27}$. The verification
of the ternary Goldbach conjecture up to $n\leq 10^{27}$ can be done
on a home computer over a weekend, 
as of the time of writing (2014). It must be said that this uses 
the verification of the binary Goldbach conjecture for $n\leq 4\cdot 10^{18}$
\cite{OSHP}, which itself required computational resources far outside the 
home-computing range. 
Checking the conjecture up to $n\leq 10^{27}$ was not even 
 the main computational task that needed to be accomplished to establish
the Main Theorem -- that task was
the finite verification of zeros of $L$-functions in 
\cite{Plattfresh}, 
a general-purpose computation that should be useful elsewhere.)

What was the strategy of the proof?
The basic framework is the one pioneered by
Hardy and Littlewood for a variety of problems -- namely, the
{\em circle method}, which, as we shall see, is an application of
Fourier analysis over $\mathbb{Z}$. 
(There are other, later routes to Vinogradov's result;
see \cite{MR834356}, \cite{MR1670069} 
and especially the recent work \cite{MR3165421}, which avoids using anything
about zeros of $L$-functions inside the critical strip.) 
Vinogradov's proof, 
like much of the later work on the subject, was based on a detailed analysis of
exponential sums, i.e., Fourier transforms over $\mathbb{Z}$. So is
the proof that we will sketch.

At the same time, the distance between $2\cdot 10^{1346}$ and $10^{27}$
is such that we cannot hope to get to $10^{27}$ (or any other
reasonable constant) by fine-tuning previous work.
Rather, we must work from scratch, using the basic outline
in Vinogradov's original proof and other, initially unrelated, developments
in analysis and number theory (notably, the large sieve). Merely improving
 constants will not do; rather, we must do qualitatively better than
previous work (by non-constant factors) if we are to have any chance to
succeed. It is on these qualitative improvements that we
will focus.

\begin{center}
* * *
\end{center}

It is only fair to review some of the progress made between Vinogradov's
time and ours. Here we will focus on results; later, 
we will discuss some of the progress made in the techniques of proof.
See \cite[Ch.~XVIII]{MR0245499} for the early history of the problem
(before Hardy and Littlewood); see R. Vaughan's
ICM lecture notes on the ternary Goldbach problem \cite{MR562631} for some further
details on the history up to 1978.

In 1933, Schnirelmann proved 
\cite{MR1512821} that every integer $n>1$ can be written as the sum
of at most $K$ primes for some unspecified constant $K$. (This pioneering work
is now considered to be
part of the early history of additive combinatorics.) 
In 1969, Klimov gave an explicit value for $K$ (namely, 
$K = 6\cdot 10^9$); he later improved the constant to
$K = 115$ (with G. Z. Piltay and T. A. Sheptickaja) and $K = 55$. 
Later, there were results by Vaughan \cite{MR0437478} ($K=27$), 
Deshouillers
\cite{MR0466050} ($K=26$) and Riesel-Vaughan \cite{MR706639} ($K=19$).

Ramar\'e showed in 1995 that every even number $n>1$ can be written
as the sum of at most $6$ primes \cite{MR1375315}. In 2012, Tao proved
\cite{Tao} that every odd number $n>1$ is the sum of at most $5$ primes.

There have been other avenues of attack towards the strong conjecture.
Using ideas close to those of Vinogradov's,
Chudakov \cite{zbMATH03028355}, 
\cite{Chudtoo}, Estermann  \cite{MR1576891} and
van der Corput \cite{zbMATH02522863} 
 proved (independently from each other) that almost every even number
(meaning: all elements of a subset of density $1$ in the even numbers)
can be written as the sum of two primes. In 1973, J.-R. Chen
showed \cite{MR0434997} that every even number $n$ larger than a constant
$C$ can be written as the sum of a prime number and the product of
at most two primes ($n = p_1 + p_2$ or 
$n = p_1 + p_2 p_3$). Incidentally, J.-R. Chen himself, together with
T.-Z. Wang, was responsible for the best bounds on $C$ (for ternary
 Goldbach) before Lui and Wang: 
$C = \exp(\exp(11.503)) < 4\cdot 10^{43000}$
\cite{MR1046491} and
$C = \exp(\exp(9.715)) < 6\cdot
10^{7193}$ \cite{MR1411958}.

Matters are different if one assumes the Generalized Riemann
Hypothesis (GRH). A careful analysis \cite{MR1715106}
of Hardy and Littlewood's work \cite{MR1555183} gives that every 
odd number $n\geq 1.24\cdot 10^{50}$ is the sum of three primes if GRH
is true\footnote{In fact, Hardy, Littlewood and Effinger use an assumption
  somewhat weaker than GRH: they assume that Dirichlet $L$-functions have
  no zeroes satisfying $\Re(s)\geq \theta$, where $\theta<3/4$ is arbitrary.
(We will review Dirichlet $L$-functions in a minute.)}.
According to \cite{MR1715106}, the same statement with
 $n\geq 10^{32}$ was proven in the unpublished doctoral dissertation
of B. Lucke, a student of E. Landau's, in 1926. Zinoviev \cite{MR1462848}
improved this to $n\geq 10^{20}$. A computer check (\cite{MR1469323}; see also
\cite{MR1451327}) showed that the conjecture is true for $n<10^{20}$,
thus completing the proof of the ternary Goldbach conjecture 
under the assumption of GRH. What was open until now was, of course,
the problem of giving an unconditional proof.
 

\section{The circle method: Fourier analysis on $\mathbb{Z}$}

It is common for a first course on Fourier analysis to focus
on functions over the reals satisfying $f(x)=f(x+1)$, or, what is the same,
functions $f:\mathbb{R}/\mathbb{Z} \to \mathbb{C}$. Such a function
(unless it is fairly pathological) has a Fourier series converging to it;
this is just the same as saying that $f$ has a Fourier transform
$\widehat{f}:\mathbb{Z}\to \mathbb{C}$ defined by 
$\widehat{f}(n) = \int_{\mathbb{R}/\mathbb{Z}} f(\alpha) e(-\alpha n) d\alpha$ 
and satisfying $f(\alpha) = \sum_{n\in \mathbb{Z}} \widehat{f}(n) e(\alpha
n)$
({\em Fourier inversion theorem}), where $e(t) = e^{2\pi i t}$.

In number theory, we are especially interested in functions $f:\mathbb{Z}
\to \mathbb{C}$. Then things are exactly the other way around: provided
that $f$ decays reasonably fast as $n\to \pm \infty$ (or becomes $0$ for $n$ large enough), 
$f$ has a Fourier transform $\widehat{f}:\mathbb{R}/\mathbb{Z} \to
\mathbb{C}$ defined by 
$\widehat{f}(\alpha) = \sum_n f(n) e(-\alpha n)$
and satisfying $f(n) = \int_{\mathbb{R}/\mathbb{Z}} \widehat{f}(\alpha) e(\alpha n)d\alpha$.
(Highbrow talk: we already knew that $\mathbb{Z}$ is the Fourier dual of
$\mathbb{R}/\mathbb{Z}$, and so, of course, 
$\mathbb{R}/\mathbb{Z}$ is the Fourier dual of $\mathbb{Z}$.) 
``Exponential sums'' (or ``trigonometrical sums'', as in the title of
\cite{MR0062183})
are sums of the form $\sum_n f(\alpha) e(-\alpha n)$; of course, the ``circle''
in ``circle method'' is just a name for $\mathbb{R}/\mathbb{Z}$.
(To see an actual circle in the complex plane, look at the image of
$\mathbb{R}/\mathbb{Z}$ under the map $\alpha\mapsto e(\alpha)$.)

The study of the Fourier transform $\widehat{f}$ 
is relevant to additive problems in number theory, i.e., questions
on the number of ways of writing $n$ as a sum of $k$ integers of a particular
form. Why? One answer could be that $\widehat{f}$ gives us information
about the ``randomness'' of $f$; if $f$ were the characteristic function
of a random set, then $\widehat{f}(\alpha)$ would be very small outside a sharp
peak at $\alpha=0$.

We can also give a more concrete and immediate answer. Recall that, in 
general, the Fourier transform of a convolution
equals the product of the transforms; over $\mathbb{Z}$, this means that
for the additive convolution
\[(f\ast g)(n) = \mathop{\sum_{m_1, m_2\in \mathbb{Z}}}_{m_1+m_2 = n} f(m_1) g(m_2),\]
the Fourier transform satisfies the simple rule
\[\widehat{f\ast g}(\alpha) = \widehat{f}(\alpha) \cdot \widehat{g}(\alpha).\]

We can see right away from this that $(f\ast g)(n)$ can be non-zero only if $n$ 
can be written as $n=m_1+m_2$ for some $m_1$, $m_2$ such that $f(m_1)$ and 
$g(m_2)$ are non-zero. Similarly, $(f\ast g\ast h)(n)$ can be non-zero only if 
$n$ can be written as $n=m_1+m_2+m_3$ for some $m_1$, $m_2$, $m_3$ such that 
$f(m_1)$, $f_2(m_2)$ and $f_3(m_3)$ are all non-zero. This suggests that, to 
study the ternary Goldbach problem, we define $f_1, f_2, f_3: \mathbb{Z}\to
\mathbb{C}$ so that they take non-zero values only at the primes.

Hardy and Littlewood defined $f_1(n) = f_2(n) = f_3(n) = 0$ for $n$ non-prime 
(and also for $n\leq 0$), and $f_1(n) = f_2(n) = f_3(n) = (\log n) e^{-n/x}$ 
for $n$ prime (where $x$ is a parameter to be fixed later). 
Here the factor $e^{-n/x}$ is there to provide ``fast decay'', so that 
everything converges; as we will see later, Hardy and Littlewood's choice of 
$e^{-n/x}$ (rather than some other function of fast decay) comes across in
hindsight as being very clever, though not quite best-possible. (Their
``choice'' was, to some extent, not a choice, but
 an artifact of their version of the circle method,
which was framed in terms of power series, not in terms of exponential sums
with arbitrary smoothing functions.) The term $\log n$ is there for technical reasons -- in essence, it makes sense to put it there because a random
integer around $n$ has a chance of about $1/(\log n)$ of being prime.

We can see that $(f_1\ast f_2\ast f_3)(n) \ne 0$ if and only if $n$ can be written as  the sum of three primes. Our task is then to show that 
$(f_1\ast f_2\ast f_3)(n)$ 
(i.e., $(f\ast f\ast f)(n)$) is non-zero for every $n$ larger than a constant
$C\sim 10^{27}$. 
Since the transform of a convolution equals a product of 
transforms,
\begin{equation}\label{eq:vanes}
(f_1\ast f_2\ast f_3)(n) 
= \int_{\mathbb{R}/\mathbb{Z}} 
\widehat{f_1\ast f_2\ast f_3}(\alpha) e(\alpha n) d\alpha = 
\int_{\mathbb{R}/\mathbb{Z}} (\widehat{f_1} \widehat{f_2} \widehat{f_3})(\alpha) e(\alpha n) d\alpha.\end{equation}
 Our task is thus to show that the integral 
$\int_{\mathbb{R}/\mathbb{Z}} (\widehat{f_1} \widehat{f_2} \widehat{f_3})(\alpha) e(\alpha n) d\alpha$
is non-zero.

As it happens, $\widehat{f}(\alpha)$ is particularly large when $\alpha$ is close to a rational with small denominator.
Moreover, for such $\alpha$, it turns out we can actually give rather
precise estimates for $\widehat{f}(\alpha)$. Define $\mathfrak{M}$
(called the set of {\em major arcs}) to be a union of narrow arcs around the rationals with small denominator:
\[\mathfrak{M} = \bigcup_{q\leq r} \mathop{\bigcup_{a \mo q}}_{(a,q)=1}
\left(\frac{a}{q} - \frac{1}{q Q} , \frac{a}{q} + \frac{1}{q Q}\right),\]
where $Q$ is a constant times $x/r$, and $r$ will be set later.
(This is a slight simplification:
the major-arc set we will actually use in the course of the proof
will be a little different, due to a distinction between odd and even $q$.)
We can write
\begin{equation}\label{eq:willow}
\int_{\mathbb{R}/\mathbb{Z}} (\widehat{f_1} \widehat{f_2} \widehat{f_3})(\alpha) e(\alpha n) d\alpha = 
\int_{\mathfrak{M}} (\widehat{f_1} \widehat{f_2} \widehat{f_3})(\alpha) e(\alpha n) d\alpha +
\int_{\mathfrak{m}} (\widehat{f_1} \widehat{f_2} \widehat{f_3})(\alpha) e(\alpha n) d\alpha,\end{equation}
where $\mathfrak{m}$ is the complement 
$(\mathbb{R}/\mathbb{Z})\setminus \mathfrak{M}$ (called {\em minor arcs}).

Now, we simply do not know how to give precise estimates for 
$\widehat{f}(\alpha)$ when $\alpha$ is in $\mathfrak{m}$. However, as
Vinogradov realized, one can give reasonable upper bounds on 
$|\widehat{f}(\alpha)|$ for $\alpha\in \mathfrak{m}$. This suggests
the following strategy: show that
\begin{equation}\label{eq:wtree}
\int_{\mathfrak{m}} |\widehat{f_1}(\alpha)| |\widehat{f_2}(\alpha)| 
|\widehat{f_3}(\alpha)| d\alpha 
<
\int_{\mathfrak{M}} \widehat{f_1}(\alpha) \widehat{f_2}(\alpha)
 \widehat{f_3}(\alpha) e(\alpha n) d\alpha.
\end{equation}
By (\ref{eq:vanes}) and (\ref{eq:willow}), this will imply immediately that
$(f_1\ast f_2\ast f_3)(n)  > 0$, and so we will be done.

The name of {\em circle method} is given to the study of additive problems
by means of Fourier analysis over $\mathbb{Z}$, and, in particular,
to the use of a subdivision of the circle $\mathbb{R}/\mathbb{Z}$ into
major and minor arcs to estimate the integral of a Fourier transform.
There was a ``circle'' already in Hardy and Ramanujan's work \cite{MR2280879},
but 
the subdivision into major and minor arcs is due to Hardy and Littlewood,
who also applied their method to a wide variety
of additive problems. (Hence ``the Hardy-Littlewood method'' as an alternative
name for the circle method.) For instance,
before working on the ternary Goldbach conjecture, they 
studied the question of whether every $n>C$ can be
written as the sum of $k$th powers (Waring's problem). In fact, they used
a subdivision into major and minor arcs to study Waring's problem, and not
for the ternary Goldbach problem: they had no
minor-arc bounds for ternary Goldbach, 
and their use of GRH had the effect of making
every $\alpha\in \mathbb{R}/\mathbb{Z}$ yield to a major-arc treatment.

Vinogradov worked with finite exponential sums, i.e., $f_i$ compactly 
supported. From today's perspective, it is clear that there are applications
(such as ours) in which it can be more important for $f_i$ to be
smooth than compactly supported; still, Vinogradov's simplifications
were an incentive to further developments. In the case of the ternary
 Goldbach's problem, his key contribution consisted in the fact that he
 could give bounds on $\widehat{f}(\alpha)$ for $\alpha$ in the minor arcs
without using GRH.

An important note: in the case of the binary Goldbach conjecture, the method
fails at (\ref{eq:wtree}), and not before; if our understanding of
the actual value of $\widehat{f_i}(\alpha)$ is at all correct, it is simply
not true in general that
\[\int_{\mathfrak{m}} |\widehat{f_1}(\alpha)| |\widehat{f_2}(\alpha)| d\alpha
<
\int_{\mathfrak{M}} \widehat{f_1}(\alpha) 
\widehat{f_2}(\alpha) e(\alpha n) d\alpha.\]
Let us see why this is not surprising. 
Set $f_1=f_2=f_3=f$ for simplicity, so that we have the integral of the square
$(\widehat{f}(\alpha))^2$
for the binary problem, and the integral of the cube 
$(\widehat{f}(\alpha))^3$ for the ternary problem.
Squaring, like cubing, amplifies the peaks of $\widehat{f}(\alpha)$, which are at
the rationals of small denominator and their immediate neighborhoods
(the major arcs); however, cubing amplifies the peaks much more than squaring.
This is why, even though the arcs making up $\mathfrak{M}$ are very narrow,
$\int_{\mathfrak{M}} (\widehat{f}(\alpha))^3 e(\alpha n) d\alpha$ is larger than
$\int_{\mathfrak{m}} |\widehat{f}(\alpha)|^3 d\alpha$; that explains the name
{\em major arcs} -- they are not large, but they
give the major part of the contribution. In contrast, squaring
amplifies the peaks less, and this is why the absolute value of
$\int_{\mathfrak{M}} \widehat{f}(\alpha)^2 e(\alpha n) d\alpha$ is in general smaller than
$\int_{\mathfrak{m}} |\widehat{f}(\alpha)|^2 d\alpha$. As nobody knows how to prove
a precise estimate (and, in particular, lower bounds) on
$\widehat{f}(\alpha)$ for $\alpha \in \mathfrak{m}$, the binary Goldbach conjecture is
still very much out of reach.

To prove the ternary Goldbach conjecture, it is enough to estimate both sides
of (\ref{eq:wtree}) for carefully chosen $f_1$, $f_2$, $f_3$, and compare them.
This is our task from now on.

\section{The major arcs $\mathfrak{M}$}\label{sec:rubosa}
\subsection{What do we really know about $L$-functions and their zeros?}
Before we start, let us give 
a very brief review of basic analytic number theory
(in the sense of, say, \cite{MR0217022}). A {\em Dirichlet character}
$\chi:\mathbb{Z}\to \mathbb{C}$ of modulus $q$ is a character of 
$(\mathbb{Z}/q \mathbb{Z})^*$ lifted to $\mathbb{Z}$. (In other words,
$\chi(n) = \chi(n+q)$ for all $n$, $\chi(a b) = \chi(a) \chi(b)$ for all $a$, $b$
and $\chi(n)=0$ for $(n,q)\ne 1$.) A {\em Dirichlet $L$-series}
is defined by 
\[L(s,\chi) = \sum_{n=1}^\infty \chi(n) n^{-s}\]
for $\Re(s)>1$, and by analytic continuation for $\Re(s)\leq 1$. (The Riemann
zeta function $\zeta(s)$ is the $L$-function for the trivial character, i.e.,
the character $\chi$ such that $\chi(n)=1$ for all $n$.) Taking logarithms
and then derivatives, we see that
\begin{equation}\label{eq:koj}
- \frac{L'(s,\chi)}{L(s,\chi)} = \sum_{n=1}^\infty \chi(n) \Lambda(n) n^{-s},
\end{equation}
for $\Re(s)>1$,
where $\Lambda$ is the {\em von Mangoldt function} ($\Lambda(n) = \log p$ if $n$ is some prime power $p^\alpha$, $\alpha\geq
1$, and $\Lambda(n)=0$ otherwise).

Dirichlet introduced his characters
and $L$-series so as to study primes in arithmetic progressions.
In general, and after some work,
(\ref{eq:koj}) allows us to restate many sums over the primes (such as our 
Fourier
transforms $\widehat{f}(\alpha)$) as sums over the zeros of $L(s,\chi)$. 
 A {\em non-trivial zero} of $L(s,\chi)$
is a zero of $L(s,\chi)$ such that $0<\Re(s)<1$. (The other zeros 
are called trivial because we know where they are, namely, at negative
integers and, in some cases, also on the line $\Re(s)=0$. In order
to eliminate all zeros on $\Re(s)=0$ outside $s=0$, it suffices to
assume that $\chi$ is {\em primitive}; a primitive character modulo $q$
is one that is not induced by (i.e., not the restriction of) 
any character modulo $d|q$, $d<q$.) 

The Generalized
Riemann Hypothesis for Dirichlet $L$-functions is the statement that, for
every Dirichlet character $\chi$, every non-trivial zero of $L(s,\chi)$
satisfies $\Re(s) = 1/2$. Of course, the Generalized Riemann Hypothesis
(GRH) -- and the Riemann Hypothesis, which is the special case of
$\chi$ trivial --
remains unproven. Thus, if we want to prove unconditional statements, we
need to make do with partial results towards GRH. 
Two kinds of such results have been proven:
\begin{itemize}
\item \textbf{Zero-free regions.} Ever since the late nineteenth century
(Hadamard, de la Vall\'ee-Poussin) we have known that there are
hourglass-shaped
regions (more precisely, of the shape $\frac{c}{\log t} \leq \sigma \leq
1 - \frac{c}{\log t}$, where $c$ is a constant and where we write
$s = \sigma + i t$) outside which non-trivial zeros cannot lie. Explicit
values for $c$ are known \cite{MR751161}, \cite{MR2140161}, \cite{Habiba}. 
There is
also the Vinogradov-Korobov region \cite{MR0106205}, \cite{MR0103861},
which is broader asymptotically but narrower in most of
the practical range (see \cite{MR1936814}, however).

\item \textbf{Finite verifications of GRH.} It is possible to (ask a
  computer to) prove small, finite fragments of GRH, in the sense of verifying 
that all non-trivial zeros of a given finite set of $L$-functions 
with imaginary part less than some constant $H$ lie on the critical line 
$\Re(s)=1/2$. Such verifications go back to Riemann, who checked the first
few zeros of $\zeta(s)$. Large-scale, rigorous
computer-based verifications are now a possibility.
\end{itemize}

Most work in the literature follows the first alternative, though
\cite{Tao} did use a finite verification of RH (i.e., GRH for the trivial
character). Unfortunately, zero-free regions seem too narrow to be useful
for the ternary Goldbach problem. Thus, we are left with the second
alternative. 

In coordination with the present work, Platt \cite{Plattfresh} verified that
all zeros $s$ of $L$-functions for characters $\chi$ with modulus $q\leq
300000$ satisfying $\Im(s)\leq H_q$ lie on the line $\Re(s)=1/2$, where
\begin{itemize}
\item $H_q = 10^8/q$ for $q$ odd, and
\item $H_q = \max(10^8/q,200+7.5\cdot 10^7/q)$ for $q$ even.
\end{itemize}
This was a medium-large computation, taking a few hundreds of thousands
of core-hours on a parallel computer. It used {\em interval arithmetic} for
the sake of rigor; we will later discuss what this means.

The choice to use a finite verification of GRH, rather than zero-free
regions, had consequences on the manner in which the major and minor
arcs had to be chosen. As we shall see, such a verification can be used
to give very precise bounds on the major arcs, but also forces us to
define them so that they are narrow and their number is constant.
To be precise: the major arcs were defined around rationals $a/q$ with
$q\leq r$, $r=300000$; moreover, as will become clear,
the fact that $H_q$ is finite will
force their width to be bounded by $c_0 r/q x$, where $c_0$ is a constant
(say $c_0=8$). 


\subsection{Estimates of $\widehat{f}(\alpha)$ for $\alpha$ in the major
  arcs}\label{subs:karm}

Recall that we want to estimate sums of the type 
$\widehat{f}(\alpha) = \sum f(n) e(-\alpha n)$, where $f(n)$ is something
like $(\log n) \eta(n/x)$ for $n$ equal to a prime, and $0$ otherwise;
here $\eta:\mathbb{R}\to \mathbb{C}$ is some function of fast decay, 
such as Hardy
and Littlewood's choice, \[\eta(t) =
\begin{cases} e^{-t} &\text{for $t\geq 0$},\\ 0 &\text{for $t<0$.}
  \end{cases}\]
 Let us modify this just a little -- we will actually estimate 
\begin{equation}\label{eq:holo}
S_\eta(\alpha,x) = \sum \Lambda(n) e(\alpha n) \eta(n/x),\end{equation}
where $\Lambda$ is the von Mangoldt function (as in (\ref{eq:koj})) .
The use of $\alpha$ rather than $-\alpha$ is just a bow to tradition, as is
the use of the letter $S$ (for ``sum''); however, the use of $\Lambda(n)$
rather than just plain $\log p$ does actually simplify matters.

The function $\eta$ here is sometimes called a {\em smoothing function} or
simply a {\em smoothing}. It will
indeed be helpful for it to be smooth on $(0,\infty)$, but,
in principle, it need not even be continuous. (Vinogradov's work implicitly
uses, in effect, the ``brutal truncation'' $1_{\lbrack 0,1\rbrack}(t)$,
defined to be $1$ when $t\in \lbrack 0,1\rbrack$ and $0$ otherwise; that
would be fine for the minor arcs, but, as it will become clear, it is a bad idea as
far as the major arcs are concerned.)

Assume $\alpha$ is on a major arc, meaning that we can write 
$\alpha = a/q + \delta/x$ for some $a/q$ ($q$ small) and some $\delta$ 
(with $|\delta|$ small). We can write
 $S_\eta(\alpha,x)$ as a linear combination 
\begin{equation}\label{eq:sorrow}
S_\eta(\alpha,x) = \sum_\chi c_\chi
S_{\eta,\chi}\left(\frac{\delta}{x},x\right)
+ \text{tiny error term},\end{equation} where 
\begin{equation}\label{eq:battle}
S_{\eta,\chi}\left(\frac{\delta}{x},x\right) = \sum \Lambda(n) \chi(n)
e(\delta n/x) \eta(n/x).\end{equation}
In (\ref{eq:sorrow}),
 $\chi$ runs over primitive Dirichlet characters of moduli $d|q$, and
$c_\chi$ is small ($|c_\chi|\leq \sqrt{d}/\phi(q)$).

Why are we expressing the sums $S_\eta(\alpha,x)$ 
in terms of the sums $S_{\eta,\chi}(\delta/x,x)$,
which look more complicated?
The argument has become $\delta/x$, whereas before it was $\alpha$. 
Here $\delta$ is relatively small 
-- smaller than the constant $c_0 r$, in our setup. 
In other words, $e(\delta n/x)$ will go around the circle a bounded number of 
times as $n$ goes from $1$ up to a constant times $x$ 
(by which time $\eta(n/x)$ has become small, because $\eta$ is of fast decay). 
This makes the sums much easier to estimate.

To estimate the sums $S_{\eta,\chi}$, we will use $L$-functions, together
with one of the most common tools of analytic number theory, the Mellin
transform.
This transform is essentially a Laplace transform with a change of
variables, and a Laplace transform, in turn, is a Fourier transform taken
on a vertical line in the complex plane. For $f$ of fast enough decay,
the Mellin transform $F=Mf$ of $f$ is given by
\[F(s) = \int_0^\infty f(t) t^s \frac{dt}{t};\]
we can express $f$ in terms of $F$ by the {\em Mellin inversion formula}
\[f(t) = \frac{1}{2\pi i} \int_{\sigma-i\infty}^{\sigma+i\infty} F(s) t^{-s} ds
\]
for any $\sigma$ within an interval. We can thus express
$e(\delta t) \eta(t)$ in terms of its Mellin transform $F_\delta$ and
then use (\ref{eq:koj}) to express $S_{\eta,\chi}$ in terms of $F_\delta$
and $L'(s,\chi)/L(s,\chi)$; shifting the integral in the
Mellin inversion formula to the left, we obtain what is known in analytic
number theory as an {\em explicit formula}: 
\[S_{\eta,\chi}(\delta/x,x) = \left\lbrack \widehat{\eta}(
-\delta) x\right\rbrack 
 - \sum_\rho F_\delta(\rho) x^\rho + \text{tiny error term}.\]
Here the term between brackets appears only for $\chi$ trivial. In the sum, 
$\rho$ goes over all non-trivial zeros of $L(s,\chi)$, and $F_\delta$ 
is the Mellin transform of $e(\delta t) \eta(t)$. (The tiny error term
comes from a sum over the trivial zeros of $L(s,\chi)$.)
We will obtain the estimate we desire
 if we manage to show that the sum over $\rho$ is small.

The point is this: if we verify GRH for $L(s,\chi)$ up to imaginary part
$H$, i.e., if we check that
all zeroes $\rho$ of $L(s,\chi)$ 
with $|\Im(\rho)|\leq H$ satisfy $\Re(\rho)=1/2$, we have
$|x^\rho| = \sqrt{x}$. In other words, $x^\rho$ 
is very small (compared to $x$). However, for any $\rho$ whose imaginary part has absolute value greater than $H$, we know next to nothing about its real part, other than $0\leq \Re(\rho)\leq 1$. (Zero-free regions
are notoriously weak for $\Im(\rho)$ large; we will not use them.) 
Hence, our only chance is to make sure that $F_\delta(\rho)$ is very small when $|\Im(\rho)|\geq H$. 

This has to be true for both $\delta$ very small 
(including the case $\delta=0$) and for $\delta$ not so small ($|\delta|$
up to $c_0 r/q$, which can be large because 
 $r$ is a large constant). How can we choose $\eta$ so that
$F_\delta(\rho)$ is very small in both cases for $\tau=\Im(\rho)$ large?

The method of {\em stationary phase} is useful as an exploratory tool here.
In brief, it suggests (and can sometimes prove) that the main 
contribution to the integral 
\begin{equation}\label{eq:narimsi}
F_\delta(t) = \int_0^\infty e(\delta t) \eta(t) t^s \frac{dt}{t}\end{equation}
can be found where 
the phase of the integrand has derivative $0$. 
This happens when $t= -\tau/2\pi \delta$ (for $\sgn(\tau)\ne \sgn(\delta)$);
the contribution is then a moderate factor times $\eta(-\tau/2\pi \delta)$.
In other words, if $\sgn(\tau)\ne \sgn(\delta)$ and $\delta$ is not too small
($|\delta|\geq 8$, say),
$F_\delta(\sigma + i\tau)$ behaves like $\eta(-\tau/2\pi \delta)$;
if $\delta$ is small ($|\delta|<8$), then $F_\delta$ behaves like $F_0$,
which is the Mellin transform $M\eta$ of $\eta$. Here is our goal, then: the decay of
$\eta(t)$ as $|t|\to \infty$ should be as fast as possible, and the
decay of the transform $M\eta(\sigma+ i \tau)$ should also be as fast as possible.

This is a classical dilemma, often called the {\em uncertainty principle} because
it is the mathematical fact underlying the physical principle of the same name:
you cannot have a function  
$\eta$ that decreases extremely rapidly and whose Fourier
transform (or, in this case, its Mellin transform) also decays extremely 
rapidly.

What does ``extremely rapidly'' mean here? It means (as Hardy himself proved)
``faster than any exponential $e^{- C t}$''. Thus, Hardy and Littlewood's
choice $\eta(t) = e^{-t}$ seems essentially optimal at first sight.

However, it is not optimal. We can choose $\eta$ so that $M\eta$
decreases exponentially (with a constant $C$ somewhat worse than for
$\eta(t)=e^{-t}$),
but $\eta$ decreases faster than exponentially. This is a particularly
appealing possibility because it is $t/|\delta|$, and not so much $t$,
that risks being fairly small. (To be explicit: say we check GRH for
characters of modulus $q$ up to $H_q\sim 50 \cdot c_0 r/q \geq 50 |\delta|$.
Then we only know that $|\tau/2\pi\delta| \gtrsim 8$. So, for
$\eta(t) = e^{-t}$, $\eta(-\tau/2\pi \delta)$ may be as large as $e^{-8}$,
which is not negligible. Indeed, since this term will be multiplied
later by other terms, $e^{-8}$ is simply not small enough. 
On the other hand,
we can assume that $H_q\geq 200$ (say), and so $M\eta(s) \sim e^{-(\pi/2)
|\tau|}$ is completely negligible, and will remain negligible even
if we replace $\pi/2$ by a somewhat smaller constant.)

We shall take $\eta(t) = e^{-t^2/2}$ (that is, the Gaussian). 
This is not the only 
possible choice, but it is in some sense natural. It is easy to show that 
the Mellin transform $F_\delta$ for $\eta(t) = e^{-t^2/2}$ is a multiple of what
is called a {\em parabolic cylinder function} $U(a,z)$ 
with imaginary values for $z$. There are plenty of 
estimates on parabolic cylinder functions in the literature -- but mostly
for $a$ and $z$ real, in part because that is one
 of the cases occuring most often in applications. There
are some asymptotic expansions and estimates for $U(a,z)$, $a$, $z$,
general, due to Olver \cite{MR0094496}, 
\cite{MR0109898}, \cite{MR0131580}, \cite{MR0185350},
but unfortunately they come without fully explicit error terms for $a$
and $z$ within our range of interest. (The same holds for \cite{MR1993339}.)

In the end, I derived bounds for $F_\delta$ using the {\em saddle-point method}.
(The method of stationary phase, which we used to choose $\eta$, seems to lead
to error terms that are too large.) The saddle-point method
consists, in brief, in changing the contour of an integral to be bounded
(in this case, (\ref{eq:narimsi})) so as to minimize the maximum of
the integrand.
(To use a metaphor in \cite{MR671583}:
find the lowest mountain pass.)

Here we strive to get clean bounds, rather than the best possible
constants. Consider the case $k=0$ of Corollary \ref{cor:amanita1} with
$k=0$; it states the following. 
For $s = \sigma + i\tau$ with $\sigma\in \lbrack 0,1\rbrack$ and
$|\tau|\geq \max(100,4\pi^2 |\delta|)$, we obtain that the Mellin transform
$F_\delta$ of $\eta(t) e(\delta t)$ with $\eta(t) = e^{-t^2/2}$ satisfies
\begin{equation}\label{eq:sturmo}
|F_\delta(s+k)|+ |F_\delta((1-s)+k)| \leq
\begin{cases}
3.001 e^{-0.1065 \left(\frac{2 |\tau|}{|\ell|}\right)^2}
& \text{if $4 |\tau|/\ell^2 < 3/2$.}\\
3.286 e^{- 0.1598 |\tau|}
& \text{if $4 |\tau|/\ell^2 \geq 3/2$.}
\end{cases}
\end{equation}

Similar bounds hold for $\sigma$ in other ranges, thus giving us
estimates on the Mellin transform $F_\delta$ for $\eta(t) = t^k e^{-t^2/2}$
and $\sigma$ in the critical range $\lbrack 0,1\rbrack$. 
(We could do a little better if we knew the value of $\sigma$, but, 
in our applications, 
we do not, once we leave the range in which GRH has been checked.
We will give a bound (Theorem~\ref{thm:princo}) that does take $\sigma$
into account, and also reflects and takes advantage of
the fact  that there is a transitional region around $|\tau|\sim 
(3/2) (\pi/\delta)^2$; in practice, however, we will use Cor.~\ref{cor:amanita1}.)

A moment's thought shows that we can also use (\ref{eq:sturmo}) to deal
with the Mellin transform of $\eta(t) e(\delta t)$ for any function of the
form $\eta(t) = e^{-t^2/2} g(t)$ (or, more generally, $\eta(t) = t^k e^{-t^2/2}
g(t)$), where $g(t)$ is any {\em band-limited function}. By a band-limited
function, we could mean a function whose Fourier transform is compactly
supported; while that is a plausible choice, it turns out to be better to
work with functions that are band-limited with respect to the Mellin transform
-- in the sense of being of the form
\[g(t) = \int_{-R}^R h(r) t^{-ir} dr,\]
where $h:\mathbb{R}\to \mathbb{C}$ is supported on a compact interval 
$\lbrack -R,R\rbrack$, with $R$ not too large (say $R=200$). What happens
is that the Mellin transform of the product $e^{-t^2/2} g(t) e(\delta t)$ 
is a convolution of the Mellin transform $F_\delta(s)$ 
of $e^{-t^2/2} e(\delta t)$ 
(estimated in (\ref{eq:sturmo})) and that of $g(t)$ (supported in 
$\lbrack -R,R\rbrack$); the effect of the convolution is just to delay
decay of $F_\delta(s)$ by, at most, a shift by $y\mapsto y-R$.

We wish to estimate $S_{\eta,\chi}(\delta/x)$ for several functions
$\eta$. This motivates us to derive an explicit formula (\S )
general enough to work with all the weights $\eta(t)$ 
we will work with, while being also completely explicit, and free of any 
integrals that may be tedious to evaluate. 

 Once that is done, and once we consider the
input provided by Platt's finite verification of GRH up to $H_q$, we
obtain simple bounds for different weights.

For $\eta(t)= e^{-t^2/2}$, $x\geq 10^8$, $\chi$ a primitive character
of modulus $q\leq r = 300000$, and any $\delta\in \mathbb{R}$ with
$|\delta|\leq 4r/q$, we obtain
\begin{equation}\label{eq:frank}
S_{\eta,\chi}\left(\frac{\delta}{x}, x\right)
= I_{q=1} \cdot \widehat{\eta}(-\delta) x
+ E\cdot x,\end{equation}
where $I_{q=1}=1$ if $q=1$, $I_{q=1}=0$ if $q\ne 1$,
and
\begin{equation}\label{eq:lynno}
|E|\leq
4.306 \cdot 10^{-22} + \frac{1}{\sqrt{x}}
\left( \frac{650400}{\sqrt{q}} + 112\right).\end{equation}
Here $\widehat{\eta}$ stands for the Fourier transform from $\mathbb{R}$
to $\mathbb{R}$ normalized as follows: $\widehat{\eta}(t) = 
\int_{-\infty}^\infty e(-xt) \eta(x) dx$. Thus, $\widehat{\eta}(-\delta)$
is just $\sqrt{2\pi} e^{-2 \pi^2 \delta^2}$ (self-duality of the
Gaussian).

This is one of the main results of Part \ref{part:maj}; see \S \ref{subs:results}. Similar bounds are
also proven there for $\eta(t) = t^2 e^{-t^2/2}$, as well as for a weight
of type $\eta(t) = t e^{-t^2/2} g(t)$, where $g(t)$ is a band-limited
function, and also for a weight $\eta$ defined by a multiplicative convolution. 
The conditions on $q$ (namely, $q\leq r = 300000$) and $\delta$ are what we
expected from the outset.

Thus concludes our treatment of the major arcs. This is arguably the easiest
part of the proof; it was actually what I left for the end, as I was fairly
confident it would work out. Minor-arc estimates are more delicate; 
let us now examine them.

\section{The minor arcs $\mathfrak{m}$}\label{sec:melusa}

\subsection{Qualitative goals and main ideas}
What kind of bounds do we need? What is there in the literature?

We wish to obtain upper bounds on $|S_\eta(\alpha,x)|$ for some weight $\eta$
and any $\alpha\in \mathbb{R}/\mathbb{Z}$ not very close to a rational
with small denominator. Every $\alpha$ is close to some rational $a/q$;
what we are looking for is a bound on $|S_\eta(\alpha,x)|$ that
decreases rapidly when $q$ increases.

Moreover, we want our bound to decrease rapidly when $\delta$ increases,
where $\alpha = a/q+\delta/x$. In fact, the main terms in our bound will
be decreasing functions of $\max(1,|\delta|/8)\cdot q$.
(Let us write $\delta_0 = \max(2,|\delta|/4)$ from now on.)
This will allow our bound to be good enough outside 
 narrow major arcs, which will get narrower and narrower as $q$ increases
-- that is, precisely the kind of major arcs we were presupposing in 
our major-arc bounds.

It would be possible to work with narrow major arcs that
become narrower as $q$ increases simply by allowing $q$ to be very large 
(close to $x$), and assigning each angle to the fraction closest to it.
This is, in fact, the common procedure. 
However, this makes matters more difficult, in that we would have to 
minimize at the same time the factors in front of terms $x/q$, $x/\sqrt{q}$,
etc., and those in front of terms $q$, $\sqrt{q x}$, and so on. 
(These terms are being compared to the trivial bound $x$.)
Instead, 
we choose to strive for a direct dependence on $\delta$ throughout; this
will allow us to cap $q$ at a much lower level, thus making terms such as
$q$ and $\sqrt{q x}$ negligible. (This choice has been taken
elsewhere in applications of the circle method, but, strangely, seems absent
from previous work on the ternary Goldbach conjecture.)

How good must our bounds be? Since the major-arc bounds are valid only
for $q\leq r=300000$ and $|\delta|\leq 4r/q$, we cannot afford even
a single factor of $\log x$ (or any other function tending to $\infty$
as $x\to \infty$) in front of terms such as $x/\sqrt{q |\delta_0|}$:
a factor like that would make the term larger than the trivial bound $x$
if $q |\delta_0|$ is equal to a constant ($r$, say) and $x$ is very large. 
Apparently, there was no such ``log-free bound''
with explicit constants in the literature, even though
such bounds
 were considered to be in principle feasible, and even though previous
work (\cite{MR813837}, \cite{MR1399341}, \cite{MR1803131}, \cite{Tao})
had gradually decreased the number of factors of $\log x$.
(In limited ranges for $q$, there were log-free bounds without
explicit constants; see \cite{MR1399341}, \cite{MR2607306}. The estimate in
\cite[Thm. 2a, 2b]{MR0062183} was almost log-free, but not quite.
There were also bounds \cite{MR1215269}, \cite{MR2776653}
that used $L$-functions, and thus were not
really useful in a truly minor-arc regime.)

It also seemed clear that a main bound proportional to $(\log q)^2 x/\sqrt{q}$
(as in \cite{Tao}) was too large. At the same time, it was not really
necessary to reach a bound of the best possible form that could
be found through Vinogradov's basic approach, namely
\begin{equation}\label{eq:astora}
|S_\eta(\alpha,x)|\leq C \frac{x \sqrt{q}}{\phi(q)}.\end{equation}
Such a bound had been proven by Ramar\'e \cite{MR2607306} for $q$
in a limited range and $C$ non-explicit; later, in \cite{Ramlater}
-- which postdates the first version of \cite{Helf} --
Ramar\'e  broadened the
range to $q\leq x^{1/48}$ and gave an explicit value for $C$, namely,
$C=13000$. Such a bound is a notable achievement, but, unfortunately, it
is not useful for our purposes. Rather, we will aim at a bound whose
main term is bounded by a constant around $1$ times 
$x (\log \delta_0 q)/\sqrt{\delta_0 \phi(q)}$; this is slightly worse
asymptotically than (\ref{eq:astora}), but it is much better in the delicate
range of $\delta_0 q \sim 300000$, and in fact for a much wider range as well.

\begin{center}
* * *
\end{center}

We see that we have several tasks. One of them is the removal of
logarithms: we cannot afford a single factor of $\log x$, and, in practice,
we can afford at most one factor of $\log q$. Removing logarithms will
be possible in part because of the use of previously existing
 efficient techniques (the large
sieve for sequences with prime support) but also because we will be able to find cancellation
at several places in sums coming from a combinatorial identity (namely, 
Vaughan's identity). The task of finding cancellation is particularly delicate
because we cannot afford large constants or, for that matter,
 statements valid only for large
$x$. (Bounding a sum such as $\sum_n \mu(n)$ efficiently, where $\mu$
is the {\em M\"obius function}
\[\mu(n) = \begin{cases} (-1)^k & \text{if $n=p_1 p_2\dotsc p_k$, all $p_i$
distinct}\\
0 & \text{if $p^2|n$ for some prime $p$,}\end{cases}\]
is harder than estimating a sum such as $\sum_n \Lambda(n)$ equally efficiently, even though we are used to thinking of the two problems as equivalent.)

We have said that our bounds will improve as $|\delta|$ increases. This
dependence on $\delta$ will be secured in different ways at different
places. Sometimes $\delta$ will appear as an argument, as in
$\widehat{\eta}(-\delta)$; for $\eta$ piecewise continuous with $\eta' \in
L_1$, we know that
$|\widehat{\eta}(t)| \to 0$ as $|t|\to \infty$. Sometimes we will
obtain a dependence on $\delta$ by using several different rational 
approximations to the same $\alpha \in \mathbb{R}$. Lastly, 
 we will obtain a good dependence on $\delta$ in bilinear
sums by supplying a scattered input to a large sieve.

If there is a main moral to the argument, it lies in the 
close relation between the circle method and
the large sieve. The circle method rests on the
estimation of an integral involving a Fourier transform $\widehat{f}:
\mathbb{R}/\mathbb{Z}\to \mathbb{C}$; as we will later see, this leads
naturally to estimating the $\ell_2$-norm of $\widehat{f}$ on subsets 
(namely, unions of arcs) of the circle $\mathbb{R}/\mathbb{Z}$. The large
sieve can be seen as an approximate discrete version of Plancherel's identity,
which states that $|\widehat{f}|_2 = |f|_2$. 

Both in this section and in \S \ref{sec:putall}, 
we shall use the large sieve in part so as to use the fact that some of the 
functions we work with have prime support, i.e., are non-zero only
on prime numbers. There are
ways to use prime support to improve the output of the large sieve.
 In \S \ref{sec:putall}, these techniques will be refined and
then translated to the context of the circle method,
where $f$ has (essentially) prime support and $|\widehat{f}|^2$ must be 
integrated over unions of arcs. (This allows us to remove a logarithm.)
The main point is that the large sieve is not being used as a black
box; rather, we can adapt ideas from (say) the large-sieve context and apply
them to the circle method.

Lastly, there are the benefits of a continuous $\eta$. Hardy and
Littlewood already used a continuous $\eta$; this was abandoned by Vinogradov, 
presumably for the sake of simplicity. The idea that smooth weights $\eta$
can be superior to sharp truncations is now commonplace. As we shall see,
using a continuous $\eta$ is helpful in the minor-arcs regime, but not
as crucial there as for the major arcs. We will not use a smooth $\eta$; we 
will prove our estimates for any continuous $\eta$ that is piecewise $C_1$,
and then, towards the end, we will choose to 
use the same weight $\eta=\eta_2$ as in \cite{Tao}, in part because it
has compact support, and in part for the sake of comparison. The moral here
is not quite the common dictum ``always smooth'', but rather that different kinds of smoothing can 
be appropriate for different tasks; in the end, we will show how
to coordinate different smoothing functions $\eta$. 

There are other ideas involved; for instance, some of Vinogradov's lemmas  
are improved. Let us now go into some of the details.
 
\subsection{Combinatorial identities}

Generally, since Vinogradov, a 
treatment of the minor arcs starts with a combinatorial identity expressing
$\Lambda(n)$ (or the characteristic function of the primes) as a sum of
two or more convolutions. (In this section, by a convolution $f\ast g$, we will 
mean the {\em Dirichlet convolution} $(f\ast g)(n) = \sum_{d|n} f(d) g(n/d)$,
i.e., the multiplicative convolution on the semigroup of positive integers.)

In some sense, the archetypical identity is \[\Lambda = \mu\ast \log,\]
but it will not usually do: the contribution of $\mu(d) \log(n/d)$ with
$d$ close to $n$ is too difficult
to estimate precisely. There are alternatives: for example, there is
the identity
\begin{equation}\Lambda(n) \log n = \mu\ast \log^2 - \Lambda\ast \Lambda,
\end{equation}
which underlies an estimate of Selberg's that, in turn, is the basis for 
the Erd\"os-Selberg proof of the prime number theorem;
see, e.g., \cite[\S 8.2]{MR2378655}.
More generally, one can decompose $\Lambda(n) (\log n)^k$ as
$\mu\ast \log^{k+1}$ minus a linear combination of convolutions; this 
kind of decomposition -- really
just a direct consequence of the development of $(\zeta'(s)/\zeta(s))^{(k)}$
-- will be familiar to some from the exposition of Bombieri's work  
\cite{MR0396435} in \cite[\S 3]{MR2647984} (for instance).
Another useful identity was that used by
Daboussi \cite{MR1399341}; witness its application in 
 \cite{MR1803131}, which gives explicit 
estimates on exponential sums over primes.

The proof of Vinogradov's three-prime result was simplified substantially 
\cite{MR0498434} by the introduction of {\em Vaughan's identity}:
\begin{equation}\label{eq:vaughan}
\Lambda(n) = \mu_{\leq U} \ast \log - \Lambda_{\leq V} \ast \mu_{\leq U} \ast 1
+ 1 \ast \mu_{>U} \ast \Lambda_{>V} + \Lambda_{\leq V},\end{equation}
where we are using the notation
\[f_{\leq W} = \begin{cases} f(n) &\text{if $n\leq W$,}\\ 0 
&\text{if $n>W$,}\end{cases}\;\;\;\;\;\;\;
f_{>W} = \begin{cases} 0 &\text{if $n\leq W$,}\\ f(n) 
&\text{if $n>W$.}\end{cases}\]
Of the resulting sums ($\sum_n (\mu_{\leq U}\ast \log)(n) e(\alpha n) \eta(n/x)$,
etc.), the first three
 are said to be of {\em type I}, {\em type I} (again) and {\em type II};
the last sum, $\sum_{n\leq V} \Lambda(n)$, is negligible.

One of the advantages of Vaughan's identity is its flexibility: we can
set $U$ and $V$ to whatever values we wish. Its main disadvantage is that
it is not ``log-free'', in that it seems to impose the loss of two factors of
$\log x$: if we sum each side of (\ref{eq:vaughan}) from $1$ to $x$,
we obtain $\sum_{n\leq x} \Lambda(n) \sim x$ on the left side, whereas,
if we bound the sum on the right side without the use of cancellation,
we obtain a bound of $x (\log x)^2$. Of course, we will obtain some
 cancellation from the phase $e(\alpha n)$; still, even if this gives
us a factor of, say, $1/\sqrt{q}$, we will get a bound of 
$x (\log x)^2/\sqrt{q}$,
which is worse than the trivial bound $x$ for $q$ bounded and $x$ large.
Since we want a bound that is useful for all $q$ larger than the constant $r$
and all $x$ larger than a constant, this will not do. 

As was pointed out in \cite{Tao}, it is possible to get a factor of
$(\log q)^2$ instead of a factor of $(\log x)^2$ in the type II sums
by setting $U$ and $V$ appropriately. Unfortunately, a factor of $(\log q)^2$
is still too large in practice, and there is also the issue of factors of
$\log x$ in type I sums.

Vinogradov had already managed to get an essentially log-free result
(by a rather difficult procedure) in \cite[Ch. IX]{MR0062183}. 
The result in \cite{MR1399341} is log-free. Unfortunately, the explicit result
in \cite{MR1803131} -- the study of which encouraged me at the beginning of 
the project -- is not. For a while, I worked with the case $k=2$ of
the expansion of $(\zeta'(s)/\zeta(s))^{(k)}$, which gives
\begin{equation}\label{eq:amtoro}
\Lambda\cdot \log^2 = \mu \ast \log^3 - 3\cdot
 (\Lambda \cdot \log) \ast \Lambda 
- \Lambda\ast \Lambda \ast \Lambda.\end{equation}
This identity
 is essentially log-free: while a trivial bound on the sum of the right
side for $n$ from $1$ to $N$ does seem to have two extra factors of $\log$,
they are present only in the term $\mu\ast \log^3$, which is not the hardest
one to estimate.
Ramar\'e obtained a log-free bound in \cite{MR2607306}
using an identity introduced by Diamond and Steinig in the course of their
own work on elementary proofs of the prime number theorem 
\cite{MR0280449};
that identity gives a decomposition for $\Lambda\cdot \log^k$ that can also 
be derived from the expansion of $(\zeta'(s)/\zeta(s))^{(k)}$, by 
a clever grouping of terms.

In the end, I decided to use Vaughan's identity, motivated in part by
\cite{Tao}, and in part by the lack of free parameters in (\ref{eq:amtoro});
as can be seen in (\ref{eq:vaughan}), Vaughan's identity has two 
parameters $U$, $V$ that we can set to whatever values we think best.
The form of the identity allowed me to reuse much of my work
up to that point, but it also 
posed a challenge: since Vaughan's identity is by no means log-free, one
has obtain cancellation in Vaughan's identity 
at every possible step, beyond the cancellation given by the phase 
$e(\alpha n)$. (The presence of a phase, in fact, makes the task of
getting cancellation from the identity more complicated.)
The removal of logarithms will be one of
our main tasks in what follows. It is clear that the presence of the 
M\"{o}bius function $\mu$ should give, in principle, some cancellation; we
will show how to use it to obtain as much cancellation as we need -- 
with good constants, and not just asymptotically.

\subsection{Type I sums}
There are two type I sums, namely,
\begin{equation}\label{eq:rada1}
\sum_{m\leq U} \mu(m) \sum_n (\log n) e(\alpha m n) \eta\left(\frac{m n}{x}\right)
\end{equation}
and
\begin{equation}\label{eq:rada2}
\sum_{v\leq V} \Lambda(v) \sum_{u\leq U} \mu(u) \sum_n
e(\alpha v u n) \eta\left(\frac{v u n}{x}\right).\end{equation}
In either case, $\alpha = a/q + \delta/x$, where $q$ is larger than a constant
$r$ and $|\delta/x|\leq 1/q Q_0$ for some $Q_0>\max(q,\sqrt{x})$.
For the purposes of this exposition, we will set it as our task to 
estimate the slightly simpler
sum
\begin{equation}\label{eq:sadun}
\sum_{m\leq D} \mu(m) \sum_n e(\alpha m n) \eta\left(\frac{m n}{x}\right),\end{equation}
where $D$ can be $U$ or $UV$ or something else less than $x$.

Why can we consider this simpler sum without omitting anything essential?
It is clear that (\ref{eq:rada1}) is of the same kind as (\ref{eq:sadun}).
The inner double sum in (\ref{eq:rada2}) is just (\ref{eq:sadun})
with $\alpha v$ instead of $\alpha$; this enables us to estimate
(\ref{eq:rada2}) by means of (\ref{eq:sadun}) for $q$ small, i.e., the
more delicate case. If $q$ is not small, then the approximation
$\alpha v \sim a v/q$ may not be accurate enough. In that case, we
collapse the two outer sums in (\ref{eq:rada2}) into 
a sum $\sum_n (\Lambda_{\leq V}\ast \mu_{\leq U})(n)$, and treat all of
(\ref{eq:rada2}) much as we will treat (\ref{eq:sadun}); since $q$ is not
small, we can afford to bound $(\Lambda_{\leq V}\ast \mu_{\leq U})(n)$ trivially
(by $\log n$) in the less sensitive terms.

Let us first outline Vinogradov's procedure for bounding type I sums.
Just by summing a geometric series, we get
 \begin{equation}\label{eq:salome}
\left|\sum_{n\leq N} e(\alpha n) \right| \leq
\min\left(N,\frac{c}{\{\alpha\}}\right),\end{equation}
where $c$ is a constant and
 $\{\alpha\}$ is the distance from $\alpha$ to the nearest integer.
Vinogradov splits the outer sum in (\ref{eq:sadun}) into sums of length
$q$. When $m$ runs on an interval of length $q$, the angle $a m/q$ runs
through all fractions of the form $b/q$; due to the error $\delta/x$,
$\alpha m$ could be close to $0$ for two values of $n$, but otherwise
$\{\alpha m\}$ takes values bounded below by $1/q$ (twice), $2/q$ (twice),
$3/q$ (twice), etc. Thus
\begin{equation}\label{eq:kolt}\left|\sum_{y<m\leq y+q} \mu(m) \sum_{n\leq N} e(\alpha m n)\right|
\leq \sum_{y<m\leq y+q} \left|\sum_{n\leq N} e(\alpha m n)\right| \leq
\frac{2 N}{m} + 2 c q \log e q\end{equation}
for any $y\geq 0$.

There are several ways to improve this. One is simply to estimate
the inner sum more precisely; this was already done in \cite{MR1803131}.
One can also define a smoothing function $\eta$, as in (\ref{eq:sadun});
it is easy to get
\[\left|\sum_{n\leq N} e(\alpha n) \eta\left(\frac{n}{x}\right)\right|\leq
\min\left(x |\eta|_1 + \frac{|\eta'|_1}{2}, \frac{|\eta'|_1}{2
|\sin(\pi \alpha)|}, \frac{|\widehat{\eta''}|_\infty}{4 
x (\sin \pi \alpha)^2}\right).\]
Except for the third term, this is as in \cite{Tao}. We could also choose
carefully which bound to use for each $m$; surprisingly, this gives
an improvement -- in fact, an important one, for $m$ large. 
However, even with these improvements, we still
have a term proportional to $N/m$ as in (\ref{eq:kolt}), and this contributes
about $(x \log x)/q$ to the sum (\ref{eq:sadun}), thus giving us an estimate
that is not log-free.

What we have to do, naturally, is to take out the terms with $q|m$ for
$m$ small. (If $m$ is large, then those may not be the terms for which
$m\alpha$ is close to $0$; we will later see what to do.) For $y+q\leq Q/2$,
$|\alpha - a/q|\leq 1/q Q$, 
we get that
\begin{equation}\label{eq:cami}\mathop{\sum_{y<m\leq y+q}}_{q\nmid m} 
 \min\left(A, \frac{B}{|\sin \pi \alpha n|},
\frac{C}{|\sin \pi \alpha n|^2}\right)\end{equation}
is at most
\begin{equation}\label{eq:crims}
\min\left(\frac{20}{3 \pi^2} C q^2, 2 A + \frac{4 q}{\pi} \sqrt{AC},
\frac{2 B q}{\pi} \max\left(2, \log \frac{C e^3 q}{B \pi}\right)\right).
\end{equation}
This is satisfactory. We are left with all the terms $m\leq M = \min(D,Q/2)$
with $q|m$ -- and also with all the terms $Q/2 < m\leq D$.
For $m\leq M$ divisible by $q$, we can estimate (as opposed
to just bound from above) 
the inner sum in (\ref{eq:sadun}) by the Poisson summation
formula, and then sum over $m$, but without taking absolute values; writing
$m=aq$, we get a main term 
\begin{equation}\label{eq:jut}
\frac{x \mu(q)}{q} \cdot \widehat{\eta}(-\delta) \cdot
\mathop{\sum_{a\leq M/q}}_{(a,q)=1} \frac{\mu(a)}{a},\end{equation}
where $(a,q)$ stands for the greatest common divisor of $a$ and $q$.

It is clear that we have to get cancellation over $\mu$ here.
There is an elegant elementary argument \cite{MR1401709} showing that
the absolute value of the sum in (\ref{eq:jut}) is at most $1$.
We need to gain one more log, however. Ramar\'e \cite{Ramsev} helpfully
furnished the following bound:
\begin{equation}\label{eq:hujt}
\left|\mathop{\sum_{a\leq x}}_{(a,q)=1} \frac{\mu(a)}{a}\right|\leq
\frac{4}{5} \frac{q}{\phi(q)} \frac{1}{\log x/q}\end{equation}
for $q\leq x$. (Cf. \cite{MR1378588}, \cite{ElMarraki})
This is neither trivial nor elementary.\footnote{The current state of knowledge may seem surprising: after all, we expect
nearly square-root cancellation -- for instance, $|\sum_{n\leq x} \mu(n)/n|\leq \sqrt{2/x}$
holds for all real $0<x\leq 10^{12}$; see also the stronger bound 
\cite{MR1259423}).
The classical zero-free region of the
Riemann zeta function ought to give a factor of $\exp(-\sqrt{(\log x)/c})$,
which looks much better than $1/\log x$. What happens is that (a) such a factor
is not actually much better than $1/\log x$ for $x\sim 10^{30}$, say; (b)
estimating sums involving the M\"obius function by means of an explicit formula
is harder than estimating sums involving $\Lambda(n)$:
the residues of $1/\zeta(s)$ at the non-trivial zeros of $s$ come into play.
As a result, getting non-trivial
 explicit results on sums of $\mu(n)$ is harder than
one would naively expect from the quality of classical effective (but non-explicit)
results. See 
\cite{RamEtatLieux} for a survey of explicit bounds.} We are, so to speak, allowed
to use non-elementary means (that is, methods based on $L$-functions)
because the only $L$-function we need to use here is the Riemann zeta function.

What shall we do for $m>Q/2$? We can always give a bound
\begin{equation}\label{eq:tuffblo}
\sum_{y<m\leq y+q}
 \min\left(A, 
\frac{C}{|\sin \pi \alpha n|^2}\right) \leq 3 A + \frac{4 q}{\pi} \sqrt{A C} 
\end{equation}
for $y$ arbitrary; since $A C$ will be of constant size, $(4 q/\pi) \sqrt{A C}$
is pleasant enough, but the contribution of $3 A \sim 3 |\eta|_1 x/y$
is nasty (it adds a multiple of $(x \log x)/q$ to the total) and seems
 unavoidable: the values of $m$ for which
$\alpha m$ is close to $0$ no longer correspond to the congruence class
$m\equiv 0 \mo q$, and thus cannot be taken out. 

The solution is to switch approximations. (The idea of using different
approximations to the same $\alpha$ is neither new nor recent in the general
context of the circle method: see \cite[\S 2.8, Ex. 2]{MR1435742}. What
may be new is its use to clear a hurdle in type I sums.) What does this mean?
If $\alpha$ 
were exactly, or almost exactly, $a/q$, then there would be no other very
good approximations in a reasonable range. However, note that we can
{\em define} $Q = \lfloor x/ |\delta q|\rfloor$ for 
$\alpha = a/q + \delta/x$, and still have
$|\alpha - a/q|\leq 1/q Q$. If $\delta$ is very small, $Q$ will be
larger than $2 D$, and there will be no terms with $Q/2 < m \leq D$ to
worry about.

What happens if $\delta$ is not very small? We know that, for any $Q'$, there is
an approximation $a'/q'$ to $\alpha$ with $|\alpha - a'/q'|\leq 1/q' Q'$
and $q'\leq Q'$. However, for $Q' > Q$, we know that $a'/q'$ cannot equal
$a/q$: by the definition of $Q$, the approximation $a/q$ is not good
enough, i.e., $|\alpha - a/q|\leq 1/q Q'$ does not hold.
Since $a/q\ne a'/q'$, we see that $|a/q - a'/q'|\geq 1/q q'$, and this implies that $q'\geq
(\epsilon/(1+\epsilon)) Q$. 

Thus, for $m>Q/2$, the solution is to apply (\ref{eq:tuffblo}) with
$a'/q'$ instead of $a/q$. The contribution of $A$ fades into
insignificance: for the first sum over a range $y<m\leq y+q'$, $y\geq Q/2$,
it contributes at most $x/(Q/2)$, and all the other contributions of $A$
sum up to at most a constant times $(x \log x)/q'$.

Proceeding in this way, we obtain a total bound for
(\ref{eq:sadun}) whose main terms are proportional to
\begin{equation}\label{eq:marxophone}\frac{1}{\phi(q)} \frac{x}{\log \frac{x}{q} }
\min\left(1, \frac{1}{\delta^2}\right),
\;\;\; \frac{2}{\pi} \sqrt{|\widehat{\eta''}|_\infty} \cdot D\;\;\; 
\text{and} \;\;\;\;
q \log \max\left(\frac{D}{q},q\right),
\end{equation}
with good, explicit constants. The first term -- usually the largest one --
is precisely what we needed: it is proportional to $(1/\phi(q)) x/\log x$
for $q$ small, and decreases rapidly as $|\delta|$ increases.

\subsection{Type II, or bilinear, sums}
We must now bound
\[S = \sum_m (1\ast \mu_{>U})(m) \sum_{n>V} \Lambda(n) e(\alpha m n) \eta(m n/x).
\]
At this point it is convenient to assume that $\eta$ is the Mellin
convolution of two functions. The
{\em multiplicative} or {\em Mellin convolution} on
$\mathbb{R}^+$ is defined by
\[(\eta_0\ast_M \eta_1)(t) = \int_0^\infty \eta_0(r)
\eta_1\left(\frac{t}{r}\right) \frac{dr}{r}.\]
Tao \cite{Tao} takes $\eta = \eta_2 = \eta_1 \ast_M \eta_1$, where 
$\eta_1$ is a brutal truncation, viz.,
the function taking the value $2$ on $\lbrack 1/2,1\rbrack$
and $0$ elsewhere. We take the same $\eta_2$, in part for comparison
purposes, and in part because this will allow us to use off-the-shelf 
estimates on the large sieve. (Brutal truncations are rarely optimal in
principle,
but, as they are very common, results for them have been carefully optimized
in the literature.) Clearly
\begin{equation}\label{eq:salom}
S = \int_V^{X/U} \sum_m \left(\mathop{\sum_{d>U}}_{d|m}
\mu(d)\right) \eta_1\left(\frac{m}{x/W}\right) \cdot \sum_{n\geq V}
\Lambda(n) e(\alpha m n) \eta_1\left(\frac{n}{W}\right) \frac{dW}{W}. \end{equation}
By Cauchy-Schwarz, the integrand is at most $\sqrt{S_1(U,W) S_2(V,W)}$, where
\begin{equation}\label{eq:katar}\begin{aligned}
S_1(U,W) &= \sum_{\frac{x}{2 W} < m\leq \frac{x}{W}} \left|\mathop{\sum_{d>U}}_{d|m} \mu(d)\right|^2,
\\ S_2(V,W) &= \sum_{\frac{x}{2 W} \leq m\leq \frac{x}{W}} \left|
\sum_{\max\left(V,\frac{W}{2}\right) \leq n\leq W}
\Lambda(n) e(\alpha m n)\right|^2.\end{aligned}\end{equation}

We must bound $S_1(U,W)$ by a constant times $x/W$. We are able to do this
-- with a good constant. (A careless bound would have given a multiple of
$(x/U) \log^3 (x/U)$, which is much too large.) First,
we reduce $S_1(W)$ to an expression involving an integral of
\begin{equation}\label{eq:rotobo}
\mathop{\sum_{r_1\leq x} \sum_{r_2\leq x}}_{(r_1,r_2)=1} \frac{\mu(r_1) \mu(r_2)}{
\sigma(r_1) \sigma(r_2)}.
\end{equation}
We can bound (\ref{eq:rotobo}) by the use of bounds on $\sum_{n\leq t}
\mu(n)/n$, combined with the estimation of infinite products by means
of approximations to $\zeta(s)$ for $s\to 1^+$. 
 After some additional manipulations, we
obtain a bound for $S_1(U,W)$ whose main term is at most $(3/\pi^2) (x/W)$
for each $W$, and closer to $0.22482 x/W$ on average over $W$.

(This is as good a point as any to say that, throughout, we can use
a trick in \cite{Tao} that allows us to work with odd values of integer
variables throughout, instead of letting $m$ or $n$ range over all
integers. Here, for instance, if $m$ and $n$ are restricted to
be odd, we obtain a bound of $(2/\pi^2) (x/W)$ for individual $W$,
and $0.15107 x/W$ on average over $W$. This is so even though we are losing
some cancellation in $\mu$ by the restriction.)

Let us now bound $S_2(V,W)$. This is traditionally done by Linnik's
dispersion method. However, it should be clear that the thing to do nowadays
is to use a large sieve, and, more specifically, a large sieve for primes;
that kind of large sieve is nothing other than a tool for estimating expressions
such as $S_2(V,W)$.
(Incidentally, even though we are trying to save every factor of $\log$ we
can, we choose not to use small sieves at all, either here or elsewhere.) 
In order to take advantage of prime support, we use Montgomery's inequality
(\cite{MR0224585}, \cite{MR0311618}; see the expositions in
\cite[pp. 27--29]{MR0337847} and  \cite[\S 7.4]{MR2061214}) combined with
Montgomery and Vaughan's large sieve with weights \cite[(1.6)]{MR0374060},
following the general procedure in \cite[(1.6)]{MR0374060}.
We obtain a bound of the form
\begin{equation}\label{eq:ursu}
\frac{\log W}{\log \frac{W}{2 q}} \left(\frac{x}{4 \phi(q)} + \frac{q W}{\phi(q)}
\right) \frac{W}{2} 
\end{equation}
on $S_2(V,W)$, where, of course, we can also choose {\em not} to gain a factor
of $\log W/2q$ if $q$ is close to or greater than $W$.

It remains to see how to gain a factor of $|\delta|$ in the major arcs,
and more specifically in $S_2(V,W)$. To explain this, let us step back and
take a look at what the large sieve is. Given a civilized 
function $f:\mathbb{Z} \to \mathbb{C}$, Plancherel's identity tells us that
\[\int_{\mathbb{R}/\mathbb{Z}} \left|\widehat{f}\left(\alpha\right)\right|^2 
d\alpha = \sum_n |f(n)|^2.\]
The large sieve can be seen as an approximate, or statistical, version of this:
for a ``sample'' of points $\alpha_1,\alpha_2,\dotsc,\alpha_k$ satisfying
$|\alpha_i-\alpha_j|\geq \beta$ for $i\ne j$, it tells us that
\begin{equation}\label{eq:rut}
\sum_{1\leq j\leq k} \left|\widehat{f}\left(\alpha_i\right)\right|^2 
\leq (X + \beta^{-1}) \sum_n |f(n)|^2,\end{equation}
 assuming that $f$ is supported on
an interval of length $X$.

Now consider $\alpha_1 = \alpha, \alpha_2 = 2\alpha, \alpha_3 = 3\alpha\dotsc
$. If $\alpha=a/q$, then the angles $\alpha_1,\dotsc,\alpha_q$ are
well-separated, i.e., they satisfy $|\alpha_i-\alpha_j|\geq 1/q$,
and so we can apply (\ref{eq:rut}) with $\beta = 1/q$. However,
$\alpha_{q+1} =\alpha_1$. Thus, if we have an outer sum of length
$L>q$ -- in (\ref{eq:katar}), we have an outer sum of length $L=x/2W$ --
we need to split it into $\lceil L/q\rceil$ blocks of length $q$,
and so the total bound given by (\ref{eq:rut}) is 
$\lceil L/q\rceil (X+q) \sum_n |f(n)|^2$. Indeed, this is what gives us
(\ref{eq:ursu}), which is fine, but we want to do better for $|\delta|$
larger than a constant.

Suppose, then, that $\alpha = a/q + \delta/x$, where $|\delta|>8$, say.
Then the angles $\alpha_1$ and $\alpha_{q+1}$ are not identical:
$|\alpha_1 - \alpha_{q+1}|\leq q |\delta|/x$. We also see that $\alpha_{q+1}$
is at a distance at least $q |\delta|/x$ from
 $\alpha_2, \alpha_3,\dotsc \alpha_q$, provided that $q |\delta|/x < 1/q$.
We can go on with $\alpha_{q+2}, \alpha_{q+3},\dotsc$, and stop only once 
there is overlap, i.e., only once we
reach $\alpha_m$ such that $m |\delta|/x \geq 1/q$. We then give all
the angles $\alpha_1,\dotsc,\alpha_m$ -- which are separated by at least
$q |\delta|/x$ from each other 
-- to the large sieve at the same time. We do this 
$\lceil L/m\rceil \leq \lceil L/(x/|\delta| q) \rceil$ times, and obtain
a total bound of $\lceil L/(x/|\delta| q) \rceil (X + x/|\delta| q)
\sum_n |f(n)|^2$, which, for $L=x/2W$, $X = W/2$, gives us about
\[\left(\frac{x}{4 Q} \frac{W}{2}
+ \frac{x}{4}\right) \log W\]
provided that $L\geq x/|\delta| q$ and, as usual, $|\alpha-a/q|\leq 1/q Q$.
This is very small compared to the trivial bound $\lesssim 
x W/8$.

What happens if $L<x/|\delta q|$? Then there is never any overlap: we
consider all angles $\alpha_i$, and give them all together to the large sieve.
The total bound is $(W^2/4 + x W/2 |\delta| q) \log W$. If $L=x/2W$
is smaller than, say, $x/3 |\delta q|$, then we see clearly that there
are non-intersecting swarms of angles $\alpha_i$ around the rationals $a/q$.
We can thus save a factor of $\log$ (or rather $(\phi(q)/q) \log(W/|\delta q|)$)
by applying Montgomery's inequality, which operates by strewing
displacements of given angles (or, here, swarms around angles) 
around the circle to the extent possible while keeping everything 
well-separated. In this way, we obtain a bound of the form
\[\frac{\log W}{\log \frac{W}{|\delta| q}}
\left(\frac{x}{|\delta| \phi(q)} + \frac{q}{\phi(q)} \frac{W}{2}\right) \frac{W}{2}
.\]
Compare this to (\ref{eq:ursu}); we have gained a factor of $|\delta|/4$,
and so we use this estimate when $|\delta|>4$.
(We will actually use the criterion $|\delta|>8$, but, since we will
be working with approximations of the form
$2\alpha = a/q + \delta/x$, the value of $\delta$ in our actual work is
twice of what it is in this introduction.
This is a consequence of working with sums over the odd integers, 
as in \cite{Tao}.)

\begin{center}
* * *
\end{center}

We have succeeded in eliminating all factors of $\log$ we came across.
The only factor of $\log$ that remains is $\log x/UV$, coming from
the integral $\int_V^{x/U} dW/W$. Thus, we want $U V$ to be close to $x$,
but we cannot let it be too close,
since we also have a term proportional to
 $D = UV$ in (\ref{eq:marxophone}), and we need to keep it substantially
smaller than $x$. We set $U$ and $V$
so that $UV$ is $x/\sqrt{q \max(4,|\delta|)}$ or thereabouts.

In the end, after some work, we obtain our main minor-arcs bound
(Theorem \ref{thm:minmain}). It states the following.
Let $x\geq x_0$, $x_0 = 2.16\cdot 10^{20}$.
Tecall that $S_{\eta}(\alpha,x) = \sum_n \Lambda(n) e(\alpha n) \eta(n/x)$
and $\eta_2 = \eta_1 \ast_M \eta_1 = 4\cdot 1_{\lbrack 1/2,1\rbrack} \ast
1_{\lbrack 1/2,1\rbrack}$.
Let $2 \alpha = a/q + \delta/x$, $q\leq Q$,
$\gcd(a,q)=1$, $|\delta/x|\leq 1/q Q$, where $Q = (3/4) x^{2/3}$.
If $q\leq x^{1/3}/6$, then
\begin{equation}\label{eq:umnov}
\begin{aligned}
&|S_{\eta}(\alpha,x)| \leq 
 \frac{R_{x,\delta_0 q} \log \delta_0 q + 0.5}{\sqrt{\delta_0 \phi(q)}} \cdot x
+ \frac{2.5 x}{\sqrt{\delta_0 q}} + 
 \frac{2x}{\delta_0 q} \cdot L_{x,\delta_0 q, q}
+ 3.36 x^{5/6},\end{aligned}\end{equation} where
\begin{equation}\begin{aligned}
\delta_0 &= \max(2,|\delta|/4),\;\;\;\;\;
R_{x,t} = 0.27125 \log 
\left(1 + \frac{\log 4 t}{2 \log \frac{9 x^{1/3}}{2.004 t}}\right)
 + 0.41415, \\
L_{x,t,q} &=
\frac{q}{\phi(q)}
\left(\frac{13}{4} \log t + 7.82\right)
+ 13.66 \log t + 37.55.\end{aligned}\end{equation}

The factor $R_{x,t}$ is small in practice; for typical ``difficult'' values
of $x$ and $\delta_0 x$, it is less than $1$. The crucial things to notice
in (\ref{eq:umnov}) are that there is no factor of $\log x$, and that,
in the main term, there is only one factor of $\log \delta_0 q$.
The fact that $\delta_0$ helps us as it grows is precisely what enables
us to take major arcs that get narrower and narrower as $q$ grows.

\section{Integrals over the major and minor arcs}\label{sec:putall}

So far, we have sketched (\S \ref{sec:rubosa}) how to estimate
$S_\eta(\alpha,x)$ for $\alpha$ in the major arcs and $\eta$ based on the
Gaussian $e^{-t^2/2}$, and also (\S \ref{sec:melusa}) how to bound $|S_\eta(\alpha,x)|$ for
$\alpha$ in the minor arcs and $\eta = \eta_2$, where $\eta_2 =
4 \cdot 1_{\lbrack 1/2,1\rbrack}\ast_M 1_{\lbrack 1/2,1\rbrack}$.
 We now must show how to use such information to 
estimate integrals such as the ones in (\ref{eq:wtree}).

We will use two smoothing functions $\eta_+$, $\eta_*$; 
in the notation of
(\ref{eq:willow}), we set $f_1 = f_2 = \Lambda(n) \eta_+(n/x)$,
 $f_3 = \Lambda(n) \eta_*(n/x)$, and so
we must give a lower bound for
\begin{equation}\label{eq:centi}
\int_{\mathfrak{M}} (S_{\eta_+}(\alpha,x))^2 S_{\eta_*}(\alpha,x) e(- \alpha n)
d\alpha\end{equation}
and an upper bound for
\begin{equation}\label{eq:wizgor}
\int_{\mathfrak{m}} \left|S_{\eta_+}(\alpha,x)\right|^2
S_{\eta_*}(\alpha,x) e(- \alpha n) d\alpha
\end{equation}
so that we can verify (\ref{eq:wtree}). 

The traditional approach to (\ref{eq:wizgor}) is to bound
\begin{equation}\label{eq:katju}\begin{aligned}
\int_{\mathfrak{m}} (S_{\eta_+}(\alpha,x))^2
S_{\eta_*}(\alpha,x) e(-\alpha n) d\alpha &\leq
\int_{\mathfrak{m}} \left|S_{\eta_+}(\alpha,x)\right|^2 d\alpha \cdot \max_{\alpha\in \mathfrak{m}}
\widehat{\eta_*}(\alpha) \\ &\leq \sum_n
\Lambda(n)^2 \eta_+^2\left(\frac{n}{x}\right) \cdot 
 \max_{\alpha\in \mathfrak{m}} S_{\eta_*}(\alpha,x).\end{aligned}\end{equation}
Since the sum over $n$ is of the order of $x \log x$, this is not
log-free, and so cannot be good enough; we will later see how to do better.
Still, this gets the main shape right: our bound on (\ref{eq:wizgor})
will be proportional to $|\eta_+|_2^2 |\eta_*|_1$. Moreover,
we see that $\eta_*$ has to be such that we know how to bound
$|S_{\eta_*}(\alpha,x)|$ for $\alpha\in \mathfrak{m}$, while our choice of
$\eta_+$ is more or less free, at least as far as the minor arcs are concerned.

What about the major arcs? In order to do anything on them, we will have
to be able to estimate both $\eta_+(\alpha)$ and $\eta_*(\alpha)$ for
$\alpha \in \mathfrak{M}$. If that is the case, then, as we shall see,
we will be able to obtain that
the main term of (\ref{eq:centi}) is
an infinite product (independent of the smoothing functions), times $x^2$,
times
\begin{equation}\label{eq:sust}\begin{aligned}
\int_{-\infty}^\infty &(\widehat{\eta_+}(- \alpha))^2 
\widehat{\eta_*}(- \alpha) e(- \alpha n/x) d\alpha \\ &=
\int_0^\infty \int_0^\infty \eta_+(t_1) \eta_+(t_2) 
\eta_*\left(\frac{n}{x} - (t_1+t_2)\right) dt_1 dt_2.
\end{aligned}\end{equation}
In other words, we want to maximize (or nearly maximize)
 the expression on the right of (\ref{eq:sust})
divided by $|\eta_+|_2^2 |\eta_*|_1$.

One way to do this is to let $\eta_*$ be concentrated on a small interval
$\lbrack 0,\epsilon)$. Then the right side of (\ref{eq:sust}) is
approximately
\begin{equation}\label{eq:korl}
\left|\eta_*\right|_1\cdot \int_0^\infty \eta_+(t) \eta_+\left(
\frac{n}{x} - t\right) dt.\end{equation}
To maximize (\ref{eq:korl}), we should make sure that $\eta_+(t) \sim
\eta_+(n/x-t)$. We set $x\sim n/2$, and see that we should define
$\eta_+$ so that it is supported on $\lbrack 0,2\rbrack$
and symmetric around $t=1$, or nearly so; this will maximize the ratio of
(\ref{eq:korl}) to $|\eta_+|_2^2 |\eta_*|_1$.

We should do this while making sure that we will know how to estimate
$S_{\eta_+}(\alpha,x)$ for $\alpha\in \mathfrak{M}$.
 We know how to estimate $S_\eta(\alpha,x)$ very precisely 
for functions of the form $\eta(t) = g(t) e^{-t^2/2}$, 
$\eta(t) = g(t) t e^{-t^2/2}$, etc., where $g(t)$ is band-limited. We will
work with a function $\eta_+$ of that form, chosen so as to be very
close (in $\ell_2$ norm) 
to a function $\eta_\circ$ that is in fact supported on $\lbrack 0,2\rbrack$
and symmetric around $t=1$. 

We choose
\[\eta_\circ(t) = \begin{cases}
t^3 (2-t)^3 e^{-(t-1)^2/2} & \text{if $t\in \lbrack 0,2\rbrack$,}\\
0 &\text{if $t\not\in \lbrack 0,2\rbrack$.}\end{cases}\]
This function is obviously symmetric ($\eta_\circ(t) = \eta_\circ(2-t)$)
and vanishes to high order at $t=0$, besides being supported on $\lbrack 0,2
\rbrack$.

We set $\eta_+(t) = h_R(t) t e^{-t^2/2}$, where $h_R(t)$ is an approximation to
the function \[h(t) = \begin{cases} t^2 (2-t)^3 e^{t-\frac{1}{2}}
&\text{if $t\in \lbrack 0,2\rbrack$}\\
0 &\text{if $t\not\in \lbrack 0,2\rbrack$.}\end{cases}\]
We just let $h_R(t)$ be the inverse Mellin transform of the truncation
of $Mh$ to an interval $\lbrack -iR, iR\rbrack$. (Explicitly,
\[h_R(t) = \int_0^\infty h(t y^{-1}) F_R(y) \frac{dy}{y},\]
where $F_R(t) = \sin(R \log y)/(\pi \log y)$, that is, $F_R$ is
the Dirichlet kernel with
a change of variables.) 

Since the Mellin transform of
$t e^{-t^2/2}$ is regular at $s=0$, the Mellin transform $M\eta_+$
will be holomorphic in a neighborhood of $\{s: 0\leq \Re(s)\leq 1\}$, even
though the truncation of $Mh$ to $\lbrack - i R, i R\rbrack$ is brutal.
Set $R=200$, say. By the fast decay of 
$Mh(it)$ and the fact that the Mellin transform $M$ is an isometry,
$|(h_R(t)-h(t))/t|_2$ is very small, and hence so is $|\eta_+-\eta_\circ|_2$,
as we desired.

But what about the requirement that we be able to
estimate $S_{\eta_*}(\alpha,x)$ for both $\alpha\in \mathfrak{m}$ and
$\alpha\in \mathfrak{M}$?

Generally speaking, if we know how to estimate
$S_{\eta_1}(\alpha,x)$ for some $\alpha\in \mathbb{R}/\mathbb{Z}$ and we
also know how to estimate $S_{\eta_2}(\alpha,x)$ for all other
$\alpha \in \mathbb{R}/\mathbb{Z}$, where $\eta_1$ and $\eta_2$ are
two smoothing functions, then we know how to estimate 
$S_{\eta_3}(\alpha,x)$ for all $\alpha\in \mathbb{R}/\mathbb{Z}$, where
$\eta_3 = \eta_1 \ast_M \eta_2$, or, more generally,
$\eta_*(t) = (\eta_1 \ast_M \eta_2)(\kappa t)$, $\kappa>0$ a constant.
This is an easy exercise on exchanging the order of integration and 
summation:
\begin{equation}\label{eq:asco}\begin{aligned} 
S_{\eta_*}(\alpha,x) &=
\sum_n \Lambda(n) e(\alpha n) (\eta_1\ast_M \eta_2)\left(\kappa \frac{n}{x}\right)
\\
&= \int_0^\infty \sum_n \Lambda(n) e(\alpha n) \eta_1(\kappa r) \eta_2\left(\frac{n}{r
x}\right) \frac{dr}{r} = \int_0^\infty \eta_1(\kappa r) S_{\eta_2}(r x) \frac{dr}{r},
\end{aligned}\end{equation}
and similarly with $\eta_1$ and $\eta_2$ switched. 
Of course, this trick is valid for all exponential sums:
 any function $f(n)$ would do in place of $\Lambda(n)$.
The only caveat is that
$\eta_1$ (and $\eta_2$) should be small very near $0$, since, for $r$ small,
we may not be able to estimate $S_{\eta_2}(r x)$
(or $S_{\eta_1}(r x)$) with any precision. This is not a problem; one
of our functions will be $t^2 e^{-t^2/2}$, which vanishes to second
order at $0$, and the other one will
be $\eta_2 =
4 \cdot 1_{\lbrack 1/2,1\rbrack}\ast_M 1_{\lbrack 1/2,1\rbrack}$, which has 
support bounded away from $0$. We will set $\kappa$ large (say $\kappa=49$)
so that the support of $\eta_*$ is indeed concentrated on a small interval
$\lbrack 0, \epsilon)$, as we wanted.

\begin{center}
* * *
\end{center}

Now that we have chosen our smoothing weights $\eta_+$ and $\eta_*$,
we have to estimate the major-arc integral (\ref{eq:centi}) and
the minor-arc integral (\ref{eq:wizgor}). What follows can actually be
done for general $\eta_+$ and $\eta_*$; we could have left our particular
 choice of $\eta_+$ and $\eta_*$ for the end.

Estimating the major-arc integral (\ref{eq:centi}) may
sound like an easy task, since we have rather precise estimates for
$S_{\eta}(\alpha,x)$ ($\eta=\eta_+,\eta_*$) when $\alpha$ is on the major
arcs; we could just replace $S_\eta(\alpha,x)$ in (\ref{eq:centi}) by
the approximation given by (\ref{eq:sorrow}) and (\ref{eq:frank}).
It is, however, more efficient to express (\ref{eq:centi}) as the sum of 
the contribution of the trivial character (a sum of integrals of
$(\widehat{\eta}(-\delta) x)^3$, where $\widehat{\eta}(-\delta) x$ comes
from (\ref{eq:frank})), plus a term of the form
\[(\text{maximum of $\sqrt{q}\cdot E(q)$ for $q\leq r$})\cdot
\int_{\mathfrak{M}} \left|S_{\eta_+}(\alpha,x)\right|^2 d\alpha
,\]
where $E(q)=E$ is as in (\ref{eq:lynno}), plus two other terms of 
essentially the same form.
As usual, the major arcs $\mathfrak{M}$ are the arcs around rationals
$a/q$ with $q\leq r$.
We will soon discuss how to bound the integral of 
$\left|S_{\eta_+}(\alpha,x)\right|^2$ over arcs around rationals $a/q$ with
 $q\leq s$,
$s$ arbitrary. Here, however, it is best to estimate the integral over 
$\mathfrak{M}$ using the estimate on
$S_{\eta_+}(\alpha,x)$ from (\ref{eq:sorrow}) and (\ref{eq:frank}); we obtain
a great deal of cancellation, with the effect that,
for $\chi$ non-trivial, the error term in 
(\ref{eq:lynno}) appears only when it gets squared,
and thus becomes negligible.

The contribution of the trivial character has an easy approximation, 
thanks to the fast decay of $\widehat{\eta_\circ}$.
We obtain that the major-arc integral (\ref{eq:centi}) equals a main
term $C_0 C_{\eta_\circ,\eta_*} x^2$, where
\[\begin{aligned}
C_0 &= \prod_{p|n} \left(1 - \frac{1}{(p-1)^2}\right)
\cdot \prod_{p\nmid n} \left(1 + \frac{1}{(p-1)^3}\right),\\
C_{\eta_\circ,\eta_*} &=
 \int_0^\infty \int_0^\infty \eta_\circ(t_1) \eta_\circ(t_2) 
\eta_*\left(\frac{n}{x}-(t_1+t_2)\right) dt_1 dt_2,
\end{aligned}\]
plus several small error terms. We have already chosen $\eta_\circ$,
$\eta_*$ and $x$ so as
to (nearly) maximize $C_{\eta_\circ,\eta_*}$.

It is time to bound the minor-arc integral (\ref{eq:wizgor}).
As we said in \S \ref{sec:putall}, we must do better than the usual bound
(\ref{eq:katju}). Since our minor-arc bound (\ref{eq:kraw})
on $|S_\eta(\alpha,x)|$, $\alpha\sim a/q$, decreases as $q$ increases,
it makes sense to use partial summation together with bounds on
\[\int_{\mathfrak{m}_s}  |S_{\eta_+}(\alpha,x)|^2
= \int_{\mathfrak{M}_s} |S_{\eta_+}(\alpha,x)|^2
d\alpha -
\int_{\mathfrak{M}} |S_{\eta_+}(\alpha,x)|^2
d\alpha,
\]
where
$\mathfrak{m}_s$ denotes the arcs around $a/q$, $r<q\leq s$, and
$\mathfrak{M}_s$ denotes the arcs around all $a/q$, $q\leq s$.  We already
know how to estimate the integral on $\mathfrak{M}$. How do we bound the
integral on $\mathfrak{M}_s$?

In order to do better than the trivial bound $\int_{\mathfrak{M}_s} \leq
\int_{\mathbb{R}/\mathbb{Z}}$,
we will need to use the fact that the series (\ref{eq:holo}) 
defining $S_{\eta_+}(\alpha,x)$ is essentially supported on prime numbers.
Bounding the integral on $\mathfrak{M}_s$ is closely related to the problem
of bounding
\begin{equation}\label{eq:kolro}
\sum_{q\leq s} \mathop{\sum_{a \mo q}}_{(a,q)=1} \left|
\sum_{n\leq x} a_n e(a/q)\right|^2
\end{equation}
efficiently for $s$ considerably smaller than $\sqrt{x}$ and
$a_n$ supported on the primes $\sqrt{x}<p\leq x$. This is a classical
problem in the study of the large sieve. The usual bound on (\ref{eq:kolro})
(by, for instance, Montgomery's inequality) has a gain of a factor of
\[2 e^{\gamma} (\log s)/(\log x/s^2)\] relative to the bound of $(x+s^2)
\sum_n |a_n|^2$ that one would get from the large sieve without using prime 
support. Heath-Brown proceeded similarly to bound
\begin{equation}\label{eq:kokoto}\int_{\mathfrak{M}_s} |S_{\eta_+}(\alpha,x)|^2
d\alpha 
\lesssim \frac{2 e^{\gamma} \log s}{\log x/s^2}
\int_{\mathbb{R}/\mathbb{Z}} |S_{\eta_+}(\alpha,x)|^2
d\alpha.\end{equation}

This already gives us the gain of
$C (\log s)/\log x$ that we absolutely need,
but the constant $C$ is suboptimal; the factor in the right side of
(\ref{eq:kokoto}) should really be
$(\log s)/\log x$, i.e., $C$ should be $1$. We cannot reasonably hope to
obtain a factor better than $2 (\log s)/\log x$ in the minor arcs due to what is known as the {\em parity problem} 
in sieve theory. As it turns out, Ramar\'e \cite{MR2493924} had given general
bounds on the large sieve that were clearly conducive to better bounds
on (\ref{eq:kolro}), though they involved a ratio that was not easy to
bound in general.

I used several careful estimations (including \cite[Lem.~3.4]{MR1375315})
to reduce the problem of bounding this ratio to a finite number of cases, which
 I then checked by a rigorous computation. This approach gave a bound on
(\ref{eq:kolro}) with a factor of size close to $2 (\log s)/\log x$.
(This solves the large-sieve problem for $s\leq x^{0.3}$; it would still
be worthwhile to give a computation-free proof for all $s\leq x^{1/2-\epsilon}$,
$\epsilon>0$.) It was then easy to give an analogous bound for the 
integral over $\mathfrak{M}_s$, namely,
\[\int_{\mathfrak{M}_s} |S_{\eta_+}(\alpha,x)|^2
d\alpha 
\lesssim \frac{2 \log s}{\log x}
\int_{\mathbb{R}/\mathbb{Z}} |S_{\eta_+}(\alpha,x)|^2
d\alpha,\]
where $\lesssim$ can easily be made precise by replacing $\log s$ by
$\log s + 1.36$ and $\log x$ by $\log x + c$, where $c$ is a small constant.
Without this improvement, the main theorem would still have been proved, but
the required computation time would have been multiplied by a factor of
considerably more than $e^{3\gamma} = 5.6499\dotsc$. 

What remained then was just to compare the estimates on 
(\ref{eq:centi}) and (\ref{eq:wizgor}) and check that (\ref{eq:wizgor})
is smaller for $n\geq 10^{27}$. This final step was just bookkeeping.
As we already discussed, a check for $n< 10^{27}$ is easy.
Thus ends the proof of the main theorem.
\section{Some remarks on computations}

There were two main computational tasks: verifying the ternary conjecture
for all $n\leq C$, and checking the Generalized Riemann Hypothesis
for modulus $q\leq r$ up to a certain height. 

The first task was not very
demanding. Platt and I verified in \cite{MR3171101}
 that every odd integer 
$5 < n\leq 8.8\cdot 10^{30}$
can be written as the sum of three primes. (In the end, only a check
for $5<n\leq 10^{27}$ was needed.) 
We proceeded as follows. In a major computational effort,
Oliveira e Silva, Herzog and Pardi \cite{OSHP}) had already checked
that the binary Goldbach conjecture is true up to $4\cdot 10^{18}$ --  
that is, every even number up to $4\cdot 10^{18}$ is the sum of two primes. 
Given that, all we had to do was to construct a ``prime ladder'', that is, a 
list of primes from $3$ up to $8.8\cdot 10^{30}$ such that the difference between any two consecutive primes in the list is at least $4$ and
at most $4\cdot 10^{18}$. 
(This is a known strategy: see \cite{MR1451327}.) Then,
for any odd integer $5 < n\leq 8.8\cdot 10^{30}$, there is a prime $p$ in 
the list such that $4\leq n-p \leq 4\cdot 10^{18}+2$. (Choose the largest $p<n$
in the ladder, or, if $n$ minus that prime is $2$, choose the prime immediately
under that.)
By \cite{OSHP} (and the fact that $4\cdot 10^{18}+2$ equals $p+q$, where
$p=2000000000000001301$ and $q=1999999999999998701$ are both prime), 
we can write $n-p = p_1 + p_2$ for some primes $p_1$, $p_2$, 
and so $n = p + p_1 + p_2$.

Building a prime ladder involves only integer arithmetic, that is, computer
manipulation of integers, rather than of real numbers.
Integers are something that computers can handle rapidly and reliably.
We look for primes for our ladder only among a special set of integers 
whose primality can be tested deterministically quite quickly (Proth numbers: 
$k\cdot 2^m+1$, $k< 2^m$). Thus, we can build a prime ladder by
 a rigorous, deterministic
algorithm that can be (and was) parallelized trivially.

The second computation is more demanding. It consists in 
verifying that, for every $L$-function $L(s,\chi)$ with $\chi$ 
of conductor $q\leq r = 300000$ (for $q$ even) or $q\leq r/2$ (for $q$ odd), 
all zeroes of $L(s,\chi)$ such that $|\Im(s)|\leq H_q = 10^8/q$ (for $q$ odd)
and $|\Im(s)|\leq H_q = \max(10^8/q,200+7.5\cdot 10^7/q$ (for $q$ even)
 lie on the critical line. 
As a matter of fact, Platt 
went up to conductor $q\leq 200000$ (or twice that for $q$ even) \cite{Plattfresh}; 
he had already gone up to conductor $100000$ in his PhD thesis \cite{Platt}. 
The 
verification took, in total, about $400000$ core-hours (i.e., the total number 
of processor cores used times the number of hours they ran equals $400000$; 
nowadays, a top-of-the-line processor typically has eight cores). In the end, 
since I used only $q\leq 150000$ (or twice that for $q$ even), 
the number of hours actually needed was closer to $160000$; since 
I could have made do with $q\leq 120000$ (at the cost of increasing $C$ to $10^{29}$
or $10^{30}$), it is likely, in retrospect, that only about $80000$ 
core-hours were needed.

Checking zeros of $L$-functions computationally goes back to Riemann 
(who did it by hand for the special case of the Riemann zeta function). 
It is also one of the things that were tried on digital computers in 
their early days (by Turing \cite{MR0055785}, for instance; see the 
exposition in \cite{MR2263990}). One of the main issues to be careful about 
arises whenever one manipulates real numbers via a computer: generally speaking,
a computer cannot store an irrational number; moreover, while a computer can handle rationals, it is really most comfortable handling just those rationals whose denominators are powers of two. Thus, one cannot really say: ``computer, give me the sine of that number'' and expect a precise result. What one should do, if one really wants to prove something (as is the case here!), is to say: 
``computer, I am giving you an interval $I=\lbrack a/2^k,b/2^k\rbrack$; 
give me an interval $I'=\lbrack c/2^\ell,d/2^\ell
\rbrack$, preferably very short, such that $\sin(I) \subset I'$''. This is called interval arithmetic; it is arguably the easiest way to do floating-point computations rigorously.

Processors do not do this natively, and if interval arithmetic is
implemented purely on software, computations can be slowed down by a factor of
about $100$. Fortunately, there are ways of running interval-arithmetic 
computations
partly on hardware, partly on software. 

Incidentally, there are some basic functions (such as $\sin$) that should
always be done on software, not just if one wants to use interval arithmetic, but
even if one just wants reasonably precise results: the implementation of
transcendental functions in some of the most popular processors 
does not always round correctly, and errors can accumulate quickly.
Fortunately, this problem is already well-known, and there is software
that takes care of this.
(Platt and I used the crlibm library \cite{crlibm}.) 

Lastly, there were several relatively minor computations strewn here and
there in the proof.
There is some numerical 
integration, done rigorously; once or twice, this was done using a standard
package based on interval arithmetic \cite{VNODELP}, but most of the time
I wrote my own routines in C (using Platt's interval arithmetic package) for
the sake of speed. Another kind of computation (employed much more in
\cite{HelfMaj} than in the somewhat
more polished version of the proof given here)
was a rigorous version of a ``proof by graph'' (``the maximum of a function $f$ is clearly less than $4$ because I can see it on the screen''). There is a 
standard way to do this (see, e.g.,  \cite[\S 5.2]{MR2807595}); 
essentially, the bisection method combines naturally with interval arithmetic,
as we shall describe in \S \ref{sec:koloko}.
Yet another computation (and not a very small one)
was that involved in verifying a large-sieve 
inequality in an intermediate range (as we discussed in \S \ref{sec:putall}).

 It may be interesting to note that one of the 
inequalities used to estimate (\ref{eq:rotobo})
was proven with the help of automatic quantifier elimination
\cite{QEPCAD}. Proving this inequality was a very 
minor task, both computationally and mathematically;
 in all likelihood, it is feasible to give a human-generated 
proof. Still, it is nice to know from first-hand experience
that computers can nowadays (pretend to) do something
other than just perform numerical computations -- and that this is already
applicable in current mathematical practice.

\chapter{Notation and preliminaries}
\section{General notation}
Given positive integers $m$, $n$, we say $m|n^\infty$ if every prime
dividing
$m$ also divides $n$. We say a positive integer $n$ is {\em square-full} 
if, for every prime $p$ dividing $n$, the square $p^2$ also divides $n$.
(In particular, $1$ is square-full.) We say $n$ is {\em square-free} if
$p^2\nmid n$ for every prime $p$. For $p$ prime, $n$ a non-zero integer,
we define $v_p(n)$ to be the largest non-negative integer $\alpha$ such
that $p^\alpha|n$.

When we write $\sum_n$, we mean $\sum_{n=1}^{\infty}$, unless the contrary
is stated. As always,
$\Lambda(n)$ denotes
the {\em von Mangoldt function}:
\[\Lambda(n) = \begin{cases}
\log p & \text{if $n=p^\alpha$ for some prime $p$ and some integer $\alpha\geq 1$,}\\
0 & \text{otherwise,}\end{cases}\]
and $\mu$ denotes the
{\em M\"{o}bius function}: 
\[\mu(n) = \begin{cases} (-1)^k & \text{if $n=p_1 p_2\dotsc p_k$, all $p_i$
distinct}\\
0 & \text{if $p^2|n$ for some prime $p$,}\end{cases}\]
We let $\tau(n)$ be the number of divisors of an
integer $n$, $\omega(n)$ the number of prime divisors of $n$, and
$\sigma(n)$ the sum of the divisors of $n$.

We write $(a,b)$ for the greatest common divisor of $a$ and $b$. If there
is any risk of confusion with the pair $(a,b)$, we write $\gcd(a,b)$.
Denote by $(a,b^\infty)$ the divisor $\prod_{p|b} p^{v_p(a)}$ of $a$.
(Thus, $a/(a,b^\infty)$ is coprime to $b$, and is in fact the maximal
divisor of $a$ with this property.)

As is customary, we write $e(x)$ for $e^{2\pi i x}$.
We denote the $L_r$ norm of a function $f$ by $|f|_r$.
We write $O^*(R)$ to mean a quantity at most $R$ in absolute value. 
Given a set $S$, we write $1_S$ for its characteristic function:
\[1_S(x) = \begin{cases} 1 &\text{if $x\in S$,}\\
0 &\text{otherwise.}\end{cases}\]

Write $\log^+ x$ for $\max(\log x,0)$.



\section{Dirichlet characters and $L$ functions}\label{subs:durian}
Let us go over some basic terms.
A {\em Dirichlet character} $\chi:\mathbb{Z}\to \mathbb{C}$ of modulus $q$
is a character $\chi$ of $(\mathbb{Z}/q \mathbb{Z})^*$ lifted to 
$\mathbb{Z}$ with the convention that $\chi(n)=0$ when
$(n,q)\ne 1$. (In other words: $\chi$ is completely multiplicative and periodic
modulo $q$, and vanishes on integers not coprime to $q$.)
Again by convention, there is a Dirichlet character of modulus
$q=1$, namely, the {\em trivial character} $\chi_T:\mathbb{Z}\to \mathbb{C}$
defined by $\chi_T(n)=1$ for every $n\in \mathbb{Z}$. 

If $\chi$ is a character modulo $q$ and $\chi'$ is a
character modulo $q'|q$ such that $\chi(n)=\chi'(n)$ for all $n$ coprime to $q$,
we say that $\chi'$ {\em induces} $\chi$. A character is 
{\em primitive} if it is not induced by any character of smaller modulus.
Given a character $\chi$, we write $\chi^*$ for the (uniquely defined)
primitive character inducing $\chi$. If a character $\chi$ mod $q$ is induced
by the trivial character $\chi_T$, we say that $\chi$ is {\em principal}
and write $\chi_0$ for $\chi$ (provided the modulus $q$ is clear from the 
context). In other words, $\chi_0(n)=1$ when $(n,q)=1$ and $\chi_0(n)=0$ when
$(n,q)=0$.

A Dirichlet $L$-function $L(s,\chi)$ ($\chi$ a Dirichlet character) is
defined as the analytic continuation of $\sum_n \chi(n) n^{-s}$ to the
entire complex plane; there is a pole at $s=1$ if $\chi$ is principal. 

A non-trivial zero of $L(s,\chi)$ is any $s\in \mathbb{C}$
such that $L(s,\chi)=0$ and $0 < \Re(s) < 1$. (In particular, a zero
at $s=0$ is called ``trivial'', even though its contribution can be a little
tricky to work out. The same would go for the other zeros with $\Re(s)=0$
occuring for $\chi$ non-primitive, though we will avoid this issue by
working mainly with $\chi$ primitive.) The zeros that occur at (some) negative 
integers are called {\em trivial zeros}.

The {\em critical line} is the line $\Re(s)=1/2$ in the complex
plane. Thus, the generalized Riemann hypothesis for Dirichlet $L$-functions
reads: for every Dirichlet character $\chi$,
all non-trivial zeros of $L(s,\chi)$ lie on the critical line. 
Verifiable finite versions of the generalized Riemann hypothesis generally 
read: for every Dirichlet character $\chi$ of modulus $q\leq Q$,
all non-trivial zeros of $L(s,\chi)$ with $|\Im(s)|\leq f(q)$ lie
on the critical line (where $f:\mathbb{Z}\to \mathbb{R}^+$ is some given 
function).

\section{Fourier transforms and exponential sums}

The Fourier transform on $\mathbb{R}$ is normalized here as follows:
\[\widehat{f}(t) = \int_{-\infty}^\infty e(-xt) f(x) dx.
\]

The trivial bound is $|\widehat{f}|_\infty \leq |f|_1$.
If $f$ is compactly supported (or of fast enough 
decay as $t\mapsto \pm \infty$) 
and piecewise continuous,
$\widehat{f}(t) = \widehat{f'}(t)/(2\pi i t)$ by integration by parts.
Iterating, we obtain that if $f$ is of fast decay and differentiable 
$k$ times outside finitely many points,
then
\begin{equation}\label{eq:madge}\begin{aligned}
\widehat{f}(t) &= 
O^*\left(\frac{|\widehat{f^{(k)}}|_\infty}{(2\pi t)^k}\right) =
O^*\left(\frac{|f^{(k)}|_1}{(2\pi t)^k}\right).
\end{aligned} \end{equation}
Thus, for instance, if $f$ is compactly supported,
continuous and piecewise $C^1$, then $\widehat{f}$ 
decays at least quadratically.

It could happen that $|f^{(k)}|_1=\infty$, in which case (\ref{eq:madge})
is trivial (but not false). In practice, we require $f^{(k)}\in L_1$. 
In a typical situation, $f$ is differentiable $k$ times except at
$x_1,x_2,\dotsc,x_k$, where it is differentiable only $(k-2)$ times;
the contribution of $x_i$ (say) to $|f^{(k)}|_1$ is then 
$|\lim_{x\to x_i^+} f^{(k-1)}(x) - \lim_{x\to x_i^-} f^{(k-1)}(x)|$.


The following bound is standard (see, e.g., \cite[Lemma 3.1]{Tao}):
for $\alpha\in \mathbb{R}/\mathbb{Z}$ and 
$f:\mathbb{R}\to \mathbb{C}$ compactly supported and piecewise continuous,
\begin{equation}\label{eq:ra}
\left|\sum_{n\in \mathbb{Z}} f(n) e(\alpha n)\right| \leq 
\min\left(|f|_1 + \frac{1}{2} |f'|_1,\frac{\frac{1}{2} |f'|_1}{|\sin(\pi
\alpha)|}\right).\end{equation} (The first bound follows from
$\sum_{n\in \mathbb{Z}} |f(n)| \leq |f|_1 + (1/2) |f'|_1$, which, in turn
is a quick consequence of the fundamental theorem of calculus; the second
bound is proven by summation by parts.)
The alternative bound $(1/4) |f''|_1/|\sin(\pi \alpha)|^2$ given
in \cite[Lemma 3.1]{Tao} (for $f$ continuous and piecewise $C^1$) can usually be 
improved by the following estimate.

\begin{lemma}\label{lem:areval}
Let $f:\mathbb{R}\to \mathbb{C}$ be compactly supported, continuous and
piecewise $C^1$. Then
\begin{equation}\label{eq:quisquil}
\left|\sum_{n\in \mathbb{Z}} f(n) e(\alpha n)\right| \leq 
\frac{\frac{1}{4} |\widehat{f''}|_{\infty}}{(\sin \pi \alpha)^2} 
\end{equation}
for every $\alpha \in \mathbb{R}$.
\end{lemma}
As usual, the assumption of compact support could easily be relaxed to 
an assumption of fast decay.
\begin{proof}
By the Poisson summation formula,
\[\sum_{n=-\infty}^\infty f(n) e(\alpha n) = \sum_{n=-\infty}^\infty \widehat{f}(
n-\alpha).\]
Since $\widehat{f}(t) = \widehat{f'}(t)/(2\pi i t)$,
\[
\sum_{n=-\infty}^\infty \widehat{f}(n-\alpha) = 
 \sum_{n=-\infty}^\infty \frac{\widehat{f'}(n-\alpha)}{2\pi i (n-\alpha)} = 
 \sum_{n=-\infty}^\infty \frac{\widehat{f''}(n-\alpha)}{(2\pi i (n-\alpha))^2} .
\]
By Euler's formula $\pi \cot s \pi = 1/s + \sum_{n=1}^{\infty} (1/(n+s) - 
1/(n-s))$,
\begin{equation}\label{eq:euler}
\sum_{n=-\infty}^{\infty} \frac{1}{(n+s)^2} = - (\pi \cot s \pi)' = 
\frac{\pi^2}{(\sin s \pi)^2} .\end{equation}
Hence
\[\left|\sum_{n=-\infty}^\infty \widehat{f}(n-\alpha)\right| \leq
|\widehat{f''}|_\infty \sum_{n=-\infty}^\infty \frac{1}{(2\pi (n-\alpha))^2}
= |\widehat{f''}|_\infty \cdot \frac{1}{(2\pi)^2} \cdot \frac{\pi^2}{(\sin
\alpha \pi)^2} .\]
\end{proof}
The trivial bound $|\widehat{f''}|_\infty \leq |f''|_1$, applied to
(\ref{eq:quisquil}), recovers the bound in \cite[Lemma 3.1]{Tao}. In
order to do better, we will give a tighter bound for $|\widehat{f''}|_\infty$
in Appendix \ref{sec:norms} when $f$ is equal to one of our main smoothing
functions ($f=\eta_2$).

Integrals of multiples of $f''$ (in particular,
 $|f''|_1$ and $\widehat{f''}$) can still be made sense of 
when $f''$ is undefined at a finite number
of points, provided $f$ is understood as a distribution (and $f'$ has
finite total variation). This is the case, in particular, for $f=\eta_2$.

\begin{center}
* * *
\end{center}

When we need to estimate $\sum_n f(n)$ precisely, we will use the Poisson 
summation formula: \[\sum_n f(n) = \sum_n \widehat{f}(n).\]
We will not have to worry about convergence here, since we will apply the Poisson summation formula only to compactly supported functions $f$ whose
Fourier transforms decay at least quadratically.

\section{Mellin transforms}\label{subs:milly}

The {\em Mellin transform} of a function $\phi:(0,\infty)\to \mathbb{C}$ is
\begin{equation}\label{eq:souv}
M \phi(s) := \int_0^{\infty} \phi(x) x^{s-1} dx .\end{equation}
If $\phi(x) x^{\sigma-1}$ is in $\ell_1$ with respect to $dt$ 
(i.e., $\int_0^\infty |\phi(x)| x^{\sigma-1} dx < \infty$), then the Mellin transform
is defined on the line $\sigma + i \mathbb{R}$. Moreover, if
$\phi(x) x^{\sigma-1}$ is in $\ell_1$ for $\sigma=\sigma_1$ and for
$\sigma = \sigma_2$, where $\sigma_2>\sigma_1$, then it is easy to see that
it is also in $\ell_1$ for all $\sigma \in (\sigma_1,\sigma_2)$, 
and that, moreover,
the Mellin transform is holomorphic on $\{s: \sigma_1 < \Re(s) < \sigma_2\}$. We
then say that $\{s: \sigma_1 < \Re(s) < \sigma_2\}$ is a {\em strip of
holomorphy} for the Mellin transform.

The Mellin transform becomes a Fourier transform (of $\eta(e^{-2\pi v})
e^{-2\pi v \sigma}$) by means of the change of variables $x = e^{-2\pi v}$.
We thus obtain, for example, that 
the Mellin transform is an isometry, in the sense that
\begin{equation}\label{eq:victi}
\int_0^\infty |f(x)|^2 x^{2\sigma} \frac{dx}{x} = \frac{1}{2\pi} \int_{-\infty}^\infty
|Mf(\sigma+it)|^2 dt.\end{equation}
Recall that, in the case of the Fourier transform,
for $|\widehat{f}|_2 = |f|_2$ to hold, it is enough that $f$
be in $\ell_1 \cap \ell_2$. This gives us that, for (\ref{eq:victi}) to hold,
it is enough that $f(x) x^{\sigma-1}$ be in $\ell_1$ and $f(x) x^{\sigma-1/2}$
be in $\ell_2$ (again, with respect to $dt$, in both cases).

We write $f\ast_M g$ for the multiplicative, or Mellin, convolution of $f$
and $g$:
\begin{equation}\label{eq:solida}
(f\ast_M g)(x) = \int_0^{\infty} f(w) g\left(\frac{x}{w}\right) \frac{dw}{w}.
\end{equation}
In general, \begin{equation}\label{eq:zorbag}
M (f\ast_M g) = Mf \cdot Mg\end{equation} and 
\begin{equation}\label{eq:mouv}
M(f\cdot g)(s) = \frac{1}{2\pi i}\int_{\sigma-i\infty}^{\sigma+i\infty}
Mf(z) Mg(s-z) dz\;\;\;\;\;\;\;\;\text{\cite[\S 17.32]{MR1243179}}\end{equation}
provided that $z$ and $s-z$ are within the strips on which $Mf$ and $Mg$
(respectively) are well-defined.

We also have several useful transformation rules, just as for the Fourier
transform. For example, 
\begin{equation}\label{eq:harva}\begin{aligned}
M(f'(t))(s) &= - (s-1)\cdot Mf(s-1),\\
M(t f'(t))(s) &= - s\cdot  Mf(s),\\
M((\log t) f(t))(s) &= (Mf)'(s)\end{aligned}\end{equation}
(as in, e.g., \cite[Table 1.11]{Mellin}).

Let 
\[\eta_2 = (2 \cdot 1_{\lbrack 1/2,1\rbrack}) \ast_M 
(2 \cdot 1_{\lbrack 1/2,1\rbrack}).\]
 Since (see, e.g., \cite[Table 11.3]{Mellin}
or \cite[\S 16.43]{MR1243179})
\[(M I_{\lbrack a,b\rbrack})(s) = \frac{b^s - a^s}{s},\]
we see that
\begin{equation}\label{eq:envy}
M\eta_2(s) = \left(\frac{1 - 2^{-s}}{s}\right)^2,\;\;\;\;\;
M\eta_4(s) = \left(\frac{1 - 2^{-s}}{s}\right)^4 .\end{equation}

Let $f_z = e^{-zt}$, where $\Re(z)>0$. Then
\[\begin{aligned}
(Mf)(s) &= \int_0^\infty e^{-zt} t^{s-1} dt = \frac{1}{z^s} \int_0^\infty
e^{-t} dt\\&= \frac{1}{z^s} \int_0^{z\infty} e^{-u} u^{s-1} du = 
\frac{1}{z^s} \int_0^\infty e^{-t} t^{s-1} dt = \frac{\Gamma(s)}{z^s},
\end{aligned}\]
where the next-to-last step holds by contour integration,
and the last step holds by the definition of the Gamma function $\Gamma(s)$.

\section{Bounds on sums of $\mu$ and $\Lambda$}
We will need some simple explicit bounds on sums involving the von Mangoldt
function $\Lambda$ and the Moebius function $\mu$. In non-explicit work,
such sums are usually bounded using the prime number theorem, or rather
using the properties of the zeta function $\zeta(s)$ underlying the prime number theorem.
Here, however, we need robust, fully explicit bounds valid over just about
any range. 

For the most part, we will just be quoting the literature, 
supplemented with some computations when needed. The proofs in the literature
are sometimes based on properties of $\zeta(s)$, and sometimes on more
elementary facts.

First, let us see some bounds involving $\Lambda$.
The following bound can be easily derived from \cite[(3.23)]{MR0137689},
supplemented by a quick calculation of the contribution of powers of primes
$p<32$:
\begin{equation}\label{eq:rala}
\sum_{n\leq x} \frac{\Lambda(n)}{n} \leq \log x.
\end{equation}
We can derive a bound in the other direction from \cite[(3.21)]{MR0137689}
(for $x>1000$, adding the contribution of all prime powers $\leq 1000$)
and a numerical verification for $x\leq 1000$:
\begin{equation}\label{eq:ralobio}
\sum_{n\leq x} \frac{\Lambda(n)}{n} \geq \log x - \log \frac{3}{\sqrt{2}} .
\end{equation}

We also use the following older bounds:
\begin{enumerate}
\item By the second table in \cite[p. 423]{MR1320898}, supplemented by a computation for $2\cdot 10^6\leq V\leq 4\cdot 10^6$,
\begin{equation}\label{eq:trado1}\sum_{n\leq y} \Lambda(n)\leq 1.0004 y
\end{equation}
for $y\geq 2\cdot 10^6$.
\item \begin{equation}\label{eq:trado2}\sum_{n\leq y} \Lambda(n) < 1.03883 y\end{equation}
 for every $y>0$ \cite[Thm. 12]{MR0137689}.
\end{enumerate}

For all $y>663$,
\begin{equation}\label{eq:chronop}
\sum_{n\leq y} \Lambda(n) n < 1.03884 \frac{y^2}{2} ,\end{equation}
where we use (\ref{eq:trado2}) and partial summation for $y>200000$, and
a computation for $663 < y\leq 200000$. Using instead the second table in 
\cite[p. 423]{MR1320898}, together with computations for small $y<10^7$ and
partial summation, we get that
\begin{equation}\label{eq:nicro}
\sum_{n\leq y} \Lambda(n) n < 1.0008 \frac{y^2}{2} \end{equation}
for $y>1.6\cdot 10^6$.

Similarly,
\begin{equation}\label{eq:charol}
\sum_{n\leq y} \frac{\Lambda(n)}{\sqrt{n}} < 2\cdot 1.0004 \sqrt{y}\end{equation}
for all $y\geq 1$.

It is also true that
\begin{equation}\label{eq:kast}
\sum_{y/2<p\leq y} (\log p)^2 \leq \frac{1}{2} y (\log y)
\end{equation}
for $y\geq 117$: this holds for $y\geq 2\cdot 758699$ by \cite[Cor. 2]{MR0457373}
(applied to $x = y$, $x = y/2$ and $x=2 y/3$) and for $117\leq y< 2\cdot 758699$
by direct computation.

Now let us see some estimates on sums involving $\mu$. 
The situation here is less satisfactory than for sums involving $\Lambda$. 
 The main reason is that the complex-analytic approach to
estimating $\sum_{n\leq N} \mu(n)$ would involve $1/\zeta(s)$ rather than 
$\zeta'(s)/\zeta(s)$, and thus strong explicit bounds on the residues of 
$1/\zeta(s)$ would be needed. Thus, explicit estimates on sums involving $\mu$
are harder to obtain than estimates on sums involving $\Lambda$. This is so
 even
though analytic number theorists are generally used (from the habit of
non-explicit work) to see the estimation of one kind of sum or the other 
as essentially the same task.

Fortunately, in the case of sums of the type $\sum_{n\leq x} \mu(n)/n$
for $x$ arbitrary (a type of sum that will be rather important for us), all we need is a saving of $(\log n)$ or $(\log n)^2$ 
on the trivial bound.
This is provided by the following.
\begin{enumerate}
\item (Granville-Ramar\'e \cite{MR1401709}, Lemma 10.2)
\begin{equation}\label{eq:grara}
\left|\sum_{n\leq x: \gcd(n,q)=1} \frac{\mu(n)}{n}\right|\leq 1\end{equation}
for all $x$, $q\geq 1$,
\item (Ramar\'e \cite{MR3019422}; cf. El Marraki \cite{MR1378588}, \cite{ElMarraki})
\begin{equation}\label{eq:marraki}
\left|\sum_{n\leq x} \frac{\mu(n)}{n}\right|\leq \frac{0.03}{\log x}
\end{equation}
for $x\geq 11815$. 
\item (Ramar\'e \cite{Ramsev}) 
\begin{equation}\label{eq:ronsard}
\sum_{n\leq x: \gcd(n,q)=1}
\frac{\mu(n)}{n} = 
O^*\left(\frac{1}{\log x/q} \cdot \frac{4}{5} \frac{q}{\phi(q)}\right)
\end{equation} for all $x$ and all $q\leq x$;
\begin{equation}\label{eq:meproz}
\sum_{n\leq x: \gcd(n,q)=1}
\frac{\mu(n)}{n} \log \frac{x}{n} = O^*\left(1.00303 \frac{q}{\phi(q)}\right) 
\end{equation}
for all $x$ and all $q$.
\end{enumerate}
Improvements on these bounds would lead to improvements on type I estimates, but not in what are the worst terms overall at this point.

A computation carried out by the author has proven 
the following inequality for all real $x\leq 10^{12}$:
\begin{equation}\label{eq:ramare}
\left|\sum_{n\leq x} \frac{\mu(n)}{n}\right|\leq \sqrt{\frac{2}{x}}
\end{equation}
The computation was conducted rigorously by means of interval arithmetic.
For the sake of verification, we record that
\[5.42625\cdot 10^{-8}
\leq \sum_{n\leq 10^{12}} \frac{\mu(n)}{n}\leq
5.42898\cdot 10^{-8}.
\]

Computations also show that the stronger bound 
\[\left|\sum_{n\leq x} \frac{\mu(n)}{n}\right|\leq \frac{1}{2\sqrt{x}}\]
holds for all $3\leq x\leq 7727068587$, but not for $x=7727068588-\epsilon$.

Earlier, numerical work carried out by Olivier Ramar\'e 
\cite{MR3263938}
had shown 
that (\ref{eq:ramare}) holds for all $x\leq 10^{10}$.

\section{Interval arithmetic and the bisection method}\label{sec:koloko}

{\em Interval arithmetic} has, at its basic data type, intervals of the form
$I=\lbrack a/2^\ell,b/2^\ell\rbrack$, where $a,b,\ell\in \mathbb{Z}$ and $a\leq b$.
Say we have a real number $x$, and we want to know $\sin(x)$. In general,
we cannot represent $x$ in a computer, in part because it may have no finite description.
The best we can do is to construct
an interval of the form $I=\lbrack a/2^\ell,b/2^\ell\rbrack$
in which $x$ is contained. 

What we ask of a routine
in an interval-arithmetic package is to construct an interval 
$I' = \lbrack a'/2^{\ell'},b'/2^{\ell'}\rbrack$ in which $\sin(I)$ is contained.
(In practice, this is done partly in software, by means of polynomial
approximations to $\sin$ with precise error terms, and partly in hardware,
by means of an efficient usage of rounding conventions.) This gives us,
in effect, a value for $\sin(x)$ (namely, $(a'+b')/2^{\ell'+1}$)
and a bound on the error term (namely, $(b'-a')/2^{\ell'+1}$). 

There are several implementations of interval arithmetic available.
We will almost always use D. Platt's implementation
\cite{Platt} of double-precision interval arithmetic based on Lambov's 
\cite{Lamb} ideas. (At one point, we will use the PROFIL/BIAS interval 
arithmetic 
package \cite{Profbis}, since it underlies the VNODE-LP \cite{VNODELP}
package, which we use to bound an integral.)

The {\em bisection method} is a particularly simple method for finding
maxima and minima of functions, as well as roots. 
It combines rather nicely with interval arithmetic, 
which makes the method rigorous.
We follow an implementation
 based on \cite[\S 5.2]{MR2807595}. Let us go over the basic ideas.


Let us use
the bisection method to find the minima (say) of a function $f$ on a 
compact interval $I_0$. (If the interval is non-compact, we 
generally apply the bisection method to a compact sub-interval and use other
tools, e.g., power-series expansions, in the complement.)
The method proceeds by splitting
an interval into two repeatedly, discarding the halfs where the minimum
cannot be found. More precisely, if we implement it by interval arithmetic,
it proceeds as follows. First, in an optional initial step,
 we subdivide (if necessary) the interval
$I_0$ into smaller intervals $I_k$
to which the algorithm will actually be applied. For each $k$, 
interval arithmetic gives
us a lower bound $r_k^-$ and an upper bound $r_k^+$ on $\{f(x):x\in I_k\}$;
here $r_k^-$ and $r_k^+$ are both of the form $a/2^\ell$, $a,\ell\in \mathbb{Z}$.
Let $m_0$ be the minimum of $r_k^+$ over all $k$. We can discard all the
intervals $I_k$ for which $r_k^- > m_0$.
Then we apply the main procedure: starting with $i=1$,
split each surviving interval into two equal halves,
recompute the lower and upper bound on each half, define $m_i$, as before, 
to be the minimum of all upper bounds, and discard, again,
the intervals on which the lower bound is larger than $m_i$; increase $i$
by $1$. We repeat the
main procedure as often as needed. In the end, we obtain that
the minimum is no smaller than
the minimum of the lower bounds (call them $(r^{(i)})_k^-$) on all 
surviving intervals $I_k^{(i)}$. Of course, we also obtain that the minimum 
(or minima, if there is more than one) must
lie in one of the surviving intervals.

It is easy to see how the same method can be applied (with a trivial modification) to find maxima, or (with very slight changes) to find the roots of a 
real-valued function on a compact interval. 

\part{Minor arcs}\label{part:min}
\chapter{Introduction}
The circle method expresses the number of solutions to a given problem in
terms of exponential sums.
Let $\eta:\mathbb{R}^+\to \mathbb{C}$ be a smooth function,
$\Lambda$ the von Mangoldt function (defined as in (\ref{eq:koj})) and
$e(t) = e^{2\pi i t}$.
The estimation of exponential sums of the type
\begin{equation}\label{eq:trintrabo}
S_{\eta}(\alpha,x) = \sum_n \Lambda(n) e(\alpha n) \eta(n/x),\end{equation}
where $\alpha\in \mathbb{R}/\mathbb{Z}$, already lies at the basis of 
 Hardy and Littlewood's
approach to the ternary Goldbach problem by means of the circle method
\cite{MR1555183}. 
The division of the circle $\mathbb{R}/\mathbb{Z}$ into
``major arcs'' and ``minor arcs'' goes back to 
Hardy and Littlewood's development of the circle method for other problems.
As they themselves noted, assuming GRH means that, for the ternary Goldbach
problem, all of the circle can be, in effect, subdivided into major arcs --
that is, under GRH, (\ref{eq:trintrabo}) can be estimated with major-arc techniques
for $\alpha$ arbitrary. 
They needed to make such an assumption precisely because they did not yet
know how to estimate $S_\eta(\alpha,x)$ on the minor arcs.

Minor-arc techniques for Goldbach's problem were first developed by Vinogradov
\cite{zbMATH02522879}. These techniques make it possible to work without GRH. 
The main obstacle to a full proof of the ternary Goldbach conjecture since
then has been that, in spite of gradual improvements, minor-arc bounds
have simply not been strong enough.

As in all work to date, our aim will be to give useful upper bounds on
(\ref{eq:trintrabo}) for $\alpha$ in the minor bounds, rather than 
the precise estimates that are typical of the major-arc case. We
will have to give upper bounds that are qualitatively stronger than those
known before. (In Part \ref{part:concl}, we will also show how to use them
more efficiently.)

Our main challenge will be to give a good upper bound whenever $q$ is larger
than a constant $r$. Here ``sufficiently good'' means ``smaller than the trivial
bound divided by a large constant, and getting even smaller quickly as $q$
grows''. Our bound must also be good for $\alpha = a/q + \delta/x$, where
$q<r$ but $\delta$ is large. (Such an $\alpha$ may be said to lie on the
{\em tail} ($\delta$ large) of a major arc ($q$ small).)

Of course, all expressions must be explicit and all
constants in the leading terms of the bound must be small. Still, the
main requirement is a qualitative one. For instance, we know in advance that
a single factor of $\log x$ would
be the end of us. That is, we know that, if there is
a single term of the form, say, $(x \log x)/q$, and the trivial bound
is about $x$, we are lost: $(x \log x)/q$ is greater than $x$ for $x$
large and $q$ constant. 

The quality of the results here is due to several new ideas of
general applicability. In particular, \S \ref{subs:vaucanc} introduces
a way to obtain cancellation from Vaughan's identity. 
Vaughan's identity is a two-log gambit, in that it introduces two 
convolutions (each of them at a cost of $\log$)
 and offers a great deal of flexibility in compensation. 
One of the ideas presented here is that at least one of two $\log$s 
can be successfully recovered after having been given away in the first stage of
the proof. This reduces
the cost of the use of this basic identity in this and, presumably, many
other problems. 

There are several other improvements that make a qualitative
difference; see the discussions at the beginning of \S \ref{sec:typeI}
and \S \ref{sec:typeII}. Considering smoothed sums 
-- now a common idea -- also helps.
(Smooth sums here go back to
Hardy-Littlewood \cite{MR1555183} --
both in the general context of the circle method and in the context
of Goldbach's ternary problem. In recent work on the problem, they reappear in
\cite{Tao}.)

\section{Results}

The main bound we are about to see is essentially proportional to 
$((\log q)/\sqrt{\phi(q)})\cdot x$. The term $\delta_0$ serves to
 improve the bound when we are on the tail of an arc.
\begin{theorem}\label{thm:minmain}
Let $x\geq x_0$, $x_0 = 2.16\cdot 10^{20}$.
Let $S_{\eta}(\alpha,x)$ be as in (\ref{eq:trintrabo}),
with $\eta$ defined in (\ref{eq:eqeta}).
Let $2 \alpha = a/q + \delta/x$, $q\leq Q$,
$\gcd(a,q)=1$, $|\delta/x|\leq 1/q Q$, where $Q = (3/4) x^{2/3}$.
If $q\leq x^{1/3}/6$, then
\begin{equation}\label{eq:kraw}
\begin{aligned}
&|S_{\eta}(\alpha,x)| \leq 
 \frac{R_{x,\delta_0 q} \log \delta_0 q + 0.5}{\sqrt{\delta_0 \phi(q)}} \cdot x
+ \frac{2.5 x}{\sqrt{\delta_0 q}} + 
 \frac{2x}{\delta_0 q} \cdot L_{x,\delta_0 q, q}
+ 3.36 x^{5/6},\end{aligned}\end{equation} where
\begin{equation}\label{eq:tosca}\begin{aligned}
\delta_0 &= \max(2,|\delta|/4),\;\;\;\;\;
R_{x,t} = 0.27125 \log 
\left(1 + \frac{\log 4 t}{2 \log \frac{9 x^{1/3}}{2.004 t}}\right)
 + 0.41415, \\
L_{x,t,q} &=
\frac{q}{\phi(q)}
\left(\frac{13}{4} \log t + 7.82\right)
+ 13.66 \log t + 37.55.\end{aligned}\end{equation}

If $q > x^{1/3}/6$, then
\[|S_{\eta}(\alpha,x)|\leq 
0.276 x^{5/6} (\log x)^{3/2} + 1234 x^{2/3} \log x.\]
\end{theorem}
The factor $R_{x,t}$ is small in practice; for instance, for
$x = 10^{25}$ and $\delta_0 q = 5\cdot 10^5$ 
(typical ``difficult'' values), $R_{x,\delta_0 q}$ equals $0.59648\dotsc$.

The classical choice\footnote{Or, more precisely,
the choice made by Vinogradov and followed by most of the literature since him.
 Hardy and Littlewood
\cite{MR1555183} worked with $\eta(t) = e^{-t}$.} for $\eta$ in (\ref{eq:trintrabo}) 
is $\eta(t) = 1$ for $t\leq 1$, 
$\eta(t) = 0$ for $t>1$, which, of course, is not smooth, or even
continuous.
We use
\begin{equation}\label{eq:eqeta}
\eta(t) = \eta_2(t) = 4 \max(\log 2 - |\log 2 t|,0),\end{equation}
as in Tao \cite{Tao}, in part for purposes of comparison. (This is
the multiplicative convolution of the characteristic function of an
interval with itself.) Nearly all work
should be applicable to any other sufficiently smooth function $\eta$ 
of fast decay. It is important that $\widehat{\eta}$ decay at least
quadratically.

We are not forced
to use the same smoothing function as in Part \ref{part:maj}, and we do not.
As was explained in the introduction, the simple technique
(\ref{eq:asco}) allows us
to work with one smoothing function on the major arcs and 
with another one on the minor arcs.

\section{Comparison to earlier work}
Table \ref{tab:bloid} compares the bounds for the ratio
$|S_{\eta}(a/q,x)|/x$
given by this paper and by \cite{Tao}[Thm. 1.3] for $x= 10^{27}$ and different
values of $q$.
We are comparing worst cases: $\phi(q)$ as small
as possible ($q$ divisible by $2\cdot 3\cdot 5\dotsb$) in the result here,
and $q$ divisible by $4$ (implying $4\alpha \sim a/(q/4)$) in Tao's result.
The main term in the result in this paper improves
slowly with increasing $x$; the results in \cite{Tao} worsen slowly
with increasing $x$.
The qualitative gain with respect to the main term in 
\cite[(1.10)]{Tao} is in the order of
$\log(q) \sqrt{\phi(q)/q}$. Notice also that the bounds in \cite{Tao}
are not log-free; in \cite[(1.10)]{Tao}, there is a term proportional to 
$x (\log x)^2/q$. This becomes larger than the trivial bound $x$ for
$x$ very large.

The results in \cite{MR1803131} are unfortunately
worse than the trivial bound in the range covered by Table \ref{tab:bloid}.
Ramar\'e's results (\cite[Thm. 3]{MR2607306}, \cite[Thm. 6]{Ramlater}) 
are not applicable within the range, since
neither of the conditions $\log q \leq (1/50) (\log x)^{1/3}$,
$q\leq x^{1/48}$ is satisfied. Ramar\'e's bound in
\cite[Thm. 6]{Ramlater} is
\begin{equation}\label{eq:jomor}
\left|\sum_{x<n\leq 2x} \Lambda(n) e(an/q)\right|\leq 13000
\frac{\sqrt{q}}{\phi(q)} x\end{equation}
for $20\leq q\leq x^{1/48}$.
We should underline that, while both the constant $13000$ and the
condition $q\leq x^{1/48}$ keep (\ref{eq:jomor}) from being immediately
useful in the present context, (\ref{eq:jomor}) is asymptotically better
than the results here as $q\to \infty$. (Indeed, qualitatively speaking,
the form of (\ref{eq:jomor}) is the best one can expect from results
derived by the family of methods stemming from Vinogradov's work.) There
is also unpublished work by Ramar\'e (ca. 1993) with different constants for 
$q\ll (\log x/\log \log x)^4$.
\begin{table}
\begin{center} \begin{tabular}{l|l|l}
$q_0$ &$\frac{|S_{\eta}(a/q,x)|}{x}$, HH& 
$\frac{|S_{\eta}(a/q,x)|}{x}$, Tao\\ \hline
$10^5$ & $0.04661$ & $0.34475$\\
$1.5\cdot 10^5$ & $0.03883$ & $0.28836$\\
$2.5\cdot 10^5$ & $0.03098$ & $0.23194$\\
$5\cdot 10^5$ & $0.02297$ & $0.17416$\\
$7.5\cdot 10^5$ & $0.01934$ & $0.14775$\\
$10^6$ & $0.01756$ & $0.13159$\\
$10^7$ & $0.00690$ & $0.05251$
\end{tabular} 
\caption{Worst-case upper bounds on $x^{-1} |S_{\eta}(a/2q,x)|$ for
$q\geq q_0$,
$|\delta|\leq 8$, $x= 10^{27}$. The trivial bound is $1$.}\label{tab:bloid}
\end{center}\end{table}

\section{Basic setup}

In the minor-arc regime, the first step in estimating an exponential sum on 
the primes generally consists in the application of an identity expressing
the von Mangoldt function $\Lambda(n)$ in terms of a sum of convolutions
of other functions.

\subsection{Vaughan's identity}

We recall Vaughan's identity \cite{MR0498434}: 
\begin{equation}\label{eq:vaughin}
\Lambda = \mu_{\leq U} \ast \log + \mu_{\leq U}\ast \Lambda_{\leq V}\ast 1 + 
\mu_{>U}\ast \Lambda_{>V} \ast 1 + \Lambda_{\leq V},\end{equation}
where $1$ is the constant function $1$, and where
 we write 
\[f_{\leq z}(n) = \begin{cases} f(n) &\text{if $n\leq z$}\\ 0 & \text{if
$n> z$,}\end{cases} \;\;\;\;\;\;\;
f_{>z}(n) = \begin{cases} 0 &\text{if $n\leq z$}\\ f(n) & \text{if
$n> z$.}\end{cases}\]
Here $f\ast g$ denotes the Dirichlet convolution 
$(f\ast g)(n) = \sum_{d|n} f(d) g(n/d)$. We can set the values of $U$ and $V$
however we wish.

Vaughan's identity is essentially
a consequence of the M\"obius inversion formula
\begin{equation}\label{eq:mobinv}
(1\ast \mu)(n) = \begin{cases} 1 &\text{if $n=1$,}\\ 0 & \text{otherwise.}
\end{cases}\end{equation}
Indeed, by
(\ref{eq:mobinv}),
\[\begin{aligned}
\Lambda_{>V}(n) &= \sum_{d m |n} \mu(d) \Lambda_{>V}(m)\\
&= \sum_{dm|n} \mu_{\leq U}(d) \Lambda_{>V}(m) +
\sum_{dm|n} \mu_{>U}(d) \Lambda_{>V}(m). 
\end{aligned}\]
Applying to this the trivial equality 
$\Lambda_{>V} = \Lambda - \Lambda_{\leq V}$, as well as the simple fact
that $1\ast \Lambda = \log$, we obtain that 
\[\Lambda_{>V}(n) = 
\sum_{d|n} \mu_{\leq U}(d) \log(n/d) - 
\sum_{dm|n} \mu_{\leq U}(d) \Lambda_{\leq V}(m) +
\sum_{dm|n} \mu_{>U}(d) \Lambda_{>V}(m).\]
By $\Lambda_V = \Lambda_{>V} + \Lambda_{\geq V}$, we conclude that
 Vaughan's identity (\ref{eq:vaughin}) holds.

Applying Vaughan's identity, we easily get that,
for any function 
$\eta:\mathbb{R}\to \mathbb{R}$, any completely multiplicative function
$f:\mathbb{Z}^+\to \mathbb{C}$ and any $x>0$, $U,V\geq 0$,
\begin{equation}\label{eq:bob}
\sum_n \Lambda(n) f(n) e(\alpha n) \eta(n/x) = S_{I,1} - S_{I,2} + S_{II} + 
S_{0,\infty},
\end{equation}
where
\begin{equation}\label{eq:nielsen}\begin{aligned}
S_{I,1} &= \sum_{m\leq U} \mu(m) f(m)\sum_n (\log n) e(\alpha m n) f(n) \eta(m n/x),\\
S_{I,2} &= \sum_{d\leq V} \Lambda(d) f(d) \sum_{m\leq U} \mu(m) f(m) \sum_n e(\alpha d m n) f(n)
 \eta(d m n/x),\\
S_{II} &= \sum_{m>U} f(m) \left(\mathop{\sum_{d>U}}_{d|m} \mu(d)\right) \sum_{n>V}
\Lambda(n) e(\alpha m n) f(n) \eta(m n/x),\\
S_{0,\infty} &= \sum_{n\leq V} \Lambda(n) e(\alpha n) f(n) \eta(n/x) .
\end{aligned}\end{equation}
We will use the function
\begin{equation}\label{eq:joroy}
f(n) = \begin{cases} 1 &\text{if $\gcd(n,v)=1$,}\\
0 &\text{otherwise,}\end{cases}\end{equation}
where $v$ is a small, positive, square-free integer.
(Our final choice will be $v=2$.)
 Then
\begin{equation}\label{eq:sofot}
S_\eta(x,\alpha) = S_{I,1} - S_{I,2} + S_{II} + S_{0,\infty} + S_{0,w},\end{equation}
where $S_{\eta}(x,\alpha)$ is as in (\ref{eq:trintrabo}) and
\[S_{0,v} = \sum_{n|v} \Lambda(n) e(\alpha n) \eta(n/x).\]

The sums $S_{I,1}$, $S_{I,2}$ are called ``of type I'',
the sum $S_{II}$ is called ``of type II'' (or bilinear).
(The not-all-too colorful nomenclature goes back to Vinogradov.)
 The sum $S_{0,\infty}$ is in general negligible; for our later choice of $V$ and
$\eta$, it will be in fact $0$. The sum $S_{0,v}$ will be negligible as well.

As we already discussed in the introduction, Vaughan's identity is highly 
flexible (in that we can choose $U$ and $V$ at will) but somewhat
inefficient in practice (in that a trivial estimate for the right side of
(\ref{eq:sofot}) is actually larger than a trivial estimate for the left
side of (\ref{eq:sofot})). Some of our work will consist in regaining part
of what is given up when we apply Vaughan's identity.

\subsection{An alternative route}

There is an alternative route -- namely, to use a less sacrificial, though
also more inflexible, identity. While this was not, in the end, the route
that was followed, let us nevertheless discuss it in some detail, in part
so that we can understand to what extent it was, in retrospect, viable, and
in part so as to see how much of the work we will undertake is really more
or less independent of the particular identity we choose.

Since $\zeta'(s)/\zeta(s) = \sum_n \Lambda(n) n^{-s}$ and
\begin{equation}\label{eq:midine}\begin{aligned}
\left(\frac{\zeta'(s)}{\zeta(s)}\right)^{(2)} &= 
\left(\frac{\zeta''(s)}{\zeta(s)} - 
\frac{\left(\zeta'(s)\right)^2}{\zeta(s)^2}\right)' \\ &= 
\frac{\zeta^{(3)}(s)}{\zeta(s)} - \frac{3 \zeta''(s) \zeta'(s)}{\zeta(s)^2} + 
2 \left(\frac{\zeta'(s)}{\zeta(s)}\right)^3\\
&= \frac{\zeta^{(3)}(s)}{\zeta(s)} - 3 \left(\frac{\zeta'(s)}{\zeta(s)}\right)'
\cdot \frac{\zeta'(s)}{\zeta(s)} - \left(\frac{\zeta'(s)}{\zeta(s)}\right)^3,
\end{aligned}\end{equation}
we can see, comparing coefficients, that
\begin{equation}\label{eq:somoro}
\Lambda\cdot \log^2 = \mu\ast \log^3 - 3 (\Lambda\cdot \log)\ast \Lambda -
\Lambda\ast\Lambda\ast \Lambda,\end{equation}
as was stated by Bombieri in \cite{MR0396435}.

Here the term $\mu\ast \log^3$ is of the same kind as the term
$\mu_{\leq U} \ast \log$ we have to estimate if we use Vaughan's identity,
though the fact that there is no truncation at $U$ means that one of the
error terms will get larger -- it will be proportional to $x$, in fact,
if we sum from $1$ to $x$. The trivial upper bound on the sum of
$\Lambda\cdot \log^2$ from $1$ to $x$ is $x (\log x)^2$; thus, an error
term of size $x$ is barely acceptable.

In general, when we have a double or triple sum,
we are not very good at getting better than trivial bounds
in ranges in which all but one of the variables
are very small. This is the source of the large error term that appears
in the sum involving $\mu\ast \log^3$ because we are no longer truncating as 
for $\mu_{\leq U} \ast \log$. It will also be the source of other large error
terms, including one that would be too large -- namely, the one
 coming from the term
$(\Lambda \cdot \log) \ast \Lambda$ when the variable of 
$\Lambda\cdot \log$ is large and that of $\Lambda$ is small. (The trivial
bound on that range is $\gg x \log x$.) 

We avoid this problem by substituting the identity
$\Lambda \cdot \log = \mu\ast \log^2 - \Lambda\ast \Lambda$ inside
(\ref{eq:somoro}):
\begin{equation}\label{eq:saidso}
\Lambda\cdot \log^2 = \mu\ast \log^3 - 3 (\mu\ast \log^2)\ast \Lambda +
2\Lambda\ast\Lambda\ast \Lambda.\end{equation}
(We could also have got this directly from the next-to-last line in
(\ref{eq:midine}).) When the variable of $\Lambda$ in 
$(\mu\ast \log^2)\ast \Lambda$ is small, the variable of $\mu\ast \log^2$
is large, and we can estimate the resulting term using the same techniques
as for $\mu\ast \log^3$.

It is easy to see that we can in fact mix (\ref{eq:somoro}) and 
(\ref{eq:saidso}):
\begin{equation}\label{eq:gumilev}
\begin{aligned}\Lambda\cdot \log^2 &= \mu\ast \log^3 - 3 
\left((\Lambda\cdot\log)\ast \Lambda_{>V} + 
(\mu\ast \log^2) \ast \Lambda_{\leq V}\right) \\&+ \left(
- \Lambda_{>V}\ast \Lambda\ast \Lambda 
+ 2\Lambda_{\leq V} \ast\Lambda\ast \Lambda\right)\end{aligned}\end{equation}
for $V$ arbitrary. Note here that there is some cancellation
in the last term: writing
\begin{equation}\label{eq:f3V}
F_{3,V}(n) = \left(- \Lambda_{>V}\ast \Lambda\ast \Lambda 
+ 2\Lambda_{\leq V} \ast\Lambda\ast \Lambda\right)(n),\end{equation}
we can check easily that, for $n=p_1 p_2 p_3$ square-free with $V^3<n$,
we have 
\[F_{3,V}(n) = \begin{cases} - 6 \log p_1 \log p_2 \log p_3 & 
\text{if all $p_i>V$,}\\
0  & \text{if $p_1 < p_2 \leq V < p_3$,}\\
6 \log p_1 \log p_2 \log p_3 & \text{if $p_1 \leq V < p_2 < p_3$,}\\
12 \log p_1 \log p_2 \log p_3 & \text{if all $p_i\leq V$.}\end{cases}\]
In contrast, for $n$ square-free,
 $-\Lambda\ast\Lambda\ast\Lambda(n)$ is $-6$ if $n$ is of the form
$p_1 p_2 p_3$, and $0$ otherwise.

We may find it useful to take aside two large terms that may need
to be bounded trivially, namely, $\mu\ast\log_{\leq u}^3$ and
$(\Lambda\cdot \log)_{\leq u} \ast \Lambda_{>V}$, where $u$ will be a small
parameter. (We can let, for instance, $u=3$.) We conclude that
\begin{equation}\label{eq:sult}
\Lambda\cdot \log^2 = F_{I,1,u}(n) - 3 F_{I,2,V,u}(n) 
- 3 F_{II,V,u}(n) + F_{3,V}(n) + 
F_{0,V,u}(n),\end{equation}
where 
\[\begin{aligned}
F_{I,1,u} &= \mu\ast \log_{>u}^3,\\
F_{I,2,V,u} &= (\mu\ast \log^2) \ast \Lambda_{\leq V},\\
F_{II,V,u}(n) &= (\Lambda\cdot\log)_{>u}\ast \Lambda_{>V},\\
F_{0,V,u}(n) &= \mu\ast\log_{\leq u}^3 - 3
(\Lambda\cdot \log)_{\leq u} \ast \Lambda_{>V}\\
\end{aligned}\]
and $F_{3,V}$ is as in (\ref{eq:f3V}). 

In the bulk of the present work -- in particular, in all steps that are
part of the proof of Theorem \ref{thm:minmain} or the Main Theorem --
we will use Vaughan's identity, rather than (\ref{eq:sult}). This choice was
made while the proof was still underway; it was due mainly to 
back-of-the-envelope estimates that showed that the error terms could be
too large if (\ref{eq:saidso}) was used. Of course, this might have been
the case with Vaughan's identity as well, but the fact that the parameters
$U$, $V$ there have a large effect on the outcome meant that one could
hope to improve on insufficient estimates in part by adjusting $U$ and $V$,
without losing all previous work. (This is what was meant by
the ``flexibility'' of Vaughan's identity.)

The question remains: can one prove ternary Goldbach using (\ref{eq:sult})
rather than Vaughan's identity? This seems likely.
If so, which proof would be more complicated? This is not clear.


There are large parts of the work that are the essentially the
same in both cases:
\begin{itemize}
\item estimates for sums involving $\mu_{\leq U}\ast \log^k$ (``type I''),
\item estimates for sums involving $\Lambda_{>u}\ast \Lambda_{>V}$ and
the like (``type II'').
\end{itemize}
Trilinear sums, i.e., sums
 involving $\Lambda\ast \Lambda\ast \Lambda$, can be estimated much like
bilinear sums, i.e., sums involving $\Lambda\ast \Lambda$. 

There are also challenges that appear only for Vaughan's identity and
others that appear only for (\ref{eq:sult}). An example of a challenge that is
successfully faced in the main proof, but does not appear if
(\ref{eq:sult}) is used, consists in bounding sums of type
\[\sum_{U<m\leq x/W} \left(\mathop{\sum_{d>U}}_{d|m} \mu(d) \right)^2.\]
(In \S \ref{subs:vaucanc}, we will be able to bound sums of this type by
a constant times $x/W$.) 
Likewise, large
tail terms that have to be estimated trivially seem unavoidable in 
(\ref{eq:sult}). (The choice of a parameter $u>1$, as above, is meant to
alleviate the problem.)


In the end, losing a factor of
about $\log x/UV$ seems inevitable when one uses Vaughan's identity, but not
when one uses (\ref{eq:sult}). Another reason why a 
 full treatment based on (\ref{eq:sult})
would also be worthwhile is that it is a somewhat less
familiar, and arguably under-used, identity and deserves more exploration. 
With these comments, we close
the discussion of
(\ref{eq:sult}); we will henceforth use Vaughan's identity.


\chapter{Type I sums}\label{sec:typeI}
Here, we must bound sums of the basic type
\[\sum_{m\leq D} \mu(m) \sum_n e(\alpha m n) \eta\left(\frac{m n}{x}\right)\]
and variations thereof.
There are three main improvements
in comparison to standard treatments: 
\begin{enumerate}
\item\label{it:sept} The terms with $m$ divisible by $q$ get taken out and treated separately by analytic means. This all but eliminates what would otherwise be the main term.
\item\label{it:sept2} The other terms get handled by improved estimates on
trigonometric sums. For large $m$, the improvements have a substantial total
effect -- more than a constant factor is gained.
\item\label{it:sept3}
The ``error'' term $\delta/x = \alpha - a/q$ is used to our advantage. This happens both
through the Poisson summation formula and through the use of two alternative
approximations to the same number $\alpha$.
\end{enumerate}
The fact that a continuous weight $\eta$ is used (``smoothing'') 
is a difference with respect to
the classical literature (\cite{zbMATH02522879} and what followed), 
but not with respect to more recent work (including \cite{Tao}); using
smooth or continuous weights is an idea that has become commonplace in
analytic number theory, even though it is not consistently applied. The 
improvements due to smoothing in type I are both relatively minor and
essentially independent of the improvements due to (\ref{it:sept}) and
(\ref{it:sept3}). The use of a continuous weight combines nicely with
(\ref{it:sept2}), but the ideas given here would give qualitative
improvements in the treatment of trigonometric sums even in the absence of
smoothing.

\section{Trigonometric sums}

The following lemmas on trigonometric sums improve on the best
Vinogradov-type lemmas in the literature. (By this, we mean results
 of the type of 
Lemma 8a and Lemma 8b in \cite[Ch. I]{MR2104806}. See, in particular,
the work of Daboussi and Rivat \cite[Lemma 1]{MR1803131}.) The main
idea is to switch between different types of approximation within the sum,
rather than just choosing between bounding all terms either trivially 
(by $A$) or non-trivially (by $C/|\sin(\pi \alpha n)|^2$). There will
also\footnote{This is a change with respect to the first version of
the preprint \cite{Helf}. The version of Lemma \ref{lem:gotog} 
there has, however, the advantage
of being immediately comparable to results in the literature.}
be improvements in our applications stemming from the fact that Lemmas
\ref{lem:gotog} and Lemma \ref{lem:couscous}  take quadratic 
($|\sin(\pi \alpha n)|^2$) rather than linear ($|\sin(\pi \alpha n)|$)
inputs. (These improved inputs come from the use of smoothing elsewhere.)

\begin{lemma}\label{lem:gotog}
Let $\alpha = a/q + \beta/q Q$, $(a,q)=1$, $|\beta|\leq 1$, $q\leq Q$.
Then, for any $A, C\geq 0$,
\begin{equation}\label{eq:betblu}
\sum_{y<n\leq y+q} \min\left(A, 
\frac{C}{|\sin (\pi \alpha n)|^2}\right)\leq 
\min\left(2 A + \frac{6 q^2}{\pi^2} C,
3 A + \frac{4q}{\pi} \sqrt{A C}\right)
.\end{equation}
\end{lemma}
\begin{proof}
We start by letting
$m_0 = \lfloor y \rfloor + \lfloor (q+1)/2\rfloor$,
$j = n-m_0$, so that $j$ ranges in the interval $(-q/2,q/2\rbrack$.
We write
\[\alpha n = \frac{aj + c}{q} + \delta_1(j) + \delta_2 \mo 1,\]
where $|\delta_1(j)|$ and $|\delta_2|$ are both $\leq 1/2q$; we can assume
$\delta_2\geq 0$. The variable $r = aj+c \mo q$ occupies each residue class
$\mo p$ 
exactly once. 

One option is to bound the terms corresponding to $r=0, -1$ by $A$ each
and all the other terms by $C/|\sin(\pi \alpha n)|^2$. (This can be
seen as the simple case; it will take us about a page just because we should
estimate all sums and all terms here with great care -- as in 
\cite{MR1803131}, only more so.)

The terms 
corresponding to $r=-k$ and $r=k-1$ ($2\leq k\leq q/2$) contribute at most
\[\frac{1}{\sin^2 \frac{\pi}{q} (k - \frac{1}{2} - q \delta_2)}
+  \frac{1}{\sin^2 \frac{\pi}{q} (k - \frac{3}{2} + q \delta_2)}
\leq \frac{1}{\sin^2 \frac{\pi}{q} \left(k-\frac{1}{2}\right)} +
\frac{1}{\sin^2 \frac{\pi}{q} \left(k-\frac{3}{2}\right)},\]
since $x\mapsto \frac{1}{(\sin x)^2}$ is convex-up on $(0,\infty)$.
Hence the terms with $r\ne 0, 1$ contribute at most
\[\frac{1}{\left(\sin \frac{\pi}{2q}\right)^2} + 
2 \sum_{2\leq r \leq \frac{q}{2}} \frac{1}{\left(\sin \frac{\pi}{q} (r-1/2)
\right)^2} \leq \frac{1}{\left(\sin \frac{\pi}{2q}\right)^2} + 
2 \int_1^{q/2} \frac{1}{\left(\sin \frac{\pi}{q} x \right)^2} ,\]
where we use again the convexity of $x\mapsto 1/(\sin x)^2$. (We can 
assume $q>2$, as otherwise we have no terms other than $r=0,1$.)
Now
\[\int_1^{q/2} \frac{1}{\left(\sin \frac{\pi}{q} x\right)^2} dx = 
\frac{q}{\pi} \int_{\frac{\pi}{q}}^{\frac{\pi}{2}} \frac{1}{(\sin u)^2} du =
\frac{q}{\pi} \cot \frac{\pi}{q} .\] Hence
\[\sum_{y<n\leq y+q} \min\left(A, \frac{C}{(\sin \pi \alpha n)^2}\right)
\leq 2 A + \frac{C}{\left(\sin \frac{\pi}{2q}\right)^2} + 
C \cdot \frac{2 q}{\pi} \cot \frac{\pi}{q} .\]
Now, by \cite[(4.3.68)]{MR0167642} and \cite[(4.3.70)]{MR0167642},
for $t\in (-\pi,\pi)$,
\begin{equation}\label{eq:pordo1}\begin{aligned}
\frac{t}{\sin t} &= 1 + \sum_{k\geq 0} a_{2k+1} t^{2k+2} = 
1 + \frac{t^2}{6} + \dotsc\\
t \cot t &= 1 - \sum_{k\geq 0} b_{2k+1} t^{2k+2} = 1 - \frac{t^2}{3}
- \frac{t^4}{45} - \dotsc,
\end{aligned}\end{equation}
where $a_{2k+1}\geq 0$, $b_{2k+1}\geq 0$. Thus, for $t\in \lbrack 0,t_0\rbrack$,
$t_0<\pi$,
\begin{equation}\label{eq:pordo2}
\left(\frac{t}{\sin t}\right)^2 = 1 + \frac{t^2}{3} + c_0(t) t^4
\leq 1 + \frac{t^2}{3} + 
c_0(t_0) t^4,\end{equation} where
\[c_0(t) = \frac{1}{t^4} \left(\left(\frac{t}{\sin t}\right)^2 - 
\left(1 + \frac{t^2}{3}\right)\right),\]
which is an increasing function because $a_{2k+1}\geq 0$.
For $t_0=\pi/4$, $c_0(t_0) \leq 0.074807$. Hence,
\[\begin{aligned}\frac{t^2}{\sin^2 t} + t \cot 2 t &\leq
\left(1 + \frac{t^2}{3} + c_0\left(\frac{\pi}{4}\right)
 t^4\right) + \left(\frac{1}{2} - 
\frac{2 t^2}{3} - \frac{8 t^4}{45}\right)\\ &= \frac{3}{2} - \frac{t^2}{3} 
+\left(c_0\left(\frac{\pi}{4}\right) - \frac{8}{45}\right) t^4 \leq \frac{3}{2} - \frac{t^2}{3}
\leq \frac{3}{2}\end{aligned}\]
for $t\in \lbrack 0, \pi/4\rbrack$.

Therefore, the left
side of (\ref{eq:betblu}) is at most
\[2 A + C \cdot \left(\frac{2 q}{\pi}\right)^2 \cdot \frac{3}{2} = 2 A + \frac{6}{\pi^2}
C q^2 .\]
 
The following is an alternative approach; it yields the other estimate in
(\ref{eq:betblu}).
We bound the terms corresponding to $r=0$, $r=-1$, $r=1$ 
by $A$ each. We let $r = \pm r'$ for $r'$ ranging from $2$ to $q/2$.
We obtain that the sum is at most
\begin{equation}\label{eq:cloison}
\begin{aligned}
3 A &+ \sum_{2\leq r'\leq q/2} 
 \min\left(A,\frac{C}{\left(\sin \frac{\pi}{q} 
\left(r' - \frac{1}{2} - q\delta_2 \right)\right)^2}\right) \\
&+ \sum_{2\leq r'\leq q/2} 
 \min\left(A,\frac{C}{\left(\sin \frac{\pi}{q} 
\left(r' - \frac{1}{2} + q\delta_2 \right)\right)^2}\right).\end{aligned}
\end{equation}

We bound a term $\min(A,C/\sin((\pi/q) (r'-1/2\pm q\delta_2))^2)$ by $A$ if
and only if $C/\sin((\pi/q) (r'-1\pm q \delta_2))^2 \geq A$. 
(In other words, we are choosing which of the two bounds $A$, 
$C/|\sin(\pi \alpha n)|^2$ on a case-by-case basis, i.e., for each $n$,
instead of making a single choice for all $n$ in one go. This is hardly
anything deep, but it does result in a marked improvement with respect to
the literature, and would give an improvement even if we were given a bound
$B/|\sin(\pi \alpha n)|$ instead of a bound
$C/|\sin(\pi \alpha n)|^2$ as input.)
The number
of such terms is \[\leq \max(0,\lfloor
(q/\pi) \arcsin(\sqrt{C/A}) \mp q\delta_2\rfloor),\] and thus at most
$(2 q/\pi) \arcsin(\sqrt{C/A})$ in total. (Recall that $q\delta_2\leq 1/2$.)
Each other term gets bounded
by the integral of $C/\sin^2(\pi\alpha/q)$ from $r'-1\pm q\delta_2$
($\geq (q/\pi) \arcsin(\sqrt{C/A})$) to $r'\pm q\delta_2$, by convexity.
Thus (\ref{eq:cloison}) is at most
\[\begin{aligned}
3 A &+ \frac{2q}{\pi} A \arcsin \sqrt{\frac{C}{A}} +
2 \int_{\frac{q}{\pi}  \arcsin \sqrt{\frac{C}{A}}}^{q/2}\; \frac{C}{\sin^2 \frac{\pi t}{q}}
dt\\
&\leq 3 A + \frac{2q}{\pi} A \arcsin \sqrt{\frac{C}{A}} +
\frac{2 q}{\pi} C \sqrt{\frac{A}{C}-1} \end{aligned}\]

We can easily show (taking derivatives) that
$\arcsin x +
x (1-x^2) \leq 2 x$ for $0\leq x\leq 1$. Setting $x = C/A$, we see
that this implies that
\[3 A + \frac{2q}{\pi} A \arcsin \sqrt{\frac{C}{A}} +
\frac{2 q}{\pi} C \sqrt{\frac{A}{C}-1} \leq
 3 A + \frac{4q}{\pi} \sqrt{A C}.\]
(If $C/A>1$, then $3 A + (4q/\pi) \sqrt{A C}$ is greater than
$A q$, which is an obvious upper bound for the left side of (\ref{eq:betblu}).)
\end{proof}

Now we will see that, if we take out terms with $n$ divisible by $q$ and $n$
is not too large, then we can give a bound that does not involve a constant
term $A$ at all. (We are referring to the bound $(20/3\pi^2) C q^2$ below; of course,
$2A + (4q/\pi) \sqrt{AC}$ does have a constant term $2A$ -- it is just
smaller than the constant term $3 A$ in the corresponding bound in
(\ref{eq:betblu}).)
\begin{lemma}\label{lem:couscous}
Let $\alpha = a/q + \beta/q Q$, $(a,q)=1$, $|\beta|\leq 1$, $q\leq Q$.
Let $y_2>y_1\geq 0$. If $y_2-y_1\leq q$ and $y_2\leq Q/2$, then, for any
$A, C \geq 0$,
\begin{equation}\label{eq:dijkre}
\mathop{\sum_{y_1 < n \leq y_2}}_{q\nmid n} 
  \min\left(A,
\frac{C}{|\sin (\pi \alpha n)|^2}\right) \leq 
\min\left(\frac{20}{3 \pi^2} C q^2, 2 A + \frac{4 q}{\pi} \sqrt{A C}\right).
\end{equation}
\end{lemma}
\begin{proof}
Clearly, $\alpha n$ equals $an/q + (n/Q) \beta/q$; since
$y_2\leq Q/2$, this means that $|\alpha n - an/q| \leq 1/2q$
for $n\leq y_2$; moreover, again for $n\leq y_2$, the sign of
$\alpha n - an/q$ remains constant. Hence the left side of
(\ref{eq:dijkre}) is at most
\[\begin{aligned}
\sum_{r=1}^{q/2} \min\left(A, \frac{C}{(\sin \frac{\pi}{q}
    (r-1/2))^2}\right)
+ \sum_{r=1}^{q/2} \min\left(A, \frac{C}{(\sin \frac{\pi}{q} r)^2}\right)
.
\end{aligned}\]
Proceeding as in the
proof of Lemma \ref{lem:gotog}, we obtain a bound of at most
\[
C \left(\frac{1}{(\sin \frac{\pi}{2q})^2}  + \frac{1}{(\sin \frac{\pi}{q})^2} +
\frac{q}{\pi} \cot \frac{\pi}{q} + \frac{q}{\pi} \cot \frac{3 \pi}{2 q}\right)\]
for $q\geq 2$. (If $q=1$, then the left-side of (\ref{eq:dijkre}) is
trivially zero.) Now, by (\ref{eq:pordo1}),
\[\begin{aligned}
\frac{t^2}{(\sin t)^2} + \frac{t}{2} \cot 2t &\leq
\left(1 + \frac{t^2}{3} + c_0\left(\frac{\pi}{4}\right) t^4\right) + 
\frac{1}{4} \left(1 - \frac{4 t^2}{3} - \frac{16 t^4}{45}\right)\\
&\leq \frac{5}{4} + \left(c_0\left(\frac{\pi}{4}\right) - \frac{4}{45}\right)
t^4  \leq \frac{5}{4}
\end{aligned}\]
for $t\in \lbrack 0,\pi/4\rbrack$, and
\[\begin{aligned}
\frac{t^2}{(\sin t)^2} + t \cot \frac{3 t}{2} &\leq
\left(1 + \frac{t^2}{3} + c_0\left(\frac{\pi}{2}\right) t^4\right) + 
\frac{2}{3} \left(1 - \frac{3 t^2}{4} - \frac{81 t^4}{2^4\cdot 45}\right)\\
&\leq \frac{5}{3} + \left(- \frac{1}{6} + \left( c_0\left(\frac{\pi}{2}\right)
- \frac{27}{360}\right) \left(\frac{\pi}{2}\right)^2\right) t^2 \leq \frac{5}{3}
\end{aligned}\]
for $t\in \lbrack 0,\pi/2\rbrack$. Hence,
\[
\left(\frac{1}{(\sin \frac{\pi}{2q})^2}  + \frac{1}{(\sin \frac{\pi}{q})^2} +
\frac{q}{\pi} \cot \frac{\pi}{q} + \frac{q}{\pi} \cot \frac{3 \pi}{2 q}\right)
\leq \left(\frac{2 q}{\pi}\right)^2 \cdot \frac{5}{4} + \left(\frac{q}{\pi}
\right)^2\cdot \frac{5}{3} \leq
\frac{20}{3 \pi^2} q^2.\]
Alternatively, we can follow the second approach in the proof of Lemma
\ref{lem:gotog}, and obtain an upper bound of $2 A + (4q/\pi) \sqrt{AC}$.

\end{proof}

The following bound will be useful when the constant $A$ in an application
of Lemma \ref{lem:couscous} would be too large. (This tends to happen for
$n$ small.)
\begin{lemma}\label{lem:thina}
Let $\alpha = a/q + \beta/q Q$, $(a,q)=1$, $|\beta|\leq 1$, $q\leq Q$.
Let $y_2>y_1\geq 0$. If $y_2-y_1\leq q$ and $y_2\leq Q/2$, then, for
any $B,C\geq 0$,
\begin{equation}\label{eq:shtru}
\mathop{\sum_{y_1<n\leq y_2}}_{q\nmid n} 
\min\left(\frac{B}{|\sin(\pi \alpha n)|},\frac{C}{|\sin(\pi \alpha n)|^2}
\right) \leq 2 B \frac{q}{\pi} \max\left(2, \log \frac{C e^3 q}{B \pi}\right)
.
 \end{equation}
The upper bound $\leq (2 B q/\pi) \log (2 e^2 q/\pi)$ is also valid.
\end{lemma}
\begin{proof}
As in the proof of Lemma \ref{lem:couscous}, we can bound the left side
of (\ref{eq:shtru}) by
\[
2 \sum_{r=1}^{q/2} \min\left(\frac{B}{
\sin \frac{\pi}{q} \left(r - \frac{1}{2}\right)}, 
\frac{C}{\sin^2 \frac{\pi}{q} \left(r - \frac{1}{2}\right)}\right)
.\]
Assume $B \sin(\pi/q) \leq C\leq B$. 
By the convexity of $1/\sin(t)$ and $1/\sin(t)^2$ for $t\in (0,\pi/2\rbrack$,
\[\begin{aligned}
\sum_{r=1}^{q/2} &\min\left(\frac{B}{
\sin \frac{\pi}{q} \left(r - \frac{1}{2}\right)}, 
\frac{C}{\sin^2 \frac{\pi}{q} \left(r - \frac{1}{2}\right)}\right) \\
&\leq \frac{B}{\sin \frac{\pi}{2 q}} + 
\int_1^{\frac{q}{\pi} \arcsin \frac{C}{B}} \frac{B}{\sin \frac{\pi}{q} t} 
dt + \int_{\frac{q}{\pi} \arcsin \frac{C}{B}}^{q/2} \frac{1}{\sin^2 \frac{\pi}{q} t}
dt \\ &\leq \frac{B}{\sin \frac{\pi}{2 q}} + 
\frac{q}{\pi} \left(B \left(\log \tan \left(\frac{1}{2} \arcsin \frac{C}{B}\right) -
\log \tan \frac{\pi}{2 q}\right) + C \cot \arcsin \frac{C}{B}\right)\\
&\leq \frac{B}{\sin \frac{\pi}{2 q}} + 
\frac{q}{\pi} \left(B \left(\log \cot \frac{\pi}{2 q} - \log \frac{C}{B-
\sqrt{B^2 - C^2}}\right) + \sqrt{B^2 - C^2}\right).
\end{aligned}\]

Now, for all $t\in (0,\pi/2)$,
\[\frac{2}{\sin t} + \frac{1}{t} \log \cot t 
< \frac{1}{t} \log\left(\frac{e^2}{t}\right);\]
we can verify this by comparing series. Thus
\[\frac{B}{\sin \frac{\pi}{2q}} + \frac{q}{\pi} B \log \cot \frac{\pi}{2 q}
\leq B \frac{q}{\pi} \log \frac{2 e^2 q}{\pi}\]
for $q\geq 2$.  (If $q=1$, the sum on the left of (\ref{eq:shtru}) is empty,
and so the bound we are trying to prove is trivial.)
We also have
\begin{equation}\label{eq:somos}
t \log(t - \sqrt{t^2-1}) + \sqrt{t^2 - 1} < -t \log 2t + t\end{equation}
for $t\geq 1$ (as this is equivalent to $\log(2 t^2 (1 - \sqrt{1 - t^{-2}}))
< 1 - \sqrt{1 - t^{-2}}$, which we check easily after changing variables to
$\delta = 1 - \sqrt{1 - t^{-2}}$). Hence
\[\begin{aligned}
\frac{B}{\sin \frac{\pi}{2 q}} &+ 
\frac{q}{\pi} \left(B \left(\log \cot \frac{\pi}{2 q} - \log \frac{C}{B-
\sqrt{B^2 - C^2}}\right) + \sqrt{B^2 - C^2}\right)\\ &\leq
B \frac{q}{\pi} \log \frac{2 e^2 q}{\pi} + \frac{q}{\pi} \left(B - B \log \frac{2 B}{C}
\right)
\leq B \frac{q}{\pi} \log \frac{C e^3 q}{B \pi}\end{aligned}\]
for $q\geq 2$.

Given any $C$, we can apply the above with $C=B$ instead, as, for any
$t>0$, $\min(B/t, C/t^2) \leq B/t \leq \min(B/t, B/t^2)$. 
(We refrain from applying (\ref{eq:somos}) so as to avoid worsening 
a constant.)
If $C < B \sin \pi/q$ (or even if $C < (\pi/q) B$), we relax the input to
$C = B \sin \pi/q$ and go through the above.
\end{proof}

\section{Type I estimates}

Let us give our first main type I estimate.\footnote{The current version of
Lemma \ref{lem:bosta1} is an improvement over that included in the first
version of the preprint \cite{Helf}.} One of the main innovations is the 
manner 
in which the ``main term'' ($m$ divisible by $q$) is separated; we are able
to keep error terms small thanks to the particular way in which we switch 
between two different approximations.

(These are {\em not} necessarily successive approximations in the sense
of continued fractions; we do not want to assume that the approximation
$a/q$ we are given arises from a continued fraction, and at any rate
we need more control on the denominator $q'$ of the new approximation 
$a'/q'$ than continued fractions would furnish.)

The following lemma is a theme, so to speak, to which several variations
will be given. Later, in practice, we will always use one
of the variations, rather than the original lemma itself. This is so just
because, even though (\ref{eq:gorio}) is the basic type of sum we treat in
type I, the sums that we will have to estimate in practice will always 
present some minor additional complication. Proving the lemma we are about
to give in full will give us a chance to see all the main ideas at work,
leaving complications for later.
 
\begin{lemma}\label{lem:bosta1}
Let $\alpha = a/q + \delta/x$, $(a,q)=1$, $|\delta/x|\leq 1/q Q_0$,
$q\leq Q_0$, $Q_0\geq 16$.
 Let $\eta$ be continuous, piecewise $C^2$ and compactly supported, with
$|\eta|_1 = 1$ and $\eta''\in L_1$. Let $c_0 \geq |\widehat{\eta''}|_\infty$.

Let $1\leq D\leq x$. Then, if $|\delta|\leq 1/2c_2$, where
$c_2 = (3 \pi/5\sqrt{c_0}) (1+\sqrt{13/3})$, the absolute value of 
\begin{equation}\label{eq:gorio}
\sum_{m\leq D} \mu(m) \sum_n e(\alpha m n) \eta\left(\frac{m n}{x}\right)
\end{equation} is at most
\begin{equation}\label{eq:cupcake}
\frac{x}{q} \min\left(1,\frac{c_0}{(2\pi \delta)^2}\right)
\left|
\mathop{\sum_{m\leq \frac{M}{q}}}_{(m,q)=1} \frac{\mu(m)}{m} \right| + 
O^*\left(c_0 
\left(\frac{1}{4} - \frac{1}{\pi^2}
\right)
\left(\frac{D^2}{2 x q} + \frac{D}{2x} \right)\right)
\end{equation}
plus
\begin{equation}\label{eq:kuche1}\begin{aligned} 
&\frac{2 \sqrt{c_0 c_1}}{\pi} D +
3 c_1 \frac{x}{q} \log^+ \frac{D}{c_2 x/q} 
+ \frac{\sqrt{c_0 c_1}}{\pi} q \log^+ \frac{D}{q/2}\\
&+ \frac{|\eta'|_1}{\pi} q \cdot \max\left(2, \log \frac{c_0 e^3 q^2}{4 \pi |\eta'|_1 x}\right) +
 \left(\frac{2 \sqrt{3 c_0 c_1}}{\pi} + 
\frac{3 c_1}{c_2}
+ \frac{55 c_0 c_2}{12 \pi^2} \right) q
 ,\end{aligned}\end{equation}
where $c_1 = 1 + |\eta'|_1/(2 x/D)$
 and $M\in \lbrack \min(Q_0/2,D),D\rbrack$.
The same bound holds if $|\delta|\geq 1/2c_2$ but $D\leq Q_0/2$.

In general, if $|\delta|\geq 1/2c_2$, the absolute value of
(\ref{eq:gorio}) is at most (\ref{eq:cupcake}) plus
\begin{equation}\label{eq:kallervo}\begin{aligned}
&\frac{2 \sqrt{c_0 c_1}}{\pi} \left(
D  + (1+\epsilon) \min\left(\left\lfloor \frac{x}{|\delta| q}\right\rfloor + 1, 2 D\right)
 \left(\varpi_\epsilon +
 \frac{1}{2} \log^+ \frac{2 D}{
\frac{x}{|\delta| q}}\right)\right)\\
&+ 3 c_1
 \left(2 + \frac{(1+\epsilon)}{\epsilon} \log^+ \frac{2D}{
\frac{x}{|\delta| q}}\right) \frac{x}{Q_0} +
\frac{35 c_0 c_2}{6 \pi^2} q,
\end{aligned}\end{equation}
for $\epsilon\in (0,1\rbrack$ arbitrary, where $\varpi_\epsilon = 
\sqrt{3+2\epsilon} + ((1+\sqrt{13/3})/4-1)/(2 (1+\epsilon))$.
\end{lemma}
In (\ref{eq:cupcake}), $\min(1,c_0/(2\pi \delta)^2)$ always equals $1$ when
$|\delta|\leq 1/2c_2$ (since $(3/5) (1 + \sqrt{13/3}) > 1$).
\begin{proof}
Let $Q = \lfloor x/|\delta q|\rfloor$. Then
$\alpha = a/q + O^*(1/q Q)$ and $q\leq Q$. (If $\delta=0$, we 
let $Q=\infty$ and ignore the rest of the paragraph, since then we will never
need $Q'$ or the alternative approximation $a'/q'$.) Let
$Q' = \lceil (1+\epsilon) Q\rceil \geq Q+1$. 
Then $\alpha$ is {\em not} $a/q + O^*(1/q Q')$, and so
there must be a different approximation $a'/q'$, $(a',q')=1$, $q'\leq Q'$
such that $\alpha = a'/q' + O^*(1/q' Q')$ (since such an approximation
always exists).
Obviously, $|a/q - a'/q'|\geq 1/q q'$, yet, at the same time,
$|a/q- a'/q'|\leq 1/q Q + 1/q' Q' \leq 1/q Q + 1/((1+\epsilon) q' Q)$. 
Hence $q'/Q + q/((1+\epsilon)Q) \geq 1$, and so $q'\geq Q-q/(1+\epsilon)
\geq (\epsilon/(1+\epsilon)) Q$.
(Note also that $(\epsilon/(1+\epsilon)) Q \geq (2|\delta q|/x)\cdot
\lfloor x/\delta q\rfloor>1$, and so $q'\geq 2$.)

Lemma \ref{lem:couscous} will enable us to treat separately the
contribution from terms with $m$ divisible by $q$ and $m$ not divisible by
$q$, provided that $m\leq Q/2$. Let $M = \min(Q/2,D)$.
We start by considering all terms with $m\leq M$ divisible by $q$.
Then $e(\alpha m n)$ equals $e((\delta m/x) n)$. By Poisson summation,
\[\sum_n e(\alpha m n) \eta(m n/x) = \sum_n \widehat{f}(n),\]
where $f(u) = e((\delta m/x) u) \eta((m/x) u)$. Now
\[\widehat{f}(n) = \int e(-un) f(u) du = \frac{x}{m} \int
e\left(\left(\delta-\frac{x n}{m}\right) u\right) \eta(u) du = 
\frac{x}{m} \widehat{\eta}\left(\frac{x}{m} n -\delta\right).\]
By assumption, $m\leq M \leq 
Q/2\leq x/2|\delta q|$, and so $|x/m|\geq 2 |\delta q| \geq
2 \delta$. Thus, by (\ref{eq:madge}) (with $k=2$),
\begin{equation}\label{eq:crinfo}\begin{aligned}
\sum_n \widehat{f}(n) &= \frac{x}{m} \left(\widehat{\eta}(-\delta) + 
 \sum_{n\neq 0} \widehat{\eta}\left(\frac{n x}{m} - \delta\right)\right)\\
&= \frac{x}{m} \left(\widehat{\eta}(-\delta) + O^*\left(\sum_{n\neq 0}
 \frac{1}{\left(2\pi \left(\frac{n x}{m} - \delta\right)\right)^2}\right)
\cdot \left|\widehat{\eta''}\right|_\infty\right)\\
&= 
\frac{x}{m} \widehat{\eta}(-\delta) + \frac{m}{x} \frac{c_0}{(2\pi)^2}
O^*\left(\max_{|r|\leq \frac{1}{2}} \sum_{n\neq 0} \frac{1}{(n-r)^2}\right) .
\end{aligned}\end{equation}
Since $x\mapsto 1/x^2$ is convex on $\mathbb{R}^+$, 
\[\max_{|r|\leq \frac{1}{2}}
\sum_{n\neq 0} \frac{1}{(n-r)^2} = \sum_{n\neq 0} \frac{1}{
\left(n-\frac{1}{2}\right)^2} = 
\pi^2 - 4.\]

Therefore, the sum of all terms with $m\leq M$
and $q|m$ is
\[\begin{aligned}
&\mathop{\sum_{m\leq M}}_{q|m} \frac{x}{m} \widehat{\eta}(-\delta) +
\mathop{\sum_{m\leq M}}_{q|m} \frac{m}{x} 
\frac{c_0}{(2 \pi)^2} (\pi^2 - 4) \\
&=\frac{x \mu(q)}{q}\cdot
\widehat{\eta}(-\delta) \cdot 
\mathop{\sum_{m\leq \frac{M}{q}}}_{(m,q)=1} \frac{\mu(m)}{m} \\
&+ 
O^*\left(\mu(q)^2 c_0 
\left(\frac{1}{4} - \frac{1}{\pi^2}
\right)
\left(\frac{D^2}{2 x q} + \frac{D}{2 x} \right)\right).
\end{aligned}.\]
We will bound $|\widehat{\eta}(-\delta)|$ by (\ref{eq:madge}).
 
As we have just seen, estimating the contribution of the terms with $m$
divisible by $q$ and not too large ($m\leq M$) involves isolating
a main term, estimating it carefully (with cancellation) and then bounding
the remaining error terms. 

We will now bound the contribution of
all other $m$ -- that is, $m$ not divisible by $q$ and $m$ larger than $M$. Cancellation will now be used only within the inner sum; that is, we
will bound each inner sum
\[T_m(\alpha) = \sum_n e(\alpha m n) \eta\left(\frac{m n}{x}\right),\]
and then we will carefully consider how to bound sums of $|T_m(\alpha)|$ over
$m$ efficiently.

By (\ref{eq:ra}) and Lemma \ref{lem:areval},
\begin{equation}\label{eq:aoro}\begin{aligned}
|T_m(\alpha)| \leq \min\left(\frac{x}{m} + \frac{1}{2} |\eta'|_1,
\frac{\frac{1}{2} |\eta'|_1 }{|\sin(\pi m \alpha)|},
\frac{m}{x} \frac{c_0}{4} \frac{1}{(\sin \pi m \alpha)^2}\right).
\end{aligned}\end{equation}

For any $y_2>y_1>0$ with $y_2-y_1\leq q$ and $y_2\leq Q/2$,
(\ref{eq:aoro}) gives us that
\begin{equation}\label{eq:gowo}
\mathop{\sum_{y_1<m\leq y_2}}_{q\nmid m} |T_m(\alpha)|\leq
\mathop{\sum_{y_1<m\leq y_2}}_{q\nmid m} \min\left(A, \frac{C}{
(\sin \pi m \alpha)^2}\right)\end{equation}
for $A = (x/y_1) ( 1 + |\eta'|_1/(2(x/y_1)))$ and $C = (c_0/4) (y_2/x)$.
We must now estimate the sum
\begin{equation}\label{eq:esthel}
\mathop{\sum_{m\leq M}}_{q\nmid m} 
|T_m(\alpha)| + \sum_{\frac{Q}{2} < m \leq D} |T_m(\alpha)|.
\end{equation}


To bound the terms with $m\leq M$, we can use Lemma \ref{lem:couscous}.
The question is then which one is smaller: the first or the second
bound given by Lemma \ref{lem:couscous}? A brief calculation gives that
the second bound is smaller (and hence preferable) exactly when 
$\sqrt{C/A} > (3 \pi/10 q) (1 + \sqrt{13/3})$. Since $\sqrt{C/A} \sim
(\sqrt{c_0}/2) m/x$, this means that it is sensible to prefer the
 second bound in Lemma \ref{lem:couscous} when $m> c_2 x /q$,
where $c_2 = (3 \pi/ 5\sqrt{c_0}) (1+\sqrt{13/3})$.

It thus makes sense to ask: does $Q/2\leq c_2 x/q$ (so that $m\leq M$
implies $m\leq c_2 x/q$)? This question divides our work into two
basic cases.

Case (a). {\em $\delta$ large:  $|\delta|\geq 1/2c_2$,
where $c_2 = (3 \pi/5\sqrt{c_0}) (1+\sqrt{13/3})$.}
Then $Q/2 \leq c_2 x/q$; this will induce us to 
bound the first sum in (\ref{eq:esthel})
by the first bound in Lemma \ref{lem:couscous}. 

Recall that $M = \min(Q/2,D)$, and so $M\leq c_2 x/q$.
By (\ref{eq:gowo}) and Lemma \ref{lem:couscous}, 
\begin{equation}\label{eq:jenuf}\begin{aligned}
\mathop{\sum_{1 \leq m \leq M}}_{q\nmid m} &|T_m(\alpha)| \leq
\sum_{j=0}^{\infty} \mathop{\sum_{j q < m\leq \min((j+1) q ,
    M)}}_{q\nmid m}
\min\left( 
\frac{x}{jq+1} +\frac{|\eta'|_1}{2}, 
\frac{\frac{c_0}{4} \frac{(j+1) q}{x}}{(\sin \pi m \alpha)^2}\right) \\
&\leq \frac{20}{3 \pi^2} \frac{c_0 q^3}{4 x} 
\sum_{0\leq j\leq \frac{M}{q}} (j+1) \leq \frac{20}{3 \pi^2} \frac{c_0 q^3}{4 x} \cdot
\left(\frac{1}{2} \frac{M^2}{q^2}+ 
\frac{3}{2} \frac{c_2 x}{q^2} + 1\right) \\ &\leq
 \frac{5 c_0 c_2}{6  \pi^2} M + 
\frac{5 c_0 q}{3 \pi^2}  \left(\frac{3}{2} c_2 + 
\frac{q^2}{x}\right) \leq
 \frac{5 c_0 c_2}{6  \pi^2} M + 
\frac{35 c_0 c_2}{6 \pi^2} q,
\end{aligned}\end{equation}
where, to bound the smaller terms, we are using the inequality 
$Q/2\leq c_2 x/q$, and where we are also 
using the observation that, since $|\delta/x| \leq 1/q Q_0$, 
the assumption $|\delta|\geq 1/2c_2$ implies that $q\leq 2 c_2 x/Q_0$;
moreover, since $q\leq Q_0$, this gives us that $q^2\leq 2 c_2 x$.
In the main term, we are bounding
$q M^2/x$ from above by $M \cdot q Q/2 x\leq M/2\delta \leq c_2 M$.

If $D\leq (Q+1)/2$, then $M\geq \lfloor D\rfloor$ 
and so (\ref{eq:jenuf}) is all we need: the second sum in (\ref{eq:esthel})
is empty.
Assume from now on that $D> (Q+1)/2$. The first sum in (\ref{eq:esthel}) is then
bounded by (\ref{eq:jenuf}) (with $M=Q/2$).
To bound the second sum in (\ref{eq:esthel}), we will use
the approximation $a'/q'$
 instead of $a/q$. 
The motivation is the following: if we used the approximation $a/q$
even for $m>Q/2$, the contribution of the terms with $q|m$ would be too large.
When we use $a'/q'$, the contribution of the terms with $q'|m$ 
(or $m\equiv \pm 1 \mo q'$) is very small: only a fraction $1/q'$ (tiny,
since $q'$ is large) of all terms are like that, and their individual
contribution is always small, precisely because $m>Q/2$.

By (\ref{eq:gowo}) (without the restriction $q\nmid m$ on either side) 
and Lemma \ref{lem:gotog},
\[\begin{aligned}
&\sum_{Q/2 < m \leq D} |T_m(\alpha)| 
\leq \sum_{j=0}^\infty \sum_{jq' + \frac{Q}{2}< m \leq 
\min((j+1)q'+Q/2,D)} |T_m(\alpha)|\\
&\leq \sum_{j=0}^{\left\lfloor \frac{D-(Q+1)/2}{q'}\right\rfloor} 
\left(3 c_1 \frac{x}{jq'+ \frac{Q+1}{2}} 
+ \frac{4 q'}{\pi} \sqrt{\frac{c_1 c_0}{4}
 \frac{x}{jq'+ (Q+1)/2}
\frac{(j+1)q'+Q/2}{x}}\right)\\
&\leq \sum_{j=0}^{\left\lfloor \frac{D-(Q+1)/2}{q'}\right\rfloor} 
\left(3 c_1 \frac{x}{jq'+ \frac{Q+1}{2}}
+ \frac{4 q'}{\pi}  \sqrt{\frac{c_1 c_0}{4}
 \left(1 + \frac{q'}{jq'+ (Q+1)/2}\right)}\right),
\end{aligned}\]
where we recall that $c_1 = 1 + |\eta'|_1/(2x/D)$.
Since $q'\geq (\epsilon/(1+\epsilon)) Q$,
\begin{equation}\label{eq:tenda}
\sum_{j=0}^{\left\lfloor \frac{D-(Q+1)/2}{q'}\right\rfloor} 
\frac{x}{jq'+ \frac{Q+1}{2}} 
\leq \frac{x}{Q/2} + \frac{x}{q'} \int_{\frac{Q+1}{2}}^D \frac{1}{t} dt \leq 
\frac{2 x}{Q} + \frac{(1+\epsilon) x}{\epsilon Q} \log^+ \frac{D}{
\frac{Q+1}{2}} .
\end{equation}
Recall now that $q' \leq (1+\epsilon) Q + 1 \leq (1+\epsilon)(Q+1)$.
Therefore,
\begin{equation}\label{eq:beatri}\begin{aligned}
q' &\sum_{j=0}^{\lfloor \frac{D-(Q+1)/2}{q'}\rfloor}
\sqrt{1 + \frac{q'}{jq'+ (Q+1)/2}} \leq 
q' \sqrt{1 + \frac{(1+\epsilon)Q+1}{(Q+1)/2}} + 
\int_{\frac{Q+1}{2}}^{D} \sqrt{1 + \frac{q'}{t}} dt\\
&\leq q'\sqrt{3 + 2\epsilon} + \left(D-\frac{Q+1}{2}\right) 
+ \frac{q'}{2} \log^+ \frac{D}{
\frac{Q+1}{2}} 
.\end{aligned}\end{equation}
We conclude that 
$\sum_{Q/2 < m\leq D} |T_m(\alpha)|$ is at most
\begin{equation}\label{eq:sauna}\begin{aligned}
&\frac{2 \sqrt{c_0 c_1}}{\pi} \left(
D  + \left((1+\epsilon) \sqrt{3+2\epsilon}  - \frac{1}{2} \right) (Q+1) +
 \frac{(1+\epsilon)Q+1}{2} \log^+ \frac{D}{\frac{Q+1}{2}}\right)\\
&+ 3 c_1 \left(2 + \frac{(1+\epsilon)}{\epsilon} \log^+ \frac{D}{\frac{Q+1}{2}} \right)
\frac{x}{Q}
\end{aligned}\end{equation}
 We sum this to
(\ref{eq:jenuf})
(with $M=Q/2$), and obtain that (\ref{eq:esthel}) is at most
\begin{equation}\label{eq:kullervo}\begin{aligned}
&\frac{2 \sqrt{c_0 c_1}}{\pi} \left(
D  + (1+\epsilon) (Q+1) \left(\varpi_\epsilon +
 \frac{1}{2} \log^+ \frac{D}{\frac{Q+1}{2}}\right)\right)\\
&+ 3 c_1
 \left(2 + \frac{(1+\epsilon)}{\epsilon} \log \frac{D}{\frac{Q+1}{2}}\right) 
\frac{x}{Q} + \frac{35 c_0 c_2}{6 \pi^2} q,
\end{aligned}\end{equation}
where we are bounding
\begin{equation}\label{eq:semin1}
\frac{5 c_0 c_2}{6 \pi^2} = 
\frac{5 c_0}{6 \pi^2} \frac{3\pi}{5 \sqrt{c_0}} \left(1 +
  \sqrt{\frac{13}{3}} \right) = \frac{\sqrt{c_0}}{2 \pi}
\left(1 + \sqrt{\frac{13}{3}}\right) \leq \frac{2 \sqrt{c_0 c_1}}{\pi} \cdot 
\frac{1}{4}
\left(1 + \sqrt{\frac{13}{3}}\right)
\end{equation}
and defining
\begin{equation}\label{eq:semin2}
\varpi_\epsilon = \sqrt{3 + 2 \epsilon} + 
\left(\frac{1}{4} \left(1 + \sqrt{\frac{13}{3}}\right) -
  1\right) \frac{1}{2 (1+\epsilon)}.\end{equation}
(Note that $\varpi_\epsilon<\sqrt{3}$ for $\epsilon < 0.1741$).
A quick check against (\ref{eq:jenuf}) shows that (\ref{eq:kullervo})
is valid also when $D\leq Q/2$, even when $Q+1$ is replaced by $\min(Q+1,2D)$. 
We bound $Q$ from above by $x/|\delta| q$
and $\log^+ D/((Q+1)/2)$ by $\log^+ 2 D/(x/|\delta| q+1)$,
and obtain the result. 

Case (b): {\em $|\delta|$ small: $|\delta|\leq 1/2 c_2$ or $D\leq Q_0/2$.}
Then $\min(c_2 x/q,D) \leq Q/2$. We start by bounding the first $q/2$ terms 
in (\ref{eq:esthel}) by (\ref{eq:aoro}) and Lemma \ref{lem:thina}:
\begin{equation}\label{eq:prokof}\begin{aligned}
\sum_{m\leq q/2} |T_m(\alpha)| &\leq
\sum_{m\leq q/2} \min\left(
\frac{\frac{1}{2} |\eta'|_1}{|\sin(\pi m \alpha)|},
\frac{c_0 q/8 x}{|\sin(\pi m \alpha)|^2}\right)\\ &\leq
\frac{|\eta'|_1}{\pi} q \max\left(2, \log \frac{c_0 e^3 q^2}{4 \pi |\eta'|_1 x}\right)
.\end{aligned}\end{equation}

If $q^2 < 2 c_2 x$, we estimate the terms with $q/2 < m \leq c_2 x/q$ by
Lemma \ref{lem:couscous}, which is applicable because $\min(c_2 x/q,D) <
Q/2$:
\begin{equation}\label{eq:martinu}\begin{aligned}
&\mathop{\sum_{\frac{q}{2} < m \leq D'}}_{q\nmid m} |T_m(\alpha)|\;
\leq \; \sum_{j=1}^{\infty} 
\mathop{\mathop{\sum_{\left(j-\frac{1}{2}\right) q < m\leq \left(j + \frac{1}{2}
\right) q}}_{m \leq \min\left(\frac{c_2 x}{q},D\right)}}_{q\nmid m} \min\left(\frac{x}{\left(j - 
\frac{1}{2}\right) q} + \frac{|\eta_1'|}{2}, \frac{\frac{c_0}{4}
\frac{(j+1/2) q}{x}}{(\sin \pi m \alpha)^2}\right)\\
&\leq \frac{20}{3 \pi^2} \frac{c_0 q^3}{4 x}
\sum_{1\leq j\leq \frac{D'}{q} + \frac{1}{2}} \left(j+\frac{1}{2}\right)
\leq \frac{20}{3 \pi^2} \frac{c_0 q^3}{4 x} 
\left(\frac{c_2 x}{2 q^2} \frac{D'}{q} + \frac{3}{2}
\left(\frac{c_2 x}{q^2}\right) + \frac{5}{8}\right)\\
&\leq \frac{5 c_0}{6 \pi^2} \left(c_2 D' + 3 c_2 q + 
\frac{5}{4} \frac{q^3}{x}\right) \leq
\frac{5 c_0 c_2}{6 \pi^2} \left(D' + \frac{11}{2} q \right),
\end{aligned}\end{equation}
where we write $D' = \min(c_2 x/q,D)$.
If $c_2 x/q\geq D$, we stop here. 
Assume that $c_2 x/q < D$. Let $R = \max(c_2 x/q, q/2)$. The terms we
have already estimated are precisely those with $m\leq R$.
We bound the terms $R < m\leq D$ by
the second bound in Lemma \ref{lem:gotog}:
\begin{equation}\label{eq:caron}\begin{aligned}
\sum_{R< m \leq D} &|T_m(\alpha)| \leq
\sum_{j=0}^{\infty} \mathop{\sum_{m > j q + R}}_{m\leq
\min\left((j+1) q + R,D\right)}
\min\left(\frac{c_1 x}{jq + R},
\frac{\frac{c_0}{4} \frac{(j+1) q + R}{x}}{(\sin \pi m \alpha)^2}
\right)\\
&\leq \sum_{j=0}^{\left\lfloor \frac{1}{q} \left(D - R\right)
\right\rfloor} \frac{3 c_1 x}{j q + R} + 
\frac{4q}{\pi} \sqrt{\frac{c_1 c_0}{4}
\left(1 + \frac{q}{j q + R}\right)}
\end{aligned}\end{equation}
(Note there is no need to use two successive approximations 
$a/q$, $a'/q'$ as in case (a). We are also including all terms with $m$
divisible by $q$, as we may, since $|T_m(\alpha)|$ is non-negative.) 
Now, much as before,
\begin{equation}\label{eq:kosto}
\sum_{j=0}^{\left\lfloor \frac{1}{q} \left(D - R\right)
\right\rfloor} \frac{x}{j q + R} \leq
\frac{x}{R} + \frac{x}{q} \int_{R}^{D} \frac{1}{t} dt
\leq \min\left(\frac{q}{c_2}, \frac{2 x}{q}\right) 
+ \frac{x}{q} \log^+ \frac{D}{c_2 x/q} ,\end{equation}
and
\begin{equation}\label{eq:kostas}\begin{aligned}
\sum_{j=0}^{\left\lfloor \frac{1}{q} \left(D - R\right)
\right\rfloor}  \sqrt{1 + \frac{q}{j q + R}} &\leq
\sqrt{1 + \frac{q}{R}} + \frac{1}{q} \int_R^D \sqrt{1 + \frac{q}{t}} dt\\
&\leq \sqrt{3} + \frac{D-R}{q} + \frac{1}{2} \log^+ \frac{D}{q/2}.
\end{aligned}\end{equation}
We sum with (\ref{eq:prokof}) and (\ref{eq:martinu}),
and we obtain that (\ref{eq:esthel}) is at most
\begin{equation}\label{eq:kabanova}\begin{aligned}
&\frac{2 \sqrt{c_0 c_1}}{\pi} \left(\sqrt{3} q + D + \frac{q}{2}
\log^+ \frac{D}{q/2}\right) 
+ 
\left(3 c_1 \log^+ \frac{D}{c_2 x/q}\right) \frac{x}{q}\\
&+ 3 c_1 \min\left(\frac{q}{c_2}, \frac{2 x}{q}\right) 
+ \frac{55 c_0 c_2}{12 \pi^2} q
+ \frac{|\eta'|_1}{\pi} q\cdot \max\left(2, \log \frac{c_0 e^3 q^2}{4 \pi |\eta'|_1 x}\right),\end{aligned}\end{equation}
where we are using the fact that $5 c_0 c_2/6\pi^2 < 2 \sqrt{c_0 c_1}/\pi$
to make sure that the term $(5 c_0 c_2/6\pi^2) D'$ from (\ref{eq:martinu})
is more than compensated by the term $- 2\sqrt{c_0 c_1} R/\pi$ coming
from $-R/q$ in (\ref{eq:kostas}) (by the definition of $D'$ and $R$,
we have $R\geq D$). We can also use $5 c_0 c_2/6\pi^2 < 2 \sqrt{c_0 c_1}/\pi$
 to bound  
 the term $(5 c_0 c_2/6\pi^2) D'$ from (\ref{eq:martinu}) by the term
$2\sqrt{c_0 c_1} D/\pi$ in (\ref{eq:kabanova}), in case $c_2 x/q \geq D$.
(Again by definition, $D'\leq D$.)
Thus,
(\ref{eq:kabanova}) is valid both when $c_2 x/q < D$ and when $c_2 x/q \geq D$.
\end{proof}

\subsection{Type I: variations}

We will need a version of Lemma
\ref{lem:bosta1} with $m$ and $n$ restricted to the
odd numbers. (We will barely be using
the restriction of $m$, whereas the restriction on $n$ is both (a) slightly
harder to deal with, (b) something that can be turned to our advantage.)
\begin{lemma}\label{lem:bosta2}
Let $\alpha\in \mathbb{R}/\mathbb{Z}$ with
$2 \alpha = a/q + \delta/x$, $(a,q)=1$, $|\delta/x|\leq 1/q Q_0$,
$q\leq Q_0$, $Q_0\geq 16$.
 Let $\eta$ be continuous, piecewise $C^2$ and compactly supported, with
$|\eta|_1 = 1$ and $\eta''\in L_1$. Let $c_0 \geq |\widehat{\eta''}|_\infty$.

Let $1\leq D\leq x$. Then, if $|\delta|\leq 1/2c_2$, where
$c_2 = 6 \pi/5 \sqrt{c_0}$, the absolute value of 
\begin{equation}\label{eq:gorio2}
\mathop{\sum_{m\leq D}}_{\text{$m$ odd}} \mu(m) 
\sum_{\text{$n$ odd}} e(\alpha m n) \eta\left(\frac{m n}{x}\right)
\end{equation} is at most
\begin{equation}\label{eq:asparto}
\frac{x}{2 q} \min\left(1,\frac{c_0}{(\pi \delta)^2}\right)
\left|
\mathop{\sum_{m\leq \frac{M}{q}}}_{(m,2q)=1} \frac{\mu(m)}{m} \right| + 
O^*\left(\frac{c_0 q}{x} 
\left(\frac{1}{8} - \frac{1}{2 \pi^2}\right)
\left(\frac{D}{q} + 1 \right)^2\right)
\end{equation}
plus
\begin{equation}\label{eq:keks}\begin{aligned} 
&\frac{2 \sqrt{c_0 c_1}}{\pi} D +
\frac{3 c_1}{2} \frac{x}{q} 
\log^+ \frac{D}{c_2 x/q} 
+ \frac{\sqrt{c_0 c_1}}{\pi} q \log^+ \frac{D}{q/2}\\
&+ \frac{2 |\eta'|_1}{\pi} q \cdot \max\left(1, \log \frac{c_0 e^3 q^2}{4 \pi |\eta'|_1 x}\right) +
 \left(\frac{2 \sqrt{3 c_0 c_1}}{\pi} + 
\frac{3 c_1}{2 c_2}
+ \frac{55 c_0 c_2}{6 \pi^2} \right) q ,\end{aligned}\end{equation}
where $c_1 = 1 + |\eta'|_1/(x/D)$
 and $M\in \lbrack \min(Q_0/2,D),D\rbrack$.
The same bound holds if $|\delta|\geq 1/2c_2$ but $D\leq Q_0/2$.

In general, if $|\delta|\geq 1/2c_2$, the absolute value of
(\ref{eq:gorio}) is at most (\ref{eq:asparto}) plus
\begin{equation}\label{eq:kallervo2}\begin{aligned}
&\frac{2 \sqrt{c_0 c_1}}{\pi} \left(
D  + (1+\epsilon) \min\left(\left\lfloor \frac{x}{|\delta| q}\right\rfloor + 1, 2 D\right)
 \left(\sqrt{3+2\epsilon} +
 \frac{1}{2} \log^+ \frac{2 D}{\frac{x}{|\delta| q}}\right)\right)\\
&+ \frac{3}{2} c_1
 \left(2 + \frac{(1+\epsilon)}{\epsilon} \log^+ \frac{2D}{
\frac{x}{|\delta| q}}\right) \frac{x}{Q_0} +
\frac{35 c_0 c_2}{3 \pi^2} q,
\end{aligned}\end{equation}
for $\epsilon\in (0,1\rbrack$ arbitrary.
\end{lemma}

If $q$ is even, the sum (\ref{eq:asparto}) can be replaced by $0$.
\begin{proof}
The proof is almost exactly that of Lemma \ref{lem:bosta1}; we go
over the differences. The parameters $Q$, $Q'$, $a'$, $q'$ and $M$
are defined just as before (with $2\alpha$ wherever we had $\alpha$).

Let us first consider $m\leq M$ odd and divisible by $q$.
(Of course, this case arises only if $q$ is odd.) For $n = 2 r + 1$,
\[\begin{aligned}
e(\alpha m n) &= e(\alpha m (2 r + 1)) = e(2\alpha r m) e(\alpha m)
 \\ &= e\left(\frac{\delta}{x} r m\right) 
e\left(\left(\frac{a}{2 q} + \frac{\delta}{ 2 x} + \frac{\kappa}{2}\right)
m\right)\\ &= 
e\left(\frac{\delta (2 r + 1)}{2 x} m\right)
e\left(\frac{a + \kappa q}{2} \frac{m}{q}\right) = 
\kappa' e\left(\frac{\delta (2 r + 1)}{2 x} m\right),\end{aligned}\]
where $\kappa \in \{0,1\}$ and $\kappa' = e((a+\kappa q)/2) \in \{-1,1\}$ 
are independent of $m$ and $n$. Hence, by Poisson summation,
\begin{equation}\label{eq:kormo}\begin{aligned}
\sum_{\text{$n$ odd}}
 e(\alpha m n) \eta(m n/x) &= \kappa' \sum_{\text{$n$ odd}}
 e((\delta m/2 x) n) \eta(m n/x) \\&= \frac{\kappa'}{2}
\left(\sum_n \widehat{f}(n) - \sum_n \widehat{f}(n+1/2)\right),\end{aligned}
\end{equation}
where $f(u) = e((\delta m/2 x) u) \eta((m/x) u)$. Now
\[\widehat{f}(t) = \frac{x}{m} \widehat{\eta}\left(\frac{x}{m} t - 
\frac{\delta}{2}\right).\] Just as before, $|x/m|\geq 2 |\delta q|\geq 2 \delta$. Thus
\begin{equation}\label{eq:karma}\begin{aligned}
\frac{1}{2}&\left|\sum_n \widehat{f}(n) - \sum_n \widehat{f}(n+1/2)\right|
\leq \frac{x}{m} \left(\frac{1}{2}\left|
\widehat{\eta}\left(-\frac{\delta}{2}\right) \right|
+ \frac{1}{2} \sum_{n\ne 0} \left|\widehat{\eta}\left(\frac{x}{m} \frac{n}{2} -
\frac{\delta}{2}\right)\right|\right)\\
&= \frac{x}{m} \left(\frac{1}{2}\left|
\widehat{\eta}\left(-\frac{\delta}{2}\right) \right|
+ \frac{1}{2} \cdot  O^*\left(\sum_{n\ne 0} \frac{1}{\left(\pi \left(\frac{n x}{m}
- \delta\right)\right)^2}\right) \cdot  \left|\widehat{\eta''}\right|_\infty
\right)\\
&= \frac{x}{2 m} \left|\widehat{\eta}\left(-\frac{\delta}{2}\right)\right| + \frac{m}{x} \frac{c_0}{2 \pi^2}
 (\pi^2 - 4)x
.\end{aligned}\end{equation}
The contribution of the second term in the last line of (\ref{eq:karma})
is
\[\begin{aligned}
\mathop{\mathop{\sum_{m\leq M}}_{\text{$m$ odd}}}_{q|m} \frac{m}{x} \frac{c_0}{
2\pi^2} (\pi^2-4) &= \frac{q}{x} \frac{c_0}{2\pi^2} (\pi^2-4) \cdot
\mathop{\sum_{m\leq M/q}}_{\text{$m$ odd}} m \\ 
&= \frac{q c_0}{x} \left(\frac{1}{8}
- \frac{1}{2\pi^2}\right) \left(\frac{M}{q}+1\right)^2.
\end{aligned}\]
Hence, the absolute value of the sum of all terms with $m\leq M$ and 
$q|m$ is given by (\ref{eq:asparto}).

We define $T_{m,\circ}(\alpha)$ by
\begin{equation}\label{eq:baxter}
T_{m,\circ}(\alpha) = \sum_{\text{$n$ odd}} e(\alpha m n) \eta\left(\frac{m n}{x}\right)
.\end{equation}
Changing variables by $n = 2 r + 1$, we see that
 \[|T_{m,\circ}(\alpha)| = \left|\sum_r e(2\alpha \cdot m r) \eta(m (2 r + 1)/x)
\right|.\]
Hence,
instead of (\ref{eq:aoro}), we get that
\begin{equation}\label{eq:trompais}
|T_{m,\circ}(\alpha)| \leq \min\left(\frac{x}{2 m} + \frac{1}{2} |\eta'|_1,
\frac{\frac{1}{2} |\eta'|_1}{|\sin(2\pi m \alpha)|},
\frac{m}{x} \frac{c_0}{2} \frac{1}{(\sin 2 \pi m \alpha)^2}\right)
.\end{equation}
We obtain (\ref{eq:gowo}), but with $T_{m,\circ}$ instead of $T_m$,
$A = (x/2 y_1) (1 + |\eta'|_1/(x/y_1))$ and $C = (c_0/2) (y_2/x)$,
and so $c_1 = 1 + |\eta'|_1/(x/D)$.

The rest of the proof of Lemma \ref{lem:bosta1} carries almost over 
word-by-word. (For the sake of simplicity, 
we do not really try to take advantage of the odd support
of $m$ here.) Since $C$ has doubled, it would seem to make
sense to reset the value of $c_2$ to be
$c_2 = (3 \pi/5\sqrt{2 c_0}) (1+\sqrt{13/3})$; this would cause
complications related to the fact that $5 c_0 c_2/3 \pi^2$ would become
larger than $2\sqrt{c_0}/\pi$, and so we set $c_2$ to the slightly smaller
value $c_2 = 6\pi/5\sqrt{c_0}$ instead. This implies 
\begin{equation}\label{eq:sosot}
\frac{5 c_0 c_2}{3 \pi^2} = \frac{2\sqrt{c_0}}{\pi}.\end{equation}
 The bound from (\ref{eq:jenuf}) gets multiplied by $2$ 
(but the value of $c_2$ has changed),
the second line in (\ref{eq:sauna}) gets halved,
(\ref{eq:semin1}) gets replaced by (\ref{eq:sosot}),
the second term in the maximum in the second line of (\ref{eq:prokof}) gets
doubled, the bound from (\ref{eq:martinu}) gets doubled,
and the bound from
(\ref{eq:kosto}) gets halved. 
\end{proof}

We will also need a version of Lemma \ref{lem:bosta1} (or
rather Lemma \ref{lem:bosta2}; we will decide to work with the
restriction that $n$ and $m$ be odd)
 with
a factor of $(\log n)$ within the inner sum. This is the sum $S_{I,1}$ in
(\ref{eq:nielsen}).
\begin{lemma}\label{lem:bostb1}
Let $\alpha\in \mathbb{R}/\mathbb{Z}$ with
$2 \alpha = a/q + \delta/x$, $(a,q)=1$, $|\delta/x|\leq 1/q Q_0$,
$q\leq Q_0$, $Q_0\geq \max(16,2 \sqrt{x})$.
 Let $\eta$ be continuous, piecewise $C^2$ and compactly supported, with
$|\eta|_1 = 1$ and $\eta''\in L_1$. Let $c_0 \geq |\widehat{\eta''}|_\infty$.
Assume that, 
for any $\rho\geq \rho_0$, $\rho_0$ a constant, the function
$\eta_{(\rho)}(t) = \log(\rho t) \eta(t)$ satisfies
\begin{equation}\label{eq:puella}
|\eta_{(\rho)}|_1 \leq \log(\rho) |\eta|_1,\;\;\;\;
|\eta_{(\rho)}'|_1 \leq \log(\rho) |\eta'|_1,\;\;\;\;
|\widehat{\eta_{(\rho)}''}|_\infty \leq c_0 \log(\rho) 
\end{equation}

Let $\sqrt{3}\leq D\leq \min(x/\rho_0,x/e)$. Then, if $|\delta|\leq
1/2c_2$, where $c_2 = 6 \pi/5 \sqrt{c_0}$,
 the absolute value of 
\begin{equation}\label{eq:bovary}
\mathop{\sum_{m\leq D}}_{\text{$m$ odd}}
 \mu(m) \mathop{\sum_n}_{\text{$n$ odd}}
 (\log n) e(\alpha m n) \eta\left(\frac{m n}{x}\right)
\end{equation} is at most
\begin{equation}\label{eq:cupcake2}\begin{aligned}
&\frac{x}{q} 
 \min\left(1,\frac{c_0/\delta^2}{(2\pi)^2}\right)
\left|\mathop{\sum_{m\leq \frac{M}{q}}}_{(m,q)=1} \frac{\mu(m)}{m} 
\log \frac{x}{m q}\right| + \frac{x}{q}
|\widehat{\log \cdot \eta}(-\delta)| \left|\mathop{\sum_{m\leq \frac{M}{q}}}_{(m,q)=1} \frac{\mu(m)}{m}\right|\\
&+ 
O^*\left(c_0 
\left(\frac{1}{2} - \frac{2}{\pi^2}
\right)
\left(\frac{D^2}{4 q x} \log \frac{e^{1/2} x}{
D} + \frac{1}{e} \right) 
\right)
\end{aligned}\end{equation}
plus
\begin{equation}\label{eq:kuche2}\begin{aligned} 
&\frac{2 \sqrt{c_0 c_1}}{\pi}  
D \log \frac{e x}{D} + 
\frac{3 c_1}{2} \frac{x}{q} \log^+ \frac{D}{c_2 x/q} \log \frac{q}{c_2}\\
&+ \left( \frac{2 |\eta'|_1}{\pi} 
\max\left(1, \log \frac{c_0 e^3 q^2}{4 \pi |\eta'|_1 x}\right) 
\log x + \frac{2 \sqrt{c_0 c_1}}{\pi} 
\left(\sqrt{3} + \frac{1}{2} \log^+ \frac{D}{q/2}\right)
\log \frac{q}{c_2}\right) q\\
&+  \frac{3 c_1}{2} \sqrt{\frac{2 x}{c_2}} \log \frac{2 x}{c_2} 
+ \frac{20 c_0 c_2^{3/2}}{3 \pi^2} \sqrt{2 x} \log \frac{2 \sqrt{e} x}{c_2} 
\end{aligned}\end{equation}
for $c_1 = 1 + |\eta'|_1/(x/D)$.
The same bound holds if $|\delta|\geq 1/2c_2$ but $D\leq Q_0/2$.

In general, if $|\delta|\geq 1/2c_2$, the absolute value of (\ref{eq:bovary}) is at most
\begin{equation}\label{eq:zygota}
\begin{aligned}
&\frac{2 \sqrt{c_0 c_1}}{\pi} 
D \log \frac{e x}{D} +\\
&\frac{2 \sqrt{c_0 c_1}}{\pi} 
(1+\epsilon) \left(\frac{x}{|\delta| q}+1\right) \left(\sqrt{3+2\epsilon} 
\cdot \log^+ 2 \sqrt{e} |\delta| q
 + \frac{1}{2} \log^+ \frac{2 D}{\frac{x}{|\delta| q}} \log^+ 2 |\delta| q
\right)\\
&+ \left(
\frac{3 c_1}{4} \left(\frac{2}{\sqrt{5}} + \frac{1+\epsilon}{2\epsilon}
\log x\right) +
\frac{40}{3} \sqrt{2} c_0 c_2^{3/2} \right) \sqrt{x} \log x
\end{aligned}\end{equation}
for $\epsilon\in (0,1\rbrack$.
\end{lemma}
\begin{proof}
Define $Q$, $Q'$, $M$, $a'$ and $q'$  as in the proof of Lemma
\ref{lem:bosta1}. The same method of proof works 
as for Lemma \ref{lem:bosta1}; we go over the differences.
When applying Poisson summation or (\ref{eq:ra}), use $\eta_{(x/m)}(t) = 
(\log xt/m) \eta(t)$ instead of $\eta(t)$. Then use 
the bounds in (\ref{eq:puella}) with $\rho = x/m$; in particular,
 \[|\widehat{\eta_{(x/m)}''}|_\infty \leq c_0 \log \frac{x}{m}.\]
For $f(u) = e((\delta m/2 x) u) (\log u) \eta((m/x) u)$,
\[\widehat{f}(t) = \frac{x}{m} \widehat{\eta_{(x/m)}}\left(\frac{x}{m} t - 
\frac{\delta}{2}\right)\]
and so
\[\begin{aligned}
\frac{1}{2} &\sum_n \left|\widehat{f}(n/2)\right| \leq
\frac{x}{m} \left(\frac{1}{2}
\left|\widehat{\eta_{(x/m)}}\left(-\frac{\delta}{2}\right) \right|+
\frac{1}{2} \sum_{n\ne 0} \left|\widehat{\eta}\left(\frac{x}{m} \frac{n}{2}
- \frac{\delta}{2}\right) \right|\right)\\
&= \frac{1}{2} \frac{x}{m}
\left(\widehat{\log \cdot \eta}\left(-\frac{\delta}{2}\right) + 
\log\left(\frac{x}{m}\right) \widehat{\eta}\left(-\frac{\delta}{2}\right)
 \right) + \frac{m}{x} \left(\log \frac{x}{m}\right) \frac{c_0}{2\pi^2} 
(\pi^2 - 4) .\end{aligned}\]

The part of
the main term involving $\log(x/m)$ becomes
\[
\frac{x \widehat{\eta}(-\delta)}{2} 
\mathop{\mathop{\sum_{m\leq M}}_{\text{$m$ odd}}}_{q|m}
\frac{\mu(m)}{m} \log \left(\frac{x}{m }\right)
= \frac{x \mu(q)}{q} 
 \widehat{\eta}(-\delta) \cdot
\mathop{\sum_{m\leq M/q}}_{(m,2q)= 1} \frac{\mu(m)}{m} 
\log\left(\frac{x}{m q}\right) 
\]
for $q$ odd. (We can see that this, like the rest of the main term,
vanishes for $m$ even.)

In the term in front of $\pi^2-4$, we find the sum
\[\begin{aligned}
\mathop{\mathop{\sum_{m\leq M}}_{\text{$m$ odd}}}_{q|m} 
\frac{m}{x} \log\left(\frac{x}{m}\right)
&\leq \frac{M}{x} \log \frac{x}{M} + \frac{q}{2}  
\int_0^{M/q} t \log \frac{x/q}{t} dt \\ &= 
\frac{M}{x} \log \frac{x}{M} + \frac{M^2}{4 q x} \log \frac{e^{1/2} x}{M},
\end{aligned}\]
where we use the fact that $t\mapsto t \log(x/t)$ is increasing for
$t\leq x/e$. By the same fact (and by $M\leq D$), 
$(M^2/q) \log(e^{1/2} x/M)\leq (D^2/q) \log(e^{1/2} x/D)$. It is also easy
to see that $(M/x) \log(x/M) \leq 1/e$ (since $M\leq D \leq x$).

The basic estimate for the rest of the proof (replacing (\ref{eq:aoro}))
is
\[\begin{aligned}
&T_{m,\circ}(\alpha) =
\sum_{\text{$n$ odd}} e(\alpha m n) (\log n) \eta\left(\frac{m n}{x}\right) =
\sum_{\text{$n$ odd}} e(\alpha m n) \eta_{(x/m)} \left(\frac{m n}{x}\right)\\
&= O^*\left( \min\left(\frac{x}{2m} |\eta_{(x/m)}|_1 + \frac{|\eta_{(x/m)}'|_1}{2},
\frac{\frac{1}{2} |\eta_{(x/m)}'|_1 }{|\sin(2 \pi m \alpha)|},
\frac{m}{x} 
\frac{\frac{1}{2} |\widehat{\eta_{(x/m)}''}|_\infty}{(\sin 2 \pi m \alpha)^2}\right)\right)\\
&= O^*\left(\log \frac{x}{m} \cdot
\min\left(\frac{x}{2 m}  + \frac{|\eta'|_1}{2},
\frac{\frac{1}{2} |\eta'|_1 }{|\sin(2 \pi m \alpha)|},
\frac{m}{x} \frac{c_0}{2}
\frac{1}{(\sin 2 \pi m \alpha)^2}\right)
\right).
\end{aligned}\]

We wish to bound
\begin{equation}\label{eq:esthel2}
\mathop{\mathop{\sum_{m\leq M}}_{q\nmid m}}_{\text{$m$ odd}} 
|T_{m,\circ}(\alpha)| + \sum_{\frac{Q}{2} < m \leq D} |T_{m,\circ}(\alpha)|.
\end{equation}

Just as in the proofs of Lemmas \ref{lem:bosta1} and \ref{lem:bosta2}, 
we give two bounds,
one valid for $|\delta|$ large ($|\delta| \geq
1/2 c_2$) and the other for $\delta$ small ($|\delta|\leq 1/2 c_2$).
Again as in the proof of Lemma \ref{lem:bosta2}, we ignore
the condition that $m$ is odd in (\ref{eq:esthel}).

Consider the case of
$|\delta|$ large first. Instead of (\ref{eq:jenuf}), we have
\begin{equation}
\mathop{\sum_{1\leq m\leq M}}_{q\nmid m} |T_m(\alpha)|
\leq \frac{40}{3 \pi^2} \frac{c_0 q^3}{2 x}
\sum_{0\leq j\leq \frac{M}{q}} (j+1) \log \frac{x}{j q+1}.
\end{equation}
Since
\[\begin{aligned}
\sum_{0\leq j\leq \frac{M}{q}} &(j+1) \log \frac{x}{j q+1}
\\ &\leq \log x + 
\frac{M}{q} \log \frac{x}{M} + \sum_{1\leq j\leq \frac{M}{q}} \log 
\frac{x}{j q} + \sum_{1\leq j \leq \frac{M}{q}-1} j \log \frac{x}{j q}\\
&\leq  \log x +  \frac{M}{q} \log \frac{x}{M} +
\int_0^{\frac{M}{q}} \log \frac{x}{t q} dt + \int_1^{\frac{M}{q}}
t \log \frac{x}{t q} dt\\
&\leq  \log x +  \left(\frac{2 M}{q} +
\frac{M^2}{2 q^2}\right)
 \log \frac{e^{1/2} x}{M},
\end{aligned}\]
this means that
\begin{equation}\label{eq:luiw}\begin{aligned}
\mathop{\sum_{1\leq m\leq M}}_{q\nmid m} |T_m(\alpha)|
&\leq  \frac{40}{3 \pi^2} \frac{c_0 q^3}{4 x} \left(\log x +  \left(\frac{2 M}{q} +
\frac{M^2}{2 q^2}\right)
 \log \frac{e^{1/2} x}{M}\right)\\
&\leq \frac{5 c_0 c_2}{3 \pi^2} M \log \frac{\sqrt{e} x}{M}
+ \frac{40}{3} \sqrt{2} c_0 c_2^{3/2} \sqrt{x} \log x,
\end{aligned}\end{equation}
where we are using the bounds $M\leq Q/2\leq c_2 x/q$ and
$q^2\leq 2c_2 x$ (just as in (\ref{eq:jenuf})).
Instead of (\ref{eq:tenda}), we have
\[\begin{aligned}
\sum_{j=0}^{\left\lfloor \frac{D- (Q+1)/2}{q'}\right\rfloor}
\left(\log \frac{x}{j q' + \frac{Q+1}{2}}\right) 
\frac{x}{j q' + \frac{Q+1}{2}}
&\leq \frac{x}{Q/2} \log \frac{2 x}{Q} + \frac{x}{q'}
\int_{\frac{Q+1}{2}}^D \log \frac{x}{t} \frac{dt}{t}\\
&\leq \frac{2 x}{Q} \log \frac{2 x}{Q} + 
\frac{x}{q'} \log \frac{2x}{Q} \log^+ \frac{2D}{Q};
\end{aligned}\]
recall that the coefficient in front of this sum will be halved by
the condition that $n$ is odd.
 Instead of (\ref{eq:beatri}), we obtain
\[\begin{aligned}
&q' \sum_{j=0}^{\lfloor \frac{D-(Q+1)/2}{q'}\rfloor}
\sqrt{1 + \frac{q'}{jq'+ (Q+1)/2}}
\left(\log \frac{x}{j q' + \frac{Q+1}{2}}\right) \\
 &\leq q' \sqrt{3+2\epsilon} \cdot \log \frac{2 x}{Q+1} + 
\int_{\frac{Q+1}{2}}^D \left(1 + \frac{q'}{2 t}\right) 
\left(\log \frac{x}{t}\right) dt \\ &\leq q' \sqrt{3 + 2\epsilon} \cdot 
\log \frac{2 x}{Q+1} + D \log \frac{e x}{D} \\ &- \frac{Q+1}{2} \log
\frac{2 e x}{Q+1} 
+ \frac{q'}{2} \log \frac{2 x}{Q+1} \log \frac{2 D}{Q+1}.
\end{aligned}\]
(The bound $\int_a^b \log(x/t) dt/t \leq \log(x/a) \log(b/a)$
will be more practical than the exact expression for the integral.)
Hence $\sum_{Q/2<m\leq D} |T_m(\alpha)|$ is at most
\[\begin{aligned}
&\frac{2 \sqrt{c_0 c_1}}{\pi} D \log \frac{e x}{D} \\
&+\frac{2 \sqrt{c_0 c_1}}{\pi} \left(
(1+\epsilon) \sqrt{3 + 2\epsilon}  +
\frac{(1+\epsilon)}{2} \log \frac{2 D}{Q+1}
\right) (Q+1) \log \frac{2 x}{Q+1}\\
&- \frac{2 \sqrt{c_0 c_1}}{\pi} \cdot \frac{Q+1}{2} \log
\frac{2 e x}{Q+1}
+ \frac{3 c_1}{2} \left(\frac{2}{\sqrt{5}} + \frac{1+\epsilon}{\epsilon}
\log^+ \frac{D}{Q/2} \right) \sqrt{x} \log \sqrt{x}.
\end{aligned}\]
Summing this to (\ref{eq:luiw}) (with $M = Q/2$), and using
(\ref{eq:semin1}) and (\ref{eq:semin2}) as before,
 we obtain that (\ref{eq:esthel2}) is at most
\[\begin{aligned}
&\frac{2 \sqrt{c_0 c_1}}{\pi} 
D \log \frac{e x}{D} \\ &+ \frac{2 \sqrt{c_0 c_1}}{\pi} 
(1+\epsilon)  (Q+1) 
\left(\sqrt{3+2\epsilon} 
\log^+ \frac{2 \sqrt{e} x}{Q+1}
 + \frac{1}{2} \log^+ \frac{2 D}{Q+1} \log^+ \frac{2 x}{Q+1}
\right)\\
&+ \frac{3 c_1}{2}\left(\frac{2}{\sqrt{5}} + \frac{1+\epsilon}{\epsilon}
\log^+ \frac{D}{Q/2} \right) \sqrt{x} \log \sqrt{x}
+ \frac{40}{3} \sqrt{2} c_0 c_2^{3/2} \sqrt{x} \log x .
\end{aligned}\]

Now we go over the case of $|\delta|$ small
(or $D\leq Q_0/2$).
Instead of (\ref{eq:prokof}),
we have
\begin{equation}\label{eq:bobo}
\sum_{m\leq q/2} |T_{m,\circ}(\alpha)| \leq 
\frac{2 |\eta'|_1}{\pi} q
\max\left(1, \log \frac{c_0 e^3 q^2}{4 \pi |\eta'|_1 x}\right) 
\log x.\end{equation}
Suppose $q^2<2 c_2 x$. (Otherwise, the sum we are about to estimate is
empty.)
Instead of (\ref{eq:martinu}), we have
\begin{equation}\label{eq:bocio}\begin{aligned}
\mathop{\sum_{\frac{q}{2} < m\leq D'}}_{q\nmid m} &|T_{m,\circ}(\alpha)|
\leq \frac{40}{3 \pi^2} \frac{c_0 q^3}{6 x} \sum_{1\leq j\leq 
\frac{D'}{q} + \frac{1}{2}} \left(j + \frac{1}{2}\right)
\log \frac{x}{\left(j - \frac{1}{2}\right) q}\\
&\leq \frac{10 c_0 q^3}{3 \pi^2 x} \left(\log \frac{2 x}{q} + 
\frac{1}{q} \int_0^{D'} \log \frac{x}{t} dt + 
\frac{1}{q} \int_0^{D'} t \log \frac{x}{t} dt 
+ \frac{D'}{q} \log \frac{x}{D'}\right)\\
&=
\frac{10 c_0 q^3}{3 \pi^2 x} \left(\log \frac{2 x}{q} + 
\left(\frac{2 D'}{q} + \frac{(D')^2}{2 q^2}\right) 
\log \frac{\sqrt{e} x}{D'}\right)\\
&\leq \frac{5 c_0 c_2}{3 \pi^2} \left(4 \sqrt{2 c_2 x} \log \frac{2 x}{q} +
4 \sqrt{2 c_2 x} \log \frac{\sqrt{e} x}{D'} + D' \log \frac{\sqrt{e} x}{D'}
\right) \\ &\leq
\frac{5 c_0 c_2}{3 \pi^2} \left(D' \log \frac{\sqrt{e} x}{D'} +
4 \sqrt{2 c_2 x} \log \frac{2 \sqrt{e} x}{c_2} 
\right)
\end{aligned}\end{equation}
where $D' = \min(c_2 x/q,D)$. (We are using the bounds
$q^3/x \leq (2 c_2)^{3/2}$, $D' q^2 / x \leq c_2 q < c_2^{3/2} \sqrt{2 x}$
and $D' q/x \leq c_2$.)
Instead of (\ref{eq:caron}), we have
\[\begin{aligned}\sum_{R< m \leq D} &|T_{m,\circ}(\alpha)| \leq
\sum_{j=0}^{\left\lfloor \frac{D-R}{q} \right\rfloor} 
\left(\frac{\frac{3 c_1}{2} x}{j q + R} + 
\frac{4q}{\pi} \sqrt{\frac{c_1 c_0}{4}
\left(1 + \frac{q}{j q + R}\right)}\right) \log \frac{x}{jq + R},
\end{aligned}\]
where $R = \max(c_2 x/q,q/2)$. We can simply reuse (\ref{eq:kosto}),
multiplying it by $\log x/R$; the only difference is that now we take
care to bound $\min(q/c_2,2x/q)$ by the geometric mean 
$\sqrt{(q/c_2) (2x/q)} = \sqrt{2 x/c_2}$.
We replace (\ref{eq:kostas}) by
\begin{equation}\label{eq:binbed}\begin{aligned}
&\sum_{j=0}^{\left\lfloor \frac{1}{q} \left(D - R\right)\right\rfloor}
\sqrt{1 + \frac{q}{jq+R}} \log \frac{x}{jq + R}
\leq \sqrt{1 + \frac{q}{R}} \log \frac{x}{R} 
+  \frac{1}{q} \int_R^D \sqrt{1 + \frac{q}{t}}
\log \frac{x}{t} dt\\
&\leq \sqrt{3} \log \frac{q}{c_2}
+ \left(\frac{D}{q} \log \frac{e x}{D} - \frac{R}{q} \log \frac{e x}{R}\right)
 + \frac{1}{2} \log \frac{q}{c_2} \log^+ 
\frac{D}{R} .
\end{aligned}\end{equation}
We sum with (\ref{eq:bobo}) and (\ref{eq:bocio}), and obtain
(\ref{eq:kuche2}) as an upper bound for
(\ref{eq:esthel2}). (Just as in the proof of Lemma \ref{lem:bosta1},
the term $(5 c_0 c_2/(3 \pi^2)) D' \log(\sqrt{e} x/D')$ is smaller than
the term $(2 \sqrt{c_1 c_0}/\pi) R \log e x/R$ in (\ref{eq:binbed}),
and thus gets absorbed by it when $D>R$. If $D\leq R$, then,
 again as in Lemma \ref{lem:bosta1}, the sum 
$\sum_{R<m\leq D} |T_{m,\circ}(\alpha)|$ is empty, and we bound
$(5 c_0 c_2/(3 \pi^2)) D' \log(\sqrt{e} x/D')$ by the term
$(2 \sqrt{c_1 c_0}/\pi) D \log e x/D$, which would not appear otherwise.)
\end{proof}

Now comes the time to focus on our second type I sum, namely,
\[\mathop{\sum_{v\leq V}}_{\text{$v$ odd}} \Lambda(v) 
\mathop{\sum_{u\leq U}}_{\text{$u$ odd}} \mu(u) \mathop{\sum_n}_{\text{$n$ odd}}
 e(\alpha v u n)  \eta(v u n/x),\]
which corresponds to the term $S_{I,2}$ in (\ref{eq:nielsen}).
The innermost two sums, on their own, are a sum of type I we have
already seen. Accordingly, for $q$ small, we will be able to bound them
using Lemma \ref{lem:bosta2}. If $q$ is large, then that approach does not
quite work, since then the approximation $av/q$ to $v \alpha$ is not always
good enough. (As we shall later see, we need $q\leq Q/v$ for the approximation
to be sufficiently close for our purposes.) 

Fortunately, when $q$ is large,
we can also afford to lose a factor of $\log$, since the gains from $q$ will
be large. Here is the estimate we will
use for $q$ large.
\begin{lemma}\label{lem:bogus}
Let $\alpha\in \mathbb{R}/\mathbb{Z}$ with
$2 \alpha = a/q + \delta/x$, $(a,q)=1$, $|\delta/x|\leq 1/q Q_0$,
$q\leq Q_0$, $Q_0\geq \max(2e,2 \sqrt{x})$. 
 Let $\eta$ be continuous, piecewise $C^2$ and compactly supported, with
$|\eta|_1 = 1$ and $\eta''\in L_1$. Let $c_0 \geq |\widehat{\eta''}|_\infty$.
Let $c_2 = 6 \pi/5 \sqrt{c_0}$. Assume that $x \geq e^2 c_2/2$.

Let $U,V\geq 1$ satisfy $UV + (19/18) Q_0\leq x/5.6$. 
Then, if $|\delta| \leq 1/2c_2$,
 the absolute value of 
\begin{equation}\label{eq:gargam}
\left|\mathop{\sum_{v\leq V}}_{\text{$v$ odd}} \Lambda(v) 
\mathop{\sum_{u\leq U}}_{\text{$u$ odd}} \mu(u) \mathop{\sum_n}_{\text{$n$ odd}}
 e(\alpha v u n)  \eta(v u n/x)\right|\end{equation}
is at most
\begin{equation}\label{eq:cupcake3}\begin{aligned}
&\frac{x}{2 q} \min\left(1, \frac{c_0}{(\pi \delta)^2}\right) \log V q \\
&+ 
O^*\left(\frac{1}{4} - \frac{1}{\pi^2}\right) \cdot
c_0 \left(\frac{D^2 \log V}{2 q x} + \frac{3 c_4}{2} \frac{U V^2}{x} 
+ \frac{(U+1)^2 V}{2 x} \log q\right)
\end{aligned}\end{equation}
plus
\begin{equation}\label{eq:piececake}\begin{aligned}
&\frac{2 \sqrt{c_0 c_1}}{\pi} \left(
D \log \frac{D}{\sqrt{e}} +
q \left(\sqrt{3} \log \frac{c_2 x}{q} +
\frac{\log D}{2} 
\log^+ \frac{D}{q/2} \right)\right)\\ 
&+ \frac{3 c_1}{2} \frac{x}{q} \log D \log^+ \frac{D}{c_2 x/q}
+ 
\frac{2 |\eta'|_1}{\pi} q
\max\left(1, \log \frac{c_0 e^3 q^2}{4 \pi |\eta'|_1 x}\right) \log \frac{q}{2}
\\
&+ \frac{3 c_1}{2 \sqrt{2 c_2}} \sqrt{x} \log \frac{c_2 x}{2} + 
\frac{25 c_0}{4 \pi^2} (2 c_2)^{3/2} \sqrt{x} \log x ,
\end{aligned} \end{equation}
where $D = UV$ and $c_1 = 1 + |\eta'|_1/(2x/D)$ and $c_4 = 1.03884$.
The same bound holds if $|\delta| \geq 1/2 c_2$ but
$D\leq Q_0/2$.

In general,
 if $|\delta| \geq 1/2c_2$, the absolute value of (\ref{eq:gargam}) is at most
(\ref{eq:cupcake3}) plus
\begin{equation}\label{eq:tvorog}\begin{aligned}
&\frac{2 \sqrt{c_0 c_1}}{\pi} D \log \frac{D}{e}\\ 
&+ \frac{2 \sqrt{c_0 c_1}}{\pi} (1+\epsilon) \left(\frac{x}{|\delta| q}+1
\right) \left(
(\sqrt{3+2\epsilon}-1) \log \frac{\frac{x}{|\delta| q}+1}{\sqrt{2}}
 + \frac{1}{2} \log D \log^+ \frac{e^2 D}{\frac{x}{|\delta|q}}\right)\\
&+
\left(\frac{3 c_1}{2} \left( \frac{1}{2}  + 
\frac{3(1+\epsilon)}{16 \epsilon} \log x\right) 
+ \frac{20 c_0}{3 \pi^2} (2 c_2)^{3/2}
\right) \sqrt{x} \log x
\end{aligned}\end{equation}
for $\epsilon\in (0,1\rbrack$.
\end{lemma}
\begin{proof}
We proceed essentially as in Lemma \ref{lem:bosta1} and Lemma \ref{lem:bosta2}.
Let $Q$, $q'$ and $Q'$ be as in the proof of Lemma \ref{lem:bosta2}, that
is, with $2 \alpha$ where Lemma \ref{lem:bosta1} uses $\alpha$.


 Let $M = \min(UV, Q/2)$. We first consider the terms with $uv\leq M$, 
$u$ and $v$ odd, $uv$ divisible by $q$. If $q$ is even, there are no such terms.
Assume $q$ is odd. Then, by (\ref{eq:kormo}) and (\ref{eq:karma}),
the absolute value of the contribution of these terms is at most 
\begin{equation}\label{eq:hoho}
\mathop{\mathop{\sum_{a\leq M}}_{\text{$a$ odd}}}_{q|a} 
\left(\mathop{\sum_{v|a}}_{a/U\leq v\leq V}
\Lambda(v) \mu(a/v)\right)
\left(
\frac{x \widehat{\eta}(-\delta/2)}{2a} +
O\left(\frac{a}{x} \frac{|\widehat{\eta''}|_\infty}{2 \pi^2} \cdot
(\pi^2 - 4)\right)\right).
\end{equation}

Now
\[\begin{aligned}
\mathop{\mathop{\sum_{a\leq M}}_{\text{$a$ odd}}}_{q|a} \;
&\mathop{\sum_{v|a}}_{a/U\leq v\leq V} \frac{\Lambda(v) \mu(a/v)}{a}\\
&= \mathop{\mathop{\sum_{v\leq V}}_{\text{$v$ odd}}}_{(v,q)=1} 
 \frac{\Lambda(v)}{v} 
\mathop{\mathop{\sum_{u\leq \min(U,M/V)}}_{\text{$u$ odd}}}_{q|u}
\frac{\mu(u)}{u} + 
\mathop{\mathop{\sum_{p^\alpha\leq V}}_{\text{$p$ odd}}}_{p|q}
 \frac{\Lambda(p^{\alpha})}{p^{\alpha}}
\mathop{\mathop{\sum_{u\leq \min(U,M/V)}}_{\text{$u$ odd}}}_{\frac{q}{(q,p^{\alpha})}|u} \frac{\mu(u)}{u}, \end{aligned}\] which equals
\[\begin{aligned} &\frac{\mu(q)}{q}
\mathop{\mathop{\sum_{v\leq V}}_{\text{$v$ odd}}}_{(v,q)=1} 
 \frac{\Lambda(v)}{v} 
\mathop{\sum_{u\leq \min(U/q,M/Vq)}}_{(u,2q)=1}
\frac{\mu(u)}{u} \\ &+ \frac{\mu\left(\frac{q}{(q,p^\alpha)}\right)}{q}
\mathop{\mathop{\sum_{p^\alpha\leq V}}_{\text{$p$ odd}}}_{p|q}
\frac{\Lambda(p^{\alpha})}{p^{\alpha}/(q,p^{\alpha})}
\mathop{\mathop{\sum_{u\leq \min\left(\frac{U}{q/(q,p^\alpha)},\frac{M/V}{
q/(q,p^{\alpha})}\right)}}_{\text{$u$ odd}}}_{
\left(u,\frac{q}{(q,p^{\alpha})}\right) = 1} \frac{\mu(u)}{u}\\
&= \frac{1}{q} \cdot O^*\left(\mathop{\sum_{v\leq V}}_{(v,2q)=1}
\frac{\Lambda(v)}{v} + \mathop{\mathop{\sum_{p^\alpha\leq V}}_{\text{$p$ odd}}}_{p|q}
\frac{\log p}{p^{\alpha}/(q,p^{\alpha})}\right),
\end{aligned}\]
where we are using (\ref{eq:grara}) to bound the sums on $u$ by $1$.
We notice that
\[\begin{aligned}
&\mathop{\mathop{\sum_{p^\alpha\leq V}}_{\text{$p$ odd}}}_{p|q}
\frac{\log p}{p^{\alpha}/(q,p^{\alpha})} \leq
\mathop{\sum_{\text{$p$ odd}}}_{p|q} (\log p) \left(v_p(q) + 
\mathop{\sum_{\alpha>v_p(q)}}_{p^\alpha\leq V} \frac{1}{p^{\alpha-v_p(q)}}\right)\\
&\leq \log q + 
\mathop{\sum_{\text{$p$ odd}}}_{p|q} (\log p) 
\mathop{\sum_{\beta>0}}_{p^\beta \leq \frac{V}{p^{v_p(q)}}} \frac{\log p}{p^\beta}
\leq \log q + \mathop{\mathop{\sum_{v\leq V}}_{\text{$v$ odd}}}_{(v,q)=1}
\frac{\Lambda(v)}{v},\end{aligned}\]
and so
\[\begin{aligned}
\mathop{\mathop{\sum_{a\leq M}}_{\text{$a$ odd}}}_{q|a} \;
\mathop{\sum_{v|a}}_{a/U\leq v\leq V} \frac{\Lambda(v) \mu(a/v)}{a} &=
\frac{1}{q} \cdot O^*\left(\log q + 
\mathop{\sum_{v\leq V}}_{(v,2)=1} \frac{\Lambda(v)}{v}\right)\\
&= \frac{1}{q} \cdot O^*(\log q + \log V)\end{aligned}\]
by (\ref{eq:rala}). The absolute value of the sum of the terms with
$\widehat{\eta}(-\delta/2)$ in (\ref{eq:hoho}) is thus at most
\[\frac{x}{q} \frac{\widehat{\eta}(-\delta/2)}{2} (\log q + \log V)
\leq
\frac{x}{2 q} \min\left(1, \frac{c_0}{(\pi \delta)^2}\right) \log Vq,\]
where we are bounding $\widehat{\eta}(-\delta/2)$ by
(\ref{eq:madge}) (with $k=2$).

The other terms in (\ref{eq:hoho}) contribute at most
\begin{equation}\label{eq:billy}
(\pi^2 - 4) \frac{|\widehat{\eta''}|_\infty}{2 \pi^2} \frac{1}{x} 
\mathop{\mathop{\mathop{\sum_{u\leq U} \sum_{v\leq V}}_{\text{$uv$ odd}}}_{
uv \leq M,\; q|uv}}_{\text{$u$ sq-free}}
\Lambda(v) u v.\end{equation}

For any $R$, $\sum_{u\leq R, \text{$u$ odd}, q|u} \leq R^2/4q + 3 R/4$.
Using the estimates (\ref{eq:rala}), (\ref{eq:trado2}) and (\ref{eq:chronop}),
we obtain that the double sum in (\ref{eq:billy}) is at most
\begin{equation}\label{eq:etoile}
\begin{aligned}
\mathop{\sum_{v\leq V}}_{(v,2q)=1} &\Lambda(v) v 
\mathop{\mathop{\sum_{u\leq \min(U,M/v)}}_{\text{$u$ odd}}}_{q|u} u +
\mathop{\mathop{\sum_{p^\alpha\leq V}}_{\text{$p$ odd}}}_{p|q} (\log p) p^\alpha 
\mathop{\mathop{\sum_{u\leq U}}_{\text{$u$ odd}}}_{\frac{q}{(q,p^\alpha)} | u} u\\
&\leq \mathop{\sum_{v\leq V}}_{(v,2q)=1} \Lambda(v) v \cdot
\left(\frac{(M/v)^2}{4 q} + \frac{3 M}{4 v}\right) +
\mathop{\mathop{\sum_{p^\alpha\leq V}}_{\text{$p$ odd}}}_{p|q} (\log p) p^\alpha \cdot
\frac{(U+1)^2}{4}\\
&\leq \frac{M^2 \log V}{4 q} + \frac{3 c_4}{4} M V 
+ \frac{(U+1)^2}{4} V \log q,
\end{aligned}\end{equation}
where $c_4 = 1.03884$.

From this point onwards, we use the easy bound
\[
\left|\mathop{\sum_{v|a}}_{a/U\leq v\leq V}
\Lambda(v) \mu(a/v)\right| \leq \log a .
\]
What we must bound now is
\begin{equation}\label{eq:esthel3}
\mathop{\mathop{\sum_{m\leq UV}}_{\text{$m$ odd}}}_{\text{$q\nmid m$ or 
$m>M$}} (\log m) \sum_{\text{$n$ odd}} e(\alpha m n) \eta(m n/x).\end{equation}
The inner sum is the same as the sum $T_{m,\circ}(\alpha)$ in (\ref{eq:baxter});
we will be using the bound (\ref{eq:trompais}). Much as before, we will be able
to ignore the condition that $m$ is odd.

Let $D= UV$. What remains to do is similar to what we did
 in the proof of Lemma \ref{lem:bosta1} (or Lemma \ref{lem:bosta2}).

Case (a). {\em $\delta$ large: $|\delta|\geq 1/2c_2$.}
 Instead of (\ref{eq:jenuf}),
we have
\[\mathop{\sum_{1\leq m\leq M}}_{q\nmid m} (\log m) |T_{m,\circ}(\alpha)| \leq
 \frac{40}{3 \pi^2} \frac{c_0 q^3}{4 x} \sum_{0\leq j\leq \frac{M}{q}}
(j+1) \log (j+1)q,\]
and, since $M\leq \min(c_2 x/q,D)$,
$q\leq \sqrt{2 c_2 x}$ (just as in the proof of Lemma \ref{lem:bosta1}) and
\[\begin{aligned}
\sum_{0\leq j\leq \frac{M}{q}}
&(j+1) \log (j+1)q \\&\leq \frac{M}{q} \log M + \left(\frac{M}{q} + 1\right) 
\log (M+1) + \frac{1}{q^2} \int_0^M t \log t\; dt\\
&\leq \left(2 \frac{M}{q} + 1\right) \log x 
 + \frac{M^2}{2 q^2} \log \frac{M}{\sqrt{e}},
\end{aligned}\]
we conclude that 
\begin{equation}\label{eq:cocolo}\begin{aligned}
\mathop{\sum_{1\leq m\leq M}}_{q\nmid m} |T_{m,\circ}(\alpha)| \leq
 \frac{5 c_0 c_2}{3 \pi^2} M \log \frac{M}{\sqrt{e}} +
\frac{20 c_0}{3 \pi^2} (2 c_2)^{3/2} 
\sqrt{x} \log x
.\end{aligned}\end{equation}

Instead of (\ref{eq:tenda}), we have
\[\begin{aligned}
\sum_{j=0}^{\lfloor \frac{D - (Q+1)/2}{q'}\rfloor} \frac{x}{j q' + \frac{Q+1}{2}}
&\log \left(j q' + \frac{Q+1}{2}\right)
\leq \frac{x}{\frac{Q+1}{2}} \log \frac{Q+1}{2}  + \frac{x}{q'} \int_{\frac{Q+1}{2}}^D \frac{\log t}{t} dt\\
&\leq \frac{2 x}{Q} \log \frac{Q}{2} + \frac{(1+\epsilon) x}{2 \epsilon Q} \left((\log D)^2 - 
 \left(\log \frac{Q}{2}\right)^2\right)
.\end{aligned}\]
Instead of (\ref{eq:beatri}), we estimate
\[\begin{aligned}
q' &\sum_{j=0}^{\left\lfloor \frac{D - \frac{Q+1}{2}}{q'}\right\rfloor} 
\left(\log \left(\frac{Q+1}{2} + j q'\right)\right)
\sqrt{1 + \frac{q'}{j q' + \frac{Q+1}{2}}}\\ &\leq
q' \left(\log D + (\sqrt{3+2\epsilon} - 1) \log \frac{Q+1}{2} \right) + 
\int_{\frac{Q+1}{2}}^D \log t\; dt + 
\int_{\frac{Q+1}{2}}^D \frac{q' \log t}{2 t} dt\\
&\leq q' \left(\log D + \left(\sqrt{3+2\epsilon} - 1\right) \log \frac{Q+1}{2} \right) + 
 \left(D \log \frac{D}{e} - 
\frac{Q+1}{2} \log \frac{Q+1}{2 e}\right) \\ &+ 
\frac{q'}{2} \log D \log^+ \frac{D}{\frac{Q+1}{2}} .\end{aligned}\]
We conclude that, when $D\geq Q/2$, the sum 
 $\sum_{Q/2 < m\leq D} (\log m) |T_m(\alpha)|$ is at most
\[\begin{aligned}
&\frac{2 \sqrt{c_0 c_1}}{\pi} \left(D \log \frac{D}{e}
+ (Q+1)
\left((1+\epsilon) (\sqrt{3 + 2 \epsilon} - 1) \log \frac{Q+1}{2} -
\frac{1}{2} \log \frac{Q+1}{2 e}\right)\right)\\
&+ \frac{\sqrt{c_0 c_1}}{\pi} (Q+1) 
(1+\epsilon) \log D \log^+ \frac{e^2 D}{\frac{Q+1}{2}} \\
&+
\frac{3 c_1}{2} \left( \frac{2 x}{Q} \log \frac{Q}{2} + \frac{(1+\epsilon) x}{2 \epsilon Q} \left((\log D)^2 - 
 \left(\log \frac{Q}{2}\right)^2\right)\right).
\end{aligned}\]
We must now add this to (\ref{eq:cocolo}). Since
\[
 (1+\epsilon) (\sqrt{3 + 2\epsilon} - 1) \log \sqrt{2}
- \frac{1}{2} \log 2 e + \frac{1 + \sqrt{13/3}}{2} \log 2 \sqrt{e}
> 0\]
and $Q\geq 2 \sqrt{x}$,
 we conclude that (\ref{eq:esthel3}) is at most
\begin{equation}\label{eq:iulia}\begin{aligned}
&\frac{2 \sqrt{c_0 c_1}}{\pi} D \log \frac{D}{e}\\ 
&+ \frac{2 \sqrt{c_0 c_1}}{\pi} (1+\epsilon) (Q+1) \left(
(\sqrt{3+2\epsilon}-1) \log \frac{Q+1}{\sqrt{2}}
 + \frac{1}{2} \log D \log^+ \frac{e^2 D}{\frac{Q+1}{2}}\right)\\
&+
\left(\frac{3 c_1}{2} \left( \frac{1}{2}  + 
\frac{3(1+\epsilon)}{16 \epsilon} \log x\right) 
+ \frac{20 c_0}{3 \pi^2} (2 c_2)^{3/2}
\right) \sqrt{x} \log x.
\end{aligned}\end{equation}

{\em Case (b). $\delta$ small: $|\delta|\leq 1/2 c_2$
or $D\leq Q_0/2$.}
The analogue of (\ref{eq:prokof}) is a bound of
\[\leq \frac{2 |\eta'|_1}{\pi} q
\max\left(1, \log \frac{c_0 e^3 q^2}{4 \pi |\eta'|_1 x}\right) \log \frac{q}{2}
\] 
for the terms with $m\leq q/2$. 
If $q^2 < 2 c_2 x$, then, much
as in (\ref{eq:martinu}), we have
\begin{equation}\label{eq:jotoy}\begin{aligned}
\mathop{\sum_{\frac{q}{2} < m \leq D'}}_{q\nmid m} |T_{m,\circ}(\alpha)|
(\log m)
&\leq \frac{10}{\pi^2} \frac{c_0 q^3}{3 x} \sum_{1\leq j\leq \frac{D'}{q} + 
\frac{1}{2}} \left(j + \frac{1}{2}\right) \log (j+1/2) q\\
&\leq \frac{10}{\pi^2} \frac{c_0 q}{3 x} \int_q^{D' + \frac{3}{2} q}
x \log x \; dx.\end{aligned}\end{equation}
Since
\[\begin{aligned}
&\int_q^{D' + \frac{3}{2} q} x \log x \; dx = 
\frac{1}{2} \left(D' +
\frac{3}{2} q \right)^2 \log \frac{D'+\frac{3}{2} q}{\sqrt{e}} 
- \frac{1}{2} q^2 \log \frac{q}{\sqrt{e}}\\
&= \left(\frac{1}{2} D'^2 + \frac{3}{2} D' q\right)
\left(\log \frac{D'}{\sqrt{e}} + \frac{3}{2} \frac{q}{D'}\right) + 
\frac{9}{8} q^2 \log \frac{D' + \frac{3}{2} q}{\sqrt{e}}
 - \frac{1}{2} q^2 \log \frac{q}{\sqrt{e}}\\
&= \frac{1}{2} D'^2 \log \frac{D'}{\sqrt{e}} + \frac{3}{2} D' q \log D' +
\frac{9}{8} q^2 \left(\frac{2}{9} + \frac{3}{2} + 
\log \left(D' + \frac{19}{18} q\right)\right),
\end{aligned}\]
where $D' = \min(c_2 x/q,D)$, and since the assumption
$(UV + (19/18) Q_0)\leq x/5.6$ implies that $(2/9 + 3/2 + 
\log(D' + (19/18) q)) \leq x$, 
 we conclude that
\begin{equation}\label{eq:prok}\begin{aligned}
&\mathop{\sum_{\frac{q}{2} < m \leq D'}}_{q\nmid m} |T_{m,\circ}(\alpha)| (\log m)
\\ &\leq \frac{5 c_0 c_2}{3 \pi^2} D' \log \frac{D'}{\sqrt{e}} +
\frac{10 c_0}{3 \pi^2} \left( \frac{3}{4} (2 c_2)^{3/2} \sqrt{x} \log x
+ \frac{9}{8} (2 c_2)^{3/2} \sqrt{x} \log x\right)\\
&\leq \frac{5 c_0 c_2}{3 \pi^2} D' \log \frac{D'}{\sqrt{e}} +
\frac{25 c_0}{4 \pi^2} (2 c_2)^{3/2} \sqrt{x} \log x 
.\end{aligned}\end{equation}
Let $R = \max(c_2 x/q,q/2)$.
We bound the terms $R<m\leq D$ as in (\ref{eq:caron}), with a factor of
$\log (j q + R)$ inside the sum. The analogues of (\ref{eq:kosto}) and
(\ref{eq:kostas}) are
\begin{equation}\label{eq:gator1}\begin{aligned}
\sum_{j=0}^{\left\lfloor \frac{1}{q} (D- R)\right\rfloor} &\frac{x}{jq + R} \log(jq + R) \leq
\frac{x}{R} \log R + \frac{x}{q} \int_R^D \frac{\log t}{t} dt\\
&\leq \sqrt{ \frac{ 2 x}{c_2}} \log \sqrt{\frac{c_2 x}{2}} + \frac{x}{q}
\log D \log^+ \frac{D}{R}, 
\end{aligned}\end{equation} where we use the assumption that
$x\geq e^2 c/2$, and
\begin{equation}\label{eq:gator2}\begin{aligned}
\sum_{j=0}^{\left\lfloor \frac{1}{q} (D- R)\right\rfloor} &\log(jq + R) 
\sqrt{1 + \frac{q}{jq + R}}
\leq \sqrt{3} \log R\\
&+ \frac{1}{q} \left(D \log \frac{D}{e} - R \log \frac{R}{e}\right)
+ \frac{1}{2} \log D \log \frac{D}{R} 
\end{aligned}\end{equation}
(or $0$ if $D<R$).
We sum with (\ref{eq:prok}) and the terms with $m\leq q/2$, and 
obtain, for $D' = c_2 x/q = R$,
\[\begin{aligned}
&\frac{2 \sqrt{c_0 c_1}}{\pi} \left(
D \log \frac{D}{\sqrt{e}} +
q \left(\sqrt{3} \log \frac{c_2 x}{q} +
\frac{\log D}{2} 
\log^+ \frac{D}{q/2} \right)\right)\\ 
&+ \frac{3 c_1}{2} \frac{x}{q} \log D \log^+ \frac{D}{c_2 x/q}
+ 
\frac{2 |\eta'|_1}{\pi} q
\max\left(1, \log \frac{c_0 e^3 q^2}{4 \pi |\eta'|_1 x}\right) \log \frac{q}{2}
\\
&+ \frac{3 c_1}{2 \sqrt{2 c_2}} \sqrt{x} \log \frac{c_2 x}{2} + 
\frac{25 c_0}{4 \pi^2} (2 c_2)^{3/2} \sqrt{x} \log x 
,\end{aligned}\]
which, it is easy to check, is also valid even if $D'=D$ (in which case
(\ref{eq:gator1}) and (\ref{eq:gator2}) do not appear) or $R=q/2$ (in which
case (\ref{eq:prok}) does not appear). 
\end{proof}

\chapter{Type II sums}\label{sec:typeII}
We must now consider the sum
\begin{equation}\label{eq:adoucit}
S_{II} = \mathop{\sum_{m>U}}_{(m,v)=1}
 \left(\mathop{\sum_{d>U}}_{d|m} \mu(d)\right) \mathop{\sum_{n>V}}_{(n,v)=1}
 \Lambda(n)
 e(\alpha m n) \eta(m n/x).\end{equation}

Here the main improvements over classical treatments of type II sums
are as follows:
\begin{enumerate}
\item\label{it:bbc}
 obtaining cancellation in the term 
\[
\mathop{\sum_{d>U}}_{d|m} \mu(d),
\] leading to
a gain of a factor of $\log$; 
\item\label{it:umor} using a large sieve for primes, getting
rid of a further $\log$;
\item\label{it:gomro}
 exploiting, via a non-conventional application of the principle of the
 large sieve (Lemma \ref{lem:ogor}),
the fact that $\alpha$ is in the tail of an interval (when that is the case).
\end{enumerate}
It should be clear that these techniques are of general applicability.
(It is also clear that (\ref{it:umor}) is not new, though, strangely enough,
it seems not to have been applied to Goldbach's problem. Perhaps this oversight
is due to the fact that
proofs of Vinogradov's result given in textbooks often follow Linnik's 
dispersion method, rather than the large sieve. Our treatment of
the large sieve for primes will follow
the lines set by Montgomery and Montgomery-Vaughan \cite[(1.6)]{MR0374060}.
The fact that the large sieve for primes
can be combined with the new technique (\ref{it:gomro}) is, of course, 
a novelty.)

While (\ref{it:bbc}) is particularly useful for the treatment of a term
that generally arises in applications of Vaughan's identity, all of the points
above address issues that can arise in more general situations in number theory.

It is technically helpful to express $\eta$ as the 
(multiplicative) convolution of two functions of compact support --
preferrably the same function:
\begin{equation}\label{eq:conque}
\eta(x) = \eta_1\ast_M \eta_1 = 
\int_0^\infty \eta_1(t) \eta_1(x/t) \frac{dt}{t}.\end{equation}
For the smoothing function $\eta(t) = \eta_2(t) = 4 \max(\log 2 - |\log 2 t|, 
0)$, equation 
(\ref{eq:conque}) holds with $\eta_1 = 2 \cdot 1_{\lbrack 1/2,1\rbrack}$,
where $1_{\lbrack 1/2,1\rbrack}$ is the characteristic function of the interval
$\lbrack 1/2,1\rbrack$.
We will work with $\eta = \eta_2$, 
yet most of our work will be valid for any $\eta$ of the
form $\eta = \eta_1\ast \eta_1$.

By (\ref{eq:conque}), the sum (\ref{eq:adoucit}) equals
\begin{equation}\label{eq:bycaus}\begin{aligned}
4 &\int_0^{\infty}
\mathop{\sum_{m>U}}_{(m,v)=1}
 \left(\mathop{\sum_{d>U}}_{d|m} \mu(d)\right)
 \mathop{\sum_{n>V}}_{(n,v)=1} \Lambda(n)
 e(\alpha m n) \eta_1(t) \eta_1\left(\frac{m n/x}{t}\right) \frac{dt}{t}\\
&= 4 \int_V^{x/U} 
\mathop{\sum_{\max\left(\frac{x}{2W},U\right)<m\leq \frac{x}{W}}}_{(m,v)=1} 
\left(\mathop{\sum_{d>U}}_{d|m} \mu(d)\right)
 \mathop{\sum_{\max\left(V,\frac{W}{2}\right)<n\leq W}}_{(n,v)=1} \Lambda(n)
 e(\alpha m n)  \frac{dW}{W}
\end{aligned}\end{equation}
by the substitution $t = (m/x) W$.
(We can assume $V \leq W \leq x/U$ because otherwise one of the sums in
(\ref{eq:costo}) is empty.) As we can see, the sums within the integral
are now unsmoothed. This will not be truly harmful, and to some extent
it will be convenient, in that ready-to-use
large-sieve estimates in the literature have
been optimized more carefully for unsmoothed sums than for smooth sums.
The fact that the sums start at $x/2W$ and $W/2$ rather than at $1$ will also
be slightly helpful.

(This is presumably why the weight $\eta_2$ was introduced in \cite{Tao}, which
also uses the large sieve. As we will later see, the weight $\eta_2$ -- or 
anything like it -- will simply not do on the major arcs, which are much
more sensitive to the choice of weights. On the minor arcs,
however, $\eta_2$ is convenient, and this is why we use it here. For type I sums --
as should be clear from our work so far, which was stated for general weights --
 any function whose second derivative exists almost everywhere and lies in
$\ell_1$ would do just as well. The option of having no smoothing whatsoever
-- as in Vinogradov's work, or as in most textbook accounts -- would not be
quite as good for type I sums, and would lead to a routine but inconvenient
splitting of sums into short intervals in place of (\ref{eq:bycaus}).)

We now do what is generally the first thing in type II treatments: we
use Cauchy-Schwarz. A minor note, however, that may help avoid confusion: 
the treatments familiar to some
readers (e.g., the dispersion method, not followed here) 
start with the special case of Cauchy-Schwarz that is most common
in number theory
\[\left|\sum_{n\leq N} a_n\right|^2 \leq N \sum_{n\leq N} |a_n|^2,\]
whereas here we apply the general rule
\[\sum_{m} a_m b_m \leq \sqrt{\sum_{m} |a_m|^2} \sqrt{\sum_{m} |b_m|^2}.\]
to the integrand in (\ref{eq:bycaus}). At any rate, we will have
reduced the estimation of a sum to the estimation of two simpler sums
$\sum_{m} |a_m|^2$, $\sum_{m} |b_m|^2$,
but each of these two simpler sums will be of a kind that we will lead
to a loss of a factor of $\log x$ (or $(\log x)^3$) if not estimated carefully.
Since we cannot afford to lose a single factor of $\log x$, we will have
to deploy and develop techniques to eliminate these factors of $\log x$.
The procedure followed will be quite different for the two sums; a variety
of techniques will be needed.

We separate $n$ prime and $n$ non-prime in the integrand of (\ref{eq:bycaus}),
and, as we were saying, we apply Cauchy-Schwarz.
We obtain that the expression within the integral in (\ref{eq:bycaus})
is at most $\sqrt{S_1(U,W)\cdot S_2(U,V,W)} +
\sqrt{S_1(U,W) \cdot S_3(W)}$, where
\begin{equation}\label{eq:costo}\begin{aligned}
S_1(U,W) &= \mathop{\sum_{\max\left(\frac{x}{2W},U\right)<m\leq \frac{x}{W}}}_{
(m,v)=1} 
\left(\mathop{\sum_{d>U}}_{d|m} \mu(d)\right)^2 ,\\
S_2(U,V,W) &= \mathop{\sum_{\max\left(\frac{x}{2W},U\right)<m\leq \frac{x}{W}}}_{
(m,v)=1} 
\left|
 \mathop{\sum_{\max\left(V,\frac{W}{2}\right)<p\leq W}}_{(p,v)=1} (\log p)
 e(\alpha m p)\right|^2.\end{aligned}
\end{equation}
and
\begin{equation}\label{eq:negli}\begin{aligned}
S_3(W) &= 
 \mathop{\sum_{\frac{x}{2 W} < m\leq \frac{x}{W}}}_{(m,v)=1}
\left| \mathop{\sum_{n\leq W}}_{\text{$n$ non-prime}} \Lambda(n)\right|^2\\
&= \mathop{\sum_{\frac{x}{2 W} < m\leq \frac{x}{W}}}_{(m,v)=1}
 \left(1.42620 W^{1/2}\right)^2
\leq 1.0171 x + 2.0341 W\end{aligned}\end{equation}
(by \cite[Thm. 13]{MR0137689}). We will assume $V\leq w$; thus the condition
$(p,v)=1$ will be fulfilled automatically and can be removed.

The contribution of $S_3(W)$ will be negligible.
We must bound $S_1(U,W)$ and $S_2(U,V,W)$ from above.

\section{The sum $S_1$: cancellation}\label{subs:vaucanc}
We shall bound
\begin{equation}\label{eq:mahalobi}
S_1(U,W) = \mathop{\sum_{\max(U,x/2W)<m\leq x/W}}_{(m,v)=1}
 \left(\mathop{\sum_{d>U}}_{d|m} \mu(d)\right)^2.\end{equation}

There will be
 a surprising amount of cancellation: the expression within
the sum will be bounded by a constant on average -- a constant less than $1$,
and usually less than $1/2$, in fact. In other words, the inner sum in 
(\ref{eq:mahalobi}) is exactly $0$ most of the time. 

Recall that we need explicit constants throughout, and that this essentially
constrains us to elementary means. (We will at one point use Dirichlet
series and $\zeta(s)$ for $s$ real and greater than $1$.)

\subsection{Reduction to a sum with $\mu$}

It is tempting to start by applying M\"obius inversion to change $d>U$
to $d\leq U$ in (\ref{eq:mahalobi}), but this just makes matters worse.
We could also try changing variables so that $m/d$ (which is smaller than
$x/UW$) becomes the variable instead of $d$, but this leads to complications
for $m$ non-square-free. Instead, we write
\begin{equation}\label{eq:crusto}\begin{aligned}
\mathop{\sum_{\max(U,x/2W)<m\leq x/W}}_{(m,v)=1}
 &\left(\mathop{\sum_{d>U}}_{d|m} \mu(d)\right)^2
= \mathop{\sum_{\frac{x}{2W} < m \leq \frac{x}{W}}}_{(m,v)=1} \sum_{d_1,d_2|m} \mu(d_1>U) \mu(d_2>U)\\
&= \mathop{\mathop{\sum_{r_1<x/WU} \sum_{r_2<x/WU}}_{(r_1,r_2)=1}}_{(r_1 r_2,
v)=1} 
\mathop{
\mathop{\mathop{\sum_l}_{(l,r_1 r_2)=1}}_{r_1 l, r_2 l>U}}_{(\ell,v)=1} \mu(r_1 l) \mu(r_2 l)
 \mathop{\mathop{\sum_{\frac{x}{2W} < m\leq \frac{x}{W}}}_{r_1 r_2 l|m}}_{(m,v)=1} 1,\end{aligned}
\end{equation}
where $d_1 = r_1 l$, $d_2 = r_2 l$, $l=(d_1,d_2)$.
(The inequality $r_1<x/WU$ comes from $r_1 r_2 l|m$, $m\leq x/W$, $r_2 l>U$;
$r_2<x/WU$ is proven in the same way.)
Now (\ref{eq:crusto}) equals
\begin{equation}\label{eq:cudo}\mathop{\sum_{s<\frac{x}{WU}}}_{(s,v)=1}
 \mathop{\mathop{\sum_{r_1< \frac{x}{WUs}} \sum_{r_2<\frac{x}{WUs}}}_{(r_1,r_2)=1}}_{
(r_1 r_2,v)=1}
\mu(r_1) \mu(r_2) \mathop{\mathop{\sum_{
\max\left(\frac{U}{\min(r_1,r_2)},\frac{x/W}{2 r_1 r_2 s}
\right) < l \leq \frac{x/W}{r_1 r_2 s}}}_{
(l,r_1 r_2)=1, (\mu(l))^2 = 1}}_{(\ell,v)=1} 1,\end{equation}
where we have set $s = m/(r_1 r_2 l)$. We begin by simplifying the innermost
triple sum. This we do in the following Lemma; it is not a trivial task,
and carrying it out efficiently actually takes an idea.

\begin{lemma}\label{lem:monro}
Let $z,y>0$. Then
\begin{equation}\label{eq:srto}
\mathop{\mathop{\sum_{r_1<y} \sum_{r_2<y}}_{(r_1,r_2)=1}}_{(r_1 r_2,v)=1}
 \mu(r_1) \mu(r_2) 
 \mathop{
\mathop{\sum_{\min\left(\frac{z/y}{\min(r_1,r_2)}, \frac{z}{2 r_1 r_2}\right)
    < l \leq \frac{z}{r_1 r_2}}}_{(l,r_1r_2)=1,
(\mu(l))^2=1}}_{(\ell,v)=1} 1
\end{equation} equals
\begin{equation}\label{eq:muted}\begin{aligned}
\frac{6 z}{\pi^2}  &\frac{v}{\sigma(v)}
\mathop{\mathop{\sum_{r_1<y}\; \sum_{r_2<y}}_{(r_1,r_2)=1}}_{(r_1 r_2,v)=1} 
\frac{\mu(r_1) \mu(r_2)}{\sigma(r_1) \sigma(r_2)} \left(1 - \max\left(
\frac{1}{2}, \frac{r_1}{y}, \frac{r_2}{y}\right)\right) \\ &+ 
O^*\left(5.08\; \zeta\left(\frac{3}{2}\right)^2 y \sqrt{z} \cdot
\prod_{p|v} \left(1 + \frac{1}{\sqrt{p}}\right) \left(1 - \frac{1}{p^{3/2}}\right)^2\right).\end{aligned}\end{equation}
If $v=2$, the error term in (\ref{eq:muted}) can be replaced by
\begin{equation}\label{eq:mudo}
O^*\left( 1.27 \zeta\left(\frac{3}{2}\right)^2 y \sqrt{z} \cdot
\left(1 + \frac{1}{\sqrt{2}}\right)
\left(1 - \frac{1}{2^{3/2}}\right)^2\right).   
\end{equation}
\end{lemma}
\begin{proof}
By M\"obius inversion, (\ref{eq:srto}) equals
\begin{equation}\label{eq:etze}\begin{aligned}
\mathop{\mathop{\sum_{r_1<y}\; \sum_{r_2<y}}_{(r_1,r_2)=1}}_{(r_1 r_2,v)=1}
 \mu(r_1) \mu(r_2) \mathop{\mathop{\sum_{l \leq \frac{z}{r_1 r_2}}}_{l>
\min\left(\frac{z/y}{\min(r_1,r_2)}, \frac{z}{2 r_1 r_2}\right)}}_{(\ell,v)=1}
  &\mathop{\sum_{d_1|r_1, d_2|r_2}}_{d_1 d_2|l} \mu(d_1) \mu(d_2) \\
&\mathop{\sum_{d_3|v}}_{d_3|l} \mu(d_3)
  \mathop{\sum_{m^2|l}}_{(m,r_1 r_2 v)=1} \mu(m).
\end{aligned}\end{equation}
We can change the order of summation of $r_i$ and $d_i$ by defining
$s_i=r_i/d_i$, and we can also use the obvious fact that the number
of integers in an interval $(a,b\rbrack$ divisible by $d$ 
is $(b-a)/d + O^*(1)$. Thus (\ref{eq:etze}) equals
\begin{equation}\label{eq:dbamain}\begin{aligned}
&\mathop{\mathop{\sum_{d_1, d_2 <y}}_{(d_1, d_2)=1}}_{(d_1 d_2,v)=1}
 \mu(d_1) \mu(d_2)
 \mathop{\mathop{\mathop{\sum_{s_1<y/d_1}}_{s_2<y/d_2}}_{(d_1 s_1, d_2 s_2)=1}}_{
(s_1 s_2,v)=1}
\mu(d_1 s_1) \mu(d_2 s_2)\\
&\sum_{d_3|v} \mu(d_3) 
\mathop{\sum_{m\leq \sqrt{\frac{z}{d_1^2 s_1 d_2^2 s_2 d_3}}}}_{(m,d_1 s_1
  d_2 s_2 v)=1} \frac{\mu(m)}{d_1 d_2 d_3 m^2} \frac{z}{s_1 d_1 s_2 d_2} 
\left(1-\max\left(\frac{1}{2}, \frac{s_1 d_1}{y}, \frac{s_2 d_2}{y}\right)\right)
\end{aligned}\end{equation} 
plus
\begin{equation}\label{eq:dbaerr}
O^*\left(
\mathop{\sum_{d_1, d_2 <y}}_{(d_1 d_2,v)=1}
\mathop{\mathop{\sum_{s_1<y/d_1}}_{s_2<y/d_2}}_{(s_1 s_2,v)=1} \sum_{d_3|v}
\mathop{\sum_{m\leq \sqrt{\frac{z}{d_1^2 s_1 d_2^2 s_2 d_3}}}}_{\text{$m$
sq-free}} 1\right)
.\end{equation} 
If we complete the innermost sum in (\ref{eq:dbamain}) by removing the condition
\[m\leq \sqrt{z/(d_1^2 s d_2^2 s_2)},\] we obtain 
(reintroducing the
variables $r_i = d_i s_i$)
\begin{equation}\label{eq:johnny}\begin{aligned}z\cdot
\mathop{\mathop{\sum_{r_1, r_2<y}}_{(r_1,r_2)=1}}_{(r_1 r_2,v)=1} 
\frac{\mu(r_1) \mu(r_2)}{r_1 r_2} 
&\left(1 - \max\left(
\frac{1}{2}, \frac{r_1}{y}, \frac{r_2}{y}\right)\right) \\
&\mathop{\sum_{d_1|r_1}}_{d_2|r_2} \sum_{d_3|v} 
\mathop{\sum_{m}}_{(m,r_1 r_2 v)=1} \frac{\mu(d_1) \mu(d_2) \mu(m) \mu(d_3)}{d_1
  d_2 d_3 m^2}\end{aligned}\end{equation} times $z$. Now (\ref{eq:johnny}) equals
\[\begin{aligned}
\mathop{\mathop{\sum_{r_1, r_2<y}}_{(r_1,r_2)=1}}_{(r_1 r_2,v)=1} 
&\frac{\mu(r_1) \mu(r_2) z}{r_1 r_2} 
\left(1 - \max\left(
\frac{1}{2}, \frac{r_1}{y}, \frac{r_2}{y}\right)\right) 
\mathop{\prod_{p|r_1 r_2}}_{\text{or $v$}} \left(1 - \frac{1}{p}\right)
\mathop{\prod_{p\nmid r_1 r_2}}_{p\nmid v} \left(1 - \frac{1}{p^2}\right)\\
&= \frac{6 z}{\pi^2} \frac{v}{\sigma(v)} 
\mathop{\mathop{\sum_{r_1, r_2<y}}_{(r_1,r_2)=1}}_{(r_1 r_2,v)=1} 
 \frac{\mu(r_1) \mu(r_2)}{\sigma(r_1) \sigma(r_2)} 
\left(1 - \max\left(
\frac{1}{2}, \frac{r_1}{y}, \frac{r_2}{y}\right)\right),
\end{aligned}\]
i.e., the main term in (\ref{eq:muted}). It remains to estimate
the terms used to complete the sum; their total is, by definition,
 given exactly by (\ref{eq:dbamain}) with the inequality
$m\leq \sqrt{z/(d_1^2 s d_2^2 s_2 d_3)}$ changed to
$m>\sqrt{z/(d_1^2 s d_2^2 s_2 d_3)}$. This is a total of size at most
\begin{equation}\label{eq:shalco}
\frac{1}{2}
\mathop{\sum_{d_1, d_2 <y}}_{(d_1 d_2,v)=1}
 \mathop{\mathop{\sum_{s_1<y/d_1}}_{s_2<y/d_2}}_{(s_1 s_2,v)=1} \sum_{d_3|v}
\mathop{\sum_{m>\sqrt{\frac{z}{d_1^2 s_1 d_2^2 s_2 d_3}}}}_{\text{$m$ sq-free}}
\frac{1}{d_1 d_2 d_3 m^2} \frac{z}{s_1 d_1 s_2 d_2} .\end{equation}
Adding this to (\ref{eq:dbaerr}), we obtain, as our total error term,
\begin{equation}\label{eq:totor}
\mathop{\sum_{d_1, d_2 <y}}_{(d_1 d_2,v)=1}
 \mathop{\mathop{\sum_{s_1<y/d_1}}_{s_2<y/d_2}}_{(s_1 s_2,v)=1} \sum_{d_3|v}
f\left(\sqrt{\frac{z}{d_1^2 s_1 d_2^2 s_2 d_3}}\right),\end{equation}
where \[f(x) := \mathop{\sum_{m\leq x}}_{\text{$m$ sq-free}} 1 +
\frac{1}{2} \mathop{\sum_{m>x}}_{\text{$m$ sq-free}} \frac{x^2}{m^2} .\]
It is easy to see that $f(x)/x$ has a local maximum exactly when $x$
is a square-free (positive) integer. We can hence check that
\[f(x) \leq \frac{1}{2} \left(2 + 2 \left(\frac{\zeta(2)}{\zeta(4)}-1.25\right)
\right) x = 1.26981\dotsc x \]
for all $x\geq 0$ by checking all integers smaller than a constant, using
$\{m: \text{$m$ sq-free}\}\subset \{m: 4\nmid m\}$ and $1.5\cdot (3/4) <
1.26981$ to bound $f$ from below for $x$ larger than a constant.
Therefore, (\ref{eq:totor}) is at most
\[\begin{aligned}
1.27 &\mathop{\sum_{d_1, d_2 <y}}_{(d_1 d_2,v)=1}
 \mathop{\mathop{\sum_{s_1<y/d_1}}_{s_2<y/d_2}}_{(s_1 s_2,v)=1} \sum_{d_3|v}
\sqrt{\frac{z}{d_1^2 s_1 d_2^2 s_2 d_3}}\\
&= 1.27 \sqrt{z}
\prod_{p|v} \left(1 + \frac{1}{\sqrt{p}}\right) \cdot 
\left(\mathop{\sum_{d<y}}_{(d,v)=1} \mathop{\sum_{s<y/d}}_{(s,v)=1}
\frac{1}{d \sqrt{s}}\right)^2.\end{aligned}\]
We can bound the double sum simply by
\[\mathop{\sum_{d<y}}_{(d,v)=1} \sum_{s<y/d} \frac{1}{\sqrt{s} d} \leq
2 \sum_{d<y} \frac{\sqrt{y/d}}{d} \leq
2 \sqrt{y} \cdot \zeta\left(\frac{3}{2}\right) \prod_{p|v} \left(1 - \frac{1}{p^{3/2}}\right).\]
Alternatively, if $v=2$, we bound
\[\mathop{\sum_{s<y/d}}_{(s,v)=1} \frac{1}{\sqrt{s}} = 
\mathop{\sum_{s<y/d}}_{\text{$s$ odd}} \frac{1}{\sqrt{s}} \leq
1 + \frac{1}{2} \int_1^{y/d} \frac{1}{\sqrt{s}} ds = \sqrt{y/d}\]
and thus
\[\mathop{\sum_{d<y}}_{(d,v)=1} 
\mathop{\sum_{s<y/d}}_{(s,v)=1} \frac{1}{\sqrt{s} d} \leq
\mathop{\sum_{d<y}}_{(d,2)=1} \frac{\sqrt{y/d}}{d} \leq \sqrt{y} \left(1 - \frac{1}{2^{3/2}}\right) \zeta\left(\frac{3}{2}\right)
.\]

\end{proof}

Applying Lemma \ref{lem:monro} with $y=S/s$ and $z=x/Ws$, where $S = x/WU$,
we obtain that
(\ref{eq:cudo}) equals
\begin{equation}\label{eq:flatow}\begin{aligned}
&\frac{6 x}{\pi^2 W} \frac{v}{\sigma(v)} \mathop{\sum_{s<S}}_{(s,v)=1}
 \frac{1}{s}
\mathop{\mathop{\sum_{r_1<S/s} \sum_{r_2<S/s}}_{(r_1,r_2)=1}}_{(r_1 r_2,v)=1} 
\frac{\mu(r_1) \mu(r_2)}{\sigma(r_1) \sigma(r_2)} \left(1 - \max\left(
\frac{1}{2}, \frac{r_1}{S/s}, \frac{r_2}{S/s}\right)\right) \\ &+ 
O^*\left(5.04 \zeta\left(\frac{3}{2}\right)^3 S \sqrt{\frac{x}{W}}
\prod_{p|v} \left(1 + \frac{1}{\sqrt{p}}\right)
\left(1 - \frac{1}{p^{3/2}}\right)^3
\right),
\end{aligned}\end{equation}
with $5.04$ replaced by $1.27$ if $v=2$.
The main term in (\ref{eq:flatow}) can be written as 
\begin{equation}\label{eq:fleming}
\frac{6 x}{\pi^2 W} \frac{v}{\sigma(v)} \mathop{\sum_{s\leq S}}_{(s,v)=1}
 \frac{1}{s} \int_{1/2}^1 
\mathop{\mathop{\sum_{r_1\leq \frac{uS}{s}} \sum_{r_2\leq \frac{uS}{s}}}_{(r_1,r_2)=1}}_{(r_1 r_2,v)=1}
\frac{\mu(r_1) \mu(r_2)}{\sigma(r_1) \sigma(r_2)} du .
\end{equation}
As we can see, the use of an integral eliminates the unpleasant factor
\[\left(1 - \max\left(
\frac{1}{2}, \frac{r_1}{S/s}, \frac{r_2}{S/s}\right)\right).\]
From now on, we will focus on the cases $v=1$ and $v=2$ for simplicity. (Higher
values of $v$ do not seem to be really profitable in the last analysis.)

\subsection{Explicit bounds for a sum with $\mu$}

We must estimate the expression within parentheses in (\ref{eq:fleming}).
It is not too hard to show that it tends to $0$; the first part of the proof
of Lemma \ref{lem:yutto} will reduce this to the fact that
$\sum_n \mu(n)/n = 0$. Obtaining good bounds is a more delicate matter.
For our purposes, we will need the expression to converge to $0$ at least
as fast as $1/(\log)^2$, with a good constant in front. For this task, the bound
(\ref{eq:marraki}) on $\sum_{n\leq x} \mu(n)/n$ is enough.

\begin{lemma}\label{lem:yutto}
Let
\[g_v(x) := \mathop{\mathop{\sum_{r_1\leq x} \sum_{r_2\leq x}}_{(r_1,r_2)=1}}_{
(r_1 r_2,v)=1} 
\frac{\mu(r_1) \mu(r_2)}{\sigma(r_1) \sigma(r_2)},\]
where $v=1$ or $v=2$.
Then
\[|g_1(x)| \leq \begin{cases}
1/x &\text{if $33\leq x\leq 10^6$,}\\
\frac{1}{x} (111.536 + 55.768 \log x) &\text{if $10^6\leq x< 10^{10}$,}\\
\frac{0.0044325}{(\log x)^2} + \frac{0.1079}{\sqrt{x}} 
&\text{if $x\geq 10^{10}$,}
\end{cases}\]
\[|g_2(x)| \leq \begin{cases}
2.1/x &\text{if $33\leq x\leq 10^6$,}\\
\frac{1}{x} (1634.34 + 817.168\log x) 
 &\text{if $10^6\leq x< 10^{10}$,}\\
\frac{0.038128}{(\log x)^2}  + \frac{0.2046}{\sqrt{x}}.
 &\text{if $x\geq 10^{10}$.}
\end{cases}\]
\end{lemma}
Tbe proof involves what may be called a version of
 Rankin's trick, using Dirichlet series and the behavior of $\zeta(s)$ near
$s=1$. 
\begin{proof}
We prove the statements for $x\leq 10^6$ by a direct computation,
using interval arithmetic.
 (In fact, in that range, one gets $2.0895071/x$ instead of $2.1/x$.)
Assume from now on that $x>10^6$.

Clearly
\begin{equation}\label{eq:anna}\begin{aligned} 
g(x) &= \mathop{\sum_{r_1\leq x} \sum_{r_2\leq x}}_{(r_1 r_2,v)=1} 
\left(\sum_{d|(r_1,r_2)} \mu(d)\right)
\frac{\mu(r_1) \mu(r_2)}{\sigma(r_1) \sigma(r_2)}
\\ &= \mathop{\sum_{d\leq x}}_{(d,v)=1} 
\mu(d) \mathop{\mathop{\sum_{r_1\leq x} \sum_{r_2\leq x}}_{d|(r_1,r_2)}}_{
(r_1 r_2, v)=1}
\frac{\mu(r_1) \mu(r_2)}{\sigma(r_1) \sigma(r_2)}\\
&= \mathop{\sum_{d\leq x}}_{(d,v)=1} \frac{\mu(d)}{(\sigma(d))^2}
 \mathop{\sum_{u_1\leq x/d}}_{(u_1,d v)=1} \mathop{\sum_{u_2\leq x/d}}_{(u_2,d v)=1}
\frac{\mu(u_1) \mu(u_2)}{\sigma(u_1) \sigma(u_2)} \\ 
&= \mathop{\sum_{d\leq x}}_{(d,v)=1} \frac{\mu(d)}{(\sigma(d))^2}
\left(\mathop{\sum_{r\leq x/d}}_{(r,d v)=1} 
\frac{\mu(r)}{\sigma(r)}\right)^2.\end{aligned}\end{equation}
Moreover,
\[\begin{aligned}
\mathop{\sum_{r\leq x/d}}_{(r,d v)=1} 
\frac{\mu(r)}{\sigma(r)} &=
\mathop{\sum_{r\leq x/d}}_{(r,d v)=1} 
\frac{\mu(r)}{r} \sum_{d'|r} \prod_{p|d'} 
  \left(\frac{p}{p+1}-1\right)\\
&= \mathop{\mathop{\sum_{d'\leq x/d}}_{\mu(d')^2=1}}_{(d',d v)=1}
 \left(\prod_{p|d'} \frac{-1}{p+1}\right)
  \mathop{\mathop{\sum_{r\leq x/d}}_{(r,d v)=1}}_{d'|r} 
\frac{\mu(r)}{r}\\
&= \mathop{\mathop{\sum_{d'\leq x/d}}_{\mu(d')^2=1}}_{(d',d v)=1} 
\frac{1}{d' \sigma(d')}
  \mathop{\sum_{r\leq x/dd'}}_{(r,dd' v)=1} \frac{\mu(r)}{r}\end{aligned}\]
and
\[
\mathop{\sum_{r\leq x/dd'}}_{(r,dd'v)=1} \frac{\mu(r)}{r} =
\mathop{\sum_{d''\leq x/dd'}}_{d''|(d d'v)^{\infty}}
 \frac{1}{d''} \sum_{r\leq x/d d' d''} \frac{\mu(r)}{r}.\]
Hence
\begin{equation}\label{eq:onno}|g(x)|\leq \mathop{\sum_{d\leq x}}_{(d,v)=1}
 \frac{(\mu(d))^2}{(\sigma(d))^2}
 \left(\mathop{\mathop{\sum_{d'\leq x/d}}_{\mu(d')^2=1}}_{(d',d v)=1}
\frac{1}{d' \sigma(d')} \mathop{\sum_{d''\leq x/d d'}}_{d''|(d d' v)^{\infty}}
 \frac{1}{d''} f(x/d d' d'')\right)^2,\end{equation}
where $f(t) = \left|\sum_{r\leq t} \mu(r)/r\right|$.

We intend to bound the function $f(t)$ by 
a linear combination of terms of the form $t^{-\delta}$, 
$\delta\in \lbrack 0,1/2)$. Thus it makes sense now to estimate
$F_v(s_1,s_2,x)$, defined to be
the quantity
\[\begin{aligned}
\mathop{\sum_d}_{(d,v)=1} \frac{(\mu(d))^2}{(\sigma(d))^2} 
&\left(\mathop{\sum_{d_1'}}_{(d_1',d v)=1} 
\frac{\mu(d_1')^2}{d_1' \sigma(d_1')}
\sum_{d_1''|(d d_1' v)^{\infty}}
 \frac{1}{d_1''} \cdot (d d_1' d_1'')^{1-s_1}\right)\\
 &\left(\mathop{\sum_{d_2'}}_{(d_2',d v)=1} 
\frac{\mu(d_2')^2}{d_2' \sigma(d_2')}
\sum_{d_2''|(d d_2' v)^{\infty}}
 \frac{1}{d_2''} \cdot (d d_2' d_2'')^{1-s_2}\right) .
\end{aligned}\]
for $s_1,s_2\in \lbrack 1/2,1\rbrack$. This is equal to
\[\begin{aligned}
\mathop{\sum_d}_{(d,v)=1} &\frac{\mu(d)^2}{d^{s_1+s_2}} \prod_{p|d}
\frac{1}{\left(1+p^{-1}\right)^2 \left(1-p^{-s_1}\right) \prod_{p|v}
\frac{1}{(1-p^{-{s_1}}) (1 - p^{-s_2})}
\left(1-p^{-s_2}\right)}\\
&\cdot \left(\mathop{\sum_{d'}}_{(d',d v)=1} 
\frac{\mu(d')^2}{(d')^{s_1+1}} \prod_{p'|d'}
\frac{1}{\left(1+p'^{-1}\right) \left(1-p'^{-s_1}\right)}\right)\\
&\cdot \left(\mathop{\sum_{d'}}_{(d',d v)=1} 
\frac{\mu(d')^2}{(d')^{s_2+1}} \prod_{p'|d'}
\frac{1}{\left(1+p'^{-1}\right) \left(1-p'^{-s_2}\right)}\right) ,
\end{aligned}\]
which in turn can easily be seen to equal
\begin{equation}\label{eq:prof}\begin{aligned}
\prod_{p\nmid v} &\left(1 + 
\frac{p^{-s_1} p^{-s_2}}{(1 - p^{-s_1} +p^{-1})
(1 - p^{-s_2} +p^{-1})}\right)
\prod_{p|v} \frac{1}{(1 - p^{-s_1}) (1 - p^{-s_2})}
\\
&\cdot \prod_{p\nmid v} \left(1 + \frac{p^{-1} p^{-s_1}}{(1+p^{-1}) 
(1-p^{-s_1})}\right)
\cdot \prod_{p\nmid v} \left(1 + \frac{p^{-1} p^{-s_2}}{(1+p^{-1}) 
(1-p^{-s_2})}\right)
\end{aligned}\end{equation}
Now, for any $0<x\leq y\leq x^{1/2}<1$, 
\[(1+x-y) (1-xy) (1-xy^2) - (1+x) (1-y) (1- x^3) = (x-y) ( y^2-x) ( xy - x
-1) x \leq 0,\] and so
\begin{equation}1 + \frac{xy}{(1+x)(1-y)} = 
\frac{(1+x-y) (1-xy) (1-x y^2)}{(1+x) (1-y) (1-xy) (1-xy^2)} \leq
\frac{(1-x^3)}{(1-xy) (1-xy^2)}.\end{equation}

For any $x\leq y_1,y_2<1$ with $y_1^2\leq x$, $y_2^2\leq x$,
\begin{equation}\label{eq:odmalicka}
1 + \frac{y_1 y_2}{(1-y_1+x) (1-y_2+x)} \leq
 \frac{(1-x^3)^2 (1-x^4)}{(1- y_1 y_2) (1 - y_1 y_2^2) (1 - y_1^2 y_2)} .
\end{equation}
This can be checked as follows: multiplying by the denominators and changing
variables to $x$, $s=y_1+y_2$ and $r=y_1 y_2$, we obtain an inequality where
the left side, quadratic on $s$ with positive leading coefficient, must be
less than or equal to the right side, which is linear on $s$. The left side
minus the right side can be maximal for given $x$, $r$ only when $s$ is
maximal or minimal. This happens when $y_1=y_2$ or when either
 $y_i = \sqrt{x}$ or $y_i = x$ for at least one of $i=1,2$. In each of these
cases, we have reduced (\ref{eq:odmalicka}) to an inequality in two variables
that can be proven automatically\footnote{In practice, the case $y_i = \sqrt{x}$
leads to a polynomial of high degree, and quantifier elimination increases
sharply in complexity as the degree increases; a stronger inequality of
lower degree (with $(1-3 x^3)$ instead of $(1-x^3)^2 (1-x^4)$) 
was given to QEPCAD to prove in this case.} 
by a quantifier-elimination program; the 
author has used QEPCAD \cite{QEPCAD} to do this.

Hence $F_v(s_1,s_2,x)$ is at most
\begin{equation}\label{eq:tausend}\begin{aligned}
\prod_{p\nmid v} &\frac{(1-p^{-3})^2 (1-p^{-4})}{(1-p^{-s_1-s_2}) (1-p^{-2s_1-s_2})
(1-p^{-s_1-2s_2})} \cdot \prod_{p|v} \frac{1}{(1-p^{-s_1}) (1-p^{-s_2})}
 \\ &\cdot \prod_{p\nmid v}
\frac{1-p^{-3}}{(1+p^{-s_1-1}) (1 + p^{-2s_1-1})}
\prod_{p\nmid v}
\frac{1-p^{-3}}{(1+p^{-s_2 -1}) (1 + p^{-2 s_2 - 1})}\\
&= C_{v,s_1,s_2} \cdot \frac{
\zeta(s_1+1) \zeta(s_2+1)
\zeta(2s_1+1) \zeta(2s_2 +1 ) 
}{\zeta(3)^4 \zeta(4)
(\zeta(s_1+s_2)  \zeta(2 s_1 + s_2) \zeta(s_1 + 2 s_2))^{-1}
},
\end{aligned}\end{equation}
where $C_{v,s_1,s_2}$ equals $1$ if $v=1$, and
\[\frac{
 (1-2^{-s_1-2s_2}) 
(1+2^{-s_1-1}) 
(1+2^{-2s_1-1}) (1+ 2^{-s_2-1}) (1+2^{-2s_2-1})}{(1-2^{-{s_1+s_2}})^{-1} (1-2^{-2s_1-s_2})^{-1}
(1-2^{-s_1})(1-2^{-s_2}) (1-2^{-3})^4
(1-2^{-4})}\] if $v=2$.

For $1\leq t\leq x$,
(\ref{eq:marraki}) and (\ref{eq:ramare}) imply
\begin{equation}\label{eq:kustor}
f(t) \leq \begin{cases} \sqrt{\frac{2}{t}} & \text{if $x\leq 10^{10}$}\\
\sqrt{\frac{2}{t}} + \frac{0.03}{\log x} \left(\frac{x}{t}\right)^{ 
  \frac{\log \log 10^{10}}{\log x - \log 10^{10}}} &\text{if $x>10^{10}$},
\end{cases}
\end{equation}
where we are using the fact that $\log x$ is convex-down. Note that,
again by convexity,
\[\frac{\log \log x - \log \log 10^{10}}{\log x - \log 10^{10}}
< (\log t)'|_{t=\log 10^{10}} = \frac{1}{\log 10^{10}} = 0.0434294\dots
\]
Obviously, $\sqrt{2/t}$ in (\ref{eq:kustor}) can be replaced by $(2/t)^{1/2-
\epsilon}$ for any $\epsilon\geq 0$. 

By (\ref{eq:onno}) and (\ref{eq:kustor}),
\[|g_v(x)|\leq \left(\frac{2}{x}\right)^{1-2\epsilon}
 F_v(1/2+\epsilon,1/2+\epsilon,x)\]
for $x\leq 10^{10}$. We set $\epsilon=1/\log x$ and obtain
from (\ref{eq:tausend}) that
\begin{equation}\label{eq:rot}\begin{aligned}
F_v(1/2+\epsilon,1/2+\epsilon,x) &\leq
C_{v,\frac{1}{2} + \epsilon,\frac{1}{2} + \epsilon}
\frac{\zeta(1+2\epsilon) \zeta(3/2)^4 \zeta(2)^2}{\zeta(3)^4
\zeta(4)}\\
&\leq 55.768\cdot C_{v,\frac{1}{2} + \epsilon,\frac{1}{2} + \epsilon}
\cdot \left(1 + \frac{\log x}{2}\right) 
,\end{aligned}\end{equation}
where we use the easy bound $\zeta(s)< 1+ 1/(s-1)$ obtained by
\[\sum n^s < 1 + \int_1^{\infty} t^s dt.\] (For sharper bounds, see
\cite{MR1924708}.) Now
\[\begin{aligned}
C_{2,\frac{1}{2}+\epsilon,\frac{1}{2} + \epsilon} &\leq
\frac{
 (1-2^{-3/2-\epsilon})^2 
(1+2^{-3/2})^2  (1+2^{-2})^2 (1-2^{-1-2\epsilon})}{
(1-2^{-1/2})^2 (1-2^{-3})^4 (1-2^{-4})}\\
&\leq 14.652983
,\end{aligned}\]
whereas $C_{1,\frac{1}{2}+\epsilon,\frac{1}{2}+\epsilon}=1$.
(We are assuming $x\geq 10^6$, and so $\epsilon\leq 1/(\log 10^6)$.)
Hence
\[|g_v(x)|\leq
\begin{cases} 
\frac{1}{x} (111.536 + 55.768\log x) &\text{if $v=1$,}\\
\frac{1}{x} (1634.34 + 817.168\log x) 
& \text{if $v=2$.}\end{cases}\] 
for $10^6\leq x< 10^{10}$.

For general $x$, we must use the second bound in (\ref{eq:kustor}).
Define $c = 1/(\log 10^{10})$. We see that, if $x>10^{10}$,
\[\begin{aligned}
|g_v(x)|&\leq \frac{0.03^2}{(\log x)^2} F_1(1-c,1-c) \cdot C_{v,1-c,1-c}\\&+ 2
\cdot \frac{\sqrt{2}}{\sqrt{x}} \frac{0.03}{\log x} F(1-c,1/2) 
\cdot C_{v,1-c,1/2}\\ &+
\frac{1}{x} (111.536 + 55.768 \log x)\cdot C_{v,\frac{1}{2}+\epsilon,
\frac{1}{2}+\epsilon}.\end{aligned}\]
For $v=1$, this gives
\[\begin{aligned}
|g_1(x)| 
&\leq \frac{0.0044325}{(\log x)^2} + \frac{2.1626}{\sqrt{x} \log x} +
\frac{1}{x} (111.536 + 55.768 \log x)\\
&\leq \frac{0.0044325}{(\log x)^2} + \frac{0.1079}{\sqrt{x}} ;
\end{aligned}\]
for $v=2$, we obtain
\[\begin{aligned}
|g_2(x)| &\leq
\frac{0.038128}{(\log x)^2} 
+ \frac{25.607}{\sqrt{x} \log x} + 
\frac{1}{x} (1634.34 + 817.168\log x) \\
&\leq \frac{0.038128}{(\log x)^2}  + \frac{0.2046}{\sqrt{x}}.
\end{aligned}\]
\end{proof}

\subsection{Estimating the triple sum}
We will now be able to bound the triple sum in (\ref{eq:fleming}), viz.,
\begin{equation}\label{eq:grotto}
\mathop{\sum_{s\leq S}}_{(s,v)=1} \frac{1}{s} \int_{1/2}^{1} g_v(uS/s) du,\end{equation}
where $g_v$ is as in Lemma \ref{lem:yutto}.

As we will soon see, Lemma \ref{lem:yutto} that (\ref{eq:grotto}) is bounded
by a constant (essentially because the integral $\int_0^{1/2} 1/t(\log t)^2$
converges).
We must give as good a constant as we can, since it will affect the largest
term in the final result.

Clearly $g_v(R) = g_v(\lfloor R\rfloor)$. The contribution of each $g_v(m)$,
$1\leq m\leq S$, to (\ref{eq:grotto}) is exactly $g_v(m)$ times
\begin{equation}\label{eq:greco}\begin{aligned}
&\mathop{\sum_{\frac{S}{m+1} < s \leq \frac{S}{m}} \frac{1}{s}}_{(s,v)=1} 
\int_{ms/S}^1 1 du + 
\mathop{\sum_{\frac{S}{2 m} < s \leq \frac{S}{m+1}} \frac{1}{s}}_{(s,v)=1}  
\int_{ms/S}^{(m+1)s/S} 1 du\\ 
&+ \mathop{\sum_{\frac{S}{2 (m+1)} < s \leq \frac{S}{2 m}} \frac{1}{s}}_{(s,v)=1}  
\int_{1/2}^{(m+1)s/S} du = 
\mathop{
\sum_{\frac{S}{m+1} < s \leq \frac{S}{m}}}_{(s,v)=1} 
 \left(\frac{1}{s} - \frac{m}{S}\right) \\
&+ \mathop{\sum_{\frac{S}{2 m} < s \leq \frac{S}{m+1}}}_{(s,v)=1}  \frac{1}{S} +
\mathop{\sum_{\frac{S}{2 (m+1)} < s \leq \frac{S}{2 m}}}_{(s,v)=1} 
 \left(\frac{m+1}{S} - \frac{1}{2s}\right).
\end{aligned}\end{equation}
Write $f(t) = 1/S$ for $S/2m < t\leq S/(m+1)$, $f(t)=0$ for $t>S/m$ or
$t<S/2(m+1)$, $f(t) = 1/t - m/S$ for $S/(m+1) <t\leq S/m$ and 
$f(t) = (m+1)/S - 1/2t$ for $S/2(m+1) < t\leq S/2m$; then (\ref{eq:greco})
equals $\sum_{n: (n,v)=1} f(n)$. By Euler-Maclaurin (second order),
\begin{equation}\label{eq:etex}\begin{aligned}
\sum_n f(n) &= \int_{-\infty}^{\infty} f(x) - \frac{1}{2} B_2(\{ x\}) f''(x) dx
= \int_{-\infty}^{\infty} f(x) + O^*\left(\frac{1}{12} |f''(x)|\right) dx\\
&= \int_{-\infty}^{\infty} f(x) dx + \frac{1}{6} \cdot O^*\left(\left|f'\left(
\frac{3}{2m}\right)\right| + \left|f'\left(\frac{s}{m+1}\right)\right|\right)\\
&= \frac{1}{2} \log\left(1 + \frac{1}{m}\right) +\frac{1}{6}\cdot
 O^*\left(\left(\frac{2 m}{s}\right)^2 + \left(\frac{m+1}{s}\right)^2\right).\end{aligned}\end{equation}
Similarly,
\[\begin{aligned}
\sum_{\text{$n$ odd}} f(n) &= \int_{-\infty}^{\infty} f(2x+1) - \frac{1}{2} B_2(\{ x\}) 
\frac{d^2 f(2x+1)}{d x^2} dx \\ &= \frac{1}{2} \int_{-\infty}^{\infty} f(x) dx -
2\int_{-\infty}^{\infty} \frac{1}{2} B_2\left(\left\{\frac{x-1}{2}\right\}\right)
f''(x) dx\\ &= \frac{1}{2} \int_{-\infty}^{\infty} f(x) dx + 
\frac{1}{6}  \int_{-\infty}^\infty
O^*\left(
|f''(x)|\right) dx \\ &=
\frac{1}{4} \log\left(1 + \frac{1}{m}\right) + \frac{1}{3} 
\cdot O^*\left(\left(\frac{2 m}{s}\right)^2 + \left(\frac{m+1}{s}\right)^2\right).
\end{aligned}\]

We use these expressions for $m\leq C_0$, where $C_0\geq 33$ is a
constant to be computed later; they will give us the main term. For $m> C_0$,
we use the bounds on $|g(m)|$ that Lemma \ref{lem:yutto} gives us.

(Starting now and for the rest of the paper, we will focus on the cases
$v=1$, $v=2$ when giving explicit computational estimates. All of our 
procedures would allow higher values of $v$ as well, but, as will become clear
much later, the gains 
from higher values of $v$ are offset by losses and complications elsewhere.)

Let us estimate (\ref{eq:grotto}).
Let \[c_{v,0} = \begin{cases} 1/6 &\text{if $v=1$,}\\ 1/3 &\text{if
    $v=2$,}\end{cases}
\;\;\;\;c_{v,1} = \begin{cases} 1 &\text{if $v=1$,}\\ 2.5 &\text{if
    $v=2$,}\end{cases}\]
\[c_{v,2} = \begin{cases} 55.768\dotsc &\text{if $v=1$,}\\ 817.168\dotsc &\text{if
    $v=2$,}\end{cases}\;\;\;\;
c_{v,3} = \begin{cases} 111.536\dotsc &\text{if $v=1$,}\\ 1634.34\dotsc &\text{if
    $v=2$,}\end{cases}\]
\[c_{v,4} = \begin{cases} 0.0044325\dotsc &\text{if $v=1$,}\\ 0.038128\dotsc&\text{if
    $v=2$,}\end{cases}\;\;\;\;
c_{v,5} = \begin{cases}  0.1079\dotsc&\text{if $v=1$,}\\ 0.2046\dotsc &\text{if
    $v=2$.}\end{cases}\]
Then (\ref{eq:grotto}) equals
\[\begin{aligned}
\sum_{m\leq C_0} &g_v(m) \cdot\left(\frac{\phi(v)}{2 v} \log\left(1+ \frac{1}{m}\right)
+ O^*\left(c_{v,0} \frac{5 m^2 + 2 m+1}{S^2}\right)\right)\\
&+ \sum_{S/10^6 \leq s < S/C_0} \frac{1}{s} \int_{1/2}^1 
O^*\left(\frac{c_{v,1}}{uS/s}\right) du \\ &+ 
\sum_{S/10^{10}\leq s<S/10^6} \frac{1}{s} 
\int_{1/2}^1 
O^*\left(
\frac{c_{v,2} \log(uS/s) + c_{v,3}}{uS/s}\right) 
du  \\ &+
\sum_{s < S/10^{10}} \frac{1}{s} \int_{1/2}^1 O^*\left(\frac{c_{v,4}}{(\log uS/s)^2}
+ \frac{c_{v,5}}{\sqrt{uS/s}}\right) du,\end{aligned}\]
which is
\[\begin{aligned}
\sum_{m\leq C_0} &g_v(m) \cdot \frac{\phi(v)}{2 v} \log\left(1+ \frac{1}{m}\right)
+ \sum_{m\leq C_0} |g(m)|\cdot O^*\left(
c_{v,0} \frac{5 m^2 + 2 m+1}{S^2}\right)\\
&+ O^*\left(c_{v,1} \frac{\log 2}{C_0}
    + \frac{\log 2}{10^6} \left( c_{v,3} + c_{v,2} (1 + \log 10^6)\right)
+ \frac{2-\sqrt{2}}{10^{10/2}} c_{v,5}\right)
\\ &+ O^*\left(\sum_{s<S/10^{10}} \frac{c_{v,4}/2}{s (\log
    S/2s)^2}\right)
\end{aligned}\]
for $S\geq (C_0+1)$. 
Note that $\sum_{s<S/10^{10}} \frac{1}{s (\log S/2s)^2} = 
\int_0^{2/10^{10}} \frac{1}{t (\log t)^2} dt$.


Now 
\[\frac{c_{v,4}}{2} \int_0^{2/10^{10}} \frac{1}{t (\log t)^2} dt =
\frac{c_{v,4}/2}{\log(10^{10}/2)} = 
\begin{cases} 0.00009923\dotsc &\text{if $v=1$}\\
0.000853636\dotsc &\text{if $v=2$.}\end{cases}\]
and
\[\frac{\log 2}{10^6} \left( c_{v,3} + c_{v,2} (1 + \log 10^6)\right)
+ \frac{2-\sqrt{2}}{10^{5}} c_{v,5} = 
\begin{cases}
0.0006506\dotsc &\text{if $v=1$}\\
0.009525\dotsc &\text{if $v=2$.}\end{cases}\]
For $C_0 = 10000$,
\[\begin{aligned}
\frac{\phi(v)}{v} \frac{1}{2} \sum_{m\leq C_0} g_v(m) \cdot
               \log\left(1+ \frac{1}{m}\right)
&= \begin{cases} 0.362482\dotsc &\text{if $v=1$,}\\
0.360576\dotsc &\text{if $v=2$,}\end{cases}\\
c_{v,0} \sum_{m\leq C_0} |g_v(m)| (5 m^2 + 2 m + 1) &\leq 
\begin{cases} 6204066.5\dotsc &\text{if $v=1$,}\\
15911340.1\dotsc &\text{if $v=2$,}\end{cases}
\end{aligned}\]
and 
\[c_{v,1}\cdot (\log 2)/C_0 = \begin{cases} 0.00006931\dotsc
&\text{if $v=1$,}\\ 0.00017328\dotsc &\text{if $v=2$.}\end{cases}\]

Thus, for $S\geq 100000$, 
\begin{equation}\label{eq:corto}
\mathop{\sum_{s\leq S}}_{(s,v)=1} \frac{1}{s} \int_{1/2}^{1} g_v(uS/s) du
\leq \begin{cases}0.36393
&\text{if $v=1$,}\\ 
0.37273 &\text{if $v=2$.}\end{cases}
\end{equation}

For $S < 100000$, we proceed as above, but using the exact expression
(\ref{eq:greco}) instead of (\ref{eq:etex}). Note (\ref{eq:greco}) is
of the form $f_{s,m,1}(S) + f_{s,m,2}(S)/S$, where both $f_{s,m,1}(S)$ and
$f_{s,m,2}(S)$ depend only on $\lfloor S\rfloor$ (and on $s$ and $m$).
Summing over $m\leq S$, we obtain a bound of the form
\[\mathop{\sum_{s\leq S}}_{(s,v)=1} \frac{1}{s} \int_{1/2}^{1} g_v(uS/s) du \leq
G_v(S)\]
with \[G_v(S) = K_{v,1}(|S|) + K_{v,2}(|S|)/S,\]
where $K_{v,1}(n)$ and $K_{v,2}(n)$ can be computed explicitly for 
each integer $n$. (For example, $G_v(S) = 1 - 1/S$ for
$1\leq S < 2$ and $G_v(S) = 0$ for $S<1$.)

It is easy to check numerically
 that this implies that (\ref{eq:corto}) holds not just
for $S\geq 100000$ but also for $40\leq S<100000$ (if $v=1$) or
$16\leq S < 100000$ (if $v=2$). Using the fact that $G_v(S)$ is non-negative,
we can compare $\int_1^T G_v(S) dS/S$ with $\log(T+1/N)$ for each
$T\in \lbrack 2,40\rbrack \cap \frac{1}{N} \mathbb{Z}$ ($N$ a large integer)
to show, again numerically, that
\begin{equation}\label{eq:passi}
\int_1^T G_v(S) \frac{dS}{S} \leq \begin{cases} 0.3698 \log T
&\text{if $v=1$,}\\
0.37273 \log T &\text{if $v=2$.}\end{cases}\end{equation}
(We use $N=100000$ for $v=1$; already $N=10$ gives us the answer above for
$v=2$. Indeed, computations suggest the better bound $0.358$ instead of
$0.37273$; we are committed to using 
$0.37273$ because of (\ref{eq:corto}).)

Multiplying by $6 v/\pi^2\sigma(v)$, we conclude that
\begin{equation}\label{eq:menson1}S_1(U,W) = 
\frac{x}{W} \cdot H_1\left(\frac{x}{W U}\right)
+ O^*\left(5.08 \zeta(3/2)^3 \frac{x^{3/2}}{W^{3/2} U} \right)
\end{equation}
if $v=1$,
\begin{equation}\label{eq:menson2}S_1(U,W) = 
\frac{x}{W} \cdot H_2\left(\frac{x}{W U}\right)
+O^*\left(1.27 \zeta(3/2)^3 \frac{x^{3/2}}{W^{3/2} U} \right)
\end{equation}
if $v=2$, where 
\begin{equation}\label{eq:palmiped}
H_1(S) = \begin{cases} \frac{6}{\pi^2} G_1(S)
 &\text{if $1\leq S < 40$,}\\ 0.22125 &\text{if $S\geq 40$,}
\end{cases}\;\;\;\;\;\;\;\;
H_2(s) = \begin{cases} \frac{4}{\pi^2} G_2(S) 
&\text{if $1\leq S < 16$,}\\ 0.15107 &\text{if $S\geq 16$.}
\end{cases}\end{equation}
Hence (by (\ref{eq:passi}))
\begin{equation}\label{eq:velib}\begin{aligned}
\int_1^{T} H_v(S) \frac{dS}{S} &\leq
\begin{cases} 0.22482 \log T &\text{if $v=1$,}\\
0.15107 \log T &\text{if $v=2$;}\end{cases}
\end{aligned}\end{equation}
 moreover,
\begin{equation}\label{eq:demimond}
H_1(S)\leq \frac{3}{\pi^2},\;\;\;\;\;\;\;\;\;\;
H_2(S)\leq \frac{2}{\pi^2}\end{equation} for all $S$.
\begin{center}
* * *
\end{center}

{\em Note.} There is another way to obtain cancellation on $\mu$, applicable
when $(x/W)> Uq$ (as is unfortunately never the case in our main
application). For this alternative
to be taken, one must either apply Cauchy-Schwarz on $n$ rather than $m$
(resulting in exponential sums over $m$) or lump together all $m$ near each 
other and in the same
congruence class modulo $q$ before applying Cauchy-Schwarz on $m$ (one can
indeed do this if $\delta$ is small). We could then write
\[\mathop{\sum_{m\sim W}}_{m\equiv r \mo q} \mathop{\sum_{d|m}}_{d>U} \mu(d) =
 - \mathop{\sum_{m\sim W}}_{m\equiv r \mo q} \mathop{\sum_{d|m}}_{d\leq U} \mu(d) =
 -\sum_{d\leq U} \mu(d) (W/qd + O(1))\]
and obtain cancellation on $d$. If $Uq\geq (x/W)$, however,
the error term dominates.

\section{The sum $S_2$: the large sieve, primes and tails}
We must now bound
\begin{equation}\label{eq:honi}
S_2(U',W',W) = 
\mathop{\sum_{U'<m\leq \frac{x}{W}}}_{(m,v)=1} 
\left|\sum_{W' < p\leq W} (\log p)
 e(\alpha m p)\right|^2.\end{equation}
for $U' = \max(U,x/2W)$, $W' = \max(V,W/2)$.
(The condition $(p,v)=1$ will be fulfilled automatically by the
assumption $V>v$.)

From a modern perspective, this is clearly a case for a large sieve. 
It is also clear that we ought to try to apply
a large sieve for sequences of prime support. What is subtler here is how to do
things well for very large $q$ (i.e., $x/q$ small). This is in some sense
a dual problem to that of $q$ small, but it poses additional complications;
for example, it is not obvious how to take advantage of prime support
for very large $q$.

As in type I, we avoid this entire issue by forbidding $q$ large and then taking
advantage of the error term $\delta/x$ in the approximation 
$\alpha = \frac{a}{q} + \frac{\delta}{x}$. This is one of the main innovations
here. Note this alternative method will allow us to take advantage of prime
support.

A key situation to study is that of
frequencies $\alpha_i$ clustering around given rationals $a/q$ while
nevertheless keeping at a certain small distance from each other.

\begin{lemma}\label{lem:ogor}
Let $q\geq 1$.
Let $\alpha_1,\alpha_2,\dotsc,\alpha_k\in \mathbb{R}/\mathbb{Z}$ be of
the form $\alpha_i = a_i/q + \upsilon_i$, $0\leq a_i<q$, 
where the elements $\upsilon_i\in \mathbb{R}$
all lie in an interval of length $\upsilon>0$, and where $a_i=a_j$ implies
$|\upsilon_i - \upsilon_j|>\nu>0$. Assume $\nu+\upsilon\leq 1/q$.
Then, for any $W, W'\geq 1$, $W'\geq W/2$,
\begin{equation}\label{eq:crut}\begin{aligned}
\sum_{i=1}^{k} \left|\sum_{W'<p\leq W} (\log p) e(\alpha_i p)\right|^2
&\leq \min\left(1, \frac{2 q}{\phi(q)} \frac{1}{\log\left((q (\nu+\upsilon))^{-1}
\right)}\right)\\ &\cdot \left(W-W'+\nu^{-1}\right) \sum_{W'<p\leq W} (\log p)^2.
\end{aligned}
\end{equation}
\end{lemma}
\begin{proof}
For any distinct $i$, $j$, the angles $\alpha_i$, $\alpha_j$ are separated
by at least $\nu$ (if $a_i=a_j$) or at least $1/q - |\upsilon_i-\upsilon_j|\geq
1/q-\upsilon\geq \nu$ (if $a_i\ne a_j$). Hence we can apply the large sieve
(in the optimal $N+\delta^{-1}-1$
form due to Selberg \cite{Sellec} and Montgomery-Vaughan
\cite{MR0337775})
and obtain the bound in (\ref{eq:crut}) with $1$ instead of $\min(1,\dotsc)$
immediately.

We can also apply Montgomery's inequality
(\cite{MR0224585}, \cite{MR0311618}; see the expositions in
\cite[pp. 27--29]{MR0337847} and  \cite[\S 7.4]{MR2061214}). 
This gives us that the left side of (\ref{eq:crut}) is at most
\begin{equation}\label{eq:toyor}
\left(\mathop{\sum_{r\leq R}}_{(r,q)=1} \frac{(\mu(r))^2}{\phi(r)}
\right)^{-1} 
\mathop{\sum_{r\leq R}}_{(r,q)=1} \mathop{\sum_{a' \mo r}}_{(a',r)=1} 
 \sum_{i=1}^{k} \left|\sum_{W'<p\leq W} (\log p) e((\alpha_i+a'/r) p)\right|^2
\end{equation}
If we add all possible fractions
of the form $a'/r$, $r\leq R$, $(r,q)=1$, to the fractions $a_i/q$,
we obtain fractions that are
separated by at least $1/qR^2$. If $\nu+\upsilon\geq 1/qR^2$, then the resulting
angles $\alpha_i + a'/r$ are still separated by at least $\nu$. Thus
we can apply the large sieve to (\ref{eq:toyor}); setting $R = 
1/\sqrt{(\nu+\upsilon) q}$, we see that we gain a factor of
\begin{equation}\label{eq:werst}
\mathop{\sum_{r\leq R}}_{(r,q)=1} \frac{(\mu(r))^2}{\phi(r)} \geq
\frac{\phi(q)}{q} \sum_{r\leq R} \frac{(\mu(r))^2}{\phi(r)} \geq
\frac{\phi(q)}{q} \sum_{d\leq R} \frac{1}{d} \geq
\frac{\phi(q)}{2 q} \log\left((q (\nu+\upsilon))^{-1}\right),\end{equation}
since $\sum_{d\leq R} 1/d \geq \log(R)$ for all $R\geq 1$ (integer or not).
\end{proof}

Let us first give a bound on sums of the type of 
$S_2(U,V,W)$ using prime support
but not the error terms (or Lemma \ref{lem:ogor}). This is something that
can be done very well using tools available in the literature. (Not
all of these tools seem to be known as widely as they should be.) 
Bounds (\ref{eq:pokor1}) and (\ref{eq:zerom}) are completely standard
large-sieve bounds. To obtain the gain of a factor of $\log$ in
(\ref{eq:pokor2}), we use a lemma of Montgomery's, for whose modern proof
(containing an improvement by Huxley) we refer to the standard source
\cite[Lemma 7.15]{MR2061214}. The purpose of Montgomery's lemma is precisely
to gain a factor of $\log$ in applications of the large sieve to sequences
supported on the primes. To use the lemma efficiently, we apply Montgomery
and Vaughan's large sieve with weights \cite[(1.6)]{MR0374060}, rather
than more common forms of the large sieve. (The idea -- used in \cite{MR0374060} to prove an improved version of the Brun-Titchmarsh inequality -- 
is that Farey fractions
(rationals with bounded denominator) are not equidistributed;
this fact can be exploited if a large sieve with weights is used.)
\begin{lemma}\label{lem:kastor1}
Let $W\geq 1$, $W'\geq W/2$. Let 
$\alpha = a/q + O^*(1/q Q)$, $q\leq Q$.  Then
\begin{equation}\label{eq:pokor1}\begin{aligned}
\sum_{A_0<m\leq A_1} &\left|\sum_{W'<p\leq W} (\log p) e(\alpha m p)\right|^2
\\ &\leq 
\left\lceil \frac{A_1-A_0}{\min(q,\lceil Q/2\rceil)} \right\rceil\cdot
(W-W'+2q) \sum_{W'<p\leq W} (\log p)^2.\end{aligned}\end{equation}
If $q<W/2$ and $Q\geq 3.5 W$, the following bound also holds:
\begin{equation}\label{eq:pokor2}\begin{aligned}
\sum_{A_0<m\leq A_1} &\left|\sum_{W'<p\leq W} (\log p) e(\alpha m p)\right|^2\\
&\leq \left\lceil \frac{A_1-A_0}{q}\right\rceil\cdot
\frac{q}{\phi(q)} \frac{W}{\log(W/2q)}
\cdot \sum_{W'<p\leq W} (\log p)^2.\end{aligned}\end{equation}
If $A_1-A_0\leq \varrho q$ and $q\leq \rho Q$, $\varrho,\rho \in \lbrack 0,1
\rbrack$, the following bound also holds:
\begin{equation}\label{eq:zerom}\begin{aligned}
\sum_{A_0<m\leq A_1} &\left|\sum_{W'<p\leq W} (\log p) e(\alpha m p)\right|^2\\
&\leq (W-W'+q/(1-\varrho \rho)) \sum_{W'<p\leq W} (\log p)^2.\end{aligned}
\end{equation}
\end{lemma}
\begin{proof}
Let $k = \min(q,\lceil Q/2\rceil) \geq \lceil q/2\rceil$. We split
$(A_0,A_1\rbrack$ into $\lceil (A_1-A_0)/k\rceil$ blocks of at most $k$
consecutive integers $m_0+1, m_0+2,\dotsc$. 
For $m$, $m'$ in such a block, $\alpha m$ and
$\alpha m'$ are separated by a distance of at least 
\[|\{(a/q) (m-m')\}| - O^*(k/qQ) = 1/q - O^*(1/2q) \geq 1/2q.\]
By the large sieve 
\begin{equation}\label{eq:roussel}
\sum_{a=1}^{q} 
\left|\sum_{W'<p\leq W} (\log p) e(\alpha (m_0+a) p)\right|^2
\leq ((W-W') + 2q) \sum_{W'<p\leq W} (\log p)^2.\end{equation}
We obtain (\ref{eq:pokor1}) by summing over all
$\lceil (A_1-A_0)/k\rceil$ blocks.

If $A_1-A_0\leq |\varrho q|$ and $q\leq \rho Q$, $\varrho,\rho \in \lbrack 0,1
\rbrack$, we obtain (\ref{eq:zerom}) simply by applying the large sieve
without splitting the interval $A_0<m\leq A_1$.

Let us now prove (\ref{eq:pokor2}). We will
 use Montgomery's inequality, followed by Montgomery and Vaughan's
large sieve with weights.
An angle $a/q+a_1'/r_1$ is separated from other angles $a'/q+a_2'/r_2$
($r_1, r_2\leq R$, $(a_i,r_i)=1$) 
by at least $1/q r_1 R$, rather than just $1/q R^2$.
We will choose $R$ so that $q R^2 < Q$; this implies $1/Q < 1/q R^2\leq
1/q r_1 R$. 

By a lemma of Montgomery's 
\cite[Lemma 7.15]{MR2061214}, applied
(for each $1\leq a\leq q$)
 to $S(\alpha) = \sum_n a_n e(\alpha n)$
with $a_n = \log(n) e(\alpha(m_0+a) n)$ if $n$ is prime and $a_n=0$ otherwise,
\begin{equation}\label{eq:cortomalt}\begin{aligned}
\frac{1}{\phi(r)} &\left|\sum_{W'<p\leq W} (\log p) e(\alpha (m_0+a)p)\right|^2 \\
&\leq
\mathop{\sum_{a' \mo r}}_{(a',r)=1} \left|
\sum_{W'<p\leq W} (\log p) e\left(\left(\alpha \left(m_0+a\right) +
\frac{a'}{r}\right) p\right)\right|^2 .\end{aligned}\end{equation}
for each square-free $r\leq W'$.
We multiply both sides of (\ref{eq:cortomalt}) by 
\[\left(\frac{W}{2} + \frac{3}{2} \left(\frac{1}{q r R} - \frac{1}{Q}
\right)^{-1}\right)^{-1}\] and sum over all $a=0,1,\dotsc,q-1$ and
all square-free $r\leq R$ coprime to $q$;
we will later make sure that $R\leq W'$. We obtain that
\begin{equation}\label{eq:malheur}\begin{aligned}
\mathop{\sum_{r\leq R}}_{(r,q)=1}
 &\left(\frac{W}{2} + \frac{3}{2} \left(\frac{1}{q r R} - \frac{1}{Q}\right)^{-1}\right)^{-1} 
\frac{\mu(r)^2}{\phi(r)}\\ &\cdot \sum_{a=1}^q 
\left|\sum_{W'<p\leq W} (\log p) e(\alpha (m_0+a)p)\right|^2\end{aligned}
\end{equation}
is at most
\begin{equation}\label{eq:douben}\begin{aligned}
\mathop{\mathop{\sum_{r\leq R}}_{(r,q)=1}}_{\text{$r$ sq-free}}
&\left(\frac{W}{2} + \frac{3}{2} \left(\frac{1}{q r R} - \frac{1}{Q}\right)^{-1}\right)^{-1} \\
&\sum_{a=1}^q \mathop{\sum_{a' \mo r}}_{(a',r)=1} 
\left|\sum_{W'<p\leq W} (\log p) e\left(\left(\alpha \left(m_0+a\right) +
\frac{a'}{r}\right) p\right)\right|^2 
\end{aligned}\end{equation}
We now apply the large sieve with weights
\cite[(1.6)]{MR0374060}, recalling that each angle 
$\alpha (m_0+a) + a'/r$ is separated from the others by at least 
$1/q r R - 1/Q$; we obtain that (\ref{eq:douben}) is at most
$\sum_{W'<p\leq W} (\log p)^2$. It remains to estimate the sum in the
first line of (\ref{eq:malheur}).
(We are following here a procedure analogous to that used in
 \cite{MR0374060} to prove the Brun-Titchmarsh theorem.)

Assume first that $q\leq W/13.5$. Set
\begin{equation}\label{eq:gorg}
R = \left(\sigma \frac{W}{q}\right)^{1/2} ,
\end{equation}
where $\sigma = 1/2 e^{2\cdot 0.25068} = 0.30285\dotsc$.
It is clear that $q R^2<Q$,
$q<W'$ and $R\geq 2$. Moreover, for $r\leq R$,
\[\frac{1}{Q} \leq \frac{1}{3.5 W}\leq \frac{\sigma}{3.5} \frac{1}{\sigma W} =
\frac{\sigma}{3.5} \frac{1}{q R^2} \leq \frac{\sigma/3.5}{q r R}.\]
Hence
\[\begin{aligned}
\frac{W}{2} + \frac{3}{2} \left(\frac{1}{q r R} - \frac{1}{Q}\right)^{-1}
&\leq \frac{W}{2} + \frac{3}{2} \frac{ q r R}{1-\sigma/3.5} = 
\frac{W}{2} + \frac{3 r}{2 \left(1-\frac{\sigma}{3.5}\right) R} \cdot
2\sigma \frac{W}{2}\\
&= \frac{W}{2} \left(1 + \frac{3\sigma}{1-\sigma/3.5} \frac{r W}{R}\right) < 
\frac{W}{2} \left(1 + \frac{r W}{R}\right)\end{aligned}\]
and so
\[\begin{aligned}
\mathop{\sum_{r\leq R}}_{(r,q)=1} &\left(\frac{W}{2} + \frac{3}{2} \left(\frac{1}{q r R} - \frac{1}{Q}\right)^{-1}\right)^{-1} 
\frac{\mu(r)^2}{\phi(r)}\\ 
&\geq \frac{2}{W} \mathop{\sum_{r\leq R}}_{(r,q)=1} (1 + r R^{-1})^{-1} 
\frac{\mu(r)^2}{\phi(r)}
\geq \frac{2}{W} \frac{\phi(q)}{q} 
\sum_{r\leq R} (1 + r R^{-1})^{-1} 
\frac{\mu(r)^2}{\phi(r)}
.\end{aligned}\]
For $R\geq 2$, 
\[\sum_{r\leq R} (1 + r R^{-1})^{-1} 
\frac{\mu(r)^2}{\phi(r)} > \log R + 0.25068;\]
this is true for $R\geq 100$ by \cite[Lemma 8]{MR0374060} 
and easily verifiable numerically for $2\leq R<100$. (It suffices to verify
this for $R$ integer with $r<R$ instead of $r\leq R$, as that is the worst
case.) 

Now
\[\log R = \frac{1}{2} \left(\log \frac{W}{2 q} + \log 2\sigma\right)
= \frac{1}{2} \log \frac{W}{2 q} - 0.25068 .\]
Hence
\[\sum_{r\leq R} (1 + r R^{-1})^{-1} 
\frac{\mu(r)^2}{\phi(r)} > \frac{1}{2} \log \frac{W}{2 q}\] and
the statement follows. 

Now consider the case $q > W/13.5$. If  $q$ is even, then, in this range,
inequality (\ref{eq:pokor1}) is always better than
(\ref{eq:pokor2}), and so we are done.
 Assume, then, that $W/13.5 <q \leq W/2$ and $q$ is odd. We
set $R=2$; clearly $q R^2 < W\leq Q$ and $q<W/2\leq W'$, and so this
choice of $R$ is valid. It remains to check that
\[\frac{1}{\frac{W}{2} + \frac{3}{2} 
\left(\frac{1}{2q} - \frac{1}{Q}\right)^{-1}} + 
\frac{1}{\frac{W}{2} + \frac{3}{2} 
\left(\frac{1}{4q} - \frac{1}{Q}\right)^{-1}} \geq \frac{1}{W} \log
\frac{W}{2q}.\]
This follows because
\[\frac{1}{\frac{1}{2} + \frac{3}{2} 
\left(\frac{t}{2} - \frac{1}{3.5}\right)^{-1}} + 
\frac{1}{\frac{1}{2} + \frac{3}{2} 
\left(\frac{t}{4} - \frac{1}{3.5}\right)^{-1}} \geq \log \frac{t}{2}\]
for all $2\leq t\leq 13.5$.

\end{proof} 
We need a version of Lemma \ref{lem:kastor1} with $m$ restricted to the
odd numbers, since we plan to set the parameter $v$ equal to $2$.
\begin{lemma}\label{lem:kastor2}
Let $W\geq 1$, $W'\geq W/2$. Let 
$2\alpha = a/q + O^*(1/q Q)$, $q\leq Q$.  Then
\begin{equation}\label{eq:pokor1b}\begin{aligned}
\mathop{\sum_{A_0<m\leq A_1}}_{\text{$m$ odd}}
 &\left|\sum_{W'<p\leq W} (\log p) e(\alpha m p)\right|^2
\\ &\leq 
\left\lceil \frac{A_1-A_0}{\min(2q,Q)} \right\rceil\cdot
(W-W'+2q) \sum_{W'<p\leq W} (\log p)^2.\end{aligned}\end{equation}
If $q<W/2$ and $Q\geq 3.5 W$, the following bound also holds:
\begin{equation}\label{eq:pokor2b}\begin{aligned}
\mathop{\sum_{A_0<m\leq A_1}}_{\text{$m$ odd}}
 &\left|\sum_{W'<p\leq W} (\log p) e(\alpha m p)\right|^2\\
&\leq \left\lceil \frac{A_1-A_0}{2 q}\right\rceil\cdot
\frac{q}{\phi(q)} \frac{W}{\log(W/2q)}
\cdot \sum_{W'<p\leq W} (\log p)^2.\end{aligned}\end{equation}
If $A_1-A_0\leq 2\varrho q$ and $q\leq \rho Q$, $\varrho,\rho \in \lbrack 0,1
\rbrack$, the following bound also holds:
\begin{equation}\label{eq:zeromb}\begin{aligned}
\sum_{A_0<m\leq A_1} &\left|\sum_{W'<p\leq W} (\log p) e(\alpha m p)\right|^2\\
&\leq (W-W'+q/(1-\varrho \rho)) \sum_{W'<p\leq W} (\log p)^2.\end{aligned}
\end{equation}
\end{lemma}
\begin{proof}
We follow the proof of Lemma \ref{lem:kastor1}, noting the differences.
Let \[k=\min(q,\lceil Q/2\rceil)\geq \lceil q/2\rceil,\] just as before.
We split $(A_0,A_1\rbrack$ into $\lceil (A_1-A_0)/k\rceil$ blocks of at most
$2k$ consecutive integers; any such block contains at most $k$ odd numbers.
For odd $m$, $m'$ in such a block, $\alpha m$ and $\alpha m'$ are
separated by a distance of 
\[|\{\alpha (m-m')\}| = \left|\left\{2\alpha \frac{m-m'}{2}\right\}\right|
= |\{(a/q) k\}| - O^*(k/qQ) \geq 1/2q.\]

We obtain (\ref{eq:pokor1b}) and (\ref{eq:zeromb})
 just as we obtained (\ref{eq:pokor1}) and (\ref{eq:zerom}) before.
To obtain (\ref{eq:pokor2b}), proceed again as before, noting that
the angles we are working with can be labelled as $\alpha (m_0+2a)$,
$0\leq a < q$.
\end{proof}

The idea now (for large $\delta$) is that, if $\delta$ is not negligible, then, 
as $m$ increases and $\alpha m$ loops around the circle $\mathbb{R}/\mathbb{Z}$, $\alpha m$ 
roughly repeats itself every $q$ steps -- but with a slight displacement.
This displacement gives rise to a configuration to which Lemma \ref{lem:ogor}
is applicable. The effect is that we can apply the large sieve once instead
of many times, thus leading to a gain of a large factor (essentially, the number
of times the large sieve would have been used). This is how we obtain
the factor of $|\delta|$ in the denominator of the main term
$x/|\delta| q$ in (\ref{eq:procida2}) and (\ref{eq:procida3}).

\begin{prop}\label{prop:kraken}
Let $x\geq W\geq 1$, $W'\geq W/2$, $U'\geq x/2W$. Let 
$Q\geq
3.5W$.
Let $2 \alpha = a/q + 
\delta/x$, $(a,q)=1$, $|\delta/x|\leq 1/qQ$, $q\leq Q$. 
Let $S_2(U',W',W)$ be as in (\ref{eq:honi}) with $v=2$.

For $q\leq \rho Q$, where $\rho\in \lbrack 0,1\rbrack$,
\begin{equation}\label{eq:garn1b}\begin{aligned}
S_2(U',W',W) &\leq
\left(\max(1,2\rho) \left(\frac{x}{8q} + \frac{x}{2W}\right) + \frac{W}{2} + 2 q\right)
\cdot  \sum_{W'<p\leq W} (\log p)^2\\
\end{aligned}
\end{equation}

If $q< W/2$,
\begin{equation}\label{eq:garn1a}
S_2(U',W',W)\leq 
\left(
\frac{x}{4 \phi(q)} \frac{1}{\log(W/2q)}
 + \frac{q}{\phi(q)} \frac{W}{\log(W/2q)} \right)\cdot
 \sum_{W'<p\leq W} (\log p)^2 .\end{equation}


If $W> x/4 q$,
the following bound also holds:
\begin{equation}\label{eq:gargamel}
S_2(U',W',W) \leq \left(\frac{W}{2} + \frac{q}{1-x/4Wq}\right)\sum_{W'< p\leq W} (\log p)^2.
\end{equation}

If $\delta\ne 0$ and $x/4W + q \leq x/|\delta| q$, 
\begin{equation}\label{eq:procida2}\begin{aligned}
S_2(U',W',W) &\leq \min\left(1,\frac{2 q/\phi(q)}{\log\left(\frac{x}{|\delta q|} 
\left(q + \frac{x}{4W}\right)^{-1}
\right)}\right) \\&\cdot 
\left(\frac{x}{|\delta q|} + \frac{W}{2}\right)
\sum_{W'<p\leq W} (\log p)^2 .
\end{aligned}\end{equation}

Lastly, if $\delta\ne 0$ and $q\leq \rho Q$, where $\rho\in \lbrack 0,1)$,
\begin{equation}\label{eq:procida3}
 S_2(U',W',W) \leq 
\left(\frac{x}{|\delta q|} + \frac{W}{2} + \frac{x}{8 (1-\rho) Q} + 
\frac{x}{4 (1-\rho) W}\right) \sum_{W'<p\leq W} (\log p)^2 .\end{equation}
\end{prop}

The trivial bound would be in the order
of \[S_2(U',W',W) = (x/2\log x) \sum_{W'<p\leq W} (\log p)^2.\]
In practice, (\ref{eq:gargamel}) gets applied when $W\geq x/q$.

\begin{proof}
Let us first prove statements (\ref{eq:garn1a}) and (\ref{eq:garn1b}), which
do not involve $\delta$. Assume first $q\leq W/2$. Then, by
(\ref{eq:pokor2b}) with $A_0 = U'$, $A_1 = x/W$,
\[S_2(U',W',W)\leq \left(\frac{x/W - U'}{2 q} + 1\right)
\frac{q}{\phi(q)} \frac{W}{\log(W/2q)} \sum_{W'<p\leq W} (\log p)^2.\]
Clearly $(x/W-U') W\leq (x/2W) \cdot W = x/2$. 
 Thus (\ref{eq:garn1a}) holds.

Assume now that $q\leq \rho Q$. Apply (\ref{eq:pokor1b}) with 
$A_0 = U'$, $A_1 = x/W$. Then
\[S_2(U',W',W)\leq \left(\frac{x/W - U'}{q\cdot \min(2,\rho^{-1})} + 1\right)
(W - W' + 2q) 
\sum_{W'<p\leq W} (\log p)^2.\]
Now
\[\begin{aligned}
&\left(\frac{x/W - U'}{q\cdot \min(2,\rho^{-1})} + 1\right)\cdot 
(W - W' + 2q) \\&\leq \left(\frac{x}{W}-U'\right) 
\frac{W-W'}{q \min(2,\rho^{-1})} + 
\max(1,2\rho) \left(\frac{x}{W} - U'\right) + W/2 + 2q 
\\ &\leq \frac{x/4}{q \min(2,\rho^{-1})}
 + \max(1,2\rho) \frac{x}{2 W} + W/2 + 2 q .
\end{aligned}\]
This implies (\ref{eq:garn1b}).

If $W> x/4 q$, apply (\ref{eq:zerom}) with $\varrho = x/4Wq$, $\rho=1$.
This yields (\ref{eq:gargamel}).

Assume now that $\delta\ne 0$ and $x/4W+q\leq x/|\delta q|$.
Let $Q' = x/|\delta q|$. For any $m_1$, $m_2$ with $x/2W< m_1,m_2\leq x/W$,
we have
$|m_1-m_2|\leq x/2W \leq 2 (Q'-q)$, and so
\begin{equation}\label{eq:radclas}
\left|\frac{m_1-m_2}{2} \cdot \delta/x + q\delta/x\right|\leq Q'|\delta|/x = \frac{1}{q}.\end{equation}
The conditions of Lemma \ref{lem:ogor}
are thus 
fulfilled with $\upsilon = (x/4W)\cdot |\delta|/x$ and $\nu=|\delta q|/x$.
We obtain that $S_2(U',W',W)$ is at most
\[
 \min\left(1, \frac{2 q}{\phi(q)} \frac{1}{\log\left((q (\nu+\upsilon))^{-1}
\right)}\right)
 \left(W-W'+\nu^{-1}\right) \sum_{W'<p\leq W} (\log p)^2.\]
Here $W -W' + \nu^{-1} = W-W' + x/|q\delta|\leq W/2 + x/|q\delta|$
and \[(q (\nu+\upsilon))^{-1} = 
\left(q \frac{|\delta|}{x}\right)^{-1}
\left(q + \frac{x}{4W}\right)^{-1} .\]

Lastly, assume $\delta\ne 0$ and $q\leq \rho Q$. We let
$Q' = x/|\delta q| \geq Q$ again, and we split the range $U'<m\leq x/W$
into intervals of length $2(Q'-q)$, so that (\ref{eq:radclas}) still
holds within each interval. We apply Lemma \ref{lem:ogor}
with $\upsilon = (Q'-q)\cdot |\delta|/x$ and $\nu=|\delta q|/x$.
We obtain that 
$S_2(U',W',W)$ is at most
\[
 \left(1 + \frac{x/W-U}{2(Q'-q)}\right)
 \left(W-W'+\nu^{-1}\right) \sum_{W'<p\leq W} (\log p)^2.\]
Here $W-W'+\nu^{-1} \leq W/2 + x/q|\delta|$ as before. Moreover,
\[\begin{aligned}
\left(\frac{W}{2} + \frac{x}{q|\delta|}\right) \left(1 + \frac{x/W-U}{2(Q'-q)}\right)
&\leq \left(\frac{W}{2} + Q'\right) \left(1 + 
\frac{x/2W}{2 (1-\rho) Q'}\right)\\
&\leq \frac{W}{2} + Q' + \frac{x}{8 (1-\rho) Q'} + 
\frac{x}{4 W (1-\rho)}\\
&\leq \frac{x}{|\delta q|} + \frac{W}{2} + \frac{x}{8 (1-\rho) Q} + 
\frac{x}{4 (1-\rho) W}.
\end{aligned}\]
Hence (\ref{eq:procida3}) holds.
\end{proof}

\chapter{Minor-arc totals}\label{ch:totcon}

It is now time to make all of our estimates fully explicit, choose our
parameters, put our type I and type II estimates together and give our
final minor-arc estimates.

Let $x>0$ be given. Starting in section \ref{subs:cojor}, we will assume that
$x\geq x_0 = 2.16 \cdot 10^{20}$. We will choose our main parameters $U$ and $V$
gradually, as the need arises; we assume from the start that
 $2\cdot 10^6 \leq V< x/4$ and $UV\leq x$. 

We are also given an angle $\alpha \in \mathbb{R}/\mathbb{Z}$. 
We choose an approximation
$2 \alpha = a/q + \delta/x$, $(a,q)=1$, $q\leq Q$, $|\delta/x|\leq 1/qQ$.
The parameter $Q$ will be chosen later;
we assume from the start that
 $Q\geq \max(16,2 \sqrt{x})$ and $Q\geq \max(2 U,x/U)$.

(Actually, $U$ and $V$ will be chosen in different ways depending on 
the size of $q$. Actually, even $Q$ will depend on the size of $q$; this 
may seem circular, but what actually happens is the following: we will first
set a value for $Q$ depending only on $x$, and if the corresponding value of
$q\leq Q$ is larger than a certain parameter $y$ depending on $x$, then
we reset $U$, $V$ and $Q$, and obtain a new value of $q$.)

Let $S_{I,1}$, $S_{I,2}$, $S_{II}$, $S_0$ be as in (\ref{eq:nielsen}), with
the smoothing function $\eta=\eta_2$ as in (\ref{eq:eqeta}). (We bounded
the type I sums $S_{I,1}$, $S_{I,2}$ for a general smoothing function $\eta$;
it is only here that we are specifying $\eta$.)

The term $S_0$ is $0$ because $V<x/4$ and $\eta_2$ is supported on
$\lbrack -1/4,1\rbrack$.  We set $v=2$.

\section{The smoothing function}

For the smoothing function $\eta_2$ in
(\ref{eq:eqeta}), 
\begin{equation}\label{eq:muggle}
|\eta_2|_1 = 1,\;\;\;\;\; |\eta_2'|_1 = 8 \log 2,\;\;\;\;\; |\eta_2''|_1 = 48,
\end{equation}
as per \cite[(5.9)--(5.13)]{Tao}.
Similarly, for $\eta_{2,\rho}(t) = \log(\rho t) \eta_2(t)$, where $\rho\geq 4$,
\begin{equation}\label{eq:cloclo}\begin{aligned}
|\eta_{2,\rho}|_1 &< \log(\rho) |\eta_2|_1 = \log(\rho)\\
|\eta_{2,\rho}'|_1 &= 2 \eta_{2,\rho}(1/2) = 2 \log(\rho/2) \eta_2(1/2) < 
(8 \log 2) \log \rho,\\
|\eta_{2,\rho}''|_1 &= 4 \log(\rho/4) + |2 \log \rho - 4 \log(\rho/4)| +
|4\log 2 - 4 \log \rho| \\ &+ |\log \rho - 4 \log 2| +
|\log \rho| < 48 \log \rho.
\end{aligned}\end{equation}
In the first inequality, we are using the fact that $\log(\rho t)$
is always positive (and less than $\log(\rho)$) when $t$ is in the support
of $\eta_2$.

Write $\log^+ x$ for $\max(\log x,0)$.

\section{Contributions of different types}\label{subs:putmal}
\subsection{Type I terms: $S_{I,1}$.}\label{subs:renzo}
The term $S_{I,1}$ can be handled directly by Lemma \ref{lem:bostb1},
with $\rho_0=4$ and $D = U$. (Condition (\ref{eq:puella}) is 
valid thanks to (\ref{eq:cloclo}).)
Since $U\leq Q/2$,
 the contribution of $S_{I,1}$ gets bounded by (\ref{eq:cupcake2}) and
(\ref{eq:kuche2}): the absolute value of $S_{I,1}$ is at most
\begin{equation}\label{eq:cosI1}\begin{aligned}&\frac{x}{q} 
 \min\left(1,\frac{c_0/\delta^2}{(2\pi)^2}\right)
\left|\mathop{\sum_{m\leq \frac{U}{q}}}_{(m,q)=1} \frac{\mu(m)}{m} 
\log \frac{x}{m q}\right| + \frac{x}{q}
|\widehat{\log \cdot \eta}(-\delta)| \left|\mathop{\sum_{m\leq \frac{U}{q}}}_{(m,q)=1} \frac{\mu(m)}{m}\right|\\
&+ \frac{2 \sqrt{c_0 c_1}}{\pi}  \left( 
U \log \frac{e x}{U} + \sqrt{3} q \log \frac{q}{c_2} +
\frac{q}{2} \log \frac{q}{c_2} \log^+ \frac{2 U}{q}\right)
+ \frac{3 c_1 x}{2 q}
 \log \frac{q}{c_2} \log^+ \frac{U}{\frac{c_2 x}{q}}\\
&+ \frac{3 c_1}{2} \sqrt{\frac{2 x}{c_2}} \log \frac{2 x}{c_2} 
+ 
\left(\frac{c_0}{2} - \frac{2 c_0}{\pi^2}
\right)
\left(\frac{U^2}{4 q x} \log \frac{e^{1/2} x}{
U} + \frac{1}{e} \right) 
\\ &+ \frac{2 |\eta'|_1}{\pi} q \max\left(1, \log \frac{c_0 e^3 q^2}{4 \pi |\eta'|_1 x}\right) \log x + \frac{20 c_0 c_2^{3/2}}{3 \pi^2} \sqrt{2 x} \log
\frac{2 \sqrt{e} x}{c_2},
\end{aligned}\end{equation}
where $c_0 = 31.521$ (by Lemma \ref{lem:octet}), 
$c_1 =  1.0000028 > 1 + (8 \log 2)/V \geq 1 + (8 \log 2)/(x/U)$ and
$c_2 = 6 \pi/5\sqrt{c_0}
= 0.67147\dotsc$. By (\ref{eq:madge}) (with $k=2$),
(\ref{eq:koasl}) and Lemma \ref{lem:marengo},
\[|\widehat{\log \cdot \eta}(-\delta)| \leq
\min\left(2-\log 4,\frac{24 \log 2}{\pi^2 \delta^2}\right).\]

By (\ref{eq:grara}), (\ref{eq:ronsard}) and (\ref{eq:meproz}),
the first line of (\ref{eq:cosI1}) is at most
\[\begin{aligned}
&\frac{x}{q} \min\left(1,\frac{c_0'}{\delta^2}\right)
\left(\min\left(\frac{4}{5} \frac{q/\phi(q)}{\log^+ 
\frac{U}{q^2}}, 1\right) \log \frac{x}{U} + 1.00303 \frac{q}{\phi(q)}\right)
\\+ &\frac{x}{q} \min\left(2 - \log 4,\frac{c_0''}{\delta^2}\right)
\min\left(\frac{4}{5} \frac{q/\phi(q)}{\log^+ 
\frac{U}{q^2}}, 1\right),
\end{aligned}\]
where $c_0' = 0.798437 > c_0/(2\pi)^2$, $c_0'' = 1.685532$. Clearly 
$c_0''/c_0 > 1 > 2 - \log 4$. 

Taking derivatives, we see that $t\mapsto (t/2) \log(t/c_2) \log^+ 2 U/t$
takes its maximum (for $t\in \lbrack 1,2 U\rbrack$) 
when $\log(t/c_2) \log^+ 2 U/t =
\log t/c_2 - \log^+ 2 U /t$; since $t\to \log t/c_2 - \log^+ 2U/t$
is increasing on $\lbrack 1, 2U\rbrack$, we conclude that
\[\frac{q}{2} \log \frac{q}{c_2} \log^+ \frac{2 U}{q} \leq U \log \frac{2 U}{c_2}.\]
Similarly, $t\mapsto t \log(x/t) \log^+(U/t)$ takes its maximum at a point
$t\in \lbrack 0, U$ for which
$\log(x/t) \log^+(U/t) = \log(x/t) + \log^+(U/t)$, and so
\[\frac{x}{q} \log \frac{q}{c_2} \log^+ \frac{U}{\frac{c_2 x}{q}} \leq
\frac{U}{c_2} (\log x + \log U).\]

We conclude that
\begin{equation}\label{eq:lavapie}\begin{aligned}
|S_{I,1}| &\leq \frac{x}{q} \min\left(1, \frac{c_0'}{\delta^2}\right)
\left(\min\left(\frac{4 q/\phi(q)}{5\log^+ 
\frac{U}{q^2}}, 1\right) \left(\log \frac{x}{U}
+ c_{3,I}\right)
 + c_{4,I} \frac{q}{\phi(q)} \right)\\
&+ \left(c_{7,I}
\log \frac{q}{c_2} + c_{8,I} \log x \max\left(
1,  \log \frac{c_{11,I} q^2}{x}\right) \right) q 
+
c_{10,I} \frac{U^2}{4 q x} \log \frac{e^{1/2} x}{
U}\\
&+ \left(c_{5,I} \log \frac{2 U}{c_2} +
c_{6,I} \log x U\right) U 
+ c_{9,I} \sqrt{x} \log \frac{2 \sqrt{e} x}{c_2} 
+ \frac{c_{10,I}}{e}
,\end{aligned}\end{equation}
where $c_2$ and $c_0'$ are as above, $c_{3,I} = 2.11104 > c_0''/c_0'$, 
$c_{4,I} = 1.00303$, $c_{5,I}= 3.57422 > 2 \sqrt{c_0 c_1}/\pi$,
$c_{6,I} = 2.23389 > 3 c_1/2 c_2$, $c_{7,I} = 6.19072 > 2 \sqrt{3 c_0 c_1}/\pi$, 
$c_{8,I} = 3.53017 > 2 (8 \log 2)/\pi$, 
\[c_{9,I} = 19.1568 > 
\frac{3 \sqrt{2} c_1}{2 \sqrt{c_2}} + \frac{20 \sqrt{2} 
c_0 c_2^{3/2}}{3\pi^2},\] $c_{10,I} = 
9.37301 > c_0 (1/2 - 2/\pi^2)$ and $c_{11,I} = 9.0857
> c_0 e^3/(4 \pi \cdot 8 \log 2)$.

\subsection{Type I terms: $S_{I,2}$.}
{\em The case $q\leq Q/V$.} If $q\leq Q/V$, then, for $v\leq V$,
\[2 v\alpha = \frac{v a}{q} + O^*\left(\frac{v}{Q q}\right) = 
\frac{v a}{q} + O^*\left(\frac{1}{q^2}\right) ,\]
and so $va/q$ is a valid approximation to $2 v\alpha$. (Here we are using
$v$ to label an integer variable bounded above by $v\leq V$; we no longer
need $v$ to label the quantity in (\ref{eq:joroy}), since that
has been set equal to the constant $2$.)
Moreover, for
$Q_v = Q/v$, we see that $2 v\alpha = (va/q) + O^*(1/q Q_v)$.
 If $\alpha = a/q + 
\delta/x$, then $v\alpha = v a/q + \delta/(x/v)$.
Now
\begin{equation}\label{eq:jotoco}
S_{I,2} = \mathop{\sum_{v\leq V}}_{\text{$v$ odd}} \Lambda(v) 
\mathop{\sum_{m\leq U}}_{\text{$m$ odd}} \mu(m) \mathop{\sum_n}_{
\text{$n$ odd}}
e((v\alpha) \cdot m n) \eta(mn/(x/v)).\end{equation}
 We can thus
estimate $S_{I,2}$ by applying Lemma \ref{lem:bosta2}
to each inner double sum in (\ref{eq:jotoco}).
We obtain that, if $|\delta|\leq 1/2 c_2$, where 
$c_2 = 6\pi/5\sqrt{c_0}$ and $c_0 = 31.521$, then $|S_{I,2}|$ is at most
\begin{equation}\label{eq:putbarat}
\sum_{v\leq V} \Lambda(v) \left(
\frac{x/v}{2 q_v} \min\left(1,\frac{c_0}{(\pi \delta)^2}
\right) \left|\mathop{\sum_{m\leq M_v/q}}_{(m,2 q)=1} \frac{\mu(m)}{m}\right|
+ \frac{c_{10,I} q}{4 x/v} \left(\frac{U}{q_v} + 1\right)^2
\right)
\end{equation}
plus
\begin{equation}\label{eq:douze}\begin{aligned}
&\sum_{v\leq V} \Lambda(v) 
\left( \frac{2 \sqrt{c_0 c_+}}{\pi} U + 
\frac{3 c_+}{2} \frac{x}{v q_v}
\log^+ \frac{U}{\frac{c_2 x}{v q_v}} + \frac{\sqrt{c_0 c_+}}{\pi}
q_v \log^+ \frac{U}{q_v/2}\right)\\
+ &\sum_{v\leq V} \Lambda(v) \left(c_{8,I}
\max\left(
 \log \frac{c_{11,I} q_v^2}{x/v}, 1\right) q_v 
+ \left(\frac{2 \sqrt{3 c_0 c_+}}{\pi} + \frac{3 c_+}{2 c_2} + 
\frac{55 c_0 c_2}{6 \pi^2} \right) q_v
\right),
\end{aligned}\end{equation}
where $q_v = q/(q,v)$, $M_v \in \lbrack \min(Q/2v,U),U\rbrack$ 
and $c_+ = 1 + (8 \log 2)/(x/UV)$;
if $|\delta|\geq 1/2c_2$, then $|S_{I,2}|$ is at most (\ref{eq:putbarat})
plus
\begin{equation}\label{eq:cheaslu}\begin{aligned}
&\sum_{v\leq V} \Lambda(v) 
\left(\frac{\sqrt{c_0 c_1}}{\pi/2} U  +
 \frac{3 c_1}{2}
 \left(2 + \frac{(1+\epsilon)}{\epsilon} \log^+ \frac{2 U}{
\frac{x/v}{|\delta| q_v}}
\right) \frac{x}{Q} + \frac{35 c_0 c_2}{3 \pi^2} q_v\right)\\
&+\sum_{v\leq V} \Lambda(v) \frac{\sqrt{c_0 c_1}}{\pi/2} 
 (1+\epsilon) \min\left(\left\lfloor\frac{x/v}{|\delta| q_v}\right\rfloor + 1, 2 U\right)
 \left(\sqrt{3+2\epsilon} +
 \frac{\log^+ \frac{2 U}{\left\lfloor \frac{x/v}{|\delta| q_v}\right\rfloor + 
1}}{2}\right)
\end{aligned}\end{equation}

Write $S_V = \sum_{v\leq V} \Lambda(v)/(v q_v)$.
By (\ref{eq:rala}),
\begin{equation}\label{eq:avamys}\begin{aligned}
S_V &\leq \sum_{v\leq V} \frac{\Lambda(v)}{v q} +
\mathop{\sum_{v\leq V}}_{(v,q)>1} \frac{\Lambda(v)}{v} \left(\frac{(q,v)}{q} - 
\frac{1}{q}\right)\\ &\leq \frac{\log V}{q} +
\frac{1}{q} \sum_{p|q} (\log p) \left(v_p(q) + \mathop{\sum_{\alpha \geq 1}}_{
p^{\alpha+v_p(q)} \leq V} \frac{1}{p^\alpha} - \mathop{\sum_{\alpha \geq 1}}_{
p^{\alpha}\leq V} \frac{1}{p^\alpha} \right)\\ &\leq \frac{\log V}{q} + \frac{1}{q}
\sum_{p|q} (\log p) v_p(q) = \frac{\log V q}{q}.
\end{aligned}\end{equation}
This helps us to estimate (\ref{eq:putbarat}). We could also use this
to estimate the second term in the first line of (\ref{eq:douze}),
but, for that purpose, it will actually be wiser to use the simpler bound
\begin{equation}\label{eq:qotro}\sum_{v\leq V} \Lambda(v) 
\frac{x}{v q_v} \log^+ \frac{U}{\frac{c_2 x}{v q_v}}
\leq \sum_{v\leq V} \Lambda(v) \frac{U/c_2}{e} \leq
\frac{1.0004}{e c_2} UV 
\end{equation}
(by (\ref{eq:trado1}) and the fact that $t \log^+ A/t$ takes its maximum at
$t=A/e$).



We bound the sum over $m$ in (\ref{eq:putbarat}) by (\ref{eq:grara}) and
(\ref{eq:ronsard}):
\[\left|\mathop{\sum_{m\leq M_v/q}}_{(m,2 q)=1} \frac{\mu(m)}{m}\right|
\leq \min\left(\frac{4}{5} \frac{q/\phi(q)}{\log^+ \frac{M_v}{2q^2},1}\right)
.\]
 To bound the terms involving $(U/q_v+1)^2$, we use
\[\begin{aligned}
\sum_{v\leq V} \Lambda(v) v &\leq 0.5004 V^2 \;\;\;\;\;\;\text{(by
(\ref{eq:nicro})),}\\
\sum_{v\leq V} \Lambda(v) v (v,q)^j &\leq \sum_{v\leq V} \Lambda(v) v +
V \mathop{\sum_{v\leq V}}_{(v,q)\ne 1} \Lambda(v) (v,q)^j
\end{aligned},\]
\[\begin{aligned}
\mathop{\sum_{v\leq V}}_{(v,q)\ne 1} \Lambda(v) (v,q) &\leq
 \sum_{p|q} (\log p) \sum_{1\leq \alpha \leq \log_p V} p^{v_p(q)}
\leq \sum_{p|q} (\log p) \frac{\log V}{\log p} p^{v_p(q)} \\ &\leq (\log V)
\sum_{p|q} p^{v_p(q)} \leq q \log V
\end{aligned}\] and
\[\begin{aligned}
\mathop{\sum_{v\leq V}}_{(v,q)\ne 1} \Lambda(v) (v,q)^2 &\leq
\sum_{p|q} (\log p) \sum_{1\leq \alpha\leq \log_p V} p^{v_p(q)+\alpha}\\&\leq
\sum_{p|q} (\log p) \cdot 2 p^{v_p(q)} \cdot p^{\log_p V} \leq 2 q V \log q .
\end{aligned}\]
 Using (\ref{eq:trado1}) and
(\ref{eq:avamys}) as well, we conclude that (\ref{eq:putbarat}) is at most
\[\begin{aligned}
\frac{x}{2 q} &\min\left(1, \frac{c_0}{(\pi \delta)^2}\right)
\min\left(\frac{4}{5} \frac{q/\phi(q)}{\log^+ \frac{\min(Q/2V,U)}{2q^2}},1\right) \log Vq\\
&+ \frac{c_{10,I}}{4 x} \left(0.5004 V^2 q \left(\frac{U}{q}+1\right)^2 +
2 U V q \log V + 2 U^2 V \log V\right).\end{aligned}\]

Assume $Q \leq 2 U V/e$.
Using (\ref{eq:trado1}), (\ref{eq:qotro}), (\ref{eq:charol})
 and the inequality $vq \leq V q\leq Q$ (which implies $q/2\leq U/e$), 
we see that (\ref{eq:douze})
is at most
\[\begin{aligned}
&1.0004 \left(\left(\frac{2\sqrt{c_0 c_+}}{\pi} + \frac{3 c_+}{2 e c_2}
\right) UV + 
\frac{\sqrt{c_0 c_+}}{\pi} Q \log \frac{U}{q/2}\right)
\\ + &\left(c_{5,I_2} 
\max\left(
 \log \frac{c_{11,I} q^2 V}{x}, 2\right)
 + c_{6,I_2}\right) Q,\end{aligned}\]
where $c_{5,I_2} = 3.53312 > 1.0004 \cdot c_{8,I}$ 
and \begin{equation}\label{eq:czmil}
c_{6,I_2} = 1.0004 \left(\frac{2 \sqrt{3 c_0 c_+}}{\pi} + \frac{3 c_+}{2 c_2} + 
\frac{55 c_0 c_2}{6 \pi^2}\right) .\end{equation}

The expressions in (\ref{eq:cheaslu}) get estimated similarly. 
The first line of (\ref{eq:cheaslu}) is at most
\[1.0004 \left(\frac{2 \sqrt{c_0 c_+}}{\pi} U V + \frac{3 c_+}{2}
\left(2 + \frac{1+\epsilon}{\epsilon} \log^+ \frac{2 U V |\delta| q}{x}
\right) \frac{x V}{Q} +\frac{35 c_0 c_2}{3 \pi^2} q V\right)\]
by (\ref{eq:trado1}). Since $q\leq Q/V$, we can obviously bound $q V$ by
$Q$. As for the second line of (\ref{eq:cheaslu}) --
\[\begin{aligned}
\sum_{v\leq V} &\Lambda(v) \min\left(\left\lfloor
\frac{x/v}{|\delta| q_v}\right\rfloor + 1, 2 U\right) 
\cdot \frac{1}{2}
\log^+ \frac{2 U}{\left\lfloor \frac{x/v}{|\delta| q_v}\right\rfloor + 1}\\
&\leq \sum_{v\leq V} \Lambda(v) \max_{t>0} t \log^+ \frac{U}{t}
 \leq
\sum_{v\leq V} \Lambda(v)  \frac{U}{e} = \frac{1.0004}{e} UV,
\end{aligned}\]
but
\[\begin{aligned}
\sum_{v\leq V} &\Lambda(v) \min\left(\left\lfloor
\frac{x/v}{|\delta| q_v}\right\rfloor + 1, 2 U\right) 
 \leq \sum_{v\leq \frac{x}{2 U |\delta| q}} \Lambda(v) \cdot 2 U \\ &+ 
\mathop{\sum_{\frac{x}{2 U |\delta| q} < v \leq V}}_{(v,q)=1} \Lambda(v)
\frac{x/|\delta|}{
v q} + \sum_{v\leq V} \Lambda(v) + \mathop{\sum_{v\leq V}}_{(v,q)\ne 1}
\Lambda(v) \frac{x/|\delta|}{v} \left(\frac{1}{q_v} - \frac{1}{q}\right)\\
&\leq 1.03883 \frac{x}{|\delta| q} + \frac{x}{|\delta| q} 
\max\left(\log V - \log \frac{x}{2 U |\delta| q} + \log \frac{3}{\sqrt{2}},0\right)\\
&+ 1.0004 V + \frac{x}{|\delta|} \frac{1}{q} \sum_{p|q} (\log p) v_p(q)\\
&\leq \frac{x}{|\delta| q} \left(1.03883+ \log q + 
 \log^+ \frac{6 U V |\delta| q}{\sqrt{2} x}\right) + 1.0004 V 
\end{aligned}\]
by (\ref{eq:rala}), (\ref{eq:ralobio}), (\ref{eq:trado1})
 and (\ref{eq:trado2}); we are proceeding
much as in (\ref{eq:avamys}).

Let us collect our bounds.
If $|\delta|\leq 1/2 c_2$, then, assuming $Q \leq 2 U V/e$,
we conclude that
$|S_{I,2}|$ is at most
\begin{equation}\label{eq:chusan}\begin{aligned}
&\frac{x}{2 \phi(q)} \min\left(1, \frac{c_0}{(\pi \delta)^2}\right)
\min\left(\frac{4/5}{\log^+ \frac{Q}{4 V q^2}},1\right) \log V q \\ &+
c_{8,I_2} \frac{x}{q} \left(\frac{U V}{x}\right)^2
\left(1 + \frac{q}{U}\right)^2 + \frac{c_{10,I}}{2} \left(\frac{UV}{x}
q \log V + \frac{U^2 V}{x} \log V\right)
\end{aligned}\end{equation}
plus
\begin{equation}\label{eq:fausto}
(c_{4,I_2} + c_{9,I_2}) UV + (c_{10,I_2} \log \frac{U}{q} +
c_{5,I_2} \max\left(
 \log \frac{c_{11,I} q^2 V}{x}, 2\right) 
 + c_{12,I_2})\cdot  Q,
\end{equation}
where \[\begin{aligned}
c_{4,I_2} &= 3.57565 (1+\epsilon_0) >
1.0004 \cdot 2 \sqrt{c_0 c_+}/\pi,\\ 
c_{5,I_2} &= 3.53312 > 1.0004 \cdot c_{8,I},\\
c_{8,I_2} &= 1.17257 > \frac{c_{10,I}}{4}\cdot 0.5004,\\
c_{9,I_2} &= 0.82214 (1 + 2 \epsilon_0) > 3 c_+ \cdot 1.0004/2 e c_2,\\
c_{10,I_2} &= 1.78783 \sqrt{1 + 2 \epsilon_0} > 1.0004 \sqrt{c_0 c_+}/\pi,\\ 
c_{12,I_2} &= 29.3333 + 11.902 \epsilon_0 \\ &> 1.0004 \left(
\frac{3}{2 c_2} c_+ + 
\frac{2 \sqrt{3 c_0}}{\pi} 
\sqrt{c_+} + \frac{55 c_0 c_2}{6 \pi^2}\right) + 1.78783 (1 + \epsilon_0) \log 2
\\ &=
c_{6,I_2} + c_{10,I_2} \log 2\end{aligned}\]
and $c_{10,I} = 9.37301$ as before.
Here
$\epsilon_0 = (4 \log 2)/(x/UV)$, and $c_{6,I_2}$ is as in (\ref{eq:czmil}).
If $|\delta|\geq 1/2 c_2$, then
$|S_{I,2}|$ is at most
(\ref{eq:chusan}) plus 
\begin{equation}\label{eq:magus}\begin{aligned}
&(c_{4,I_2} + (1+\epsilon) c_{13,I_2}) U V +  c_\epsilon \left(
c_{14,I_2} \left(\log q + \log^+ \frac{6 UV |\delta| q}{\sqrt{2} x}\right)
+ c_{15,I_2}\right) \frac{x}{|\delta| q} \\ &+ 
c_{16,I_2} \left(2 + \frac{1+\epsilon}{\epsilon} \log^+ \frac{2 UV
|\delta| q}{x}\right) \frac{x}{Q/V} + c_{17,I_2} Q + 
c_\epsilon\cdot 
c_{4,I_2} V,
\end{aligned}\end{equation}
where 
\[\begin{aligned}
c_{13,I_2} &= 1.31541 (1+\epsilon_0) > \frac{2\sqrt{c_0 c_+}}{\pi} \cdot \frac{1.0004}{e},\\
c_{14,I_2} &= 3.57422 \sqrt{1 + 2 \epsilon_0} > \frac{2 \sqrt{c_0 c_+}}{\pi} ,\\
c_{15,I_2} &= 3.71301 \sqrt{1+ 2 \epsilon_0} > \frac{2 \sqrt{c_0 c_+}}{\pi} \cdot
1.03883 ,\\ c_{16,I_2} &= 
1.5006 (1 + 2 \epsilon_0) > 1.0004\cdot 3 c_+/2\\
c_{17,I_2} &= 25.0295 > 1.0004\cdot \frac{35 c_0 c_2}{3 \pi^2},
\end{aligned}\]
and $c_\epsilon = (1+\epsilon) \sqrt{3+2\epsilon}$.
We recall that $c_2 = 6\pi/5\sqrt{c_0} = 0.67147\dotsc$.
We will choose $\epsilon \in (0,1)$ later; we also leave the task of bounding
$\epsilon_0$ for later.

{\em The case $q>Q/V$.}
We use Lemma \ref{lem:bogus} in this case.

\subsection{Type II terms.}\label{subs:absur} As we showed in 
(\ref{eq:adoucit})--(\ref{eq:negli}), $S_{II}$ (given in (\ref{eq:adoucit}))
is at most
\begin{equation}\label{eq:secint}
4 \int_V^{x/U} \sqrt{S_1(U,W) \cdot S_2(U,V,W)} \frac{d W}{W} +
4 \int_V^{x/U} \sqrt{S_1(U,W) \cdot S_3(W)} \frac{d W}{W},\end{equation}
where $S_1$, $S_2$ and $S_3$ are as in (\ref{eq:costo}) and (\ref{eq:negli}).
We bounded $S_1$ in (\ref{eq:menson1}) and (\ref{eq:menson2}), $S_2$ in Prop. \ref{prop:kraken}
and $S_3$ in (\ref{eq:negli}).

Let us try to give some structure to the bookkeeping we must now inevitably do.
The second integral in (\ref{eq:secint}) will be negligible (because 
$S_3$ is); let us focus on the first integral. 

Thanks to our work in \S \ref{subs:vaucanc}, the term $S_1(U,W)$ is bounded by 
a (small)
constant times $x/W$. (This represents a gain of several factors of $\log$
with respect to the trivial bound.) We bounded $S_2(U,V,W)$ using the large 
sieve; we expected, and got, a bound that is better than trivial by a factor of 
size roughly $\gg \sqrt{q} \log x$
 -- the exact factor in the bound depends on the value of $W$. In particular, it is only in the central part of the range for $W$
that we will really be able to save a factor of $\gg \sqrt{q} \log x$, as 
opposed to just $\gg \sqrt{q}$. We will have to be slightly clever in order to
get a good total bound in the end.

\begin{center} 
* * *
\end{center}

We first recall our estimate for $S_1$. In the whole range 
$\lbrack V,x/U\rbrack$ for $W$, we know from (\ref{eq:menson1}),
 (\ref{eq:menson2}) and (\ref{eq:demimond}) that
$S_1(U,W)$ is at most
\begin{equation}\label{eq:nehru2}\frac{2}{\pi^2} \frac{x}{W} + 
\kappa_{0} \zeta(3/2)^3 \frac{x}{W} \sqrt{\frac{x/WU}{U}} ,
\end{equation} where
\[\kappa_{0} = 1.27 .\]
(We recall we are working with $v=2$.)

We have better estimates for the constant in front in some parts of the
range; in what is usually the main part, 
 (\ref{eq:menson2}) and (\ref{eq:velib}) give us a constant of $0.15107$ 
instead of $2/\pi^2$. Note that $1.27 \zeta(3/2)^3 =
22.6417\dotsc$. We should choose $U$, $V$ so that
the first term in (\ref{eq:nehru2})
dominates. For the while being, assume only
\begin{equation}\label{eq:curious}
U\geq 5\cdot 10^5 \frac{x}{VU};
\end{equation} then (\ref{eq:nehru2}) gives
\begin{equation}\label{eq:crudo}
S_1(U,W)\leq \kappa_{1} \frac{x}{W},\end{equation}
where \[\kappa_{1} =  
\frac{2}{\pi^2} + \frac{22.6418}{\sqrt{10^6/2}} \leq  0.2347.\]
This will suffice for our cruder estimates.

The second integral in (\ref{eq:secint}) is now easy to bound.
By (\ref{eq:negli}),
\[S_3(W)\leq 1.0171 x + 2.0341 W \leq 1.0172 x,\]
since $W \leq x/U \leq x/5 \cdot 10^5$. Hence
\[\begin{aligned}
4 \int_V^{x/U} \sqrt{S_1(U,W) \cdot S_3(W)}\; \frac{d W}{W}
&\leq 
4\int_V^{x/U} \sqrt{\kappa_{1} \frac{x}{W}\cdot 1.0172 x}\; \frac{dW}{W}\\
&\leq \kappa_{9} \frac{x}{\sqrt{V}},
\end{aligned}\]
where \[\kappa_{9} = 8 \cdot \sqrt{1.0172\cdot \kappa_{1}}
\leq 3.9086 .\]

Let us now examine $S_2$, which was bounded in Prop. \ref{prop:kraken}. 
We set the parameters $W'$, $U'$ as follows, in accordance with
(\ref{eq:costo}): \[W' = \max(V,W/2),\;\;\;\;\;\;\;
U' = \max(U,x/2W).\] 
Since $W'\geq W/2$ and $W\geq V> 117$, we can always bound
\begin{equation}\label{eq:immer}
\sum_{W'<p\leq W} (\log p)^2 \leq \frac{1}{2} W (\log W).\end{equation}
by (\ref{eq:kast}).

{\em Bounding $S_2$ for $\delta$ arbitrary.}
We set
\[W_0 = \min(\max(2 \theta q,V),x/U),\]
where $\theta\geq e$ is a parameter that will be set later.

For $V\leq W<W_0$, we use the bound (\ref{eq:garn1b}):
\[\begin{aligned}
S_2(U',W',W) &\leq 
\left(\max(1,2\rho) 
\left(\frac{x}{8 q} + \frac{x}{2 W}\right) + \frac{W}{2} + 2 q\right)
\cdot \frac{1}{2} W (\log W) \\
&\leq 
\max\left(\frac{1}{2},\rho\right) 
\left(\frac{W}{8 q} + \frac{1}{2}\right) x \log W + \frac{W^2 \log W}{4} + q W \log W
,\end{aligned}\]
where 
$\rho = q/Q$. 

If $W_0>V$, the contribution of the terms with $V\leq W<W_0$ 
to (\ref{eq:secint}) is 
(by \ref{eq:crudo}) bounded by
\begin{equation}\label{eq:rook}\begin{aligned}
4 &\int_V^{W_0} \sqrt{\kappa_{1}
\frac{x}{W} \left(\frac{\rho_0}{4}
\left(\frac{W}{4q} + 1\right) x \log W + \frac{W^2 \log W}{4} + q W \log W
\right)}\; \frac{dW}{W} \\ 
&\leq \frac{\kappa_{2}}{2} \sqrt{\rho_0}
x \int_V^{W_0} \frac{\sqrt{\log W}}{W^{3/2}} dW +
\frac{\kappa_{2}}{2} \sqrt{x} \int_V^{W_0}  \frac{\sqrt{\log W}}{W^{1/2}} dW \\ &+
\kappa_{2} \sqrt{\frac{\rho_0 x^2}{16 q} + q x}
 \int_V^{W_0}\frac{\sqrt{\log W}}{W} dW\\
&\leq \left(\kappa_{2} \sqrt{\rho_0}
\frac{x}{\sqrt{V}} + \kappa_{2} \sqrt{x W_0}\right) \sqrt{\log W_0}\\
 &+ \frac{2 \kappa_{2}}{3} \sqrt{\frac{\rho_0 x^2}{16 q} + q x}
 \left((\log W_0)^{3/2} - (\log V)^{3/2} \right) 
,\end{aligned}\end{equation}
where $\rho_0 = \max(1,2\rho)$
and
\[\kappa_{2} = 4 \sqrt{\kappa_{1}} \leq 
1.93768 .
\]
(We are using the easy bound $\sqrt{a+b+c}\leq \sqrt{a}+\sqrt{b}+
\sqrt{c}$.)


We now examine the terms with $W\geq W_0$. 
If $2 \theta q > x/U$, then $W_0=x/U$, the contribution of the case is nil,
and the computations below can be ignored. Thus, we can assume that
$2 \theta q\leq x/U$.

We use (\ref{eq:garn1a}): 
\[
S_2(U',W',W)\leq 
\left(\frac{x}{4 \phi(q)} \frac{1}{\log(W/2q)}
 + \frac{q}{\phi(q)} \frac{W}{\log(W/2q)} \right)\cdot \frac{1}{2} W \log W.\]
By $\sqrt{a+b}\leq \sqrt{a}+\sqrt{b}$, we can take out the 
$q/\phi(q) \cdot W/\log(W/2q)$
term and estimate its contribution on its own; it is at most
\begin{equation}\label{eq:vivaldi}\begin{aligned}
4 \int_{W_0}^{x/U} &\sqrt{\kappa_{1} \frac{x}{W} \cdot \frac{q}{\phi(q)}
\cdot \frac{1}{2} W^2 \frac{\log W}{\log W/2q}}\; \frac{dW}{W}\\
&= \frac{\kappa_{2}}{\sqrt{2}} \sqrt{\frac{q}{\phi(q)}}
\int_{W_0}^{x/U} \sqrt{\frac{x \log W}{W \log W/2q}} dW\\
&\leq \frac{\kappa_{2}}{\sqrt{2}} \sqrt{\frac{q x}{\phi(q)}}
\int_{W_0}^{x/U} \frac{1}{\sqrt{W}} \left(1 + \sqrt{\frac{\log 2q}{\log W/2q}}
\right) dW
\end{aligned}\end{equation}
\\ 

Now
\[
\int_{W_0}^{x/U} 
 \frac{1}{\sqrt{W}} \sqrt{\frac{\log 2q}{\log W/2q}} dW =
\sqrt{2q\log 2q} \int_{\max(\theta,V/2q)}^{x/2Uq} \frac{1}{\sqrt{t \log t}} dt.\]
We bound this last integral somewhat crudely: for $T\geq e$,
\begin{equation}\label{eq:notung}
\int_e^T \frac{1}{\sqrt{t \log t}} dt \leq
2.3 \sqrt{\frac{T}{\log T}},
 \end{equation}
(This is shown as follows: since
\[\frac{1}{\sqrt{T \log T}} < \left(
2.3 \sqrt{\frac{T}{\log T}} \right)'\]
if and only if $T>T_0$, where $T_0=e^{(1-2/2.3)^{-1}} = 2135.94\dotsc$,
it is enough to check (numerically) that (\ref{eq:notung}) holds for
$T=T_0$.) Since $\theta\geq e$,
this gives us that
\[\begin{aligned}
\int_{W_0}^{x/U} 
&\frac{1}{\sqrt{W}} \left(1+ 
 \sqrt{\frac{\log 2q}{\log W/2q}}\right) dW 
\\ &\leq 2 \sqrt{\frac{x}{U}} 
+ 2.3 \sqrt{2 q\log 2q}
 \cdot \sqrt{\frac{x/2Uq}{
\log x/2Uq}} ,
\end{aligned}\]
and so (\ref{eq:vivaldi}) is at most
\[\sqrt{2} \kappa_{2} 
\sqrt{\frac{q}{\phi(q)}} \left(1 + 1.15 \sqrt{\frac{\log 2q}{\log x/2Uq}}
\right) \frac{x}{\sqrt{U}}.\]

We are left with what will usually be the main term, viz.,
\begin{equation}\label{eq:soledad}
4\int_{W_0}^{x/U} \sqrt{S_1(U,W)\cdot \left(\frac{x}{8 \phi(q)}
\frac{\log W}{\log W/2q}\right) W} \frac{dW}{W},\end{equation}
which, by (\ref{eq:menson2}),  is at most $x/\sqrt{\phi(q)}$ times the 
integral of
\[
\frac{1}{W}
\sqrt{\left(2 H_2\left(\frac{x}{WU}\right) + \frac{\kappa_{4}}{2}
\sqrt{\frac{x/WU}{U}}\right) \frac{\log W}{\log W/2q}}
\]
for $W$ going from $W_0$ to $x/U$, where $H_2$ is as in (\ref{eq:palmiped})
and \[\kappa_{4} = 4 \kappa_{0} \zeta(3/2)^3 \leq 90.5671 .\]
By the arithmetic/geometric mean inequality, the integrand is at most
$1/W$ times
\begin{equation}\label{eq:dane}
\frac{\beta + \beta^{-1}\cdot 2 H_2(x/WU)}{2} + \frac{\beta^{-1}}{2} 
\frac{\kappa_{4}}{2} \sqrt{\frac{x/WU}{U}} 
+  \frac{\beta}{2} \frac{\log 2q}{\log W/2q}\end{equation}
for any $\beta>0$. We will choose $\beta$ later.

The first summand in (\ref{eq:dane})
gives what we can think of as the main or worst term
in the whole paper; let us compute it first. The integral is
\begin{equation}\label{eq:quartma}\begin{aligned}
\int_{W_0}^{x/U} \frac{\beta+\beta^{-1}\cdot 2 H_2(x/WU)}{2} \frac{dW}{W} &=
\int_1^{x/U W_0} \frac{\beta+\beta^{-1}\cdot 2 H_2(s)}{2} \frac{ds}{s}
\\ &\leq \left(\frac{\beta}{2} + \frac{\kappa_{6}}{4\beta}\right)
 \log \frac{x}{U W_0} 
\end{aligned}\end{equation}
by (\ref{eq:velib}),
where \[\kappa_{6} = 0.60428.\] 

Thus the main term is simply
\begin{equation}\label{eq:gogolo}
\left(\frac{\beta}{2} + \frac{\kappa_{6}}{4 \beta}\right) 
\frac{x}{\sqrt{\phi(q)}} \log \frac{x}{U W_0}.
\end{equation}

The integral of the second summand is at most
\[\begin{aligned}
\beta^{-1} \cdot \frac{\kappa_{4}}{4} \frac{\sqrt{x}}{U} 
\int_V^{x/U} \frac{d W}{W^{3/2}} 
&\leq \beta^{-1} \cdot \frac{\kappa_{4}}{2} \sqrt{\frac{x/UV}{U}} .
\end{aligned}\]
By (\ref{eq:curious}), this is at most 
\[\frac{\beta^{-1}}{\sqrt{2}}\cdot  10^{-3} \cdot \kappa_{4} \leq 
\beta^{-1} \kappa_{7}/2,\]
where \[\kappa_{7}= \frac{\sqrt{2} \kappa_{4}}{1000} \leq
 0.1281 .\]
Thus the contribution
of the second summand is at most
\[\frac{\beta^{-1} \kappa_{7}}{2}\cdot \frac{x}{\sqrt{\phi(q)}} .\]
The integral of the third summand in (\ref{eq:dane}) is
\begin{equation}\label{eq:resist}
\frac{\beta}{2} \int_{W_0}^{x/U} \frac{\log 2q}{\log W/2q} \frac{dW}{W}.
\end{equation}
If $V < 2 \theta q\leq x/U$, this is
\[\begin{aligned}
\frac{\beta}{2} \int_{2 \theta q}^{x/U} \frac{\log 2q}{\log W/2q} 
\frac{dW}{W} &=
\frac{\beta}{2} \log 2q \cdot \int_{\theta}^{x/2Uq} \frac{1}{\log t} \frac{dt}{t}\\
&= \frac{\beta}{2} \log 2q \cdot \left(\log \log \frac{x}{2Uq} - \log \log 
\theta\right).
\end{aligned}\]
If $2\theta q > x/U$, 
the integral is over an empty range and its 
contribution is hence $0$.

If $2\theta q\leq V$, (\ref{eq:resist}) is 
\begin{equation}\label{eq:wofov}\begin{aligned}
\frac{\beta}{2} \int_V^{x/U} \frac{\log 2q}{\log W/2q} \frac{dW}{W} &=
\frac{\beta \log 2q}{2} \int_{V/2q}^{x/2Uq} \frac{1}{\log t} \frac{dt}{t} \\&=
\frac{\beta \log 2q}{2} \cdot (\log \log \frac{x}{2Uq} - \log \log V/2q)\\ &=
\frac{\beta \log 2q}{2} \cdot \log \left(1+ \frac{\log x/UV}{\log V/2q}\right).
\end{aligned}\end{equation}
(Let us stop for a moment and ask ourselves when this will be smaller than what we
can see as the main term, namely, the term $(\beta/2) \log x/U W_0$
in (\ref{eq:quartma}). Clearly,
 $\log (1+(\log x/UV)/(\log V/2q))\leq (\log x/UV)/(\log V/2q)$,
and that
is smaller than $(\log x/UV)/\log 2q$ when $V/2q>2q$. Of course, it does not 
actually matter if (\ref{eq:wofov}) is smaller than the term from
(\ref{eq:quartma}) or not, since we are looking for upper bounds here, not
for asymptotics.)

The total bound for (\ref{eq:soledad}) is thus
\begin{equation}\label{eq:egmont}
\frac{x}{\sqrt{\phi(q)}} \cdot \left(
\beta \cdot \left(\frac{1}{2} \log \frac{x}{U W_0} + \frac{\Phi}{2}\right) +
\beta^{-1} \left(\frac{1}{4} \kappa_{6} \log \frac{x}{U W_0} + 
\frac{\kappa_{7}}{2}
\right)\right),\end{equation}
where
\begin{equation}\label{eq:cocot}
\Phi = 
\begin{cases}
\log 2q
\left(\log \log \frac{x}{2Uq} - \log \log \theta\right) 
&\text{if $V/2\theta<q<x/(2\theta U)$.}\\
\log 2q
\log \left(1+ \frac{\log x/UV}{\log V/2q}\right) 
&\text{if $q\leq V/2\theta$.}
\end{cases}
\end{equation}
Choosing $\beta$ optimally, we obtain that (\ref{eq:soledad}) is at most
\begin{equation}\label{eq:valmont}
  \frac{x}{\sqrt{2 \phi(q)}} \sqrt{\left(\log \frac{x}{U W_0} + \Phi\right)
\left(\kappa_{6} \log \frac{x}{U W_0} + 2 \kappa_{7}\right)},
\end{equation}
where $\Phi$ is as in (\ref{eq:cocot}).

{\em Bounding $S_2$ for $|\delta|\geq 8$.} 
Let us see how much a non-zero $\delta$ can help us. It makes
sense to apply (\ref{eq:procida2}) only when $|\delta|\geq 8$; otherwise
(\ref{eq:garn1a}) is almost certainly better. Now, by definition,
$|\delta|/x\leq 1/qQ$, and so $|\delta|\geq 8$ can happen only when 
$q \leq x/8Q$. 

With this in mind, let us apply (\ref{eq:procida2}), assuming $|\delta|>8$.
Note first that
\[\begin{aligned}
\frac{x}{|\delta q|} \left(q + \frac{x}{4 W}\right)^{-1} &\geq
\frac{1/|\delta q|}{\frac{q}{x} + \frac{1}{4 W}} \geq
\frac{4/|\delta q|}{\frac{1}{2 Q} + \frac{1}{W}}\\ &\geq
\frac{4 W}{|\delta| q} \cdot \frac{1}{1 + \frac{W}{2 Q}}
\geq \frac{4 W}{|\delta| q} \cdot \frac{1}{1 + \frac{x/U}{2 Q}}
.\end{aligned}\]
This is at least $2 \min(2 Q,W)/|\delta q|$.
Thus we are allowed to apply (\ref{eq:procida2})
when $|\delta q|\leq 2 \min(2Q,W)$.
Since $Q\geq x/U$, we know that $\min(2Q,W)=W$ for all $W\leq x/U$, and
so it is enough to assume that $|\delta q|\leq 2 W$. 
We will soon be making a stronger assumption.

Recalling also (\ref{eq:immer}), we see that (\ref{eq:procida2}) gives us
\begin{equation}\label{eq:thislife}\begin{aligned}
S_2(U',W',W) &\leq \min\left(1,\frac{2 q/\phi(q)}{
\log \left(\frac{4 W}{|\delta| q} \cdot \frac{1}{1 + \frac{x/U}{2 Q}}
\right)}\right)
 \left(\frac{x}{|\delta q|} + \frac{W}{2}\right)
\cdot \frac{1}{2} W (\log W).\end{aligned}\end{equation}

Similarly to before, we define $W_0 = \max(V, \theta |\delta q|)$, where
$\theta\geq 3 e^2/8$ will be set later. (Here $\theta\geq 3 e^2/8$ is an assumption we do not yet need, but we will be using it soon to simplify matters 
slightly.)
For $W\geq W_0$, we certainly have
$|\delta q|\leq 2 W$. Hence the part of the first term of
 (\ref{eq:secint}) coming from
the range $W_0\leq W < x/U$ is
\begin{equation}\label{eq:tort}
\begin{aligned}4 &\int_{W_0}^{x/U} \sqrt{S_1(U,W) \cdot S_2(U,V,W)} \frac{dW}{W} 
\\ &\leq 4 \sqrt{\frac{q}{\phi(q)}} \int_{W_0}^{x/U} 
 \sqrt{S_1(U,W) \cdot \frac{\log W}{
\log \left(\frac{4 W}{|\delta| q} \cdot \frac{1}{1 + \frac{x/U}{2 Q}}\right)}
   \left(\frac{W x}{|\delta q|} + \frac{W^2}{2}\right)} \frac{dW}{W}.
\end{aligned}\end{equation}
By (\ref{eq:menson2}), the contribution of the term $W x/|\delta q|$ to
(\ref{eq:tort}) is at most
\[\frac{4 x}{\sqrt{|\delta| \phi(q)}} 
\int_{W_0}^{x/U} \sqrt{\left(H_2\left(\frac{x}{WU}\right) + 
\frac{\kappa_4}{4} \sqrt{\frac{x/WU}{U}}\right) 
\frac{\log W}{\log \left(
\frac{4 W}{|\delta| q} \cdot \frac{1}{1 + \frac{x/U}{2 Q}}\right)}}
 \frac{dW}{W}\]
Note that $1+(x/U)/2Q\leq 3/2$.
Proceeding as in (\ref{eq:soledad})--(\ref{eq:valmont}), we obtain that this is
at most 
\[\frac{2 x}{\sqrt{|\delta| \phi(q)}}
\sqrt{\left(\log \frac{x}{U W_0} + \Phi\right) \left(\kappa_{6}
\log \frac{x}{U W_0} + 2 \kappa_{7}\right)},\]
where \begin{equation}\label{eq:regxo1}
\Phi = \begin{cases}
\log \frac{(1+\epsilon_1) |\delta q|}{4} \log \left(1 + \frac{\log x/UV}{
\log 4 V/|\delta| (1+ \epsilon_1) q}\right) &\text{if $|\delta q|\leq V/\theta$,}\\
\log \frac{3 |\delta q|}{8} 
\left(\log \log \frac{8 x}{3 U |\delta q|} - \log \log \frac{8 \theta}{3}\right)
&\text{if $V/\theta < |\delta q| \leq x/\theta U$,}
\end{cases}\end{equation}
where $\epsilon_1 = x/2 U Q$. This is what we think of as the main term.

By (\ref{eq:crudo}), the contribution of the term 
$W^2/2$ to (\ref{eq:tort}) is at most
\begin{equation}\label{eq:heeren}
4 \sqrt{\frac{q}{\phi(q)}} \int_{W_0}^{x/U} \sqrt{\frac{\kappa_1}{2} x} \frac{
dW}{\sqrt{W}}
\cdot \max_{W_0\leq W\leq \frac{x}{U}} \sqrt{\frac{\log W}{\log \frac{8 W}{3 |\delta
q|}}}.\end{equation}
Since $t\to (\log t)/(\log t/c)$ is decreasing for $t>c$, (\ref{eq:heeren})
 is at most
\begin{equation}\label{eq:filmot}
 4 \sqrt{2 \kappa_1} \sqrt{\frac{q}{\phi(q)}} \left(\frac{x}{\sqrt{U}}
- \sqrt{x W_0}\right)
\sqrt{\frac{\log W_0}{\log \frac{8 W_0}{3 |\delta q|}}}.\end{equation}

If $W_0 > V$, we also have to consider the range $V\leq W < W_0$.
By Prop. \ref{prop:kraken} 
and (\ref{eq:immer}), the part of (\ref{eq:secint}) coming from this is
\[
4 \int_{V}^{\theta |\delta q|} \sqrt{S_1(U,W) \cdot
(\log W) \left(\frac{W x}{2 |\delta q|} + \frac{W^2}{4} + 
\frac{W x}{16 (1-\rho) Q} +
\frac{x}{8 (1-\rho)}\right)} \frac{dW}{W}. 
\]
The contribution of $W^2/4$ is at most
\[4 \int_V^{W_0} \sqrt{\kappa_1 \frac{x}{W} \log W \cdot \frac{W^2}{4}} \frac{dW}{W}
\leq 4 \sqrt{\kappa_1} \cdot \sqrt{x W_0} \cdot \sqrt{\log W};\]
the sum of this and (\ref{eq:filmot}) is at most
\[\begin{aligned}
4 \sqrt{\kappa_1} &\left(\sqrt{\frac{2 q}{\phi(q)}} \left(\frac{x}{\sqrt{U}}
- \sqrt{x W_0}\right) \sqrt{\frac{\log W_0}{\log \frac{8 \theta}{3}}}
+ \sqrt{x W_0} \sqrt{\log W_0}\right)\\
&\leq \kappa_2 \cdot \sqrt{\frac{q}{\phi(q)}} \frac{x}{\sqrt{U}}
\sqrt{\log W_0},
\end{aligned}\]
where we use the facts that $W_0 = \theta |\delta q|$ (by $W_0>V$) and
$\theta\geq 3 e^2/8$, and where we recall that $\kappa_2 = 4 \sqrt{\kappa_1}$.

The terms $W x/2 |\delta| q$ and $W x/(16 (1-\rho) Q)$  contribute at most
\[\begin{aligned} 4 \sqrt{\kappa_1} \int_{V}^{\theta |\delta q|} 
&\sqrt{\frac{x}{W} \cdot (\log W) W \left(
\frac{x}{2 |\delta q|} + \frac{x}{16 (1-\rho) Q}\right)
} \frac{dW}{W}\\
&= 
\kappa_2 x \left(\frac{1}{\sqrt{2 |\delta| q}} +
\frac{1}{4 \sqrt{(1-\rho) Q}}\right) 
 \int_{V}^{\theta |\delta q|}  
\sqrt{\log W}\; \frac{dW}{W}\\
&= \frac{2 \kappa_2}{3} x \left(\frac{1}{\sqrt{2 |\delta| q}} +
\frac{1}{4 \sqrt{(1-\rho) Q}}\right) \left((\log \theta |\delta| q)^{3/2} -
(\log V)^{3/2}\right).\end{aligned}
\]
The term $x/8(1-\rho)$ contributes
\[\begin{aligned}
\sqrt{2 \kappa_1} x \int_{V}^{\theta |\delta q|} \sqrt{\frac{\log W}{W (1-\rho)}}
\frac{dW}{W} &\leq \frac{\sqrt{2 \kappa_1} x}{\sqrt{1-\rho}} 
\int_V^\infty \frac{\sqrt{\log W}}{W^{3/2}} dW\\ &\leq
\frac{\kappa_2 x}{\sqrt{2 (1-\rho) V}} (\sqrt{\log V} + \sqrt{1/\log V}), 
\end{aligned}\]

where we use the estimate 
\[\begin{aligned}
\int_V^{\infty} \frac{\sqrt{\log W}}{W^{3/2}} dW &= 
\frac{1}{\sqrt{V}} \int_1^{\infty} \frac{\sqrt{\log u + \log V}}{u^{3/2}} du\\
&\leq \frac{1}{\sqrt{V}} \int_1^{\infty} \frac{\sqrt{\log V}}{u^{3/2}} du +
\frac{1}{\sqrt{V}} \int_1^{\infty} \frac{1}{2\sqrt{\log V}} 
\frac{\log u}{u^{3/2}} du \\
&= 2 \frac{\sqrt{\log V}}{\sqrt{V}} + \frac{1}{2\sqrt{ V\log V}}\cdot 4
\leq \frac{2}{\sqrt{V}} \left(\sqrt{\log V} + \sqrt{1/\log V}\right) .
\end{aligned}\]

\begin{center} * * * \end{center}

It is time to collect all type II terms. Let us start with the case of
general $\delta$. We will set $\theta\geq e$ later. If $q\leq V/2\theta$,
then $|S_{II}|$ is at most
\begin{equation}\label{eq:vinland1}\begin{aligned}
&\frac{x}{\sqrt{2 \phi(q)}} \cdot 
\sqrt{\left(\log \frac{x}{U V} + \log 2q \log \left(1 + \frac{\log x/UV}{\log V/2q}\right)\right) \left(\kappa_{6} \log
\frac{x}{U V} + 2 \kappa_{7}\right)}\\
&+ \sqrt{2} \kappa_{2} \sqrt{\frac{q}{\phi(q)}} \left(1 + 1.15
\sqrt{\frac{\log 2q}{\log x/2Uq}}
\right) \frac{x}{\sqrt{U}} +
 \kappa_{9} \frac{x}{\sqrt{V}} .
\end{aligned}
\end{equation}
If $V/2\theta <q \leq x/2\theta U$, then $|S_{II}|$ is at most
\begin{equation}\label{eq:vinland2}\begin{aligned}
&\frac{x}{\sqrt{2 \phi(q)}} \cdot 
\sqrt{\left(\log \frac{x}{U\cdot 2\theta q} + \log 2q \log \frac{\log x/2Uq}{\log \theta}\right) 
\left(\kappa_{6} \log
\frac{x}{U \cdot 2 \theta q} + 2 \kappa_{7}\right)}\\
&+ \sqrt{2} \kappa_{2} 
\sqrt{\frac{q}{\phi(q)}} \left(1 + 1.15\sqrt{\frac{\log 2q}{\log x/2Uq}}
\right) \frac{x}{\sqrt{U}} +
(\kappa_{2} \sqrt{\log 2\theta q} + \kappa_{9}) \frac{x}{\sqrt{V}}\\
&+
\frac{\kappa_{2}}{6}
\left((\log 2\theta q)^{3/2} - (\log V)^{3/2}\right) \frac{x}{\sqrt{q}} \\
&+ \kappa_{2} \left(\sqrt{2\theta \cdot \log 2\theta q} + \frac{2}{3}
((\log 2\theta q)^{3/2} - (\log V)^{3/2}) \right) \sqrt{q x} ,
\end{aligned}\end{equation}
where we use the fact that $Q\geq x/U$ (implying that
$\rho_0 = \max(1,2q/Q)$ equals $1$ for $q\leq x/2 U$).
Finally, if $q>x/2\theta U$,
\begin{equation}\label{eq:vinland3}\begin{aligned}|S_{II}|&\leq 
(\kappa_{2} \sqrt{2 \log x/U} + \kappa_{9}) \frac{x}{\sqrt{V}}
+ \kappa_{2} \sqrt{\log x/U} \frac{x}{\sqrt{U}}\\
&+ \frac{2 \kappa_{2}}{3} ((\log x/U)^{3/2} - (\log V)^{3/2})
\left(\frac{x}{2 \sqrt{2 q}} + \sqrt{q x}\right).
\end{aligned}\end{equation}

Now let us examine the alternative bounds for $|\delta|\geq 8$.
Here we assume $\theta\geq 3 e^2/8$.
If $|\delta q|\leq V/\theta$, then $|S_{II}|$ is at most 
\begin{equation}\label{eq:eriksaga}
\begin{aligned}
&\frac{2 x}{\sqrt{|\delta| \phi(q)}}
\sqrt{
\log \frac{x}{U V} + \log \frac{|\delta q| (1+\epsilon_1)}{4} \log 
\left(1 + \frac{\log x/UV}{\log \frac{4 V}{
|\delta| (1+\epsilon_1) q}}\right)}\\
&\cdot
\sqrt{\kappa_{6}
\log \frac{x}{U V} + 2 \kappa_{7}}
\\ &+ \kappa_{2} \sqrt{\frac{2 q}{\phi(q)}} 
\cdot \sqrt{\frac{\log V}{\log 2 V/|\delta q|}} \cdot \frac{x}{\sqrt{U}}
+ \kappa_9 \frac{x}{\sqrt{V}},
\end{aligned}
\end{equation}
where $\epsilon_1 = x/2UQ$.
If $V/\theta < |\delta| q \leq x/\theta U$, then $|S_{II}|$ is at most 
\begin{equation}\label{eq:vinlandsaga}
\begin{aligned}
&\frac{2 x}{\sqrt{|\delta| \phi(q)}}
\sqrt{\left(\log \frac{x}{U\cdot \theta |\delta| q} + 
\log \frac{3 |\delta q|}{8} 
\log \frac{\log \frac{8 x}{ 3 U |\delta q|}}{\log 8\theta/3}\right) 
\left(\kappa_{6}
\log \frac{x}{U\cdot \theta |\delta q|} + 2 \kappa_{7}\right)}\\
&+ \frac{2 \kappa_{2}}{3} 
\left(\frac{x}{\sqrt{2 |\delta q|}} +
\frac{x}{4 \sqrt{Q-q}}\right) 
\left((\log \theta |\delta q|)^{3/2} - (\log V)^{3/2}\right)
\\ &+
\left(\frac{\kappa_2}{\sqrt{2 (1-\rho)}} \left(\sqrt{\log V} + 
\sqrt{1/\log V}\right) + \kappa_{9}\right) \frac{x}{\sqrt{V}} 
\\ &+
\kappa_2 \sqrt{\frac{q}{\phi(q)}} \cdot \sqrt{\log \theta |\delta q|}
\cdot \frac{x}{\sqrt{U}},
\end{aligned}
\end{equation}
where $\rho = q/Q$.
Note that $|\delta|\leq x/Q q$ implies
$\rho\leq x/Q^2$, and so $\rho$ will be very small
and $Q-q$ will be very close to $Q$.

The case $|\delta q|>x/\theta U$ will not arise in practice,
essentially because of $|\delta| q \leq x/Q$.

%

\section{Adjusting parameters. Calculations.}\label{subs:totcho}
We must bound the exponential sum $\sum_n \Lambda(n) e(\alpha n)
\eta(n/x)$. By (\ref{eq:bob}), it is enough to sum the bounds
we obtained in \S \ref{subs:putmal}.
We will now see how it will be best to set $U$, $V$ and other parameters.

Usually, the largest terms will be
\begin{equation}\label{eq:llama}
C_0 U V,\end{equation}
where $C_0$ equals
\begin{equation}\label{eq:guanaco}
\begin{cases}
c_{4,I_2} + c_{9,I_2} = 4.39779+5.21993 \epsilon_0 &\text{if $|\delta|\leq 1/2c_2 \sim 0.74463$,}\\
c_{4,I_2} + (1+\epsilon) c_{13,I_2} = (4.89106 + 1.31541 \epsilon) (1+\epsilon_0)
&\text{if $|\delta|> 1/2 c_2$}\end{cases}
\end{equation}
(from (\ref{eq:fausto}) and (\ref{eq:magus}), type I; we will specify
$\epsilon$ and $\epsilon_0 = (4\log 2)/(x/UV)$ later) and
\begin{equation}\label{eq:codo1}\begin{aligned}
\frac{x}{\sqrt{\delta_0 \phi(q)}}
&\sqrt{
\log \frac{x}{U V} + (\log \delta_0 (1+\epsilon_1) q) \log 
\left(1 + \frac{\log \frac{x}{UV}}{\log \frac{V}{\delta_0 (1+\epsilon_1) q}}\right)}\\ 
&\cdot \sqrt{\kappa_{6}
\log \frac{x}{U V} + 2 \kappa_{7}}
\end{aligned}\end{equation}
(from (\ref{eq:vinland1}) and (\ref{eq:eriksaga}), type II; here
 $\delta_0 = \max(2,|\delta|/4)$, while $\epsilon_1 = x/2 U Q$ for 
$|\delta|>8$ and $\epsilon_1 = 0$ for $|\delta|<8$.


We set $UV = \varkappa x/\sqrt{q \delta_0}$; we must choose $\varkappa>0$.

Let us first optimize (or, rather, almost optimize) $\varkappa$ in the case $|\delta|\leq 4$, so that
$\delta_0 = 2$ and $\epsilon_1=0$. For the purpose of choosing $\varkappa$,
we replace $\sqrt{\phi(q)}$ by 
$\sqrt{q}/C_1$, where $C_1 = 2.3536 \sim 510510/\phi(510510)$,
and also replace $V$ by $q^2/c$, $c$ a constant.
We use the approximation
\[\begin{aligned}
\log \left(1 + \frac{\log \frac{x}{U V}}{\log \frac{V}{|2 q|}}\right)
&= \log \left(1 + \frac{\log(\sqrt{2 q}/\varkappa)}{\log(q/2c)}\right) =
\log \left(\frac{3}{2} + \frac{\log 2 \sqrt{c}/\varkappa}{\log q/2c}\right)\\ 
&\sim
\log \frac{3}{2} + \frac{2 \log 2 \sqrt{c}/\varkappa}{
3 \log q/2c}.
\end{aligned}\]
What we must minimize, then, is
\begin{equation}\label{eq:cojono}\begin{aligned}
&\frac{C_0 \varkappa}{\sqrt{2 q}} + \frac{C_1}{\sqrt{2 q}}
\sqrt{\left(\log \frac{\sqrt{2 q}}{\varkappa} + \log 2q \left(\log \frac{3}{2}
+ \frac{2 \log \frac{2 \sqrt{c}}{\varkappa}}{
3 \log \frac{q}{2c}}
\right)\right) \left(\kappa_{6}
\log \frac{\sqrt{2 q}}{\varkappa} + 
2 \kappa_{7}\right)}\\
 &\leq \frac{C_0 \varkappa}{\sqrt{2 q}}  + \frac{C_1}{2 \sqrt{q}}
\frac{\sqrt{\kappa_{6}}}{\sqrt{\kappa_1'}}
\sqrt{\kappa_1' \log q 
- \left(\frac{5}{3} + \frac{2}{3} \frac{\log 4c }{\log \frac{q}{2c}}\right) \log \varkappa + \kappa_2'}
\\ &\cdot \sqrt{
\kappa_1' \log q - 2 \kappa_1'
 \log \varkappa + \frac{4 \kappa_1' \kappa_{7}}{\kappa_{6}} +
\kappa_1' \log 2} \\
&\leq
\frac{C_0}{\sqrt{2 q}} \left(\varkappa  + \kappa_{4}' 
\left(\kappa_1' \log q - \left(\left(\frac{5}{6} + \kappa_1'\right) 
+ \frac{1}{3} \frac{\log 4c}{\log \frac{q}{2c}}\right)
\log \varkappa + \kappa_{3}'\right)\right),\end{aligned}\end{equation}
where 
\[\begin{aligned}
\kappa_1' &= \frac{1}{2} + \log \frac{3}{2},\;\;\;\;
\kappa_2' = \log \sqrt{2} + \log 2 \log \frac{3}{2} + \frac{\log 4 c \log 2 q}{3 \log q/2c},\\
\kappa_{3}' &= \frac{1}{2} \left(\kappa_2' + \frac{4 \kappa_1'
\kappa_{7}}{\kappa_{6}} + \kappa_1' \log 2 \right) = 
\frac{\log 4 c}{6} + \frac{(\log 4c)^2}{6 \log \frac{q}{2 c}} + 
\kappa_5',\\
\kappa_{4}' &= 
\frac{C_1}{C_0} \sqrt{\frac{\kappa_{6}}{2 \kappa_1'}} \sim
\begin{cases}
\frac{0.30915}{1 + 1.18694\epsilon_0} & \text{if $|\delta|\leq 4$}\\
\frac{0.27797}{(1 + 0.26894 \epsilon)(1+\epsilon_0)} & \text{if $|\delta|>4$},\end{cases}\\
\kappa_5' &= \frac{1}{2} (\log \sqrt{2} + \log 2 \log \frac{3}{2} +
\frac{4 \kappa_1' \kappa_7}{\kappa_6} + \kappa_1' \log 2) \sim
1.01152 . 
\end{aligned}\]
Taking derivatives, we see that the minimum is attained when
\begin{equation}\label{eq:jotoka}
\varkappa = \left(\frac{5}{6}+\kappa_1'+
\frac{1}{3} \frac{\log 4c}{\log \frac{q}{2c}}\right) 
\kappa_{4}' \sim \left(1.7388 + \frac{\log 4c}{3 \log \frac{q}{2c}}\right) 
\cdot \frac{0.30915}{1+1.19 \epsilon_0} 
\end{equation}
provided that $|\delta|\leq 4$.
(What we obtain for $|\delta|>4$ is essentially the same, only with 
$\delta_0 q = \delta q/4$ instead of $2 q$
and $0.27797/((1+0.27 \epsilon) (1+\epsilon_0))$ in place of $0.30915$.) 
For $q=5\cdot 10^5$, $c=2.5$ and $|\delta|\leq 4$ 
(typical values in the most delicate range), we get that
 $\varkappa$ should be about $0.5582/(1+1.19\epsilon_0)$.
Values of $q$, $c$ nearby give similar values for $\varkappa$, whether
$|\delta|\leq 4$ or for $|\delta|>4$.

(Incidentally, at this point, we could already give a back-of-the-envelope 
estimate for the last line of (\ref{eq:cojono}), i.e., our main term.
It suggests that choosing $w=1$
instead of $w=2$ would have given bounds worse by about $15$ percent.)

We make the choices
\[
\varkappa = 1/2,\;\;\;\; \text{and so}\;\;\;\;\;\; UV = \frac{x}{2 \sqrt{q \delta_0}}
\]
for the sake of simplicity. (Unsurprisingly, (\ref{eq:cojono}) changes
very slowly around its minimum.) Note, by the way, that this means that
$\epsilon_0 = (2 \log 2)/\sqrt{ q \delta_0}$.


Now we must decide how to choose $U$, $V$ and $Q$, given our choice of 
$UV$. We will actually make two sets of choices. 

First, we will use the
$S_{I,2}$ estimates for $q\leq Q/V$ to treat all $\alpha$ of the form
$\alpha = a/q + O^*(1/qQ)$, $q\leq y$. (Here $y$ is a parameter satisfying
$y\leq Q/V$.) 

Then, the remaining $\alpha$ will
get treated with the (coarser) $S_{I,2}$ estimate for $q>Q/V$, with 
 $Q$ reset to a lower value (call it $Q'$). If $\alpha$ was
not treated in the first go (so that it must be dealt with the coarser
estimate) then $\alpha = a'/q'+\delta'/x$, where either $q'>y$ or
$\delta' q' > x/Q$. (Otherwise, $\alpha = a'/q'+O^*(1/q'Q)$ would be a valid 
estimate with $q'\leq y$.)
The value of $Q'$ is set to be smaller than $Q$
both because this is helpful (it diminishes error terms that would be large
for large $q$) and because this is harmless (since we are no longer
assuming that $q\leq Q/V$). 

\subsection{First choice of parameters: $q\leq y$}\label{subs:cojor}

The largest items affected strongly by our choices at this point are
\begin{equation}\label{eq:ned1}\begin{aligned}
c_{16,I_2} \left(2 + \frac{1+\epsilon}{\epsilon} \log^+ \frac{2 U V |\delta| q}{
x}\right) \frac{x}{Q/V} + c_{17,I_2} Q
 \;\;\;\;\;\;\text{(from $S_{I,2}$, $|\delta|> 1/2c_2$)},\\
\left(c_{10,I_2} \log \frac{U}{q} + 2 c_{5,I_2} + c_{12,I_2}\right) Q
 \;\;\;\;\;\;\text{(from $S_{I,2}$, $|\delta|\leq 1/2c_2$)},
\end{aligned}\end{equation}
and
\begin{equation}\label{eq:ned2}
\kappa_{2} \sqrt{\frac{2 q}{\phi(q)}} \left(1 + 1.15\sqrt{\frac{\log 2q}{\log x/2Uq}}
\right) \frac{x}{\sqrt{U}} +
 \kappa_{9} \frac{x}{\sqrt{V}} 
\;\;\;\; \text{(from $S_{II}$, any $|delta|$)},
\end{equation}
with 
\[
\kappa_{2} \sqrt{\frac{2 q}{\phi(q)}} 
\cdot \sqrt{\frac{\log V}{\log 2 V/|\delta q|}} \cdot \frac{x}{\sqrt{U}}
\;\;\;\; \text{(from $S_{II}$)}
\]
as an alternative to (\ref{eq:ned2}) for $|\delta|\geq 8$. (In several
of these expressions, we are applying some minor simplifications that 
our later choices will justify. Of course, even if these 
simplifications were not justified, we would not be getting incorrect results,
only potentially suboptimal ones; we are trying to decide how choose certain
parameters.)

In addition, we have a relatively mild but important dependence on $V$
in the main term (\ref{eq:codo1}), even when we hold $U V$ constant
(as we do, in so far as we have already chosen $U V$).
 We must also respect the 
condition $q\leq Q/V$, the lower bound on
$U$ given by (\ref{eq:curious}), and the assumptions made at the
beginning of the chapter (e.g. $Q\geq x/U$, $V\geq 2\cdot 10^6$). 
Recall that $UV = x/2 \sqrt{q \delta_0}$.

We set 
\[Q = \frac{x}{8 y},\]
since we will then have not just $q\leq y$ but also $q |\delta| \leq
x/Q = 8y$, and so $q \delta_0 \leq 2 y$.
We want $q\leq Q/V$ to be true whenever $q\leq y$; this means that
\[q\leq \frac{Q}{V} = \frac{Q U}{U V} = \frac{Q U}{x/2 \sqrt{q \delta_0}}
= \frac{U \sqrt{q \delta_0}}{4 y}\]
must be true when $q\leq y$, and so it is enough to set
$U = 4 y^2/\sqrt{q \delta_0}$.
The following choices make sense: we will work with the
parameters
\begin{equation}\label{eq:humid}\begin{aligned}
y &= \frac{x^{1/3}}{6},\;\;\;\;\;\; 
Q=\frac{x}{8 y} = \frac{3}{4} x^{2/3},\;\;\;\;\;\; x/UV = 2 \sqrt{q \delta_0}\leq 2 \sqrt{2 y},\\
U &= \frac{4 y^2}{\sqrt{q \delta_0}}
= \frac{x^{2/3}}{9 \sqrt{q \delta_0}},
\;\;\;\;\;\; 
V = \frac{x}{(x/UV)\cdot U} 
= 
\frac{x}{8 y^2} = \frac{9 x^{1/3}}{2},
\end{aligned}\end{equation}
where, as before, $\delta_0 = \max(2,|\delta|/4)$. 
So, for instance, we obtain $\epsilon_1 \leq x/2 U Q = 6 \sqrt{q \delta_0}/x^{1/3} \leq
2 \sqrt{3}/x^{1/6}$. Assuming
\begin{equation}\label{eq:voil}
x \geq 2.16 \cdot 10^{20},\end{equation}
we obtain that 
$U/(x/UV) \geq (x^{2/3}/9\sqrt{q\delta_0})/(2 \sqrt{q \delta_0}) = x^{2/3}/18
q \delta_0\geq x^{1/3}/6 \geq 10^6$,
and so (\ref{eq:curious}) holds. We also get that $\epsilon_1\leq 0.002$.

Since $V=x/8 y^2 = (9/2) x^{1/3}$, (\ref{eq:voil}) also implies that
$V\geq 2\cdot 10^6$ (in fact, $V\geq 27\cdot 10^6$). It is easy to check that
\begin{equation}\label{eq:herring}
V<x/4,\;\;\; UV\leq x,\;\;\;\;
Q\geq \max(16,2\sqrt{x}),\;\;\;\;Q\geq \max(2U,x/U),\end{equation}
as stated at the beginning of the chapter.
Let $\theta = (3/2)^3 = 27/8$. Then
\begin{equation}\label{eq:werto}\begin{aligned}
\frac{V}{2 \theta q} &= \frac{x/8 y^2}{2 \theta q} \geq \frac{x}{16 \theta y^3
} = \frac{x}{54 y^3} = 4 > 1,\\
\frac{V}{\theta |\delta q|} &\geq \frac{x/8 y^2}{8 \theta y} 
\geq \frac{x}{64 \theta y^3} = \frac{x}{216 y^3} = 1.\end{aligned}
\end{equation}



The first type I bound is
\begin{equation}\label{eq:dikaiopolis}\begin{aligned}
&|S_{I,1}| \leq 
\frac{x}{q} \min\left(1, \frac{c_0'}{\delta^2}\right)
\left(\min\left(\frac{\frac{4}{5} \frac{q}{\phi(q)}}{\log^+ 
\frac{x^{\frac{2}{3}}/9}{q^{\frac{5}{2}} \delta_0^\frac{1}{2}}}, 1\right) \left(\log 9 x^{\frac{1}{3}} \sqrt{q \delta_0}
+ c_{3,I}\right)
 + \frac{c_{4,I} q}{\phi(q)} \right)\\
&+ \left(c_{7,I}
\log \frac{y}{c_2} + c_{8,I} \log x \right) y 
+ \frac{c_{10,I} x^{1/3}}{3^4 2^2 q^{3/2} \delta_0^{\frac{1}{2}}} (\log 9 x^{1/3}
\sqrt{e q \delta_0})\\
&+ \left(c_{5,I} \log \frac{2 x^{2/3}}{9 c_2 \sqrt{q \delta_0}} +
c_{6,I} \log \frac{x^{5/3}}{9 \sqrt{q \delta_0}}\right) \frac{x^{2/3}}{9 \sqrt{q
\delta_0}} 
+ c_{9,I} \sqrt{x} \log \frac{2 \sqrt{e} x}{c_2} 
+ \frac{c_{10,I}}{e}
,\end{aligned}\end{equation}
where the constants are as in \S \ref{subs:renzo}.
For any $c,R\geq 1$, 
the function \[x\to (\log cx)/(\log x/R)\]
attains its maximum on
$\lbrack R',\infty\rbrack$, $R'>R$, at $x=R'$. Hence, for
$q \delta_0$ fixed,
\begin{equation}\label{eq:peergynt}
\min\left(\frac{4/5}{\log^+ 
\frac{4 x^{2/3}}{9 (\delta_0 q)^{\frac{5}{2}}}}, 1\right) \left(\log 9 x^{\frac{1}{3}} 
\sqrt{q \delta_0} + c_{3,I}\right)
\end{equation}
attains its maximum for $x \in \lbrack
 (9 e^{4/5} (\delta_0 q)^{5/2}/4)^{3/2},
\infty)$ at
\begin{equation}\label{eq:koloso}
x = \left(9 e^{4/5} (\delta_0 q)^{5/2}/4\right)^{3/2}
= (27/8) e^{6/5} (q \delta_0)^{15/4}.\end{equation}
Now, notice that, for smaller values of $x$, (\ref{eq:peergynt}) increases
as $x$ increases, since the term $\min(\dotsc,1)$ equals the constant $1$. Hence,
(\ref{eq:peergynt}) attains its maximum for $x\in (0,\infty)$ at
(\ref{eq:koloso}),
and so 
\[\begin{aligned}
&\min\left(\frac{4/5}{\log^+ 
\frac{4 x^{2/3}}{9 (\delta_0 q)^{\frac{5}{2}}}}, 1\right) \left(\log 9 x^{\frac{1}{3}} 
\sqrt{q \delta_0} + c_{3,I}\right) + c_{4,I}
\\ &\leq \log \frac{27}{2} e^{2/5} (\delta_0 q)^{7/4} + c_{3,I} + c_{4,I} \leq 
\frac{7}{4} \log \delta_0 q + 6.11676 .
\end{aligned}\]
Examining the other terms in (\ref{eq:dikaiopolis})
and using (\ref{eq:voil}), we conclude that
\begin{equation}\label{eq:therwald}\begin{aligned}
|S_{I,1}| &\leq \frac{x}{q} \min\left(1, \frac{c_0'}{\delta^2}\right)
\cdot \frac{q}{\phi(q)} \left(\frac{7}{4} \log \delta_0 q + 6.11676
\right)
\\ &+
\frac{x^{2/3}}{\sqrt{q \delta_0}} (0.67845 \log x - 1.20818) + 
0.37864 x^{2/3},\end{aligned}\end{equation}
where we are using (\ref{eq:voil}) (and, of course, the trivial bound
$\delta_0 q\geq 2$) to simplify the smaller error terms.
We recall that $c_0' = 0.798437 > c_0/(2\pi)^2$.

Let us now consider $S_{I,2}$. The terms that appear both for $|\delta|$ small
and $|\delta|$ large are given in (\ref{eq:chusan}). The second line
in (\ref{eq:chusan}) equals 
\[\begin{aligned}
& c_{8,I_2} \left(\frac{x}{4 q^2 \delta_0} + \frac{2 U V^2}{x} + 
\frac{q V^2}{x}\right) +
\frac{c_{10,I}}{2} \left( \frac{q}{2 \sqrt{q \delta_0}} +
\frac{x^{2/3}}{18 q \delta_0}\right) \log \frac{9 x^{1/3}}{2}\\
&\leq c_{8,I_2} \left(\frac{x}{4 q^2 \delta_0} + 
\frac{9 x^{1/3}}{2 \sqrt{2}} + \frac{27}{8}\right) +
 \frac{c_{10,I}}{2} \left( \frac{y^{1/6}}{2^{3/2}} +
 \frac{x^{2/3}}{18 q \delta_0}\right) \left(\frac{1}{3}\log x + \log \frac{9}{2}\right)\\
&\leq 0.29315 \frac{x}{q^2 \delta_0} + (0.08679 \log x + 0.39161) \frac{x^{2/3}}{q \delta_0} +
 0.00153 \sqrt{x},
\end{aligned}\]
where we are using (\ref{eq:voil}) to simplify. Now
\begin{equation}\label{eq:gorachy}
\min\left(\frac{4/5}{\log^+ \frac{Q}{4 V q^2}},1\right) \log V q =
\min\left(\frac{4/5}{\log^+ \frac{y}{4 q^2}},1\right) \log 
\frac{9 x^{1/3} q}{2} 
\end{equation} can be bounded
trivially by $\log(9 x^{1/3} q/2)\leq 
(2/3) \log x + \log 3/4$. We can also bound (\ref{eq:gorachy})
as we bounded (\ref{eq:peergynt}) before, namely, by fixing $q$
and finding the maximum for $x$ variable. In this way, we obtain that
(\ref{eq:gorachy}) is maximal for $y = 4 e^{4/5} q^2$; since, by
definition, $x^{1/3}/6=y$,
(\ref{eq:gorachy}) then equals
\[\log \frac{9 (6\cdot 4 e^{4/5} q^2) q}{2} = 3 \log q + \log 108 + \frac{4}{5}
\leq 3 \log q + 5.48214.\]

We conclude that (\ref{eq:chusan}) is at most
\begin{equation}\label{eq:clums}
\begin{aligned}&\min\left(1, \frac{4 c_0'}{\delta^2}\right)\cdot 
\left(\frac{3}{2} \log q + 2.74107\right) \frac{x}{\phi(q)} \\ &+
0.29315 \frac{x}{q^2 \delta_0} + 
(0.0434 \log x + 0.1959) x^{2/3}.
\end{aligned}\end{equation}

If $|\delta|\leq 1/2c_2$, we must consider (\ref{eq:fausto}). This is at most
\[\begin{aligned}
&(c_{4,I_2}+c_{9,I_2}) \frac{x}{2 \sqrt{q \delta_0}} +
(c_{10,I_2} \log \frac{x^{2/3}}{9 q^{3/2} \sqrt{\delta_0}} + 2 c_{5,I_2} + c_{12,I_2}) \cdot \frac{3}{4}
x^{2/3}\\
&\leq \frac{2.1989 x}{\sqrt{q \delta_0}} + 
\frac{3.61818 x}{q \delta_0} + 
(1.77019 \log x + 29.2955) x^{2/3},
\end{aligned}\]
where we recall that $\epsilon_0 = (4 \log 2)/(x/UV) = (2 \log 2)/\sqrt{q \delta_0}$, which can be bounded crudely by $\sqrt{2} \log 2$. (Thus,
$c_{10,I_2} \leq \sqrt{1 + \sqrt{8} \log 2}\cdot 1.78783 < 3.54037$ and
$c_{12,I_2} \leq 29.3333 + 11.902 \sqrt{2} \log 2 \leq 41.0004$.)

If $|\delta|>1/2c_2$, we must consider (\ref{eq:magus}) instead. For
$\epsilon = 0.07$, that is at most
\[\begin{aligned}
&(c_{4,I_2}+(1+\epsilon) c_{13,I_2}) \frac{x}{2 \sqrt{q \delta_0}} 
\left(1 + \frac{2 \log 2}{\sqrt{q \delta_0}}\right) \\&+
(3.38845 \left(1 + \frac{2 \log 2}{\sqrt{q \delta_0}}\right)
 \log \delta q^3 + 20.8823) \frac{x}{|\delta| q}\\
&+ \left(68.8133 \left(1 + \frac{4 \log 2}{\sqrt{q \delta_0}}\right) 
\log |\delta| q + 72.0828\right) x^{2/3} + 60.4141 x^{1/3}\\
&= 2.49157 \frac{x}{\sqrt{q \delta_0}} 
\left(1 + \frac{2 \log 2}{\sqrt{q \delta_0}}\right) +
(3.38845 \log \delta q^3 + 32.6771) \frac{x}{|\delta| q}\\ &+
\left(22.9378 \log x + 190.791 \frac{\log |\delta| q}{\sqrt{q \delta_0}} +
130.691\right) x^{\frac{2}{3}}\\
&\leq
2.49157 \frac{x}{\sqrt{q \delta_0}} 
+ (3.59676 \log \delta_0 + 27.3032 \log q + 91.2218) \frac{x}{q \delta_0}
\\ &+ (22.9378 \log x + 411.228) x^{\frac{2}{3}},
\end{aligned}\]
where, besides the crude bound $\epsilon_0\leq \sqrt{2} \log 2$, 
we use the inequalities
\[\begin{aligned} 
\frac{\log |\delta| q}{\sqrt{q \delta_0}} \leq \frac{
\log 4 q \delta_0}{\sqrt{q 
\delta_0}} &\leq \frac{\log 8}{\sqrt{2}},\;\;\;\;\;\;\;
\frac{\log q}{\sqrt{q \delta_0}} \leq \frac{1}{\sqrt{2}} \frac{\log q}{\sqrt{q}}
\leq \frac{1}{\sqrt{2}} \frac{\log e^2}{e} = \frac{\sqrt{2}}{e},
\\
\frac{1}{|\delta|}&\leq \frac{4 c_2}{\delta_0},\;\;\;\;\;\;\;
\frac{\log |\delta|}{|\delta|}\leq \frac{2}{e \log 2}\cdot \frac{\log \delta_0}{\delta_0}.\end{aligned}\]
(Obviously, $1/|\delta|\leq 4 c_2/\delta_0$ is based on the assumption
$|\delta| > 1/2 c_2$ and on the inequality $16 c_2 \geq 1$. The bound on
$(\log |\delta|)/|\delta|$ is based on the fact that $(\log t)/t$ reaches
its maximum at $t=e$, and $(\log \delta_0)/\delta_0 = (\log 2)/2$ for
$|\delta|\leq 8$.)

We sum (\ref{eq:clums}) and whichever one of our bounds for
(\ref{eq:fausto}) and (\ref{eq:magus}) is greater (namely, the latter).
We obtain that, for any $\delta$,
\begin{equation}\label{eq:cleson}\begin{aligned}
&|S_{I,2}|\leq 2.49157 \frac{x}{\sqrt{q \delta_0}} +
\min\left(1, \frac{4 c_0'}{\delta^2}\right)\cdot 
\left(\frac{3}{2} \log q + 2.74107\right) \frac{x}{\phi(q)} 
\\ &
+ (3.59676 \log \delta_0 + 27.3032 \log q + 91.515) \frac{x}{q \delta_0} +
(22.9812 \log x + 411.424) x^{2/3},\end{aligned}\end{equation}
where we bound one of the lower-order terms in (\ref{eq:clums}) by
$x/q^2 \delta_0 \leq x/q \delta_0$.

 For type II, we have to consider two cases: (a) 
$|\delta|<8$, and (b)  $|\delta|\geq 8$. Consider first $|\delta|<8$. 
Then $\delta_0 = 2$.
Recall that $\theta =27/8$.
We have $q \leq V/2\theta$ and $|\delta q| \leq V/\theta$ thanks to (\ref{eq:werto}).
We apply (\ref{eq:vinland1}), and obtain that, for $|\delta|<8$,
\begin{equation}\label{eq:pell}\begin{aligned}
|S_{II}|&\leq
\frac{x}{\sqrt{2 \phi(q)}} \cdot 
\sqrt{\frac{1}{2} \log 4q\delta_0 + \log 2q \log \left(1+ 
\frac{\frac{1}{2} \log 4q\delta_0}{\log \frac{V}{2q}}\right)}\\ &\cdot 
\sqrt{0.30214 \log 4 q \delta_0 + 0.2562}\\
&+ 8.22088 \sqrt{\frac{q}{\phi(q)}} \left(1 + 1.15 
\sqrt{\frac{\log 2q}{\log \frac{9 x^{1/3} \sqrt{\delta_0}}{2 \sqrt{q}}}}\right) (q \delta_0)^{1/4} x^{2/3}  +  1.84251 x^{5/6}\\
&\leq \frac{x}{\sqrt{2 \phi(q)}} \cdot 
\sqrt{C_{x,2 q} \log 2q + \frac{\log 8 q}{2}} \cdot 
\sqrt{0.30214 \log 2q  + 0.67506}\\
&+ 16.406 \sqrt{\frac{q}{\phi(q)}} x^{3/4} +  1.84251 x^{5/6}
\end{aligned}\end{equation}
where we bound
\[\frac{\log 2q}{\log \frac{9 x^{1/3} \sqrt{\delta_0}}{2 \sqrt{q}}}
\leq \frac{\log \frac{x^{1/3}}{3}}{\log \frac{9 x^{1/6} \sqrt{2}}{2 \sqrt{1/6}}}
< \lim_{x\to \infty}
\frac{\log \frac{x^{1/3}}{3}}{\log \frac{9 x^{1/6} \sqrt{2}}{2 \sqrt{1/6}}}
= 2,
\]
and where we define
\[C_{x,t} := 
\log \left(1 + \frac{\log 4 t}{2 \log \frac{9 x^{1/3}}{2.004 t}}\right)\]
for $0 < t < 9 x^{1/3}/2$. (We have $2.004$ here instead of $2$ because
we want a constant $\geq 2 (1+\epsilon_1)$ in later occurences of
$C_{x,t}$, for reasons that will soon become clear.)

For purposes of later comparison, we remark
that $16.404\leq 1.57863 x^{4/5-3/4}$ for $x\geq 2.16\cdot 10^{20}$.

Consider now case (b), namely, $|\delta|\geq 8$. Then $\delta_0  = |\delta|/4$.
By (\ref{eq:werto}), $|\delta q|\leq V/\theta$. Hence,
 (\ref{eq:eriksaga}) gives us that
\begin{equation}\label{eq:meli}\begin{aligned}
|S_{II}|&\leq
\frac{2 x}{\sqrt{|\delta| \phi(q)}} \cdot 
\sqrt{\frac{1}{2} \log |\delta q|  + \log \frac{|\delta q| (1+\epsilon_1)}{4} 
\log \left(1+ \frac{\log |\delta| q}{2 \log \frac{18 x^{1/3}}{|\delta|
(1+ \epsilon_1) q}}
\right)}\\
&\cdot \sqrt{0.30214 \log |\delta| q + 0.2562} \\ &+ 8.22088
\sqrt{\frac{q}{\phi(q)}}\sqrt{\frac{\log \frac{9 x^{1/3}}{2}}{\log
\frac{9 x^{1/3}}{|\delta q|}}}\cdot (q \delta_0)^{1/4} x^{2/3}  
 +  1.84251 x^{5/6}\\
&\leq \frac{x}{\sqrt{\delta_0 \phi(q)}} 
\sqrt{C_{x, \delta_0 q} \log \delta_0 (1+\epsilon_1) q + \frac{\log 4 \delta_0 q}{2}}
\sqrt{0.30214 \log \delta_0 q + 0.67506} \\ 
&+ 1.79926 \sqrt{\frac{q}{\phi(q)}} x^{4/5} + 
1.84251 x^{5/6} ,
\end{aligned}\end{equation}
since
\[\begin{aligned}8.22088 \sqrt{\frac{\log \frac{9 x^{1/3}}{2}}{\log
\frac{9 x^{1/3}}{|\delta q|}}}\cdot (q \delta_0)^{1/4} &\leq
8.22088 \sqrt{\frac{\log \frac{9 x^{1/3}}{2}}{\log \frac{27}{4}}}
\cdot (x^{1/3}/3)^{1/4} \\ &\leq 1.79926 x^{4/5-2/3}\end{aligned}\]
for $x\geq 2.16\cdot 10^{20}$. Clearly
\[\log \delta_0 (1+\epsilon_1) q = \log \delta_0 q + \log(1+\epsilon_1)
\leq \log \delta_0 q + \epsilon_1.\]

By Lemma \ref{lem:merkel},
$q/\phi(q)\leq \digamma(y) = \digamma(x^{1/3}/6)$ (since $x\geq 18^3$).
It is easy to check that $x\to \sqrt{\digamma(x^{1/3}/6)} x^{4/5-5/6}$
is decreasing for $x\geq 2.16\cdot 10^{20}$ (in fact, for $18^3$). Using (\ref{eq:voil}), we conclude that
 $1.67718 \sqrt{q/\phi(q)} x^{4/5} \leq 0.89657 x^{5/6}$ and, by the way,
$16.406 \sqrt{q/\phi(q)} x^{3/4} \leq 0.78663 x^{5/6}$.
This allows us to simplify the last lines of (\ref{eq:pell}) and 
(\ref{eq:meli}). We obtain that, for $\delta$ arbitrary,
\begin{equation}\label{eq:senorburns}
\begin{aligned}
|S_{II}| &\leq 
\frac{x}{\sqrt{\delta_0 \phi(q)}} 
\sqrt{C_{x, \delta_0 q} (\log \delta_0 q + \epsilon_1) + \frac{\log 4 \delta_0 q}{2}}
\sqrt{0.30214 \log \delta_0 q + 0.67506} \\ 
&+ 2.73908 x^{5/6}.
\end{aligned}\end{equation}

It is time to sum up $S_{I,1}$, $S_{I,2}$ and $S_{II}$. The main terms
come from the first line of (\ref{eq:senorburns}) and
the first term of (\ref{eq:cleson}). Lesser-order terms can be dealt with
roughly: we bound $\min(1,c_0'/\delta^2)$ and $\min(1,4 c_0'/\delta^2)$
from above by $2/\delta_0$ (using the fact that
$c_0' = 0.798437 < 16$, which implies that
 $8/\delta > 4 c_0'/\delta^2$ for $\delta>8$; of course, for
 $\delta\leq 8$, we have $\min(1,4 c_0'/\delta^2)\leq 1 = 2/2 = 2/\delta_0$).

The terms inversely proportional to $q$, $\phi(q)$ or $q^2$ thus add up to 
at most
\[\begin{aligned}
&\frac{2 x}{\delta_0 q} \cdot \frac{q}{\phi(q)} \left(\frac{7}{4} \log \delta_0 q 
 + 6.11676\right) + 
\frac{2 x}{\delta_0 \phi(q)} \left(\frac{3}{2} \log q + 2.74107\right)\\
&+ 
 (3.59676 \log \delta_0 + 27.3032 \log q + 91.515) \frac{x}{q \delta_0}\\
&\leq \frac{2 x}{\delta_0 \phi(q)} \left(\frac{13}{4} \log \delta_0 q + 
7.81811\right) + \frac{2 x}{\delta_0 q} (13.6516 \log \delta_0 q 
+ 37.5415),
\end{aligned}\]
where, for instance, we bound $(3/2) \log q + 2.74107$ by
$(3/2) \log \delta_0 q + 2.74107 - (3/2) \log 2$.

As for the other terms -- we use the assumption $x\geq 2.16 \cdot 10^{20}$
 to bound $x^{2/3}$ and
$x^{2/3} \log x$ by a small constant times $x^{5/6}$. We bound $x^{2/3}/\sqrt{
q\delta_0}$ by $x^{2/3}/\sqrt{2}$ (in (\ref{eq:therwald})).
We obtain
\[\begin{aligned}
\frac{x^{2/3}}{\sqrt{2}} &(0.67845 \log x - 1.20818) + 0.37864 x^{\frac{2}{3}}
\\ &+ (22.9812 \log x + 411.424) x^{\frac{2}{3}} + 2.73908 x^{\frac{5}{6}}
\;\;\leq 3.35531 x^{5/6}.\end{aligned}\]

The sums $S_{0,\infty}$ and $S_{0,w}$
in (\ref{eq:sofot}) are $0$ (by (\ref{eq:voil}) and the fact that $\eta_2(t)=0$
for $t\leq 1/4$). 
We conclude that, for $q\leq y = x^{1/3}/6$, $x\geq 2.16\cdot 10^{20}$
and $\eta=\eta_2$ as in (\ref{eq:eqeta}),
\begin{equation}\label{eq:duaro}\begin{aligned}
&|S_\eta(x,\alpha)| \leq |S_{I,1}| + |S_{I,2}| + |S_{II}|\\
 &\leq \frac{x}{\sqrt{\delta_0 \phi(q)}} 
\sqrt{C_{x, \delta_0 q} (\log \delta_0 q + 0.002)+ 
\frac{\log 4 \delta_0 q}{2}}
\sqrt{0.30214 \log \delta_0 q + 0.67506} 
\\ &+ \frac{2.49157 x}{\sqrt{\delta_0 q}}
+ \frac{2 x}{\delta_0 \phi(q)} \left(\frac{13}{4} \log \delta_0 q + 
7.81811\right) + \frac{2 x}{\delta_0 q} (13.6516 \log \delta_0 q 
+ 37.5415)
\\ &+ 3.35531 x^{5/6},
\end{aligned}\end{equation}
where
\begin{equation}\label{eq:raisin}\begin{aligned}
\delta_0 &= \max(2,|\delta|/4),\;\;\;\;\;\;
C_{x,t} =
\log \left(1 + \frac{\log 4 t}{2 \log \frac{9 x^{1/3}}{2.004 t}}\right).
\end{aligned}\end{equation}
Since $C_{x,t}$ is an increasing function as a function of
$t$ (for $x$ fixed and $t\leq 9 x^{1/3}/2.004$) 
and $\delta_0 q\leq 2y$, we see that
$C_{x,t}\leq C_{x,2y}$. It is clear that
$x\mapsto C_{x,t}$ (fixed $t$) is a decreasing function of $x$. 
For $x = 2.16 \cdot 10^{20}$, $C_{x,2y} = 1.39942\dotsc$.

\subsection{Second choice of parameters}\label{subs:espan}
If, with the original choice of parameters, we obtained $q>y= x^{1/3}/6$, 
we now reset our parameters ($Q$, $U$ and $V$).
Recall that, while the value of $q$ may now change (due to the change in 
$Q$),
we will be able to assume that either $q>y$ or $|\delta q|>x/(x/8y)
= 8y$.

We want $U/(x/UV)\geq 5\cdot 10^5$ (this is (\ref{eq:curious})).
We also want $UV$ small. With this in mind, we let
\begin{equation}\label{eq:elpozer}
V = \frac{x^{1/3}}{3},\;\;\;\;\;\;\; U = 500 \sqrt{6} x^{1/3},\;\;\;\;\;\;\;\;
Q = \frac{x}{U} = \frac{x^{2/3}}{500 \sqrt{6}}.
\end{equation}
Then
(\ref{eq:curious}) holds (as an equality). 
Since we are assuming (\ref{eq:voil}), we have
$V\geq 2\cdot 10^6$.
It is easy to check that (\ref{eq:voil}) also
implies that $U\leq \sqrt{x}/2$ and $Q\geq 2 \sqrt{x}$, and so
the inequalities in (\ref{eq:herring}) all hold.

Write $2\alpha = a/q + \delta/x$ for the new approximation; we must have 
either $q>y$ or $|\delta| > 8 y/q$, since otherwise $a/q$ would already be
a valid approximation under the first choice of parameters. 
Thus, either (a) $q>y$, or both (b1) $|\delta|>8$ and (b2) $|\delta| q > 
8y$. Since now $V=2y$, we have $q>V/2\theta$ in case (a) and 
$|\delta q|> V/\theta$ in case (b) 
for any $\theta\geq 1$. We set $\theta = 4$.

(Thanks to this choice of $\theta$, we have
$|\delta q| \leq x/Q \leq x/\theta U$, as we commented at the end of 
\S \ref{subs:absur}; this will help us avoid some case-work later.)


By (\ref{eq:lavapie}),
\[\begin{aligned}
&|S_{I,1}|\leq \frac{x}{q} \min\left(1,\frac{c_0'}{\delta^2}\right)
\left(\log x^{2/3} - \log 500 \sqrt{6} + c_{3,I} + c_{4,I} \frac{q}{\phi(q)}
\right)\\
&+ \left(c_{7,I} \log \frac{Q}{c_2} + c_{8,I} \log x \log c_{11,I} \frac{Q^2}{x}
\right) Q + c_{10,I} \frac{U^2}{4 x} \log \frac{e^{1/2} x^{2/3}}{500 \sqrt{6}}
 + \frac{c_{10,I}}{e}\\
&+ \left(c_{5,I} \log \frac{1000 \sqrt{6} x^{1/3}}{c_2} + 
c_{6,I} \log 500 \sqrt{6} x^{4/3}\right) \cdot 500 \sqrt{6} x^{1/3} +
c_{9,I} \sqrt{x} \log \frac{2 \sqrt{e} x}{c_2}\\
&\leq \frac{x}{q} \min\left(1,\frac{c_0'}{\delta^2}\right)
\left(\frac{2}{3} \log x - 4.99944 + 1.00303 \frac{q}{\phi(q)}\right) + 
\frac{2.89}{1000} x^{2/3} (\log x)^2,
\end{aligned}\]
where we are bounding 
\[\begin{aligned}
&c_{7,I} \log \frac{Q}{c_2} + c_{8,I} \log x \log c_{11,I} \frac{Q^2}{x}\\
= &c_{8,I} (\log x)^2 - \left(c_{8,I} (\log 1500000
- \log c_{11,I}) - \frac{2}{3} c_{7,I}\right) \log x + c_{7,I} 
\log \frac{1}{500 \sqrt{6} c_2}
\\ 
\leq &c_{8,I} (\log x)^2 - 38 \log x.
\end{aligned}\]
We are also using the assumption (\ref{eq:voil}) repeatedly in order
to show that the sum of all lower-order terms is less than
$(38 c_{8,I} \log x)/(500 \sqrt{6})$. Note that
$c_{8,I} (\log x)^2 Q \leq 0.00289 x^{2/3} (\log x)^2$.

We have $q/\phi(q)\leq \digamma(Q)$ (where $\digamma$ is as in
 (\ref{eq:locos})) and, since $Q> \sqrt{6}\cdot 12 \cdot 10^9$ for
$x\geq 2.16\cdot 10^{20}$,
\[\begin{aligned}1.00303 \digamma(Q)&\leq 
1.00303 \left(e^\gamma \log \log Q + \frac{2.50637}{\log \log 
\sqrt{6}\cdot 12 \cdot 10^9}\right)
\\ &\leq 0.2359 \log Q + 0.79 < 0.1573 \log x.\end{aligned}\]
(It is possible to give a much better estimation, but it is not worthwhile,
since this will be a very minor term.)
We have either $q>y$ or $q |\delta| > 8 y$; if $q |\delta|> 8 y$ but 
$q\leq y$, then $|\delta|\geq 8$, and so $c_0'/\delta^2 q < 1/8 |\delta| q
< 1/64 y < 1/y$. Hence 
\[\begin{aligned}
|S_{I,1}|&\leq \frac{x}{y} \left(\left(\frac{2}{3} + 0.1573\right) \log x\right)
+ 0.00289 x^{2/3} (\log x)^2
\\ &\leq 2.4719 x^{2/3} \log x + 0.00289 x^{2/3} (\log x)^2.
\end{aligned}\]

We bound $|S_{I,2}|$ using Lemma \ref{lem:bogus}. First we bound
(\ref{eq:cupcake3}): this is at most
\[\begin{aligned}
&\frac{x}{2 q} \min\left(1,\frac{4 c_0'}{\delta^2}\right)
\log \frac{x^{1/3} q}{3}\\&+
c_0 \left(\frac{1}{4}-\frac{1}{\pi^2}\right) 
\left(\frac{(UV)^2 \log \frac{x^{\frac{1}{3}}}{3}}{2 x} + \frac{3 c_4}{2}
 \frac{500 \sqrt{6}}{9} + \frac{(500 \sqrt{6} x^{1/3} + 1)^2 x^{\frac{1}{3}} \log x^{\frac{2}{3}}}{6 x}
\right),
\end{aligned}\]
where $c_4 = 1.03884$.
We bound the second line of this using (\ref{eq:voil}). As for the first
line, we have either $q\geq y$ (and so the first line is at most
$(x/2y) (\log x^{1/3} y/3)$) or $q<y$ and 
$4 c_0'/\delta^2 q < 1/16 y < 1/y$
(and so the same bound applies). Hence (\ref{eq:cupcake3}) is at most
\[3 x^{2/3} \left(\frac{2}{3} \log x - \log 18\right) + 
0.02017 x^{2/3} \log x = 2.02017 x^{2/3} \log x - 3 (\log 18) x^{2/3}.\]

Now we bound (\ref{eq:piececake}), which comes up when $|\delta|\leq 1/2 c_2$,
where $c_2 = 6\pi/5\sqrt{c_0}$, $c_0 =31.521$ (and so $c_2 = 0.6714769\dotsc$).
Since $1/2c_2 <8$, it follows that $q>y$ (the alternative $q\leq y$,
$q |\delta|>8y$ is impossible, since it implies $|\delta|>8$). Then (\ref{eq:piececake}) is at most
\begin{equation}\label{eq:octet}\begin{aligned}
&\frac{2 \sqrt{c_0 c_1}}{\pi} \left(UV \log \frac{UV}{\sqrt{e}} + 
Q \left(\sqrt{3} \log \frac{c_2 x}{Q} + \frac{\log UV}{2} \log \frac{UV}{Q/2}
\right)\right)\\
&+ \frac{3 c_1}{2} \frac{x}{y} \log UV \log \frac{UV}{c_2 x/y} + 
\frac{16 \log 2}{\pi} Q \log \frac{c_0 e^3 Q^2}{4\pi \cdot 8 \log 2 \cdot x}
\log \frac{Q}{2}\\&+ \frac{3 c_1}{2 \sqrt{2 c_2}} \sqrt{x} \log \frac{c_2 x}{2}
+ \frac{25 c_0}{4\pi^2} (3 c_2)^{1/2} \sqrt{x} \log x,\end{aligned}\end{equation}
where $c_1 = 1.000189 > 1+ (8 \log 2)/(2x/UV)$. 

The first line of (\ref{eq:octet}) is a linear combination of terms of the
form $x^{2/3} \log C x$, $C>1$; using (\ref{eq:voil}), we obtain that it is
at most $1144.693 x^{2/3} \log x$. (The main contribution comes from the first 
term.) Similarly, we can bound the first term in the second line by
$33.0536 x^{2/3} \log x$. 
Since $\log(c_0 e^3 Q^2/(4\pi \cdot 8\log 2 \cdot x)) \log Q/2$ is at most
$\log x^{1/3} \log x^{2/3}$,
the second term in the second line is at most $0.0006406 x (\log x)^2$.
The third line of (\ref{eq:octet}) can be bounded easily by
$0.0122 x^{2/3} \log x$.

Hence,
(\ref{eq:octet}) is at most
\[
1177.76 x^{2/3} \log x + 0.0006406 x^{2/3} (\log x)^2.\]

If $|\delta|> 1/2 c_2$, then we know that $|\delta q| > \min(y/2 c_2, 8 y) = 
y/2 c_2$. Thus (\ref{eq:tvorog}) (with $\epsilon=0.01$) is at most
\[\begin{aligned}
&\frac{2 \sqrt{c_0 c_1}}{\pi} UV \log \frac{UV}{\sqrt{e}} \\ &+ 
\frac{2.02 \sqrt{c_0 c_1}}{\pi} \left(\frac{x}{y/2c_2}+1\right)
\left( (\sqrt{3.02}-1) \log \frac{\frac{x}{y/2c_2} + 1}{\sqrt{2}}
+ \frac{1}{2} \log UV \log \frac{e^2 UV}{\frac{x}{y/2c_2}}\right)\\
&+ \left(\frac{3 c_1}{2} \left(\frac{1}{2} + \frac{3.03}{0.16} \log x\right)
+ \frac{20 c_0}{3 \pi^2} (2 c_2)^{3/2}\right) \sqrt{x} \log x .\end{aligned}\]
Again by (\ref{eq:voil}), and in much the same way as before,
 this simplifies to
\[\begin{aligned}
&\leq (1144.66 + 15.107 + 68.523)  x^{2/3} \log x + 29.136 x^{1/2} (\log x)^2 \\
&\leq
1228.85 x^{2/3} (\log x).\end{aligned}\]

Hence, in total and for any $|\delta|$,
\[\begin{aligned}
|S_{I,2}| &\leq
2.02017 x^{2/3} \log x +
 1228.85 x^{2/3} (\log x) + 0.0006406 x^{2/3} (\log x)^2\\
&\leq 1230.9 x^{2/3} (\log x) + 0.0006406 x^{2/3} (\log x)^2.\end{aligned}\]

Now we must estimate $S_{II}$. As we said before, either (a) $q>y$, or both
(b1) $|\delta|>8$ and (b2) $|\delta| q>8 y$. Recall that $\theta = 4$.
In case (a), we have $q > x^{1/3}/6 = V/2 > V/2\theta$; thus, we can use 
(\ref{eq:vinland2}), and obtain that, if $q\leq x/8 U$,
$|S_{II}|$ is at most
\begin{equation}\label{eq:hust}\begin{aligned}
&\frac{x \sqrt{\digamma(q)}}{\sqrt{2 q}}
\sqrt{\left(\log \frac{x}{U \cdot 8 q} + \log 2q 
\log \frac{\log x/(2 U q)}{\log 4}\right) \left(\kappa_6 \log \frac{x}{
U\cdot 8 q} + 2 \kappa_7\right)}\\
&+ \sqrt{2} \kappa_2 \sqrt{\digamma\left(\frac{x}{8 U}\right)}
\left(1 + 1.15 \sqrt{\frac{\log x/4 U}{\log 4}}\right) \frac{x}{\sqrt{U}} + 
(\kappa_2 \sqrt{\log x/U} + \kappa_9) \frac{x}{\sqrt{V}}\\
&+ \frac{\kappa_2}{6} \left((\log 8 y)^{3/2} - (\log 2 y)^{3/2}\right) \frac{x}{\sqrt{y}} \\ &+ \kappa_2 \left(\sqrt{8 \log x/U} + 
\frac{2}{3} ((\log x/U)^{3/2} - (\log V)^{3/2})\right) \frac{x}{\sqrt{8 U}}
,\end{aligned}\end{equation}
where $\digamma$ is as in (\ref{eq:locos}). (We are already simplifying
the third line; the bound given is justified by a derivative test.)
It is easy to check that $q\to (\log 2q) (\log \log q)/q$ is decreasing for
$q\geq y$ (indeed for $q\geq 9$), and so the first line of
(\ref{eq:hust}) is maximal for $q=y$. 

We can thus bound (\ref{eq:hust}) by $x^{5/6}$ times
\begin{equation}\label{eq:quan}\begin{aligned}
&\sqrt{3 \digamma(e^{t/3}/6)
\left(\frac{t}{3} - \log 8 c +
\left(\frac{t}{3} - \log 3\right) \log \frac{\frac{t}{3} - \log
2 c}{\log 4}\right) 
\left(\frac{\kappa_6}{3} t - 4.214\right)}\\
&+ \frac{\sqrt{2} \kappa_2}{\sqrt{6 c}}
\sqrt{\digamma\left(\frac{e^{2t/3}}{48 c}\right)}
\left(1 + 1.15 \sqrt{\frac{\frac{2}{3} t - \log 24 c}{\log 4}}\right)
\\ &+ \left(\kappa_2 \sqrt{\frac{2 t}{3} - \log 6 c} + \kappa_9\right)
 \sqrt{3}\\
&+ \frac{\kappa_2}{\sqrt{6}} 
\left(\left(\frac{t}{3} + \log \frac{8}{6}\right)^{
\frac{3}{2}} - \left(\frac{t}{3} + \log \frac{2}{6}\right)^{\frac{3}{2}}\right) 
\\
&+ \frac{\kappa_2}{\sqrt{48 c}} \left(\sqrt{8\left (\frac{2t}{3} - \log
6 c\right)
} + \frac{2}{3} \left(\left(\frac{2t}{3} - \log
6 c\right)^{\frac{3}{2}} - 
\left(\frac{t}{3} - \log 3\right)^{\frac{3}{2}}\right)
\right) 
\end{aligned}\end{equation}
where $t = \log x$ and $c = 500/\sqrt{6}$. 
Asymptotically, the largest term
in (\ref{eq:hust}) comes from the last line (of order $t^{3/2}$),
even if the first line is larger in practice (while being of
order at most $t \log t$). Let us bound (\ref{eq:quan}) by a multiple of
$t^{3/2}$.

First of all, notice that
\begin{equation}\label{eq:ceneres}
\begin{aligned}\frac{d}{dt} \frac{\digamma\left(\frac{e^{t/3}}{6}\right)}{
\log t}
 &= \frac{\left(e^\gamma 
\log \left(\frac{t}{3} - \log 6\right) 
+ \frac{2.50637}{\log \left(\frac{t}{3} - \log 6\right)}
\right)'}{\log t}
-  \frac{\digamma\left(\frac{e^{t/3}}{6}\right)}{t (\log t)^2}\\
&=
\frac{e^\gamma - \frac{2.50637}{\log^2 \left(\frac{t}{3} - \log 6\right)}
}{(t  - 3 \log 6) \log t}
- \frac{e^\gamma 
+ \frac{2.50637}{\log^2 \left(\frac{t}{3} - \log 6\right)}}{t \log t}\cdot 
\frac{\log \left(\frac{t}{3} - \log 6\right)}{\log t},
\end{aligned}\end{equation}
which, for $t\geq 100$, is 
\[> \frac{e^\gamma \log 3 - \frac{2\cdot 2.50637 \log t}{\log^2
\left(\frac{t}{3} - \log 6\right)}}{t (\log t)^2} \geq
\frac{1.95671 - \frac{8.92482}{\log t}}{t (\log t)^2}>0.\]
Similarly, for $t\geq 2000$,
\[\begin{aligned}
\frac{d}{dt} \frac{\digamma\left(\frac{e^{2t/3}}{48 c}\right)}{
\log t}
&> \frac{e^\gamma \log \frac{3}{2} - \frac{2.50637 \log t}{\log^2
\left(\frac{2 t}{3} - \log 48 c\right)} - 
\frac{2.50637}{\log \left(\frac{2 t}{3} - \log 48 c\right)} 
}{t (\log t)^2} \\ &\geq
\frac{0.72216 - \frac{5.45234}{\log t}}{t (\log t)^2} 
> 0.\end{aligned}\] 
Thus,
\begin{equation}\label{eq:jogaila}\begin{aligned}
\digamma\left(\frac{e^{t/3}}{6}\right) &\leq
(\log t) \cdot \lim_{s\to \infty} \frac{
\digamma\left(\frac{e^{s/3}}{6}\right)}{\log s} = e^{\gamma} \log t\;\;\;\;\;\;\;\;\;\;\;
\text{ for $t\geq 100$},\\
\digamma\left(\frac{e^{2t/3}}{48 c}\right) &\leq
(\log t) \cdot \lim_{s\to \infty} \frac{
\digamma\left(\frac{e^{2 s/3}}{48 c}\right)}{\log s} = e^{\gamma} \log t\;\;\;\;\;\;\;\;\;\;\;
\text{ for $t\geq 2000$}.
\end{aligned}\end{equation}

Also note that, since $(x^{3/2})' = (3/2) \sqrt{x}$,
\[
\left(\left(\frac{t}{3} + \log \frac{8}{6}\right)^{
\frac{3}{2}} - \left(\frac{t}{3} + \log \frac{2}{6}\right)^{\frac{3}{2}}\right) 
\leq \frac{3}{2} \sqrt{\frac{t}{3} + \log \frac{8}{6}} \cdot \log 4
\leq 1.20083 \sqrt{t}.
\]
for $t\geq 2000$. We also have
\[\begin{aligned}
&\left(\frac{2t}{3} - \log 6 c\right)^{\frac{3}{2}} - 
\left(\frac{t}{3} - \log 3\right)^{\frac{3}{2}}
 < \left(\frac{2t}{3} - \log 9\right)^{\frac{3}{2}} - 
\left(\frac{t}{3} - \log 3\right)^{\frac{3}{2}}\\ &=
 (2^{3/2} - 1)
\left(\frac{t}{3} - \log 3\right)^{\frac{3}{2}} <
 (2^{3/2} - 1) \frac{t^{3/2}}{3^{3/2}} \leq 0.35189 t^{3/2}.
\end{aligned}\]
Of course,
\[\frac{t}{3} - \log 8 c +
\left(\frac{t}{3} - \log 3\right) \log \frac{\frac{t}{3} - \log
2 c}{\log 4} <
\left(\frac{t}{3}+ \frac{t}{3} \log \frac{t}{3}\right) 
 < \frac{t}{3} \log t.\]
We conclude that, for $t\geq 2000$, (\ref{eq:quan}) is at most
\[\begin{aligned}
&\sqrt{3 \cdot e^\gamma \log t \cdot \frac{t}{3} \log t
\cdot \frac{\kappa_6}{3} t} + 
\frac{\sqrt{2} \kappa_2}{\sqrt{6 c}} \sqrt{e^\gamma \log t}
\left(1 + 0.79749 \sqrt{t}\right)\\
&+ \left(\kappa_2 \sqrt{\frac{2}{3}} t^{1/2} + \kappa_9\right) \sqrt{3}
+ \frac{\kappa_2}{\sqrt{6}} \cdot 1.2009 \sqrt{t}
+ \frac{\kappa_2}{\sqrt{48 c}} \left(\sqrt{\frac{16 t}{3}} + \frac{2}{3}
\cdot 0.35189 t^{3/2} \right)\\
&\leq (0.10181 + 0.00012 + 0.00145 + 0.000048 + 0.00462) t^{3/2}
\leq 0.10848 t^{3/2}.
\end{aligned}\]

On the remaining interval $\log(2.16\cdot 10^{20}) \leq t \leq \log 2000$, we use interval arithmetic (as in \S \ref{sec:koloko}, with $30$ iterations) 
to bound the ratio of
(\ref{eq:quan}) to $t^{3/2}$. We obtain that it is at most
\[0.275964 t^{3/2}.\]

Hence, for all $x\geq 2.16 \cdot 10^{20}$,
\begin{equation}\label{eq:hostoma}
|S_{II}| \leq 0.275964 x^{5/6} (\log x)^{3/2}.
\end{equation}
in the case $y < q \leq x/8 U$.


If $x/8 U < q \leq Q$, we use (\ref{eq:vinland3}). In this range,
$x/2 \sqrt{2q} +\sqrt{qx}$ adopts its maximum at $q=Q$ (because 
$x/2\sqrt{2q}$ for $q=x/8 U$ is smaller than $\sqrt{qx}$ for
$q=Q$, by (\ref{eq:elpozer}) and (\ref{eq:voil})). 
Hence, (\ref{eq:vinland3}) is at most $x^{5/6}$ times
\[\begin{aligned}
&\left(\kappa_2 \sqrt{2 \left(\frac{2}{3} t - \log c'\right)} + \kappa_9\right)
\sqrt{3} + \kappa_2 \sqrt{\frac{2}{3} t - \log c'} \cdot \frac{1}{\sqrt{c'}} \\
&+ \frac{2 \kappa_2}{3} \left(\left(\frac{2}{3} t - \log c'\right)^{\frac{3}{2}}
- \left(\frac{t}{3} - \log 3\right)^{\frac{3}{2}}\right)
\left(\frac{\sqrt{c'}}{2 \sqrt{2}} e^{-t/6} + 
\frac{1}{\sqrt{c'}}\right),\end{aligned}\]
where $t = \log x$ (as before) and $c' = 500 \sqrt{6}$.
This is at most
\[\begin{aligned} (2 \kappa_2 + \sqrt{3} \kappa_9) &\sqrt{t}
+ \frac{\kappa_2}{\sqrt{c'}} \sqrt{\frac{2}{3}} \sqrt{t} +
\frac{2 \kappa_2}{3} \frac{2^{3/2} - 1}{3^{3/2}} t^{\frac{3}{2}}
\left(\frac{\sqrt{c'}}{2 \sqrt{2}} e^{-t/6} + 
\frac{1}{\sqrt{c'}}\right)\\
&\leq 0.10327\end{aligned}\]
for $t\geq \log\left(2.16\cdot 10^{20}\right)$,
and so
\[|S_{II}|\leq 0.10327 x^{5/6} (\log x)^{3/2},\]
for $x/8U < q \leq Q$, using the assumption $x\geq 2.16\cdot 10^{20}$.

Finally, let us treat case (b), that is, $|\delta|>8$ and $|\delta| q>8y$;
we can also assume $q\leq y$, as otherwise we are in case (a),
which has already been treated. Since $|\delta/x|\leq 1/q Q$, we know that
\[|\delta q|\leq \frac{x}{Q} = U = 500 \sqrt{6} x^{1/3} \leq
\frac{x^{2/3}}{2000 \sqrt{6}} = \frac{x}{4 U} = \frac{x}{\theta U},\]
again under assumption (\ref{eq:voil}).
We apply (\ref{eq:vinlandsaga}), and obtain
that $|S_{II}|$ is at most
\begin{equation}\label{eq:bilal}\begin{aligned}
&\frac{2 x \sqrt{\digamma(y)}}{\sqrt{8 y}}
\sqrt{\left(\log \frac{x}{U\cdot 4 \cdot 8y} + \log 3 y 
\log \frac{\log x/3 U y}{\log 32/3}\right) \left(\kappa_6 
\log \frac{x}{U\cdot 4 \cdot 8 y} + 2 \kappa_7\right)}\\
&+ \frac{2\kappa_2}{3} \left(\frac{x}{\sqrt{16 y}} 
((\log 32 y)^{\frac{3}{2}} - (\log 2 y)^{\frac{3}{2}}) + 
\frac{x/4}{\sqrt{Q-y}}
((\log 4 U)^{\frac{3}{2}} - (\log 2 y)^{\frac{3}{2}} )\right) \\
&+ \left(\frac{\kappa_2}{\sqrt{2 (1-y/Q)}} \left(\sqrt{\log V} + 
\sqrt{1/\log V}\right) + \kappa_{9}\right) \frac{x}{\sqrt{V}} 
\\ &+
\kappa_2 \sqrt{\digamma(y)} \cdot \sqrt{\log 4 U}
\cdot \frac{x}{\sqrt{U}},
\end{aligned}\end{equation}
where we are using the facts that $(\log 3t/8)/t$ is increasing for
$t\geq 8 y > 8e/3$ and that
\[\begin{aligned}\frac{d}{dt} \frac{(\log t)^{3/2} - (\log V)^{3/2}}{\sqrt{t}}
&= \frac{3 (\log t)^{1/2} - ((\log t)^{3/2} - (\log V)^{3/2})}
{2 t^{3/2}}\\ &= - \frac{\log \frac{t}{e^3} \cdot \sqrt{\log t}
- (\log V)^{3/2}
}{2 t^{3/2}} < 0\end{aligned}\]
for $t\geq \theta\cdot 8 y = 16 V$, thanks to 
\[\begin{aligned} \left(\log \frac{16 V}{e^3}\right)^2 \log 16 V &>
(\log V)^3 + \left(\log 16 - 2 \log \frac{e^3}{16}
\right) (\log V)^2\\ &+ \left(\left(\log \frac{16}{e^3}\right)^2 -
2 \log \frac{e^3}{16} \log 16\right) \log V > (\log V)^3\end{aligned}\]
(valid for $\log V \geq 1$). Much as before, we can rewrite (\ref{eq:bilal})
as $x^{5/6}$ times
\begin{equation}\label{eq:jadwi}\begin{aligned}
&\frac{2 \sqrt{\digamma(e^{t/3}/6)}}{\sqrt{8/6}}
\sqrt{\frac{t}{3} - \log 32 c + 
\left(\frac{t}{3} - \log 2\right)
  \log \frac{\frac{t}{3} - \log 3 c}{\log 32/3}}\\
&\cdot \sqrt{\kappa_6 \left(\frac{t}{3} - \log 32 c \right) +
2 \kappa_7} + \frac{2 \kappa_2}{3} \sqrt{\frac{3}{8}}
\left(\left(\frac{t}{3} + \log \frac{32}{6}\right)^{\frac{3}{2}} - 
\left(\frac{t}{3} - \log 3\right)^{\frac{3}{2}}
\right) \\ &+ \frac{2 \kappa_2}{3} 
\frac{1/4}{\sqrt{\frac{e^{t/3}}{6 c} - \frac{1}{6}}}
\left(\left(\frac{t}{3} + \log 24 c\right)^{3/2} - \left(\frac{t}{3} - \log 3\right)^{3/2}\right)\\
&+ \frac{\kappa_2 \sqrt{3}}{\sqrt{2 \left(1 - \frac{c}{e^{t/3}}\right)}} \left(\sqrt{t/3 - \log 3} + \frac{1}{\sqrt{t/3 - \log 3} }\right)
+ \kappa_9 \sqrt{3}\\
&+ \kappa_2 \sqrt{\digamma(e^{t/3}/6)} \sqrt{\frac{t/3 + \log 24 c}{
6 c}},
\end{aligned}\end{equation}
where $t = \log x$ and $c = 500/\sqrt{6}$. For $t\geq 100$, 
we use (\ref{eq:jogaila}) to bound $\digamma(e^{t/3}/6)$, and we obtain that
(\ref{eq:jadwi}) is at most
\begin{equation}\label{eq:asex}
\begin{aligned}&\frac{2 \sqrt{e^{\gamma}}}{\sqrt{8/6}}
\sqrt{\frac{1}{3} \cdot \frac{\kappa_6}{3}} \cdot (\log t) t +
\frac{2\kappa_2}{3} \sqrt{\frac{3}{8}} \cdot \frac{1}{2} \left(
\frac{t}{3} + \log \frac{32}{6}\right)^{1/2} \cdot \log 16\\
&+ \frac{2 \kappa_2}{3} \frac{1/4}{\sqrt{\frac{e^{100/3}}{6 c} - \frac{1}{6}}}
\cdot \frac{1}{2} \left(\frac{t}{3} + \log 24 c\right)^{1/2} \cdot
\log 72 c\\
&+ \frac{\kappa_2 \sqrt{3}}{\sqrt{2 \left(1 - \frac{c}{e^{100/3}}\right)}} \left(\sqrt{t/3} + \frac{1}{\sqrt{t/3} }\right)
+ \kappa_9 \sqrt{3} + \kappa_2 \sqrt{e^\gamma \log t}
\sqrt{\frac{t/3 + \log 24 c}{
6 c}},
\end{aligned}\end{equation}
where we have bounded expressions of the form $a^{3/2} - b^{3/2}$ ($a>b$)
by $(a^{1/2}/2)\cdot (a-b)$. The ratio of (\ref{eq:asex}) to $t^{3/2}$ is
clearly a decreasing function of $t$. For $t=200$, this ratio is
$0.23747\dotsc$; hence, (\ref{eq:asex}) (and thus (\ref{eq:jadwi}))
is at most $0.23748 t^{3/2}$ for $t\geq 200$.

On the range $\log(2.16\cdot 10^{20}) \leq t \leq 200$, the
bisection method (with $25$ iterations)
gives that the ratio of (\ref{eq:jadwi}) to $t^{3/2}$
is at most $0.23511$.

We conclude that, when $|\delta|>8$ and $|\delta| q> 8 y$,
\[|S_{II}|\leq 0.23511 x^{5/6} (\log x)^{3/2}.\]
Thus (\ref{eq:hostoma}) gives the worst case.

We now take totals, and obtain
\begin{equation}\label{eq:bague}\begin{aligned}
S_{\eta}(x,\alpha) &\leq |S_{I,1}|+ |S_{I,2}|+ |S_{II}| \\ &\leq
(2.4719 + 1230.9) x^{2/3} \log x +
(0.00289  + 0.0006406) x^{2/3} (\log x)^2\\ &+ 0.275964 x^{5/6} (\log x)^{3/2}\\
&\leq 0.27598 x^{5/6} (\log x)^{3/2} + 1233.38 x^{2/3} \log x,\end{aligned}\end{equation}
where we use (\ref{eq:voil}) yet again.

\section{Conclusion}
\begin{proof}[Proof of Theorem \ref{thm:minmain}]
We have shown that $|S_{\eta}(\alpha,x)|$ is at most (\ref{eq:duaro}) for 
$q\leq x^{1/3}/6$ and at most (\ref{eq:bague}) for $q> x^{1/3}/6$.
It remains to simplify (\ref{eq:duaro}) slightly. 
By the
geometric mean/arithmetic mean inequality,
\begin{equation}\label{eq:moryo}
\sqrt{C_{x,\delta_0 q} (\log \delta_0 q +0.002) + \frac{\log 4 \delta_0 q}{2}}
\sqrt{0.30214 \log \delta_0 q + 0.67506}\end{equation}
is at most
\[\begin{aligned}
&\frac{1}{2 \sqrt{\rho}} 
\left(C_{x,\delta_0 q} (\log \delta_0 q +0.002) + \frac{\log 4 \delta_0 q}{2}
\right) + \frac{\sqrt{\rho}}{2} (0.30214 \log \delta_0 q + 0.67506)
\end{aligned}\]
for any $\rho>0$. We recall that
\[C_{x,t} =
\log \left(1 + \frac{\log 4 t}{2 \log \frac{9 x^{1/3}}{2.004 t}}\right).\]

Let
\[\rho = \frac{C_{x_1,2 q_0} (\log 2 q_0  +0.002) + 
\frac{\log 8 q_0}{2}}{0.30214 \log 2 q_0 + 0.67506} = 3.397962\dotsc,\]
where $x_1 = 10^{25}$, $q_0 = 2\cdot 10^5$. 
(In other words, we are optimizing matters for $x=x_1$,
$\delta_0 q = 2 q_0$; the losses in nearby ranges will be very slight.) 
We obtain that (\ref{eq:moryo})
is at most
\begin{equation}\label{eq:torot}\begin{aligned}
&\frac{C_{x,\delta_0 q}}{2\sqrt{\rho}} (\log \delta_0 q +0.002) +
\left(\frac{1}{4\sqrt{\rho}} +  \frac{\sqrt{\rho} \cdot 0.30214}{2}\right) \log
\delta_0 q  \\ &+ 
\frac{1}{2} \left(\frac{\log 2}{\sqrt{\rho}} + \frac{\sqrt{\rho}}{2} \cdot
0.67506\right)
\\
&\leq 0.27125 C_{x,t}
 (\log \delta_0 q + 0.002) + 
0.4141 \log \delta_0 q + 0.49911.\end{aligned}\end{equation}
 
Now, for $x\geq x_0 = 2.16\cdot 10^{20}$,
\[\frac{C_{x,t}}{\log t} \leq \frac{C_{x_0,t}}{\log t} 
= \frac{1}{\log t}
\log \left(1 + \frac{\log 4 t}{2 \log \frac{54\cdot 10^6}{2.004 t}}\right)
\leq 0.08659\]
for $8\leq t\leq 10^6$ (by the bisection method, with $20$ iterations), and
\[\frac{C_{x,t}}{\log t} \leq \frac{C_{(6 t)^3,t}}{\log t} \leq
\frac{1}{\log t} \log \left(1 + \frac{\log 4 t}{2 \log \frac{9\cdot 6}{2.004}}
\right)\leq 0.08659.\]
if $10^6 < t \leq x^{1/3}/6$. Hence
\[0.27125 \cdot C_{x,\delta_0 q} \cdot 0.002 \leq 0.000047 \log \delta_0 q.\]


We conclude that, for $q\leq x^{1/3}/6$,
\[\begin{aligned}&|S_{\eta}(\alpha,x)| \leq 
 \frac{R_{x,\delta_0 q} \log \delta_0 q +
 0.49911}{\sqrt{\phi(q) \delta_0}} \cdot x
+ \frac{2.492 x}{\sqrt{q \delta_0}}
\\ &+ 
\frac{2 x}{\delta_0 \phi(q)} 
\left(\frac{13}{4} \log \delta_0 q + 7.82\right)
+ \frac{2 x}{\delta_0 q} (13.66 \log \delta_0 q 
+ 37.55) + 3.36 x^{5/6},\end{aligned}\]
where \[R_{x,t} = 0.27125 \log \left(1 + \frac{\log 4t}{2 \log \frac{9 x^{1/3}}{2.004 t}}\right) + 0.41415.\]
\end{proof}

\part{Major arcs}\label{part:maj}
\chapter{Major arcs: overview and results}

Our task, as in Part \ref{part:min}, will be to estimate
\begin{equation}\label{eq:trinny}
S_{\eta}(\alpha,x) = \sum_n \Lambda(n) e(\alpha n) \eta(n/x),\end{equation}
where
$\eta:\mathbb{R}^+\to \mathbb{C}$ us a smooth function,
$\Lambda$ is the von Mangoldt function and
$e(t) = e^{2\pi i t}$.
Here, we will treat the case of $\alpha$ lying on the major arcs.

We will see how we can obtain good estimates by using smooth functions $\eta$
based on the Gaussian $e^{-t^2/2}$. 
This will involve proving new, fully explicit bounds
for the Mellin transform
of the twisted Gaussian, or, what is the same, bounds on 
 parabolic cylindrical functions in certain ranges. It will also require
explicit formulae that are general and strong enough, even for moderate
values of $x$.

Let $\alpha = a/q + \delta/x$. For us, saying
that $\alpha$ lies on a major arc will be the same as saying that $q$
and $\delta$ are bounded; more precisely, $q$ will be bounded by a constant $r$ 
and $|\delta|$ will be bounded by a constant times $r/q$.
As is customary on the major arcs, 
we will express our exponential sum (\ref{eq:trintrabo})
as a linear combination of twisted sums
\begin{equation}\label{eq:mindy}
S_{\eta,\chi}(\delta/x,x) = \sum_{n=1}^\infty \Lambda(n) \chi(n) e(\delta n/x)
\eta(n/x),\end{equation}
for $\chi:\mathbb{Z}\to \mathbb{C}$ a Dirichlet character mod $q$, i.e., 
a multiplicative character on $(\mathbb{Z}/q\mathbb{Z})^*$ lifted to
$\mathbb{Z}$. (The advantage here is that the phase term is now
$e(\delta n/x)$ rather than $e(\alpha n)$, and $e(\delta n/x)$
varies very slowly as $n$ grows.) Our task, then, is to estimate
$S_{\eta,\chi}(\delta/x,x)$ for $\delta$ small.

Estimates on $S_{\eta,\chi}(\delta/x,x)$ rely on
the properties of Dirichlet 
$L$-functions $L(s,\chi) = \sum_n \chi(n) n^{-s}$.
What is crucial is the location of the zeroes of $L(s,\chi)$
in the critical strip $0\leq \Re(s)\leq 1$ (a region in which $L(s,\chi)$
can be defined by analytic continuation). In contrast to most previous work,
we will not use zero-free regions, which are too narrow for our purposes.
Rather, we use a verification of the Generalized Riemann Hypothesis
up to bounded height
for all conductors $q\leq 300000$ 
(due to D. Platt \cite{Plattfresh}).

A key feature of the present work is that it allows one to mimic a wide 
variety of smoothing functions by means of estimates on the Mellin transform
of a single smoothing function -- here, the Gaussian $e^{-t^2/2}$.
\section{Results}\label{subs:results}
Write $\eta_\heartsuit(t) = e^{-t^2/2}$.
Let us first give a bound for exponential sums on the primes
using $\eta_\heartsuit$ as the smooth weight.
Without loss of generality, we may assume that our character $\chi \mo q$
is primitive, i.e., that it is not really a character to a smaller modulus
$q'|q$.

\begin{theorem}\label{thm:gowo1}
Let $x$ be a real number $\geq 10^8$. 
Let $\chi$ be a primitive Dirichlet character mod $q$, $1\leq q\leq r$, where
$r=300000$.

Then, for any $\delta \in \mathbb{R}$ with $|\delta|\leq 4r/q$,
\[\sum_{n=1}^\infty \Lambda(n) \chi(n) e\left(\frac{\delta}{x} n\right) 
e^{-\frac{(n/x)^2}{2}}
= I_{q=1}\cdot \widehat{\eta_\heartsuit}(-\delta) \cdot x
+ E \cdot x,\]
where $I_{q=1}=1$ if $q=1$, $I_{q=1}=0$ if $q\ne 1$,
and
\[|E|\leq
4.306 \cdot 10^{-22} + \frac{1}{\sqrt{x}}
\left( \frac{650400}{\sqrt{q}} + 112\right).\]
\end{theorem}
We normalize the Fourier transform $\widehat{f}$ as follows:
$\widehat{f}(t) = \int_{-\infty}^\infty e(-xt) f(x) dx$.
Of course, 
$\widehat{\eta_\heartsuit}(-\delta)$ is just $\sqrt{2\pi} e^{-2 \pi^2 \delta^2}$.

As it turns out, smooth weights based on the Gaussian are often better
in applications than the Gaussian $\eta_\heartsuit$ itself. Let us
give a bound based on $\eta(t) = t^2 \eta_\heartsuit(t)$.
\begin{theorem}\label{thm:janar}
Let $\eta(t) = t^2 e^{-t^2/2}$.
Let $x$ be a real number $\geq 10^8$. 
Let $\chi$ be a primitive character mod $q$, $1\leq q\leq r$, where
$r=300000$.

Then, for any $\delta \in \mathbb{R}$ with $|\delta|\leq 4r/q$,
\[\sum_{n=1}^\infty \Lambda(n) \chi(n) e\left(\frac{\delta}{x} n\right) 
\eta(n/x)
= I_{q=1}\cdot \widehat{\eta}(-\delta) \cdot x
+ E \cdot x,\]
where $I_{q=1}=1$ if $q=1$, $I_{q=1}=0$ if $q\ne 1$,
and
\[|E|\leq 2.485\cdot 10^{-19} + 
\frac{1}{\sqrt{x}} \left(\frac{281200}{\sqrt{q}} + 56\right).\]
\end{theorem}
The advantage of $\eta(t) = t^2 \eta_\heartsuit(t)$ over $\eta_\heartsuit$ is
that it vanishes at the origin (to second order); as we shall see, this makes
it is easier to estimate exponential sums with the smoothing $\eta \ast_M g$,
where $\ast_M$ is a Mellin convolution and $g$ is nearly arbitrary.
Here is a good example that is used, crucially, in Part \ref{part:concl}.

\begin{corollary}\label{cor:coprar}
Let $\eta(t) = t^2 e^{-t^2/2} \ast_M \eta_2(t)$, where $\eta_2 = \eta_1 \ast_M
\eta_1$ and $\eta_1 = 2\cdot I_{\lbrack 1/2,1\rbrack}$.
Let $x$ be a real number $\geq 10^8$. 
Let $\chi$ be a primitive character mod $q$, $1\leq q\leq r$, where
$r=300000$.

Then, for any $\delta \in \mathbb{R}$ with $|\delta|\leq 4r/q$,
\[\sum_{n=1}^\infty \Lambda(n) \chi(n) e\left(\frac{\delta}{x} n\right) 
\eta(n/x)
= I_{q=1}\cdot \widehat{\eta}(-\delta) \cdot x
+ E \cdot x,\]
where $I_{q=1}=1$ if $q=1$, $I_{q=1}=0$ if $q\ne 1$,
and
\[|E|\leq 2.485\cdot 10^{-19} + 
\frac{1}{\sqrt{x}} \left(\frac{381500}{\sqrt{q}} + 76\right).\]
\end{corollary}

Let us now look at a different kind of modification of the Gaussian smoothing.
Say we would like a weight of a specific shape; for example, what we will
need to do in Part \ref{part:concl}, we would like an approximation to the function
\begin{equation}\label{eq:cleo2}
\eta_\circ:t\mapsto \begin{cases}
t^3 (2-t)^3 e^{-(t-1)^2/2} &\text{for $t\in \lbrack 0,2\rbrack$,}\\
0 &\text{otherwise.}\end{cases}\end{equation}
At the same time, what we have is an estimate for the Mellin transform of
the Gaussian $e^{-t^2/2}$, centered at $t=0$. 

The route taken here is to work with an approximation $\eta_+$ to $\eta_\circ$.
We let
\begin{equation}\label{eq:patra2}\eta_+(t) = h_H(t) \cdot 
t e^{-t^2/2},\end{equation} 
where $h_H$ is a band-limited approximation to 
\begin{equation}\label{eq:hortor}
h(t) = \begin{cases}
t^2 (2-t)^3 e^{t-1/2} &\text{if $t\in \lbrack 0,2\rbrack$,}\\ 
0 &\text{otherwise.}\end{cases}\end{equation}
By {\em band-limited} we mean that the restriction of the Mellin transform
of $h_H$ to the imaginary axis is of compact support. (We could, alternatively,
let $h_H$ be a function whose Fourier transform is of compact support; this
would be technically easier in some ways, but it would also lead to using GRH
verifications less efficiently.) 

To be precise: we define
\begin{equation}\label{eq:dirich2}\begin{aligned}
F_H(t) &= \frac{\sin(H \log y)}{\pi \log y},\\
h_H(t) &= (h \ast_M F_H)(y) = \int_0^\infty h(t y^{-1}) F_H(y)
\frac{dy}{y}
\end{aligned}\end{equation}
and $H$ is a positive constant. It is easy to check that $M F_H(i\tau) = 1$ for 
$-H<\tau<H$ and $M F_H(i \tau) =0$ for $\tau>H$ or $\tau<-H$ (unsurprisingly,
since $F_H$ is a Dirichlet kernel under a change of variables). Since,
in general, the Mellin transform of a multiplicative convolution $f\ast_M g$
equals $M f \cdot M g$, we see that the Mellin transform of $h_H$,
on the imaginary axis, equals the truncation of the Mellin transform of $h$
to $\lbrack - i H, i H\rbrack$. Thus, $h_H$ is a band-limited approximation to
$h$, as we desired.

The distinction between the odd and the even case in the statement that follows
simply reflects the two different points up to which computations where carried
out in \cite{Plattfresh}; these computations were, in turn, to some
extent tailored to the
needs of the present work (as was the shape of $\eta_+$ itself).
\begin{theorem}\label{thm:malpor}
Let $\eta(t) = \eta_+(t) = h_H(t) t e^{-t^2/2}$,
where $h_H$ is as in (\ref{eq:dirich2}) and $H=200$.
Let $x$ be a real number $\geq 10^{12}$. 
Let $\chi$ be a primitive character mod $q$, where
$1\leq q\leq 150000$ if $q$ is odd, and
$1\leq q\leq 300000$ if $q$ is even.

Then, for any $\delta \in \mathbb{R}$ with 
$|\delta|\leq 600000 \cdot \gcd(q,2)/q$,
\[\sum_{n=1}^\infty \Lambda(n) \chi(n) e\left(\frac{\delta}{x} n\right) 
\eta(n/x)
= I_{q=1}\cdot \widehat{\eta}(-\delta) \cdot x
+ E \cdot x,\]
where $I_{q=1}=1$ if $q=1$, $I_{q=1}=0$ if $q\ne 1$,
and
\[|E|\leq
1.3482\cdot 10^{-14}+
\frac{1.617\cdot 10^{-10}}{q} 
+ \frac{1}{\sqrt{x}}  \left(\frac{499900}{\sqrt{q}} + 52\right).\]

If $q=1$, we have the sharper bound
\[|E|\leq  4.772 \cdot 10^{-11} +  
\frac{251400}{\sqrt{x}}.\]
\end{theorem}
This is a paradigmatic example, in that, following the proof given in 
\S \ref{subs:astardo}, we can bound exponential sums with weights of the form
$h_H(t) e^{-t^2/2}$, where $h_H$ is a band-limited approximation to just
about any continuous function of our choosing.


Lastly, we will need an explicit estimate of the $\ell_2$ norm corresponding to
the sum in Thm.~\ref{thm:malpor}, for the trivial character.
\begin{prop}\label{prop:malheur}
Let $\eta(t) = \eta_+(t) = h_H(t) t e^{-t^2/2}$,
where $h_H$ is as in (\ref{eq:dirich2}) and $H=200$.
Let $x$ be a real number $\geq 10^{12}$. 

Then
\[\begin{aligned}
\sum_{n=1}^\infty \Lambda(n) (\log n)
\eta^2(n/x)
&= x\cdot \int_0^\infty \eta_+^2(t) \log x t\; dt + E_1 \cdot x \log x\\
&= 0.640206 x \log x - 0.021095 x + E_2 \cdot x \log x,\end{aligned}\]
where
\[|E_1|\leq 5.123\cdot 10^{-15} + \frac{366.91}{\sqrt{x}}\;\;\;\;\;\; |E_2|\leq
2\cdot 10^{-6} + \frac{366.91}{\sqrt{x}}.\]
\end{prop}
\section{Main ideas}



An {\em explicit formula} gives an expression
\begin{equation}\label{eq:manon}
S_{\eta,\chi}(\delta/x,x) = I_{q=1} \widehat{\eta}(-\delta) x - 
\sum_{\rho} F_{\delta}(\rho) x^\rho + \text{small error},\end{equation}
where $I_{q=1}=1$ if $q=1$ and $I_{q=1}=0$ otherwise. Here $\rho$ runs
over the complex numbers $\rho$ with $L(\rho,\chi)=0$ and 
$0<\Re(\rho)<1$ (``non-trivial zeros''). The function $F_\delta$ is the
Mellin transform of $e(\delta t) \eta(t)$ (see \S \ref{subs:milly}).

The questions are then: where are the non-trivial zeros $\rho$ of $L(s,\chi)$?
How fast does $F_\delta(\rho)$ decay as $\Im(\rho)\to \pm \infty$?

Write $\sigma = \Re(s)$, $\tau = \Im(s)$.
The belief is, of course, that $\sigma=1/2$ for every non-trivial zero
(Generalized Riemann Hypothesis), but this is far from proven.
Most work to date has used zero-free regions of the form 
$\sigma \leq 1 - 1/C \log q |\tau|$, $C$ a constant. This is a classical
zero-free region, going back, qualitatively, to de la Vall\'ee-Poussin (1899).
The best values of $C$ known are due to McCurley \cite{MR726004} and Kadiri
 \cite{MR2140161}.

These regions seem too narrow to yield a proof of the three-primes theorem.
What we will use instead is a finite verification of GRH ``up to $T_q$'', 
i.e., a computation 
showing that, for every Dirichlet character of conductor $q\leq r_0$
($r_0$ a constant, as above), every non-trivial zero $\rho=\sigma+i\tau$
with $|\tau|\leq T_q$ satisfies $\Re(\sigma)=1/2$. Such
verifications go back to Riemann; modern computer-based
methods are descended in part from a paper by Turing \cite{MR0055785}.
(See the historical article \cite{MR2263990}.)
In his thesis \cite{Platt}, D. Platt gave a rigorous verification for
$r_0 = 10^5$, $T_q = 10^8/q$. In coordination with the present work, he has
extended this to
\begin{itemize}
\item all odd $q\leq 3\cdot 10^5$, with $T_q = 10^8/q$,
\item all even $q\leq 4\cdot 10^5$, with 
$T_q = \max(10^8/q,200 + 7.5\cdot 10^7/q)$.
\end{itemize}  
This was a major computational effort, involving, in particular, a fast
implementation of interval arithmetic (used for the sake of rigor). 

What remains to discuss, then, is how to choose $\eta$ in such a way
$F_\delta(\rho)$ decreases fast enough as $|\tau|$ increases, so
that (\ref{eq:manon}) gives a good estimate. We cannot hope for 
$F_\delta(\rho)$ to start decreasing consistently before $|\tau|$ is at least
as large as a constant times $|\delta|$. Since $\delta$ varies within
$(-c r_0/q, c r_0/q)$, this explains why $T_q$ is taken inversely proportional
to $q$ in the above. As we will work with $r_0\geq 150000$, we also see
that we have little margin for maneuver: we want $F_\delta(\rho)$ to be
extremely small already for, say, $|\tau|\geq 80 |\delta|$.
 We also have a Scylla-and-Charybdis situation, courtesy
of the uncertainty principle: roughly speaking, $F_\delta(\rho)$ cannot decrease
faster than exponentially on $|\tau|/|\delta|$ both for $|\delta|\leq 1$ and for
$\delta$ large.

The most delicate case is that of $\delta$ large, since then $|\tau|/|\delta|$
is small. It turns out we can manage to get decay that is much faster
than exponential for $\delta$ large, while no slower than exponential
 for $\delta$ small. This we will achieve
by working with smoothing functions based on the (one-sided) 
Gaussian $\eta_\heartsuit(t) = e^{-t^2/2}$.

The Mellin transform of the twisted Gaussian $e(\delta t) e^{-t^2/2}$ is
a parabolic cylinder function $U(a,z)$ with $z$ purely imaginary.
Since fully explicit estimates for $U(a,z)$, $z$ imaginary, have not been
worked in the literature, we will have to derive them ourselves. 

Once we have fully explicit estimates for the Mellin transform of the twisted
Gaussian, we are able to use essentially
any smoothing function based on the Gaussian $\eta_\heartsuit(t) = e^{-t^2/2}$. 
As we already saw, we
can and will consider smoothing functions obtained by convolving the twisted
Gaussian with another function and also functions obtained by multiplying the twisted Gaussian
with another function. All we need to do is use an explicit formula of the 
right kind -- that is, a formula that does not assume too much about the smoothing
function or the region of holomorphy of its Mellin transform, but still gives
very good error terms, with simple expressions. 

All results here will be based on a single, general
 explicit formula (Lem.~\ref{lem:agamon}) 
valid for all our purposes. The contribution of the zeros in the critical trip
can be handled in a unified way (Lemmas \ref{lem:garmola} and 
\ref{lem:hausierer}). All that has to be done for each smoothing function
is to bound a simple integral (in (\ref{eq:jotok})). We then apply a finite
verification of GRH and are done.

\chapter{The Mellin transform of the twisted Gaussian}\label{chap:meltra}

Our aim in this chapter is 
to give fully explicit, yet relatively simple bounds for 
the Mellin transform $F_\delta(\rho)$ of $e(\delta t) \eta_\heartsuit(t)$, where
 $\eta_\heartsuit(t) = e^{-t^2/2}$ and $\delta$ is arbitrary.
The rapid decay that results will establish that the Gaussian 
$\eta_\heartsuit$ is a very good choice for a smoothing, particularly when
the smoothing has to be twisted by an additive character $e(\delta t)$.

The Gaussian smoothing has been used before in number theory; see, notably,
Heath-Brown's well-known paper on the fourth power moment of the
Riemann zeta function \cite{MR532980}. 
What is new here is that we will derive fully explicit
bounds on the Mellin transform of the twisted Gaussian. This means that
the Gaussian smoothing will be a real option in explicit work on exponential 
sums in number theory and elsewhere from now on.\footnote{
There has also been work using the Gaussian after a logarithmic change
of variables; see, in particular, \cite{MR0202686}. In that case, the Mellin
transform is simply a Gaussian (as in, e.g.,
\cite[Ex. XII.2.9]{MR2378655}). However, for $\delta$ non-zero,
the Mellin transform of a twist $e(\delta t) e^{-(\log t)^2/2}$ decays
very slowly, and thus would not be useful for our purposes, or, in general,
for most applications in which GRH is not assumed.}

\begin{theorem}\label{thm:princo}
Let $f_\delta(t) = e^{-t^2/2} e(\delta t)$, $\delta\in \mathbb{R}$. 
Let $F_\delta$ be the Mellin transform of $f_\delta$.
 Let $s = \sigma+i \tau$, $\sigma\geq 0$,
$\tau \ne 0$. Let $\ell = - 2 \pi \delta$. Then, if $\sgn(\delta)\ne \sgn(\tau)$ and $\delta\ne 0$,
\begin{equation}\label{eq:wilen}|F_{\delta}(s)| \leq
|\Gamma(s)| e^{\frac{\pi}{2} \tau} e^{-E(\rho) \tau} \cdot
\begin{cases} 
c_{1,\sigma,\tau}/\tau^{\sigma/2} &\text{for $\rho$ arbitrary,}\\
c_{2, \sigma,\tau}/\ell^\sigma &\text{for $\rho\leq 3/2$.}\end{cases}
\end{equation}
where $\rho = 4 \tau/\ell^2$,
\begin{equation}\label{eq:cormo}\begin{aligned}
E(\rho) &=\frac{1}{2} 
\left(\arccos \frac{1}{\upsilon(\rho)} -
\frac{2 (\upsilon(\rho)- 1)}{\rho}
\right),\\
c_{1,\sigma,\tau} &= 
\frac{1}{2}
\left(1 + 2^{\frac{1}{4}}
\left(\frac{2}{1 + \sin^2 \frac{\pi}{8}}\right)^{\sigma/2}
+ \frac{e^{-\left(\frac{\sqrt{2}-1}{2}\right) 
\tau}}{\left(\tan \frac{\pi}{8}\right)^\sigma}\right)\\
c_{2,\sigma,\tau} &= 
\frac{1}{2} \left(1 + 
\min\left(2^{\sigma + \frac{1}{2}},
\frac{\sqrt{\sec \frac{2\pi}{5}}}{\left(\sin \frac{\pi}{5}\right)^\sigma}
\right) + \frac{e^{-\frac{\tau}{6}}}{(1/\sqrt{3})^\sigma}\right).
\end{aligned}\end{equation}
and
\[\upsilon(\rho) = \sqrt{\frac{1+\sqrt{\rho^2+1}}{2}}.\]
If $\sgn(\delta)=\sgn(\tau)$ or $\delta=0$,
\begin{equation}\label{eq:octop}
|F_{\delta}(s)| \leq 
|x_0|^{-\sigma} \cdot e^{-\frac{1}{2} \ell^2} |\Gamma(s)| e^{\frac{\pi}{2}|\tau|}
\cdot \left(\left(1 + \frac{\pi}{2^{3/2}}\right) e^{-\frac{\pi}{4} |\tau|} + 
\frac{1}{2}  e^{-\pi |\tau|}\right),
\end{equation}
where \begin{equation}\label{eq:pastafrola}
|x_0|\geq \begin{cases} 0.51729 \sqrt{\tau}
& \text{for $\rho$ arbitrary},\\ 0.84473 \frac{|\tau|}{|\ell|}&
\text{for $\rho \leq 3/2$.}
\end{cases}\end{equation}
\end{theorem}

As we shall see, the choice of smoothing function
$\eta(t) = e^{-t^2/2}$ can be easily motivated
by the method of stationary phase, but the problem is actually solved by
the saddle-point method. One of the challenges here is to keep all expressions
explicit and practical. 

(In particular, the more critical estimate, (\ref{eq:wilen}), 
is optimal up to a constant depending on $\sigma$; the constants we give
will be good rather than optimal.)

 
The expressions in Thm.~\ref{thm:princo} can be easily simplified further, 
especially if one is ready to introduce some mild constraints and 
make some sacrifices in the main term. 


\begin{corollary}\label{cor:amanita1}
Let $f_\delta(t) = e^{-t^2/2} e(\delta t)$, $\delta\in \mathbb{R}$. Let
$F_\delta$ be the Mellin transform of $f_\delta$. Let 
$s = \sigma+i\tau$, where $\sigma\in \lbrack 0,1\rbrack$ and
$|\tau|\geq 20$. 
Then, for $0\leq k\leq 2$,
\[|F_\delta(s+k)|+ |F_\delta((1-s)+k)| \leq
\begin{cases}
\kappa_{k,0} \left(\frac{|\tau|}{|\ell|}\right)^k
e^{-0.1065 \left(\frac{2 |\tau|}{|\ell|}\right)^2}
& \text{if $4 |\tau|/\ell^2 < 3/2$.}\\
\kappa_{k,1} |\tau|^{k/2}
e^{- 0.1598 |\tau|}
&\text{if $4 |\tau|/\ell^2 \geq 3/2$.}
\end{cases}
\]
where
\[\begin{aligned}
\kappa_{0,0}&\leq 3.001,\;\;\;\;\;
\kappa_{1,0}\leq
4.903,\;\;\;\;\;
\kappa_{2,0} \leq
7.96,\\
\kappa_{0,1} &\leq 3.286,\;\;\;\;\;
\kappa_{1,1} \leq 4.017,\;\;\;\;\;
\kappa_{2,1} \leq 5.13
.\end{aligned}
\]
\end{corollary}
We are considering $F_\delta(s+k)$, and not just $F_\delta(s)$, because
bounding $F_\delta(s+k)$ enables us to work with
smoothing functions equal to or based on $t^k e^{-t^2/2}$. Clearly, we can
easily derive bounds with $k$ arbitrary from Thm.~\ref{thm:princo}.
It is just that we will
use $k=0,1,2$ in practice.
Corollary \ref{cor:amanita1} is meant to be applied to cases where
$\tau$ is larger than a constant ($10$, say) times $|\ell|$, and $\sigma$
cannot be bounded away from $1$; if either condition fails to hold, it is
better to apply Theorem \ref{thm:princo} directly.



Let us end by a remark that may be relevant to applications outside number
theory. By (\ref{eq:mosot}), 
Thm.~\ref{thm:princo} gives us bounds on the parabolic cylinder
function $U(a,z)$ for $z$ purely imaginary. (Surprisingly, there seem to have
been no fully explicit bounds for this case in the literature.)
The bounds are useful when $|\Im(a)|$ is at least somewhat larger than
$|\Im(z)|$ (i.e., when $|\tau|$ is large compared to $\ell$). 
While the Thm.~\ref{thm:princo} is stated for $\sigma \geq 0$ 
(i.e., for $\Re(a)\geq -1/2$),
 extending the result to larger half-planes for $a$ is not
hard.


\section{How to choose a smoothing function?}

Let us motivate our choice of smoothing function $\eta$. 
The method of {\em stationary phase} (\cite[\S 4.11]{MR0435697}, 
\cite[\S II.3]{MR1851050}))
 suggests that the main contribution to the integral
\begin{equation}\label{eq:madejap}
F_\delta(t) = \int_0^\infty e(\delta t) \eta(t) t^s \frac{dt}{t}
\end{equation}
should come when the phase has derivative $0$. The
phase part of (\ref{eq:madejap}) is
\[e(\delta t) t^{\Im(s) i} = e^{(2\pi \delta t + \tau \log t) i}\]
(where we write $s=\sigma+i\tau$); clearly,
\[(2\pi \delta t + \tau \log t)' = 2\pi \delta + \frac{\tau}{t} = 0\]
when $t = -\tau/2\pi \delta$. This is meaningful when $t\geq 0$, i.e., 
$\sgn(\tau) \ne \sgn(\delta)$. The contribution of
$t = -\tau/2\pi \delta$ to (\ref{eq:madejap}) is then
\begin{equation}\label{eq:rateda}
\eta(t) e(\delta t) t^{s-1} = \eta\left(\frac{-\tau}{2\pi \delta}\right)
e^{-i \tau} \left(\frac{-\tau}{2\pi \delta}\right)^{\sigma+i\tau-1}
\end{equation}
multiplied by a ``width'' approximately equal to a constant divided by
\[\sqrt{|(2\pi i \delta t + \tau \log t)''|} = \sqrt{|-\tau/t^2|} = \frac{2\pi |\delta|}{\sqrt{|\tau|}}.\]
The absolute value of (\ref{eq:rateda}) is
\begin{equation}\label{eq:maloko}
\eta\left(-\frac{\tau}{2\pi \delta}\right) \cdot 
\left|\frac{- \tau}{2\pi \delta}\right|^{\sigma-1}.
\end{equation}

In other words, if $\sgn(\tau)\ne \sgn(\delta)$ and $\delta$ is not too small,
asking that $F_\delta(\sigma + i \tau)$ decay rapidly as $|\tau|\to \infty$
amounts to asking that $\eta(t)$ decay rapidly as $t\to 0$. Thus, if 
we ask for $F_\delta(\sigma + i \tau)$ to decay rapidly as $|\tau|\to \infty$
for all moderate $\delta$, we are requesting that
\begin{enumerate}
\item $\eta(t)$ decay rapidly as $t\to \infty$,
\item\label{it:reko} the Mellin transform $F_0(\sigma+i \tau)$ decay rapidly as $\tau
\to \pm \infty$.
\end{enumerate}
Requirement (\ref{it:reko}) is there because we also need to consider
$F_\delta(\sigma+it)$ for $\delta$ very small, and, in particular, for
$\delta=0$.

There is clearly an uncertainty-principle issue here; one cannot do arbitrarily
well in both aspects at the same time. Once we are conscious of this, the
choice $\eta(t)=e^{-t}$ in Hardy-Littlewood actually looks fairly good:
obviously, $\eta(t) = e^{-t}$ decays exponentially, and its Mellin transform
$\Gamma(s+i\tau)$ also decays exponentially as $\tau\to \pm \infty$.
Moreover, for this choice of $\eta$, the Mellin transform $F_\delta(s)$ can
be written explicitly:
$F_\delta(s) = \Gamma(s)/(1-2\pi i\delta)^s$.

It is not hard to work out an explicit formula\footnote{There may be a minor gap in the
  literature in this respect. The explicit formula given
in \cite[Lemma 4]{MR1555183} does not make all constants explicit.
The constants and trivial-zero terms were fully worked out for $q=1$ 
by \cite{Wigert}  (cited in \cite[Exercise 12.1.1.8(c)]{MR2378655}; 
the sign of $\hyp_{\kappa,q}(z)$ there seems to be off).
As was pointed out by Landau (see \cite[p. 628]{MR0201267}),
\cite{MR1555183} seems to neglect the effect of the zeros $\rho$ with
$\Re(\rho)=0$, $\Im(\rho)\ne 0$ for $\chi$
non-primitive. (The author thanks R. C. Vaughan for this information
and the references.)} for $\eta(t) = e^{-t}$.
However, it is not hard to see that, for
$F_\delta(s)$ as above, $F_\delta(1/2+it)$ decays like $e^{-t/2 \pi
  |\delta|}$,
just as we expected from (\ref{eq:maloko}).
This is a little too slow for our purposes: we will often have to work
with relatively large $\delta$, and we would like to have to check
the zeroes of $L$ functions only up to relatively low heights $t$ --
say, up to $50 |\delta|$. Then $e^{-t/2\pi |\delta|}> e^{-8}=0.00033\dotsc$,
which is not very small.
We will settle for a different choice of $\eta$: the Gaussian. 

The decay of
the Gaussian smoothing function $\eta(t) = e^{-t^2/2}$ is much faster than
exponential. Its Mellin transform is $\Gamma(s/2)$, which
decays exponentially as $\Im(s) \to \pm \infty$. Moreover,
the Mellin transform
$F_\delta(s)$ ($\delta \ne 0$), while not an elementary or 
very commonly
occurring function, equals (after a change of variables) a relatively 
well-studied special function, namely, a parabolic cylinder function
$U(a,z)$ (or, in Whittaker's \cite{Whi} notation, $D_{-a-1/2}(z)$).


For $\delta$ not too small, the main term will indeed work
out to be proportional to $e^{-(\tau/2\pi \delta)^2/2}$, as the method of stationary
phase indicated. This is, of course, much better than $e^{-\tau/2\pi |\delta|}$.
The ``cost'' is that the Mellin transform $\Gamma(s/2)$ for $\delta=0$
now decays like $e^{- (\pi/4) |\tau|}$ rather than $e^{- (\pi/2) |\tau|}$. This we can
certainly afford.


\section{The twisted Gaussian: overview and setup}\label{subs:melltwist}
\subsection{Relation to the existing literature}
We wish to approximate the Mellin transform
\begin{equation}\label{eq:ivsnow}
F_{\delta}(s) = \int_0^\infty e^{-t^2/2} e(\delta t) t^s
\frac{dt}{t},\end{equation}
where $\delta\in \mathbb{R}$. 
The parabolic cylinder function $U:\mathbb{C}^2\to \mathbb{C}$ is given by
\[U(a,z) = \frac{e^{-z^2/4}}{\Gamma\left(\frac{1}{2} + a\right)}
\int_0^\infty t^{a - \frac{1}{2}} e^{-\frac{1}{2} t^2 - z t} dt\]
for $\Re(a)>-1/2$; the function can be extended to all $a,z\in \mathbb{C}$
either by analytic continuation or by other integral representations
(\cite[\S 19.5]{MR0167642}, \cite[\S 12.5(i)]{MR2655352}). Hence
\begin{equation}\label{eq:mosot}F_\delta(s) = e^{\left(\pi i \delta
\right)^2} \Gamma(s) 
U\left(s-\frac{1}{2},- 2\pi i \delta
\right).\end{equation}
The second argument of $U$ is purely imaginary; it would be otherwise if 
a Gaussian of non-zero mean were chosen.

Let us briefly discuss the state of knowledge up to date 
on Mellin transforms of ``twisted''
Gaussian smoothings, that is, $e^{-t^2/2}$ multiplied 
by an additive character $e(\delta t)$. 
As we have just seen, these Mellin transforms are precisely the parabolic
cylinder functions $U(a,z)$.

The function $U(a,z)$ has been well-studied for $a$ and $z$ real; see, e.g., 
\cite{MR2655352}. Less attention has been paid to the more general case of
$a$ and $z$ complex. The most notable exception is by far the work of
Olver \cite{MR0094496}, 
\cite{MR0109898}, \cite{MR0131580}, \cite{MR0185350};
 he gave asymptotic series 
for $U(a,z)$, $a,z\in \mathbb{C}$. These were asymptotic series in the 
sense of Poincar\'e, and thus not in general convergent; they would solve
our problem if and only if they came with error term bounds. Unfortunately, 
it would seem that all fully explicit
error terms in the literature are either for $a$ and $z$
real, or for $a$ and $z$ outside our range of interest (see both Olver's work
and \cite{MR1993339}.) The bounds in \cite{MR0131580} involve non-explicit 
constants.
 Thus, we will have to find expressions with explicit error bounds
ourselves. Our case is that of
$a$ in the critical strip, $z$ purely imaginary.


\subsection{General approach} We will use the
{\em saddle-point method} (see, e.g., \cite[\S 5]{MR671583}, 
\cite[\S 4.7]{MR0435697}, \cite[\S II.4]{MR1851050}) 
to obtain bounds with an optimal
leading-order term and small error terms. (We used the stationary-phase
method solely as an exploratory tool.)

What do we expect to obtain? Both
the asymptotic expressions in \cite{MR0109898} and the bounds in
\cite{MR0131580} make clear that, if the sign of $\tau = \Im(s)$ is
different from that of $\delta$,
there will a change in behavior 
when $\tau$ gets to be of size about $(2 \pi \delta)^2$. 
This is unsurprising,
given our discussion using stationary phase: for $|\Im(a)|$ smaller
than a constant times
$|\Im(z)|^2$, the term proportional to $e^{-(\pi/4) |\tau|} = e^{-|\Im(a)|/2}$
should be dominant, whereas for $|\Im(a)|$ much larger than a constant
times $|\Im(z)|^2$, the term proportional to $e^{- \frac{1}{2}
\left(\frac{\tau}{2\pi \delta}\right)^2}$ should be dominant.

There is one important difference between the approach we will follow here
and that in \cite{HelfMaj}. In \cite{HelfMaj}, the integral
(\ref{eq:ivsnow}) was estimated by a direct application of the saddle-point 
method. Here, following a suggestion of
N. Temme, we will use the identity 
\begin{equation}\label{eq:hustadom}
U(a,z) = \frac{e^{\frac{1}{4} z^2}}{\sqrt{2 \pi} i}
\int_{c-i\infty}^{c+i\infty} e^{-z u + \frac{u^2}{2}} u^{-a-\frac{1}{2}} du
\end{equation}
(see, e.g., \cite[(12.5.6)]{MR2723248}; $c>0$ is arbitrary). Together,
(\ref{eq:mosot}) and (\ref{eq:hustadom}) give us that
\begin{equation}\label{eq:melis}
F_\delta(s) = \frac{e^{- 2 \pi^2 \delta^2} \Gamma(s)}{\sqrt{2 \pi} i}
\int_{c-i\infty}^{c+i\infty} e^{2\pi i \delta u + \frac{u^2}{2}} u^{-s} du.
\end{equation}
Estimating the integral in (\ref{eq:melis})
turns out to be a somewhat cleaner task than estimating (\ref{eq:ivsnow}).
The overall procedure, however, is in essence the same in both cases.


We write 
\begin{equation}\label{eq:fribar}
\phi(u) = - \frac{u^2}{2} - (2 \pi i \delta) u + i \tau \log u 
\end{equation}
for $u$ real or complex, so that the integral in (\ref{eq:melis}) equals
\begin{equation}\label{eq:drogon}
I(s) = \int_{c-i\infty}^{c+i\infty} e^{-\phi(u)} u^{-\sigma} du.\end{equation}

We wish to find a saddle point. A saddle point is a point $u$ at which
 $\phi'(u) = 0$.
This means that 
\begin{equation}\label{eq:arbre}
- u - 2 \pi i \delta + \frac{i \tau}{u} = 0,\;\;\;\;\;\;
\text{i.e.,}\;\;\;\;\;\;
u^2 - i \ell u - i \tau = 0,\end{equation}
where $\ell = - 2\pi \delta$. The solutions to
$\phi'(u)=0$ are thus
\begin{equation}\label{eq:rooto}
u_0 = \frac{i\ell \pm \sqrt{-\ell^2 + 4 i \tau}}{2} .\end{equation}
The value of $\phi(u)$ at $u_0$ is
\begin{equation}\label{eq:direwolf}\begin{aligned}
\phi(u_0) &= - \frac{i\ell u_0 + i\tau}{2} + i \ell u_0 + i \tau \log u_0
\\ &= \frac{i \ell}{2} u_0 + i\tau \log \frac{u_0}{\sqrt{e}} . 
\end{aligned}\end{equation}
The second derivative at $u_0$ is
\begin{equation}\label{eq:polan}
\phi''(u_0) = - \frac{1}{u_0^2} \left(u_0^2 + i \tau\right) =
 - \frac{1}{u_0^2} (i \ell u_0 + 2 i \tau).\end{equation}

Assign the names $u_{0,+}$, $u_{0,-}$ to the roots in (\ref{eq:rooto})
according to the sign in front of
the square-root (where the square-root is defined so as to have its argument in 
the interval $(-\pi/2,\pi/2\rbrack$). We will actually have to pay attention
just to $u_{0,+}$, since, unlike $u_{0,-}$, it lies on the right half of the plane, where
our contour of integration also lies. We remark that
\begin{equation}\label{eq:picpio}
u_{0,+} = \frac{i\ell + |\ell| \sqrt{-1 + \frac{4 i \tau}{\ell^2}}}{2} 
= \frac{\ell}{2} \left(i \pm \sqrt{-1 + \frac{4 \tau}{\ell^2} i}
\right)\end{equation}
where the sign $\pm$ is $+$ if $\ell> 0$ and $-$ if
$\ell < 0$. If $\ell = 0$, then $u_{0,+} = 
(1/\sqrt{2} + i/\sqrt{2}) \sqrt{\tau}$.

We can assume without loss of generality that $\tau\geq 0$.
We will find it convenient to assume $\tau>0$, since we can deal with $\tau = 0$
simply by letting $\tau\to 0^+$.

\section{The saddle point}
\subsection{The coordinates of the saddle point}
We should start by determining $u_{0,+}$ explicitly, both in rectangular and 
polar coordinates. For one thing, we will need to estimate the integrand
in (\ref{eq:drogon}) for $u = u_{0,+}$. 
The absolute value of the
integrand is then $\left|e^{- \phi(u_{0,+})}
u_{0,+}^{-\sigma}\right| = 
|u_{0,+}|^{-\sigma} e^{- \Re \phi(u_{0,+})}$, and, by
(\ref{eq:direwolf}),
\begin{equation}\label{eq:petpan}\Re \phi(u_{0,+})
= - \frac{\ell}{2} \Im(u_{0,+}) - \arg(u_{0,+}) \tau
.\end{equation}

If $\ell = 0$, we already know that $\Re(u_{0,+}) = \Im(u_{0,+}) = 
\sqrt{\tau/2}$, $|u_{0,+}| = \sqrt{\tau}$ and $\arg u_{0,+} = \pi/4$.
Assume from now on that $\ell\ne 0$.

We will use the expression for $u_{0,+}$ in (\ref{eq:picpio}).
Solving a quadratic equation, we see that
\begin{equation}\label{eq:masha}
\sqrt{-1 + \frac{4 \tau}{\ell^2} i} = 
\sqrt{\frac{j(\rho)-1}{2}} + i \sqrt{\frac{j(\rho)+1}{2}},\end{equation}
where $j(\rho) = (1+\rho^2)^{1/2}$ and $\rho = 4 \tau/\ell^2$. 
Hence
\begin{equation}\label{eq:damherr}
\Re(u_{0,+}) = \pm
\frac{\ell}{2} \sqrt{\frac{j(\rho) - 1}{2}},\;\;\;\;\;\;\;\;
\Im(u_{0,+}) =  \frac{\ell}{2} \left(1 \pm \sqrt{\frac{j(\rho)+1}{2}}
\right).\end{equation}
Here and in what follows, the sign $\pm$ is $+$ if $\ell> 0$ and $-$ if
$\ell < 0$. (Notice that $\Re(u_{0,+})$ and $\Im(u_{0,+})$ are always positive, except for
$\tau = \ell = 0$, in which case $\Re(u_{0,+}) = \Im(u_{0,+}) = 0$.)
By (\ref{eq:damherr}), 
\begin{equation}\label{eq:dada}\begin{aligned}
|u_{0,+}| &= 
 \frac{|\ell|}{2}\cdot \left|\sqrt{\frac{-1 + j(\rho)}{2}} +
\left(1 \pm \sqrt{\frac{1+j(\rho)}{2}}\right) i\right|\\
&= \frac{|\ell|}{2} \sqrt{\frac{-1+j(\rho)}{2} + \frac{1+j(\rho)}{2} + 1
\pm 2 \sqrt{\frac{1+j(\rho)}{2}}} \\ 
&= \frac{|\ell|}{2} \sqrt{1+j(\rho) \pm 2\sqrt{\frac{1+j(\rho)}{2}
}} 
=  \frac{|\ell|}{\sqrt{2}}\sqrt{\upsilon(\rho)^2 \pm \upsilon(\rho)},
\end{aligned}\end{equation}
where $\upsilon(\rho) = \sqrt{(1+j(\rho))/2}$.
We now compute the argument of $u_{0,+}$:
\begin{equation}\label{eq:niels}\begin{aligned}
\arg(u_{0,+}) &= \arg\left(
\ell \left(i \pm \sqrt{- 1 + i \rho}\right)\right) \\ &= 
\arg\left(
\sqrt{\frac{-1 + j(\rho)}{2}} + i\left(\pm 1 +  
\sqrt{\frac{1 + j(\rho)}{2}}\right)\right)\\
&= \arcsin\left(\frac{\pm 1+
\sqrt{\frac{1 + j(\rho)}{2}}}{
\sqrt{1+j(\rho) \pm 2\sqrt{\frac{1+j(\rho)}{2}
}} }
\right) = \arcsin\left(\frac{\sqrt{\pm 1+
\sqrt{\frac{1 + j(\rho)}{2}}}}{\sqrt{2 \sqrt{\frac{1 + j(\rho)}{2}}}}\right)
\\
&=  
\arcsin\left(\sqrt{\frac{1}{2} \left(1 \pm \sqrt{\frac{2}{1+j(\rho)}}\right)}\right)
= \frac{\pi}{2} - \frac{1}{2} \arccos \left(\pm \sqrt{\frac{2}{1+ j(\rho)}}
\right)
\end{aligned}\end{equation}
(by $\cos(\pi-2\theta) = - \cos 2\theta = 2\sin^2 \theta - 1$). Thus
\begin{equation}\label{eq:phal}
\arg(u_{0,+}) = \begin{cases} 
\frac{\pi}{2} - \frac{1}{2} \arccos \frac{1}{\upsilon(\rho)}
= \frac{1}{2} \arccos \frac{-1}{\upsilon(\rho)}
&\text{if $\ell>0$,}\\
\frac{1}{2} \arccos \frac{1}{\upsilon(\rho)}
 &\text{if $\ell<0$.}
\end{cases}\end{equation}
In particular, $\arg(u_{0,+})$ lies in $\lbrack 0,\pi/2\rbrack$, and is close
to $\pi/2$ only when $\ell>0$ and $\rho \to 0^+$.
Here and elsewhere, we follow the convention that $\arcsin$
and $\arctan$ have image in $\lbrack -\pi/2,\pi/2\rbrack$, whereas
$\arccos$ has image in $\lbrack 0,\pi\rbrack$.

\subsection{The direction of steepest descent}\label{subs:stedes}
As is customary in the saddle-point method, it is 
 now time to determine the direction of steepest descent at the
saddle-point $u_{0,+}$. Even if we decide to use a contour that
goes through the saddle-point in a direction that is not quite optimal,
it will be useful to know what the direction $w$ of steepest descent
actually is. A contour that passes through the saddle-point making an
angle between $-\pi/4 + \epsilon$ and $\pi/4 -\epsilon$ with $w$
may be acceptable, in that the contribution of the saddle point is then suboptimal by at most a bounded factor depending on $\epsilon$; an angle approaching $-\pi/4$ or $\pi/4$ leads
to a contribution suboptimal by an unbounded factor.

Let $w\in \mathbb{C}$ be the unit vector
pointing in the direction of steepest descent.
Then, by definition, $w^2 \phi''(u_{0,+})$ is real and positive, where
$\phi$ is as in (\ref{eq:fribar}). Thus $\arg(w) = -\arg(\phi''(u_{0,+}))/2
\mo \pi$. (The direction of steepest descent is defined only modulo $\pi$.)
By (\ref{eq:polan}),
\[\begin{aligned}
\arg(\phi''(u_{0,+})) &= - \pi +
\arg(i \ell u_{0,+} + 2 i \tau) - 2 \arg(u_{0,+}) \mo 2\pi\\
&= - \frac{\pi}{2} +
\arg(\ell u_{0,+} + 2 \tau) - 2 \arg(u_{0,+}) \mo 2\pi.\end{aligned}\]
By (\ref{eq:damherr}),
\[\begin{aligned}
\Re(\ell u_{0,+} + 2 \tau) &= 
\frac{\ell^2}{2} \left(\pm \sqrt{\frac{j(\rho)-1}{2}} +
\frac{4 \tau}{\ell^2}\right) = 
\frac{\ell^2}{2} \left(\rho \pm \sqrt{\frac{j(\rho)-1}{2}}\right),\\
\Im(\ell u_{0,+} + 2 \tau) &= 
\frac{\ell^2}{2} \left(1 \pm \sqrt{\frac{j(\rho)+1}{2}}\right).\\
\end{aligned}\]
Therefore,
$\arg(\ell u_{0,+} + 2 \tau) = \arctan \varpi$, where
\[\varpi = 
\frac{1 \pm \sqrt{\frac{j(\rho)+1}{2}}}{\rho \pm \sqrt{\frac{j(\rho)-1}{2}}}
\]
It is easy to check that $\sgn \varpi = \sgn \ell$. Hence,
\[\arctan \varpi = \pm \frac{\pi}{2} - 
\arctan\left(\frac{\rho \pm \sqrt{\frac{j(\rho)-1}{2}}}{
1 \pm \sqrt{\frac{j(\rho)+1}{2}}}
\right).\]
At the same time,
\begin{equation}\label{eq:jomo}\begin{aligned}
\frac{\rho \pm \sqrt{\frac{j-1}{2}}}{
1 \pm \sqrt{\frac{j+1}{2}}} &=
\frac{\left(\rho \pm \sqrt{\frac{j-1}{2}}\right) \left(1 \mp 
\sqrt{\frac{j+1}{2}}\right)
}{1 - \frac{j+1}{2}}
= \frac{\rho \pm \sqrt{2 (j-1)} \mp \rho \sqrt{2 (j+1)}}{1-j}\\
&= \frac{\rho \pm \sqrt{\frac{2}{j+1}} \left(\sqrt{j^2-1} - \rho
    \cdot (j+1)\right)}{1-j} = \frac{\rho \pm \frac{1}{\upsilon} (\rho - \rho
  \cdot (j+1))}{1-j}\\
&= \frac{\rho (1 \mp j/\upsilon)}{1-j} = \frac{(-1 \pm j/\upsilon) (j+1)}{\rho}
= \frac{2 \upsilon (- \upsilon \pm j)}{\rho}.\end{aligned}\end{equation}
Hence, modulo $2 \pi$,
\[\arg(\phi''(u_{0,+})) 
= 
- \arctan \frac{2 \upsilon (- \upsilon \pm j)}{\rho} - 2 \arg(u_{0,+}) 
- \begin{cases} 0 &\text{if $\ell\geq 0$}\\ \pi &\text{if $\ell < 0$.}
\end{cases}\]
Therefore, the direction of steepest descent is
\begin{equation}\label{eq:heraus}
\arg(w) = -\frac{\arg(\phi''(u_{0,+}))}{2} 
= \arg(u_{0,+}) + 
\frac{1}{2} \arctan \frac{2 \upsilon (-\upsilon\pm j)}{\rho} +
\begin{cases} 0 &\text{if $\ell\geq 0$}\\ \frac{\pi}{2} &\text{if $\ell < 0$.}
\end{cases}\end{equation}
By
(\ref{eq:phal}) and $\arccos 1/\upsilon = \arctan \sqrt{\upsilon^2-1}
= \arctan \sqrt{(j-1)/2}$, 
we conclude that
\begin{equation}\label{eq:rugros}\begin{aligned}
\arg(w) &= \begin{cases}
\frac{\pi}{2} + \frac{1}{2} \left(-\arctan \frac{2 \upsilon (j+\upsilon)}{\rho}
+ \arctan \sqrt{\frac{j-1}{2}}\right) &\text{if $\ell< 0$,}\\
\frac{\pi}{2} 
 + \frac{1}{2} \left(\arctan \frac{2 \upsilon (j-\upsilon)}{\rho}
- \arctan \sqrt{\frac{j-1}{2}}\right) &\text{if $\ell\geq 0$.}
\end{cases}\end{aligned}\end{equation}

There is nothing wrong in using plots here to get an idea of the behavior of
$\arg(w)$, since, at any rate, the direction of steepest descent will play only
an advisory role in our choices. See Figures \ref{fig:argwminus}
and \ref{fig:argwplus}.

\begin{figure}
\begin{minipage}[b]{0.5\linewidth}
                \centering \includegraphics[height=1.9in]{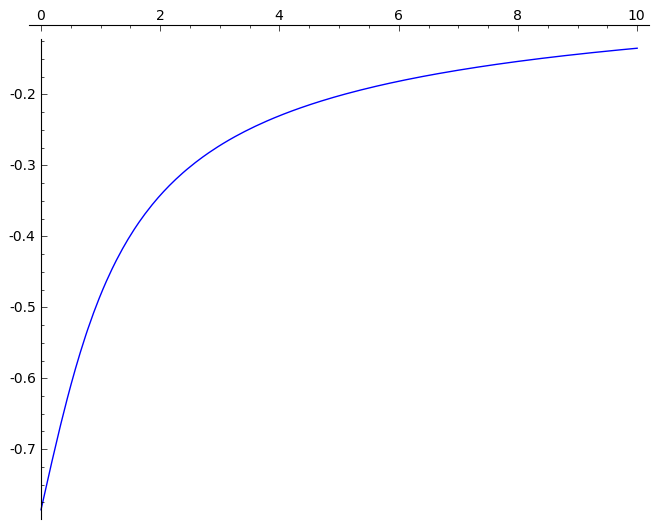}
\caption{$\arg(w)-\pi/2$ as a function of $\rho$ for $\ell <0$}\label{fig:argwminus}
\end{minipage}%
\;\;
\begin{minipage}[b]{0.5\linewidth}
                \centering \includegraphics[height=1.9in]{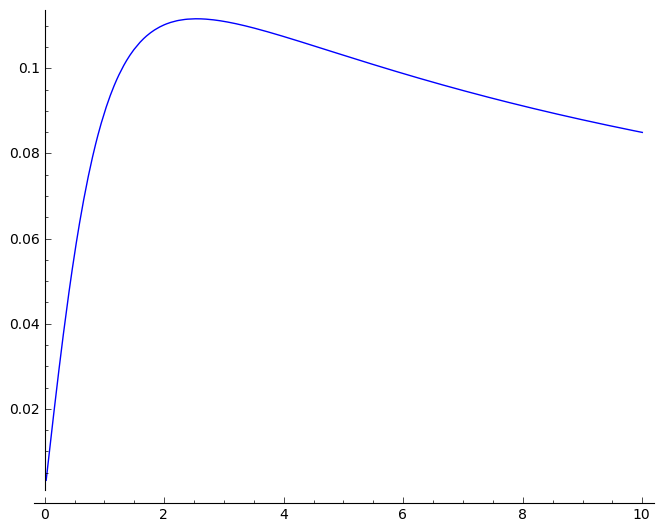}
\caption{$\arg(w)-\pi/2$ as a function of $\rho$ for $\ell\geq 0$}\label{fig:argwplus}
\end{minipage}%
\end{figure}

\section{The integral over the contour}
We must now choose the contour of integration. 
The optimal contour should be one on which the phase of the integrand in
(\ref{eq:drogon}) is constant, i.e., $\Im(\phi(u))$ is constant. 
This is so because, throughout the contour, we want to keep descending from
the saddle as rapidly as possible, and so we want to maximize the 
absolute value of the derivative
of the real part of the exponent $-\phi(u)$. At any point $u$,
if we are to maximize $|\Re(d\phi(u)/dt)|$,
we want our contour to be such that $\Im(d\phi(u)/dt) = 0$. (We can also see this
as follows: if $\Im(\phi(u))$ is constant, there
is no cancellation in (\ref{eq:drogon}) for us to miss.)

Writing $u = x + i y$, we obtain from (\ref{eq:fribar}) that
\begin{equation}\label{eq:morb}
\Im(\phi(u)) = - x y + \ell x + \tau \log \sqrt{x^2+ y^2}.\end{equation}
We would thus be considering the curve $\Im(\phi(u)) = c$, where $c$ is
a constant.
Since we need the contour to pass through the saddle point $u_{0,+}$,
we set $c = \Im(\phi(u_{0,+}))$.
The only problem is that the curve
$\Im(\phi(u)) = 0$ given by (\ref{eq:morb}) is rather uncomfortable
to work with.


Instead, we shall use several rather simple contours, each appropriate 
for different values of $\ell$ and $\tau$.

\subsection{A simple contour}
Assume first that $\ell>0$.
We could just let our contour $L$ be the vertical line going through
$u_{0,+}$. Since the direction of steepest descent is never far from 
vertical (see (\ref{fig:argwplus})), this would be a good choice. 
However, the vertical line has
the defect of going too close to the origin when $\rho\to 0$.

Instead, we will let $L$ consist of three segments: (a) the straight 
vertical ray
\[\{(x_0,y): y\geq y_0\},\] where $x_0 = \Re u_{0,+}\geq 0$, $y_0 = \Im u_{0,+}>0$;
(b) the straight segment going downwards and to the right 
from $u_{0,+}$ to the $x$-axis, forming an angle of $\pi/2-\beta$ 
(where $\beta>0$
will be determined later) with the $x$-axis
at a point $(x_1,0)$;
(c) the straight vertical ray $\{(x_1,y): y\leq 0\}$. Let us call these
three segments $L_1$, $L_2$, $L_3$.
Shifting the contour in (\ref{eq:drogon}), we obtain
\[I = \int_L e^{-\phi(u)} u^{-\sigma} du,\]
and so $|I|\leq I_1 + I_2 + I_3$, where
\begin{equation}\label{eq:destined}
I_j = \int_{L_j} \left|e^{-\phi(u)} u^{-\sigma}\right| |du|.\end{equation}
As we shall see, we have chosen the segments $L_j$ so that 
each of the three integrals $I_j$ will be easy to bound.

Let us start with $I_1$. Since $\sigma\geq 0$,
\[I_1 \leq |u_{0,+}|^{-\sigma} \int_{y_0}^{\infty} e^{-\Re \phi(x_0+iy)} dy,
\]
where, by (\ref{eq:fribar}),
\begin{equation}\label{eq:bankal}
\Re \phi(x + i y) = \frac{y^2 - x^2}{2} - \ell y - \tau \arg(x + i y).
\end{equation}
Let us expand the expression on the right of (\ref{eq:bankal}) for 
$x=x_0$ and $y$ around $y_0 = \Im u_{0,+}>0$. The constant term is
\begin{equation}\label{eq:borgnin}\begin{aligned}\Re \phi(u_{0,+}) &= 
- \frac{\ell}{2} y_0  - \tau \arg(u_{0,+})
= - \frac{\ell^2}{4} (1 + \upsilon(\rho)) - 
\frac{\tau}{2} \arccos \frac{-1}{\upsilon(\rho)}\\
&= - \left( \frac{1+\upsilon(\rho)}{\rho} +
\frac{1}{2} \arccos \frac{-1}{\upsilon(\rho)}\right) \tau,\end{aligned}
\end{equation}
where we are using (\ref{eq:petpan}),
(\ref{eq:damherr}) and (\ref{eq:phal}).

The linear term vanishes because $u_{0,+}$ is a saddle-point (and thus a local
extremum on $L$). It remains to estimate the quadratic term. Now, in 
(\ref{eq:bankal}), 
the term $\arg(x+iy)$ equals $\arctan(y/x)$, whose quadratic term we should
now examine -- but, instead, we are about to see that we can bound it trivially.
 In general,
for $t_0,t\in \mathbb{R}$ and $f\in C^2$,
\begin{equation}\label{eq:kaylor}f(t) = f(t_0) + f'(t_0) \cdot (t-t_0) + 
\int_{t_0}^t \int_{t_0}^r f''(s) ds dr.\end{equation}
Now, $\arctan''(s) = - 2 s/(s^2+1)^2$, and this is negative for $s>0$ and
obeys \[\arctan''(-s) = - \arctan''(s)\] for all $s$. Hence, for $t_0\geq 0$ and $t\geq - t_0$,
\begin{equation}\label{eq:ardan}
\arctan t \leq \arctan t_0 + (\arctan' t_0) \cdot (t-t_0).\end{equation}
Therefore, in (\ref{eq:bankal}), we can consider only the quadratic term
coming from $(y^2-x^2)/2$ -- namely, $(y-y_0)^2/2$ -- and ignore the quadratic
term coming from $\arg(x+i y)$. Thus,
\begin{equation}\label{eq:apprec}
\Re \phi(x_0+i y) \geq \frac{(y-y_0)^2}{2} + \Re \phi(u_{0,+})\end{equation}
for $y\geq - y_0$, and, in particular, for $y\geq y_0$. Hence,
\begin{equation}\label{eq:chnehr}
\int_{y_0}^{\infty} e^{-\Re \phi(x_0+iy)} dy \leq
e^{-\Re \phi(u_{0,+})} 
\int_{y_0}^\infty e^{- \frac{1}{2} (y-y_0)^2} dy = 
\sqrt{\pi/2} \cdot e^{-\Re \phi(u_{0,+})}.\end{equation}
Notice that, once we choose to use 
the approximation (\ref{eq:ardan}), the vertical 
direction is actually optimal. (In turn, the fact that the direction of
steepest descent is close to vertical shows us that we are not losing much
by using the approximation (\ref{eq:ardan}).)

As for $|u_{0,+}|^{-\sigma}$, we will estimate it by the easy bound
\begin{equation}\label{eq:blondie}
|u_{0,+}|= \frac{\ell}{\sqrt{2}} \sqrt{\upsilon^2+\upsilon}
\geq  \frac{\ell}{\sqrt{2}}
\max\left(\sqrt{\frac{\rho}{2}},\sqrt{2}\right) = \max(\sqrt{\tau},\ell),
\end{equation}
where we use (\ref{eq:dada}).

Let us now bound $I_2$. As we already said, the linear term at $u_{0,+}$
vanishes. Let $u_\circ$ be the point at which $L_2$ meets the line normal
to it through the origin. We must take care that the angle formed by the origin,
$u_{0,+}$ and $u_\circ$ be no larger than the angle formed by the origin,
$(x_1,0)$ and $u_0$; this will ensure that 
we are in the range in which the approximation
(\ref{eq:ardan}) is valid (namely, $t\geq -t_0$, where $t_0 = \tan \alpha_0$).
The first angle is $\pi/2+\beta-\arg u_{0,+}$, whereas the second angle is 
$\pi/2-\beta$. Hence, it is enough to set $\beta\leq (\arg u_{0,+})/2$.
%
Then we obtain from (\ref{eq:fribar}) and (\ref{eq:ardan}) that
\begin{equation}\label{eq:journ}
\Re \phi(u) \geq \Re \phi(u_{0,+}) - \Re \frac{(u-u_{0,+})^2}{2}.\end{equation}
If we let $s = |u- u_{0,+}|$, we see that 
\[\Re \frac{(u-u_{0,+})^2}{2}  = \frac{s^2}{2} \cos \left(2\cdot
 \left(\frac{\pi}{2} -\beta\right)\right) = - \frac{s^2}{2} \cos 2 \beta.\]
Hence,
\begin{equation}\label{eq:cuchara}
\begin{aligned} I_2 &\leq |u_\circ|^{-\sigma} \int_{L_2} e^{-\Re \phi(u)} |du| 
\\ &<
|u_\circ|^{-\sigma} \int_{0}^\infty e^{-\Re \phi(u_{0,+}) - \frac{s^2}{2} \cos 2 \beta} ds
= |u_\circ|^{-\sigma} e^{-\Re \phi(u_{0,+})} \sqrt{\frac{\pi}{2 \cos 2 \beta}}
.\end{aligned}\end{equation}

Since $\arg u_0 = \arg u_{0,+} - \beta$, we see that, by (\ref{eq:damherr}),
\begin{equation}\label{eq:empo}\begin{aligned}
|u_\circ| &= 
\Re\left((x_0 + i y_0) \left(\cos \beta - i \sin \beta\right)
\right)\\
&= \frac{\ell}{2} \left(\sqrt{\frac{j-1}{2}} \cos \beta + 
\left(1+ \sqrt{\frac{j+1}{2}}\right) \sin \beta\right).
\end{aligned}\end{equation}
The square of the expression within the outer parentheses is at least
\[\begin{aligned}
\frac{j-1}{2} \cos^2 \beta &+
\left(1+ \frac{j+1}{2} + \sqrt{2 (j+1)}\right) \sin^2 \beta +
\left(\sqrt{\frac{j^2-1}{4}} + \sqrt{\frac{j-1}{2}}\right) \sin 2 \beta
\\
&\geq
\frac{j}{2} + \frac{7}{2} \sin^2 \beta - 
\frac{1}{2} \cos^2 \beta + \frac{j}{2} \sin^2 \beta.
\end{aligned}\]
If $\beta\geq \pi/8$, then $\tan \beta > 1/\sqrt{7}$, and so, since $j>\rho$,
we obtain
\[|u_\circ|\geq \frac{\ell}{2} \sqrt{\frac{j}{2} (1+\sin^2 \beta)} >
\frac{\ell \sqrt{\rho}}{2^{3/2}} \sqrt{1+\sin^2 \beta}
.\]
We can also apply the trivial bound
$j\geq 1$ directly to (\ref{eq:empo}). Thus,
\[|u_\circ| \geq \max\left(\sqrt{\frac{\tau}{2}} \sqrt{1+\sin^2 \beta},
\ell \sin \beta
\right).\]
Let us choose $\beta$ as follows. We could always set $\beta = \pi/8$:
since $\arg u_{0,+}\geq \pi/4$, we then have $\beta \leq (\arg u_{0,+})/2$,
as required. However, if $\rho\leq 3/2$, then 
$\upsilon(\rho) \leq 1.18381$, and so, by (\ref{eq:phal}),
$\arg u_{0,+}\geq 1.28842$. We can thus set either $\beta = \pi/6
= 0.523598\dotsc$ or
$\beta = \pi/5 = 0.628318\dotsc$, say, either of which is smaller than 
$(\arg u_{0,+})/2$.
Going back to (\ref{eq:cuchara}), 
we conclude that
\[I_2\leq e^{-\Re \phi(u_{0,+})} \cdot \frac{\sqrt{\pi}}{2^{1/4}}
\left|\sqrt{\frac{\tau}{2}} \sqrt{1+\sin^2 \frac{\pi}{8}}
\right|^{-\sigma}\] for $\rho$ arbitrary, and
\[I_2\leq e^{-\Re \phi(u_{0,+})} \cdot \min\left(\sqrt{\frac{\pi/2}{\cos 2\pi/5}}
\cdot \left|\ell \sin \frac{\pi}{5}\right|^{-\sigma},
\sqrt{\pi} \left|\frac{\ell}{2}\right|^{-\sigma}\right)\] when 
$\upsilon(\rho)\leq 3/2$.


It remains to estimate $I_3$. For $u = x_1$, 
\[\begin{aligned}- \Re \frac{(u-u_{0,+})^2}{2} &= 
- \Re \frac{y_0^2 \left(\tan \beta - i\right)^2}{2} = 
\frac{1}{2} \left(1 - \tan^2 \beta\right) y_0^2\\
&\geq \left(1 - \tan^2 \beta\right) \cdot \frac{\ell^2}{8} 
\left(1 + \frac{j+1}{2}\right) \geq
\frac{\ell^2}{8} \left(1 - \tan^2 \beta\right) 
\cdot \frac{\rho}{2}\\ &\geq 
\frac{1}{4} \left(1 - \tan^2 \beta\right) \tau  
,\end{aligned}\]
where we are using (\ref{eq:damherr}). Thus, (\ref{eq:journ}) tells us that
\[\Re \phi(x_1) \geq \Re \phi(u_{0,+}) + \frac{1 - \tan^2 \beta}{4} \tau.\]
At the same time, by (\ref{eq:bankal}) and $\tau,\ell\geq 0$,
\[\Re \phi(x_1+ i y) \geq \Re \phi(x_1) + \frac{y^2}{2}\]
for $y\leq 0$. Hence
\[\begin{aligned}I_3 &\leq
|x_1|^{-\sigma} \int_{L_3} e^{-\Re \phi(u)} |du| \leq |x_1|^{-\sigma}
e^{-\Re \phi(x_1)} \int_{-\infty}^0 e^{-y^2/2} dy\\
&\leq |x_1|^{-\sigma}  \cdot \sqrt{\frac{\pi}{2}} e^{-\frac{1 - \tan^2 \beta}{4} 
\tau} e^{-\Re \phi(u_{0,+})} .
\end{aligned}\]
Here note that $x_1\geq (\tan \beta) |u_{0,+}|$, and so, by (\ref{eq:blondie}),
\[x_1\geq \tan \beta \cdot \max\left(\sqrt{\tau},\ell\right).\]

We conclude that, for $\ell>0$,
\[|I|\leq 
\left(1 + 2^{\frac{1}{4}}
\left(\frac{2}{1 + \sin^2 \frac{\pi}{8}}\right)^{\sigma/2}
+ \frac{e^{-\left(\frac{\sqrt{2}-1}{2}\right) 
\tau}}{\left(\tan \frac{\pi}{8}\right)^\sigma}\right)
\cdot \frac{\sqrt{\pi/2}
}{\tau^{\sigma/2}} e^{-\Re \phi(u_{0,+})}\]
(since $(1-\tan^2 \pi/8)/4 = (\sqrt{2}-1)/2$) and, when $\rho\leq 3/2$,
\[|I|\leq 
\left(1 + \min\left(2^{\sigma + \frac{1}{2}},
\frac{\sqrt{\sec \frac{2\pi}{5}}}{\left(\sin \frac{\pi}{5}\right)^\sigma}
\right) +
\frac{e^{-\frac{\tau}{6}}}{(1/\sqrt{3})^\sigma}
\right)
\cdot \frac{\sqrt{\pi/2}
}{\ell^{\sigma}} e^{-\Re \phi(u_{0,+})}.\]

We know $\Re \phi(u_{0,+})$ from (\ref{eq:borgnin}). Write
\begin{equation}\label{eq:jame}E(\rho) = 
\frac{1}{2} \arccos \frac{1}{\upsilon(\rho)} - 
\frac{\upsilon(\rho)-1}{\rho},\end{equation}
so that 
\[- \Re \phi(u_{0,+}) = \frac{1 + \upsilon(\rho)}{\rho} +
\frac{1}{2} \arccos \frac{- 1}{\upsilon(\rho)} = 
\frac{\pi}{2} - E(\rho) + \frac{2}{\rho}.\]

To finish, we just need to apply (\ref{eq:melis}).
It makes sense to group together $\Gamma(s) e^{\frac{\pi}{2} \tau}$, since it
is bounded on the critical line (by the classical formula
$|\Gamma(1/2+i \tau)| = \sqrt{\pi/\cosh \pi \tau}$, as in 
\cite[Exer.~C.1(b)]{MR2378655}), and, in general, 
of slow growth on bounded strips.
Using (\ref{eq:melis}), 
and noting that $2\pi^2 \delta^2 = \ell^2/2 = (2/\rho)\cdot \tau$,
we obtain
\begin{equation}\label{eq:huj}\left|F_\delta(s)\right| \leq
|\Gamma(s)| e^{\frac{\pi}{2} \tau} e^{-E(\rho) \tau} \cdot
\begin{cases} 
c_{1,\sigma,\tau}/\tau^{\sigma/2} &\text{for $\rho$ arbitrary,}\\
c_{2, \sigma,\tau}/\ell^\sigma &\text{for $\rho\leq 3/2$.}\end{cases}
\end{equation}
where
\begin{equation}\label{eq:marmota}\begin{aligned}
c_{1,\sigma,\tau} &= 
\frac{1}{2}
\left(1 + 2^{\frac{1}{4}}
\left(\frac{2}{1 + \sin^2 \frac{\pi}{8}}\right)^{\sigma/2}
+ \frac{e^{-\left(\frac{\sqrt{2}-1}{2}\right) 
\tau}}{\left(\tan \frac{\pi}{8}\right)^\sigma}\right)\\
c_{2,\sigma,\tau} &= 
\frac{1}{2} \left(1 + 
\min\left(2^{\sigma + \frac{1}{2}},
\frac{\sqrt{\sec \frac{2\pi}{5}}}{\left(\sin \frac{\pi}{5}\right)^\sigma}
\right) + \frac{e^{-\frac{\tau}{6}}}{(1/\sqrt{3})^\sigma}\right).
\end{aligned}\end{equation}
We have assumed throughout that $\ell\geq 0$ and $\tau\geq 0$.
We can immediately obtain a bound valid for $\ell\leq 0$, $\tau\leq 0$,
by reflection on the $x$-axis; we simply put
absolute values around $\tau$ and $\ell$ in (\ref{eq:huj}).

We see that we have obtained a bound in a neat, closed form without too
much effort. Of course, this
effortlessness is usually in part illusory; the contour
we have used here is actually the product of some trial and error, in that some 
other contours give results that are comparable in quality but harder to 
simplify. We will have to choose a different contour when $\sgn(\ell) \ne 
\sgn(\tau)$.



\subsection{Another simple contour}

We now wish to give a bound for the case of $\sgn(\ell) \ne \sgn(\tau)$,
i.e., $\sgn(\delta) = \sgn(\tau)$. We expect a much smaller upper bound
than for $\sgn(\ell) = \sgn(\tau)$, given what we already know from the method
of stationary phase. This also means that we will not need to be as careful
in order to get a bound that is good enough for all practical purposes.


Our contour $L$ will consist of three segments: (a) the straight vertical ray
$\{(x_0,y): y\geq 0\}$, (b) the quarter-circle from $(x_0,0)$ to $(0,-x_0)$
(that is, an arc where the argument runs from
$0$ to $-\pi/2$), and (c) the straight vertical ray
$\{(0,y): y\leq -x_0\}$. We call these segments $L_1$, $L_2$, $L_3$,
and define the integrals $I_1$, $I_2$ and $I_3$ just as in (\ref{eq:destined}).

Much as before, we have
\[I_1 \leq x_0^{-\sigma} \int_0^\infty e^{-\Re \phi(x_0 +i y)} dy.\]
Since (\ref{eq:ardan}) is valid for $t\geq 0$, (\ref{eq:apprec}) holds, and so
\[I_1 \leq x_0^{-\sigma} e^{-\Re \phi(u_{0,+})} \int_{-\infty}^\infty e^{-\frac{1}{2} 
(y-y_0)^2} dy = x_0^{-\sigma} \sqrt{2 \pi}\cdot 
e^{-\Re \phi(u_{0,+})}.\]

By (\ref{eq:fribar}) and (\ref{eq:bankal}),
\begin{equation}\label{eq:sutherl}
I_2 \leq x_0^{-\sigma} \int_{L_2} e^{-\Re \phi(u)} du =
x_0^{1-\sigma} \int_0^{\pi/2} e^{-\left(- \frac{x_0^2 \cos 2 \alpha}{2} + \ell x_0 
\sin \alpha + \tau \alpha\right)} d\alpha.\end{equation}
Now, for $\alpha\geq 0$ and $\ell \leq 0$,
\[\left(\ell x_0 
\sin \alpha + \tau \alpha\right)' = \ell x_0 \cos \alpha
+\tau \geq \ell x_0 + \tau.\]
Since $j = \sqrt{1+\rho^2} \leq 1 + \rho^2/2$, we have $\sqrt{(j-1)/2} \leq
\rho/2$, and so, by (\ref{eq:damherr}), $|\ell x_0|\leq \ell^2 \rho/4 = 
\tau$, and thus
$\ell x_0 + \tau\geq 0$. In other words, the exponent in (\ref{eq:sutherl})
equals $(x_0^2 \cos 2\alpha)/2$ minus an increasing function, and so,
since $\Re \phi(x_0) = - x_0^2/2$,
\[\begin{aligned}I_2 &\leq x_0^{-\sigma} \cdot
 x_0
\int_0^{\pi/2} e^{\frac{x_0^2 \cos 2\alpha}{2}} d\alpha = 
x_0^{-\sigma} \cdot
 \frac{\pi}{2} x_0 \cdot I_0(x_0^2/2),
\end{aligned}\]
where $I_0(t) = \frac{1}{\pi} \int_0^\pi e^{t \cos \theta} d\theta$ is the modified
Bessel function of the first kind (and order $0$).

Since $\cos \theta = \sqrt{1-\sin^2 \theta} < 1 - (\sin^2 \theta)/2 \leq
1 - 2 \theta^2/\pi^2$, we have\footnote{It is actually not hard
to prove rigorously the better bound $I_0(t)\leq 0.468823 e^t/\sqrt{t}$.
For $t\geq 8$, this can be done directly by the change of variables
$\cos \theta = 1 - 2 s^2$, $d\theta = 2 ds/\sqrt{1-s^2}$, followed by the usage
of different upper bounds on the the integrand $\exp(-2 t s^2/\sqrt{1-s^2})$ 
for $0\leq s\leq 1/2$ and
$1/2\leq s\leq 1$. (Thanks are due G. Kuperberg for this argument.)
For $t<8$, use the Taylor expansion of $I_0(t)$ around $t=0$ \cite[(9.6.12)]{MR0167642}: truncate it after $16$ terms, and then bound the maximum of the 
truncated series by the bisection method, implemented via 
interval arithmetic (as described in \S \ref{sec:koloko}).}
\[\begin{aligned}
I_0(t) \leq \frac{1}{\pi} \int_0^\pi e^{t \left(1 - \frac{2 \theta^2}{\pi^2}\right)} d\theta < e^t \cdot \frac{1}{\pi} \int_0^\infty e^{-\frac{2 t}{\pi^2} \theta^2} d\theta= e^t \frac{\pi/\sqrt{2 t}}{\pi} \frac{\sqrt{\pi}}{2} = \frac{\sqrt{\pi}}{2^{3/2}} \frac{e^t}{\sqrt{t}}
\end{aligned}\]
for $t\geq 0$.

Using the fact that $\Re \phi(x_0) = - x_0^2/2$, we conclude that
\[\begin{aligned}
I_2 \leq x_0^{-\sigma} \cdot \frac{\pi}{2} x_0 \cdot \frac{\sqrt{\pi}}{2^{3/2}}
\frac{e^{x_0^2/2}}{x_0/\sqrt{2}}
= \frac{\pi^{3/2}}{4} x_0^{-\sigma}
 e^{-\Re \phi(x_0)}.\end{aligned}\]
By (\ref{eq:apprec}), which is valid for all $\ell$, we know that
$\Re \phi(x_0) \geq \Re \phi(u_{0,+})$.

Let us now estimate the integral on $L_3$. 
Again by (\ref{eq:bankal}), for $y<0$,
\[\Re \phi(i y) = \frac{y^2}{2} - \ell y  + \tau \frac{\pi}{2}.\]
Hence
\[\begin{aligned}
&\left|\int_{L_3} e^{-\phi(u)} u^{-\sigma} du \right|\leq
x_0^{-\sigma} \int_{-\infty}^{- x_0} e^{-\left(\frac{y^2}{2} - \ell y + \tau \frac{\pi}{2}\right)} du \\&= x_0^{-\sigma} e^{\frac{1}{2} \ell^2} e^{-\frac{\tau \pi}{2}} 
\int_{-\infty}^{-x_0} e^{-\frac{1}{2} (y-\ell)^2} dy
= x_0^{-\sigma} e^{-\frac{\tau \pi}{2}} \sqrt{\frac{\pi}{2}},
\end{aligned}\]
since $y-\ell\leq -\ell$ for $y\leq - x_0$ and $\int_{-\infty}^{-\ell}
e^{-t^2/2} dt \leq \sqrt{\pi/2} \cdot e^{-\ell^2/2}$ (by \cite[7.1.13]{MR0167642}).

Now that we have bounded the integrals over $L_1$, $L_2$ and $L_3$, 
it remains to bound $x_0$ from below, starting from (\ref{eq:damherr}).
We will bound it differently for $\rho<3/2$ and for $\rho\geq 3/2$. (The
choice of $3/2$ is fairly arbitrary.) 

Expanding $(\sqrt{1+t}-1)^2>0$, we obtain that $2(1+t)-2\sqrt{1+t}\geq t$
for all $t\geq -1$, and so 
\[\left(\frac{\sqrt{1+t}-1}{t}\right)' = \frac{1}{t^2}
\left(\frac{t}{2 \sqrt{1+t}} - (\sqrt{1+t} - 1)\right) < 0,\]
i.e., $(\sqrt{1+t}-1)/t$ decreases as $t$ increases. Hence, for 
$\rho\leq \rho_0$, where $\rho_0\geq 0$,
\begin{equation}\label{eq:malgrer}
j(\rho) = \sqrt{1+\rho^2} \geq 1 + \frac{\sqrt{1+\rho_0^2}-1}{\rho_0^2} \rho^2,
\end{equation}
which equals 
$1 + (2/9) (\sqrt{13}-2) \rho^2$ for $\rho_0 = 3/2$.
Thus, for $\rho \leq 3/2$,
\begin{equation}\label{eq:blondie1}\begin{aligned}x_0 &\geq \frac{|\ell|}{2} \sqrt{\frac{\frac{2}{9} (\sqrt{13}-2) \rho^2}{2}}
= \frac{\sqrt{\sqrt{13}-2}}{6} |\ell| \rho \\ &=
 \frac{2 \sqrt{\sqrt{13}-2}}{3} \frac{\tau}{|\ell|} \geq 0.84473 \frac{|\tau|}{\ell}.
\end{aligned}\end{equation}

On the other hand,
\[\begin{aligned}
\left(\frac{j(\rho)-1}{\rho}\right)' &= \frac{1}{\rho^2} \left(j'(\rho) \rho
- (j(\rho)-1)\right) = \frac{\rho^2 - (1+\rho^2) + \sqrt{1+\rho^2}}{\rho^2 \sqrt{1+\rho^2}}\geq 0, \end{aligned}\]
and so, for $\rho\geq 3/2$, $(j(\rho)-1)/\rho$ is minimal at $\rho=3/2$, where
it takes the value $(\sqrt{13}-2)/3$. Hence
\begin{equation}\label{eq:blondie2}
x_0 = \frac{|\ell|}{2} \sqrt{\frac{j(\rho)-1}{2}} \geq
\frac{|\ell| \sqrt{\rho}}{2} \frac{\sqrt{\sqrt{13}-2}}{\sqrt{6}}
= \frac{\sqrt{\sqrt{13}-2}}{\sqrt{6}} \sqrt{\tau} \geq 0.51729 \sqrt{\tau}.
\end{equation}

We now sum $I_1$, $I_2$ and $I_3$, and then
use (\ref{eq:melis}); we obtain that, when $\ell<0$ and $\tau\geq 0$,
\begin{equation}\label{eq:chanmayo}\begin{aligned}
|F_\delta(s)|&\leq \frac{e^{-2\pi^2 \delta^2} |\Gamma(s)|}{\sqrt{2 \pi}}
\left|\int_L e^{-\phi(u)} u^{-\sigma} du\right| 
\\&\leq 
|x_0|^{-\sigma} \left(\left(1 + \frac{\pi}{2^{3/2}}\right)
 e^{-\Re \phi(u_{0,+})} +
\frac{1}{2} e^{-\frac{\tau \pi}{2}}\right) e^{-\frac{1}{2} \ell^2}
|\Gamma(s)|.
\end{aligned}\end{equation}
By (\ref{eq:petpan}), (\ref{eq:damherr}) and (\ref{eq:phal}),
\[- \Re(\phi(u_{0,+})) = \frac{\ell^2}{4} (1 - \upsilon(\rho)) + 
\frac{\tau}{2} \arccos \frac{1}{\upsilon(\rho)} < \frac{\tau}{2} \arccos
\frac{1}{\upsilon(\rho)} \leq \frac{\pi}{4} \tau.\]
We conclude that, when $\sgn(\ell)\ne \sgn(\tau)$ (i.e.,
$\sgn(\delta) = \sgn(\tau)$),
\[|F_\delta(s)|\leq
|x_0|^{-\sigma} \cdot e^{-\frac{1}{2} \ell^2} |\Gamma(s)| e^{\frac{\pi}{2}|\tau|}
\cdot \left(\left(1 + \frac{\pi}{2^{3/2}}\right) e^{-\frac{\pi}{4} |\tau|} + 
\frac{1}{2}  e^{-\pi |\tau|}\right),\]
where $x_0$ can be bounded as in (\ref{eq:blondie1}) and (\ref{eq:blondie2}).
Here, as before, we reducing the case $\tau<0$ to the case $\tau>0$ by
reflection. This concludes the proof of Theorem \ref{thm:princo}.

\section{Conclusions}\label{sec:abulof}

We have obtained bounds on $|F_\delta(s)|$ for $\sgn(\delta)\ne \sgn(\tau)$
(\ref{eq:huj}) and for $\sgn(\delta) = \sgn(\tau)$ (\ref{eq:chanmayo}).
Our task is now to simplify them. 

\begin{figure}\label{fig:erho}
                \centering \includegraphics[height=2in]{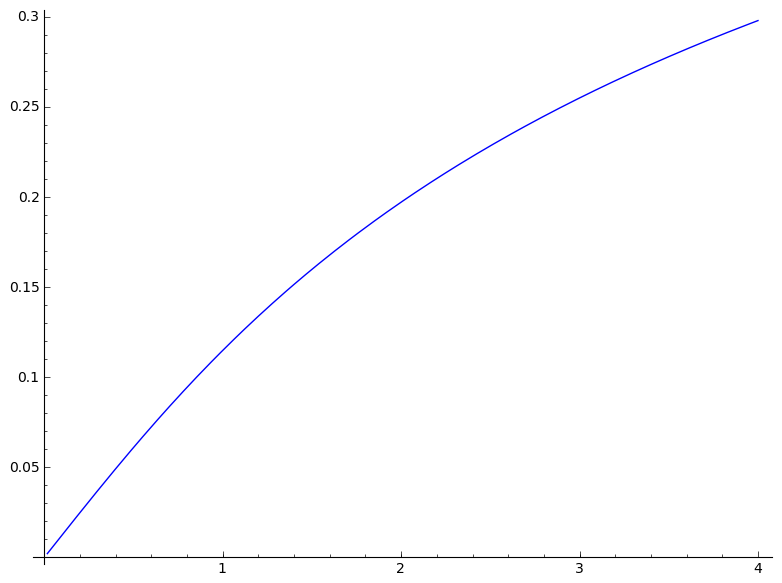}
\caption{The function $E(\rho)$}
\end{figure}

First, let us look at the exponent $E(\rho)$, defined as in
(\ref{eq:cormo}).
 Its plot can be
seen in Figure \ref{fig:erho}. We claim that
\begin{equation}\label{eq:dorof}
E(\rho) \geq \begin{cases} 0.1598 &\text{if $\rho\geq 1.5$,}\\
0.1065 \rho &\text{if $\rho< 1.5$.}\end{cases}\end{equation}
This is so for $\rho\geq 1.5$ because $E(\rho)$ is increasing on $\rho$
and $E(1.5) = 0.15982\dotsc$. The case $\rho<1.5$ is a little more delicate.
We can easily see that $\arccos(1-t^2/2)\geq t$ for $0\leq t\geq 2$ (since the
derivative of the left side is $1/\sqrt{1-t^2/4}$, which is always
$\geq 1$). We also have
\[1+\frac{\rho^2}{2} - \frac{\rho^4}{8} \leq j(\rho) \leq 1 +
\frac{\rho^2}{2}\]
for $0\leq \rho\leq \sqrt{8}$, and so
\[1 + \frac{\rho^2}{8} - \frac{5 \rho^4}{128}    \leq \upsilon(\rho)\leq
1 + \frac{\rho^2}{8}\]
for $0\leq \rho\leq \sqrt{32/5}$; this, in turn, gives us that $1/\upsilon(\rho)\leq
1 - \rho^2/8+7 \rho^4/128$ (again for $0\leq \rho\leq \sqrt{32/5}$), and so
$1/\upsilon(\rho)\leq 1 - (1-7/64) \rho^2/8$ for $0\leq \rho \leq 1/2$. 
We conclude that
\[\arccos \frac{1}{\upsilon(\rho)} \geq \frac{1}{2} \sqrt{\frac{57}{64}}
\rho;\]
therefore,
\[E(\rho) \geq \frac{1}{4} \sqrt{\frac{57}{64}} \rho - \frac{\rho}{8}
> 0.11093 \rho > 0.1065 \rho.\]

In the remaining range $1/2\leq \rho \leq 3/2$, we prove that 
$E(\rho)/\rho > 0.106551$ using the bisection method (with $20$ iterations) 
implemented by means of interval arithmetic.
This concludes the proof of (\ref{eq:dorof}).

Assume from this point onwards that $|\tau|\geq 20$.
Let us show that the contribution of (\ref{eq:octop}) is
negligible relative to that of (\ref{eq:wilen}). Indeed,
\[\left(\left(1 + \frac{\pi}{2^{3/2}}\right) e^{-\frac{\pi}{4} |\tau|} +
 \frac{1}{2} e^{-\pi |\tau|}\right) \leq 
\frac{7.8}{10^6} e^{-0.1598 \tau}.
\]
It is useful to note that $e^{-\ell^2/2} = e^{-2 \tau/\rho}$, and so,
for $\sigma\leq k+1$ and $\rho\leq 3/2$,
\begin{equation}\begin{aligned}
\frac{e^{-2\tau/\rho}}{(0.84473 |\tau|/\ell)^\sigma} &\leq
\frac{e^{-40/\rho}}{\left(\frac{0.84473}{4} \rho \right)^\sigma \ell^{\sigma}} \leq
\frac{1}{\ell^\sigma}
\left(\frac{4}{0.84473\cdot 1.5}\right)^\sigma \frac{e^{-80/(3 t)}}{t^\sigma}\\
&\leq \frac{1}{\ell^\sigma} \cdot 3.15683^{k+1} \frac{e^{-80/(3 t)}}{t^{k+1}}
,\end{aligned}\end{equation}
where $t = 2 \rho/3 \leq 1$. Since $e^{-c/t}/t^{k+1}$ attains its maximum at
$t = c/(k+1)$, 
\[\frac{e^{-80/(3 t)}}{t^{k+1}} \leq e^{-(k+1)} \left(\frac{3 (k+1)}{80}\right)^{k+1},\] and so, for $\rho\leq 3/2$,
\[|x_0|^{-\sigma} e^{-\frac{1}{2} \ell^2} \leq 
\frac{1}{\ell^{\sigma}} \cdot
 \begin{cases}
0.04355 & \text{if $0\leq \sigma \leq 1$},\\
0.00759 & \text{if $1\leq \sigma \leq 2$},\\
0.00224 & \text{if $2\leq \sigma \leq 3$,}\end{cases}
\]
whereas $|x_0|^{-\sigma} e^{-\ell^2/2} \leq |x_0|^{-\sigma} \leq
(0.51729 \sqrt{\tau})^{-\sigma}$ for $\rho\geq 3/2$.

We conclude that, for $|\tau|\geq 20$ and $\sigma\leq 3$,
\begin{equation}\label{eq:auste}|F_\delta(s)|\leq 
|\Gamma(s)| e^{\frac{\pi}{2} \tau} \cdot e^{-0.1598 \tau} \cdot
\begin{cases}
\frac{4}{10^7} \frac{1}{\ell^\sigma} & \text{if $\rho \leq 3/2$,}\\
\frac{6}{10^5} \frac{1}{\tau^{\sigma/2}} &\text{if $\rho\geq 3/2$}
\end{cases}\end{equation}
provided that $\sgn(\delta)=\sgn(\tau)$ or $\delta=0$. 
This will indeed be negligible compared to our bound for
the case $\sgn(\delta) = -\sgn(\tau)$.

Let us now deal with the factor $|\Gamma(s)| e^{\frac{\pi}{2} \tau}$.
By Stirling's formula with remainder term 
\cite[(8.344)]{MR1243179},
\[\log \Gamma(s) = \frac{1}{2} \log(2\pi) + \left(s-\frac{1}{2}\right)
\log s - s + \frac{1}{12 s} + R_2(s),\]
where \[|R_2(s)|< \frac{1/30}{12 |s|^3 \cos^3 \left(\frac{\arg s}{2}\right)}
= \frac{\sqrt{2}}{180 |s|^3}
\]
for $\Re(s)\geq 0$. The real part of $(s-1/2) \log s - s$ is 
\[(\sigma-1/2) \log |s| - \tau \arg(s) - \sigma =
(\sigma-1/2) \log |s| - \frac{\pi}{2} \tau + \tau \left(\arctan \frac{\sigma}{|\tau|}
- \frac{\sigma}{|\tau|}\right)\]
for $s= \sigma+ i \tau$, $\sigma\geq 0$. Since $\arctan(r)\leq r$ for $r\geq 0$,
we conclude that
\begin{equation}\label{eq:sitirlo}
|\Gamma(s)| e^{\frac{\pi}{2} \tau} \leq
\sqrt{2\pi} |s|^{\sigma-\frac{1}{2}} e^{\frac{1}{12 |s|} + \frac{\sqrt{2}}{180 |s|^3}}.
\end{equation}
Lastly, $|s|^{\sigma-1/2} = |\tau|^{\sigma-1/2} |1+i \sigma/\tau|^{\sigma-1/2}$.
For $|\tau|\geq 20$,
\[|1 + i \sigma/\tau|^{\sigma-1/2} \leq
\begin{cases}
1.000625 & \text{if $0\leq \sigma\leq 1$,}\\
1.007491 & \text{if $1\leq \sigma\leq 2$,}\\
1.028204 & \text{if $2\leq \sigma\leq 3$}
\end{cases}\]
and
\[e^{\frac{1}{12 |\tau|} + \frac{\sqrt{2}}{180 |\tau|^3}} \leq 1.004177.\]
Thus,
\begin{equation}\label{eq:janaka}
|\Gamma(s)| e^{\frac{\pi}{2} \tau} \leq
|\tau|^{\sigma-1/2} \cdot
\begin{cases}
2.51868 & \text{if $0\leq \sigma\leq 1$,}\\
2.53596 & \text{if $1\leq \sigma\leq 2$,}\\
2.5881 & \text{if $2\leq \sigma\leq 3$.}
\end{cases}
\end{equation}

Let us now estimate the constants $c_{1,\sigma,\tau}$ and $c_{2,\sigma,\tau}$
in (\ref{eq:cormo}). By $|\tau|\geq 20$,
\begin{equation}\label{eq:kol1}
e^{-\left(\frac{\sqrt{2}-1}{2}\right) \tau} \leq 0.015889,\;\;\;\;
e^{-\frac{\tau}{6}} \leq 0.035674.\end{equation}
Since $8 \sin(\pi/8) = 3.061467\dotsc > 1$, we obtain that
\[c_{1,\sigma,\tau} \leq \begin{cases}
1.30454 & \text{if $0\leq \sigma\leq 1$,}\\
1.58361 & \text{if $1\leq \sigma\leq 2$,}\\
1.98186 & \text{if $2\leq \sigma\leq 3$,}
\end{cases}\;\;\;\;\;\;\;\;
c_{2,\sigma,\tau} \leq \begin{cases}
1.94511 & \text{if $0\leq \sigma\leq 1$,}\\
3.15692 & \text{if $1\leq \sigma\leq 2$,}\\
5.02186 & \text{if $2\leq \sigma\leq 3$.}
\end{cases}
\]
Lastly, note that, for $k\leq \sigma\leq k+1$, we have
\[\frac{1}{\tau^{\sigma/2}} \cdot
|\tau|^{\sigma-1/2} = |\tau|^{(\sigma-1)/2}\leq \tau^{k/2},\]
whereas, for $\rho\leq 3/2$ and $0\leq \gamma\leq 1$,
\[\frac{|\tau|^{\gamma-1/2}}{|\ell|^{\gamma}}\leq 
|\tau|^{\frac{\gamma}{2} - \frac{1}{2}}
\left(\frac{\tau}{\ell^2}\right)^{\gamma/2}\leq
20^{\frac{\gamma}{2} - \frac{1}{2}} \left(\frac{3/2}{4}\right)^{\gamma/2}
\leq \left(\frac{3}{8}\right)^{1/2},\]
and so
\[
\frac{1}{\ell^{\sigma}} \cdot
|\tau|^{\sigma-1/2} = \left(\frac{|\tau|}{\ell}\right)^k 
\frac{|\tau|^{\{\sigma\}-1/2}}{|\ell|^{\{\sigma\}}} \leq \sqrt{\frac{3}{8}}\cdot
 \left(\frac{|\tau|}{\ell}\right)^k.
\]

Multiplying, and remembering to add (\ref{eq:auste}), we obtain that,
for $k=0,1,2$,
$\sigma\in \lbrack 0, 1\rbrack$ and $|\tau|\geq 20$,
\[|F_\delta(s+k)|+ |F_\delta((1-s)+k)| \leq
\begin{cases}
\kappa_{k,0} \left(\frac{|\tau|}{|\ell|}\right)^k
e^{-0.1065 \left(\frac{2 |\tau|}{|\ell|}\right)^2}
&\text{if $\rho< 3/2$,}\\
\kappa_{k,1} |\tau|^k
e^{- 0.1598 |\tau|}
 &\text{if $\rho\geq 3/2$,}
\end{cases}
\]
where
\[\begin{aligned}
\kappa_{0,0} &\leq (4\cdot 10^{-7}+1.94511
)\cdot 2.51868 \cdot \sqrt{3/8} \leq 3.001
,\\
\kappa_{1,0} &\leq (4\cdot 10^{-7}+3.15692)\cdot 2.53596 \cdot \sqrt{3/8} \leq
4.903,\\
\kappa_{2,0} &\leq (4\cdot 10^{-7}+5.02186) \cdot 2.5881 \cdot \sqrt{3/8}\leq 
7.96,\end{aligned}
\]
and, similarly,
\[\begin{aligned}
\kappa_{0,1} &\leq (6\cdot 10^{-5}+1.30454)\cdot 2.51868 \leq 3.286
 ,\\
\kappa_{1,1} &\leq (6\cdot 10^{-5}+1.58361)\cdot 2.53596 \leq 4.017
 ,\\
\kappa_{2,1} &\leq (6\cdot 10^{-5}+1.98186)\cdot 2.5881 \leq 5.13
.\end{aligned}
\]
This concludes the proof of Corollary \ref{cor:amanita1}.

\chapter{Explicit formulas}

An 
{\em explicit formula} is an expression restating a sum such as
$S_{\eta,\chi}(\delta/x,x)$ as a sum of the Mellin transform
$G_\delta(s)$ over the zeros of the $L$ function $L(s,\chi)$. 
More specifically, for us, $G_\delta(s)$ is the Mellin transform
of $\eta(t) e(\delta t)$ for some smoothing function $\eta$ and some
$\delta\in \mathbb{R}$. We want a formula whose error terms are good
both for $\delta$ very close or equal to $0$ and for $\delta$ farther away
from $0$. (Indeed, our choice(s) of $\eta$ will be made so that
$F_\delta(s)$ decays rapidly in both cases.)

We will be able to base all of our work on a single, general explicit formula,
namely, Lemma \ref{lem:agamon}. This explicit formula has simple error terms
given purely in terms of a few norms of the given smoothing function $\eta$.
We also give a common framework for
estimating the contribution of zeros on the critical strip (Lemmas
\ref{lem:garmola} and \ref{lem:hausierer}).

The first example we work out is that of the Gaussian smoothing
$\eta(t) = e^{-t^2/2}$. We actually do this in part for didactic purposes
and in part because of its likely applicability elsewhere; for
our applications, we will always use smoothing functions based on 
$t e^{-t^2/2}$ and $t^2 e^{-t^2/2}$, 
generally in combination with something else. Since
$\eta(t) = e^{-t^2/2}$ does not vanish at $t=0$, its Mellin transform
has a pole at $s=0$ -- something that requires some additional work (Lemma
\ref{lem:povoso}; see also the proof of Lemma \ref{lem:agamon}).

Other than that, for each function $\eta(t)$, all that has to be done
is to bound an integral (from Lemma \ref{lem:garmola}) and bound a few
norms. Still, both for $\eta_*$ and for $\eta_+$, we find a few interesting
complications. Since $\eta_+$ is defined in terms of a truncation of
a Mellin transform (or, alternatively, in terms of a multiplicative
convolution with a Dirichlet kernel, as in (\ref{eq:patra2}) and
(\ref{eq:dirich2})), bounding the norms of $\eta_+$ and $\eta_+'$ takes
a little work. We leave this to Appendix \ref{app:norsmo}. The effect
of the convolution is then just to delay the decay a shift, in that
a rapidly decaying function $f(\tau)$ will get replaced by
$f(\tau-H)$, $H$ a constant. 

The smoothing function
$\eta_*$ is defined as a multiplicative convolution of $t^2 e^{-t^2/2}$
with something else. Given that we have an explicit formula for
$t^2 e^{-t^2/2}$, we obtain an explicit formula for $\eta_*$ by what amounts
to just exchanging the order of a sum and an integral.
(We already went over this in the introduction, in (\ref{eq:asco}).)


\section{A general explicit formula}\label{subs:genexpf}

We will prove an explicit formula valid whenever the smoothing $\eta$
and its derivative $\eta'$ satisfy rather mild assumptions -- they will
be assumed to be $L_2$-integrable and to have strips of definition containing
$\{s: 1/2\leq \Re(s)\leq 3/2\}$, though any strip of the form
$\{s: \epsilon\leq \Re(s)\leq 1 + \epsilon\}$ would do just as well.

(For explicit formulas with different sets of assumptions, see, e.g.,
\cite[\S 5.5]{MR2061214} and \cite[Ch. 12]{MR2378655}.)

The main idea in deriving any explicit formula is to start with an expression
giving a sum as integral over a vertical line with an integrand involving
a Mellin transform (here, $G_\delta(s)$)
and an $L$-function (here, $L(s,\chi)$). We then
shift the line of integration to the left. If stronger assumptions
were made (as in Exercise 5 in \cite[\S 5.5]{MR2061214}), we could shift
the integral all the way to $\Re(s) = -\infty$; the integral would then 
disappear, replaced entirely by a sum over zeros (or even, as in the same
Exercise 5, by a particularly simple integral). Another possibility is to 
shift the line
only to $\Re(s) = 1/2+\epsilon$ for some $\epsilon>0$ -- but 
this gives a weaker result, and at any rate the factor
$L'(s,\chi)/L(s,\chi)$ can be large and messy to estimate 
within the critical strip $0<\Re(s)<1$.

Instead, we will shift the line to $\Re s = -1/2$. We can do this because
the assumptions on $\eta$ and $\eta'$ are enough to continue $G_\delta(s)$
analytically up to there (with a possible pole at $s=0$). 
The factor $L'(s,\chi)/L(s,\chi)$ is easy to estimate
for $\Re s < 0$ and $s=0$ (by the functional equation), and the part of
the integral on $\Re s = -1/2$ coming from $G_\delta(s)$ can be estimated 
easily using the fact that the Mellin transform is an isometry.

\begin{lemma}\label{lem:agamon}
 Let $\eta:\mathbb{R}_0^+\to \mathbb{R}$ be in $C^1$.
Let $x\in \mathbb{R}^+$, $\delta \in \mathbb{R}$. 
Let $\chi$ be a primitive character mod $q$, $q\geq 1$. 

Write $G_\delta(s)$ for the Mellin transform of $\eta(t) e(\delta t)$.
Assume that $\eta(t)$ and $\eta'(t)$ are in $\ell_2$ (with respect
to the measure $dt$) and that $\eta(t) t^{\sigma-1}$ and $\eta'(t) t^{\sigma-1}$
are in $\ell_1$ (again with respect to $dt$) for all $\sigma$ in an open 
interval containing $\lbrack 1/2,3/2 \rbrack$.

Then
\begin{equation}\label{eq:marmar}\begin{aligned}
\sum_{n=1}^\infty &\Lambda(n) \chi(n) e\left(\frac{\delta}{x} n\right) \eta(n/x) =
I_{q=1} \cdot \widehat{\eta}(-\delta) x - \sum_\rho G_\delta(\rho) x^\rho \\ &- R +
O^*\left(
(\log q + 6.01) \cdot
(|\eta'|_2 + 2 \pi |\delta| |\eta|_2 )\right) x^{-1/2}
,\end{aligned}\end{equation}
where \begin{equation}\label{eq:estromo}\begin{aligned}
I_{q=1} &= \begin{cases} 1 & \text{if $q=1$,} \\ 0
&\text{if $q\ne 1$,}\end{cases}\\
R &=  \eta(0) \left(\log \frac{2 \pi}{q} + \gamma -
\frac{L'(1,\chi)}{L(1,\chi)}\right) + O^*(c_0)\end{aligned}\end{equation}
for $q>1$, $R= \eta(0) \log 2\pi$ for $q=1$ and
\begin{equation}\label{eq:marenostrum}c_0= \frac{2}{3} O^*\left(
\left|\frac{\eta'(t)}{\sqrt{t}}\right|_1 + \left|\eta'(t) \sqrt{t}\right|_1 + 
2 \pi |\delta| \left(\left|\frac{\eta(t)}{\sqrt{t}}\right|_1 + |\eta(t) \sqrt{t}|_1\right)
\right).\end{equation}
The norms $|\eta|_2$, $|\eta'|_2$, $|\eta'(t)/\sqrt{t}|_1$, etc., 
are taken with respect to the usual
measure $dt$.
The sum $\sum_\rho$ is a sum over all non-trivial zeros $\rho$
of  $L(s,\chi)$.
\end{lemma}
\begin{proof}
Since (a)
$\eta(t) t^{\sigma-1}$ is in $\ell_1$ for $\sigma$ in an open interval
containing $3/2$ and (b) $\eta(t) e(\delta t)$ 
has bounded variation (since $\eta, \eta'\in \ell_1$, implying that the
derivative of $\eta(t) e(\delta t)$ is also in $\ell_1$),
 the Mellin inversion formula (as in, e.g., \cite[4.106]{MR2061214}) holds:
\[\eta(n/x) e(\delta n/x) 
= \frac{1}{2\pi i} \int_{\frac{3}{2} - i\infty}^{\frac{3}{2} + i \infty}
G_\delta(s) x^s n^{-s} ds.\]
Since $G_\delta(s)$ is bounded for $\Re(s)=3/2$ 
(by $\eta(t) t^{3/2-1} \in \ell_1$) and 
$\sum_n \Lambda(n) n^{-3/2}$ is bounded as well, 
we can change the order of summation and 
integration as follows:
\begin{equation}\label{eq:sartai}\begin{aligned}
\sum_{n=1}^\infty \Lambda(n) \chi(n) e(\delta n/x) \eta(n/x) 
&= \sum_{n=1}^\infty \Lambda(n) \chi(n) \cdot
\frac{1}{2\pi i} \int_{\frac{3}{2} - i\infty}^{\frac{3}{2} + i \infty}
G_\delta(s) x^s n^{-s} ds\\
&= \frac{1}{2 \pi i} \int_{\frac{3}{2} - i\infty}^{\frac{3}{2} + i \infty}
 \sum_{n=1}^\infty \Lambda(n) \chi(n)  G_\delta(s) x^s n^{-s} ds\\
&= \frac{1}{2 \pi i}
\int_{\frac{3}{2} - i \infty}^{\frac{3}{2} + i \infty} - \frac{L'(s,\chi)}{L(s,\chi)}
G_\delta(s) x^s ds.\end{aligned}\end{equation}
(This is the way the procedure always starts: see, for instance,
 \cite[Lemma 1]{MR1555183} or, to look at a recent standard reference, 
\cite[p. 144]{MR2378655}. We are being very scrupulous about integration 
 because we are working with general $\eta$.)

The first question we should ask ourselves is: up to where can we extend
$G_\delta(s)$? Since $\eta(t) t^{\sigma-1}$ is in $\ell_1$ for $\sigma$ in an
open interval $I$ containing $\lbrack 1/2,3/2\rbrack$, the transform
$G_\delta(s)$ is defined for $\Re(s)$ in the same interval $I$. However,
we also know that the transformation rule $M(t f'(t))(s) = -s\cdot M f(s)$ (see
(\ref{eq:harva}); by integration by parts) is valid when $s$ is in
the holomorphy
strip for both $M(t f'(t))$ and $M f$. In our case 
($f(t) = \eta(t) e(\delta t)$), this happens when
$\Re(s) \in (I-1) \cap I$ (so that both sides of the equation in the
 rule are defined). Hence $s\cdot G_\delta(s)$ (which equals $s \cdot M f(s)$) 
can be analytically continued to $\Re(s)$ in $(I-1)\cup I$, 
which is an open interval containing $\lbrack -1/2,3/2\rbrack$.
This implies immediately that $G_\delta(s)$ can be analytically continued to
the same region, with a possible pole at $s=0$.

When does $G_\delta(s)$ have a pole at $s=0$? This happens when $s G_\delta(s)$
is non-zero at $s=0$, i.e., when $M(t f'(t))(0)\ne 0$ for $f(t) = \eta(t) e(\delta t)$. Now
\[M(t f'(t))(0) = \int_0^\infty f'(t) dt = \lim_{t\to \infty} f(t) - f(0).\]
We already know that $f'(t) = (d/dt) (\eta(t) e(\delta t))$ is in $\ell_1$.
Hence, $\lim_{t\to \infty} f(t)$ exists, and must be $0$ because 
$f$ is in $\ell_1$. Hence $- M(t f'(t))(0) = f(0) = \eta(0)$.

Let us look at the next term in the Laurent expansion of $G_\delta(s)$ at
$s=0$. It is
\[\begin{aligned}
\lim_{s\to 0} \frac{s G_\delta(s) - \eta(0)}{s} &= 
\lim_{s\to 0} \frac{- M(t f'(t))(s) - f(0)}{s} = - \lim_{s\to 0}
\frac{1}{s} \int_0^{\infty} f'(t) (t^s-1) dt
\\ &= - \int_0^\infty f'(t) \lim_{s\to 0} \frac{t^s-1}{s} dt = 
-\int_0^\infty f'(t) \log t\; dt.\end{aligned}\]
Here we were able to exchange the limit and the integral because
$f'(t) t^\sigma$ is in $\ell_1$ for $\sigma$ in a neighborhood of $0$;
in turn, this is true because $f'(t) = \eta'(t) + 2\pi i \delta \eta(t)$
and $\eta'(t) t^\sigma$ and $\eta(t) t^\sigma$ are both in $\ell_1$ for
$\sigma$ in a neighborhood of $0$. In fact, we will use the easy bounds
$|\eta(t) \log t|\leq (2/3) (|\eta(t) t^{-1/2}|_1 + |\eta(t) t^{1/2}|_1)$,
$|\eta'(t) \log t|\leq (2/3) (|\eta'(t) t^{-1/2}|_1 + |\eta'(t) t^{1/2}|_1)$,
resulting from the inequality
\begin{equation}\label{eq:hutterite}
\frac{2}{3} \left(t^{-\frac{1}{2}} + t^{\frac{1}{2}}\right) \leq |\log t|,
\end{equation}
valid for all $t>0$.

We conclude that the Laurent expansion of $G_\delta(s)$ at $s=0$ is
\begin{equation}\label{eq:estabba}
G_\delta(s) = \frac{\eta(0)}{s} + c_0 + c_1 s + \dotsc,
\end{equation}
where 
\[\begin{aligned}c_0 &= O^*(|f'(t) \log t|_1)\\ &= \frac{2}{3} O^*\left(
\left|\frac{\eta'(t)}{\sqrt{t}}\right|_1 + \left|\eta'(t) \sqrt{t}\right|_1 + 
2 \pi \delta \left(\left|\frac{\eta(t)}{\sqrt{t}}\right|_1 + |\eta(t) \sqrt{t}|_1\right)
\right).\end{aligned}\]

We shift the line of integration in (\ref{eq:sartai}) to $\Re(s)=-1/2$.
We obtain
\begin{equation}\label{eq:argeri}\begin{aligned}\frac{1}{2\pi i}
\int_{2 - i \infty}^{2 + i \infty} -\frac{L'(s,\chi)}{L(s,\chi)}
G_\delta(s) x^s ds 
&= I_{q=1} G_\delta(1) x - \sum_\rho G_\delta(\rho) x^\rho - R
\\ &- \frac{1}{2\pi i} \int_{-1/2-i\infty}^{-1/2+i\infty}
\frac{L'(s,\chi)}{L(s,\chi)}
G_\delta(s) x^s ds,\end{aligned}\end{equation}
where 
\[R = \Res_{s=0} \frac{L'(s,\chi)}{L(s,\chi)} G_\delta(s).\]
Of course,
\[G_\delta(1) = M(\eta(t) e(\delta t))(1) = \int_0^\infty \eta(t) e(\delta t) dt
= \widehat{\eta}(-\delta).\]

Let us work out the Laurent expansion of $L'(s,\chi)/L(s,\chi)$ at $s=0$.
By the functional equation
(as in, e.g., \cite[Thm. 4.15]{MR2061214}),
\begin{equation}\label{eq:funeq}\begin{aligned}
\frac{L'(s,\chi)}{L(s,\chi)} &= \log \frac{\pi}{q} 
- \frac{1}{2} \psi\left(\frac{s+\kappa}{2}\right) 
- \frac{1}{2} \psi\left(\frac{1-s+\kappa}{2}\right) 
- \frac{L'(1-s,\overline{\chi})}{L(1-s,\overline{\chi})},\end{aligned}
\end{equation} where $\psi(s) = \Gamma'(s)/\Gamma(s)$
and 
\[\kappa= \begin{cases} 0 &\text{if $\chi(-1)=1$}\\
1 &\text{if $\chi(-1)=-1$.}\end{cases}\]
By $\psi(1-x)-\psi(x) = \pi \cot \pi x$ (immediate from
$\Gamma(s) \Gamma(1-s) = \pi/\sin \pi s$) and $\psi(s)+\psi(s+1/2) =
2 (\psi(2s) - \log 2)$ (Legendre; \cite[(6.3.8)]{MR0167642}), 
\begin{equation}\label{eq:gorilo}- \frac{1}{2} \left( \psi \left(\frac{s+\kappa}{2}\right)
 + \psi\left(\frac{1-s+\kappa}{2}\right)\right) = - \psi(1-s) + \log 2 +
\frac{\pi}{2} \cot \frac{\pi (s+\kappa)}{2}.\end{equation}

Hence, unless $q=1$,
the Laurent expansion of $L'(s,\chi)/L(s,\chi)$ at $s=0$ is
\[
\frac{1-\kappa}{s} + \left(\log \frac{2 \pi}{q} - \psi(1) -
\frac{L'(1,\chi)}{L(1,\chi)}\right) 
+ \frac{a_1}{s} + \frac{a_2}{s^2} + \dotsc.
\]
Here $\psi(1) = -\gamma$, the Euler gamma constant \cite[(6.3.2)]{MR0167642}.

There is a special case for $q=1$ due to the pole of $\zeta(s)$ at
$s=1$. We know that $\zeta'(0)/\zeta(0) = \log 2\pi$ 
(see, e.g., \cite[p. 331]{MR2378655}).

From this and (\ref{eq:estabba}), we conclude that, if $\eta(0)=0$, then
\[R = \begin{cases} c_0 &\text{if $q>1$ and $\chi(-1)=1$,}\\
0 & \text{otherwise,}\end{cases}\]
where $c_0 = O^*(|\eta'(t) \log t|_1 +
2 \pi |\delta| |\eta(t) \log t|_1)$. If $\eta(0)\ne 0$, then
\[R = \eta(0) \left(\log \frac{2 \pi}{q} + \gamma -
\frac{L'(1,\chi)}{L(1,\chi)}\right) 
+ \begin{cases} c_0 &\text{if $\chi(-1)=1$}\\
0 & \text{otherwise.}\end{cases}\]
for $q>1$, and
\[R = \eta(0) \log 2 \pi\]
for $q=1$.

It is time to estimate the integral on the right side of (\ref{eq:argeri}).
For that, we will need to estimate $L'(s,\chi)/L(s,\chi)$ for $\Re(s)=-1/2$
using (\ref{eq:funeq}) and (\ref{eq:gorilo}).

If $\Re(z)=3/2$, then $|t^2+z^2|\geq 9/4$ for all real $t$.
Hence, by \cite[(5.9.15)]{MR2723248} and \cite[(3.411.1)]{MR1243179},
\begin{equation}\label{eq:malgach}\begin{aligned}
\psi(z) &= \log z - \frac{1}{2z} - 2 \int_0^\infty
\frac{t dt}{(t^2+ z^2) (e^{2\pi t}-1)} \\ &=
\log z - \frac{1}{2z} + 2\cdot O^*\left(\int_0^\infty
\frac{t dt}{\frac{9}{4} (e^{2\pi t}-1)}\right)\\
&=  \log z - \frac{1}{2z} +  \frac{8}{9} 
O^*\left(\int_0^\infty \frac{t dt}{e^{2\pi t} -1}\right) \\ &= 
 \log z - \frac{1}{2z} +  \frac{8}{9} \cdot
O^*\left(\frac{1}{(2 \pi)^2} \Gamma(2) \zeta(2)\right) \\ &= 
 \log z - \frac{1}{2z} +  O^*\left(\frac{1}{27}\right) = 
\log z + O^*\left(\frac{10}{27}\right).
\end{aligned}\end{equation}
Thus, in particular, $\psi(1-s) = \log(3/2- i\tau) + O^*(10/27)$, where
we write $s = 1/2 + i \tau$. Now
\[\left|\cot \frac{\pi (s+\kappa)}{2}\right| = 
\left|\frac{e^{\mp \frac{\pi}{4} i - \frac{\pi}{2} \tau} + 
e^{\pm \frac{\pi}{4} i + \frac{\pi}{2} \tau}}{
e^{\mp \frac{\pi}{4} i - \frac{\pi}{2} \tau} - 
e^{\pm \frac{\pi}{4} i + \frac{\pi}{2} \tau}}\right| = 1.\]
Since $\Re(s)=-1/2$, a comparison
of Dirichlet series gives
\begin{equation}\label{eq:koloko}
\left|\frac{L'(1-s,\overline{\chi})}{L(1-s,\overline{\chi})}\right|
\leq \frac{|\zeta'(3/2)|}{|\zeta(3/2)|} \leq 1.50524,\end{equation}
where $\zeta'(3/2)$ and $\zeta(3/2)$ can be evaluated by Euler-Maclaurin.
Therefore, (\ref{eq:funeq}) and (\ref{eq:gorilo})
 give us that, for $s = -1/2 + i\tau$,
\begin{equation}\label{eq:peanuts}\begin{aligned}
\left|\frac{L'(s,\chi)}{L(s,\chi)}\right| &\leq
\left| \log \frac{q}{\pi}\right| + 
\log \left|\frac{3}{2} + i \tau\right| + \frac{10}{27} + \log 2 +
\frac{\pi}{2} +1.50524\\
&\leq \left| \log \frac{q}{\pi}\right| + 
\frac{1}{2} \log \left(\tau^2 + \frac{9}{4}\right)  
+ 4.1396.
\end{aligned}\end{equation}

Recall that we must bound the integral on the right side of
(\ref{eq:argeri}). The absolute value of the integral is at most $x^{-1/2}$
times
\begin{equation}\label{eq:lorel}
\frac{1}{2\pi} \int_{-\frac{1}{2}-i \infty}^{- \frac{1}{2} + i \infty} 
\left|\frac{L'(s,\chi)}{L(s,\chi)} G_\delta(s)\right| ds.\end{equation}
By Cauchy-Schwarz, this is at most
\[\sqrt{
\frac{1}{2\pi} 
\int_{-\frac{1}{2}- i \infty}^{-\frac{1}{2} + i \infty} 
\left|\frac{L'(s,\chi)}{L(s,\chi)}\cdot \frac{1}{s}\right|^2
|ds| } \cdot \sqrt{\frac{1}{2\pi} 
\int_{-\frac{1}{2}- i \infty}^{-\frac{1}{2} + i \infty} 
\left|G_\delta(s) s\right|^2 |ds|}\]
By (\ref{eq:peanuts}),
\[\begin{aligned}
\sqrt{\int_{-\frac{1}{2}- i \infty}^{-\frac{1}{2} + i \infty} 
\left|\frac{L'(s,\chi)}{L(s,\chi)}\cdot \frac{1}{s}\right|^2 |ds|}
&\leq  
\sqrt{\int_{-\frac{1}{2}- i \infty}^{-\frac{1}{2} + i \infty} 
\left|\frac{\log q}{s}\right|^2 |ds|} \\&+
\sqrt{\int_{-\infty}^{\infty} 
\frac{\left|
\frac{1}{2} \log\left(\tau^2 + \frac{9}{4}\right) +
    4.1396 + \log \pi\right|^2}{\frac{1}{4} + \tau^2} d\tau}\\
&\leq \sqrt{2\pi} \log q  + \sqrt{226.844},
\end{aligned}\]
where we compute the last integral numerically.\footnote{By a
rigorous integration from $\tau = -100000$ to $\tau = 100000$ 
using VNODE-LP \cite{VNODELP}, which runs
on the PROFIL/BIAS interval arithmetic package \cite{Profbis}.}

Again, we use the fact that,
by (\ref{eq:harva}), $s G_\delta(s)$ is the Mellin transform of
\begin{equation}\label{eq:jamon}
- t \frac{d (e(\delta t) \eta(t))}{dt} = 
- 2 \pi i \delta t e(\delta t) \eta(t) 
- t e(\delta t) \eta'(t)
\end{equation}


Hence, by Plancherel (as in (\ref{eq:victi})),
\begin{equation}\label{eq:chorizo}\begin{aligned}
\sqrt{\frac{1}{2\pi} 
\int_{-\frac{1}{2}- i \infty}^{-\frac{1}{2} + i \infty} 
\left|G_\delta(s) s\right|^2 |ds|} &= 
\sqrt{\int_0^\infty \left|
- 2 \pi i \delta t e(\delta t) \eta(t) 
- t e(\delta t) \eta'(t)
\right|^2 t^{-2} dt} \\ &=
2 \pi |\delta|
\sqrt{\int_0^\infty |\eta(t)|^2 dt} + \sqrt{\int_0^\infty |\eta'(t)|^2 dt}.
\end{aligned}\end{equation}
Thus, (\ref{eq:lorel}) is at most
\[\left(\log q + \sqrt{\frac{226.844}{2 \pi}}\right) \cdot
\left(|\eta'|_2 + 2\pi |\delta| |\eta|_2\right).\]
\end{proof}

Lemma \ref{lem:agamon} leaves us with three tasks: bounding
the sum of $G_\delta(\rho) x^{\rho}$ over all non-trivial zeroes $\rho$
with small
imaginary part, bounding the sum of $G_\delta(\rho) x^{\rho}$ over all
non-trivial zeroes $\rho$ with large imaginary part, and bounding
$L'(1,\chi)/L(1,\chi)$. Let us start with the last task: while, in a narrow
sense, it is optional -- in that, in the applications we actually need
(Thm.~\ref{thm:janar}, Cor.~\ref{cor:coprar} 
and Thm.~\ref{thm:malpor}), we will have
$\eta(0)=0$, thus making the term
$L'(1,\chi)/L(1,\chi)$ disappear -- it is also very easy and can be
dealt with quickly.

 Since we will be using
a finite GRH check in all later applications, we might as well use it here.
\begin{lemma}\label{lem:povoso}
Let $\chi$ be a primitive character mod $q$, $q>1$. 
Assume that all non-trivial zeroes $\rho= \sigma+i t$ of
$L(s,\chi)$ with $|t|\leq 5/8$ satisfy $\Re(\rho) = 1/2$.
Then
\[\left|\frac{L'(1,\chi)}{L(1,\chi)}\right|\leq 
\frac{5}{2} \log M(q) + c,\]
where $M(q) = \max_n \left|\sum_{m\leq n} \chi(m)\right|$ and
\[c = 5 \log \frac{2 \sqrt{3}}{\zeta(9/4)/\zeta(9/8)} = 
15.07016\dotsc.\]
\end{lemma}
\begin{proof}
By a lemma of Landau's (see, e.g., \cite[Lemma 6.3]{MR2378655}, where
the constants are easily made explicit) based on the Borel-Carath\'eodory Lemma
(as in \cite[Lemma 6.2]{MR2378655}), any function $f$ analytic and
zero-free on a disc $C_{s_0,R}=\{s:|s-s_0|\leq R\}$ of radius $R>0$
around $s_0$ satisfies
\begin{equation}\label{eq:landau}
\frac{f'(s)}{f(s)} = O^*\left(\frac{2 R \log M/|f(s_0)|}{(R-r)^2}\right)
\end{equation}
for all $s$ with $|s-s_0|\leq r$, where $0<r<R$ and 
$M$ is the maximum of $|f(z)|$ on $C_{s_0,R}$. Assuming $L(s,\chi)$ has
no non-trivial zeros off the critical line with $|\Im(s)|\leq H$, 
where $H>1/2$, we
set $s_0 = 1/2+H$, $r = H-1/2$, and let $R\to H^-$.
We obtain
\begin{equation}\label{eq:rabamar}
\frac{L'(1,\chi)}{L(1,\chi)} = O^*\left(8 H \log \frac{\max_{s\in C_{s_0,H}} 
|L(s,\chi)|}{|L(s_0,\chi)|}\right).
\end{equation}
Now 
\[|L(s_0,\chi)|\geq \prod_p (1+p^{-s_0})^{-1} = \prod_p
\frac{(1-p^{-2 s_0})^{-1}}{(1-p^{- s_0})^{-1}} = \frac{\zeta(2 s_0)}{\zeta(s_0)}.\]
Since $s_0 = 1/2+H$, $C_{s_0,H}$ is contained
in $\{s\in \mathbb{C}: \Re(s)>1/2\}$ for any value of $H$.
We choose (somewhat arbitrarily) $H=5/8$. 

By partial summation, for $s=\sigma+it$ with $1/2\leq \sigma<1$ and any 
$N\in \mathbb{Z}^+$,
\begin{equation}\label{eq:sellyou}\begin{aligned}
L(s,\chi) &= \sum_{n\leq N} \chi(m) n^{-s} -
\left(\sum_{m\leq N} \chi(m)\right) (N+1)^{-s} \\ &+
 \sum_{n\geq N+1}
 \left(\sum_{m\leq n} \chi(m)\right) (n^{-s}-(n+1)^{-s+1})\\
&= O^*\left(\frac{N^{1-1/2}}{1-1/2} + N^{1-\sigma} +
M(q) N^{-\sigma}\right),\end{aligned}\end{equation}
where $M(q) = \max_n \left|\sum_{m\leq n} \chi(m)\right|$. 
We set $N = M(q)/3$, and obtain
\begin{equation}\label{eq:thecrisis}
|L(s,\chi)| \leq 2 M(q) N^{-1/2} = 2 \sqrt{3} \sqrt{M(q)}.\end{equation}
We put this into (\ref{eq:rabamar}) and are done.
\end{proof}
Let $M(q)$ be as in the statement of Lem.~\ref{lem:povoso}. Since the sum
of $\chi(n)$ ($\chi \mo q$, $q>1$)
over any interval of length $q$ is $0$, it is easy to see that $M(q)\leq
q/2$. We also have the following explicit version of the P\'olya-Vinogradov
inequality:
\begin{equation}\label{eq:karbach}M(q) \leq
\begin{cases} \frac{2}{\pi^2} \sqrt{q} \log q + \frac{4}{\pi^2} \sqrt{q} 
\log \log q + \frac{3}{2} \sqrt{q} & \text{if $\chi(-1)=1$,}\\
\frac{1}{2\pi} \sqrt{q} \log q + \frac{1}{\pi} \sqrt{q} \log \log q + 
\sqrt{q} & \text{if $\chi(-1)=1$.}\end{cases}\end{equation}
Taken together with $M(q)\leq q/2$, this implies that
\begin{equation}\label{eq:wachetauf}M(q)\leq q^{4/5}\end{equation}
for all $q\geq 1$, and also that
\begin{equation}\label{eq:marlo}M(q)\leq 2 q^{3/5}\end{equation}
for all $q\geq 1$.

Notice, lastly, that
\[\left|\log \frac{2\pi}{q} + \gamma\right|\leq 
\log q + \log \frac{e^\gamma \cdot 2 \pi}{3^2}\]
for all $q\geq 3$. (There are no primitive characters modulo $2$, so 
we can omit $q=2$.)

We conclude that, for $\chi$ primitive and non-trivial,
\[\begin{aligned}
\left|\log \frac{2\pi}{q} + \gamma - \frac{L'(1,\chi)}{L(1,\chi)}\right|
&\leq \log \frac{e^\gamma \cdot 2 \pi}{3^2} + \log q + \frac{5}{2} \log q^{\frac{4}{5}} + 15.07017\\
&\leq 3 \log q + 15.289.
\end{aligned}\]
Obviously, $15.289$ is more than $\log 2 \pi$, the bound for $\chi$ trivial.
Hence,
 the absolute value of the quantity $R$ in the statement of Lemma \ref{lem:agamon}
is at most
\begin{equation}\label{eq:bleucol}
|\eta(0)| (3 \log q + 15.289) + |c_0|\end{equation}
for all primitive $\chi$.

It now remains to bound the sum $\sum_{\rho} G_{\delta}(\rho) x^\rho$
in (\ref{eq:marmar}). Clearly
\[\left|\sum_{\rho} G_{\delta}(\rho) x^\rho\right| \leq
\sum_{\rho} \left|G_{\delta}(\rho)\right| \cdot x^{\Re(\rho)}.\]
Recall that these are sums over the non-trivial zeros $\rho$ of $L(s,\chi)$.

We first prove a general lemma on sums of values of functions on the non-trivial
zeros of $L(s,\chi)$. This is little more than partial summation, given
a (classical) bound for the number of zeroes $N(T,\chi)$ of $L(s,\chi)$
with $|\Im(s)|\leq T$. 
The error term becomes particularly simple if $f$ is real-valued and 
decreasing; the statement is then practically identical to that of
\cite[Lemma 1]{MR0202686} (for $\chi$ principal), except for the fact
 that the error term is improved here.
\begin{lemma}\label{lem:garmola}
Let $f:\mathbb{R}^+\to \mathbb{C}$ be piecewise $C^1$. Assume
$\lim_{t\to \infty} f(t) t \log t = 0$. Let $\chi$ be a primitive character
$\mod q$, $q\geq 1$; let $\rho$ denote the non-trivial zeros $\rho$ of 
$L(s,\chi)$.
Then, for any $y\geq 1$,
\begin{equation}\label{eq:jotok}\begin{aligned}
\mathop{\sum_{\text{$\rho$ non-trivial}}}_{\Im(\rho) > y} f(\Im(\rho))
&= \frac{1}{2\pi} \int_y^\infty f(T) \log \frac{q T}{2 \pi} dT\\
&+ 
\frac{1}{2} O^*\left(|f(y)| g_\chi(y) + \int_{y}^\infty \left|f'(T)\right| 
\cdot g_\chi(T)  dT\right),
\end{aligned}\end{equation}
where 
\begin{equation}\label{eq:ertr}
g_\chi(T) =
0.5 \log qT + 17.7\end{equation}

If $f$ is real-valued and decreasing on $\lbrack y,\infty)$, the second line of (\ref{eq:jotok})
equals
\[O^*\left(\frac{1}{4} \int_y^\infty \frac{f(T)}{T} dT\right).\]
\end{lemma}
\begin{proof}
Write $N(T,\chi)$ for the number of non-trivial zeros of $L(s,\chi)$
satisfying $|\Im(s)|\leq T$. 
Write 
$N^+(T,\chi)$ for the number of (necessarily non-trivial) zeros of $L(s,\chi)$
with $0 < \Im(s)\leq T$.
Then, for any $f:\mathbb{R}^+\to \mathbb{C}$ with
$f$ piecewise differentiable and $\lim_{t\to \infty} f(t) N(T,\chi) = 0$,
\[\begin{aligned}
\sum_{\rho: \Im(\rho)> y} f(\Im(\rho)) &= \int_{y}^\infty f(T)\; dN^+(T,\chi)
\\ &= - \int_{y}^\infty f'(T) (N^+(T,\chi) - N^+(y,\chi)) dT\\
&= - \frac{1}{2} \int_{y}^\infty f'(T) (N(T,\chi) - N(y,\chi)) dT
.\end{aligned}\]
Now, by \cite[Thms. 17--19]{MR0003018} and \cite[Thm. 2.1]{MR726004}
(see also \cite[Thm. 1]{Trudgian}),
\begin{equation}\label{eq:melos}
N(T,\chi) = \frac{T}{\pi} \log \frac{q T}{2\pi e} + O^*\left(g_{\chi}(T)\right)
\end{equation}
for $T\geq 1$, where $g_\chi(T)$ is as in (\ref{eq:ertr}). (This is a classical
formula; the references serve to prove the explicit form (\ref{eq:ertr}) 
for the error term $g_\chi(T)$.)

Thus, for $y\geq 1$,
\begin{equation}\label{eq:bochal}\begin{aligned}
\sum_{\rho: \Im(\rho)>y} f(\Im(\rho)) 
&= - \frac{1}{2}
\int_{y}^\infty f'(T) \left(\frac{T}{\pi} \log \frac{q T}{2\pi e} -
\frac{y}{\pi} \log \frac{q y}{2\pi e}\right) dT
\\ &+ \frac{1}{2} O^*\left(|f(y)| g_\chi(y) + \int_{y}^\infty \left|f'(T)\right| 
\cdot g_\chi(T) dT\right)
.\end{aligned}\end{equation}
Here
\begin{equation}\label{eq:linstrau}- \frac{1}{2}
\int_{y}^\infty f'(T) \left(\frac{T}{\pi} \log \frac{q T}{2\pi e}
- \frac{y}{\pi} \log \frac{q y}{2\pi e}\right) dT =
\frac{1}{2 \pi}
\int_{y}^\infty f(T) \log \frac{q T}{2\pi} dT.
\end{equation}
If $f$ is real-valued and decreasing (and so, by $\lim_{t\to \infty} f(t)=0$,
non-negative), 
 \[\begin{aligned}|f(y)| g_\chi(y)  + \int_{y}^\infty \left|f'(T)\right| 
\cdot g_\chi(T)  dT &= f(y) g_\chi(y) - 
\int_{y}^\infty f'(T) g_\chi(T) dT \\
&= 0.5 \int_{y}^\infty \frac{f(T)}{T} dT,
\end{aligned}\]
since $g_\chi'(T) \leq 0.5/T$ for all $T\geq T_0$.
\end{proof}

Let us bound the part of the sum $\sum_\rho G_\delta(\rho) x^\rho$ corresponding
to $\rho$ with bounded $|\Im(\rho)|$. The bound we will
give is proportional to $\sqrt{T_0} \log q T_0$, whereas a very naive approach
(based on the trivial bound $|G_\delta(\sigma + i \tau)|\leq |G_0(\sigma)|$)
would give a bound proportional to $T_0 \log q T_0$.

We could obtain a bound proportional to $\sqrt{T_0} \log q T_0$ for
$\eta(t) = t^k e^{-t^2/2}$ by using
Theorem \ref{thm:princo}. Instead, we will give a bound of that same
quality valid for 
$\eta$ essentially arbitrary simply by using the fact that the Mellin transform
is an isometry (preceded by an application of Cauchy-Schwarz). 

\begin{lemma}\label{lem:hausierer}
Let $\eta:\mathbb{R}_0^+\to \mathbb{R}$ be such that
both $\eta(t)$ and $(\log t) \eta(t)$ lie in $L_1\cap L_2$ 
and $\eta(t)/\sqrt{t}$ lies in $L_1$ (with respect to $dt$). 
Let $\delta\in \mathbb{R}$. Let $G_\delta(s)$ be 
the Mellin transform of $\eta(t) e(\delta t)$. 

Let $\chi$ be a primitive character mod $q$, $q\geq 1$. Let $T_0\geq 1$. 
Assume that all non-trivial zeros $\rho$ of $L(s,\chi)$ with
$|\Im(\rho)|\leq T_0$ lie on the critical line.
Then
\[
\mathop{\sum_{\text{$\rho$ non-trivial}}}_{|\Im(\rho)|\leq T_0} \left|G_{\delta}(\rho)\right| 
\]
is at most
\begin{equation}\label{eq:arnbax}
\begin{aligned} (|\eta|_2 + |\eta\cdot \log|_2) &\sqrt{T_0} \log q T_0 + 
(17.21 |\eta\cdot \log|_2 - (\log 2\pi \sqrt{e}) |\eta|_2)
\sqrt{T_0}\\
&+ \left|\eta(t)/\sqrt{t}\right|_1\cdot (1.32 \log q + 34.5)
\end{aligned}\end{equation}
\end{lemma}
\begin{proof}
For $s = 1/2 + i \tau$, we have the trivial bound
\begin{equation}\label{eq:sandunga}
|G_\delta(s)| \leq \int_0^\infty |\eta(t)| t^{1/2} \frac{dt}{t} =
\left|\eta(t)/\sqrt{t}\right|_1
,\end{equation}
where $F_\delta$ is as in (\ref{eq:guason}). We also have the trivial bound
\begin{equation}\label{eq:candonga}
|G_\delta'(s)| = \left|\int_0^\infty (\log t) \eta(t) t^s \frac{dt}{t}\right|
\leq \int_0^\infty |(\log t) \eta(t)| t^{\sigma} \frac{dt}{t} =
\left|(\log t) \eta(t) t^{\sigma-1}\right|_1
\end{equation}
for $s = \sigma+ i \tau$.

Let us start by bounding the contribution of very low-lying zeros
($|\Im(\rho)|\leq 1$). 
By (\ref{eq:melos}) and (\ref{eq:ertr}),
\[N(1,\chi) = \frac{1}{\pi} \log \frac{q}{2 \pi e} + O^*\left(
0.5 \log q + 17.7\right) = O^*(0.819 \log q + 16.8).\]
Therefore,
\[\mathop{\sum_{\text{$\rho$ non-trivial}}}_{|\Im(\rho)|\leq 1} 
\left|G_{\delta}(\rho)\right|
\leq \left|\eta(t) t^{- 1/2}\right|_1\cdot
(0.819 \log q + 16.8).\]

Let us now consider zeros $\rho$ with $|\Im(\rho)|>1$.
Apply Lemma \ref{lem:garmola} with $y=1$ and
\[f(t) = \begin{cases}
\left|G_\delta(1/2 + i t)\right| &\text{if $t\leq T_0$},\\
0 &\text{if $t> T_0$}.\end{cases}\]
This gives us that
\begin{equation}\label{eq:datora}\begin{aligned}
\sum_{\rho: 1<|\Im(\rho)|\leq T_0} f(\Im(\rho)) &= \frac{1}{\pi}
\int_1^{T_0} f(T) \log \frac{q T}{2 \pi} dT 
\\ &+ 
O^*\left(|f(1)| g_\chi(1) + \int_{1}^\infty |f'(T)|\cdot g_\chi(T) \; dT\right),
\end{aligned}\end{equation}
where we are using the fact that $f(\sigma + i\tau) = f(\sigma-i\tau)$
(because $\eta$ is real-valued).
By Cauchy-Schwarz,
\[
\frac{1}{\pi}
\int_1^{T_0} f(T) \log \frac{q T}{2 \pi} dT \leq
\sqrt{\frac{1}{\pi} \int_1^{T_0} |f(T)|^2 dT} \cdot 
 \sqrt{\frac{1}{\pi} \int_1^{T_0} \left(\log \frac{q T}{2 \pi}\right)^2
   dT}.\]
Now
\[\begin{aligned}
 \frac{1}{\pi} \int_1^{T_0} |f(T)|^2 dT &\leq 
\frac{1}{2 \pi } \int_{-\infty}^{\infty} \left|G_\delta\left(\frac{1}{2}+iT\right)\right|^2 dT
\leq \int_0^\infty |e(\delta t) \eta(t)|^2 dt = |\eta|_2^2
\end{aligned}\]
by Plancherel (as in (\ref{eq:victi})).
 We also have
\[\int_1^{T_0} \left(\log \frac{q T}{2 \pi}\right)^2 dT \leq
\frac{2 \pi}{q} \int_0^{\frac{q T_0}{2\pi}} (\log t)^2 dt
\leq \left(\left(\log \frac{q T_0}{2 \pi e}\right)^2 + 1 \right) \cdot
T_0.\]
Hence \[\frac{1}{\pi}
\int_1^{T_0} f(T) \log \frac{q T}{2 \pi} dT \leq
\sqrt{\left(\log \frac{q T_0}{2 \pi e}\right)^2 + 1} \cdot
|\eta|_2 \sqrt{T_0}.\]

Again by Cauchy-Schwarz,
\[\begin{aligned}
\int_{1}^\infty |f'(T)|&\cdot g_\chi(T) \; dT \leq
\sqrt{\frac{1}{2\pi} \int_{-\infty}^\infty |f'(T)|^2 dT}\cdot 
\sqrt{\frac{1}{\pi} \int_1^{T_0} |g_{\chi}(T)|^2
  dT}.\end{aligned}\]
Since $|f'(T)| = |G_\delta'(1/2+i T)|$ and $(M\eta)'(s)$ is
the Mellin transform of $\log(t)\cdot e(\delta t) \eta(t)$
(by (\ref{eq:harva})),
\[\frac{1}{2\pi} \int_{-\infty}^\infty |f'(T)|^2 dT =
|\eta(t) \log(t)|_2.\] 
Much as before,
\[\begin{aligned}
\int_1^{T_0} |g_{\chi}(T)|^2 dT &\leq \int_0^{T_0} (0.5 \log q T +17.7)^2
dT \\&= (0.25 (\log q T_0)^2 + 17.2 (\log q T_0) + 296.09) T_0.\end{aligned}\]

Summing, we obtain 
\[\begin{aligned}\frac{1}{\pi}
\int_1^{T_0} &f(T) \log \frac{q T}{2 \pi} dT  + 
\int_{1}^\infty |f'(T)|\cdot g_\chi(T) \; dT\\
&\leq \left(
\left(\log \frac{q T_0}{2 \pi e} + \frac{1}{2}\right) |\eta|_2+
 \left(\frac{\log q T_0}{2} + 17.21\right) |\eta(t) (\log t)|_2\right)
\sqrt{T_0}
\end{aligned}\]

Finally, by (\ref{eq:sandunga}) and (\ref{eq:ertr}),
\[|f(1)| g_\chi(1) \leq 
\left|\eta(t)/\sqrt{t}\right|_1\cdot (0.5 \log q + 17.7).\]
By (\ref{eq:datora}) and the assumption that all non-trivial zeros with
$|\Im(\rho)|\leq T_0$ lie
on the line $\Re(s)=1/2$, we conclude that
\[\begin{aligned}\mathop{\sum_{\text{$\rho$ non-trivial}}}_{1 < |\Im(\rho)|\leq T_0} 
\left|G_{\delta}(\rho)\right| &\leq 
(|\eta|_2 + |\eta\cdot \log|_2) \sqrt{T_0} \log q T_0 \\ &+ 
(17.21 |\eta\cdot \log|_2 - (\log 2\pi \sqrt{e}) |\eta|_2)
\sqrt{T_0}\\
&+ \left|\eta(t)/\sqrt{t}\right|_1\cdot (0.5 \log q + 17.7)
.\end{aligned}\] 
\end{proof}

All that remains is to bound the contribution to $\sum_\rho G_\delta(\rho)
x^\rho$ corresponding to all zeroes $\rho$ with $|\Im(\rho)|>T_0$. This
will do by another application of Lemma \ref{lem:garmola}, combined
with bounds on $G_\delta(\rho)$ for $\Im(\rho)$ large. This is the only
part that will require us to take a look at the actual smoothing function
$\eta$ we are working with; it is at this point, not before, that we actually
have to look at each of our options for $\eta$ one by one.

\section{Sums and decay for the Gaussian}

It is now time to derive our bounds for the Gaussian smoothing.
As we were saying,
 there is really only one thing
left to do, namely, an estimate for the sum
$\sum_\rho |F_\delta(\rho)|$ over all zeros $\rho$ with $|\Im(\rho)|>T_0$.

\begin{lemma}\label{lem:garmonas}
Let $\eta_\heartsuit(t) = e^{-t^2/2}$. Let $x\in \mathbb{R}^+$, $\delta\in \mathbb{R}$.
Let $\chi$ be a primitive character mod $q$, $q\geq 1$. 
Assume that all non-trivial zeros $\rho$ of $L(s,\chi)$
with $|\Im(\rho)|\leq T_0$
satisfy $\Re(s)=1/2$. Assume that $T_0\geq 50$.

Write $F_\delta(s)$ for the Mellin transform of $\eta(t) e(\delta t)$. Then
\[\begin{aligned}
\mathop{\sum_{\text{$\rho$ }}}_{|\Im(\rho)|>T_0} 
\left|F_{\delta}(\rho)\right| \leq 
\log \frac{q T_0}{2 \pi} \cdot
\left(3.53 e^{-0.1598 T_0} + 
22.5 \frac{\delta^2}{T_0} e^{-0.1065 \left(\frac{T_0}{\pi |\delta|}\right)^2}\right).
\end{aligned}\]
\end{lemma}
Here we have preferred to give a bound with a simple form.
It is probably feasible to derive
from Theorem \ref{thm:princo}
 a bound
essentially proportional to $e^{-E(\rho) T_0}$, where 
$\rho = T_0/(\pi \delta)^2$ and
$E(\rho)$ is as in (\ref{eq:cormo}). (As we discussed in 
\S \ref{sec:abulof}, $E(\rho)$ behaves as
$e^{-(\pi/4) T_0}$ for $\rho$ large and as $e^{-0.125 (T_0/(\pi
  \delta))^2}$ for $\rho$ small.)
\begin{proof}
First of all,
\[\mathop{\sum_{\text{$\rho$ }}}_{|\Im(\rho)|>T_0} 
\left|F_{\delta}(\rho)\right| = 
\mathop{\sum_{\text{$\rho$ }}}_{\Im(\rho)>T_0} 
\left(\left|F_{\delta}(\rho)\right| + 
\left|F_{\delta}(1-\rho)\right|\right)
,\]
by the functional equation (which implies that non-trivial
zeros come in pairs $\rho$, $1-\rho$). Hence, by a somewhat
brutish application of Cor.~\ref{cor:amanita1},
\begin{equation}\label{eq:gargara}
\mathop{\sum_{\text{$\rho$ }}}_{|\Im(\rho)|>T_0} 
\left|F_{\delta}(\rho)\right| \leq
\mathop{\sum_{\text{$\rho$ }}}_{\Im(\rho)>T_0}
f(\Im(\rho)),\end{equation}
where
\begin{equation}\label{eq:darmo}
f(\tau) = 
3.001
e^{-0.1065 \left(\frac{\tau}{\pi \delta}\right)^2} +
3.286 e^{-0.1598 |\tau|}.
\end{equation}
Obviously, $f(\tau)$ 
is a decreasing function of $\tau$ for $\tau\geq T_0$.

We now apply Lemma \ref{lem:garmola}. We obtain 
that
\begin{equation}\label{eq:happin}
\mathop{\sum_{\text{$\rho$ }}}_{\Im(\rho)>T_0} f(\Im(\rho))
\leq \int_{T_0}^\infty f(T) \left(\frac{1}{2\pi}  \log \frac{q T}{2\pi}
+ \frac{1}{4 T}\right) dT.\end{equation}

We just need to estimate some integrals. For any $y\geq 1$, $c,c_1>0$,
\[\begin{aligned}
\int_y^\infty \left(\log t + \frac{c_1}{t}\right) e^{-c t} dt
&\leq \int_y^\infty \left(\log t - \frac{1}{c t}\right) e^{-c t} dt +
\left(\frac{1}{c}+c_1\right) 
\int_y^\infty \frac{e^{-c t}}{t} dt\\
&= \frac{(\log y) e^{-c y}}{c} +   
\left(\frac{1}{c}+c_1\right) E_1(cy),
\end{aligned}\]
where $E_1(x) = \int_x^\infty e^{-t} dt/t$. 
Clearly, $E_1(x)\leq \int_x^\infty e^{-t} dt/x = e^{-x}/x$. Hence
\[\int_y^\infty \left(\log t + \frac{c_1}{t}\right) e^{-c t} dt
\leq \left(\log y +  \left( \frac{1}{c} + c_1\right) \frac{1}{y}
 \right) \frac{e^{-c y}}{c}.
\] We conclude that
\begin{equation}\label{eq:vysyvat}\begin{aligned}
\int_{T_0}^\infty e^{-0.1598 t} &\left(\frac{1}{2\pi} \log \frac{q t}{2\pi} +
\frac{1}{4t}\right) dt\\ &\leq \frac{1}{2\pi} \int_{T_0}^\infty 
\left(\log t + \frac{\pi/2}{t}\right) e^{-ct} dt + \frac{\log \frac{q}{2\pi}}{2\pi c} 
\int_{T_0}^\infty e^{-c t} dt\\
&= \frac{1}{2\pi c} \left(\log T_0 + \log \frac{q}{2\pi} +
\left(\frac{1}{c} + \frac{\pi}{2}\right)
\frac{1}{T_0}\right) e^{-c T_0} 
\end{aligned}\end{equation}
with $c=0.1598$. Since $T_0\geq 50$ and $q\geq 1$, this is at most
\begin{equation}\label{eq:haroldo}1.072 \log \frac{q T_0}{2\pi} e^{-c T_0}.
\end{equation}

Now let us deal with the Gaussian term. (It appears only if
$T_0 < (3/2) (\pi \delta)^2$, as otherwise 
$|\tau|\geq (3/2) (\pi \delta)^2$ holds whenever $|\tau|\geq T_0$.)
For any $y\geq e$, $c\geq 0$,
\begin{equation}\label{eq:analo1}\int_{y}^\infty e^{-c t^2} dt = \frac{1}{\sqrt{c}} \int_{\sqrt{c} y}^\infty
e^{-t^2} dt \leq \frac{1}{c y} \int_{\sqrt{c} y}^\infty t e^{-t^2} dt \leq
\frac{e^{-c y^2}}{2 c y},
\end{equation}
\begin{equation}\label{eq:analo2}\int_y^\infty  \frac{e^{-c t^2}}{t} dt =
\int_{c y^2}^\infty \frac{e^{-t}}{2 t} dt = 
 \frac{E_1(c y^2)}{2} \leq \frac{e^{-c y^2}}{2 c y^2},
\end{equation}
\begin{equation}\label{eq:analo3}
\int_{y}^\infty (\log t) e^{-c t^2} dt
\leq \int_y^\infty \left(\log t + \frac{\log t - 1}{2 c t^2}\right) e^{-c t^2}
dt = \frac{\log y}{2 c y} e^{-c y^2}.
\end{equation} Hence
\begin{equation}\label{eq:milstei}
\begin{aligned}\int_{T_0}^\infty &e^{- 0.1065 \left(\frac{T}{\pi \delta}\right)^2}
\left(\frac{1}{2\pi} \log \frac{q T}{2 \pi} + \frac{1}{4 T}\right) dT
\\&= \int_{\frac{T_0}{\pi |\delta|}}^\infty e^{-0.1065 t^2} 
\left(\frac{|\delta|}{2} \log \frac{q |\delta| t}{2} + \frac{1}{4 t}\right)  dt\\
&\leq \left(\frac{\frac{|\delta|}{2} \log \frac{T_0}{\pi |\delta|}}{2 c' \frac{T_0}{\pi |\delta|}} 
+ \frac{\frac{|\delta|}{2} \log \frac{q |\delta|}{2}}{2 c' \frac{T_0}{\pi |\delta|}}+ \frac{1}{8 c' \left(\frac{T_0}{\pi |\delta|}\right)^2}\right) 
e^{-c' \left(\frac{T_0}{\pi |\delta|}\right)^2} 
\end{aligned}\end{equation}
with $c'=0.1065$. Since $T_0\geq 50$ and $q\geq 1$,
\[
\frac{2 \pi}{8 T_0} \leq 
\frac{\pi}{200} 
\leq 0.0152 \cdot \frac{1}{2} \log \frac{q T_0}{2 \pi} 
\] 
Thus, the last line of (\ref{eq:milstei}) is less than 
\begin{equation}\label{eq:certeru}
1.0152 \frac{\frac{|\delta|}{2} \log \frac{q T_0}{2 \pi}}{\frac{2 c' T_0}{\pi |\delta|}}
e^{-c' \left(\frac{T_0}{\pi |\delta|}\right)^2} = 
7.487 \frac{\delta^2}{T_0} \cdot \log \frac{q T_0}{2\pi}\cdot e^{-c' \left(\frac{T_0}{\pi |\delta|}\right)^2}
.\end{equation}
Again by $T_0\geq 4\pi^2 |\delta|$, we see that
$1.0057 \pi |\delta|/(4 c T_0) \leq 1.0057/(16 c \pi) \leq 0.18787$. 

To obtain our final bound, we simply sum (\ref{eq:haroldo}) and
(\ref{eq:certeru}), after multiplying them by the constants 
$3.286$ and  $3.001$ in (\ref{eq:darmo}).
We conclude that the integral in (\ref{eq:happin}) is at most
\[\left(3.53 e^{-0.1598 T_0} + 
22.5 \frac{\delta^2}{T_0} e^{-0.1065 \left(\frac{T_0}{\pi |\delta|}\right)^2}\right)
\log \frac{q T_0}{2 \pi}.
\]
\end{proof}

We need to record a few norms related to the Gaussian 
$\eta_\heartsuit(t)=e^{-t^2/2}$ before we proceed. Recall we are working
with the one-sided Gaussian, i.e., we set $\eta_\heartsuit(t)=0$ for $t<0$.
Symbolic integration then gives
\begin{equation}\label{eq:simphard}\begin{aligned}
|\eta_\heartsuit|_2^2 &= \int_0^\infty e^{-t^2} dt = \frac{\sqrt{\pi}}{2},\\
|\eta_\heartsuit'|_2^2 &= \int_0^\infty (t e^{-t^2/2})^2 dt = \frac{\sqrt{\pi}}{4},\\
|\eta_\heartsuit \cdot \log|_2^2 &= \int_0^\infty e^{-t^2} (\log t)^2 dt \\ &=
 \frac{\sqrt{\pi}}{16}
\left(\pi^2 + 2 \gamma^2 + 8 \gamma \log 2 + 8 (\log 2)^2\right) \leq 1.94753,
\end{aligned}\end{equation}
\begin{equation}\label{eq:simphard2}\begin{aligned}
|\eta_\heartsuit(t)/\sqrt{t}|_1 &= \int_0^\infty \frac{e^{-t^2/2}}{\sqrt{t}} dt = 
\frac{\Gamma(1/4)}{2^{3/4}} \leq 2.15581\\
|\eta_\heartsuit'(t)/\sqrt{t}| = |\eta_\heartsuit(t) \sqrt{t}|_1 &= \int_0^\infty e^{-\frac{t^2}{2}} \sqrt{t} dt = 
\frac{\Gamma(3/4)}{2^{1/4}} \leq 1.03045\\
\left|\eta_\heartsuit'(t) t^{1/2}\right|_1  = 
\left|\eta_\heartsuit(t) t^{3/2}\right|_1 
&= \int_0^\infty e^{- \frac{t^2}{2}}
t^{\frac{3}{2}} dt = 1.07791
.\end{aligned}\end{equation}

We can now state what is really our main result for the Gaussian smoothing.
(The version in \S \ref{subs:results} will, as we shall later see, follow from
this, given numerical inputs.)

\begin{prop}\label{prop:bargo}
Let $\eta(t) = e^{-t^2/2}$.
 Let $x\geq 1$, $\delta\in \mathbb{R}$.
Let $\chi$ be a primitive character mod $q$, $q\geq 1$. 
Assume that all non-trivial zeros $\rho$ of $L(s, \chi)$
with $|\Im(\rho)|\leq T_0$
lie on the critical line.
Assume that $T_0\geq 50$.

 Then
\begin{equation}
\sum_{n=1}^\infty \Lambda(n) \chi(n) e\left(\frac{\delta}{x} n\right) 
\eta\left(\frac{n}{x}\right) =
\begin{cases} \widehat{\eta}(-\delta) x + 
O^*\left(\err_{\eta,\chi}(\delta,x)\right)\cdot x
&\text{if $q=1$,}\\
O^*\left(\err_{\eta,\chi}(\delta,x)\right)\cdot x
&\text{if $q>1$,}\end{cases}\end{equation}
where
\[\begin{aligned}
\err_{\eta,\chi}(\delta,x) &= 
\log \frac{q T_0}{2 \pi} \cdot
\left(3.53 e^{-0.1598 T_0} + 
22.5 \frac{\delta^2}{T_0} e^{-0.1065 \left(\frac{T_0}{\pi |\delta|}\right)^2}\right)\\
&+
(2.337  \sqrt{T_0} \log q T_0 + 21.817 \sqrt{T_0} + 2.85 \log q + 74.38)
x^{-\frac{1}{2}}\\
&+ (3 \log q + 14 |\delta| + 17) x^{-1} +
(\log q + 6) \cdot
(1 + 5 |\delta|)\cdot x^{-3/2}.
\end{aligned}\]
\end{prop}
\begin{proof}
Let $F_\delta(s)$ be the Mellin transform of $\eta_\heartsuit(t) e(\delta t)$.
By Lemmas \ref{lem:hausierer} (with $G_\delta=F_\delta$) and Lemma
\ref{lem:garmonas},
\[\left|\sum_{\text{$\rho$ non-trivial}} F_\delta(\rho) x^\rho\right|\]
is at most (\ref{eq:arnbax}) (with $\eta=\eta_\heartsuit$) times $\sqrt{x}$, plus
\[\log \frac{q T_0}{2 \pi} \cdot
\left(3.53 e^{-0.1598 T_0} + 
22.5 \frac{|\delta|^2}{T_0} e^{-0.1065 \left(\frac{T_0}{\pi |\delta|}\right)^2}\right) \cdot x.
\]
By the norm computations in (\ref{eq:simphard}) and (\ref{eq:simphard2}), we see that
(\ref{eq:arnbax}) is at most
\[
2.337  \sqrt{T_0} \log q T_0 + 21.817 \sqrt{T_0} + 2.85 \log q + 74.38.
\]

Let us now apply Lemma \ref{lem:agamon}. We saw that the value of $R$
in Lemma \ref{lem:agamon} is bounded by (\ref{eq:bleucol}). 
We know that $\eta_\heartsuit(0)=1$. Again by (\ref{eq:simphard}) and
(\ref{eq:simphard2}), the quantity $c_0$ defined in (\ref{eq:marenostrum}) 
is at most
$1.4056 + 13.3466 |\delta|$. Hence
\[|R| \leq 3 \log q + 13.347 |\delta| + 16.695.\]
Lastly,
\[|\eta_\heartsuit'|_2 + 2\pi |\delta| |\eta_\heartsuit|_2 
\leq 0.942 + 4.183 |\delta| \leq 1 + 5 |\delta|.\]
 Clearly
\[(6.01-6)\cdot (1+5 |\delta|) + 13.347 |\delta| + 16.695 < 14 |\delta| + 17,\]
and so we are done. 
\end{proof}

\section{The case of $\eta_*(t)$}
We will now work with a weight based on the Gaussian:
\begin{equation}\label{eq:zvedzka}
\eta(t) = \begin{cases} t^2 e^{-t^2/2} &\text{if $t\geq 0$,}\\
0 &\text{if $t<0$.}\end{cases}\end{equation}
The fact that this vanishes at $t=0$ actually makes it easier to work with
at several levels.

Its Mellin transform is just a shift of that of the Gaussian. Write
 \begin{equation}\label{eq:guason}
\begin{aligned}
F_\delta(s) &= (M( e^{- \frac{t^2}{2}} e(\delta t)))(s),\\
G_\delta(s) &= (M(\eta(t) e(\delta t)))(s).\end{aligned}\end{equation}
Then, by the definition of the Mellin transform,
\[G_\delta(s) = F_\delta(s+2).\]


We start by bounding the contribution of zeros with large imaginary part,
just as before.

\begin{lemma}\label{lem:festavign}
Let $\eta(t) = t^2 e^{-t^2/2}$. Let $x\in \mathbb{R}^+$, $\delta\in \mathbb{R}$.
Let $\chi$ be a primitive character mod $q$, $q\geq 1$. 
Assume that all non-trivial zeros $\rho$ of $L(s,\chi)$
with $|\Im(\rho)|\leq T_0$
satisfy $\Re(s)=1/2$. Assume that 
$T_0\geq \max(10 \pi |\delta|,50)$.

Write $G_\delta(s)$ for the Mellin transform of $\eta(t) e(\delta t)$. Then
\[\begin{aligned}
\mathop{\sum_{\text{$\rho$ }}}_{|\Im(\rho)|>T_0} 
\left|G_{\delta}(\rho)\right| \leq 
T_0 \log \frac{q T_0}{2\pi} \cdot \left( 
6.11 e^{-0.1598 T_0} + 
1.578 e^{-0.1065 \cdot \frac{T_0^2}{(\pi \delta)^2}}\right)
.\end{aligned}\]
\end{lemma}
\begin{proof}
We start by writing
\[\mathop{\sum_{\text{$\rho$ }}}_{|\Im(\rho)|>T_0} 
\left|G_{\delta}(\rho)\right| = 
\mathop{\sum_{\text{$\rho$ }}}_{\Im(\rho)>T_0} 
\left(\left|F_{\delta}(\rho+2)\right| + 
\left|F_{\delta}((1-\rho)+2)\right|\right)
,\]
where we are using $G_\delta(\rho) = F_\delta(\rho+2)$ and the fact that
non-trivial zeros come in pairs $\rho$, $1-\rho$.

By Cor.~\ref{cor:amanita1} with $k=2$,
\[
\mathop{\sum_{\text{$\rho$ }}}_{|\Im(\rho)|>T_0} 
\left|G_{\delta}(\rho)\right| \leq \mathop{\sum_{\text{$\rho$ }}}_{\Im(\rho)>T_0} f(\Im(\rho)),
\]
where 
\begin{equation}\label{eq:asgo}
f(\tau) = \begin{cases}
\kappa_{2,1} |\tau| e^{-0.1598 |\tau|} + 
\frac{\kappa_{2,0}}{4} \left(\frac{|\tau|}{\pi \delta}\right)^2 e^{-
0.1065 \left(\frac{|\tau|}{\pi \delta}\right)^2} 
& \text{if $|\tau| < \frac{3}{2} (\pi \delta)^2$,}\\ 
\kappa_{2,1} |\tau| e^{-0.1598 |\tau|}
& \text{if $|\tau| \geq \frac{3}{2} (\pi \delta)^2$,}
\end{cases}
\end{equation}
where $\kappa_{2,0} = 7.96$ and $\kappa_{2,1} = 5.13$.
We are including the term $|\tau| e^{-0.1598 |\tau|}$ in both cases
in part because we cannot be bothered to take it out (just as we could
 not be bothered
in the proof of Lem.~\ref{lem:garmonas}) and in part to ensure that
$f(\tau)$ is a decreasing function of $\tau$ for $\tau\geq T_0$.

We can now apply Lemma \ref{lem:garmola}. We obtain, again,
\begin{equation}
\mathop{\sum_{\text{$\rho$ }}}_{\Im(\rho)>T_0} f(\Im(\rho))
\leq \int_{T_0}^\infty f(T) \left(\frac{1}{2\pi}  \log \frac{q T}{2\pi}
+ \frac{1}{4 T}\right) dT.
\end{equation}
Just as before, we will need to estimate some integrals.

For any $y\geq 1$, $c,c_1>0$ such that $\log y > 1/(c y)$,
\[\int_{y}^\infty t e^{-ct} dt 
= \left(\frac{y}{c}+\frac{1}{c^2}\right) e^{-c y},\]
\begin{equation}\label{eq:vrashchenie}\begin{aligned}
\int_{y}^\infty \left(t \log t + \frac{c_1}{t}\right) e^{-ct} dt 
&\leq \int_{y}^\infty \left(\left(t + \frac{a-1}{c}\right) \log t 
-\frac{1}{c} - \frac{a}{c^2 t}\right) e^{-ct} dt \\
&= \left(\frac{y}{c} + \frac{a}{c^2}\right) e^{- c y} \log y,
\end{aligned}\end{equation}
where \[a = \frac{\frac{\log y}{c} + \frac{1}{c} + \frac{c_1}{y}}{
\frac{\log y}{c} - \frac{1}{c^2 y}}.\]
Setting $c = 0.1598$, $c_1 = \pi/2$, $y = T_0\geq 50$, we
obtain that 
\begin{equation}\label{eq:movon}\begin{aligned}
\int_{T_0}^\infty &\left(\frac{1}{2\pi} \log \frac{q T}{2 \pi} + 
\frac{1}{4 T}\right) T e^{-0.1598 T} dT\\
&\leq \frac{1}{2\pi} \left(\log \frac{q}{2 \pi} \cdot \left(\frac{T_0}{c}
+ \frac{1}{c^2}\right) + 
\left(\frac{T_0}{c} + \frac{a}{c^2}\right) \log T_0\right)
e^{- 0.1598 T_0}
\end{aligned}\end{equation}
and
\[a= \frac{\frac{\log T_0}{0.1598} + \frac{1}{0.1598} + \frac{\pi/2}{T_0}}{
\frac{\log T_0}{0.1598} - \frac{1}{0.1598^2 T_0}}\leq 1.299.\]
It is easy to see that ratio of 
the expression within parentheses on the right side of (\ref{eq:movon})
to $T_0 \log(q T_0/2\pi)$ increases as $q$ decreases and, if we hold $q$ fixed,
decreases as $T_0\geq 2\pi$ increases; thus, it is maximal for $q=1$ and
$T_0 = 50$. Multiplying (\ref{eq:movon})
by $\kappa_{2,1}=5.13$ and simplifying by the assumption 
$T_0\geq 50$, 
we obtain that
\begin{equation}\label{eq:jokr}
\int_{T_0}^\infty 5.13 T e^{-0.1598 T} \left(\frac{1}{2\pi}
\log \frac{q T_0}{2\pi} + \frac{1}{4 T}\right) dT \leq
6.11 T_0 \log \frac{q T_0}{2\pi} \cdot e^{-0.1598 T_0}.\end{equation}


Now let us examine the Gaussian term.
First of all -- when does it arise? If $T_0\geq (3/2) (\pi \delta)^2$,
then $|\tau|\geq (3/2) (\pi \delta)^2$ holds whenever $|\tau|\geq T_0$,
and so (\ref{eq:asgo}) does not give us a Gaussian term. 
Recall that $T_0\geq 10 \pi |\delta|$,
which means that $|\delta|\leq 20/(3\pi)$ implies that 
$T_0\geq (3/2) (\pi \delta)^2$.
We can thus assume from now on that $|\delta|>20/(3\pi)$, since otherwise there is
no Gaussian term to treat.

For any $y\geq 1$, $c,c_1>0$,
\[\int_y^\infty t^2 e^{-c t^2} dt <
\int_y^\infty \left(t^2 + \frac{1}{4 c^2 t^2}\right) e^{- c t^2}
dt = \left(\frac{y}{2 c} + \frac{1}{4 c^2 y}\right)\cdot e^{- c y^2},
\]
\[\begin{aligned}
\int_y^\infty (t^2 \log t + c_1 t) \cdot e^{-c t^2} dt &\leq 
\int_y^\infty \left(t^2 \log t + \frac{a t\log e t}{2 c} - \frac{\log e t}{2 c}
- \frac{a}{4 c^2 t} \right) e^{-c t^2} dt \\&= 
 \frac{(2 c y + a) \log y + a}{4 c^2} \cdot e^{- c y^2},
\end{aligned}\]
where
\[a = 
\frac{c_1 y + \frac{\log ey}{2 c} 
}{\frac{y \log e y}{2 c} - \frac{1}{4
  c^2 y}} =
\frac{1}{y} +
\frac{c_1 y + \frac{1}{4 c^2 y^2}}{\frac{y \log e y}{2 c} - \frac{1}{4
  c^2 y}} 
=\frac{1}{y} + \frac{2 c_1 c}{\log e y} + 
\frac{\frac{c_1}{2 c y \log e y} + \frac{1}{4 c^2 y^2}}{\frac{y \log e y}{2 c} - \frac{1}{4
  c^2 y}} .\]             
(Note that $a$ decreases as $y\geq y_0$ increases, provided that
$\log e y_0 > 1/(2 c y_0^2)$.)
Setting $c=0.1065$, $c_1 = 1/(2 |\delta|)\leq 3/16$ and $y = T_0/(\pi |\delta|)
\geq 4 \pi$, we obtain
\[\begin{aligned}
\int_{\frac{T_0}{\pi |\delta|}}^\infty
&\left(\frac{1}{2\pi} \log \frac{q |\delta| t}{2} + \frac{1}{4 \pi
    |\delta| t}\right) t^2 e^{-0.1065 t^2} dt
\\ &\leq \left(\frac{1}{2 \pi} \log \frac{q |\delta|}{2}\right) \cdot 
\left(\frac{T_0}{2 \pi c |\delta|} + \frac{1}{4 c^2 \cdot 10} \right)
\cdot e^{-0.1065 \left(\frac{T_0}{\pi |\delta|}\right)^2}\\
&+ \frac{1}{2 \pi} \cdot \frac{\left(2 c \frac{T_0}{\pi |\delta|} +
    a\right)
\log \frac{T_0}{\pi |\delta|} + a}{4 c^2}
\cdot e^{-0.1065 \left(\frac{T_0}{\pi |\delta|}\right)^2}
\end{aligned}\]
and \[a\leq  \frac{1}{10} + 
\frac{\left(\frac{2\cdot 20}{3\pi}\right)^{-1}\cdot 10 
 + \frac{1}{4\cdot 0.1065^2\cdot 10^2}}{
\frac{10 \log 10 e}{2\cdot 0.1065} -
\frac{1}{4\cdot 0.1065^2 \cdot 10}}\leq 0.117.\]
Multiplying by $(\kappa_{2,0}/4) \pi |\delta|$, we get that 
\begin{equation}\label{eq:skinoad}
\int_{T_0}^\infty \frac{\kappa_{2,0}}{4} \left(\frac{T}{\pi |\delta|}\right)^2
e^{-0.1065 \left(\frac{T}{\pi |\delta|}\right)^2} 
\left(\frac{1}{2\pi} \log \frac{q T_0}{2\pi} + \frac{1}{4 T}\right) 
dT\end{equation}
is at most $e^{-0.1065
\left(\frac{T_0}{\pi |\delta|}\right)^2}$ times
\begin{equation}\label{eq:caplor}\begin{aligned}
&\left((1.487 T_0 + 2.194 |\delta|)
\cdot \log \frac{q |\delta|}{2} +
1.487 T_0 \log \frac{T_0}{\pi |\delta|}
+ 2.566 |\delta| \log \frac{e T_0}{\pi |\delta|} \right)\\
&\leq \left(1.487 +
2.566 \cdot \frac{1 + \frac{1}{\log T_0/\pi |\delta|}}{T_0/|\delta|}\right) T_0 \log \frac{q T_0}{2 \pi} \leq
1.578 \cdot T_0 \log \frac{q T_0}{2\pi},
\end{aligned}\end{equation}
where we are using several times the assumption that
$T_0\geq 4 \pi^2 |\delta|$ (and, in one occasion, the fact that
$|\delta|>20/(3 \pi) > 2$).

We sum (\ref{eq:jokr}) and the estimate for (\ref{eq:skinoad}) we have just got
to reach our conclusion.
\end{proof}

Again, we record some norms obtained by symbolic integration:
for $\eta$ as in (\ref{eq:zvedzka}), \begin{equation}\label{eq:drachcat}
\begin{aligned}
|\eta|_2^2 &= \frac{3}{8} \sqrt{\pi},\;\;\;\;\;\;
|\eta'|_2^2 = \frac{7}{16} \sqrt{\pi},\\
|\eta\cdot \log|_2^2 &= \frac{\sqrt{\pi}}{64}
\left(8 (3\gamma-8) \log 2 + 3 \pi^2 + 6 \gamma^2 + 24 (\log 2)^2
+ 16 - 32 \gamma\right) \\ &\leq 0.16364,\\
|\eta(t)/\sqrt{t}|_1 &= \frac{2^{1/4} \Gamma(1/4)}{4} \leq 1.07791,\;\;\;\;\;
|\eta(t) \sqrt{t}|_1 = \frac{3}{4} 2^{3/4} \Gamma(3/4)\leq
1.54568,\\
|\eta'(t)/\sqrt{t}|_1 &= \int_0^{\sqrt{2}} t^{3/2} e^{-\frac{t^2}{2}} dt - 
\int_{\sqrt{2}}^\infty t^{3/2} e^{-\frac{t^2}{2}} dt \leq 1.48469,\\
|\eta'(t) \sqrt{t}|_1 &\leq 1.72169
.\end{aligned}\end{equation}

\begin{prop}\label{prop:magoma}
Let $\eta(t) = t^2 e^{-t^2/2}$.
 Let $x\geq 1$, $\delta\in \mathbb{R}$.
Let $\chi$ be a primitive character mod $q$, $q\geq 1$. 
Assume that all non-trivial zeros $\rho$ of $L(s, \chi)$
with $|\Im(\rho)|\leq T_0$
lie on the critical line.
Assume that $T_0\geq \max(10 \pi |\delta|,50)$.

 Then
\begin{equation}
\sum_{n=1}^\infty \Lambda(n) \chi(n) e\left(\frac{\delta}{x} n\right) \eta(n/x) =
\begin{cases} \widehat{\eta}(-\delta) x + 
O^*\left(\err_{\eta,\chi}(\delta,x)\right)\cdot x
&\text{if $q=1$,}\\
O^*\left(\err_{\eta,\chi}(\delta,x)\right)\cdot x
&\text{if $q>1$,}\end{cases}\end{equation}
where
\begin{equation}\label{eq:camdiaw}\begin{aligned}
\err_{\eta,\chi}(\delta,x) &= 
T_0 \log \frac{q T_0}{2\pi} \cdot
\left( 6.11 e^{-0.1598 T_0} + 1.578
  e^{-0.1065 \cdot \frac{T_0^2}{(\pi \delta)^2}}\right)\\
&+
\left(1.22 \sqrt{T_0} \log q T_0 + 5.056 \sqrt{T_0} + 1.423 \log q + 37.19\right)\cdot x^{-1/2}\\
&+ (3+11 |\delta|) x^{-1} + (\log q + 6) \cdot
(1 + 6 |\delta|)\cdot x^{-3/2}.
\end{aligned}\end{equation}
\end{prop}
\begin{proof}
We proceed as in the proof of Prop.~\ref{prop:bargo}. The contribution of 
Lemma \ref{lem:festavign} is
\[T_0 \log \frac{q T_0}{2\pi} \cdot \left( 
6.11 e^{-0.1598 T_0} + 
1.578 e^{-0.1065 \cdot \frac{T_0^2}{(\pi \delta)^2}}\right)\cdot x,\]
whereas the contribution of Lemma \ref{lem:hausierer} is at most
\[
(1.22 \sqrt{T_0} \log q T_0 + 5.056 \sqrt{T_0} + 1.423 \log q + 37.188)
\sqrt{x}.
\]

Let us now apply Lemma \ref{lem:agamon}.
Since $\eta(0)=0$, we have
\[R = O^*(c_0) = O^*(2.138 + 10.99 |\delta|).\]
Lastly,
\[|\eta'|_2 + 2 \pi |\delta| |\eta|_2 \leq 0.881 + 5.123 |\delta|.\]
\end{proof}

Now that we have Prop.~\ref{prop:magoma}, we can derive from it
similar bounds for a smoothing defined as the multiplicative
convolution of $\eta$ with something else.
In general,
for $\varphi_1,\varphi_2:\lbrack 0,\infty)\to \mathbb{C}$, if we know how
to bound 
sums of the form
\begin{equation}
S_{f,\varphi_1}(x) = \sum_n f(n) \varphi_1(n/x),\end{equation}
we can bound sums of the form $S_{f,\varphi_1 \ast_M \varphi_2}$, simply by changing
the order of summation and integration:
\begin{equation}\begin{aligned}
S_{f,\varphi_1 \ast_M \varphi_2} &=
\sum_n f(n) \cdot (\varphi_1 \ast_M \varphi_2)\left(\frac{n}{x}\right) \\
&= 
 \int_0^{\infty} \sum_n f(n) \varphi_1\left(\frac{n}{w x}\right) \varphi_2(w)
   \frac{dw}{w}
= \int_{0}^{\infty} S_{f,\varphi_1}(w x) 
\varphi_2(w) \frac{dw}{w}.\end{aligned}\end{equation}
This is particularly nice if $\varphi_2(t)$ vanishes in a neighbourhood of
the origin, since then the argument $w x$ of $S_{f,\varphi_1}(w x)$ is always
large.

We will use $\varphi_1(t) = t^2 e^{-t^2/2}$, $\varphi_2(t) = \eta_1 \ast_M
\eta_1$, where $\eta_1$ is $2$ times the characteristic function of the
interval $\lbrack 1/2,1\rbrack$. 
The motivation for the choice of
$\varphi_1$ and $\varphi_2$ is clear: we have just got bounds based on
$\varphi_1(t)$ in the major arcs, and we obtained
 minor-arc bounds for the weight
$\varphi_2(t)$ in Part \ref{part:min}.

\begin{corollary}\label{cor:kolona}
Let $\eta(t) = t^2 e^{-t^2/2}$, $\eta_1 = 2 
\cdot I_{\lbrack 1/2,1\rbrack}$, $\eta_2 = \eta_1 \ast_M \eta_1$. Let
$\eta_* = \eta_2 \ast_M \eta$.
 Let $x\in \mathbb{R}^+$, $\delta\in \mathbb{R}$.
Let $\chi$ be a primitive character mod $q$, $q\geq 1$. 
Assume that all non-trivial zeros $\rho$ of $L(s,\chi)$
with $|\Im(\rho)|\leq T_0$
lie on the critical line.
Assume that $T_0\geq \max(10 \pi |\delta|,50)$.

 Then
\begin{equation}\label{eq:expec}
\sum_{n=1}^\infty \Lambda(n) \chi(n) e\left(\frac{\delta}{x} n\right) \eta_*(n/x) =
\begin{cases} \widehat{\eta_*}(-\delta) x + 
O^*\left(\err_{\eta_*,\chi}(\delta,x)\right)\cdot x
&\text{if $q=1$,}\\
O^*\left(\err_{\eta_*,\chi}(\delta,x)\right)\cdot x
&\text{if $q>1$,}\end{cases}\end{equation}
where
\begin{equation}\label{eq:monte}\begin{aligned}
\err_{\eta,\chi_*}(\delta,x) &= 
T_0 \log \frac{q T_0}{2\pi} \cdot
\left( 6.11 e^{-0.1598 T_0} + 0.0102\cdot 
  e^{-0.1065 \cdot \frac{T_0^2}{(\pi \delta)^2}}\right)\\
&+
\left(1.679 \sqrt{T_0} \log q T_0 + 6.957 \sqrt{T_0} +
1.958 \log q + 51.17\right)\cdot x^{-\frac{1}{2}}\\
&+ (6 + 22 |\delta|) x^{-1} + (\log q + 6) \cdot
(3 + 17 |\delta|)\cdot x^{-3/2}.
\end{aligned}\end{equation}
\end{corollary}
\begin{proof}
The left side of (\ref{eq:expec}) equals
\[\begin{aligned}
&\int_0^\infty \sum_{n=1}^\infty \Lambda(n) \chi(n) e\left(\frac{\delta
    n}{x}\right) \eta\left(\frac{n}{w x}\right) \eta_2(w) \frac{dw}{w}\\
\\ &=\int_{\frac{1}{4}}^1 \sum_{n=1}^\infty \Lambda(n) \chi(n) e\left(\frac{\delta w
    n}{w x}\right) \eta\left(\frac{n}{w x}\right) \eta_2(w) \frac{dw}{w},
\end{aligned}\]
since $\eta_2$ is supported on $\lbrack -1/4,1\rbrack$.
By Prop.~\ref{prop:magoma}, the main term (if $q=1$) contributes
\[\begin{aligned}
&\int_{\frac{1}{4}}^1 \widehat{\eta}(-\delta w) x w\cdot
\eta_2(w) \frac{dw}{w}
= x \int_{0}^\infty \widehat{\eta}(-\delta w) \eta_2(w) dw\\
&= x \int_{0}^\infty \int_{-\infty}^\infty
\eta(t) e(\delta w t) dt \cdot \eta_2(w) dw
= x \int_{0}^\infty \int_{-\infty}^\infty
\eta\left(\frac{r}{w}\right) e(\delta r) 
\frac{dr}{w} \eta_2(w) dw\\
&= x \int_{-\infty}^\infty \left(\int_0^\infty 
\eta\left(\frac{r}{w}\right) \eta_2(w) \frac{dw}{w}\right)
 e(\delta r) dr = \widehat{\eta_*}(-\delta)\cdot x.\end{aligned}\]
The error term is
\begin{equation}\label{eq:braca}\int_{\frac{1}{4}}^1 
\err_{\eta,\chi}(\delta w, w x) \cdot w x \cdot \eta_2(w) \frac{dw}{w}
= x \cdot 
\int_{\frac{1}{4}}^1 
\err_{\eta,\chi}(\delta w, w x) \eta_2(w) dw.
\end{equation}
Using the fact that
\[\eta_2(w) = \begin{cases} 4 \log 4 w &\text{if $w\in \lbrack 1/4,1/2
\rbrack$,}\\ 4 \log w^{-1} & \text{if $w\in \lbrack 1/2,1\rbrack$,}\\
0 &\text{otherwise,}\end{cases}\]
we can easily check that
 \[\begin{aligned}
\int_0^\infty \eta_2(w) dw &= 1,\;\;\;\;\;\;\;\;\;\;\;\;\;\;\;\;\;\;\;\;\;\;\;\;\;\;\;\;\;\;
 \int_0^\infty w^{-1/2} \eta_2(w) dw \leq 1.37259,\\
\int_0^\infty w^{-1} \eta_2(w) dw &= 4 (\log 2)^2 \leq 1.92182,\;\;\;\;\;\;\; 
\int_0^\infty w^{-3/2} \eta_2(w) dw \leq 2.74517\end{aligned}\]
and, by rigorous numerical 
integration from $1/4$ to $1/2$ and from $1/2$ to
  $1$ (using, e.g., VNODE-LP \cite{VNODELP}),
\[\int_0^\infty e^{-0.1065 \cdot 10^2 \left(\frac{1}{w^2} - 1\right)}
\eta_2(w) dw \leq 0.006446.\]
We then see that (\ref{eq:camdiaw}) and (\ref{eq:braca}) imply (\ref{eq:monte}).
\end{proof}

\section{The case of $\eta_+(t)$}\label{subs:astardo}
We will work with
\begin{equation}\label{eq:dmit}
\eta(t) = \eta_+(t) = 
h_H(t)\cdot t \eta_{\heartsuit}(t) = 
h_H(t) \cdot t e^{-t^2/2},\end{equation}
where $h_H$ is as in (\ref{eq:dirich2}). We recall that $h_H$ is
a band-limited approximation to the function $h$ defined in (\ref{eq:hortor})
-- to be more precise, $M h_H(i t)$ is the truncation of 
$M h(i t)$ to the interval $\lbrack -H, H\rbrack$.

We are actually defining $h$, $h_H$ and $\eta$ in a slightly different way from
what was done in the first version of \cite{HelfMaj}. The difference is 
instructive. There, $\eta(t)$ was defined as $h_H(t) e^{-t^2/2}$, and
$h_H$ was a band-limited approximation to a function $h$ defined as in
(\ref{eq:hortor}), but with $t^3 (2-t)^3$ instead of $t^2 (2-t)^3$. 
The reason for our
new definitions is that now the truncation of $M h (i t)$ will not break
the holomorphy of $M\eta$, and so we will be able to use the general results
we proved in \S \ref{subs:genexpf}.

In essence, $M h$ will still be holomorphic because the Mellin transform
of $t \eta_{\heartsuit}(t)$ is holomorphic in the domain we care about, unlike
the Mellin transform of $\eta_{\heartsuit}(t)$, which does have a pole at $s=0$.

As usual, we start by bounding
the contribution of zeros with large imaginary part. 
The procedure is much as before: since $\eta_+(t) = \eta_H(t) 
\eta_\heartsuit(t)$ , the Mellin transform $M\eta_+$ is a convolution
of $M(t e^{-t^2/2})$ and something of support in $\lbrack - H, H\rbrack i$,
 namely,
$M \eta_H$ restricted to the imaginary axis.
This means that the decay of
$M \eta_+$ is (at worst) like the decay of $M(t e^{-t^2/2})$, delayed by $H$.

\begin{lemma}\label{lem:schastya}
Let $\eta=\eta_+$ be as in (\ref{eq:dmit}) for some $H \geq 25$.
Let $x\in \mathbb{R}^+$, $\delta\in \mathbb{R}$.
Let $\chi$ be a primitive character mod $q$, $q\geq 1$. 
Assume that all non-trivial zeros $\rho$ of $L(s,\chi)$ 
with $|\Im(\rho)|\leq T_0$
satisfy $\Re(s)=1/2$, where $T_0\geq H+\max(10 \pi |\delta|,50)$.

Write $G_\delta(s)$ for the Mellin transform of $\eta(t) e(\delta t)$. Then
\[\begin{aligned}
\mathop{\sum_{\rho}}_{|\Im(\rho)|>T_0} 
&\left|G_{\delta}(\rho)\right|\leq 
\left(11.308 \sqrt{T_0'} e^{-0.1598 T_0'} +
16.147 |\delta| e^{-0.1065 \left(
\frac{T_0'}{\pi \delta}\right)^2}
\right) \log \frac{q T_0}{2\pi},
\end{aligned}\] 
where $T_0' = T_0-H$.
\end{lemma}
\begin{proof}
As usual,
\[\mathop{\sum_{\text{$\rho$ }}}_{|\Im(\rho)|>T_0} 
\left|G_{\delta}(\rho)\right| = 
\mathop{\sum_{\text{$\rho$ }}}_{\Im(\rho)>T_0} 
\left(\left|G_{\delta}(\rho)\right| + \left|G_{\delta}(1-\rho)\right|\right)
.\]
Let $F_\delta$ be as in (\ref{eq:guason}). Then, since
$\eta_+(t) e(\delta t)= h_H(t) t e^{-t^2/2} e(\delta t)$, where
$h_H$ is as in (\ref{eq:dirich2}), we see by (\ref{eq:mouv}) that
\[G_\delta(s) = \frac{1}{2\pi} \int_{-H}^H Mh(ir) F_\delta(s+1-ir) dr,\]
and so, since $|Mh(ir)| = |Mh(-ir)|$,
\begin{equation}\label{eq:sunodo}
\left|G_{\delta}(\rho)\right| + \left|G_{\delta}(1-\rho)\right| \leq
\frac{1}{2\pi} \int_{-H}^H |Mh(ir)| (|F_\delta(1+\rho-ir)| + |F_\delta(2-(\rho-ir)
)|) dr.\end{equation}

We apply Cor.~\ref{cor:amanita1} with $k=1$ and $T_0-H$ instead of $T_0$, 
and obtain that
$|F_\delta(\rho)|+|F_\delta(1-\rho)|\leq g(\tau)$, where
\begin{equation}\label{eq:saintsze}g(\tau) = 
\kappa_{1,1} \sqrt{|\tau|} e^{-0.1598 |\tau|} + \kappa_{1,0}
\frac{|\tau|}{2 \pi |\delta|}
e^{-0.1065 \left(\frac{\tau}{\pi \delta}\right)^2},
\end{equation}
where $\kappa_{1,0} = 4.903$ and $\kappa_{1,1} = 4.017$.
(As in the proof of Lemmas \ref{lem:garmonas} and \ref{lem:festavign},
we are putting in extra terms so as to simplify our integrals.)

From (\ref{eq:sunodo}), we conclude that
\[|G_\delta(\rho)| + |G_\delta(1-\rho)| \leq f(\tau),\]
for $\rho = \sigma+i
\tau$, $\tau>0$, where
\[f(\tau) = \frac{|M h(i r)|_1}{2\pi} \cdot g(\tau-H)\]
is decreasing for $\tau\geq T_0$ (because $g(\tau)$ is decreasing for
$\tau\geq T_0-H$). By (\ref{eq:marpales}), $|M h(i r)|_1\leq 16.193918$. 

We apply Lemma \ref{lem:garmola}, and get that
\begin{equation}\label{eq:natju}
\begin{aligned}\mathop{\sum_{\text{$\rho$ }}}_{|\Im(\rho)|>T_0} 
\left|G_{\delta}(\rho)\right| &\leq \int_{T_0}^\infty
f(T) \left(\frac{1}{2\pi}  \log \frac{q T}{2\pi}
+ \frac{1}{4 T}\right) dT\\
&= \frac{|Mh(i r)|_1}{2\pi} 
\int_{T_0}^\infty
g(T-H) \left(\frac{1}{2\pi}  \log \frac{q T}{2\pi}
+ \frac{1}{4 T }\right) dT.\end{aligned}\end{equation}

Now we just need to estimate some integrals.
For any $y\geq e^2$, $c>0$ and $\kappa,\kappa_1\geq 0$,
\[\int_y^\infty \sqrt{t} e^{-c t} dt \leq
\left(\frac{\sqrt{y}}{c} + \frac{1}{2 c^2 \sqrt{y}}\right) e^{-c y},\]
\[\begin{aligned}
\int_y^\infty \left(\sqrt{t} \log (t + \kappa) + \frac{\kappa_1}{\sqrt{t}}
\right) e^{-c t} dt \leq
\left(\frac{\sqrt{y}}{c} + \frac{a}{c^2 \sqrt{y}}\right) \log(y+\kappa) 
e^{-c y},
\end{aligned}\]
where \[a = \frac{1}{2} + 
\frac{1 + c \kappa_1}{\log(y+\kappa)}.\]

The contribution of the exponential term in (\ref{eq:saintsze})
to (\ref{eq:natju}) thus equals 
\begin{equation}\label{eq:sibenize}\begin{aligned}
&\frac{\kappa_{1,1} |Mh(i r)|_1}{2\pi} 
\int_{T_0}^\infty \left(\frac{1}{2\pi} \log \frac{q T}{2\pi} + \frac{1}{4 T}
\right) \sqrt{T-H} \cdot e^{-0.1598 (T-H)} dT\\
&\leq 10.3532
\int_{T_0 - H}^\infty \left(\frac{1}{2\pi} \log (T+H) + \frac{\log \frac{q}{2\pi}}{2\pi} + \frac{1}{4 T}\right) \sqrt{T} e^{-0.1598 T} dT\\
&\leq \frac{10.3532}{2\pi}
 \left(\frac{\sqrt{T_0-H}}{0.1598} + \frac{a}{0.1598^2 \sqrt{T_0-H}}
\right) \log \frac{q T_0}{2 \pi}\cdot  e^{-0.1598 (T_0-H)},
\end{aligned}
\end{equation}
where $a = 1/2 + (1+0.1598 \pi/2)/\log T_0$.
Since $T_0-H\geq 50$ and $T_0\geq 50+25=75$, this is at most
\[11.308 \sqrt{T_0-H} \log \frac{q T_0}{2\pi} \cdot e^{-0.1598 (T_0-H)}.\]

We now estimate a few more integrals so that we can handle the Gaussian term
in (\ref{eq:saintsze}). For any $y>1$, $c>0$, $\kappa,\kappa_1\geq 0$,
\[\int_y^\infty t e^{-c t^2} dt = \frac{e^{-c y^2}}{2 c},\]
\[\int_y^\infty (t \log (t+\kappa) + \kappa_1) e^{-c t^2} dt \leq 
\left(1 + \frac{\kappa_1 + \frac{1}{2 c y}}{y \log (y + \kappa)}
\right) \frac{\log (y+\kappa) \cdot e^{-c y^2}}{2 c}
\]
Proceeding just as before, we see that the contribution of the
Gaussian term in (\ref{eq:saintsze}) to (\ref{eq:natju}) is at most
\begin{equation}\label{eq:sibabita}\begin{aligned}
&\frac{\kappa_{1,0} |Mh(i r)|_1}{2\pi} 
\int_{T_0}^\infty \left(\frac{1}{2\pi} \log \frac{q T}{2\pi} + \frac{1}{4 T}
\right) \frac{T-H}{2 \pi |\delta|} \cdot e^{-0.1065 \left(\frac{T-H}{\pi \delta}\right)^2} dT\\
&\leq 12.6368\cdot \frac{|\delta|}{4}
\int_{\frac{T_0 - H}{\pi |\delta|}}^\infty \left(\log \left(T+
\frac{H}{\pi |\delta|}\right) + \log \frac{q |\delta|}{2} + \frac{\pi/2}{T}\right) T e^{-0.1065 T^2} dT\\
&\leq 12.6368\cdot \frac{|\delta|}{8\cdot 0.1065}
 \left(1 + \frac{\frac{\pi}{2} + \frac{\pi |\delta|}{2\cdot 0.1065\cdot (T_0-H)}
}{\frac{T_0-H}{\pi |\delta|} \log \frac{T_0}{\pi |\delta|}}
\right) \log \frac{q T_0}{2 \pi}\cdot  e^{-0.1065 \left(\frac{T_0-H}{\pi \delta}
\right)^2},
\end{aligned}
\end{equation}
Since $(T_0-H)/(\pi |\delta|) \geq 10$, this is at most
\[16.147 |\delta| \log \frac{q T_0}{2\pi} \cdot e^{-0.1065 \left(
\frac{T_0-H}{\pi \delta}\right)^2}.\]
\end{proof}

\begin{prop}\label{prop:unease}
Let $\eta=\eta_+$ be as in (\ref{eq:dmit}) for some $H \geq 25$.
Let $x\geq 10^3$, $\delta\in \mathbb{R}$.
Let $\chi$ be a primitive character mod $q$, $q\geq 1$. 
Assume that all non-trivial zeros $\rho$ of $L(s,\chi)$
with $|\Im(\rho)|\leq T_0$
lie on the critical line, where $T_0\geq H+\max(10\pi |\delta|,50)$.

Then
\begin{equation}
\sum_{n=1}^\infty \Lambda(n) \chi(n) e\left(\frac{\delta}{x} n\right) \eta_+(n/x) =
\begin{cases} \widehat{\eta_+}(-\delta) x + 
O^*\left(\err_{\eta_+,\chi}(\delta,x)\right)\cdot x
&\text{if $q=1$,}\\
O^*\left(\err_{\eta_+,\chi}(\delta,x)\right)\cdot x
&\text{if $q>1$,}\end{cases}\end{equation}
where
\begin{equation}\label{eq:nochpai}\begin{aligned}
\err_{\eta_+,\chi}(\delta,x) &= 
\left(11.308 \sqrt{T_0'} \cdot e^{-0.1598 T_0'} +
16.147 |\delta| e^{-0.1065 \left(
\frac{T_0'}{\pi \delta}\right)^2}
\right) \log \frac{q T_0}{2\pi}\\
&+ 
(1.634 \sqrt{T_0} \log q T_0 + 12.43 \sqrt{T_0} + 1.321 \log q + 34.51)
x^{1/2},\\
& + (9 + 11 |\delta|) x^{-1} + (\log q) (11 + 6 |\delta|) x^{-3/2},
\end{aligned}\end{equation}
where $T_0' = T_0 - H$.
\end{prop}
\begin{proof}
We can apply Lemmas \ref{lem:agamon} and Lemma \ref{lem:hausierer}
because $\eta_+(t)$, $(\log t) \eta_+(t)$ and $\eta_+'(t)$ are in $\ell_2$
(by (\ref{eq:mastodon}), (\ref{eq:pamiatka}) and (\ref{eq:miran}))
and $\eta_+(t) t^{\sigma-1}$ and $\eta_+'(t) t^{\sigma-1}$ are in $\ell_1$
for $\sigma$ in an open interval containing $\lbrack 1/2,3/2\rbrack$
(by (\ref{eq:paytoplay}) and (\ref{eq:uzsu})).
(Because of (\ref{eq:hutterite}), the fact that $\eta_+(t) t^{-1/2}$ and
$\eta_+(t) t^{1/2}$ are in $\ell_1$ implies that $\eta_+(t) \log t$ is also
in $\ell_1$, as is required by Lemma \ref{lem:hausierer}.)

We apply Lemmas \ref{lem:agamon}, 
 \ref{lem:hausierer} and \ref{lem:schastya}. We bound the norms
involving $\eta_+$ using the estimates in \S \ref{subs:daysold} and
\S \ref{subs:weeksold}. Since $\eta_+(0)=0$ (by the definition
(\ref{eq:patra})
of $\eta_+$), the term $R$ in (\ref{eq:estromo}) is at most $c_0$,
where $c_0$ is as in (\ref{eq:marenostrum}). We bound
\[\begin{aligned}c_0&\leq \frac{2}{3} 
\left(2.922875 \left(\sqrt{\Gamma(1/2)} + \sqrt{\Gamma(3/2)}\right) +
1.062319 \left(\sqrt{\Gamma(5/2)} + \sqrt{\Gamma(7/2)}\right)\right)
\\
&+ \frac{4\pi}{3} |\delta| \cdot 1.062319
\left(\sqrt{\Gamma(3/2)} + \sqrt{\Gamma(5/2)}\right)
\leq 6.536232 + 
9.319578 |\delta|
\end{aligned}\]
using (\ref{eq:paytoplay}) and (\ref{eq:uzsu}).
By (\ref{eq:mastodon}), (\ref{eq:miran}) and the assumption
$H\geq 25$,
\[|\eta_+|_2\leq 0.80365,\;\;\;\;\;\;\;\;
|\eta_+'|_2 \leq 10.845789.\]
Thus, the error terms in (\ref{eq:marmar}) total at most
\begin{equation}\label{eq:zwilling}\begin{aligned}
6.536232 + 
&9.319578 |\delta| + (\log q + 6.01) (
10.845789 + 2\pi\cdot 0.80365 |\delta| ) x^{-1/2}\\
&\leq 
 9 + 11 |\delta| + (\log q) (11 + 6 |\delta|) x^{-1/2}.
\end{aligned}\end{equation}

The part of the sum $\sum_\rho G_\delta(\rho) x^\rho$ in (\ref{eq:marmar})
corresponding to zeros $\rho$ with $|\Im(\rho)|>T_0$ gets estimated by
Lem \ref{lem:schastya}. By Lemma \ref{lem:hausierer},
the part of the sum corresponding to zeros $\rho$ with $|\Im(\rho)|\leq T_0$
is at most
\[(1.634 \sqrt{T_0} \log q T_0 + 12.43 \sqrt{T_0} + 1.321 \log q + 34.51)
x^{1/2},\]
where we estimate the norms $|\eta_+|_2$, $|\eta\cdot \log |_2$ and
$|\eta(t)/\sqrt{t}|_1$ by (\ref{eq:mastodon}), (\ref{eq:pamiatka}) and (\ref{eq:paytoplay}).
\end{proof}

\section{A sum for $\eta_+(t)^2$}

Using a smoothing function sometimes leads to considering sums involving the
square of the smoothing function. In particular, in Part \ref{part:concl},
we will need
a result involving $\eta_+^2$ -- something that could be
slightly challenging to prove, given the way in which $\eta_+$ is defined. 
Fortunately, we have bounds on $|\eta_+|_\infty$ and other
$\ell_\infty$-norms (see Appendix \ref{subs:byron}).
Our task will also be made easier by the fact
that we do not have a phase $e(\delta n/x)$ this time. 
All in all, this will be yet another demonstration of the generality
of the framework developed in \S \ref{subs:genexpf}.

\begin{prop}\label{prop:konechno}
Let $\eta=\eta_+$ be as in (\ref{eq:dmit}), $H\geq 25$. Let $x\geq 10^8$.
Assume that all non-trivial zeros $\rho$ of the Riemann zeta function $\zeta(s)$
with $|\Im(\rho)|\leq T_0$ lie on the critical line, where 
$T_0\geq \max(2 H + 25,200)$.

Then
\begin{equation}\label{eq:horrorshow}
\sum_{n=1}^\infty \Lambda(n) (\log n) \eta_+^2(n/x)
 = x\cdot \int_0^\infty \eta_+^2(t) \log x t\; dt
+ O^*(\err_{\ell_2,\eta_+})\cdot x \log x,\end{equation}
where
\begin{equation}\label{eq:galo}\begin{aligned}
\err_{\ell_2,\eta_+} &= 
\left(\left(0.462 \frac{(\log T_1)^2}{\log x} + 0.909 \log T_1\right) T_1 +
1.71 \left(1 + \frac{\log T_1}{\log x}\right) H \right) e^{-\frac{\pi}{4} T_1}\\
&+ (2.445 \sqrt{T_0} \log T_0 + 50.04) \cdot x^{-1/2}
\end{aligned}\end{equation}
and $T_1 = T_0 - 2 H$.
\end{prop}
The assumption $T_0\geq 200$ is stronger than what we strictly need, but,
as it happens, we could make much stronger assumptions still. Proposition
\ref{prop:konechno} relies on a verification of zeros of the Riemann
zeta function; such verifications have gone up to values of
$T_0$ much higher than $200$.
\begin{proof}
We will need to consider two smoothing functions, namely,
$\eta_{+,0}(t) = \eta_+(t)^2$ and $\eta_{+,1} =\eta_+(t)^2 \log t$. Clearly,
\[\sum_{n=1}^\infty \Lambda(n) (\log n) \eta_+^2(n/x)
= (\log x) \sum_{n=1}^\infty \Lambda(n) \eta_{+,0}(n/x) +  
\sum_{n=1}^\infty \Lambda(n) \eta_{+,1}(n/x).
\]
Since $\eta_+(t) = h_H(t) t e^{-t^2/2}$,
\[\eta_{+,0}(r) = h_H^2(t) t^2 e^{-t^2},\;\;\;\;\;\;\;\;\;\;
\eta_{+,1}(r) = h_H^2(t) (\log t) t^2 e^{-t^2}.\]
Let $\eta_{+,2} = (\log x) \eta_{+,0} + \eta_{+,1} = \eta_+^2(t) \log x t$.

We wish to apply Lemma \ref{lem:agamon}. For this, we must first
check that some norms are finite. Clearly,
\begin{equation}\label{eq:viroro}\begin{aligned}
\eta_{+,2}(t) &= \eta_{+}^2(t) \log x + \eta_+^2(t) \log t\\
\eta_{+,2}'(t) &= 2\eta_+(t) \eta_+'(t) \log x +
2 \eta_+(t) \eta_+'(t) \log t + \eta_+^2(t)/t .\end{aligned}\end{equation}
Thus, we see that
$\eta_{+,2}(t)$ is in $\ell_2$ because
$\eta_+(t)$ is in $\ell_2$ and $\eta_+(t)$, $\eta_+(t) \log t$ are both in
$\ell_\infty$ (see (\ref{eq:mastodon}), (\ref{eq:malgache}),
(\ref{eq:dalida})):
\begin{equation}\label{eq:alcmeon}\begin{aligned}
\left|\eta_{+,2}(t)\right|_2 &\leq  
\left|\eta_+^2(t)\right|_2 \log x + \left|\eta_+^2(t) \log t \right|_2\\
&\leq \left|\eta_+\right|_\infty  
\left|\eta_+\right|_2 \log x + \left|\eta_+(t) \log t \right|_\infty
\left|\eta_+\right|_2 .
\end{aligned}\end{equation}
Similarly, $\eta_{+,2}'(t)$ is in $\ell_2$ because $\eta_+(t)$ is in
$\ell_2$, $\eta_+'(t)$ is in $\ell_2$ (\ref{eq:miran}), and
$\eta_+(t)$, $\eta_+(t) \log t$ and  
$\eta_+(t)/t$ (see (\ref{eq:gobmark})) are all in $\ell_\infty$:
\begin{equation}\label{eq:comor}\begin{aligned}
\left|\eta_{+,2}'(t)\right|_2 &\leq  
\left|2 \eta_+(t) \eta_+'(t)\right|_2 \log x + 
\left|2 \eta_+(t) \eta_+'(t) \log t \right|_2 +
\left|\eta_+^2(t)/t\right|_2\\
&\leq 2 \left|\eta_+\right|_\infty \left|\eta_+'\right|_2 \log x + 
2 \left|\eta_+(t) \log t\right|_\infty \left|\eta_+'\right|_2 + 
\left|\eta_+(t)/t\right|_\infty \left|\eta_+\right|_2 . 
\end{aligned}\end{equation}
In the same way, we see that $\eta_{+,2}(t) t^{\sigma-1}$ is in
$\ell_1$ for all $\sigma$ in $(-1,\infty)$ (because
the same is true of $\eta_+(t) t^{\sigma-1}$ (\ref{eq:paytoplay}),
and $\eta_+(t)$, $\eta_+(t) \log t$ are both in
$\ell_\infty$) and $\eta_{+,2}'(t) t^{\sigma-1}$ is in
$\ell_1$ for all $\sigma$ in $(0,\infty)$ (because the same is true
of $\eta_+(t) t^{\sigma-1}$ and $\eta_+'(t) t^{\sigma-1}$ (\ref{eq:uzsu}),
and $\eta_+(t)$, $\eta_+(t) \log t$, $\eta_+(t)/t$  are all in $\ell_\infty$).

We now apply Lemma \ref{lem:agamon} with $q=1$, $\delta=0$. 
Since $\eta_{+,2}(0)=0$, the residue term $R$ equals $c_0$, which, by
(\ref{eq:viroro}),
is at most $2/3$ times
\[\begin{aligned}
2 &\left(\left|\eta_+\right|_\infty \log x + 
 \left|\eta_+(t) \log t\right|_\infty\right) \left( 
\left|\eta_+'(t)/\sqrt{t}\right|_1 + \left|\eta_+'(t) \sqrt{t}\right|_1
\right)\\
&+ \left|\eta_+(t)/t\right|_\infty \left( 
\left|\eta_+(t)/\sqrt{t}\right|_1 + \left|\eta_+(t) \sqrt{t}\right|_1
\right).
\end{aligned}\]
Using the bounds 
 (\ref{eq:malgache}), (\ref{eq:dalida}), (\ref{eq:gobmark})
(with the assumption $H\geq 25$),
(\ref{eq:paytoplay}) and (\ref{eq:uzsu}), we get that this means that
\[c_0\leq 18.57606 \log x + 8.63264 .
\]

Since $q=1$ and $\delta=0$, we get from (\ref{eq:comor}) (and
 (\ref{eq:malgache}), (\ref{eq:dalida}), (\ref{eq:gobmark}),
with the assumption $H\geq 25$, and also (\ref{eq:mastodon}) and
(\ref{eq:miran})) 
that
\[\begin{aligned}
(\log q + 6.01) \cdot &\left(\left|\eta_{+,2}'\right|_2 + 2 \pi |\delta|
\left|\eta_{+,2}\right|_2\right) x^{-1/2} \\&= 6.01
\left|\eta_{+,2}'\right|_2 x^{-1/2}
\leq (162.56 \log x + 59.325) x^{-1/2}.\end{aligned}\]
Using the assumption $x\geq 10^8$, we obtain
\begin{equation}\label{eq:croissant}c_0 + (185.26 \log x + 71.799) x^{-1/2}\leq
19.064 \log x.\end{equation}

We will now apply Lemma \ref{lem:hausierer} -- as we may, because of the 
finiteness of the norms
we have already checked, together with
\begin{equation}\label{eq:sansan}\begin{aligned}
\left|\eta_{+,2}(t) \log t\right|_2 &\leq 
\left|\eta_+^2(t) \log t\right|_2 \log x +
\left|\eta_+^2(t) (\log t)^2 \right|_2 \\ &\leq
\left|\eta_+(t) \log t\right|_\infty 
\left(\left|\eta_+(t)\right|_2 \log x + \left|\eta_+(t) \log
    t\right|_2\right)\\
&\leq 0.4976\cdot (0.80365 \log x + 0.82999) \leq 0.3999 \log x + 0.41301
\end{aligned}\end{equation}
(by (\ref{eq:dalida}), (\ref{eq:mastodon}) 
and (\ref{eq:pamiatka}); use the assumption $H\geq 25$). We also need
the bounds
\begin{equation}\label{eq:cuahcer}
\left|\eta_{+,2}(t)\right|_2\leq 1.14199 \log x + 0.39989\end{equation}
(from (\ref{eq:alcmeon}), by the norm bounds (\ref{eq:malgache}),
(\ref{eq:dalida}) and (\ref{eq:mastodon}), all with $H\geq 25$) and
\begin{equation}\label{eq:jostume}\begin{aligned}
\left|\eta_{+,2}(t)/\sqrt{t}\right|_1 &\leq 
\left(\left|\eta_+(t)\right|_\infty \log x + 
 \left|\eta_+(t) \log t\right|_\infty\right) \left|\eta_+(t)/\sqrt{t}\right|_1 \\
&\leq 1.4211 \log x + 0.49763
\end{aligned}\end{equation}
(by (\ref{eq:malgache}), (\ref{eq:dalida}) (again with $H\geq 25$)
and (\ref{eq:paytoplay})). 

Applying Lemma \ref{lem:hausierer},
we obtain that the sum $\sum_\rho |G_0(\rho)| x^\rho$ (where
$G_0(\rho) = M \eta_{+,2}(\rho)$) over
all non-trivial zeros $\rho$ with $|\Im(\rho)|\leq T_0$ is at most
$x^{1/2}$ times
\begin{equation}\label{eq:trucidor}
\begin{aligned}
(1.54189 \log x + 0.8129) \sqrt{T_0} \log T_0 &+ 
(4.21245 \log x + 6.17301) \sqrt{T_0}\\ &+ 49.1 \log x + 17.2,
\end{aligned}
\end{equation}
where we are bounding norms by (\ref{eq:cuahcer}), (\ref{eq:sansan}) 
and (\ref{eq:jostume}).
(We are using the fact that
$T_0\geq 2 \pi \sqrt{e}$ to ensure that the quantity
$\sqrt{T_0} \log T_0 - (\log 2 \pi \sqrt{e}) \sqrt{T_0}$ being multiplied
by $|\eta_{+,2}|_2$ is positive; thus, an upper bound for
$|\eta_{+,2}|_2$ suffices.) By the assumptions $x\geq 10^{8}$,
$T_0\geq 200$, (\ref{eq:trucidor}) is at most
\[
(2.445 \sqrt{T_0} \log T_0 + 50.034) \log x.\]
In comparison, $19.064 x^{-1/2} \log x \leq 0.002 \log x$, since $x\geq 10^8$.

It remains to bound the sum of $M\eta_{+,2}(\rho)$ over zeros
with $|\Im(\rho)|> T_0$. This we will do, as usual, by Lemma
\ref{lem:garmola}. For that, we will need to bound 
$M\eta_{+,2}(\rho)$ for $\rho$ in the critical strip.

The Mellin transform of $e^{-t^2}$ is $\Gamma(s/2)/2$, and so
the Mellin transform of $t^2 e^{-t^2}$ is $\Gamma(s/2+1)/2$.
By (\ref{eq:harva}),
this implies that the Mellin transform
of $(\log t) t^2 e^{-t^2}$ is $\Gamma'(s/2+1)/4$.
 Hence, by (\ref{eq:mouv}), 
\begin{equation}\label{eq:moses}
M\eta_{+,2}(s) = \frac{1}{4\pi} \int_{-\infty}^{\infty}
M(h_H^2)(ir) \cdot F_x\left(s-ir\right) dr,
\end{equation}
where
\begin{equation}\label{eq:luke}
F_x(s) = (\log x) \Gamma\left(\frac{s}{2} + 1\right) + 
\frac{1}{2} \Gamma'\left(\frac{s}{2} + 1\right).
\end{equation}

Moreover, 
\begin{equation}\label{eq:langosta}
M(h_H^2)(i r) = \frac{1}{2\pi} \int_{-\infty}^{\infty}
Mh_H(iu) Mh_H(i(r-u))\;  du, 
\end{equation}
and so  $M(h_H^2)(ir)$ is supported
on $\lbrack -2H,2H\rbrack$. We also see that 
$|Mh_H^2(ir)|_1\leq |Mh_H(ir)|_1^2/2\pi$.
We know that $|Mh_H(ir)|_1^2/2\pi\leq 41.73727$
by (\ref{eq:marpales}).

Hence
\begin{equation}\label{eq:rancune}\begin{aligned}
&|M\eta_{+,2}(s)| \leq \frac{1}{4\pi} \int_{-\infty}^\infty |M(h_H^2)(i r)| dr
\cdot \max_{|r|\leq 2H} |F_x(s-ir)|\\
&\leq \frac{41.73727}{4\pi} \cdot
\max_{|r|\leq 2H} |F_x(s-ir)| \leq 
3.32135 \cdot \max_{|r|\leq 2H} |F_x(s-ir)|.
\end{aligned}\end{equation}

By (\ref{eq:sitirlo})
(Stirling with explicit constants),
\begin{equation}
|\Gamma(s)|\leq 
\sqrt{2\pi} |s|^{\sigma-\frac{1}{2}} 
e^{\frac{1}{12 |s|} + \frac{\sqrt{2}}{180 |s|^3}}
 e^{-\pi |\Im(s)|/2}\end{equation}
when $\Re(s)\geq 0$, and so
\begin{equation}\label{eq:lucia}\begin{aligned}
|\Gamma(s)|&\leq \sqrt{2 \pi} \left(\frac{\sqrt{12.5^2+1.5^2}}{12.5}\right)
e^{\frac{1}{12 \cdot 12.5} + \frac{\sqrt{2}}{180 \cdot 12.5^3}}
 \cdot |\Im(s)| e^{-\pi |\Im(s)|/2}\\ &\leq
2.542 |\Im(s)| e^{-\pi |\Im(s)|/2}
\end{aligned}\end{equation}
for $s\in \mathbb{C}$ with $0<\Re(s)\leq 3/2$ and $|\Im(s)|\geq 25/2$.
Moreover, by \cite[5.11.2]{MR2723248} and the remarks at the beginning of
\cite[5.11(ii)]{MR2723248},
\[
\frac{\Gamma'(s)}{\Gamma(s)} = 
\log s - \frac{1}{2s} + O^*\left(\frac{1}{12 |s|^2} \cdot 
\frac{1}{\cos^3 \theta/2}\right)\]
for $|\arg(s)| < \theta$ ($\theta\in (-\pi,\pi)$). 
Again, for $s= \sigma + i \tau$ with $0< \sigma \leq 3/2$ and 
$|\tau|\geq 25/2$, this gives us
\[\begin{aligned}\frac{\Gamma'(s)}{\Gamma(s)} &= \log |\tau| + 
\log \frac{\sqrt{|\tau|^2+1.5^2}}{|\tau|} + O^*\left(\frac{1}{2 |\tau|}\right) 
+ O^*\left(\frac{1}{12 |\tau|^2} \cdot
\frac{1}{(1/\sqrt{2})^3}\right) 
\\ &= \log |\tau| + O^*\left(\frac{9}{8 |\tau|^2} + \frac{1}{2 |\tau|}\right)
+ \frac{O^*(0.236)}{|\tau|^2} \\ 
&= \log |\tau| + O^*\left(\frac{0.609}{|\tau|}\right).\end{aligned}\]
Hence, for $0\leq \Re(s)\leq 1$ (or, in fact, $-2\leq \Re(s)\leq 1$) 
and $|\Im(s)|\geq 25$,
\begin{equation}\label{eq:adela} \begin{aligned}
|F_x(s)| &\leq \left((\log x) + \frac{1}{2} \log \left|\frac{\tau}{2}\right| 
+ \frac{1}{2} O^*\left(\frac{0.609}{|\tau/2|}\right)\right)
\Gamma\left(\frac{s}{2}+1\right)\\
&\leq 2.542 ((\log x) + \frac{1}{2} \log |\tau| - 0.297)
\frac{|\tau|}{2} e^{-\pi |\tau|/2}.\end{aligned}\end{equation}

Thus, by (\ref{eq:rancune}),
for $\rho=\sigma+i\tau$ with $|\tau|\geq T_0\geq 2H+25$
and $0\leq \sigma\leq 1$,
\[|M\eta_{+,2}(\rho)|\leq f(\tau)\]
where
\begin{equation}\label{eq:cajun}
f(T) = 
 8.45 \left(\log x + \frac{1}{2} \log T\right) 
\left(\frac{|\tau|}{2}-H\right) \cdot e^{-\frac{\pi
  (|\tau|-2H)}{4}}.
\end{equation}
The functions $t\mapsto t e^{-\pi t/2}$
and $t\mapsto (\log t) t e^{-\pi t/2}$ are decreasing for $t\geq e$ (or
in fact for $t\geq 1.762$);
setting $t = T/2 - H$, we see that the right side of
(\ref{eq:cajun}) is a decreasing function of $T$ for $T\geq T_0$,
since $T_0/2 - H\geq 25/2 > e$.

We can now apply Lemma \ref{lem:garmola}, and get that
\begin{equation}\label{eq:kolmo}
\mathop{\sum_{\text{$\rho$ }}}_{|\Im(\rho)|>T_0}
|M\eta_{+,2}(\rho)| 
\leq \int_{T_0}^\infty f(T) \left(\frac{1}{2\pi} \log \frac{T}{2 \pi} +
\frac{1}{4 T}\right) dT.\end{equation}
Since $T\geq T_0\geq 75>2$, we know that
$((1/2\pi) \log(T/2\pi) + 1/4T) \leq (1/2\pi) \log T$.
Hence, the right side of 
(\ref{eq:kolmo}) is at most
\begin{equation}\label{eq:gegersh}\begin{aligned}
&\frac{8.39}{4\pi} \int_{T_0}^\infty 
\left((\log x) (\log T) + \frac{(\log T)^2}{2}\right) (T-2H)
e^{-\frac{\pi (T - 2 H)}{4}} dT\\
&\leq 0.668 \int_{T_1}^\infty \left((\log x) \left(\log t + \frac{2H}{t}\right) +
\left(\frac{(\log t)^2}{2} + 2 H \frac{\log t}{t} \right)
\right) t e^{-\frac{\pi t}{4}} dt,\end{aligned}\end{equation}
where $T_1 = T_0 - 2H$ and $t = T-2H$; we are using the facts that
$(\log t)''<0$ for $t>0$ and $((\log t)^2)''<0$ for $t>e$. (Of course,
$T_1 \geq 25 > e$.)

Of course, $\int_{T_1}^\infty e^{-(\pi/4) t} = (4/\pi) e^{-(\pi/4) T_1}$.
We recall (\ref{eq:vysyvat}) and (\ref{eq:vrashchenie}):
\[\begin{aligned}
\int_{T_1}^\infty \log t \cdot e^{-\frac{\pi}{4} t} dt &\leq
\left(\log T_1 + \frac{4/\pi}{T_1}\right) \frac{e^{-\frac{\pi}{4} T_1}}{\pi/4}\\
\int_{T_1}^\infty (\log t) t e^{-\frac{\pi}{4} t} dt &\leq
\left(T_1 + \frac{4 a}{\pi} \right)\frac{e^{-\frac{\pi}{4} T_1} \log T_1}{\pi/4}
\end{aligned}\]
for $T_1\geq 1$ satisfying $\log T_1 > 4/(\pi T_1)$, 
 where $a = 1 + (1 + 4/(\pi T_1))/(\log T_1 - 4/(\pi T_1))$.
It is easy to check that $\log T_1 > 4/(\pi T_1)$ and
$4 a/\pi \leq 1.6957$ for $T_1\geq 25$; of course, we also have
$(4/\pi)/25 \leq 0.051$.
Lastly,
\[\int_{T_1}^\infty (\log t)^2 t e^{-\frac{\pi}{4} t} dt \leq 
\left(T_1+ \frac{4 b}{\pi}\right) \frac{e^{-\frac{\pi}{4} T_1} (\log T_1)^2}{\pi/4}\]
for $T_1\geq e$, where $b = 1 + (2+8/(\pi T_1))/(\log T_1 - 8/(\pi T_1))$,
and we check that $4 b/\pi \leq 2.1319$ for $T_1\geq 25$.
We conclude that the integral on the second line of  (\ref{eq:gegersh})
is at most
\[\begin{aligned}&\frac{4}{\pi} \left(\frac{(\log T_1)^2}{2} (T_1 + 2.132) +
(\log x) (\log T_1) (T_1 + 1.696) 
\right) e^{-\frac{\pi}{4} T_1}\\+ & \frac{4}{\pi} \cdot 2H (\log T_1 + 0.051 + \log x) e^{-\frac{\pi}{4} T_1}.\end{aligned}\]
Multiplying this by $0.668$ and simplifying further (using $T_1\geq 25$),
we conclude that $\sum_{\rho: |\Im(\rho)|> T_0} |M \eta_{+,2}(\rho)|$
is at most
\[\left((0.462 \log T_1 + 0.909 \log x) (\log T_1) T_1 +
1.71 (\log T_1 + \log x) H \right) e^{-\frac{\pi}{4} T_1}.\]
\end{proof}
\section{A verification of zeros and its consequences}

David Platt verified in his doctoral thesis \cite{Platt}, 
 that, for every
primitive character $\chi$ of conductor $q\leq 10^5$, all the non-trivial
zeroes of $L(s,\chi)$ with imaginary part $\leq 10^8/q$ lie on the critical
line, i.e., have real part exactly $1/2$. (We call this a {\em
GRH verification up to $10^8/q$}.)

In work undertaken in coordination with the present work \cite{Plattfresh},
Platt has extended these computations to
\begin{itemize}
\item all odd $q\leq 3\cdot 10^5$, with $T_q = 10^8/q$,
\item all even $q\leq 4\cdot 10^5$, with 
$T_q = \max(10^8/q,200 + 7.5\cdot 10^7/q)$.
\end{itemize}  
The method used was rigorous; its implementation uses interval 
arithmetic.

Let us see what this verification gives us when used 
as an input to Prop.~\ref{prop:bargo}. We are interested in bounds on
$|\err_{\eta,\chi^*}(\delta,x)|$ for $q\leq r$ and $|\delta|\leq 4r/q$.
We set $r=3\cdot 10^5$. (We will not be using the verification for
$q$ even with $3\cdot 10^5 < q\leq 4\cdot 10^5$, though we certainly could.)

We let $T_0 = 10^8/q$. Thus,
\begin{equation}\label{eq:maljust}
\begin{aligned}T_0 &\geq \frac{10^8}{3 \cdot 10^5} = \frac{1000}{3},\\
\frac{T_0}{\pi |\delta|} &\geq 
 \frac{10^8/q}{\pi \cdot 4 r/q} 
 = \frac{1000}{12 \pi} \end{aligned}\end{equation}
and so, by $|\delta|\leq 4 r/q \leq 1.2\cdot 10^6/q\leq 1.2\cdot 10^6$,
\[\begin{aligned}3.53 e^{-0.1598 T_0} &\leq 2.597 \cdot 10^{-23},\\
22.5 \frac{\delta^2}{T_0} e^{-0.1065 \frac{T_0^2}{(\pi \delta)^2}} &\leq |\delta|\cdot
7.715 \cdot 10^{-34} \leq 9.258 \cdot 10^{-28}.
\end{aligned}\]
Since $q T_0 \leq 10^8$, this
gives us that
\[\begin{aligned}
\log \frac{q T_0}{2\pi} &\cdot \left(3.53 e^{-0.1598 T_0} 
+ 22.5 \frac{\delta^2}{T_0} e^{-0.1065 \frac{T_0^2}{(\pi \delta)^2}}\right)\\ &\leq 4.3054\cdot 10^{-22} + \frac{1.54\cdot 10^{-26}}{q} \leq
4.306 \cdot 10^{-22}.\end{aligned}\]
Again by $T_0=10^8/q$,
\[2.337  \sqrt{T_0} \log q T_0 + 21.817 \sqrt{T_0} + 2.85 \log q + 74.38\]
is at most
\[\frac{648662}{\sqrt{q}} + 111,\]
and 
\[\begin{aligned}
3 \log q + 14 |\delta| + 17 &\leq 55 + \frac{1.7\cdot 10^7}{q},\\
(\log q + 6)\cdot (1 + 5 |\delta|)&\leq 19 + \frac{1.2\cdot 10^8}{q}.
\end{aligned}\]
Hence, assuming $x\geq 10^8$ to simplify, we see that Prop.~\ref{prop:bargo}
gives us that
\[\begin{aligned}
\err_{\eta,\chi}(\delta,x) &\leq 
4.306 \cdot 10^{-22} + 
\frac{\frac{648662}{\sqrt{q}} + 111}{\sqrt{x}} + 
\frac{55 + \frac{1.7\cdot 10^7}{q}}{x} + 
\frac{19 + \frac{1.2\cdot 10^8}{q}}{x^{3/2}}\\
&\leq 4.306 \cdot 10^{-22} + \frac{1}{\sqrt{x}}
\left( \frac{650400}{\sqrt{q}} + 112\right)
\end{aligned}\]
for $\eta(t) = e^{-t^2/2}$. This proves Theorem \ref{thm:gowo1}.

Let us now see what Platt's calculations give us when used as an input to
Prop.~\ref{prop:magoma} and Cor.~\ref{cor:kolona}.
Again, we set $r=3\cdot 10^5$, $\delta_0=8$, $|\delta|\leq 4 r/q$
 and $T_0 = 10^8/q$, so 
(\ref{eq:maljust}) is still valid. We obtain
\[\begin{aligned}
&T_0 \log \frac{q T_0}{2\pi} \cdot
\left( 6.11 e^{-0.1598 T_0} + 
  1.578 e^{-0.1065 \cdot \frac{T_0^2}{(\pi \delta)^2}}\right)\\
\leq &\log \frac{10^8}{2\pi} 
\left(6.11 \cdot \frac{1000}{3} e^{-0.1598 \cdot\frac{1000}{3}}
+10^8\cdot 1.578 e^{- 0.1065  \left(\frac{1000}{12 \pi}\right)^2}\right)
\\ &\leq
2.485\cdot 10^{-19},\end{aligned}\]
since $t \exp(-0.1598 t)$ is decreasing on $t$ for $t\geq 1/0.1598$.
We use the same bound when we have $0.0102$ instead of $1.578$ on the left
side, as in (\ref{eq:monte}). (The coefficient affects what is by far
the smaller term, so we are wasting nothing.) 
Again by $T_0=10^8/q$ and $q\leq r$,
\[\begin{aligned}
1.22 \sqrt{T_0} \log q T_0 + 5.053 \sqrt{T_0} + 1.423 \log q + 37.19
&\leq \frac{279793}{\sqrt{q}} + 55.2\\
1.679 \sqrt{T_0} \log q T_0 + 6.957 \sqrt{T_0} + 1.958 \log q + 51.17
&\leq \frac{378854}{\sqrt{q}} + 75.9.
\end{aligned}\]
For $x\geq 10^8$, we use
 $|\delta|\leq 4 r/q \leq 1.2\cdot 10^6/q$ to bound
\[(3+11 |\delta|) x^{-1} + (\log q + 6) \cdot
(1 + 6 |\delta|)\cdot x^{-3/2} 
\leq \left(0.0004 + \frac{1322}{q}\right) x^{-1/2}.\]
\[
(6 + 22 |\delta|) x^{-1} + (\log q + 6) \cdot
(3 + 17 |\delta|)\cdot x^{-3/2} \leq
\left(0.0007 + \frac{2644}{q}\right) x^{-1/2}.\]
Summing, we obtain
\[\err_{\eta,\chi} \leq 2.485\cdot 10^{-19} +
\frac{1}{\sqrt{x}} \left(\frac{281200}{\sqrt{q}} + 56\right)
\]
for $\eta(t) = t^2 e^{-t^2/2}$ and
\[\err_{\eta,\chi} \leq 2.485\cdot 10^{-19} + 
\frac{1}{\sqrt{x}} \left(\frac{381500}{\sqrt{q}} + 76\right)
\]
for $\eta(t) = t^2 e^{-t^2/2} \ast_M \eta_2(t)$. This proves Theorem \ref{thm:janar}
and Corollary \ref{cor:coprar}.

Now let us work with the smoothing weight $\eta_+$. 
This time around, set $r=150000$ if $q$ is odd, and
$r=300000$ if $q$ is even.
As before, we assume
\[q\leq r,\;\;\;\;\;\;\; |\delta|\leq 4 r/q.\]
We can see that Platt's verification \cite{Plattfresh}, mentioned before,
allows us to take
\[T_0 = H + \frac{250 r}{q},\;\;\;\;\; H = 200,\]
since $T_q$ is always at least this
($T_q = 10^8/q \geq 200 + 7\cdot 10^7/q > 200 + 3.75\cdot 10^7/q$ 
for $q\leq 150000$ odd,
$T_q \geq 200 + 7.5\cdot 10^7/q$ for $q\leq 300000$ even).

Thus,
\[\begin{aligned}
T_0-H & = \frac{250 r}{q} \geq \frac{250 r}{r} = 250,\\
\\
\frac{T_0-H}{\pi \delta} &\geq \frac{250 r}{\pi \delta q} \geq
\frac{250}{4 \pi} = 19.89436\dotsc
\end{aligned}\]
and also
\[T_0 \leq 200 + 250\cdot 150000 \leq 3.751 \cdot 10^7
,\;\;\;\;\;\;\;\;
q T_0 \leq r H + 250 r \leq 1.35 \cdot 10^8.\]
Hence, since $\sqrt{t} e^{-0.1598 t}$ is decreasing on $t$ for 
$t\geq 1/(2\cdot 0.1598)$,
\[\begin{aligned}
11.308 \sqrt{T_0-H} e^{-0.1598 (T_0-H)} &+ 16.147 |\delta| 
e^{-0.1065 \frac{(T_0-H)^2}{(\pi \delta)^2}}
\\ &\leq 7.9854 \cdot 10^{-16} 
 + \frac{4 r}{q} \cdot 7.9814\cdot 10^{-18}\\
&\leq 7.9854 \cdot 10^{-16} 
+ \frac{9.5777 \cdot 10^{-12}}{q}.\end{aligned}\]
Examining (\ref{eq:nochpai}), we get
\[\begin{aligned}&\err_{\eta_+,\chi}(\delta,x)
\leq \log \frac{1.35\cdot 10^8}{2\pi} \cdot
 \left(7.9854 \cdot 10^{-16}  
+ \frac{9.5777 \cdot 10^{-12}}{q}\right)\\
&+ \left(\left(1.634 \log(1.35\cdot 10^8)
 + 12.43\right) \frac{\sqrt{1.35\cdot 10^8}}{\sqrt{q}} + 1.321 \log 300000 
+ 34.51\right) \frac{1}{\sqrt{x}}\\
&+ \left(9+11\cdot \frac{1.2\cdot 10^6}{q}\right) x^{-1} +
(\log 300000) \left(11+ 6
\cdot \frac{1.2\cdot 10^6}{q}\right) x^{-3/2}\\
&\leq 1.3482\cdot 10^{-14}+
\frac{1.617\cdot 10^{-10}}{q} \\ &+ 
\left(\frac{499845}{\sqrt{q}} +
51.17 + \frac{1.32\cdot 10^6}{q \sqrt{x}} + \frac{9}{\sqrt{x}} +
\frac{9.1\cdot 10^7}{q x} + \frac{139}{x}\right) \frac{1}{\sqrt{x}}
\end{aligned}\]
Making the assumption $x\geq 10^{12}$, we obtain 
\[\err_{\eta_+,\chi}(\delta,x)\leq 
1.3482\cdot 10^{-14}+
\frac{1.617\cdot 10^{-10}}{q} 
+ \left(\frac{499900}{\sqrt{q}} + 52\right) \frac{1}{\sqrt{x}}.\]
This proves Theorem \ref{thm:malpor} for general $q$.

Let us optimize things a little more carefully for the trivial
character $\chi_T$. Again, we will make the assumption $x\geq 10^{12}$.
We will also assume, as we did before, that $|\delta|\leq 4 r/q$; this now
gives us $|\delta|\leq 600000$, since $q=1$ and $r=150000$ for $q$ odd.
We will go up to a height $T_0 = H+600000\pi\cdot t$, where $H=200$
and $t\geq 10$. Then
\[\frac{T_0-H}{\pi \delta} = \frac{600000\pi t}{4\pi r} \geq t.\]
Hence
\[\begin{aligned}
11.308 &\sqrt{T_0-H} e^{-0.1598 (T_0 - H)} + 16.147 |\delta| e^{-0.1065
  \frac{(T_0-H)^2}{(\pi \delta)^2}}\\  &\leq 
10^{-1300000} + 9689000 e^{-0.1065
  t^2}.\end{aligned}\]
Looking at (\ref{eq:nochpai}), we get
\[\begin{aligned}
\err_{\eta_+,\chi_T}(\delta,x) &\leq \log \frac{T_0}{2\pi}
\cdot \left(10^{-1300000} + 9689000 e^{-0.1065
  t^2}\right)\\
&+ ((1.634 \log T_0 + 12.43) \sqrt{T_0} + 34.51) x^{-1/2} +
6600009 x^{-1}   .\end{aligned}\]
The value $t=20$ seems good enough; we choose it because it is not far
from optimal for $x\sim 10^{27}$.
We get that $T_0 = 12000000 \pi + 200$; since $T_0<10^8$,
we are within the range of the computations in \cite{Plattfresh} (or
for that matter \cite{Wed} or \cite{Plattpi}). We obtain
\[\err_{\eta_+,\chi_T}(\delta,x) \leq 4.772 \cdot 10^{-11} +  
\frac{251400}{\sqrt{x}}.\]

Lastly, let us look at the sum estimated in (\ref{eq:horrorshow}).
Here it will be enough to go up to just $T_0 = 2 H + \max(50,H/4) = 450$,
where, as before, $H = 200$. Of course, the verification of the
zeros of the Riemann zeta function does go that far; as we already said, 
it goes until $10^8$ (or rather more: see \cite{Wed} and \cite{Plattpi}).
We make, again, the assumption $x\geq 10^{12}$.
We look at (\ref{eq:galo}) and obtain that $\err_{\ell_2,\eta_+}$ is at most
\begin{equation}\label{eq:malko}\begin{aligned}
&\left(\left(0.462 \frac{(\log 50)^2}{\log 10^{12}} + 0.909 \log 50\right)\cdot 50
 +
1.71 \left(1 + \frac{\log 50}{\log 10^{12}}\right) \cdot 200 \right) 
e^{-\frac{\pi}{4} 50}\\
&+ (2.445 \sqrt{450} \log 450 + 50.04) \cdot x^{-1/2}\\
&\leq 5.123\cdot 10^{-15} + \frac{366.91}{\sqrt{x}}.\end{aligned}\end{equation}


It remains only to estimate the integral in (\ref{eq:horrorshow}).
First of all,
\[\begin{aligned}
\int_0^\infty &\eta_+^2(t) \log x t \;dt 
= \int_0^\infty \eta_\circ^2(t) \log x t\; dt
\\ &+ 2 \int_0^\infty (\eta_+(t)-\eta_\circ(t)) \eta_\circ(t) \log xt \;dt +
\int_0^\infty (\eta_+(t)-\eta_\circ(t))^2 \log xt \;dt.\end{aligned}\]
The main term will be given by 
\[\begin{aligned}\int_0^\infty \eta_\circ^2(t) \log x t\; dt &= 
\left(0.64020599736635 + O\left(10^{-14}\right)\right) \log x \\
&- 0.021094778698867 + O\left(10^{-15}\right),
\end{aligned}\]
where the integrals were computed rigorously using VNODE-LP \cite{VNODELP}.
(The integral $\int_0^\infty \eta_\circ^2(t) dt$ can also be computed
symbolically.) By Cauchy-Schwarz and the triangle inequality,
\[\begin{aligned}
\int_0^\infty &(\eta_+(t)-\eta_\circ(t)) \eta_\circ(t) \log xt \;dt \leq
|\eta_+-\eta_\circ|_2 |\eta_\circ(t) \log xt|_2\\
&\leq |\eta_+-\eta_\circ|_2 (|\eta_\circ|_2 \log x + |\eta_\circ\cdot \log|_2)\\
&\leq \frac{274.86}{H^{7/2}} (0.80013 \log x + 0.214)\\
&\leq 1.944 \cdot 10^{-6}\cdot \log x + 5.2\cdot 10^{-7},
\end{aligned} \]
where we are using (\ref{eq:impath}) and evaluate $|\eta_\circ \cdot \log|_2$
rigorously as above.
By (\ref{eq:impath}) and (\ref{eq:lozhka}),
\[\begin{aligned}\int_0^\infty (\eta_+(t)-\eta_\circ(t))^2 \log xt\; dt &\leq
\left(\frac{274.86}{H^{7/2}}\right)^2 \log x + \frac{27428}{H^7} \\
&\leq 5.903 \cdot 10^{-12}\cdot \log x + 2.143\cdot 10^{-12}. \end{aligned}\]
We conclude that
\begin{equation}\label{eq:kokord}\begin{aligned}
\int_0^\infty &\eta_+^2(t) \log x t \;dt \\ &= 
(0.640206 + O^*(1.95\cdot 10^{-6})) \log x -
0.021095 +O^*(5.3\cdot 10^{-7})
\end{aligned}\end{equation}
We add to this the error term $5.123\cdot 10^{-15} + 366.91/\sqrt{x}$ from
(\ref{eq:malko}), and simplify using the assumption $x\geq 10^{12}$. We obtain:
\begin{equation}\label{eq:komary}\begin{aligned}
\sum_{n=1}^\infty \Lambda(n) (\log n) \eta_+^2(n/x) &= 
0.640206 x \log x - 0.021095 x
\\ &+ O^*\left(2\cdot 10^{-6} x \log x + 366.91 \sqrt{x} \log x\right),
\end{aligned}\end{equation}
and so Prop.~\ref{prop:konechno} gives us Proposition \ref{prop:malheur}.

As we can see, 
the relatively large error term $2\cdot 10^{-6}$ comes from the fact that
we have wanted to give the main term in (\ref{eq:horrorshow})
as an explicit constant, rather than as an integral. This is satisfactory;
Prop.~\ref{prop:malheur} is an auxiliary result that will be
 needed for one specific purpose in Part \ref{part:concl}, as opposed to Thms.~\ref{thm:gowo1}--\ref{thm:malpor}, which, while crucial
for Part \ref{part:concl}, are also of general applicability and interest.

\part{The integral over the circle}\label{part:concl}
\chapter{The integral over the major arcs}\label{chap:prewo}

Let
\begin{equation}\label{eq:fellok}
S_\eta(\alpha,x) = \sum_n \Lambda(n) e(\alpha n) \eta(n/x),\end{equation}
where $\alpha \in \mathbb{R}/\mathbb{Z}$, $\Lambda$ is the von Mangoldt
function and $\eta:\mathbb{R}\to \mathbb{C}$ is of fast enough decay
for the sum to converge.

Our ultimate goal is to bound from below
\begin{equation}\label{eq:tripsum}
\sum_{n_1+n_2+n_3 = N} \Lambda(n_1) \Lambda(n_2) \Lambda(n_3)
\eta_1(n_1/x) \eta_2(n_2/x) \eta_3(n_3/x),\end{equation}
where $\eta_1, \eta_2, \eta_3:\mathbb{R}\to \mathbb{C}$. 
Once we know that this is neither zero nor very close to zero,
we will know that it is possible to write $N$ 
as the sum of three primes $n_1$, $n_2$, $n_3$ in at least one way;
that is, we will have proven the ternary Goldbach conjecture.

As can be readily seen, (\ref{eq:tripsum}) equals
\begin{equation}\label{eq:osto}\int_{\mathbb{R}/\mathbb{Z}} S_{\eta_1}(\alpha,x) 
S_{\eta_2}(\alpha,x) S_{\eta_3}(\alpha,x) 
e(-N \alpha)\; d\alpha .\end{equation}
In the circle method, the set $\mathbb{R}/\mathbb{Z}$ gets partitioned into
the set of {\em major arcs} $\mathfrak{M}$ and the set of {\em minor arcs}
$\mathfrak{m}$; the contribution of each of the two sets to the integral
(\ref{eq:osto}) is evaluated separately.

Our objective here is to treat the major arcs: we wish to estimate
\begin{equation}\label{eq:russie}\int_{\mathfrak{M}} 
S_{\eta_1}(\alpha,x) S_{\eta_2}(\alpha,x) S_{\eta_3}(\alpha,x)
 e(-N \alpha) d\alpha\end{equation}
for $\mathfrak{M} = \mathfrak{M}_{\delta_0,r}$, where
\begin{equation}\label{eq:majdef}
\mathfrak{M}_{\delta_0,r} = \mathop{\bigcup_{q\leq r}}_{\text{$q$ odd}} \mathop{\bigcup_{a \mo q}}_{(a,q)=1}
 \left(\frac{a}{q} - \frac{\delta_0 r}{2 q x}, \frac{a}{q} + \frac{\delta_0 r}{2 q x}\right) \cup 
\mathop{\bigcup_{q\leq 2 r}}_{\text{$q$ even}} \mathop{\bigcup_{a \mo q}}_{(a,q)=1}
\left(\frac{a}{q} - \frac{\delta_0 r}{q x}, 
      \frac{a}{q} + \frac{\delta_0 r}{q x}\right) 
\end{equation}
and $\delta_0>0$, $r\geq 1$ are given.

In other words, our major arcs will be few (that is, a constant number)
and narrow. While \cite{MR1932763} used relatively narrow major arcs
as well, their number, as in all previous proofs of Vinogradov's result,
was not bounded by a constant. (In his proof of the five-primes theorem, 
\cite{Tao} is able to take a single major arc around $0$; this is not possible 
here.)


What we are about to see is the general major-arc setup. This is
naturally the place where the overlap with the existing literature is largest.
Two important differences can nevertheless be singled out.
\begin{itemize}
\item
 The most obvious one
is the presence of smoothing. At this point, it improves and simplifies error
terms, but it also means that we will later need estimates for exponential sums
on major arcs, and not just at the middle of each major arc. (If there is
smoothing, we cannot use summation by parts to reduce the problem of estimating 
sums to a problem of counting primes in arithmetic progressions, or weighted
by characters.)
\item Since our $L$-function estimates for exponential sums will give bounds
that are better than the trivial one by only a constant -- even if it is a 
rather large constant -- we need to be especially careful when estimating error
terms, finding cancellation when possible.
\end{itemize}

\section{Decomposition of $S_\eta$ by characters} 
What follows is largely classical; cf. 
 \cite{MR1555183} or, say, \cite[\S 26]{MR0217022}. The only difference
from the literature lies in the treatment of $n$ non-coprime to $q$,
and the way in which we show that our exponential sum (\ref{eq:beatit}) 
is equal to a linear combination of twisted sums $S_{\eta,\chi^*}$
over {\em primitive} characters $\chi^*$. (Non-primitive characters would give
us $L$-functions with some zeroes inconveniently placed on the line $\Re(s)=0$.)

Write $\tau(\chi,b)$ for the Gauss sum 
\begin{equation}\label{eq:caundal}
\tau(\chi,b) = \sum_{a \mo q} \chi(a) e(ab/q)\end{equation}
associated to a $b\in \mathbb{Z}/q\mathbb{Z}$ and a
Dirichlet character $\chi$ with modulus $q$.
 We let $\tau(\chi) = \tau(\chi,1)$.
If $(b,q)=1$, then $\tau(\chi,b) = \chi(b^{-1}) \tau(\chi)$.

Recall that $\chi^*$ denotes the primitive character inducing a given Dirichlet
character $\chi$.
Writing $\sum_{\chi \mo q}$ for a sum over all characters $\chi$ of
$(\mathbb{Z}/q \mathbb{Z})^*$), we see that,
for any $a_0\in \mathbb{Z}/q
\mathbb{Z}$,
\begin{equation}\label{eq:mouche}\begin{aligned}
\frac{1}{\phi(q)} &\sum_{\chi \mo q} \tau(\overline{\chi},b) \chi^*(a_0) =
\frac{1}{\phi(q)} 
\sum_{\chi \mo q} \mathop{\sum_{a \mo q}}_{(a,q)=1}
\overline{\chi(a)} e(ab/q) \chi^*(a_0) \\&= 
\mathop{\sum_{a \mo q}}_{(a,q)=1}
\frac{e(ab/q)}{\phi(q)}  \sum_{\chi \mo q} \chi^*(a^{-1} a_0) = 
\mathop{\sum_{a \mo q}}_{(a,q)=1}
\frac{e(ab/q)}{\phi(q)}  \sum_{\chi \mo q'} \chi(a^{-1} a_0),
\end{aligned}\end{equation}
where $q'=q/\gcd(q,a_0^\infty)$. Now,
$\sum_{\chi \mo q'} \chi(a^{-1} a_0)=0$ unless $a = a_0$ (in which case
$\sum_{\chi \mo q'} \chi(a^{-1} a_0)=\phi(q')$). Thus, (\ref{eq:mouche})
equals 
\[\begin{aligned}
\frac{\phi(q')}{\phi(q)}
&\mathop{\mathop{\sum_{a \mo q}}_{(a,q)=1}}_{a\equiv a_0 \mo q'}
e(ab/q) =  
\frac{\phi(q')}{\phi(q)}
\mathop{\sum_{k \mo q/q'}}_{(k,q/q')=1}
e\left(\frac{(a_0+kq') b}{q}\right)\\
&=  
\frac{\phi(q')}{\phi(q)} e\left(\frac{a_0 b}{q}\right)
\mathop{\sum_{k \mo q/q'}}_{(k,q/q')=1}
e\left(\frac{k b}{q/q'}\right) = 
\frac{\phi(q')}{\phi(q)} e\left(\frac{a_0 b}{q}\right)
\mu(q/q')
\end{aligned}\]
provided that $(b,q)=1$. (We are evaluating a {\em Ramanujan sum} in the
last step.)
Hence, for $\alpha = a/q +\delta/x$, $q\leq x$, $(a,q)=1$,
\[
\frac{1}{\phi(q)} \sum_\chi \tau(\overline{\chi},a) 
\sum_n \chi^*(n) \Lambda(n) e(\delta n/x) \eta(n/x)\]
equals
\[
\sum_n \frac{\mu((q,n^\infty))}{\phi((q,n^\infty))}
\Lambda(n) e(\alpha n) \eta(n/x).\]
Since $(a,q)=1$, $\tau(\overline{\chi},a)= \chi(a) \tau(\overline{\chi})$.
The factor $\mu((q,n^\infty))/\phi((q,n^\infty))$ equals $1$ when $(n,q)=1$;
the absolute value of the factor is at most $1$ for every $n$.
Clearly
\[\mathop{\sum_n}_{(n,q)\ne 1} \Lambda(n) \eta\left(\frac{n}{x}\right) = 
\sum_{p|q} \log p \sum_{\alpha\geq 1} \eta\left(\frac{p^\alpha}{x}\right).\]
Recalling the definition (\ref{eq:fellok}) of
$S_\eta(\alpha,x)$, we conclude that

\begin{equation}\label{eq:beatit}\begin{aligned}
S_{\eta}(\alpha,x) =  \frac{1}{\phi(q)} \sum_{\chi \mo q} 
\chi(a) \tau(\overline{\chi})
 S_{\eta,\chi^*}\left(\frac{\delta}{x},x\right) + O^*\left(
2 \sum_{p|q} \log p \sum_{\alpha\geq 1} \eta\left(\frac{p^\alpha}{x}\right)
\right),
\end{aligned}\end{equation}
where
\begin{equation}\label{eq:shangh}S_{\eta,\chi}(\beta,x) = 
\sum_n \Lambda(n) \chi(n) e(\beta n) \eta(n/x) .\end{equation}

Hence
$S_{\eta_1}(\alpha,x) S_{\eta_2}(\alpha,x) S_{\eta_3}(\alpha,x) e(-N\alpha)$ equals
\begin{equation}\label{eq:orgor}\begin{aligned}
\frac{1}{\phi(q)^3} \sum_{\chi_1} \sum_{\chi_2} \sum_{\chi_3} 
 &\tau(\overline{\chi_1}) \tau(\overline{\chi_2})
 \tau(\overline{\chi_3}) \chi_1(a) \chi_2(a) \chi_3(a) e(-N a/q)\\
 &\cdot S_{\eta_1,\chi_1^*}(\delta/x,x) S_{\eta_2,\chi_2^*}(\delta/x,x) 
S_{\eta_3,\chi_3^*}(\delta/x,x) e(-\delta N/x) \end{aligned}\end{equation}
plus an error term of absolute value at most
\begin{equation}\label{eq:joko}
2 \sum_{j=1}^3 \prod_{j'\ne j} |S_{\eta_{j'}}(\alpha,x)|  
\sum_{p|q} \log p \sum_{\alpha\geq 1} \eta_j\left(\frac{p^\alpha}{x}\right)
.\end{equation}
We will later see that the integral of (\ref{eq:joko}) over $S^1$
 is negligible -- for our choices of $\eta_j$, it will, in fact, be
of size $O(x (\log x)^A)$, $A$ a constant. 
The error term $O(x (\log x)^A)$ should
be compared to the main term, which will be
 of size about a constant times $x^2$.

In (\ref{eq:orgor}), we
have reduced our problems to estimating $S_{\eta,\chi}(\delta/x,x)$
for $\chi$ {\em primitive}; a more obvious way of reaching the same goal
would have made (\ref{eq:joko}) worse
by a factor of about $\sqrt{q}$. 



\section{The integral over the major arcs: the main term}
We are to estimate the integral (\ref{eq:russie}), where the
major arcs $\mathfrak{M}_{\delta_0,r}$ are defined as in (\ref{eq:majdef}).
We will use $\eta_1 = \eta_2 = \eta_+$, $\eta_3(t) = \eta_\ast(\varkappa t)$, 
where $\eta_+$ and $\eta_\ast$ will be set later. 


We can write 
\begin{equation}\label{eq:glenkin}\begin{aligned}
S_{\eta,\chi}(\delta/x,x) = S_{\eta}(\delta/x,x) &= 
\int_0^\infty \eta(t/x) e(\delta t/x) dt + O^*(\err_{\eta,\chi}(\delta,x))\cdot x\\
&= \widehat{\eta}(- \delta) \cdot x + O^*(\err_{\eta,\chi_T}(\delta,x))\cdot x
\end{aligned} \end{equation}
for $\chi=\chi_T$ the trivial character, and
\begin{equation}\label{eq:brahms}
S_{\eta,\chi}(\delta/x) = O^*(\err_{\eta,\chi}(\delta,x)) \cdot x\end{equation}
for $\chi$ primitive and non-trivial. The estimation of the error 
terms $\err$ will come later; let us focus on (a) obtaining the contribution
of the main term, (b) using estimates on the error terms efficiently.

{\em The main term: three principal characters.}
The main contribution will be given by the term in (\ref{eq:orgor})
with
$\chi_1 = \chi_2 = \chi_3 = \chi_0$, where $\chi_0$ is the principal
character mod $q$.

The sum $\tau(\chi_0,n)$ is a {\em Ramanujan sum};
as is well-known (see, e.g., \cite[(3.2)]{MR2061214}),
\begin{equation}\label{eq:selb}\tau(\chi_0,n) = \sum_{d|(q,n)} \mu(q/d) d.
\end{equation}
This simplifies to $\mu(q/(q,n)) \phi((q,n))$ for $q$ square-free.
The special case $n=1$ gives us that $\tau(\chi_0) = \mu(q)$.

Thus, the term in (\ref{eq:orgor}) with 
$\chi_1 = \chi_2 = \chi_3 = \chi_0$ equals
\begin{equation}\label{eq:lookatme}\frac{e(-N a/q)}{\phi(q)^3}  
\mu(q)^3 S_{\eta_+,\chi_0^*}(\delta/x,x)^2 S_{\eta_*,\chi_0^*}(\delta/x,x) 
e(-\delta N/x),\end{equation}
where, of course, $S_{\eta,\chi_0^*}(\alpha,x) = S_\eta(\alpha,x)$ (since
$\chi_0^*$ is the trivial character). 
Summing (\ref{eq:lookatme}) for $\alpha = a/q+\delta/x$
and $a$ going over all residues mod $q$ coprime to $q$, 
we obtain
\[\frac{\mu\left(\frac{q}{(q,N)}\right) \phi((q,N))}{\phi(q)^3} \mu(q)^3
S_{\eta_+, \chi_0^*}(\delta/x,x)^2 S_{\eta_*, \chi_0^*}(\delta/x,x) 
e(-\delta N/x).\]
The integral of (\ref{eq:lookatme}) over all of $\mathfrak{M} =
\mathfrak{M}_{\delta_0,r}$ (see (\ref{eq:majdef})) thus equals
\begin{equation}\label{eq:henki}\begin{aligned}
&\mathop{\sum_{q\leq r}}_{\text{$q$ odd}} \frac{\phi((q,N))}{\phi(q)^3}  \mu(q)^2 \mu((q,N))
\int_{-\frac{\delta_0 r}{2 q x}}^{\frac{\delta_0 r}{2 q x}} 
S_{\eta_+, \chi_0^*}^2(\alpha,x) S_{\eta_*, \chi_0^*}(\alpha,x) 
e(-\alpha N) d\alpha \\
+ &\mathop{\sum_{q\leq 2 r}}_{\text{$q$ even}} 
\frac{\phi((q,N))}{\phi(q)^3}  \mu(q)^2 \mu((q,N))
\int_{-\frac{\delta_0 r}{q x}}^{\frac{\delta_0 r}{q x}} 
S_{\eta_+, \chi_0^*}^2(\alpha,x) S_{\eta_*, \chi_0^*}(\alpha,x) 
e(-\alpha N) d\alpha .
\end{aligned}\end{equation}

The main term in (\ref{eq:henki}) is
\begin{equation}\label{eq:arger}\begin{aligned}
&x^3 \cdot \mathop{\sum_{q\leq r}}_{\text{$q$ odd}} 
\frac{\phi((q,N))}{\phi(q)^3}  \mu(q)^2 \mu((q,N))
\int_{-\frac{\delta_0 r}{2 q x}}^{\frac{\delta_0 r}{2 q x}} 
(\widehat{\eta_+}(-\alpha x))^2 \widehat{\eta_*}(-\alpha x) e(-\alpha N) d\alpha\\
+ &x^3 \cdot \mathop{\sum_{q\leq 2 r}}_{\text{$q$ even}}
 \frac{\phi((q,N))}{\phi(q)^3}  \mu(q)^2 \mu((q,N))
\int_{-\frac{\delta_0 r}{q x}}^{\frac{\delta_0 r}{q x}} 
(\widehat{\eta_+}(-\alpha x))^2 \widehat{\eta_*}(-\alpha x) e(-\alpha N) d\alpha
.\end{aligned}
\end{equation}

We would like to complete both the sum and the integral. Before, we should
say that we will want to be able to use smoothing functions $\eta_+$ whose
Fourier transforms are not easy to deal with directly. All we want to require
is that there be a smoothing function $\eta_\circ$, easier to deal with,
such that $\eta_\circ$ be close to $\eta_+$ in $\ell_2$ norm.

Assume, then, that 
\[|\eta_+-\eta_\circ|_2\leq \epsilon_0 |\eta_\circ|,\]
where $\eta_\circ$ is thrice differentiable outside finitely many points
and satisfies $\eta_\circ^{(3)} \in L_1$. Then (\ref{eq:arger})
equals
\begin{equation}\label{eq:pasture}\begin{aligned}
&x^3 \cdot \mathop{\sum_{q\leq r}}_{\text{$q$ odd}} 
\frac{\phi((q,N))}{\phi(q)^3}  \mu(q)^2 \mu((q,N))
\int_{-\frac{\delta_0 r}{2 q x}}^{\frac{\delta_0 r}{2 q x}} 
(\widehat{\eta_\circ}(-\alpha x))^2 \widehat{\eta_*}(-\alpha x) 
e(-\alpha N) 
d\alpha\\
+ &x^3 \cdot \mathop{\sum_{q\leq 2 r}}_{\text{$q$ even}} 
\frac{\phi((q,N))}{\phi(q)^3}  \mu(q)^2 \mu((q,N))
\int_{-\frac{\delta_0 r}{q x}}^{\frac{\delta_0 r}{q x}} 
(\widehat{\eta_\circ}(-\alpha x))^2 \widehat{\eta_*}(-\alpha x) 
e(-\alpha N) 
d\alpha.
\end{aligned}\end{equation}
plus 
\begin{equation}\label{eq:pommes}
O^*\left(x^2 \cdot \sum_{q} \frac{\mu(q)^2}{\phi(q)^2}
\int_{-\infty}^{\infty} |(\widehat{\eta_+}(-\alpha))^2 - 
(\widehat{\eta_\circ}(-\alpha))^2| |\widehat{\eta_*}(-\alpha)| d\alpha\right).
\end{equation}
Here (\ref{eq:pommes}) is bounded by $2.82643 x^2$ (by (\ref{eq:massacre}))
times
\[\begin{aligned}
|\widehat{\eta_*}(-\alpha)|_\infty &\cdot
\sqrt{\int_{-\infty}^{\infty} |\widehat{\eta_+}(-\alpha)-
\widehat{\eta_\circ}(-\alpha)|^2 d\alpha \cdot 
\int_{-\infty}^{\infty} |\widehat{\eta_+}(-\alpha) +
\widehat{\eta_\circ}(-\alpha)|^2 d\alpha}\\
&\leq |\eta_*|_1 \cdot |\widehat{\eta_+}-\widehat{\eta_\circ}|_2
|\widehat{\eta_+}+\widehat{\eta_\circ}|_2 =
|\eta_*|_1 \cdot |\eta_+-\eta_\circ|_2
|\eta_++\eta_\circ|_2 \\ &\leq
|\eta_*|_1 \cdot |\eta_+-\eta_\circ|_2
(2|\eta_\circ|_2 + |\eta_+-\eta_\circ|_2) =
|\eta_*|_1 |\eta_\circ|_2^2 \cdot (2 + \epsilon_0) \epsilon_0.
\end{aligned}\]
Now, (\ref{eq:pasture}) equals
\begin{equation}\label{eq:rusko}\begin{aligned}
x^3 
&\int_{-\infty}^{\infty}
(\widehat{\eta_\circ}(-\alpha x))^2 \widehat{\eta_*}(-\alpha x) 
e(-\alpha N) \mathop{\sum_{\frac{q}{(q,2)}\leq \min\left(
\frac{\delta_0 r}{2 |\alpha| x},r\right)}}_{\mu(q)^2 = 1}
\frac{\phi((q,N))}{\phi(q)^3} \mu((q,N))
d\alpha \\
= x^3 &\int_{-\infty}^\infty
(\widehat{\eta_\circ}(-\alpha x))^2 \widehat{\eta_*}(-\alpha x) 
e(-\alpha N) d\alpha \cdot \left(
\sum_{q\geq 1}
\frac{\phi((q,N))}{\phi(q)^3}  \mu(q)^2 \mu((q,N))
\right)\\
- x^3 &\int_{-\infty}^{\infty}
(\widehat{\eta_\circ}(-\alpha x))^2 \widehat{\eta_*}(-\alpha x) 
e(-\alpha N) 
\mathop{\sum_{\frac{q}{(q,2)}> \min\left(
\frac{\delta_0 r}{2 |\alpha| x},r\right)}}_{\mu(q)^2=1}
\frac{\phi((q,N))}{\phi(q)^3}  \mu((q,N))
d\alpha .\end{aligned}\end{equation}
The last line in (\ref{eq:rusko}) is bounded\footnote{This is obviously
crude, in that we are bounding $\phi((q,N))/\phi(q)$ by $1$. We are doing so
in order to avoid a potentially harmful dependence on $N$.}
 by
\begin{equation}\label{eq:boussole}x^2 |\widehat{\eta_*}|_\infty
\int_{-\infty}^{\infty}
|\widehat{\eta_\circ}(-\alpha)|^2  \sum_{\frac{q}{(q,2)}
> \min\left(\frac{\delta_0 r}{2|\alpha|},r\right)}
\frac{\mu(q)^2}{\phi(q)^2} d\alpha .\end{equation}
By (\ref{eq:madge}) (with $k=3$), (\ref{eq:gat1o}) and (\ref{eq:gat1e}), 
this is at most
\[\begin{aligned}x^2 &|\eta_*|_1 
\int_{-\delta_0/2}^{\delta_0/2} |\widehat{\eta_\circ}(-\alpha)|^2
\frac{4.31004}{r} d\alpha\\ &+ 
2 x^2 |\eta_*|_1 
\int_{\delta_0/2}^{\infty} \left(\frac{|\eta_\circ^{(3)}|_1}{
(2\pi \alpha)^3}\right)^2  
\frac{8.62008 |\alpha|}{\delta_0 r} d\alpha\\
&\leq |\eta_*|_1 \left(4.31004 |\eta_\circ|_2^2 +
0.00113 \frac{|\eta_\circ^{(3)}|_1^2}{\delta_0^5}\right)
 \frac{x^2}{r}
.\end{aligned}\]

It is easy to see that
\[\sum_{q\geq 1} \frac{\phi((q,N))}{\phi(q)^3}  \mu(q)^2 \mu((q,N))
= \prod_{p|N} \left(1 - \frac{1}{(p-1)^2}\right)
\cdot \prod_{p\nmid N} \left(1 + \frac{1}{(p-1)^3}\right).\]

Expanding the integral implicit in the definition of $\widehat{f}$,
\begin{equation}\label{eq:hosto}\begin{aligned}
\int_\infty^\infty &(\widehat{\eta_\circ}(-\alpha x))^2 \widehat{\eta_*}(- \alpha x)
 e(-\alpha N) d\alpha =\\ 
&\frac{1}{x} \int_0^\infty \int_0^\infty \eta_\circ(t_1)  \eta_\circ(t_2) 
\eta_*\left(\frac{N}{x}-(t_1+t_2)\right) 
dt_1 dt_2.
\end{aligned}\end{equation}
(This is standard. One rigorous way to obtain (\ref{eq:hosto}) is to approximate the integral over
$\alpha\in (-\infty,\infty)$ by an integral with a smooth weight, at
different scales; as the scale becomes broader, the Fourier transform
of the weight approximates (as a distribution) the $\delta$ function.
Apply Plancherel.)



Hence, (\ref{eq:arger}) equals
\begin{equation}\label{eq:notspel}\begin{aligned}
x^2 &\cdot \int_0^\infty \int_0^\infty \eta_\circ(t_1) \eta_\circ(t_2) 
\eta_*\left(\frac{N}{x}-(t_1+t_2)\right) 
dt_1 dt_2 \\ &\cdot \prod_{p|N} \left(1 - \frac{1}{(p-1)^2}\right)
\cdot \prod_{p\nmid N} \left(1 + \frac{1}{(p-1)^3}\right).\end{aligned}
\end{equation}
(the main term) plus  
\begin{equation}\label{eq:stev}
\left(2.82643 |\eta_\circ|_2^2 (2+\epsilon_0)\cdot \epsilon_0
+ \frac{4.31004 |\eta_\circ|_2^2 +
0.00113 \frac{|\eta_\circ^{(3)}|_1^2}{\delta_0^5}}{r} \right) |\eta_*|_1 x^2
\end{equation}

Here (\ref{eq:notspel}) is
 just as in the classical case \cite[(19.10)]{MR2061214}, except
for the fact that a factor of $1/2$ has been replaced by a double integral.
Later, in chapter \ref{chap:rossini}, 
we will see how to choose our smoothing functions (and $x$,
in terms of $N$) so as to make the double integral as large as possible
in comparison with the error terms. This is an important optimization.
(We already had a first discussion of this in the introduction; see
(\ref{eq:korl}) and what follows.)

What remains to estimate is the contribution of all the terms of the
form $\err_{\eta,\chi}(\delta,x)$ in (\ref{eq:glenkin}) and (\ref{eq:brahms}).
Let us first deal with another matter -- bounding the $\ell_2$ norm
of $|S_\eta(\alpha,x)|^2$ over the major arcs.

\section{The $\ell_2$ norm over the major arcs} 
We can always bound the integral of
$|S_\eta(\alpha,x)|^2$ on the whole circle by Plancherel. If we only want
the integral on certain arcs, we use the bound in Prop.~\ref{prop:bellen}
(based on work by Ramar\'e).
If these arcs are really the major arcs -- that is,
the arcs on which we have useful analytic estimates -- then we can hope
to get better bounds using $L$-functions. This will be useful both
to estimate the error terms in this section
 and to make the use of Ramar\'e's bounds more efficient later.

By (\ref{eq:beatit}),
\[\begin{aligned}
&\mathop{\sum_{a \mo q}}_{\gcd(a,q)=1} \left|S_{\eta}\left(
\frac{a}{q}+ \frac{\delta}{x},\chi\right)\right|^2 \\
&= \frac{1}{\phi(q)^2} \sum_\chi \sum_{\chi'} \tau(\overline{\chi})
\overline{\tau(\overline{\chi'})} \left(\mathop{\sum_{a \mo q}}_{\gcd(a,q)=1} \chi(a) \overline{\chi'(a)}
\right)\cdot S_{\eta,\chi^*}(\delta/x,x) \overline{S_{\eta,\chi'^*}(\delta/x,x)}
\\ &+ O^*\left(2 (1+\sqrt{q}) (\log x)^2 |\eta|_\infty 
\max_\alpha |S_\eta(\alpha,x)| + \left((1+\sqrt{q}) (\log x)^2 |\eta|_\infty\right)^2\right)\\
&= \frac{1}{\phi(q)} \sum_\chi |\tau(\overline{\chi})|^2
|S_{\eta,\chi^*}(\delta/x,x)|^2 + K_{q,1} (2 |S_\eta(0,x)| + K_{q,1}),
\end{aligned}\]
where
\[K_{q,1} = (1+\sqrt{q}) (\log x)^2 |\eta|_\infty .\]
As is well-known (see, e.g., \cite[Lem. 3.1]{MR2061214})
\[\tau(\chi) = \mu\left(\frac{q}{q^*}\right)
\chi^*\left(\frac{q}{q^*}\right)
\tau(\chi^*),\]
where $q^*$ is the modulus of $\chi^*$ (i.e., the conductor of $\chi$), and
 \[|\tau(\chi^*)| = \sqrt{q^*}. 
\]
Using the expressions (\ref{eq:glenkin}) and (\ref{eq:brahms}), we obtain
 \[\begin{aligned}&\mathop{\sum_{a \mo q}}_{(a,q)=1}
 \left|S_{\eta}\left(
\frac{a}{q}+ \frac{\delta}{x},x\right)\right|^2 =  
\frac{\mu^2(q)}{\phi(q)} 
\left|\widehat{\eta}(-\delta) x
 + O^*\left(\err_{\eta,\chi_T}(\delta,x) \cdot x\right)\right|^2\\ &+
\frac{1}{\phi(q)} \left(\sum_{\chi\ne \chi_T}  \mu^2\left(\frac{q}{q^*}\right) q^*
\cdot O^*\left(|\err_{\eta,\chi}(\delta,x)|^2 x^2\right)\right)
 + K_{q,1} (2 |S_\eta(0,x)| + K_{q,1})
\\ &= \frac{\mu^2(q) x^2}{\phi(q)} 
\left(|\widehat{\eta}(-\delta)|^2 +
O^*\left(\left|\err_{\eta,\chi_T}(\delta,x)
(2 |\eta|_1 + \err_{\eta,\chi_T}(\delta,x))\right|\right)
\right) \\ &+ 
O^*\left(\max_{\chi\ne \chi_T} q^* |\err_{\eta,\chi^*}(\delta,x)|^2 x^2 + K_{q,2} x\right),\end{aligned}\]
where $K_{q,2} = K_{q,1} (2 |S_\eta(0,x)|/x + K_{q,1}/x)$.

Thus, the integral of $|S_\eta(\alpha,x)|^2$ over
$\mathfrak{M}$ (see (\ref{eq:majdef})) is
\begin{equation}\label{eq:juto}\begin{aligned}
&\mathop{\sum_{q\leq r}}_{\text{$q$ odd}} \mathop{\sum_{a \mo q}}_{(a,q)=1}
\int_{\frac{a}{q}-\frac{\delta_0 r}{2 q x}}^{\frac{a}{q}+\frac{\delta_0
    r}{2 q x}}
\left|S_\eta(\alpha,x)\right|^2 d\alpha +
\mathop{\sum_{q\leq 2r}}_{\text{$q$ even}} \mathop{\sum_{a \mo q}}_{(a,q)=1}
\int_{\frac{a}{q}-\frac{\delta_0 r}{q x}}^{\frac{a}{q}+\frac{\delta_0 r}{q x}}
\left|S_\eta(\alpha,x)\right|^2 d\alpha \\
&=
\mathop{\sum_{q\leq r}}_{\text{$q$ odd}} \frac{\mu^2(q) x^2}{\phi(q)} 
\int_{-\frac{\delta_0 r}{2 q x}}^{\frac{\delta_0 r}{2 q x}}
\left|\widehat{\eta}(-\alpha x)\right|^2 d\alpha
+ 
\mathop{\sum_{q\leq 2 r}}_{\text{$q$ even}} \frac{\mu^2(q) x^2}{\phi(q)} 
\int_{-\frac{\delta_0 r}{q x}}^{\frac{\delta_0 r}{q x}}
\left|\widehat{\eta}(-\alpha x)\right|^2 d\alpha\\
&+
O^*\left(\sum_q \frac{\mu^2(q) x^2}{\phi(q)} \cdot 
\frac{\gcd(q,2) \delta_0 r}{qx} \left(ET_{\eta,\frac{\delta_0 r}{2}} (2|\eta|_1 + 
ET_{\eta,\frac{\delta_0 r}{2}})\right)\right) 
 \\ &+ 
\mathop{\sum_{q\leq r}}_{\text{$q$ odd}} \frac{\delta_0 r x}{q} \cdot O^*\left(
\mathop{\mathop{\max_{\chi \mo q}}_{\chi \ne \chi_T}}_{|\delta|\leq
  \delta_0 r/2 q} q^*
 |\err_{\eta,\chi^*}(\delta,x)|^2  + \frac{K_{q,2}}{x}\right)\\
&+ 
\mathop{\sum_{q\leq 2 r}}_{\text{$q$ even}} \frac{2 \delta_0 r x}{q} \cdot O^*\left(
\mathop{\mathop{\max_{\chi \mo q}}_{\chi \ne \chi_T}}_{|\delta|\leq
  \delta_0 r/q} q^*
 |\err_{\eta,\chi^*}(\delta,x)|^2  + \frac{K_{q,2}}{x}\right)
,\end{aligned}\end{equation}
where \[ET_{\eta,s} = \max_{|\delta|\leq s} |\err_{\eta,\chi_T}(\delta,x)|\]
and $\chi_T$ is the trivial character. If all we want is an upper bound,
we can simply remark that

\[\begin{aligned}
&x \mathop{\sum_{q\leq r}}_{\text{$q$ odd}} \frac{\mu^2(q)}{\phi(q)} 
\int_{-\frac{\delta_0 r}{2 q x}}^{\frac{\delta_0 r}{2 qx }}
\left|\widehat{\eta}(-\alpha x)\right|^2 d\alpha
+ 
x \mathop{\sum_{q\leq 2 r}}_{\text{$q$ even}} \frac{\mu^2(q)}{\phi(q)} 
\int_{-\frac{\delta_0 r}{q x}}^{\frac{\delta_0 r}{q x}}
\left|\widehat{\eta}(-\alpha x)\right|^2 d\alpha\\
&\leq \left(\mathop{\sum_{q\leq r}}_{\text{$q$ odd}} \frac{\mu^2(q)}{\phi(q)} +
\mathop{\sum_{q\leq 2 r}}_{\text{$q$ even}} \frac{\mu^2(q)}{\phi(q)}
\right) |\widehat{\eta}|_2^2 =
2 |\eta|_2^2 
\mathop{\sum_{q\leq r}}_{\text{$q$ odd}} \frac{\mu^2(q)}{\phi(q)}. 
\end{aligned}\]
If we also need a lower bound, we proceed as follows.

Again, we will work with an approximation $\eta_\circ$ such that (a)
$|\eta-\eta_\circ|_2$ is small, (b) $\eta_\circ$ is thrice
differentiable outside finitely many points, (c) $\eta_\circ^{(3)}\in L_1$.
Clearly,
\[\begin{aligned}
&x \mathop{\sum_{q\leq r}}_{\text{$q$ odd}} \frac{\mu^2(q)}{\phi(q)}
 \int_{-\frac{\delta_0 r}{2 q x}}^{\frac{\delta_0 r}{2 q x}}
\left|\widehat{\eta}(-\alpha x)\right|^2 d\alpha \\ &\leq
\mathop{\sum_{q\leq r}}_{\text{$q$ odd}} \frac{\mu^2(q)}{\phi(q)}
 \left(\int_{-\frac{\delta_0 r}{2 q}}^{\frac{\delta_0 r}{2 q}}
\left|\widehat{\eta_\circ}(-\alpha)\right|^2 d\alpha 
+ 2 \langle \left|\widehat{\eta_\circ}\right|,
  \left|\widehat{\eta} - \widehat{\eta_\circ}\right|\rangle 
+ \left|\widehat{\eta} - 
        \widehat{\eta_\circ}\right|_2^2 \right)\\
&=\mathop{\sum_{q\leq r}}_{\text{$q$ odd}} \frac{\mu^2(q)}{\phi(q)}
\int_{-\frac{\delta_0 r}{2 q}}^{\frac{\delta_0 r}{2 q}}
\left|\widehat{\eta_\circ}(-\alpha)\right|^2 d\alpha \\ &+
O^*\left(\frac{1}{2} \log r + 0.85\right) \left(2 
\left|\eta_\circ\right|_2 \left| \eta - \eta_\circ\right|_2 +
 \left|\eta_\circ - \eta\right|_2^2\right),
\end{aligned}\]
where we are using (\ref{eq:marmo}) and isometry.
Also,
\[\mathop{\sum_{q\leq 2 r}}_{\text{$q$ even}} \frac{\mu^2(q)}{\phi(q)}
 \int_{-\frac{\delta_0 r}{q x}}^{\frac{\delta_0 r}{q x}}
\left|\widehat{\eta}(-\alpha x)\right|^2 d\alpha =
\mathop{\sum_{q\leq r}}_{\text{$q$ odd}} \frac{\mu^2(q)}{\phi(q)}
 \int_{-\frac{\delta_0 r}{2 q x}}^{\frac{\delta_0 r}{2 q x}}
\left|\widehat{\eta}(-\alpha x)\right|^2 d\alpha.\]
By (\ref{eq:madge}) and Plancherel,
\[\begin{aligned}\int_{-\frac{\delta_0 r}{2 q}}^{\frac{\delta_0 r}{2 q}}
\left|\widehat{\eta_\circ}(-\alpha)\right|^2 d\alpha
&= \int_{-\infty}^{\infty} \left|\widehat{\eta_\circ}(-\alpha)\right|^2 d\alpha
- O^*\left(2 \int_{\frac{\delta_0 r}{2 q}}^{\infty} \frac{|\eta_\circ^{(3)}|_1^2}{
(2 \pi \alpha)^6} d\alpha\right)\\
&= |\eta_\circ|_2^2 + 
O^*\left(\frac{|\eta_\circ^{(3)}|_1^2 q^5}{5 \pi^6 (\delta_0 r)^5}\right), 
\end{aligned}\] 
Hence
\[
\mathop{\sum_{q\leq r}}_{\text{$q$ odd}} \frac{\mu^2(q)}{\phi(q)}
 \int_{-\frac{\delta_0 r}{2 q}}^{\frac{\delta_0 r}{2 q}}
\left|\widehat{\eta_\circ}(-\alpha)\right|^2 d\alpha = 
|\eta_\circ|_2^2 \cdot \mathop{\sum_{q\leq r}}_{\text{$q$ odd}}
\frac{\mu^2(q)}{\phi(q)} +
O^*\left(\mathop{\sum_{q\leq r}}_{\text{$q$ odd}}
\frac{\mu^2(q)}{\phi(q)} 
\frac{|\eta_\circ^{(3)}|_1^2 q^5}{5 \pi^6 (\delta_0 r)^5}\right).
\]
Using (\ref{eq:gatosbuenos}), we get that
\[\begin{aligned}\mathop{\sum_{q\leq r}}_{\text{$q$ odd}}
\frac{\mu^2(q)}{\phi(q)} 
\frac{|\eta_\circ^{(3)}|_1^2 q^5}{5 \pi^6 (\delta_0 r)^5}
&\leq \frac{1}{r} \mathop{\sum_{q\leq r}}_{\text{$q$ odd}}
\frac{\mu^2(q) q}{\phi(q)} 
\cdot \frac{|\eta_\circ^{(3)}|_1^2}{5 \pi^6 \delta_0^5}\\
&\leq \frac{|\eta_\circ^{(3)}|_1^2}{5 \pi^6 \delta_0^5} \cdot 
\left(0.64787 + \frac{\log r}{4 r} + \frac{0.425}{r}\right).
\end{aligned}\]

Going back to (\ref{eq:juto}), we use (\ref{eq:nagasa}) to bound 
\[\sum_q \frac{\mu^2(q) x^2}{\phi(q)} \frac{\gcd(q,2) \delta_0 r}{q x} \leq
2.59147 \cdot \delta_0 r x.\] We also note that
\[\begin{aligned}
&\mathop{\sum_{q\leq r}}_{\text{$q$ odd}} \frac{1}{q} +
\mathop{\sum_{q\leq 2r}}_{\text{$q$ even}} \frac{2}{q} =
\sum_{q\leq r} \frac{1}{q} - \sum_{q\leq \frac{r}{2}} \frac{1}{2q} +
\sum_{q\leq r} \frac{1}{q} \\ 
&\leq 2 \log e r - \log \frac{r}{2} \leq \log 2 e^2 r.
\end{aligned}\]

We have proven the following result.
\begin{lemma}\label{lem:drujal}
Let $\eta:\lbrack 0,\infty) \to \mathbb{R}$ be in $L_1 \cap L_\infty$.
Let $S_\eta(\alpha,x)$ be as in (\ref{eq:fellok}) and
let $\mathfrak{M}=\mathfrak{M}_{\delta_0,r}$ be as in (\ref{eq:majdef}).
Let $\eta_\circ:\lbrack 0,\infty)\to \mathbb{R}$ be thrice differentiable outside finitely
many points. Assume $\eta_\circ^{(3)}\in L_1$.

Assume $r\geq 182$. Then
\begin{equation}\label{eq:bfpink}\begin{aligned}
\int_{\mathfrak{M}} |S_{\eta}(\alpha,x)|^2 d\alpha &= 
L_{r,\delta_0} x +
O^*\left(5.19 \delta_0 x r \left(ET_{\eta,\frac{\delta_0 r}{2}} \cdot \left(|\eta|_1 + 
\frac{ET_{\eta,\delta_0 r/2}}{2}\right)\right)\right) \\
&+ O^*\left(\delta_0 r (\log 2 e^2 r) \left( x\cdot
E_{\eta,r,\delta_0}^2 
+ K_{r,2}\right)\right),\end{aligned}\end{equation}
where \begin{equation}\label{eq:sreda}\begin{aligned}
E_{\eta,r,\delta_0} &= 
\mathop{\mathop{\max_{\chi \mo q}}_{q\leq r\cdot \gcd(q,2)}}_{|\delta|\leq
  \gcd(q,2) \delta_0 r/2 q}
\sqrt{q^*} |\err_{\eta,\chi^*}(\delta,x)|,\;\;\;\;\;\;\;
ET_{\eta,s} = \max_{|\delta|\leq s} |\err_{\eta,\chi_T}(\delta,x)|,\\
K_{r,2} &= (1+\sqrt{2 r}) (\log x)^2 |\eta|_\infty
(2 |S_{\eta}(0,x)|/x + (1+\sqrt{2 r}) (\log x)^2 |\eta|_\infty/x)
\end{aligned}\end{equation}
and $L_{r,\delta_0}$ satisfies both 
\begin{equation}\label{eq:mardi}
L_{r,\delta_0}\leq 2 |\eta|_2^2 
\mathop{\sum_{q\leq r}}_{\text{$q$ odd}} \frac{\mu^2(q)}{\phi(q)}
\end{equation} 
and 
\begin{equation}\label{eq:chetvyorg}\begin{aligned}
L_{r,\delta_0} &= 2 |\eta_\circ|_2^2 
  \mathop{\sum_{q\leq r}}_{\text{$q$ odd}} \frac{\mu^2(q)}{\phi(q)}
+ O^*(\log r + 1.7) \cdot 
\left(2 
\left|\eta_\circ\right|_2 \left| \eta - \eta_\circ\right|_2 +
 \left|\eta_\circ - \eta\right|_2^2\right)\\
&+ O^*\left(\frac{2 |\eta_\circ^{(3)}|_1^2}{5 \pi^6 \delta_0^5}\right)
\cdot \left(0.64787 + \frac{\log r}{4 r} + \frac{0.425}{r}\right).
\end{aligned}\end{equation}
Here, as elsewhere, $\chi^*$ denotes the primitive character inducing $\chi$,
whereas 
 $q^*$ denotes the modulus of $\chi^*$.
\end{lemma}

The error term $xr ET_{\eta,\delta_0 r}$ 
will be very small, since it will be estimated
using the Riemann zeta function; the error term involving $K_{r,2}$ will
be completely negligible. The term involving
$x r (r+1) E_{\eta,r,\delta_0}^2$; we see that it constrains us to have $|\err_{\eta,\chi}(x,N)|$ less than
a constant times $1/r$ if we do not want the main term in the bound (\ref{eq:bfpink}) to be overwhelmed.

\section{The integral over the major arcs: conclusion}
There are at least two ways we can evaluate
(\ref{eq:russie}). One is to substitute (\ref{eq:orgor}) into 
(\ref{eq:russie}). The disadvantages here are that (a) this can give
rise to pages-long formulae, 
 (b)
this gives error terms proportional to $x r |\err_{\eta,\chi}(x,N)|$, meaning
that, to win, we would have to show that $|\err_{\eta,\chi}(x,N)|$ is much
smaller than $1/r$.
What we will do instead is to use our $\ell_2$ estimate (\ref{eq:bfpink})
in order to bound the contribution of non-principal terms. This will give
us a gain of almost $\sqrt{r}$ on the error terms; in other words,
to win, it will be enough to show later that $|\err_{\eta,\chi}(x,N)|$ is
much smaller than $1/\sqrt{r}$.

The contribution of the error terms in $S_{\eta_3}(\alpha,x)$ (that is,
all terms involving the quantities $\err_{\eta,\chi}$ in expressions 
(\ref{eq:glenkin}) and (\ref{eq:brahms})) to (\ref{eq:russie}) is
\begin{equation}\label{eq:huppert}\begin{aligned}
\mathop{\sum_{q\leq r}}_{\text{$q$ odd}} \frac{1}{\phi(q)} \sum_{\chi_3 \mo q}
 &\tau(\overline{\chi_3}) 
\mathop{\sum_{a \mo q}}_{(a,q)=1}
\chi_3(a) e(-Na/q)\\
&\int_{-\frac{\delta_0 r}{2 q x}}^{\frac{\delta_0 r}{2 q x}}
 S_{\eta_+}(\alpha+a/q,x)^2 \err_{\eta_*,\chi_3^*}(\alpha x,x) e(-N \alpha)
 d\alpha\\
+ \mathop{\sum_{q\leq 2 r}}_{\text{$q$ even}} \frac{1}{\phi(q)} \sum_{\chi_3 \mo q}
 &\tau(\overline{\chi_3}) 
\mathop{\sum_{a \mo q}}_{(a,q)=1}
\chi_3(a) e(-Na/q)\\
&\int_{-\frac{\delta_0 r}{q x}}^{\frac{\delta_0 r}{q x}}
 S_{\eta_+}(\alpha+a/q,x)^2 \err_{\eta_*,\chi_3^*}(\alpha x,x) e(-N \alpha)
 d\alpha.
\end{aligned}\end{equation}
We should also remember 
 the terms in (\ref{eq:joko}); we can integrate them over all of $\mathbb{R}/\mathbb{Z}$,
and obtain that they contribute at most
\[\begin{aligned}
\int_{\mathbb{R}/\mathbb{Z}} &2 \sum_{j=1}^3 \prod_{j'\ne j} |S_{\eta_{j'}}(\alpha,x)|  
\cdot \max_{q\leq r}
\sum_{p|q} \log p \sum_{\alpha\geq 1} \eta_j\left(\frac{p^\alpha}{x}\right)
d\alpha
\\
&\leq 2 \sum_{j=1}^3 \prod_{j'\ne j} |S_{\eta_{j'}}(\alpha,x)|_2  
\cdot \max_{q\leq r} 
\sum_{p|q} \log p \sum_{\alpha\geq 1} \eta_j\left(\frac{p^\alpha}{x}\right)
\\
&=
2 \sum_n \Lambda^2(n) \eta_+^2(n/x) \cdot
\log r \cdot \max_{p\leq r}
\sum_{\alpha\geq 1} \eta_*\left(\frac{p^\alpha}{x}\right)\\
&+
4 \sqrt{ \sum_n \Lambda^2(n) \eta_+^2(n/x)
 \cdot \sum_n \Lambda^2(n) \eta_*^2(n/x)} \cdot
\log r \cdot \max_{p\leq r}
\sum_{\alpha\geq 1} \eta_*\left(\frac{p^\alpha}{x}\right)\\
\end{aligned}\]
by Cauchy-Schwarz and Plancherel.

The absolute value of (\ref{eq:huppert}) is at most
\begin{equation}\label{eq:frainf}\begin{aligned}
\mathop{\sum_{q\leq r}}_{\text{$q$ odd}} \mathop{\sum_{a \mo q}}_{(a,q)=1} 
&\int_{-\frac{\delta_0 r}{2 q x}}^{\frac{\delta_0 r}{2 q x}}
 \left|S_{\eta_+}(\alpha+a/q,x)\right|^2 d\alpha \cdot
\mathop{\max_{\chi \mo q}}_{|\delta|\leq \delta_0 r/2q}
\sqrt{q^*} |\err_{\eta_*,\chi^*}(\delta,x)| \\ +
\mathop{\sum_{q\leq 2 r}}_{\text{$q$ even}} \mathop{\sum_{a \mo q}}_{(a,q)=1} &\int_{-\frac{\delta_0 r}{q x}}^{\frac{\delta_0 r}{q x}}
 \left|S_{\eta_+}(\alpha+a/q,x)\right|^2 d\alpha \cdot
\mathop{\max_{\chi \mo q}}_{|\delta|\leq \delta_0 r/q}
\sqrt{q^*} |\err_{\eta_*,\chi^*}(\delta,x)| \\
&\leq \int_{\mathfrak{M}_{\delta_0,r}}
\left|S_{\eta_+}(\alpha)\right|^2 d\alpha \cdot
\mathop{\mathop{\max_{\chi \mo q}}_{q\leq r\cdot \gcd(q,2)}}_{
|\delta|\leq \gcd(q,2) \delta_0 r/q}
\sqrt{q^*} |\err_{\eta_*,\chi^*}(\delta,x)| 
.\end{aligned}\end{equation}
We can bound the integral of $|S_{\eta_+}(\alpha)|^2$ by 
(\ref{eq:bfpink}).

What about the contribution of the error part of $S_{\eta_2}(\alpha,x)$?
We can obviously proceed in the same way, except that, to avoid double-counting,
$S_{\eta_3}(\alpha,x)$ needs to be replaced by 
\begin{equation}\label{eq:massac}\frac{1}{\phi(q)}
\tau(\overline{\chi_0}) \widehat{\eta_3}(- \delta) \cdot x
= \frac{\mu(q)}{\phi(q)} \widehat{\eta_3}(- \delta) \cdot x,\end{equation}
which is its main term (coming from (\ref{eq:glenkin})).
Instead of having an $\ell_2$ norm as in (\ref{eq:frainf}), we have the
square-root of a product of two squares of $\ell_2$ norms (by Cauchy-Schwarz),
namely, $\int_\mathfrak{M} |S_{\eta_+}^*(\alpha)|^2 d\alpha$ and
\begin{equation}\label{eq:thaddeus}\begin{aligned}
\mathop{\sum_{q\leq r}}_{\text{$q$ odd}} 
&\frac{\mu^2(q)}{\phi(q)^2} \int_{-\frac{\delta_0 r}{2 q x}}^{
\frac{\delta_0 r}{2 q x}} \left|\widehat{\eta_*}(- \alpha x) x\right|^2
d \alpha +
\mathop{\sum_{q\leq 2 r}}_{\text{$q$ even}} 
\frac{\mu^2(q)}{\phi(q)^2} \int_{-\frac{\delta_0 r}{q x}}^{
\frac{\delta_0 r}{q x}} \left|\widehat{\eta_*}(- \alpha x) x\right|^2
d \alpha\\ 
&\leq 
x |\widehat{\eta_*}|_2^2 \cdot \sum_{q} \frac{\mu^2(q)}{\phi(q)^2} .
\end{aligned}\end{equation}
By (\ref{eq:massacre}), the sum over $q$ is at most $2.82643$.

As for the contribution of the error part of $S_{\eta_1}(\alpha,x)$, we
bound it in the same way, using solely the $\ell_2$ norm in 
(\ref{eq:thaddeus}) (and replacing both $S_{\eta_2}(\alpha,x)$ 
and $S_{\eta_3}(\alpha,x)$ by expressions as in (\ref{eq:massac})).

The total of the error terms is thus
\begin{equation}\label{eq:teresa}\begin{aligned}
&x \cdot
\mathop{\mathop{\max_{\chi \mo q}}_{q\leq r\cdot \gcd(q,2)}}_{|\delta|\leq 
\gcd(q,2) \delta_0 r/q}
\sqrt{q^*} \cdot |\err_{\eta_*,\chi^*}(\delta,x)| \cdot A\\ +\; &x \cdot 
\mathop{\mathop{\max_{\chi \mo q}}_{q\leq r\cdot \gcd(q,2)}}_{|\delta|\leq 
\gcd(q,2) \delta_0 r/q}
\sqrt{q^*} \cdot |\err_{\eta_+,\chi^*}(\delta,x)| 
(\sqrt{A} + \sqrt{B_+}) \sqrt{B_*},
\end{aligned}\end{equation}
where $A = (1/x) \int_{\mathfrak{M}} |S_{\eta_+}(\alpha,x)|^2  d\alpha$ (bounded as in
(\ref{eq:bfpink})) and
\begin{equation}
B_* = 2.82643 |\eta_*|_2^2,
 \;\;\;\;\;\;\;\; B_+ = 2.82643 |\eta_+|_2^2.
\end{equation} 

In conclusion, we have proven
\begin{prop}\label{prop:nefumo}
Let $x\geq 1$. Let $\eta_+, \eta_*:\lbrack 0,\infty) \to \mathbb{R}$.
Assume $\eta_+ \in C^2$, $\eta_+''\in L_2$ and $\eta_+, \eta_* \in L^1 \cap
L^2$. Let $\eta_\circ:\lbrack 0,\infty)\to \mathbb{R}$ be thrice differentiable outside finitely
many points. Assume $\eta_\circ^{(3)}\in L_1$ and
$|\eta_+-\eta_\circ|_2\leq \epsilon_0 |\eta_\circ|_2$, where $\epsilon_0\geq 0$.

Let $S_{\eta}(\alpha,x) = \sum_{n} \Lambda(n) e(\alpha n) \eta(n/x)$.
Let $\err_{\eta,\chi}$, $\chi$ primitive, be given as in (\ref{eq:glenkin})
and (\ref{eq:brahms}). Let  $\delta_0>0$, $r\geq 1$.
Let $\mathfrak{M}=\mathfrak{M}_{\delta_0,r}$ be as in (\ref{eq:majdef}).

Then, for any $N\geq 0$,
\[
\int_{\mathfrak{M}} 
S_{\eta_+}(\alpha,x)^2 S_{\eta_*}(\alpha,x)
 e(-N \alpha) d\alpha\]
 equals
\begin{equation}\label{eq:opus111}\begin{aligned}
C_0 C_{\eta_\circ,\eta_*} x^2 &+ 
\left(
2.82643 |\eta_\circ|_2^2 (2+\epsilon_0)\cdot \epsilon_0
+ \frac{4.31004  |\eta_\circ|_2^2 +
0.0012 \frac{|\eta_\circ^{(3)}|_1^2}{\delta_0^5}}{r} \right) |\eta_*|_1 x^2\\
+ &O^*(E_{\eta_*,r,\delta_0} A_{\eta_+} + 
E_{\eta_+,r,\delta_0} \cdot 
1.6812 (\sqrt{A_{\eta_+}} + 1.6812 |\eta_+|_2) |\eta_*|_2)
\cdot x^2\\
\\+ &O^*\left(2 Z_{\eta_+^2,2}(x) LS_{\eta_*}(x,r) \cdot x + 
4 \sqrt{Z_{\eta_+^2,2}(x) Z_{\eta_*^2,2}(x)}  LS_{\eta_+}(x,r)\cdot x\right),
\end{aligned}\end{equation}
where
\begin{equation}\label{eq:vulgo}\begin{aligned}
C_0 &= \prod_{p|N} \left(1 - \frac{1}{(p-1)^2}\right)
\cdot \prod_{p\nmid N} \left(1 + \frac{1}{(p-1)^3}\right),\\
C_{\eta_\circ,\eta_*} &=
 \int_0^\infty \int_0^\infty \eta_\circ(t_1) \eta_\circ(t_2) 
\eta_*\left(\frac{N}{x}-(t_1+t_2)\right) dt_1 dt_2 
,\end{aligned}\end{equation}
\begin{equation}\label{eq:vulgato}\begin{aligned}
E_{\eta,r,\delta_0} &= 
\mathop{\mathop{\max_{\chi \mo q}}_{q\leq \gcd(q,2) \cdot r}}_{
|\delta|\leq \gcd(q,2) \delta_0 r/ 2 q}
\sqrt{q^*}\cdot |\err_{\eta,\chi^*}(\delta,x)|,\;\;\;\;\;
ET_{\eta,s} = \max_{|\delta|\leq s/q} |\err_{\eta,\chi_T}(\delta,x)|,\\
A_{\eta} &= \frac{1}{x} \int_{\mathfrak{M}} \left|S_{\eta_+}(\alpha,x)\right|^2 d\alpha,
\;\;\;\;\;\;\;\;\;
L_{\eta,r,\delta_0} \leq 2 |\eta|_2^2 \mathop{\sum_{q\leq r}}_{\text{$q$
    odd}}  \frac{\mu^2(q)}{\phi(q)},\\
K_{r,2} &= (1+\sqrt{2 r}) (\log x)^2 |\eta|_\infty
(2 Z_{\eta,1}(x)/x + (1+\sqrt{2 r}) (\log x)^2 |\eta|_\infty/x),\\
Z_{\eta,k}(x) &= \frac{1}{x} \sum_n \Lambda^k(n) \eta(n/x),\;\;\;\;\;\;\;
LS_{\eta}(x,r) = \log r \cdot \max_{p\leq r}
 \sum_{\alpha\geq 1} \eta\left(\frac{p^\alpha}{
x}\right)
,\end{aligned}\end{equation}
and $\err_{\eta,\chi}$ is as in (\ref{eq:glenkin}) and (\ref{eq:brahms}).
\end{prop}
Here is how to read these expressions. The error term in the first line of
(\ref{eq:opus111}) will be small provided that $\epsilon_0$ is small
and $r$ is large.
The third line of
(\ref{eq:opus111}) will be negligible, as will be the term 
$2 \delta_0 r (\log e r) K_{r,2}$
in the definition of $A_\eta$.
(Clearly,
$Z_{\eta,k}(x) \ll_\eta (\log x)^{k-1}$ and 
$LS_{\eta}(x,q) \ll_\eta \tau(q) \log x$
for any $\eta$ of rapid decay.)

It remains to estimate the second line of (\ref{eq:opus111}).
This includes estimating $A_{\eta}$ -- a task that was already
accomplished in Lemma \ref{lem:drujal}.
We see that we will have to give very good bounds for $E_{\eta,r,\delta_0}$
 when $\eta=\eta_+$ or $\eta=\eta_*$.
We also see that we want to make $C_0 C_{\eta_+,\eta_*} x^2$ as large as possible;
it will be competing not just with the error terms here, but, more
importantly, with the bounds from the minor arcs, which will be proportional
to $|\eta_+|_2^2 |\eta_*|_1$.

\chapter{Optimizing and adapting smoothing functions}\label{chap:rossini}

One of our goals is to maximize the quantity $C_{\eta_\circ,\eta_*}$ in 
(\ref{eq:vulgo}) 
relative to $|\eta_\circ|_2^2 |\eta_*|_1$. One way to do this is to ensure that
(a) $\eta_*$ is concentrated on a very short\footnote{This is an idea
appearing in work by Bourgain in a related context
\cite{MR1726234}.} interval $\lbrack 0,\epsilon)$,
(b) $\eta_\circ$ is supported on the interval $\lbrack 0,2\rbrack$, and is
symmetric around $t=1$, meaning that $\eta_\circ(t) \sim \eta_\circ(2-t)$.
Then, for $x \sim N/2$, the integral
\[
\int_0^\infty \int_0^\infty \eta_\circ(t_1) \eta_\circ(t_2) \eta_\ast\left(
 \frac{N}{x} - (t_1 + t_2)\right) dt_1 dt_2\]
in (\ref{eq:vulgo}) should be approximately equal to
\begin{equation}\label{eq:mana}|\eta_*|_1 \cdot 
\int_0^\infty \eta_\circ(t) \eta_\circ\left(\frac{N}{x} - t\right) dt =
 |\eta_*|_1 \cdot  \int_0^\infty \eta_\circ(t)^2 dt = |\eta_*|_1 \cdot |\eta_\circ|_2^2,
\end{equation}
provided that $\eta_0(t)\geq 0$ for all $t$.
It is easy to check (using Cauchy-Schwarz in the second step)
that this is essentially optimal. (We will redo this rigorously in a little
while.)

At the same time, the fact is that major-arc estimates are best for smoothing
functions $\eta$ of a particular form, and we have minor-arc estimates
from Part \ref{part:min} for
a different specific smoothing $\eta_2$. The issue, then, is how do we choose $\eta_\circ$
and $\eta_*$ as above so that 
\begin{itemize}
\item $\eta_*$ is concentrated on $\lbrack 0,\epsilon)$,
\item $\eta_\circ$ is supported on $\lbrack 0,2\rbrack$ and symmetric
around $t=1$,
\item we can give 
 minor-arc and major-arc estimates for $\eta_*$, 
\item we can give
major-arc estimates for a function $\eta_+$ close
to $\eta_\circ$ in $\ell_2$ norm?
\end{itemize} 
\section{The symmetric smoothing function $\eta_\circ$}\label{subs:charme}
We will later work with a smoothing function $\eta_\heartsuit$ 
whose Mellin transform
decreases very rapidly. Because of this rapid decay, we will be able to give
strong results based on an explicit formula for $\eta_\heartsuit$.
The issue is how to define $\eta_\circ$,
given $\eta_\heartsuit$, so that $\eta_\circ$ is symmetric around $t=1$
(i.e., $\eta_\circ(2 - x) \sim \eta_\circ(x)$) and is very small for $x>2$.

We will later set $\eta_{\heartsuit}(t) = e^{-t^2/2}$.
 Let
\begin{equation}\label{eq:clog}
h:t\mapsto \begin{cases}
t^3 (2-t)^3 e^{t-1/2} &\text{if $t\in \lbrack 0,2\rbrack$,}\\ 0 &\text{otherwise}\end{cases}
\end{equation}
We define $\eta_\circ:\mathbb{R}\to \mathbb{R}$ by
\begin{equation}\label{eq:cleo}\eta_\circ(t) = h(t) \eta_{\heartsuit}(t) = 
\begin{cases}
 t^3 (2-t)^3 e^{-(t-1)^2/2} &\text{if $t\in \lbrack 0,2\rbrack$,}\\ 
0 &\text{otherwise.}\end{cases}
\end{equation}
It is clear that $\eta_\circ$
is symmetric around $t=1$ for $t\in \lbrack 0,2\rbrack$.

\subsection{The product $\eta_\circ(t) \eta_\circ(\rho-t)$.}
We now should go back and redo rigorously what we discussed informally
around (\ref{eq:mana}). More precisely, we wish to estimate
\begin{equation}\label{eq:weor}
\eta_\circ(\rho) = 
\int_{-\infty}^\infty \eta_\circ(t) \eta_\circ(\rho-t) dt
= \int_{-\infty}^{\infty} \eta_\circ(t) \eta_\circ(2-\rho+t) dt
\end{equation}
for $\rho\leq 2$ close to $2$. In this, it will be useful that the
Cauchy-Schwarz inequality degrades slowly, in the following sense.
\begin{lemma}\label{lem:gosor}
Let $V$ be a real vector space with an inner product 
$\langle \cdot,\cdot\rangle$. Then, for any $v,w\in V$ with $|w-v|_2 \leq
|v|_2/2$,
\[\langle v,w\rangle = |v|_2 |w|_2 + O^*(2.71 |v-w|_2^2).\]
\end{lemma}
\begin{proof}
By a truncated Taylor expansion,
\[\begin{aligned}
\sqrt{1+x} &= 1 + \frac{x}{2} + \frac{x^2}{2} \max_{0\leq t\leq 1}
 \frac{1}{4 (1-(tx)^2)^{3/2}}\\
 &= 1 + \frac{x}{2} + O^*\left(\frac{x^2}{2^{3/2}}\right)
\end{aligned}\]
for $|x|\leq 1/2$. Hence, for $\delta = |w-v|_2/|v|_2$,
\[\begin{aligned}
&\frac{|w|_2}{|v|_2} = \sqrt{1 + \frac{2 \langle w-v,v\rangle
+ |w-v|_2^2}{|v|_2^2}}
= 1 + \frac{2 \frac{\langle w-v,v\rangle}{|v|_2^2} + \delta^2}{2} + O^*\left(\frac{(2\delta+\delta^2)^2}{2^{3/2}}\right)\\ &= 1 + \delta + O^*\left(\left(
\frac{1}{2} + \frac{(5/2)^2}{2^{3/2}}\right) \delta^2\right)
= 1 + \frac{\langle w-v,v\rangle}{|v|_2^2} + O^*\left(2.71 \frac{|w-v|_2^2}{|v|_2^2}\right).
\end{aligned}\]
Multiplying by $|v|_2^2$, we obtain that
\[|v|_2 |w|_2 = |v|_2^2 + \langle w-v,v\rangle + O^*\left(2.71 |w-v|_2^2\right)
= \langle v,w\rangle + O^*\left(2.71 |w-v|_2^2\right).\]
\end{proof}
Applying Lemma \ref{lem:gosor} to (\ref{eq:weor}), we obtain that
\begin{equation}\label{eq:espri}\begin{aligned}
(\eta_\circ \ast \eta_\circ)(\rho) &=
\int_{-\infty}^{\infty} \eta_\circ(t) \eta_\circ((2-\rho)+t) dt \\ &=
\sqrt{\int_{-\infty}^\infty |\eta_\circ(t)|^2 dt} \sqrt{\int_{-\infty}^\infty
 |\eta_\circ((2-\rho)+t)|^2 dt}
\\ &+ O^*\left(2.71 \int_{-\infty}^\infty \left|\eta_\circ(t)-
\eta_\circ((2-\rho)+t)\right|^2 dt\right)\\
&= |\eta_\circ|_2^2 + O^*\left(2.71 \int_{-\infty}^{\infty}
\left(\int_0^{2-\rho} \left|\eta_\circ'(r+t)\right|
dr\right)^2 dt\right)\\
&= |\eta_\circ|_2^2 + O^*\left(2.71 (2-\rho)
\int_0^{2-\rho}
 \int_{-\infty}^{\infty} \left|\eta_\circ'(r+t)\right|^2
dt dr\right)\\
&= |\eta_\circ|_2^2 + O^*(2.71 (2-\rho)^2 |\eta_\circ'|_2^2).
\end{aligned}\end{equation}

We will be working with $\eta_*$ supported on the non-negative
reals; we recall that $\eta_{\circ}$ is supported on $\lbrack 0,2\rbrack$.
Hence \begin{equation}\label{eq:jaram}\begin{aligned}
\int_0^\infty \int_0^\infty
&\eta_\circ(t_1) \eta_\circ(t_2) \eta_*\left(\frac{N}{x} - (t_1+t_2)\right)
dt_1 dt_2 \\
&= \int_0^{\frac{N}{x}}  (\eta_\circ \ast \eta_\circ)(\rho)
\eta_*\left(\frac{N}{x} - \rho\right) d\rho\\
&= \int_0^{\frac{N}{x}} (|\eta_\circ|_2^2 + O^*(2.71 (2-\rho)^2 |\eta_\circ'|_2^2))
\cdot \eta_*\left(\frac{N}{x} - \rho\right) d\rho \\ &= 
|\eta_\circ|_2^2 \int_0^{\frac{N}{x}} \eta_*(\rho) d\rho + 
2.71 |\eta_\circ'|_2^2 \cdot O^*\left(
\int_0^{\frac{N}{x}} ((2-N/x)+\rho)^2 \eta_*(\rho) d\rho\right)
,\end{aligned}\end{equation}
provided that $N/x\geq 2$. 
We see that it will be wise to set $N/x$ very
slightly larger than $2$. As we said before, $\eta_*$ will be scaled so that it is concentrated
on a small interval $\lbrack 0,\epsilon)$.
\section{The smoothing function $\eta_*$: adapting minor-arc bounds}\label{subs:reddo}
Here the challenge is to define a smoothing function $\eta_*$ that is good
both for minor-arc estimates and for major-arc estimates. The two regimes
tend to favor different kinds of smoothing function. For minor-arc
estimates, we use, as \cite{Tao} did, 
\begin{equation}\label{eq:eta2}
\eta_2(t) = 4 \max(\log 2 - |\log 2 t|,0) = 
((2 I_{\lbrack 1/2,1\rbrack}) \ast_M (2 I_{\lbrack 1/2,1\rbrack}))(t),
\end{equation}
where $I_{\lbrack 1/2,1\rbrack}(t)$ is $1$ if $t\in \lbrack 1/2,1\rbrack$ and
$0$ otherwise. For major-arc estimates, we will use a function based on
\[\eta_{\heartsuit} = e^{-t^2/2}.\]
We will actually use here 
the function $t^2 e^{-t^2/2}$, whose Mellin transform is
$M\eta_{\heartsuit}(s+2)$ (by, e.g., \cite[Table 11.1]{Mellin}).)

We will follow the simple expedient of convolving the two smoothing functions,
one good for minor arcs, the other one for major arcs.
 In general, let $\varphi_1,\varphi_2:\lbrack 0,\infty)\to \mathbb{C}$.
It is easy to use bounds on sums of the form
\begin{equation}\label{eq:kostor}
S_{f,\varphi_1}(x) = \sum_n f(n) \varphi_1(n/x)\end{equation}
to bound sums of the form $S_{f,\varphi_1 \ast_M \varphi_2}$:
\begin{equation}\label{eq:chemdames}\begin{aligned}
S_{f,\varphi_1 \ast_M \varphi_2} &=
\sum_n f(n) (\varphi_1 \ast_M \varphi_2)\left(\frac{n}{x}\right) \\
&= 
 \int_0^{\infty} \sum_n f(n) \varphi_1\left(\frac{n}{w x}\right) \varphi_2(w)
   \frac{dw}{w}
= \int_{0}^{\infty} S_{f,\varphi_1}(w x) 
\varphi_2(w) \frac{dw}{w}.\end{aligned}\end{equation}
The same holds, of course, if $\varphi_1$ and $\varphi_2$ are switched, since
$\varphi_1 \ast_M \varphi_2 = \varphi_2 \ast_M \varphi_1$.
The only objection is that the bounds on (\ref{eq:kostor}) that we input
might not be valid, or non-trivial, when the argument $wx$ of $S_{f,\varphi_1}(wx)$ is very small.  Because of this, it is important that the functions
$\varphi_1$, $\varphi_2$ vanish at $0$,  and desirable that their first
 derivatives do so as well.


Let us see how this works out in practice for $\varphi_1 = \eta_2$. Here
$\eta_2:\lbrack 0,\infty)\to
\mathbb{R}$ is given by
\begin{equation}\label{eq:meichu}
\eta_2 = \eta_1 \ast_M \eta_1 = 4 \max(\log 2 - |\log 2 t|,0),\end{equation}
where $\eta_1 = 2 \cdot I_{\lbrack 1/2,1\rbrack}$.


Let us restate the bounds from Theorem \ref{thm:minmain} -- the main
result of Part \ref{part:min}. We will use Lemma \ref{lem:merkel}
to bound terms of the form $q/\phi(q)$.

Let $x\geq x_0$, $x_0 = 2.16\cdot 10^{20}$. 
Let $2 \alpha = a/q +\delta/x$,
$q\leq Q$, $\gcd(a,q)=1$, $|\delta/x|\leq 1/q Q$, where
$Q = (3/4) x^{2/3}$. Then, if $3\leq q\leq x^{1/3}/6$, Theorem 
\ref{thm:minmain} 
gives us that
\begin{equation}\label{eq:monoro}
|S_{\eta_2}(\alpha,x)| \leq g_{x}\left(\max\left(1, \frac{|\delta|}{8}
\right) \cdot q\right) x,\end{equation}
where 
\begin{equation}\label{eq:syryza}
g_x(r) = \frac{(R_{x,2r} \log 2r + 0.5) \sqrt{\digamma(r)} + 2.5}{
\sqrt{2 r}} + \frac{L_{2 r}}{r} + 3.36 x^{-1/6},\end{equation}
with
\begin{equation}\label{eq:veror}\begin{aligned}
R_{x,t} &=  0.27125 \log 
\left(1 + \frac{\log 4 t}{2 \log \frac{9 x^{1/3}}{2.004 t}}\right)
 + 0.41415 \\
L_{t} &=  \digamma(t/2) \left(\frac{13}{4} \log t + 7.82\right) + 
13.66 \log t + 37.55,
\end{aligned}\end{equation}
If $q > x^{1/3}/6$, then, again by Theorem \ref{thm:minmain},
\begin{equation}\label{eq:horm}
|S_{\eta_2}(\alpha,x)| \leq h(x) x,
\end{equation}
where
\begin{equation}\label{eq:flou}
h(x) = 0.276 x^{-1/6} (\log x)^{3/2} + 1234 x^{-1/3} \log x .
\end{equation}

We will work with $x$ varying within a range, and so we must pay some
attention to the dependence of (\ref{eq:monoro}) and (\ref{eq:horm}) on $x$.
Let us prove two auxiliary lemmas on this.
\begin{lemma}\label{lem:convet}
Let $g_x(r)$ be as in (\ref{eq:syryza}) and
$h(x)$ as in (\ref{eq:flou}). Then 
\[x\mapsto \begin{cases} h(x) &\text{if $x < (6 r)^3$}\\
g_{x}(r) &\text{if $x\geq (6 r)^3$}\end{cases}\]
is a decreasing function of $x$ for $r\geq 11$ fixed and $x\geq 21$.
\end{lemma}
\begin{proof}
It is clear from the definitions 
that $x\mapsto h(x)$ (for $x\geq 21$) 
and $x\mapsto g_{x}(r)$ are both decreasing.
Thus, we simply have to show that $h(x_r) \geq g_{x_r}(r)$ for $x_r=
(6 r)^3$. Since $x_r\geq (6 \cdot 11)^3 > e^{12.5}$,
\[\begin{aligned}R_{x_r,2r} 
&\leq 0.27125 \log( 0.065 \log x_r + 1.056) + 0.41415\\
&\leq 0.27125 \log((0.065 + 0.0845) \log x_r) + 0.41415 \leq
0.27215 \log \log x_r.\end{aligned}\]
Hence \[\begin{aligned}
R_{x_r,2r} \log 2r + 0.5 &\leq 0.27215 \log \log x_r \log x_r^{1/3} -
0.27215 \log 12.5 \log 3 + 0.5 \\ &\leq 0.09072 \log \log x_r \log x_r
- 0.255.\end{aligned}\]
At the same time,
\begin{equation}\label{eq:pust}\begin{aligned}
\digamma(r) &= e^{\gamma} \log \log \frac{x_r^{1/3}}{6} + \frac{2.50637}{\log
  \log r}\leq 
e^{\gamma} \log \log x_r - e^{\gamma} \log 3 + 1.9521 \\
&\leq e^{\gamma} \log \log x_r\end{aligned}\end{equation}
for $r\geq 37$, and we also get $\digamma(r)\leq e^{\gamma} \log \log x_r$
for $r\in \lbrack 11, 37\rbrack$ by the bisection method with $10$ iterations. 
Hence
\[\begin{aligned}
(R_{x_r,2r} &\log 2r + 0.5) \sqrt{\digamma(r)} +2.5 \\&\leq
(0.09072 \log \log x_r \log x_r - 0.255) \sqrt{e^\gamma \log \log x_r}
+2.5\\
&\leq 0.1211 \log x_r (\log \log x_r)^{3/2} + 2,\end{aligned}\]
and so
\[\frac{(R_{x_r,2r} \log 2r + 0.5) \sqrt{\digamma(r)} +2.5}{\sqrt{2 r}}
\leq (0.21 \log x_r (\log \log x_r)^{3/2} + 3.47) x_r^{-1/6}.\]

Now, by (\ref{eq:pust}),
\[\begin{aligned}
L_{2 r} &\leq e^{\gamma} \log \log x_r \cdot \left(
\frac{13}{4} \log(x_r^{1/3}/3) + 7.82\right)
+ 13.66 \log(x_r^{1/3}/3) + 37.55\\
&\leq e^{\gamma} \log \log x_r \cdot \left(\frac{13}{12} x_r + 
4.25\right) + 4.56 \log x_r + 22.55.\end{aligned}\]
It is clear that
\[\frac{4.25 e^{\gamma} \log \log x_r  + 4.56 \log x_r + 22.55}{
x_r^{1/3}/6} < 1234 x_r^{-1/3} \log x_r.\]
for $x_r\geq e$: we make the comparison for $x_r=e$ and take the derivative
of the ratio of the left side by the right side.

It remains to show that
\begin{equation}\label{eq:humo}
0.21 \log x_r (\log \log x_r)^{3/2} + 3.47 + 3.36 + \frac{13}{2} e^{\gamma}
 x_r^{-1/3} \log x_r \log \log x_r 
\end{equation}
is less than $0.276 (\log x_r)^{3/2}$ for $x_r$ large enough.
Since $t\mapsto (\log t)^{3/2}/t^{1/2}$ is decreasing for $t>e^3$, we see
that 
\[\frac{0.21 \log x_r (\log \log x_r)^{3/2} + 6.83
+ \frac{13}{2} e^{\gamma} x_r^{-1/3} \log x_r \log \log x_r}{0.276 (\log x_r)^{3/2}} < 1\]
for all $x_r\geq e^{33}$, simply because it is true for $x=e^{33}$, which
is greater than $e^{e^3}$.

We conclude that $h(x_r) \geq g_{x_r}(r) = g_{x_r}(x_r^{1/3}/6)$ for 
$x_r\geq e^{33}$. We check that 
$h(x_r)\geq g_{x_r}(x_r^{1/3}/6)$ for $\log x_r\in \lbrack \log 66^3,
33 \rbrack$ as well by the bisection method (applied with $30$ iterations,
with $\log x_r$ as the variable, on the intervals
$\lbrack \log 66^3, 20\rbrack$,
$\lbrack 20, 25\rbrack$, $\lbrack 25, 30\rbrack$ and $\lbrack 30,33\rbrack$). 
Since $r\geq 11$ implies $x_r\geq 66^3$,
we are done.
\end{proof}


\begin{lemma}\label{lem:convog}
Let $R_{x,r}$ be as in (\ref{eq:syryza}). Then
$t\to R_{e^t,r}(r)$ is convex-up for $t\geq 3 \log 6 r$.
\end{lemma}
\begin{proof}
Since $t\to e^{-t/6}$ and $t\to t$ are clearly convex-up, all we have to do
is to show that $t\to R_{e^t,r}$ is convex-up. In general,
since \[(\log f)'' = \left(\frac{f'}{f}\right)' = \frac{f'' f - (f')^2}{f^2},\]
a function of the form $(\log f)$ is convex-up exactly when $f'' f - (f')^2\geq 0$.
If $f(t) = 1 + a/(t-b)$, we have $f'' f - (f')^2\geq 0$
whenever
\[(t+a-b) \cdot (2 a) \geq  a^2,\]
i.e., $a^2 + 2 a t \geq 2 a b$,
and that certainly happens when $t\geq b$. In our case,
$b = 3 \log (2.004 r/9)$, and so $t\geq 3 \log 6 r$ implies $t\geq b$.
\end{proof}

Now we come to the point where we prove bounds on exponential sums
of the form $S_{\eta_*}(\alpha,x)$ (that is, sums based on the smoothing $\eta_*$)
based on our bounds (\ref{eq:monoro}) and (\ref{eq:horm}) 
on the exponential sums 
$S_{\eta_2}(\alpha,x)$. This is straightforward, as promised.
\begin{prop}\label{prop:gorsh}
Let $x\geq K x_0$, $x_0=2.16\cdot 10^{20}$, $K\geq 1$.
Let $S_\eta(\alpha,x)$ be as in (\ref{eq:fellok}). 
Let
$\eta_* = \eta_2 \ast_M \varphi$, where $\eta_2$ is as in (\ref{eq:meichu})
and $\varphi: \lbrack 0,\infty)\to \lbrack 0, \infty)$ is continuous and in $L^1$.

Let $2\alpha = a/q + \delta/x$, $q\leq Q$, $\gcd(a,q)=1$,
$|\delta/x|\leq 1/qQ$, where $Q = (3/4) x^{2/3}$. If 
$q\leq (x/K)^{1/3}/6$, then
\begin{equation}\label{eq:kroz}
S_{\eta_*}(\alpha,x) \leq g_{x,\varphi}\left(\max\left(1,\frac{|\delta|}{8}
\right) q \right) \cdot |\varphi|_1 x,
\end{equation}
where
\begin{equation}\label{eq:basia}\begin{aligned}
g_{x,\varphi}(r) &=
\frac{(R_{x,K,\varphi,2 r} \log 2 r + 0.5) \sqrt{\digamma(r)} + 2.5}{\sqrt{2
    r}}  + 
\frac{L_{2r}}{r} + 3.36 K^{1/6} x^{-1/6},\\
R_{x,K,\varphi,t} &= R_{x,t} + (R_{x/K,t} - R_{x,t})
\frac{C_{\varphi,2,K}/|\varphi|_1}{\log K}
\end{aligned}\end{equation}
with $R_{x,t}$ and $L_{t}$ are as in (\ref{eq:veror}),
and
\begin{equation}\label{eq:cecidad}
C_{\varphi,2,K} = - \int_{1/K}^1 \varphi(w) \log w\; dw.\end{equation}

If $q>(x/K)^{1/3}/6$, then
\[|S_{\eta_*}(\alpha,x)| \leq h_\varphi(x/K)\cdot |\varphi|_1 x,\]
where 
\begin{equation}\label{eq:midin}\begin{aligned}
h_{\varphi}(x) &= h(x) + C_{\varphi,0,K}/|\varphi|_1,\\
C_{\varphi,0,K} &= 1.04488 \int_0^{1/K} |\varphi(w)| dw 
\end{aligned}\end{equation}
and $h(x)$ is as in (\ref{eq:flou}).
\end{prop}
\begin{proof}
By (\ref{eq:chemdames}),
\[\begin{aligned}
S_{\eta_*}(\alpha,x) &= \int_0^{1/K} S_{\eta_2}(\alpha,w x) 
\varphi(w) \frac{dw}{w} +
\int_{1/K}^\infty S_{\eta_2}(\alpha,w x) \varphi(w) \frac{dw}{w}.
\end{aligned}\]
We bound the first integral by the trivial estimate 
$|S_{\eta_2}(\alpha,w x)|\leq |S_{\eta_2}(0,w x)|$ and 
Cor.~\ref{cor:austeria}:
\[\begin{aligned}\int_0^{1/K} |S_{\eta_2}(0,w x)| 
\varphi(x) \frac{dw}{w} &\leq 1.04488 \int_0^{1/K} w x \varphi(w)
\frac{dw}{w} \\ &= 1.04488 x \cdot \int_0^{1/K} \varphi(w)
dw.\end{aligned}\]

If $w\geq 1/K$, then $w x\geq x_0$, and we can use 
(\ref{eq:monoro}) or (\ref{eq:horm}). 
If $q>(x/K)^{1/3}/6$, then $|S_{\eta_2}(\alpha,w x)|\leq h(x/K) w x$
by (\ref{eq:horm}); moreover, $|S_{\eta_2}(\alpha,y)|\leq h(y) y$ for
$x/K \leq y < (6 q)^3$ (by (\ref{eq:horm})) and
$|S_{\eta_2}(\alpha,y)|\leq g_{y,1}(r)$ for $y\geq (6 q)^3$ 
(by (\ref{eq:monoro})). Thus, Lemma \ref{lem:convet} gives us that
\[\begin{aligned}\int_{1/K}^\infty |S_{\eta_2}(\alpha,w x)| 
\varphi(w) \frac{dw}{w} &\leq 
\int_{1/K}^\infty h(x/K) w x \cdot \varphi(w) \frac{dw}{w} \\ &= 
h(x/K) x \int_{1/K}^\infty \varphi(w) dw \leq h(x/K) |\varphi|_1 \cdot x.
\end{aligned}\]

If $q\leq (x/K)^{1/3}/6$, we always use (\ref{eq:monoro}).
We can use the coarse bound
\[\int_{1/K}^\infty 3.36 x^{-1/6} \cdot w x \cdot \varphi(w) \frac{dw}{w}
\leq  3.36 K^{1/6} |\varphi|_1 x^{5/6}\]
Since $L_{r}$ does not depend on $x$,
\[
\int_{1/K}^{\infty} \frac{L_{r}}{r} \cdot w x\cdot 
\varphi(w) \frac{dw}{w} \leq 
\frac{L_{r}}{r} |\varphi|_1 x.\]

By Lemma \ref{lem:convog} and $q\leq (x/K)^{1/3}/6$, 
$y\mapsto R_{e^y,t}$ is convex-up and decreasing for $y\in \lbrack \log(x/K),\infty)$.
Hence
\[R_{w x,t} \leq \begin{cases} \frac{\log w}{\log \frac{1}{K}} R_{x/K,t} 
+ \left(1 - \frac{\log w}{\log \frac{1}{K}}\right) R_{x,t}
&\text{if $w<1$,}\\ R_{x,t} &\text{if $w\geq 1$.}\end{cases}\]
Therefore
\[\begin{aligned}
&\int_{1/K}^{\infty} R_{w x,t} \cdot w x \cdot \varphi(w) \frac{dw}{w} \\ &\leq
\int_{1/K}^1 \left(\frac{\log w}{\log \frac{1}{K}} R_{x/K,t} 
+ \left(1 - \frac{\log w}{\log \frac{1}{K}}\right) R_{x,t}\right) x \varphi(w)
dw  + \int_1^{\infty} R_{x,t} \varphi(w) x dw\\
&\leq R_{x,t} x \cdot \int_{1/K}^\infty \varphi(w) dw + 
(R_{x/K,t} - R_{x,t}) \frac{x}{\log K} \int_{1/K}^1  \varphi(w) \log w dw\\
&\leq \left(R_{x,t} |\varphi|_1 +
(R_{x/K,t} - R_{x,t}) \frac{C_{\varphi,2}}{\log K}\right) \cdot x,
\end{aligned}\]
where
\[C_{\varphi,2,K} = - \int_{1/K}^1 \varphi(w) \log w\; dw.\]
\end{proof}

We finish by proving a couple more lemmas.

\begin{lemma}\label{lem:vinc}
Let $x> K> 1$. 
Let $\eta_* = \eta_2 \ast_M \varphi$, where $\eta_2$ is as in (\ref{eq:meichu})
and $\varphi: \lbrack 0,\infty)\to \lbrack 0, \infty)$ is continuous and in
$L^1$. Let $g_{x,\varphi}$ be as in (\ref{eq:basia}).

Then $g_{x,\varphi}(r)$ is a decreasing function of $r$ for $670\leq r \leq 
(x/K)^{1/3}/6$.
\end{lemma}
\begin{proof}
Taking derivatives, we can easily see that
\begin{equation}\label{eq:hopo}r\mapsto \frac{\log \log r}{r},\;\;\;
r\mapsto \frac{\log r}{r},\;\;\; 
r\mapsto \frac{\log r \log \log r}{r},\;\;\;
r\mapsto \frac{(\log r)^2 \log \log r}{r}\end{equation}
are decreasing for $r\geq 20$. The same is true if 
$\log \log r$ is replaced by $\digamma(r)$, since $\digamma(r)/\log \log r$
is a decreasing function for $r\geq e$.
Since $(C_{\varphi,2,K}/|\varphi|_1)/\log K \leq 1$ (by (\ref{eq:cecidad})), 
we see that it is enough
to prove that $r\mapsto R_{y,2r} \log 2r \sqrt{\log \log r}/\sqrt{2 r}$ is
decreasing on $r$ for $y=x$ and $y=x/K$ (under the assumption that $r\geq 670$).

Looking at (\ref{eq:veror}) and at (\ref{eq:hopo}), we see that it remains only to check that
\begin{equation}\label{eq:horko}
r\mapsto \log \left(1 + \frac{\log 8 r}{2 \log \frac{9 y^{1/3}}{4.008 r}}
\right) \log 2 r \cdot \sqrt{\frac{\log \log r}{r}}\end{equation}
is decreasing on $r$ for $r\geq 670$.
Taking logarithms, and then derivatives, we see that we have to show that
\[\frac{\frac{\frac{1}{r} \ell +
\frac{\log 8 r}{r}}{2\ell^2}}{\left(1 + \frac{\log 8 r}{2 \ell}\right)
\log \left(1 + \frac{\log 8 r}{2 \ell}\right)} + \frac{1}{r \log 2 r}
+ \frac{1}{2 r \log r \log \log r} 
< \frac{1}{2 r},\]
where $\ell = \log \frac{9 y^{1/3}}{4.008 r}$.  We multiply by $2r$, and
see that this is equivalent to 
\begin{equation}\label{eq:hut}
\frac{\frac{1}{\ell}
\left(2 - \frac{1}{1 + \frac{\log 8 r}{2 \ell}}\right)}{
\log \left(1 + \frac{\log 8 r}{2 \ell}\right)} +
\frac{2}{\log 2 r}
  + \frac{1}{\log r \log \log r} < 1.\end{equation}
A derivative test is enough to show that $s/\log(1+s)$ is an increasing
function of $s$ for $s>0$; hence, so is $s\cdot ( 2-1/(1+s))/\log(1+s)$.
Setting $s = (\log 8 r)/\ell$, we obtain that
 the left side of (\ref{eq:hut}) is a decreasing function of $\ell$ for
$r\geq 1$ fixed.

Since $r\leq y^{1/3}/6$,
$\ell \geq \log 54/4.008 > 2.6$. Thus, for (\ref{eq:hut}) to hold,
it is enough to ensure that
\begin{equation}\label{eq:valsetr}\frac{\frac{1}{2.6}
\left(2 - \frac{1}{1 + \frac{\log 8 r}{5.2}}\right)}{
\log \left(1 + \frac{\log 8 r}{5.2}\right)}
  + \frac{2}{\log 2 r} +  \frac{1}{\log r \log \log r} < 1.\end{equation}
A derivative test shows that $(2-1/s)/\log(1+s)$ is a decreasing function
of $s$ for $s\geq 1.23$; 
since $\log(8\cdot 75)/5.2 > 1.23$, this implies that the left side of
(\ref{eq:valsetr}) is a decreasing function of $r$ for $r\geq 75$.

We check that the left side of (\ref{eq:valsetr}) is indeed less than $1$ for $r=670$; we conclude that it is less than $1$ for all $r\geq 670$.

\end{proof}

\begin{lemma}\label{lem:gosia}
Let $x\geq 10^{25}$. Let $\phi:\lbrack 0,\infty)\to \lbrack 0,\infty)$
be continuous and in $L^1$. Let $g_{x,\phi}(r)$ and $h(x)$ be as in (\ref{eq:basia}) and (\ref{eq:flou}), respectively. Then 
\[g_{x,\phi}\left(\frac{3}{8} 
x^{4/15}\right)\geq h(2 x/\log x).\]
\end{lemma}
\begin{proof}
We can bound $g_{x,\phi}(r)$ from below by
\[gm_x(r) = \frac{(R_{x,r} \log 2 r + 0.5) \sqrt{\digamma(r)} + 2.5}{
\sqrt{2 r}}.\]
Let $r = (3/8) x^{4/15}$.
Using the assumption that $x\geq 10^{25}$, we see that
\begin{equation}\label{eq:habro}\begin{aligned}
R_{x,r} &= 0.27125 \log \left(1 + \frac{\log\left(\frac{3 x^{4/15}}{2}
\right)}{2 \log \left(\frac{9}{2.004 \cdot \frac{3}{8}}
\cdot x^{\frac{1}{3} - 4/15}\right)}
\right) + 0.41415\geq 0.63368.\end{aligned}\end{equation}
(It is easy to see that the left side of (\ref{eq:habro}) is increasing
on $x$.)
Using $x\geq 10^{25}$ again, we get that
\[\digamma(r) = e^\gamma \log \log r + \frac{2.50637}{\log \log r} \geq
5.68721.\]
Since $\log 2 r = (4/15) \log x + \log(3/4)$,
we conclude that
\[gm_x(r) \geq
\frac{ 0.40298 \log x + 3.25765}{\sqrt{3/4}\cdot x^{2/15}}
.\]
Recall that
\[h(x) = \frac{0.276 (\log x)^{3/2}}{x^{1/6}} + \frac{1234 \log x}{x^{1/3}} .\]
We can see that
\begin{equation}\label{eq:jutu}x\mapsto
 \frac{(\log x + 3.3)/x^{2/15}}{
(\log(2 x/\log x))^{3/2}/(2 x/\log x)^{1/6}}\end{equation}
is increasing for $x\geq 10^{25}$ (and indeed for 
$x\geq e^{27}$) by taking the logarithm of the right side of (\ref{eq:jutu}) 
and then taking its derivative with respect to $t = \log x$. We can see
in the same way that
 $(1/x^{2/15})/(\log(2 x/\log x)/(2 x/\log x)^{1/3})$ is increasing for $x\geq e^{22}$.
Since
\[\begin{aligned}\frac{0.40298 (\log x+3.3)}{\sqrt{3/4} \cdot x^{2/15}} 
&\geq \frac{0.276 (\log(2 x/\log x))^{3/2}}{(2 x/\log x)^{1/6}},\\
\frac{3.25765 - 3.3\cdot 0.40298}{\sqrt{3/4}\cdot x^{2/15}} &\geq
\frac{1234 \log (2 x/\log(x)) }{(2 x/\log(x))^{1/3}}\end{aligned}\]
for $x=10^{25}$, we are done.

%
\end{proof}

\chapter{The $\ell_2$ norm and the large sieve}\label{ch:intri}

Our aim here is to give a bound on the $\ell_2$ norm of an
exponential sum over the minor arcs. While we care
about an exponential sum in particular, we will prove a result valid
for all exponential sums $S(\alpha,x) = \sum_n a_n e(\alpha n)$ with 
$a_n$ of prime support.

We start by adapting ideas from Ramar\'e's version of the large sieve
for primes to estimate $\ell_2$ norms over parts of
the circle (\S \ref{subs:ramar}). 
We are left with the task of giving an
explicit bound on the factor in Ramar\'e's work; this we do in
\S \ref{subs:boquo}. As a side effect, this finally gives a fully explicit
large sieve for primes that is asymptotically optimal, meaning a sieve that
does not have a spurious factor of $e^\gamma$ in front; this was
an arguably important gap in the literature.

\section{Variations on the large
sieve for primes}\label{subs:ramar}

We are trying to estimate an integral 
$\int_{\mathbb{R}/\mathbb{Z}} |S(\alpha)|^3 d\alpha$. Instead of bounding it 
trivially by
$|S|_\infty |S|_2^2$, we can use the fact that large (``major'') values
of $S(\alpha)$ have to be multiplied only by 
$\int_\mathfrak{M} |S(\alpha)|^2 d\alpha$,
where $\mathfrak{M}$ is a union (small in measure) of major arcs. Now,
can we give an upper bound for $\int_\mathfrak{M} |S(\alpha)|^2 d\alpha$ 
better than
$|S|_2^2 = \int_{\mathbb{R}/\mathbb{Z}} |S(\alpha)|^2 d\alpha$?

The first version of \cite{Helf} gave an estimate on that integral
using a technique due to Heath-Brown, which in turn rests
on an inequality of Montgomery's (\cite[(3.9)]{MR0337847}; see also, e.g.,
\cite[Lem.~7.15]{MR2061214}).
 The technique was 
communicated by Heath-Brown to the present author, who communicated it to Tao, who used it in his own notable work on sums of five primes 
 (see \cite[Lem.~4.6]{Tao} and adjoining comments).
 We will be able to do better than that estimate here.

The role played by Montgomery's inequality in Heath-Brown's method is played
here by a result of Ramar\'e's (\cite[Thm. 2.1]{MR2493924}; see also
\cite[Thm. 5.2]{MR2493924}). The following proposition is based on Ramar\'e's 
result, or rather on one possible proof of it. Instead of using
the result as stated in \cite{MR2493924}, we will actually be using elements of
the proof of \cite[Thm. 7A]{MR0371840}, credited to Selberg. Simply integrating
Ramar\'e's inequality would give a non-trivial if slightly worse bound.
\begin{prop}\label{prop:ramar}
Let $\{a_n\}_{n=1}^\infty$, $a_n\in \mathbb{C}$, be supported on the primes.
Assume that $\{a_n\}$ is in $\ell_1\cap \ell_2$ and that $a_n=0$ for $n\leq \sqrt{x}$.
Let $Q_0\geq 1$, $\delta_0\geq 1$ be such that $\delta_0 Q_0^2 \leq x/2$; set
$Q = \sqrt{x/2 \delta_0} \geq Q_0$. Let
\begin{equation}\label{eq:jokor}
\mathfrak{M} = \bigcup_{q\leq Q_0} \mathop{\bigcup_{a \mo q}}_{(a,q)=1}
 \left(\frac{a}{q} - \frac{\delta_0 r}{q x}, 
\frac{a}{q} + \frac{\delta_0 r}{q x}\right)
.\end{equation}

Let $S(\alpha) = \sum_n a_n e(\alpha n)$ for $\alpha \in \mathbb{R}/\mathbb{Z}$.
Then
\[\int_{\mathfrak{M}} \left|S(\alpha)\right|^2 d\alpha \leq
\left(\max_{q\leq Q_0} \max_{s\leq Q_0/q} \frac{G_q(Q_0/sq)}{G_q(Q/sq)}\right)
\sum_n |a_n|^2,\]
where
\begin{equation}\label{eq:malbo}
G_q(R) = \mathop{\sum_{r\leq R}}_{(r,q)=1} \frac{\mu^2(r)}{\phi(r)}.\end{equation}
\end{prop}
\begin{proof}
By (\ref{eq:jokor}),
\begin{equation}\label{eq:malkr}\int_{\mathfrak{M}} \left|S(\alpha)\right|^2 d\alpha = 
\sum_{q\leq Q_0} 
\int_{-\frac{\delta_0 Q_0}{q x}}^{\frac{\delta_0 Q_0}{q x}}
\mathop{\sum_{a \mo q}}_{(a,q)=1} 
\left|S\left(\frac{a}{q}+\alpha\right)\right|^2 d\alpha.\end{equation}
Thanks to the last equations of \cite[p. 24]{MR0371840} and
\cite[p. 25]{MR0371840},
\[\mathop{\sum_{a \mo q}}_{(a,q)=1} \left|S\left(\frac{a}{q}\right)\right|^2 =
\frac{1}{\phi(q)} \mathop{\mathop{\sum_{q^*|q}}_{(q^*,q/q^*)=1}}_{\mu^2(q/q^*) = 1}
q^* \cdot \sume_{\chi \mo q^*} \left|\sum_n a_n \chi(n)\right|^2\]
for every $q\leq \sqrt{x}$, where we use the assumption that 
$n$ is prime and $>\sqrt{x}$ (and thus coprime to $q$) when $a_n\ne 0$.
Hence
\[\begin{aligned}
&\int_{\mathfrak{M}} \left|S(\alpha)\right|^2 d\alpha = 
\sum_{q\leq Q_0} 
 \mathop{\mathop{\sum_{q^*|q}}_{(q^*,q/q^*)=1}}_{\mu^2(q/q^*) = 1}
q^* \int_{-\frac{\delta_0 Q_0}{q x}}^{\frac{\delta_0 Q_0}{q x}} \frac{1}{\phi(q)}
\left|\sum_n a_n e(\alpha n) \chi(n)\right|^2 d\alpha
\\
&= \sum_{q^*\leq Q_0} \frac{q^*}{\phi(q^*)} 
\mathop{\sum_{r\leq Q_0/q^*}}_{(r,q^*)=1}
\frac{\mu^2(r)}{\phi(r)} 
\int_{-\frac{\delta_0 Q_0}{q^* r x}}^{\frac{\delta_0 Q_0}{q^* r x}}
\sume_{\chi \mo q^*}
\left|\sum_n a_n e(\alpha n) \chi(n)\right|^2 d\alpha\\
&= \sum_{q^*\leq Q_0} \frac{q^*}{\phi(q^*)} 
\int_{-\frac{\delta_0 Q_0}{q^* x}}^{\frac{\delta_0 Q_0}{q^* x}}
\mathop{\sum_{r\leq \frac{Q_0}{q^*} \min\left(1,\frac{\delta_0}{|\alpha| x}\right)}}_{(r,q^*)=1}
\frac{\mu^2(r)}{\phi(r)} 
\sume_{\chi \mo q^*}
\left|\sum_n a_n e(\alpha n) \chi(n)\right|^2 d\alpha
\end{aligned}\]
Here $|\alpha|\leq \delta_0 Q_0/q^* x$ implies $(Q_0/q) \delta_0/|\alpha| x
\geq 1$. Therefore,
\begin{equation}\label{eq:ail}\begin{aligned}
\int_{\mathfrak{M}} &\left|S(\alpha)\right|^2 d\alpha \leq
\left(\max_{q^*\leq Q_0} \max_{s\leq Q_0/q^*} \frac{G_{q^*}(Q_0/sq^*)}{G_{q^*}(Q/sq^*)}\right)
\cdot \Sigma,\end{aligned}\end{equation}
where \[\begin{aligned}
\Sigma &= \sum_{q^*\leq Q_0} \frac{q^*}{\phi(q^*)} 
\int_{-\frac{\delta_0 Q_0}{q^* x}}^{\frac{\delta_0 Q_0}{q^* x}}
\mathop{\sum_{r\leq \frac{Q}{q^*} \min\left(1,\frac{\delta_0}{|\alpha| x}\right)}}_{(r,q^*)=1}
\frac{\mu^2(r)}{\phi(r)} 
\sume_{\chi \mo q^*}
\left|\sum_n a_n e(\alpha n) \chi(n)\right|^2 d\alpha\\
&\leq 
\sum_{q\leq Q} \frac{q}{\phi(q)} 
\mathop{\sum_{r\leq Q/q}}_{(r,q)=1}
\frac{\mu^2(r)}{\phi(r)} 
\int_{-\frac{\delta_0 Q}{q r x}}^{\frac{\delta_0 Q}{q r x}}
\sume_{\chi \mo q}
\left|\sum_n a_n e(\alpha n) \chi(n)\right|^2 d\alpha.
\end{aligned}\]
As stated in the proof of \cite[Thm. 7A]{MR0371840},
\[
\overline{\chi}(r) \chi(n) \tau(\overline{\chi}) c_r(n) =
\mathop{\sum_{b=1}^{qr}}_{(b,qr)=1} \overline{\chi}(b) e^{2\pi i n \frac{b}{q r}}
\]
for $\chi$ primitive of modulus $q$. Here $c_r(n)$ stands for the
Ramanujan sum \[c_r(n) = \mathop{\sum_{u \mo r}}_{(u,r)=1} e^{2\pi n u/r}.\] 
For $n$ coprime to $r$, $c_r(n) = \mu(r)$. Since $\chi$ is primitive,
$|\tau(\overline{\chi})| = \sqrt{q}$.
Hence, for $r\leq \sqrt{x}$ coprime to $q$,
\[q \left|\sum_n a_n e(\alpha n) \chi(n)\right|^2  =
\left| \mathop{\sum_{b=1}^{qr}}_{(b,qr)=1} \overline{\chi}(b) S\left(\frac{b}{q r}
+\alpha\right)\right|^2 .
\]
Thus,
\[\begin{aligned} \Sigma &=
\sum_{q\leq Q} \mathop{\sum_{r\leq Q/q}}_{(r,q)=1}
\frac{\mu^2(r)}{\phi(r q)} 
\int_{-\frac{\delta_0 Q}{q r x}}^{\frac{\delta_0 Q}{q r x}}
\sume_{\chi \mo q} 
\left| \mathop{\sum_{b=1}^{qr}}_{(b,qr)=1} \overline{\chi}(b) S\left(\frac{b}{q r}
+\alpha\right)\right|^2 d\alpha\\
&\leq 
\sum_{q\leq Q} \frac{1}{\phi(q)}
\int_{-\frac{\delta_0 Q}{q x}}^{\frac{\delta_0 Q}{q x}}
\sum_{\chi \mo q} 
\left| \mathop{\sum_{b=1}^{q}}_{(b,q)=1} \overline{\chi}(b) S\left(\frac{b}{q}
+\alpha\right)\right|^2 d\alpha\\
&= \sum_{q\leq Q} 
\int_{-\frac{\delta_0 Q}{q x}}^{\frac{\delta_0 Q}{q x}}
\mathop{\sum_{b=1}^{q}}_{(b,q)=1} \left|S\left(\frac{b}{q}+\alpha\right)\right|^2 
d\alpha.
\end{aligned}\]
Let us now check that the intervals $(b/q-\delta_0 Q/qx,b/q+\delta_0 Q/q x)$
do not overlap. Since $Q = \sqrt{x/2\delta_0}$, we see that
$\delta_0 Q/q x = 1/2q Q$.
The difference between two distinct fractions $b/q$,
$b'/q'$ is at least $1/q q'$. For $q,q'\leq Q$, $1/q q'\geq 1/2 q Q + 
1/2 Q q'$. Hence the intervals around $b/q$ and $b'/q'$ do not overlap.
We conclude that
\[\Sigma \leq \int_{\mathbb{R}/\mathbb{Z}} \left|S(\alpha)\right|^2 
= \sum_n |a_n|^2,\]
and so, by (\ref{eq:ail}), we are done. 
\end{proof}

We will actually use Prop.~\ref{prop:ramar} in the slightly modified form
given by the following statement.
\begin{prop}\label{prop:bellen}
Let $\{a_n\}_{n=1}^\infty$, $a_n\in \mathbb{C}$, be supported on the primes.
Assume that $\{a_n\}$ is in $\ell_1\cap \ell_2$ and that $a_n=0$ for $n\leq \sqrt{x}$.
Let $Q_0\geq 1$, $\delta_0\geq 1$ be such that $\delta_0 Q_0^2 \leq x/2$; set
$Q = \sqrt{x/2 \delta_0} \geq Q_0$. Let $\mathfrak{M}=\mathfrak{M}_{\delta_0,Q_0}$ be
as in (\ref{eq:majdef}).

Let $S(\alpha) = \sum_n a_n e(\alpha n)$ for $\alpha \in \mathbb{R}/\mathbb{Z}$.
Then
\[\int_{\mathfrak{M}_{\delta_0,Q_0}} \left|S(\alpha)\right|^2 d\alpha \leq
\left(\mathop{\max_{q\leq 2 Q_0}}_{\text{$q$ even}} \max_{s\leq 2 Q_0/q} \frac{G_{q}(2 Q_0/sq)}{G_{q}(2 Q/sq)}\right)
\sum_n |a_n|^2,\]
where
\begin{equation}\label{eq:malar}
G_q(R) = \mathop{\sum_{r\leq R}}_{(r,q)=1} \frac{\mu^2(r)}{\phi(r)}.\end{equation}
\end{prop}
\begin{proof}
By (\ref{eq:majdef}),
\[\begin{aligned}
\int_{\mathfrak{M}} \left|S(\alpha)\right|^2 d\alpha &= 
\mathop{\sum_{q\leq Q_0}}_{\text{$q$ odd}} 
\int_{-\frac{\delta_0 Q_0}{2 q x}}^{\frac{\delta_0 Q_0}{2 q x}}
\mathop{\sum_{a \mo q}}_{(a,q)=1} 
\left|S\left(\frac{a}{q}+\alpha\right)\right|^2 d\alpha\\
&+ \mathop{\sum_{q\leq Q_0}}_{\text{$q$ even}} 
\int_{-\frac{\delta_0 Q_0}{q x}}^{\frac{\delta_0 Q_0}{q x}}
\mathop{\sum_{a \mo q}}_{(a,q)=1} 
\left|S\left(\frac{a}{q}+\alpha\right)\right|^2 d\alpha .
\end{aligned}\]
We proceed as in the proof of Prop.~\ref{prop:ramar}.
We still have (\ref{eq:malkr}). Hence
$\int_{\mathfrak{M}} \left|S(\alpha)\right|^2 d\alpha$ equals
\[\begin{aligned}
&\mathop{\sum_{q^*\leq Q_0}}_{\text{$q^*$ odd}} \frac{q^*}{\phi(q^*)} 
\int_{-\frac{\delta_0 Q_0}{2 q^* x}}^{\frac{\delta_0 Q_0}{2 q^* x}}
\mathop{\sum_{r\leq \frac{Q_0}{q^*} \min\left(1,\frac{\delta_0}{2 |\alpha| x}\right)}}_{(r,2 
q^*)=1}
\frac{\mu^2(r)}{\phi(r)} 
\sume_{\chi \mo q^*}
\left|\sum_n a_n e(\alpha n) \chi(n)\right|^2 d\alpha\\
+
&\mathop{\sum_{q^*\leq 2 Q_0}}_{\text{$q^*$ even}} \frac{q^*}{\phi(q^*)} 
\int_{-\frac{\delta_0 Q_0}{q^* x}}^{\frac{\delta_0 Q_0}{q^* x}}
\mathop{\sum_{r\leq \frac{2 Q_0}{q^*} \min\left(1,\frac{\delta_0}{2 |\alpha| x}\right)}}_{(r,
q^*)=1}
\frac{\mu^2(r)}{\phi(r)} 
\sume_{\chi \mo q^*}
\left|\sum_n a_n e(\alpha n) \chi(n)\right|^2 d\alpha.
\end{aligned}\]
(The sum with $q$ odd and $r$ even is equal to the first sum; 
hence the factor of $2$ in front.)
Therefore,
\begin{equation}\label{eq:nimaduro}\begin{aligned}
\int_{\mathfrak{M}} \left|S(\alpha)\right|^2 d\alpha
&\leq
\left(\mathop{\max_{q^*\leq Q_0}}_{\text{$q^*$ odd}} \max_{s\leq Q_0/q^*} \frac{G_{2 q^*}(Q_0/sq^*)}{G_{2 q^*}(Q/sq^*)}\right)
\cdot 2 \Sigma_1\\
&+
\left(\mathop{\max_{q^*\leq 2 Q_0}}_{\text{$q^*$ even}} \max_{s\leq 2 Q_0/q^*} 
\frac{G_{q^*}(2Q_0/sq^*)}{G_{q^*}(2Q/sq^*)}\right)
\cdot \Sigma_2,\end{aligned}\end{equation}
where
\[\begin{aligned}
\Sigma_1 &= \mathop{\sum_{q\leq Q}}_{\text{$q$ odd}} \frac{q}{\phi(q)} 
\mathop{\sum_{r\leq Q/q}}_{(r,2 q)=1}
\frac{\mu^2(r)}{\phi(r)} 
\int_{-\frac{\delta_0 Q}{2 q r x}}^{\frac{\delta_0 Q}{2 q r x}}
\sume_{\chi \mo q}
\left|\sum_n a_n e(\alpha n) \chi(n)\right|^2 d\alpha\\
&= \mathop{\sum_{q\leq Q}}_{\text{$q$ odd}} \frac{q}{\phi(q)} 
\mathop{\mathop{\sum_{r\leq 2 Q/q}}_{(r,q)=1}}_{\text{$r$ even}}
 \frac{\mu^2(r)}{\phi(r)} 
\int_{-\frac{\delta_0 Q}{q r x}}^{\frac{\delta_0 Q}{q r x}}
\sume_{\chi \mo q}
\left|\sum_n a_n e(\alpha n) \chi(n)\right|^2 d\alpha.
\\
\Sigma_2 &=
\mathop{\sum_{q\leq 2 Q}}_{\text{$q$ even}} \frac{q}{\phi(q)} 
\mathop{\sum_{r\leq 2 Q/q}}_{(r,q)=1}
\frac{\mu^2(r)}{\phi(r)} 
\int_{-\frac{\delta_0 Q}{q r x}}^{\frac{\delta_0 Q}{q r x}}
\sume_{\chi \mo q}
\left|\sum_n a_n e(\alpha n) \chi(n)\right|^2 d\alpha.
\end{aligned}\]
The two expressions within parentheses in (\ref{eq:nimaduro}) are actually equal.

Much as before, using \cite[Thm.~7A]{MR0371840}, we obtain that
\[\begin{aligned}
\Sigma_1 &\leq \mathop{\sum_{q\leq Q}}_{\text{$q$ odd}} \frac{1}{\phi(q)}
\int_{-\frac{\delta_0 Q}{2 q x}}^{\frac{\delta_0 Q}{2 q x}}
\mathop{\sum_{b=1}^{q}}_{(b,q)=1} \left|S\left(\frac{b}{q} + \alpha\right)\right|^2 
d\alpha,\\
\Sigma_1 + \Sigma_2 &\leq \mathop{\sum_{q\leq 2 Q}}_{\text{$q$ even}} \frac{1}{\phi(q)}
\int_{-\frac{\delta_0 Q}{q x}}^{\frac{\delta_0 Q}{q x}}
\mathop{\sum_{b=1}^{q}}_{(b,q)=1} \left|S\left(\frac{b}{q} + \alpha\right)\right|^2 
d\alpha.
\end{aligned}\]
Let us now check that the intervals of integration 
$(b/q-\delta_0 Q/2 q x,b/q+\delta_0 Q/2 q x)$ (for $q$ odd),
$(b/q-\delta_0 Q/q x,b/q+\delta_0 Q/q x)$ (for $q$ even)
do not overlap. Recall that $\delta_0 Q/q x = 1/2 q Q$.
The absolute value of the difference between two distinct
 fractions $b/q$, $b'/q'$ is at least $1/q q'$.
For $q, q'\leq Q$ odd, this is larger than
$1/4 q Q + 1/4 Q q'$, and so the intervals do not overlap. For $q\leq Q$ odd
and $q'\leq 2 Q$ even (or vice versa), $1/q q'\geq 1/4 q Q + 1/2 Q q'$,
and so, again the intervals do not overlap. If $q\leq Q$ and $q'\leq Q$
are both even, then $|b/q - b'/q'|$ is actually $\geq 2/q q'$. Clearly,
$2/q q' \geq 1/2 q Q + 1/2 Q q'$, and so again there is no overlap.
We conclude that
\[2 \Sigma_1 + \Sigma_2 \leq \int_{\mathbb{R}/\mathbb{Z}} \left|S(\alpha)\right|^2 
= \sum_n |a_n|^2.\]
\end{proof}

\section{Bounding the quotient in the large sieve for primes}\label{subs:boquo}

The estimate given by Proposition \ref{prop:ramar} involves the quotient
\begin{equation}\label{eq:gusto}
\max_{q\leq Q_0} \max_{s\leq Q_0/q}
\frac{G_q(Q_0/sq)}{G_q(Q/sq)},\end{equation}
where $G_q$ is as in (\ref{eq:malbo}).
The appearance of such a quotient (at least for $s=1$) is typical of
Ramar\'e's version of the large sieve for primes; see, e.g., \cite{MR2493924}.
We will see how to bound such a quotient in a way that is essentially
optimal, not just asymptotically, but also in the ranges that are most
relevant to us. (This includes, for example, $Q_0\sim 10^6$, $Q\sim
10^{15}$.)

As the present work shows, an approach based on Ramar\'e's work gives bounds that are, 
in some contexts, better than those of other large sieves for primes by a 
constant factor (approaching $e^\gamma = 1.78107\dotsc$).
Thus, giving a fully explicit
and nearly optimal bound for (\ref{eq:gusto}) is a task of clear
general relevance, besides being needed for our main goal.

We will obtain bounds for $G_q(Q_0/sq)/G_q(Q/sq)$ when $Q_0\leq 2\cdot 10^{10}$,
$Q\geq Q_0^2$. As we shall see, our bounds will be best when $s=q=1$
-- or, sometimes, when $s=1$ and $q=2$ instead.

Write $G(R)$ for $G_1(R) = \sum_{r\leq R} \mu^2(r)/\phi(r)$.
We will need several estimates for $G_q(R)$ and $G(R)$.
As stated in \cite[Lemma 3.4]{MR1375315},
\begin{equation}\label{eq:cante}
G(R) \leq \log R + 1.4709 \end{equation}
for $R\geq 1$.
By \cite[Lem.~7]{MR0374060}, 
\begin{equation}\label{eq:jondo}
G(R)\geq \log R + 1.07
\end{equation}
for $R\geq 6$. 
There is also the trivial bound
\begin{equation}\label{eq:coro}
\begin{aligned}G(R) &= \sum_{r\leq R} \frac{\mu^2(r)}{\phi(r)} =
\sum_{r\leq R} \frac{\mu^2(r)}{r} \prod_{p|r} \left(1
  -\frac{1}{p}\right)^{-1}\\
&= \sum_{r\leq R} \frac{\mu^2(r)}{r} \prod_{p|r} \sum_{j\geq 1}
\frac{1}{p^j} \geq \sum_{r\leq R} \frac{1}{r} > \log R.\end{aligned}
\end{equation}
The following bound, also well-known and easy,
\begin{equation}\label{eq:hosmo}
G(R)\leq \frac{q}{\phi(q)} G_q(R)\leq G(R q),
\end{equation}
can be obtained by multiplying $G_q(R) = \sum_{r\leq R: (r,q)=1}
\mu^2(r)/\phi(r)$ term-by-term by $q/\phi(q) = \prod_{p|q} (1 +
1/\phi(p))$. 

We will also use
Ramar\'e's estimate from \cite[Lem.~3.4]{MR1375315}:
\begin{equation}\label{eq:malito}
G_d(R) = \frac{\phi(d)}{d} \left(\log R + c_E + \sum_{p|d}
  \frac{\log p}{p}\right) + O^*\left(7.284 R^{-1/3} f_1(d)\right)
\end{equation}
for all $d\in \mathbb{Z}^+$ and all $R\geq 1$, where
\begin{equation}\label{eq:assur}
f_1(d) = \prod_{p|d} (1+p^{-2/3}) \left(1+\frac{p^{1/3}+p^{2/3}}{p
(p-1)}\right)^{-1}\end{equation}
and
\begin{equation}\label{eq:garno}
c_E = \gamma + \sum_{p\geq 2} \frac{\log p}{p (p-1)} = 1.3325822\dotsc
\end{equation}
by \cite[(2.11)]{MR0137689}. 

If $R\geq 182$, then
\begin{equation}\label{eq:charpy}
\log R + 1.312 \leq G(R) \leq \log R + 1.354,\end{equation}
where the upper bound is valid for $R\geq 120$.
This is true by (\ref{eq:malito}) for $R\geq 4\cdot 10^7$; we check 
(\ref{eq:charpy}) for $120\leq R\leq 4\cdot 10^7$ by a numerical 
computation.\footnote{Using D. Platt's implementation \cite{Platt}
of double-precision interval arithmetic based on Lambov's \cite{Lamb} ideas.}
Similarly, for $R\geq 200$,
\begin{equation}\label{eq:charpas}
\frac{\log R + 1.661}{2} \leq G_2(R) \leq \frac{\log R + 1.698}{2}
\end{equation}
by (\ref{eq:malito}) for $R\geq 1.6\cdot 10^8$,
and by a numerical computation for $200\leq R\leq 1.6\cdot 10^8$.

Write $\rho = (\log Q_0)/(\log Q)\leq 1$. 
We obtain immediately from (\ref{eq:charpy}) and 
(\ref{eq:charpas}) that
\begin{equation}\label{eq:rabatt}\begin{aligned}
\frac{G(Q_0)}{G(Q)} &\leq \frac{\log Q_0 + 1.354}{\log Q + 1.312} \\
\frac{G_2(Q_0)}{G_2(Q)} 
&\leq \frac{\log Q_0 + 1.698}{\log Q + 1.661}\end{aligned}\end{equation}
for $Q,Q_0\geq 200$.
What is hard is to approximate $G_q(Q_0)/G_q(Q)$ for $q$ large and $Q_0$ small.



Let us start by giving an easy bound, off from the truth by a factor of
about $e^\gamma$. (Specialists will recognize this as a factor that appears 
often in first attempts at estimates based on either large or small sieves.)
First, we need a simple explicit lemma.
\begin{lemma}\label{lem:suspiro}
Let $m\geq 1$, $q\geq 1$. Then
\begin{equation}\label{eq:ostor}
\prod_{p|q \vee p\leq m} \frac{p}{p-1} \leq e^{\gamma} (\log (m+\log q) + 0.65771).\end{equation}
\end{lemma}
\begin{proof}
Let $\mathscr{P}=\prod_{p\leq m \vee p|q} p$. Then, by
\cite[(5.1)]{MR0457373},
\[\mathscr{P}\leq q \prod_{p\leq m} p =
q e^{\sum_{p\leq m} \log p} \leq 
q e^{(1+\epsilon_0) m},\]
where $\epsilon_0 = 0.001102$. Now, by \cite[(3.42)]{MR0137689},
\[ \frac{n}{\phi(n)} \leq e^\gamma \log \log n + \frac{2.50637}{\log \log n}
\leq e^\gamma \log \log x + \frac{2.50637}{\log \log x}\]
for all $x\geq n\geq 27$ (since, given $a,b>0$, the function
 $t\mapsto a+b/t$ is increasing on $t$ 
for $t\geq \sqrt{b/a}$). Hence, if $q e^m \geq 27$,
\[\begin{aligned}
\frac{\mathscr{P}}{\phi(\mathscr{P})} &\leq e^\gamma
\log ((1+\epsilon_0) m + \log q) + \frac{2.50637}{\log (m+\log q)}\\
&\leq e^{\gamma} \left(\log (m+\log q) + \epsilon_0 +
\frac{2.50637/e^{\gamma}}{\log (m+\log q)}\right).
\end{aligned}\]
Thus (\ref{eq:ostor}) holds when $m+\log q \geq 8.53$, since then
$\epsilon_0 + (2.50637/e^{\gamma})/\log (m+\log q) \leq 0.65771$.
We verify all choices of $m,q\geq 1$ with $m+\log q\leq 8.53$ 
computationally; the worst case is that of $m=1$, $q=6$, which give the value
$0.65771$ in (\ref{eq:ostor}).
\end{proof}

Here is the promised easy bound.
\begin{lemma}\label{lem:trivo}
Let $Q_0\geq 1$, $Q\geq 182 Q_0$. Let $q\leq Q_0$, $s\leq Q_0/q$, $q$ an integer.
Then \[\frac{G_q(Q_0/sq)}{G_q(Q/sq)} \leq
\frac{e^\gamma \log \left(\frac{Q_0}{sq} + \log q\right) + 1.172}
{\log \frac{Q}{Q_0} + 1.312} \leq
\frac{e^\gamma \log Q_0 + 1.172}
{\log \frac{Q}{Q_0} + 1.312} .\]
\end{lemma}
\begin{proof}
Let $\mathscr{P}=\prod_{p\leq Q_0/sq \vee p|q} p$. Then
\[ G_q(Q_0/sq) G_{\mathscr{P}}(Q/Q_0) \leq G_q(Q/sq)\]
and so
\begin{equation}\label{eq:bete}
\frac{G_q(Q_0/sq)}{G_q(Q/sq)} \leq \frac{1}{G_{\mathscr{P}}(Q/Q_0)}.
\end{equation}

Now the lower bound in (\ref{eq:hosmo}) gives us that,
for $d=\mathscr{P}$, $R=Q/Q_0$, 
\[
G_{\mathscr{P}}(Q/Q_0) \geq \frac{\phi(\mathscr{P})}{\mathscr{P}}
G(Q/Q_0) .\]
By Lem.~\ref{lem:suspiro},
\[\frac{\mathscr{P}}{\phi(\mathscr{P})} \leq e^{\gamma} \left(\log \left(
\frac{Q_0}{s q} + \log q\right) + 0.658\right).\]
Hence, using (\ref{eq:charpy}), we get that
\begin{equation}\label{eq:mogan}
\begin{aligned}\frac{G_q(Q_0/sq)}{G_q(Q/sq)} &\leq 
\frac{\mathscr{P}/\phi(\mathscr{P})}{G(Q/Q_0)} \leq
\frac{e^\gamma \log \left(\frac{Q_0}{sq} + \log q\right) + 1.172}
{\log \frac{Q}{Q_0} + 1.312},
\end{aligned}\end{equation}
since $Q/Q_0\geq 184$. Since
\[\left(\frac{Q_0}{sq} + \log q\right)' = - \frac{Q_0}{sq^2} + \frac{1}{q}
= \frac{1}{q} \left(1 - \frac{Q_0}{sq}\right) \leq 0,\]
the rightmost expression of (\ref{eq:mogan}) is maximal for $q=1$.
\end{proof}

Lemma \ref{lem:trivo} will play a crucial role in reducing 
to a finite computation the problem
of bounding $G_q(Q_0/sq)/G_q(Q/sq)$.
As we will now
see, we can use Lemma \ref{lem:trivo} to 
obtain a bound that is useful when $sq$ is large compared to $Q_0$ --
precisely the case in which asymptotic estimates such as (\ref{eq:malito})
are relatively weak.
\begin{lemma}\label{lem:paniz}
Let $Q_0\geq 1$, $Q\geq 200 Q_0$. Let $q\leq Q_0$, $s\leq Q_0/q$.
Let $\rho = (\log Q_0)/\log Q\leq 2/3$. 
Then, for any $\sigma\geq 1.312 \rho$, 
\begin{equation}\label{eq:marasmor}
\frac{G_q(Q_0/sq)}{G_q(Q/sq)} \leq \frac{\log Q_0 + \sigma}{
\log Q + 1.312}\end{equation}
holds provided that
\[\frac{Q_0}{sq} \leq c(\sigma) \cdot Q_0^{(1-\rho) e^{-\gamma}} - \log q,\]
where $c(\sigma) = \exp(\exp(-\gamma)\cdot 
(\sigma - \sigma^2/5.248 - 1.172))$.
\end{lemma}
\begin{proof}
By Lemma \ref{lem:trivo}, we see that
 (\ref{eq:marasmor}) will hold provided that
\begin{equation}\label{eq:murey}\begin{aligned}
e^\gamma &\log \left(\frac{Q_0}{sq} + \log q\right) + 1.172
\leq
\frac{\log \frac{Q}{Q_0} + 1.312}{\log Q + 1.312} \cdot (\log Q_0 + \sigma).
\end{aligned}\end{equation}
The expression on the right of (\ref{eq:murey}) equals
\[\begin{aligned}
\log Q_0 + \sigma &- \frac{(\log Q_0 + \sigma) \log Q_0}{
\log Q + 1.312}\\
&=
(1-\rho) (\log Q_0 + \sigma) + \frac{1.312 \rho (\log Q_0 + \sigma)}{
\log Q + 1.312}\\ &\geq (1-\rho) (\log Q_0 + \sigma) +
1.312\rho^2\end{aligned}\]
and so
(\ref{eq:murey}) will
 hold provided that
\[\begin{aligned}
e^\gamma &\log \left(\frac{Q_0}{sq} + \log q\right) +1.172\leq
 (1-\rho) (\log Q_0) + (1-\rho) \sigma +
1.312\rho^2 .\end{aligned}\]
Taking derivatives, we see that
\[\begin{aligned}
(1-\rho) \sigma + 1.312 \rho^2 - 1.172 &\geq 
\left(1 - \frac{\sigma}{2.624}\right) \sigma + 1.312
\left(\frac{\sigma}{2.624}\right)^2 - 1.172\\
&= \sigma - \frac{\sigma^2}{4\cdot 1.312} - 1.172 .
\end{aligned}\]
Hence it is enough that
\[\frac{Q_0}{sq} + \log q \leq e^{e^{-\gamma} 
\left((1-\rho) \log Q_0 + \sigma - \frac{\sigma^2}{4\cdot 1.312} - 1.172\right)} =
c(\sigma)\cdot Q_0^{(1-\rho) e^{-\gamma}},\]
where $c(\sigma) = \exp(\exp(-\gamma)\cdot 
(\sigma - \sigma^2/5.248 - 1.172))$.
\end{proof}

We now pass to the main result of the section.

\begin{prop}\label{prop:espagn}
Let $Q\geq 20000 Q_0$, $Q_0\geq Q_{0,\min}$, where $Q_{0,\min} = 10^5$. 
Let $\rho = (\log Q_0)/\log Q$. Assume $\rho\leq 0.6$. Then, for every 
$1\leq q\leq Q_0$ and every $s\in \lbrack 1, Q_0/q\rbrack$,
\begin{equation}\label{eq:martinos}
\frac{G_q(Q_0/sq)}{G_q(Q/sq)} \leq \frac{\log Q_0 + c_+}{
\log Q + c_E},\end{equation}
where $c_E$ is as in (\ref{eq:garno}) and $c_+ = 1.36$.
\end{prop}

An ideal result would have $c_+$ instead of $c_E$, but this is not actually
possible: error terms do exist, even if they are in reality smaller than
the bound given in (\ref{eq:malito}); this means that a bound such as (\ref{eq:martinos}) with $c_+$ instead of $c_E$ would be false for $q=1$, $s=1$.

There is nothing special about the assumptions \[Q\geq 20000 Q_0,\;\;\;\;\;\;\; 
Q_0\geq 10^5,\;\;\;\;\;\;\; (\log Q_0)/(\log Q)\leq 0.6.\] They can all be relaxed
at the cost of an increase in $c_+$.

\begin{proof}
Define $\err_{q,R}$ so that
\begin{equation}\label{eq:mero}
G_q(R) = \frac{\phi(q)}{q} \left(\log R + c_E + \sum_{p|q} \frac{\log p}{p} \right)+ \err_{q,R}.\end{equation}
Then (\ref{eq:martinos}) will hold if
\begin{equation}\label{eq:elsyn}\begin{aligned}
\log \frac{Q_0}{s q} &+ c_E + \sum_{p|q} \frac{\log p}{p} +
\frac{q}{\phi(q)} \err_{q,\frac{Q_0}{sq}}\\ &\leq
\left(\log \frac{Q}{s q} + c_E + \sum_{p|q} \frac{\log p}{p} +
\frac{q}{\phi(q)} \err_{q,\frac{Q}{sq}}\right) \frac{\log Q_0 + c_+}{\log Q + c_E}.\end{aligned}\end{equation}
This, in turn, happens if 
\[\begin{aligned}
\left(\log sq - \sum_{p|q} \frac{\log p}{p}\right)
&\left(1 - \frac{\log Q_0 + c_+}{\log Q + c_E}\right)
+ c_+ - c_E
\\
\geq &\frac{q}{\phi(q)} \left(\err_{q,\frac{Q_0}{sq}} 
- \frac{\log Q_0 + c_+}{\log Q + c_E} \err_{q,\frac{Q}{sq}}\right).
\end{aligned}\]
Define
\[\omega(\rho) = \frac{\log Q_{0,\min} + c_+}{\frac{1}{\rho}
\log Q_{0,\min} + c_E} = \rho + \frac{c_+ - \rho c_E}{
\frac{1}{\rho} \log Q_{0,\min} + c_E}.\]
Then $\rho \leq (\log Q_0+c_+)/(\log Q + c_E) \leq \omega(\rho)$
(because $c_+ \geq \rho c_E$).
We conclude that (\ref{eq:elsyn}) (and hence (\ref{eq:martinos}))
holds provided that
\begin{equation}\label{eq:karka}\begin{aligned}
(1-\omega(\rho)) &\left(\log s q - \sum_{p|q} \frac{\log p}{p}\right)
+ c_\Delta\\ &\geq \frac{q}{\phi(q)}\left(\err_{q,\frac{Q_0}{sq}} 
+ \omega(\rho) \max\left(0, - \err_{q,\frac{Q}{sq}}\right)\right),
\end{aligned}\end{equation}
where $c_\Delta = c_+ - c_E$.
Note that $1-\omega(\rho)>0$.

First, let us give some easy bounds on the error terms; these bounds
will yield upper bounds for $s$. By (\ref{eq:cante}) and (\ref{eq:hosmo}),
\[
\err_{q,R} \leq \frac{\phi(q)}{q} \left(\log q -
 \sum_{p|q} \frac{\log p}{p} + (1.4709 - c_E)\right)\]
for $R\geq 1$; by (\ref{eq:charpy}) and (\ref{eq:hosmo}),
\[\err_{q,R} \geq - \frac{\phi(q)}{q} \left(
\sum_{p|q} \frac{\log p}{p} + (c_E - 1.312)\right)\]
for $R\geq 182$. Therefore, the right side of (\ref{eq:karka})
is at most
\[\log q -
 (1-\omega(\rho)) \sum_{p|q} \frac{\log p}{p} + 
((1.4709 - c_E) + \omega(\rho) (c_E - 1.312)),\]
and so (\ref{eq:karka}) holds provided that
\begin{equation}\label{eq:miasmar}
(1-\omega(\rho)) \log s q\geq \log q +
(1.4709 - c_E) + \omega(\rho) (c_E - 1.312) - c_\Delta.\end{equation}
We will thus be able to assume from now on that (\ref{eq:miasmar})
does not hold, or, what is the same, that
\begin{equation}\label{eq:koklo}
sq < \left(c_{\rho,2} q\right)^{\frac{1}{1-\omega(\rho)}}\end{equation}
holds, where 
$c_{\rho,2} = \exp((1.4709 - c_E) + \omega(\rho) (c_E - 1.312) - c_\Delta)$.

What values of $R = Q_0/sq$ must we consider for $q$ given? First,
by (\ref{eq:koklo}), 
we can assume $R>Q_{0,\min}/(c_{\rho,2} q)^{1/(1-\omega(\rho))}$.
We can also assume
\begin{equation}\label{eq:joho}
R >
c(c_+) \cdot \max(R q,Q_{0,\min})^{(1-\rho) e^{-\gamma}} - \log q 
\end{equation}
for $c(c_+)$ is as in Lemma \ref{lem:paniz}, since
all smaller $R$ are covered by that Lemma.
Clearly, (\ref{eq:joho}) implies that
\[R^{1-\tau} > c(c_+) \cdot q^{\tau} -
\frac{\log q}{R^{\tau}} > 
c(c_+) q^{\tau} -\log q,
\]
where $\tau = (1-\rho) e^{-\gamma}$,
and also that $R> c(c_+) Q_{0,\min}^{(1-\rho) e^{-\gamma}} - \log q$.
Iterating, we obtain that we can assume that
$R>\varpi(q)$, where
\begin{equation}\label{eq:armor}
\varpi(q) = \max\left(\varpi_0(q),
c(c_+) Q_{0,\min}^{\tau} - \log q,
\frac{Q_{0,\min}}{(c_{\rho,2} q)^{\frac{1}{1-\omega(\rho)}}}\right)\end{equation}
and
\[\varpi_0(q) = 
\begin{cases}
\left(c(c_+) q^{\tau} - \frac{\log q}{
(c(c_+) q^{\tau} -\log q)^{\frac{\tau}{1- \tau}}}
\right)^{\frac{1}{1-\tau}} &\text{if 
$c(c_+) q^{\tau} > \log q + 1$},\\
0 
&\text{otherwise.}\end{cases}\]

Looking at (\ref{eq:karka}), we see that
it will be enough to show that, for all $R$ satisfying $R>\varpi(q)$, we have
\begin{equation}\label{eq:luce}
\err_{q,R} + \omega(\rho) \max\left(0,-\err_{q,tR}\right)
 \leq \frac{\phi(q)}{q} \kappa(q)
\end{equation} for all $t\geq 20000$, where
\[\kappa(q) = 
(1 - \omega(\rho)) \left(\log q - \sum_{p|q} \frac{\log p}{p}\right) + 
c_\Delta.\]

Ramar\'e's bound (\ref{eq:malito}) implies that
\begin{equation}\label{eq:agammen}
|\err_{q,R}| \leq 7.284 R^{-1/3} f_1(q),\end{equation}
with $f_1(q)$ as in (\ref{eq:assur}), and so
\[\err_{q,R} + \omega(\rho) \max\left(0,-\err_{q,tR}\right) \leq
(1 + \beta_\rho) \cdot 7.284 R^{-1/3} f_1(q),\]
where $\beta_\rho = \omega(\rho)/20000^{1/3}$.
This is enough when 
\begin{equation}\label{eq:sosor}
R\geq \lambda(q) = 
 \left(\frac{q}{\phi(q)} \frac{7.284 (1+\beta_\rho) f_1(q)}{\kappa(q)}\right)^3.
\end{equation}

It remains to do two things. First, we have to compute how large $q$
has to be for $\varpi(q)$ to be guaranteed to be greater than 
$\lambda(q)$. (For such $q$, there is no checking to be done.) Then,
we check the inequality (\ref{eq:luce}) for all smaller $q$, letting
$R$ range through the integers in $\lbrack \varpi(q),\lambda(q)\rbrack$.
We bound $\err_{q,t R}$ using (\ref{eq:agammen}), but we compute
$\err_{q,R}$ directly.

{\em How large must $q$ be for $\varpi(q)>\lambda(q)$ to hold?} 
We claim that $\varpi(q)>\lambda(q)$ whenever
 $q\geq 2.2\cdot 10^{10}$.
Let us show this. 

It is easy to see that $(p/(p-1)) \cdot f_1(p)$ and
$p\to (\log p)/p$ are decreasing functions of $p$ for
$p\geq 3$; moreover, for both functions, the value at 
$p\geq 7$ is smaller than for $p=2$. 
Hence, we have that, for $q < \prod_{p\leq p_0} p$, $p_0$ a prime,
\begin{equation}\label{eq:modo}
\kappa(q) \geq (1-\omega(\rho)) \left(\log q - 
\sum_{p<p_0} \frac{\log p}{p}\right) + c_\Delta
\end{equation}
and
\begin{equation}\label{eq:hipo}
\lambda(q) \leq \left(\prod_{p<p_0} \frac{p}{p-1} \cdot
\frac{7.284 (1+\beta_\rho) \prod_{p<p_0} f_1(p)}{
(1-\omega(\rho)) \left(\log q - 
\sum_{p<p_0} \frac{\log p}{p}\right) + c_\Delta}\right)^3.\end{equation}
If we also assume that $2\cdot 3\cdot 5\cdot 7 \nmid q$, we obtain
\begin{equation}\label{eq:modowo}
\kappa(q) \geq (1-\omega(\rho)) \left(\log q - 
\mathop{\sum_{p<p_0}}_{p\ne 7} \frac{\log p}{p}\right) + c_\Delta
\end{equation}
and
\begin{equation}\label{eq:hipowo}
\lambda(q) \leq \left(\mathop{\prod_{p<p_0}}_{p\ne 7} \frac{p}{p-1} \cdot
\frac{7.284 (1+\beta_\rho) \prod_{p<p_0, p \ne 7} f_1(p)}{
(1-\omega(\rho)) \left(\log q - 
\sum_{p<p_0, p\ne 7} \frac{\log p}{p}\right) + c_\Delta}\right)^3\end{equation}
for $q< \prod_{p\leq p_0}$.
(We are taking out $7$ because it is the ``least helpful'' prime to omit
among all primes from $2$ to $7$, again by the fact that
 $(p/(p-1)) \cdot f_1(p)$ and
$p\to (\log p)/p$ are decreasing functions for $p\geq 3$.)

We know how to give upper bounds for the expression on the right
of (\ref{eq:hipo}).
The task is in essence simple: we can base our bounds on the classic
explicit work in \cite{MR0137689}, except that we also have to optimize matters
so that they are close to tight for $p_1=29$, $p_1=31$ and other low $p_1$.

By \cite[(3.30)]{MR0137689} and a numerical computation for $29\leq p_1\leq 43$,
\[\prod_{p\leq p_1} \frac{p}{p-1} <  
1.90516 \log p_1\]
for $p_1\geq 29$.
Since $\omega(\rho)$ is increasing on $\rho$ and we are assuming
$\rho\leq 0.6$, $Q_{0,\min} = 100000$,
\[\omega(\rho)\leq 0.627312,\;\;\;\;\;
\beta_{\rho} \leq 0.023111.\]
For $x>a$, where $a>1$ is any constant, we obviously have
\[\sum_{a< p\leq x} \log \left(1 + p^{-2/3}\right) \leq
\sum_{a< p\leq x} (\log p) \frac{p^{-2/3}}{\log a}.
\]
by Abel summation (\ref{eq:jokors}) and the estimate
\cite[(3.32)]{MR0137689} for $\theta(x) = \sum_{p\leq x} \log p$,
\[\begin{aligned}
&\sum_{a<p\leq x} (\log p) p^{-2/3} = (\theta(x) - \theta(a)) x^{- \frac{2}{3}} -
\int_a^x (\theta(u) -\theta(a)) \left(-\frac{2}{3} u^{-\frac{5}{3}}\right) du \\
&\leq (1.01624 x- \theta(a)) x^{-\frac{2}{3}} +
\frac{2}{3} \int_a^x \left(1.01624 u - \theta(a)\right) u^{-\frac{5}{3}} du\\
&= (1.01624 x- \theta(a)) x^{-\frac{2}{3}} +
2\cdot 1.01624 (x^{1/3}-a^{1/3}) + \theta(a) (x^{-2/3} - a^{-2/3})\\
&= 3\cdot 1.01624 \cdot x^{1/3}
- (2.03248 a^{1/3} + \theta(a) a^{-2/3}).
\end{aligned}\]
We conclude that
$\sum_{10^4<p\leq x} \log (1 + p^{-2/3}) \leq 0.33102 x^{1/3} - 7.06909$
for $x>10^4$. Since $\sum_{p\leq 10^4} \log p \leq 10.09062$, this means that
\[\sum_{p\leq x} \log (1 + p^{-2/3}) \leq \left(0.33102 +
\frac{10.09062 - 7.06909}{10^{4/3}}\right) x^{1/3} 
\leq 0.47126 x^{1/3}\] for $x>10^4$;
a direct computation for all $x$ prime between $29$ and $10^4$ then 
confirms that \[\sum_{p\leq x} \log (1 + p^{-2/3}) \leq 
0.74914 x^{1/3}\]
for all $x\geq 29$. Thus,
\[\prod_{p\leq x} f_1(p) \leq \frac{e^{\sum_{p\leq x} \log (1+ p^{-2/3})}}{
\prod_{p\leq 29} \left(1 + \frac{p^{1/3} + p^{2/3}}{p (p-1)}\right)}
\leq \frac{e^{0.74914 x^{1/3}}}{6.62365} 
\]
for $x\geq 29$.
Finally, by \cite[(3.24)]{MR0137689},
$\sum_{p\leq p_1} \frac{\log p}{p} < \log p_1$. 

We conclude that, for $q<\prod_{p\leq p_0} p_0$, $p_0$ a prime, and 
$p_1$ the prime immediately preceding $p_0$,
\begin{equation}\label{eq:victo}\begin{aligned}
\lambda(q) &\leq \left(
1.90516 \log p_1 \cdot \frac{7.45235 \cdot \left(
 \frac{e^{0.74914 p_1^{1/3}}}{6.62365}\right)}{
0.37268 (\log q - \log p_1) + 0.02741}
\right)^3\\
&\leq \frac{190.272 (\log p_1)^3 e^{2.24742 p_1^{1/3}}}{
(\log q - \log p_1 + 0.07354)^3}.
\end{aligned}\end{equation}

It is clear from (\ref{eq:armor})  
that $\varpi(q)$ is increasing 
as soon as \[q\geq \max(Q_{0,\min},Q_{0,\min}^{1-\omega(\rho)}/c_{\rho,2})\] 
and $c(c_+) q^\tau > \log q + 1$, since then $\varpi_0(q)$ is increasing
and $\varpi(q) = \varpi_0(q)$. Here it is useful to recall that
$c_{\rho,2}\geq \exp(1.4709-c_+)$, and to note that
$c(c_+) q^{\tau} - (\log q+1)$ is increasing for 
$q\geq 1/(\tau\cdot  c(c_+))^{1/\tau}$; we see also that 
$1/(\tau \cdot
 c(c_+))^{1/\tau} \leq 1/((1-0.6) e^{-\gamma} c(c_+))^{1/((1-0.6) e^{-\gamma})}$
for $\rho\leq 0.6$. A quick computation for our value of $c_+$ makes
us conclude that $q>1.12 Q_{0,\min}=112000$ is a sufficient condition for
$\varpi(q)$ to be equal to $\varpi_0(q)$ and for $\varpi_0(q)$ to be increasing.

Since (\ref{eq:victo}) is decreasing on $q$ for $p_1$ fixed, and
$\varpi_0(q)$ is decreasing on $\rho$ and increasing on $q$, we set 
$\rho = 0.6$ and check that then
\[\varpi_0\left(2.2\cdot 10^{10}\right) \geq 846.765,\]
whereas, by (\ref{eq:victo}),
\[\lambda(2.2\cdot 10^{10}) \leq 838.227 < 846.765;\]
this is enough to ensure that $\lambda(q)<\varpi_0(q)$ for 
$2.2\cdot 10^{10} \leq q < \prod_{p\leq 31} p$.

Let us now give some rough bounds that will be enough to cover the
case $q\geq \prod_{p\leq 31} p$. First, as we already discussed,
$\varpi(q)= \varpi_0(q)$ and, since $c(c_+) q^\tau > \log q + 1$,
\begin{equation}\label{eq:drolo}
\varpi_0(q) \geq (c(c_+) q^\tau - \log q)^{\frac{1}{1-\tau}} \geq
(0.911 q^{0.224} - \log q)^{1.289} \geq q^{0.2797}\end{equation}
by $q\geq \prod_{p\leq 31} p$. We are in the range $\prod_{p\leq p_1} p \leq
q\leq \prod_{p\leq p_0} p$, where $p_1<p_0$ are two consecutive primes with
$p_1\geq 31$. By \cite[(3.16)]{MR0137689} and a computation for 
$31\leq q< 200$, we know that $\log q \geq \prod_{p\leq p_1} \log p \geq
0.8009 p_1$. By (\ref{eq:victo}) and (\ref{eq:drolo}), 
it follows that we just have to show
that
\[e^{0.224 t} > \frac{190.272 (\log t)^3 e^{2.24742 t^{1/3}}}{
(0.8009 t - \log t + 0.07354)^3}\]
for $t\geq 31$. Now, $t\geq 31$ implies 
$0.8009 t - \log t + 0.07354 \geq 0.6924 t$, and so, taking logarithms
we see that we just have to verify
\begin{equation}\label{eq:mutuso}
0.224 t -  2.24742 t^{1/3} > 3 \log \log t - 3 \log t + 6.3513\end{equation}
for $t\geq 31$, and, since the left side is increasing and the right
side is decreasing for $t\geq 31$, this is trivial to check.

We conclude that $\varpi(q)>\lambda(q)$ whenever $q\geq 2.2\cdot 10^{10}$.

It remains to see how we can relax this assumption if we assume
that $2\cdot 3\cdot 5\cdot 7\nmid q$.  We repeat the same
analysis as before, using (\ref{eq:modowo}) and (\ref{eq:hipowo}) instead
of (\ref{eq:modo}) and (\ref{eq:hipo}). For $p_1\geq 29$,
\[\mathop{\prod_{p\leq p_1}}_{p\ne 7}
 \frac{p}{p-1} < 1.633 \log p_1,\;\;\;\;
\mathop{\prod_{p\leq p_1}}_{p\ne 7}
f_1(p) \leq \frac{e^{0.74914 x^{1/3} - \log(1+7^{-2/3})}}{5.8478} \leq
\frac{e^{0.74914 x^{1/3}}}{7.44586}\]
and $\sum_{p\leq p_1: p\ne 7} (\log p)/p < \log p_1 - (\log 7)/7$. 
So, for $q< \prod_{p\leq p_0: p\ne 7} p$, and $p_1\geq 29$ 
the prime immediately preceding $p_0$,
\[\begin{aligned}
\lambda(q) 
&\leq \left(
1.633 \log p_1 \cdot \frac{7.45235 \cdot \left(
 \frac{e^{0.74914 p_1^{1/3}}}{7.44586}\right)}{
0.37268 \left(\log q - \log p_1 + \frac{\log 7}{7}\right) + 0.02741}
\right)^3\\
&\leq \frac{84.351 (\log p_1)^3 e^{2.24742 p_1^{1/3}}}{
(\log q - \log p_1 + 0.35152)^3}.
\end{aligned}\] Thus we obtain, just like before, that
\[\varpi_0(3.3\cdot 10^9)\geq 477.465,\;\;\;\;\;\;\;
\lambda(3.3\cdot 10^9) \leq 475.513 < 477.465.\]
We also check that $\varpi_0(q_0)\geq 916.322$ is greater than
$\lambda(q_0)\leq 429.731$ for $q_0= \prod_{p\leq 31: p\ne 7} p$.
The analysis for $q\geq \prod_{p\leq 37: p\ne 7} p$ is also just like before:
since $\log q \geq 0.8009 p_1 - \log 7$, we have to show that
\[\frac{e^{0.224 t}}{7} > \frac{84.351 (\log t)^3 e^{2.24742 t^{1/3}}}{(0.8009 t - \log t + 0.07354)^3}\]
for $t\geq 37$, and that, in turn, follows from
\[0.224 t - 2.24742 t^{1/3} > 3 \log \log t - 3 \log t + 6.74849,\]
which we check for $t\geq 37$ just as we checked (\ref{eq:mutuso}).

We conclude that 
$\varpi(q)>\lambda(q)$ if $q\geq 3.3\cdot 10^{9}$ and $210\nmid q$.

{\em Computation.}
Now, for $q < 3.3 \cdot 10^9$ (and also for $3.3\cdot 10^9 \leq q <
2.2\cdot 10^{10}$, $210|q$),
 we need to check that
the maximum $m_{q,R,1}$ of $\err_{q,R}$ over all $\varpi(q)\leq
R<\lambda(q)$ satisfies (\ref{eq:luce}).
Note that there is a term $\err_{q,tR}$ in (\ref{eq:luce}); we bound
it using (\ref{eq:agammen}).

Since $\log R$ is increasing on $R$
and $G_q(R)$ depends only on $\lfloor R\rfloor$, we can tell from 
(\ref{eq:mero}) that, since we are taking the maximum of $\err_{q,R}$, 
it is enough to check integer values of $R$. We check all integers $R$
in $\lbrack \varpi(q),\lambda(q))$ for all $q<3.3\cdot 10^9$
(and all  $3.3\cdot 10^9 \leq q <
2.2\cdot 10^{10}$, $210|q$) by an explicit computation.\footnote{This is
by far the heaviest computation in the present work, though it is still
rather minor (about two weeks of computing on a single core of a fairly new
(2010) desktop computer carrying out other tasks as well; this is next
to nothing compared to the computations in \cite{Plattfresh}, or even
those in \cite{MR3171101}). For the applications here,
we could have assumed $\rho\leq 8/15$, and that would have
reduced computation time drastically; the lighter assumption $\rho\leq 0.6$
was made with views to general applicability in the future.
As elsewhere in this section, 
numerical computations were carried out by the author 
in C; all floating-point operations used 
D. Platt's interval arithmetic package.}
\end{proof}

Finally, we have the trivial bound
\begin{equation}\label{eq:triwia}
\frac{G_q(Q_0/sq)}{G_q(Q/sq)} \leq 1,\end{equation}
which we shall use for $Q_0$ close to $Q$.


\begin{corollary}\label{cor:coeur}
Let $\{a_n\}_{n=1}^\infty$, $a_n\in \mathbb{C}$, be supported on the primes.
Assume that $\{a_n\}$ is in $\ell_1\cap \ell_2$ and that $a_n=0$ for $n\leq \sqrt{x}$.
Let $Q_0\geq 10^5$, $\delta_0\geq 1$ be such that $(20000 Q_0)^2 \leq 
x/2 \delta_0$; set
$Q = \sqrt{x/2 \delta_0}$. 

Let $S(\alpha) = \sum_n a_n e(\alpha n)$ for $\alpha \in \mathbb{R}/\mathbb{Z}$.
Let $\mathfrak{M}$ as in (\ref{eq:jokor}). Then, if
$Q_0\leq Q^{0.6}$,
\[\int_{\mathfrak{M}} \left|S(\alpha)\right|^2 d\alpha \leq
\frac{\log Q_0 + c_+}{\log Q + c_E} \sum_n |a_n|^2,\]
where $c_+ = 1.36$ and $c_E = \gamma + \sum_{p\geq 2} (\log p)/(p(p-1)) = 
1.3325822\dotsc$.

Let $\mathfrak{M}_{\delta_0,Q_0}$ as in (\ref{eq:majdef}). Then, if
$(2 Q_0)\leq (2 Q)^{0.6}$,
\begin{equation}\label{eq:pensat}
\int_{\mathfrak{M}_{\delta_0,Q_0}} \left|S(\alpha)\right|^2 d\alpha \leq
\frac{\log 2 Q_0 + c_+}{\log 2 Q + c_E} \sum_n |a_n|^2.\end{equation}
\end{corollary}
Here, of course, $\int_{\mathbb{R}/\mathbb{Z}} \left|S(\alpha)\right|^2 d\alpha
= \sum_n |a_n|^2$ (Plancherel). If $Q_0>Q^{0.6}$, we will use the trivial bound
\begin{equation}\label{eq:trivo}
\int_{\mathfrak{M}_{\delta_0,r}} \left|S(\alpha)\right|^2 d\alpha \leq
\int_{\mathbb{R}/\mathbb{Z}} \left|S(\alpha)\right|^2 d\alpha =
 \sum_n |a_n|^2.
\end{equation} 
\begin{proof}
Immediate from Prop.~\ref{prop:ramar}, Prop.~\ref{prop:bellen} and
Prop.~\ref{prop:espagn}.
\end{proof}
Obviously, one can also give a statement derived from
Prop.~\ref{prop:ramar}; the resulting bound is
\[\int_{\mathfrak{M}} |S(\alpha)|^2 d\alpha \leq 
\frac{\log Q_0 + c_+}{\log Q + c_E} \sum_n |a_n|^2,\]
where $\mathfrak{M}$ is as in (\ref{eq:jokor}).

We also record the large-sieve form of the result.
\begin{corollary}\label{cor:carnatio}
Let $N\geq 1$.
Let $\{a_n\}_{n=1}^\infty$, $a_n\in \mathbb{C}$, be supported on the integers
$n\leq N$.
Let $Q_0\geq 10^5$, $Q\geq 20000 Q_0$. Assume that $a_n=0$ for every $n$ for
which there is a $p\leq Q$ dividing $n$.

Let $S(\alpha) = \sum_n a_n e(\alpha n)$ for $\alpha \in \mathbb{R}/\mathbb{Z}$.
Then, if $Q_0\leq Q^{0.6}$,
\[\sum_{q\leq Q_0} \mathop{\sum_{a \mo q}}_{(a,q)=1} 
       \left|S(a/q)\right|^2 d\alpha \leq
\frac{\log Q_0 + c_+}{\log Q + c_E} \cdot (N + Q^2) \sum_n |a_n|^2
,\]
where $c_+ = 1.36$ and $c_E = \gamma + \sum_{p\geq 2} (\log p)/(p(p-1)) = 
1.3325822\dotsc$.
\end{corollary}
\begin{proof}
Proceed as Ramar\'e does in the proof of \cite[Thm. 5.2]{MR2493924},
with $\mathscr{K}_q = \{a\in \mathbb{Z}/q\mathbb{Z}: (a,q)=1\}$ and $u_n = a_n$); in particular,
apply \cite[Thm. 2.1]{MR2493924}.
The proof of \cite[Thm. 5.2]{MR2493924} shows that
\[\sum_{q\leq Q_0} \mathop{\sum_{a \mod q}}_{(a,q)=1} 
       \left|S(a/q)\right|^2 d\alpha \leq
\max_{q\leq Q_0} \frac{G_q(Q_0)}{G_q(Q)} \cdot
\sum_{q\leq Q_0} \mathop{\sum_{a \mod q}}_{(a,q)=1} 
       \left|S(a/q)\right|^2 d\alpha.\]
Now, instead of using the easy inequality $G_q(Q_0)/G_q(Q)\leq G_1(Q_0)/G_1(Q/Q_0)$,
use Prop.~\ref{prop:espagn}.
\end{proof}

\begin{center}
* * *
\end{center}

It would seem desirable to prove a result such as Prop.~\ref{prop:espagn}
(or Cor.~\ref{cor:coeur}, or Cor.~\ref{cor:carnatio})
without computations and with conditions that are as weak as possible.
Since, as we said, we cannot make $c_+$ equal to $c_E$, and since 
$c_+$ does have to increase when the conditions are weakened (as is shown by
computations; this is not an artifact of our method of proof)
 the right goal might be to show that 
the maximum of $G_q(Q_0/sq)/G_q(Q/sq)$ is reached when $s=q=1$.

However, this is also untrue without conditions. For instance, for
$Q_0=2$ and $Q$ large, the value of $G_q(Q_0/q)/G_q(Q/q)$ at $q=2$ is larger
than at $q=1$: by (\ref{eq:malito}),
\[\begin{aligned}
\frac{G_2\left(\frac{Q_0}{2}\right)}{G_2\left(\frac{Q}{2}\right)} &\sim
\frac{1}{\frac{1}{2} \left(\log \frac{Q}{2} + c_E + \frac{\log 2}{2}\right)}
\\ &= \frac{2}{\log Q + c_E - \frac{\log 2}{2}} > \frac{2}{\log Q + c_E}
\sim \frac{G(Q_0)}{G(Q)}.\end{aligned}\]
 Thus, at the very least,
a lower bound on $Q_0$ is needed as a condition. This also dims the hopes
somewhat for a combinatorial proof of $G_q(Q_0/q) G(Q) \leq G_q(Q/q)
G(Q_0)$; at any rate, while such a proof would be welcome, it could not be
extremely straightforward, since there are terms in $G_q(Q_0/q) G(Q)$
that do not appear in $G_q(Q/q) G(Q_0)$.

\chapter{The integral over the minor arcs}

The time has come to bound the part of our triple-product integral
(\ref{eq:osto}) that comes from the minor arcs $\mathfrak{m}\subset
\mathbb{R}/\mathbb{Z}$. We have an $\ell_\infty$ estimate (from
Prop.~\ref{prop:gorsh}, based on Theorem \ref{thm:minmain}) and an $\ell_2$ estimate
(from \S \ref{subs:boquo}). Now we must put them together.

There are two ways in which we must be careful. A trivial bound
of the form $\ell_3^3 = \int |S(\alpha)|^3 d\alpha \leq \ell_2^2 \cdot \ell_\infty$
would introduce a fatal factor of $\log x$ coming from $\ell_2$. We avoid
this by using the fact that we have $\ell_2$ estimates over 
$\mathfrak{M}_{\delta_0,Q_0}$ for varying $Q_0$. 

We must also remember
to substract the major-arc contribution from our estimate for
 $\mathfrak{M}_{\delta_0,Q_0}$; this is why we were careful to give
a lower bound in Lem.~\ref{lem:drujal}, as opposed to just the upper
bound (\ref{eq:mardi}).

\section{Putting together $\ell_2$ bounds over arcs and
$\ell_\infty$ bounds}

Let us start with a simple lemma -- essentially a way to obtain upper bounds
by means of summation by parts.
\begin{lemma}\label{lem:jardinbota}
Let $f,g:\{a,a+1,\dotsc,b\}\to \mathbb{R}_0^+$, where $a,b\in
\mathbb{Z}^+$. Assume that, for all $x\in \lbrack a,b\rbrack$,
\begin{equation}\label{eq:gorto}
\sum_{a\leq n\leq x} f(n) \leq F(x),
\end{equation}
where $F:\lbrack a,b\rbrack\to \mathbb{R}$ is continuous,
piecewise differentiable and non-decreasing. Then
\[
\sum_{n=a}^b f(n) \cdot g(n) \leq (\max_{n\geq a} g(n))\cdot F(a) 
+ \int_a^{b} (\max_{n\geq u} g(n)) \cdot F'(u) du .
\]
\end{lemma}
\begin{proof}
Let $S(n) = \sum_{m=a}^n f(m)$. Then, by partial summation,
\begin{equation}\label{eq:marshti}
\sum_{n=a}^b f(n) \cdot g(n) \leq S(b) g(b) + \sum_{n=a}^{b-1} S(n) (g(n) -
g(n+1)) .
\end{equation}
Let $h(x) = \max_{x\leq n\leq b} g(n)$. Then $h$ is non-increasing. Hence
(\ref{eq:gorto}) and (\ref{eq:marshti}) imply that
 \[\begin{aligned}
\sum_{n=a}^b f(n) g(n) &\leq \sum_{n=a}^b f(n) h(n)\\
&\leq S(b) h(b) + \sum_{n=a}^{b-1} S(n) (h(n) - h(n+1)) \\
&\leq F(b) h(b) + \sum_{n=a}^{b-1} F(n) (h(n) - h(n+1)) .
\end{aligned}\]
In general, for $\alpha_n\in \mathbb{C}$,  $A(x)=\sum_{a\leq n\leq x} \alpha_n$
and $F$ continuous and piecewise differentiable on $\lbrack a,x\rbrack$,
\begin{equation}\label{eq:jokors}
\sum_{a\leq n\leq x} \alpha_n F(x) = A(x) F(x) - \int_a^x A(u) F'(u) du .
\;\;\;\;\;\;\;\text{({\em Abel summation})}\end{equation}
Applying this with $\alpha_n = h(n) - h(n+1)$ and $A(x) = \sum_{a\leq n\leq
x} \alpha_n  = h(a) - h(\lfloor x\rfloor + 1)$, we obtain
\[\begin{aligned}
\sum_{n=a}^{b-1} &F(n) (h(n)-h(n+1)) \\&= (h(a)- h(b)) F(b-1) - 
\int_a^{b-1} (h(a) - h(\lfloor u\rfloor +1)) F'(u) du\\
&= h(a) F(a) - h(b) F(b-1) + \int_a^{b-1} h(\lfloor u\rfloor+1) F'(u) du\\
&= h(a) F(a) - h(b) F(b-1) + \int_a^{b-1} h(u) F'(u) du\\
&= h(a) F(a) - h(b) F(b) + \int_a^{b} h(u) F'(u) du,\end{aligned}\]
since $h(\lfloor u\rfloor + 1) = h(u)$ for $u\notin \mathbb{Z}$.
Hence
\[\sum_{n=a}^b f(n) g(n) \leq h(a) F(a) + \int_a^{b} h(u) F'(u) du .\] 
\end{proof}

We will now see our main application of Lemma \ref{lem:jardinbota}.
We have to bound an integral of the form
$\int_{\mathfrak{M}_{\delta_0,r}} |S_1(\alpha)|^2 |S_2(\alpha)|  d\alpha$,
where $\mathfrak{M}_{\delta_0,r}$ is 
a union of arcs defined as in (\ref{eq:majdef}). Our inputs are (a) a bound
on integrals of the form $\int_{\mathfrak{M}_{\delta_0,r}} |S_1(\alpha)|^2 d\alpha$,
(b) a bound on $|S_2(\alpha)|$ for $\alpha\in 
(\mathbb{R}/\mathbb{Z})\setminus \mathfrak{M}_{\delta_0,r}$. The input of type
(a) is what we derived in \S \ref{subs:ramar} and \S \ref{subs:boquo}; the
input of type (b) is a minor-arcs bound, and as such was the main subject
of Part \ref{part:min}.

\begin{prop}\label{prop:palan}
Let $S_1(\alpha) = \sum_n a_n e(\alpha n)$, $a_n \in \mathbb{C}$,
$\{a_n\}$ in $L^1$. 
Let $S_2:\mathbb{R}/\mathbb{Z}\to \mathbb{C}$ be continuous.
Define $\mathfrak{M}_{\delta_0,r}$ as in (\ref{eq:majdef}).

Let $r_0$ be a positive integer not greater than $r_1$. 
Let $H:\lbrack r_0,r_1\rbrack
 \to \mathbb{R}^+$ be a continuous, piecewise differentiable, non-decreasing
function such that
\begin{equation}\label{eq:qewer}
\frac{1}{\sum |a_n|^2} \int_{\mathfrak{M}_{\delta_0,r+1}} |S_1(\alpha)|^2 d\alpha \leq 
H(r)
\end{equation}
for some $\delta_0 \leq x/2 r_1^2$ and
all $r\in \lbrack r_0,r_1\rbrack$. Assume, moreover, that $H(r_1)=1$.
Let $g:\lbrack r_0,
r_1\rbrack \to \mathbb{R}^+$ be a non-increasing function such that
\begin{equation}\label{eq:rien}
\max_{\alpha \in (\mathbb{R}/\mathbb{Z})\setminus \mathfrak{M}_{\delta_0,r}} 
|S_2(\alpha)| \leq g(r)
\end{equation}
for all $r\in \lbrack r_0,r_1\rbrack$ and $\delta_0$ as above. 

Then
\begin{equation}\label{eq:malmu}\begin{aligned}
\frac{1}{\sum_n |a_n|^2} 
&\int_{(\mathbb{R}/\mathbb{Z})\setminus 
\mathfrak{M}_{\delta_0,r_0}} |S_1(\alpha)|^2 |S_2(\alpha)|  d\alpha 
\\ &\leq
g(r_0) \cdot (H(r_0) - I_0) + \int_{r_0}^{r_1} g(r) H'(r) dr ,
\end{aligned}\end{equation}
where
\begin{equation}\label{eq:pentimento}
I_0 = \frac{1}{\sum_n |a_n|^2} \int_{\mathfrak{M}_{\delta_0,r_0}} |S_1(\alpha)|^2 d\alpha.\\
\end{equation}
\end{prop}
The condition $\delta_0\leq x/2 r_1^2$ is there just to ensure that the
arcs in the definition of $\mathfrak{M}_{\delta_0,r}$ do not overlap for $r\leq r_1$.
\begin{proof}
For $r_0\leq r< r_1$, let
\[f(r) = \frac{1}{\sum_n |a_n|^2}  \int_{\mathfrak{M}_{\delta_0,r+1}\setminus \mathfrak{M}_{\delta_0,r}} |S_1(\alpha)|^2 d\alpha.\]
Let
\[f(r_1) = \frac{1}{\sum_n |a_n|^2}  \int_{(\mathbb{R}/\mathbb{Z})\setminus
  \mathfrak{M}_{\delta_0,r_1}} |S_1(\alpha)|^2 d\alpha.\]
Then, by (\ref{eq:rien}),
\[\frac{1}{\sum_n |a_n|^2} \int_{(\mathbb{R}/\mathbb{Z})\setminus \mathfrak{M}_{\delta_0,r_0}} |S_1(\alpha)|^2 |S_2(\alpha)|
d\alpha \leq \sum_{r=r_0}^{r_1} f(r) g(r).\]

By (\ref{eq:qewer}),
\begin{equation}\label{eq:gogol}\begin{aligned}
\sum_{r_0\leq r\leq x} f(r) &= \frac{1}{\sum_n |a_n|^2} 
\int_{\mathfrak{M}_{\delta_0,x+1}\setminus \mathfrak{M}_{\delta_0,r_0}} |S_1(\alpha)|^2 d\alpha \\ &= 
\left(
\frac{1}{\sum_n |a_n|^2}  \int_{\mathfrak{M}_{\delta_0,x+1}} |S_1(\alpha)|^2 d\alpha
\right) - I_0 \leq H(x) - I_0
\end{aligned}\end{equation}
for $x\in \lbrack r_0,r_1)$. Moreover,
\[\begin{aligned}\sum_{r_0\leq r\leq r_1} f(r) &= 
\frac{1}{\sum_n |a_n|^2}  \int_{(\mathbb{R}/\mathbb{Z})\setminus \mathfrak{M}_{\delta_0,r_0}} |S_1(\alpha)|^2\\
&= \left(
\frac{1}{\sum_n |a_n|^2}  \int_{\mathbb{R}/\mathbb{Z}} |S_1(\alpha)|^2\right) - I_0
= 1 - I_0 = H(r_1) - I_0.\end{aligned}\]

We let $F(x) = H(x) - I_0$
and apply Lemma \ref{lem:jardinbota} with $a=r_0$, $b=r_1$. We obtain that
\[\begin{aligned}
\sum_{r=r_0}^{r_1} f(r) g(r) &\leq (\max_{r\geq r_0} g(r)) F(r_0) +
\int_{r_0}^{r_1} (\max_{r\geq u} g(r))  F'(u)\; du\\
&\leq g(r_0) (H(r_0)- I_0) + 
\int_{r_0}^{r_1} g(u) H'(u)\; du.\end{aligned}\]
\end{proof}

\section{The minor-arc total}

We now apply Prop.~\ref{prop:palan}.
Inevitably, the main statement
involves some integrals that will have to be evaluated
at the end of the section.

\begin{theorem}\label{thm:ostop}
Let $x\geq 10^{25}\cdot \varkappa$, where
$\varkappa\geq 1$.
Let
\begin{equation}\label{eq:lalaz}
S_\eta(\alpha,x) = \sum_n \Lambda(n) e(\alpha n) \eta(n/x).\end{equation}
Let $\eta_*(t) = (\eta_2 \ast_M \varphi)(\varkappa t)$, where $\eta_2$ is as in (\ref{eq:meichu})
and $\varphi: \lbrack 0,\infty)\to \lbrack 0, \infty)$ is continuous and in $\ell^1$.
Let $\eta_+:\lbrack 0,\infty)\to \lbrack 0,\infty)$ 
be a bounded, piecewise differentiable function with $\lim_{t\to \infty}
\eta_+(t)=0$.
Let $\mathfrak{M}_{\delta_0,r}$ be as in (\ref{eq:majdef}) with $\delta_0=8$.
Let $10^5 \leq r_0 < r_1$, where $r_1 = (3/8) (x/\varkappa)^{4/15}$.
Let $g(r) = g_{x/\varkappa,\varphi}(r)$, where
\begin{equation}\label{eq:jadaja}
g_{y,\varphi}(r) =
\frac{(R_{y,K,\varphi,2 r} \log 2 r + 0.5) \sqrt{\digamma(r)} + 2.5}{\sqrt{2
    r}}  +  \frac{L_{2r}}{r} +3.36 K^{1/6} y^{-1/6},\end{equation}
just as in (\ref{eq:basia}), and 
$K = \log(x/\kappa)/2$. Here $R_{y,K,\phi,t}$ is as in (\ref{eq:basia}),
and $L_t$ is as in (\ref{eq:veror}).

Denote
\[Z_{r_0} = \int_{(\mathbb{R}/\mathbb{Z})\setminus \mathfrak{M}_{8,r_0}} 
|S_{\eta_*}(\alpha,x)| |S_{\eta_+}(\alpha,x)|^2 d\alpha.\]
Then
\[Z_{r_0} \leq \left(\sqrt{\frac{|\varphi|_1 x}{\varkappa} (M + T) } + 
\sqrt{S_{\eta_*}(0,x) \cdot E}\right)^2,\]
where
\begin{equation}\label{eq:georgic}\begin{aligned}
S &= \sum_{p>\sqrt{x}} (\log p)^2 \eta_+^2(n/x),\\
T &= C_{\varphi,3}\left(\frac{1}{2} \log \frac{x}{\varkappa}\right) \cdot
(S - (\sqrt{J} - \sqrt{E})^2),\\
J&= \int_{\mathfrak{M}_{8,r_0}} |S_{\eta_+}(\alpha,x)|^2\; d\alpha,\\
E &= 
\left((C_{\eta_+,0} + C_{\eta_+,2}) \log x + (2 C_{\eta_+,0} + C_{\eta_+,1})\right)
\cdot x^{1/2},
\end{aligned}\end{equation}
\begin{equation}\label{eq:malus}\begin{aligned}
C_{\eta_+,0} &= 0.7131 \int_0^\infty \frac{1}{\sqrt{t}}
(\sup_{r\geq t} \eta_+(r))^2 dt,\\
C_{\eta_+,1} &= 0.7131 \int_1^\infty \frac{\log t}{\sqrt{t}}
(\sup_{r\geq t} \eta_+(r))^2 dt,\\
C_{\eta_+,2} &= 0.51942 |\eta_+|_\infty^2,\\
C_{\varphi,3}(K) &= \frac{1.04488}{|\varphi|_1} \int_0^{1/K} |\varphi(w)| dw
\end{aligned}\end{equation}
and 
\begin{equation}\label{eq:gypo}\begin{aligned}
M &= 
g(r_0) \cdot 
\left(\frac{\log (r_0+1)+c^+}{\log \sqrt{x} + c^-} \cdot S -
(\sqrt{J}-\sqrt{E})^2 \right)
\\ &+  \left(\frac{2}{\log x + 2 c^-}
\int_{r_0}^{r_1} \frac{g(r)}{r}  dr + \left(
\frac{7}{15}
+ \frac{- 2.14938 + \frac{8}{15} \log \varkappa}{\log x + 2 c^-}\right) g(r_1)\right)
\cdot S 
\end{aligned}\end{equation}
where 
$c_+ = 2.0532$ and
$c_- = 0.6394$.
\end{theorem}
\begin{proof}
Let $y = x/\varkappa$.
 Let $Q = (3/4) y^{2/3}$, as in Thm.~\ref{thm:minmain}
(applied with $y$ instead of $x$). 
Let $\alpha\in (\mathbb{R}/\mathbb{Z})\setminus \mathfrak{M}_{8,r}$,
where $r_0\leq r\leq y^{1/3}/6$ and $y$ is used instead of $x$ to define
$\mathfrak{M}_{8,r}$ (see (\ref{eq:majdef})). There
exists an approximation $2 \alpha =a/q + \delta/y$ with $q\leq Q$,
$|\delta|/y\leq 1/q Q$. Thus, $\alpha = a'/q' +\delta/2 y$, where
either $a'/q' = a/2q$ or $a'/q' = (a+q)/2q$ holds.
(In particular, if $q'$ is odd, then $q'=q$; if $q'$ is even, then 
$q'= 2q$.)

There are three cases:
\begin{enumerate}
\item $q\leq r$. Then either (a) $q'$ is odd and $q'\leq r$ or (b)
$q'$ is even and $q'\leq 2 r$.
Since $\alpha$ is not in $\mathfrak{M}_{8,r}$, then, by
definition (\ref{eq:majdef}), $|\delta|/2 y \geq
\delta_0 r/2 q y$, and so $|\delta|\geq  
\delta_0 r/q = 8 r/q$. In particular, $|\delta|\geq 8$.

Thus, by Prop.~\ref{prop:gorsh},
\begin{equation}\label{eq:gropa}
|S_{\eta_*}(\alpha,x)| = 
|S_{\eta_2\ast_M \phi}(\alpha,y)| 
\leq g_{y,\varphi}\left(\frac{|\delta|}{8} q\right)
\cdot |\varphi|_1 y \leq g_{y,\varphi}(r) \cdot |\varphi|_1 y,
\end{equation}
where we use the fact that $g(r)$ is a non-increasing function 
(Lemma \ref{lem:vinc}).
\item $r < q \leq y^{1/3}/6$. Then, 
 by Prop.~\ref{prop:gorsh} and Lemma \ref{lem:vinc},
\begin{equation}\label{eq:grope}\begin{aligned}
|S_{\eta_*}(\alpha,x)| &= 
|S_{\eta_2\ast_M \phi}(\alpha,y)| 
\leq g_{y,\varphi}\left(\max\left(\frac{|\delta|}{8},1\right) q
\right) \cdot |\varphi|_1 y \\ 
&\leq g_{y,\varphi}(r) \cdot |\varphi|_1 y.\end{aligned}\end{equation}
\item $q> y^{1/3}/6$. Again by Prop.~\ref{prop:gorsh},
\begin{equation}\label{eq:gropi}
|S_{\eta_*}(\alpha,x)| = 
|S_{\eta_2\ast_M \phi}(\alpha,y)| 
\leq \left(h\left(\frac{y}{K}\right) + 
C_{\varphi,3}(K)\right) |\varphi|_1 y, 
\end{equation}
where $h(x)$ is as in (\ref{eq:flou}). (Of course, $C_{\varphi,3}(K)$, as in
(\ref{eq:malus}), is equal to $C_{\varphi,0,K}/|\phi|_1$, where
$C_{\varphi,0,K}$ is as in (\ref{eq:midin}).)
 We set $K = (\log y)/2$. Since $y = 
x/\kappa\geq 10^{25}$, it follows that
$y/K = 2 y/\log y > 3.47\cdot 10^{23} >  2.16 \cdot 10^{20}$.
\end{enumerate}
Let 
\[\begin{aligned}
r_1 = \frac{3}{8} y^{4/15},\;\;\;\;\;\;\;\;
g(r) = \begin{cases} g_{y,\varphi}(r) &\text{if $r\leq r_1$},\\
g_{y,\varphi}(r_1) &\text{if $r> r_1$.}\end{cases}\end{aligned}\]
By Lemma \ref{lem:vinc}, for $r\geq 670$,
$g(r)$ is a non-increasing function and $g(r)\geq g_{y,\phi}(r)$.
Moreover, by
Lemma \ref{lem:gosia}, $g_{y,\phi}(r_1) \geq h(2 y/\log y)$, where $h$
is as in (\ref{eq:flou}), and so
$g(r)\geq h(2 y/\log y)$ for all $r\geq r_0 \geq 670$. Thus, we
have shown that
\begin{equation}\label{eq:bertru}
|S_{\eta_*}(y,\alpha)|\leq 
\left(g(r) + C_{\varphi,3}\left(\frac{\log y}{2}\right)\right) \cdot |\varphi|_1 y \end{equation}
for all $\alpha\in (\mathbb{R}/\mathbb{Z})\setminus \mathfrak{M}_{8,r}$.

We first need to undertake the fairly
dull task of getting non-prime or small $n$ out of the sum defining
$S_{\eta_+}(\alpha,x)$. 
Write \[\begin{aligned}
S_{1,\eta_+}(\alpha,x) &= 
\sum_{p>\sqrt{x}} (\log p) e(\alpha p) \eta_+(p/x),\\
S_{2,\eta_+}(\alpha,x) &= 
\mathop{\sum_{\text{$n$ non-prime}}}_{n>\sqrt{x}} \Lambda(n) e(\alpha n) \eta_+(n/x) +
\sum_{n\leq \sqrt{x}} \Lambda(n) e(\alpha n)\eta_+(n/x).\end{aligned}\]
By the triangle inequality (with weights $|S_{\eta_+}(\alpha,x)|$),
\[\begin{aligned}&\sqrt{\int_{(\mathbb{R}/\mathbb{Z})\setminus \mathfrak{M}_{8,r_0}} 
|S_{\eta_*}(\alpha,x)| |S_{\eta_+}(\alpha,x)|^2 d\alpha}\\ &\leq \sum_{j=1}^2
\sqrt{\int_{(\mathbb{R}/\mathbb{Z})\setminus \mathfrak{M}_{8,r_0}} 
|S_{\eta_*}(\alpha,x)| |S_{j,\eta_+}(\alpha,x)|^2 d\alpha}.\end{aligned}\]
Clearly, \[\begin{aligned}
&\int_{(\mathbb{R}/\mathbb{Z)}\setminus \mathfrak{M}_{8,r_0}}
|S_{\eta_*}(\alpha,x)| |S_{2,\eta_+}(\alpha,x)|^2 d\alpha \\ &\leq
\max_{\alpha \in \mathbb{R}/\mathbb{Z}}
 \left|S_{\eta_*}(\alpha,x)\right| \cdot \int_{\mathbb{R}/\mathbb{Z}}
 |S_{2,\eta_+}(\alpha,x)|^2 d\alpha\\ &\leq
\sum_{n=1}^\infty \Lambda(n) \eta_*(n/x)
\cdot
\left(\sum_{\text{$n$ non-prime}} \Lambda(n)^2 \eta_+(n/x)^2 + 
\sum_{n\leq \sqrt{x}} \Lambda(n)^2 \eta_+(n/x)^2\right).\end{aligned}\]
Let $\overline{\eta_+}(z) = \sup_{t\geq z} \eta_+(t)$.
Since $\eta_+(t)$ tends to $0$ as $t\to \infty$, so does $\overline{\eta_+}$.
By  \cite[Thm. 13]{MR0137689}, partial summation
and integration by parts,
\[\begin{aligned}
\sum_{\text{$n$ non-prime}} &\Lambda(n)^2 \eta_+(n/x)^2 \leq
\sum_{\text{$n$ non-prime}} \Lambda(n)^2 \overline{\eta_+}(n/x)^2\\ &\leq
-\int_1^\infty \left(\mathop{\sum_{n\leq t}}_{\text{$n$ non-prime}} \Lambda(n)^2\right) 
\left( \overline{\eta_+}^2 (t/x) \right)' dt\\
&\leq
-\int_1^\infty (\log t) \cdot 1.4262 \sqrt{t} 
\left( \overline{\eta_+}^2 (t/x) \right)'
dt\\
&\leq  0.7131\int_1^{\infty} \frac{\log e^2 t}{\sqrt{t}} \cdot 
\overline{\eta_+}^2
\left(\frac{t}{x}\right) dt\\ &=
\left(0.7131 \int_{1/x}^\infty \frac{2 + \log t x}{\sqrt{t}} 
\overline{\eta_+}^2(t) dt \right) \sqrt{x},
\end{aligned}\]
while, by \cite[Thm. 12]{MR0137689},
\[\begin{aligned}
\sum_{n\leq \sqrt{x}} \Lambda(n)^2 \eta_+(n/x)^2 &\leq \frac{1}{2} 
|\eta_+|_\infty^2
(\log x) \sum_{n\leq \sqrt{x}} \Lambda(n)\\
&\leq 0.51942 |\eta_+|_\infty^2 \cdot \sqrt{x} \log x.
\end{aligned}\]
This shows that
\[\int_{(\mathbb{R}/\mathbb{Z)}\setminus \mathfrak{M}_{8,r_0}}
|S_{\eta_*}(\alpha,x)| |S_{2,\eta_+}(\alpha,x)|^2 d\alpha \leq 
\sum_{n=1}^\infty \Lambda(n) \eta_*(n/x) \cdot
E = S_{\eta_*}(0,x)\cdot E,\]
where $E$ is as in (\ref{eq:georgic}).

It remains to bound
\begin{equation}\label{eq:flashgo}
\int_{(\mathbb{R}/\mathbb{Z})\setminus \mathfrak{M}_{8,r_0}} 
|S_{\eta_*}(\alpha,x)| |S_{1,\eta_+}(\alpha,x)|^2 d\alpha .
\end{equation}
We wish to apply Prop.~\ref{prop:palan}.
Corollary \ref{cor:coeur} gives us an input of type (\ref{eq:qewer});
we have just derived a bound 
(\ref{eq:bertru}) that provides an input of type (\ref{eq:rien}).
More precisely, by (\ref{eq:pensat}), 
(\ref{eq:qewer}) holds with
\[H(r) = \begin{cases}
\frac{\log (r+1) + c^+}{\log \sqrt{x} + c^-}
&\text{if $r< r_1$},\\ 1 &\text{if $r\geq r_1$,}\end{cases}\]
where $c^+ = 2.0532 > \log 2 + 1.36$ and $c^- = 0.6394
< \log(1/\sqrt{2\cdot 8}) + \log 2 + 1.3325822$.
(We can apply Corollary
\ref{cor:coeur} because $2 (r_1+1) =   (3/4) x^{4/15}+2
\leq (2 \sqrt{x/16})^{0.6}$ for $x\geq 10^{25}$ 
(or even for $x\geq 100000$).)
Since $r_1 = (3/8) y^{4/15}$
and $x\geq 10^{25} \cdot
\varkappa$,
\[\begin{aligned}
\lim_{r\to r_1^+} H(r) &- \lim_{r\to r_1^-} H(r) = 1 - 
\frac{\log ((3/8) (x/\varkappa)^{4/15}+1)+c^+}{\log \sqrt{x} + c^-}\\
&\leq 1 - \left(\frac{4/15}{1/2} +
\frac{\log \frac{3}{8} + c^+ - \frac{4}{15} \log \varkappa - \frac{8}{15} c^-}{
\log \sqrt{x} + c^-}\right)\\
&\leq \frac{7}{15}
+ \frac{- 2.14938 + \frac{8}{15} \log \varkappa}{\log x + 2 c^-} .
\end{aligned}\]
We also have (\ref{eq:rien}) with \begin{equation}\label{eq:jorge}
\left(g(r) + C_{\varphi,3}\left(\frac{\log y}{2}\right)\right)
\cdot |\varphi|_1 y\end{equation}
instead of $g(r)$ (by (\ref{eq:bertru})). 
Here (\ref{eq:jorge}) is a non-increasing function of $r$ because $g(r)$ is,
as we already checked.
Hence, Prop.~\ref{prop:palan} gives us that
 (\ref{eq:flashgo}) is at most
\begin{equation}\label{eq:lili}\begin{aligned}
g(r_0) \cdot &(H(r_0) - I_0) + (1-I_0) \cdot C_{\varphi,3}\left(\frac{\log y}{2}\right)
\\ &+  \frac{1}{\log \sqrt{x} + c^-}
\int_{r_0}^{r_1} \frac{g(r)}{r+1}  dr + 
\left(\frac{7}{15}
+ \frac{- 2.14938 + \frac{8}{15} \log \varkappa}{\log x + 2 c^-}\right) g(r_1)
\end{aligned}\end{equation}
times $|\varphi|_1 y \cdot \sum_{p>\sqrt{x}} (\log p)^2
\eta_+^2(p/x)$, where
\begin{equation}
I_0 = \frac{1}{\sum_{p>\sqrt{x}} (\log p)^2 \eta_+^2(n/x)}
\int_{\mathfrak{M}_{8,r_0}} 
 |S_{1,\eta_+}(\alpha,x)|^2\; d\alpha.
\end{equation}
By the triangle inequality,
\[\begin{aligned}
 &\sqrt{\int_{\mathfrak{M}_{8,r_0}} 
 |S_{1,\eta_+}(\alpha,x)|^2\; d\alpha} =
\sqrt{\int_{\mathfrak{M}_{8,r_0}} 
 |S_{\eta_+}(\alpha,x) - S_{2,\eta_+}(\alpha,x)|^2\; d\alpha} \\ &\geq
 \sqrt{\int_{\mathfrak{M}_{8,r_0}} 
 |S_{\eta_+}(\alpha,x)|^2\; d\alpha} -
\sqrt{\int_{\mathfrak{M}_{8,r_0}} 
 |S_{2,\eta_+}(\alpha,x)|^2\; d\alpha} \\
&\geq  \sqrt{\int_{\mathfrak{M}_{8,r_0}} 
 |S_{\eta_+}(\alpha,x)|^2\; d\alpha} -
\sqrt{\int_{\mathbb{R}/\mathbb{Z}} 
 |S_{2,\eta_+}(\alpha,x)|^2\; d\alpha}. 
\end{aligned}\] 
As we already showed,
\[\int_{\mathbb{R}/\mathbb{Z}} 
 |S_{2,\eta_+}(\alpha,x)|^2\; d\alpha = 
\mathop{\sum_{\text{$n$ non-prime}}}_{\text{or $n\leq \sqrt{x}$}}
\Lambda(n)^2 \eta_+(n/x)^2\leq E.\]
Thus,
\[I_0\cdot S \geq (\sqrt{J}- \sqrt{E})^2,\]
and so we are done.


\end{proof}

We now should estimate the integral $\int_{r_0}^{r_1} \frac{g(r)}{r} dr$
in (\ref{eq:gypo}).
It is easy to see that
\begin{equation}\label{eq:ostram}
\begin{aligned}
\int_{r_0}^\infty \frac{1}{r^{3/2}} dr &= \frac{2}{r_0^{1/2}},\;\;\;\;\;
\int_{r_0}^{\infty} \frac{\log r}{r^2} dr = \frac{\log e r_0}{r_0},\;\;\;\;\;
\int_{r_0}^{\infty} \frac{1}{r^2} dr = \frac{1}{r_0},\\
\int_{r_0}^{r_1} \frac{1}{r} dr = \log \frac{r_1}{r_0},\;\;\;\;\;
&\int_{r_0}^\infty \frac{\log r}{r^{3/2}} dr = \frac{2 \log e^2 r_0}{\sqrt{r_0}}
,\;\;\;\;\;
\int_{r_0}^\infty \frac{\log 2 r}{r^{3/2}} dr = 
\frac{2 \log 2 e^2 r_0}{\sqrt{r_0}},\\
\int_{r_0}^\infty \frac{(\log 2 r)^2}{r^{3/2}} dr = 
&\frac{2 P_2(\log 2 r_0)}{\sqrt{r_0}},\;\;\;\;\;\;
\int_{r_0}^\infty \frac{(\log 2 r)^3}{r^{3/2}} dr = 
\frac{2 P_3(\log 2 r_0)}{r_0^{1/2}},\end{aligned}\end{equation}
where 
\begin{equation}\label{eq:javich}
P_2(t) = t^2 + 4 t + 8,\;\;\;\;\;\;\;\;\;
P_3(t) = t^3 + 6 t^2 + 24 t + 48.\end{equation}
We also have
\begin{equation}\label{eq:maldicho}
\int_{r_0}^\infty \frac{dr}{r^2 \log r} = E_1(\log r_0)
\end{equation}
where $E_1$ is the {\em exponential integral}
\[E_1(z) = \int_z^\infty \frac{e^{-t}}{t} dt.\] 

We must also estimate the integrals
\begin{equation}\label{eq:kurica}
\int_{r_0}^{r_1} \frac{\sqrt{\digamma(r)}}{r^{3/2}} dr,\;\;\;\;\;
\int_{r_0}^{r_1} \frac{\digamma(r)}{r^2} dr,\;\;\;\;\;
\int_{r_0}^{r_1} \frac{\digamma(r) \log r}{r^2} dr,\;\;\;\;\;
\int_{r_0}^{r_1} \frac{\digamma(r)}{r^{3/2}} dr,
\end{equation}

Clearly, $\digamma(r) - e^\gamma \log \log r = 2.50637/\log \log r$
is decreasing on $r$. Hence, for $r\geq 10^5$, 
\[\digamma(r) \leq e^\gamma \log \log r + c_\gamma,\]
where $c_\gamma = 1.025742$. Let $F(t) = e^\gamma \log t + c_\gamma$.
Then $F''(t) = -e^\gamma/t^2 < 0$.
Hence
\[\frac{d^2 \sqrt{F(t)}}{dt^2} = \frac{F''(t)}{2 \sqrt{F(t)}} - 
\frac{(F'(t))^2}{4 (F(t))^{3/2}} < 0
\] 
for all $t>0$. In other words, 
$\sqrt{F(t)}$ is convex-down, and so we can bound $\sqrt{F(t)}$
from above by $\sqrt{F(t_0)} + \sqrt{F}'(t_0)\cdot (t-t_0)$, for any 
$t\geq t_0>0$.
Hence, for $r\geq r_0\geq 10^5$,
\[\begin{aligned}
\sqrt{\digamma(r)} &\leq \sqrt{F(\log r)} \leq
\sqrt{F(\log r_0)} + \frac{d \sqrt{F(t)}}{dt}|_{t= \log r_0} \cdot
\log \frac{r}{r_0}\\
&= \sqrt{F(\log r_0)} + \frac{e^\gamma}{\sqrt{F(\log r_0)}}
\cdot \frac{\log \frac{r}{r_0}}{2 \log r_0} .
\end{aligned}\]
Thus, by (\ref{eq:ostram}),
\begin{equation}\label{eq:jot}\begin{aligned}
\int_{r_0}^{\infty} \frac{\sqrt{\digamma(r)}}{r^{3/2}} dr &\leq
\sqrt{F(\log r_0)} \left(2 - \frac{e^\gamma}{F(\log r_0)}\right) 
\frac{1}{\sqrt{r_0}} \\&+
\frac{e^\gamma}{\sqrt{F(\log r_0)} \log r_0}
\frac{\log e^2 r_0}{\sqrt{r_0}}
\\ &= \frac{2 \sqrt{F(\log r_0)}}{\sqrt{r_0}}
\left(1 + \frac{e^\gamma}{F(\log r_0) \log r_0}\right).
\end{aligned}\end{equation}

The other integrals in (\ref{eq:kurica}) are easier. Just as in
(\ref{eq:jot}), we extend
the range of integration to $\lbrack r_0, \infty\rbrack$. 
Using (\ref{eq:ostram}) and (\ref{eq:maldicho}), we obtain
\[\begin{aligned}
\int_{r_0}^\infty \frac{\digamma(r)}{r^2} dr &\leq
\int_{r_0}^{\infty} \frac{F(\log r)}{r^2} dr =
e^\gamma \left(\frac{\log \log r_0}{r_0} + E_1(\log r_0)\right) + \frac{c_\gamma}{r_0},\\
\int_{r_0}^{\infty} \frac{\digamma(r) \log r}{r^2} dr &\leq
e^\gamma \left(\frac{(1 + \log r_0) \log \log r_0 + 1}{r_0} +
E_1(\log r_0)\right)
+ \frac{ c_\gamma \log e r_0}{r_0},
\end{aligned}\]
By \cite[(6.8.2)]{MR2723248},
\[\begin{aligned}
\frac{1}{r (\log r + 1)} \leq E_1(\log r) &\leq \frac{1}{r \log r}.\\
\end{aligned}\] (The second inequality is obvious.) Hence
\[\begin{aligned}
\int_{r_0}^{\infty} \frac{\digamma(r)}{r^2} dr &\leq
\frac{e^\gamma (\log \log r_0 + 1/\log r_0) + c_\gamma}{r_0},\\
\int_{r_0}^{\infty} \frac{\digamma(r) \log r}{r^2} dr &\leq
\frac{e^\gamma \left(\log \log r_0 + \frac{1}{\log r_0}\right) +  
c_\gamma}{r_0} \cdot \log e r_0
.\end{aligned}\]
Finally,
\begin{equation}\label{eq:hunivel}\begin{aligned}
\int_{r_0}^{\infty} \frac{\digamma(r)}{r^{3/2}} &\leq
e^{\gamma}\left( \frac{2 \log \log r_0}{\sqrt{r_0}} + 2 E_1\left(\frac{\log
    r_0}{2}\right)\right) + 
\frac{2 c_\gamma}{\sqrt{r_0}}\\
&\leq \frac{2}{\sqrt{r_0}} \left(F(\log r_0) + \frac{2 e^\gamma}{\log r_0}
\right).
\end{aligned}\end{equation}


It is time to estimate
\begin{equation}\label{eq:caushwa}
\int_{r_0}^{r_1} \frac{R_{z,2r} \log 2r \sqrt{\digamma(r)}}{r^{3/2}} dr,
\end{equation}
where $z = y$ or $z=y/((\log y)/2)$ (and $y = x/\varkappa$, as before),
and where $R_{z,t}$ is as defined in (\ref{eq:veror}).
By Cauchy-Schwarz, (\ref{eq:caushwa}) is at most
\begin{equation}\label{eq:fishfish}
\sqrt{\int_{r_0}^{r_1} \frac{(R_{z,2r} \log 2r)^2}{r^{3/2}} dr} \cdot
\sqrt{\int_{r_0}^{r_1} \frac{\digamma(r)}{r^{3/2}} dr}.\end{equation}
We have already bounded the second integral. Let us look at the first one.
We can write $R_{z,t} = 0.27125 R_{z,t}^\circ + 0.41415$, where
\begin{equation}\label{eq:jamo}
R_{z,t}^\circ  = 
\log \left(1 + \frac{\log 4 t}{2 \log \frac{9 z^{1/3}}{2.004
      t}}\right).\end{equation}
Clearly,
\[R_{z,e^t/4}^\circ = \log \left(1 + \frac{t/2}{\log \frac{36
      z^{1/3}}{2.004}
- t}\right).\]
Now, for $f(t) = \log (c + a t/(b - t))$ and $t\in \lbrack 0,b)$,
\[
f'(t) = \frac{ab}{\left(c + \frac{a t}{b-t}\right) 
(b - t)^2},\;\;\;\;\;\;\;\;
f''(t) = \frac{- a b ((a-2 c) (b-2 t) - 2 c t)}{\left(c + \frac{a t}{b-t}\right)^2 
(b - t)^4}.\]
In our case, $a=1/2$, $c=1$ and $b = \log 36 z^{1/3} - \log(2.004)>0$.
Hence, for $t<b$,
\[-a b((a-2 c) (b-2 t)-2c t) = \frac{b}{2} \left(2 t + \frac{3}{2} (b - 2
  t)
\right)
= \frac{b}{2} \left(\frac{3}{2} b - t\right)>0,
\]
and so $f''(t)>0$. In other words, $t\to R_{z,e^t/4}^\circ$ is convex-up
for $t< b$, i.e., for $e^t/4 < 9 z^{1/3}/2.004$. 
It is easy to check that, since we are assuming $y\geq 10^{25}$,
\[2 r_1 = \frac{3}{16} y^{4/15} <
\frac{9}{2.004} \left(\frac{2 y}{\log y}\right)^{1/3} \leq \frac{9 z^{1/3}}{2.004}.\]
We conclude that $r\to R_{z,2 r}^\circ$ is convex-up on $\log 8 r$ 
for $r\leq r_1$, and hence so is $r\to R_{z,r}$, and so, in turn, is
$r\to R_{z,r}^2$.  Thus, 
for $r\in \lbrack r_0,r_1\rbrack$,
\begin{equation}\label{eq:strona}R_{z,2 r}^2 \leq R_{z,2 r_0}^2 
\cdot \frac{\log r_1/r}{\log r_1/r_0}
+ R_{z,2 r_1}^2 \cdot \frac{\log r/r_0}{\log r_1/r_0}
 .\end{equation}

Therefore, by (\ref{eq:ostram}),
\begin{equation}\label{eq:basmed}\begin{aligned}
&\int_{r_0}^{r_1} \frac{(R_{z,2r} \log 2r)^2}{r^{3/2}} dr
\\ &\leq \int_{r_0}^{r_1}  \left(R_{z,2 r_0}^2 
 \frac{\log r_1/r}{\log r_1/r_0}
+ R_{z,2 r_1}^2\frac{\log r/r_0}{\log r_1/r_0}\right) (\log 2 r)^2 
\frac{dr}{r^{3/2}}\\
= & \frac{2 R_{z,2 r_0}^2}{\log \frac{r_1}{r_0}} \left(\left(
\frac{P_2(\log 2 r_0)}{\sqrt{r_0}} - \frac{P_2(\log 2 r_1)}{\sqrt{r_1}}
\right) \log 2 r_1 - 
\frac{P_3(\log 2 r_0)}{\sqrt{r_0}} + \frac{P_3(\log 2 r_1)}{\sqrt{r_1}}
\right)\\
+ & \frac{2 R_{z,2 r_1}^2}{\log \frac{r_1}{r_0}} \left(
\frac{P_3(\log 2 r_0)}{\sqrt{r_0}} - \frac{P_3(\log 2 r_1)}{\sqrt{r_1}}
 -
\left(
\frac{P_2(\log 2 r_0)}{\sqrt{r_0}} - \frac{P_2(\log 2 r_1)}{\sqrt{r_1}}
\right) \log 2 r_0
\right)\\
&= 
2 \left(R_{z,2 r_0}^2 - \frac{\log 2 r_0}{\log \frac{r_1}{r_0}}
(R_{z,2r_1}^2 - R_{z,2r_0}^2)\right)
\cdot 
 \left(\frac{P_2(\log 2 r_0)}{\sqrt{r_0}} - \frac{P_2(\log 2 r_1)}{\sqrt{r_1}}\right)\\
&+ 2 \frac{R_{z,2r_1}^2 - R_{z,2r_0}^2}{\log \frac{r_1}{r_0}}
 \left(\frac{P_3(\log 2 r_0)}{\sqrt{r_0}} - \frac{P_3(\log 2
     r_1)}{\sqrt{r_1}}\right)\\
&= 
2 R_{z,2 r_0}^2
\cdot 
 \left(\frac{P_2(\log 2 r_0)}{\sqrt{r_0}} - \frac{P_2(\log 2 r_1)}{\sqrt{r_1}}\right)\\
&+ 2 \frac{R_{z,2r_1}^2 - R_{z,2r_0}^2}{\log \frac{r_1}{r_0}}
 \left(\frac{P_2^-(\log 2 r_0)}{\sqrt{r_0}} - \frac{P_3(\log 2
     r_1)- (\log 2 r_0) P_2(\log 2 r_1)}{\sqrt{r_1}}\right) ,\end{aligned}
\end{equation}
where $P_2(t)$ and $P_3(t)$ are as in (\ref{eq:javich}),
and $P_2^-(t) = P_3(t) - t P_2(t) = 2 t^2 + 16 t + 48$.

Putting all terms together, we conclude that
\begin{equation}\label{eq:byrne}
\int_{r_0}^{r_1} \frac{g(r)}{r} dr \leq
f_0(r_0,y) + f_1(r_0) + f_2(r_0,y),\end{equation}
where 
\begin{equation}\label{eq:cymba}\begin{aligned}
f_0(r_0,y) &= 
\left(\left(1 - c_\varphi\right) \sqrt{I_{0,r_0,r_1,y}}
+ c_\varphi \sqrt{I_{0,r_0,r_1,\frac{2 y}{\log y}}} \right)
\sqrt{\frac{2}{\sqrt{r_0}}  I_{1,r_0}}\\
f_1(r_0) &= \frac{\sqrt{F(\log r_0)}}{\sqrt{2 r_0}} \left(1 
 + \frac{e^\gamma}{F(\log r_0) \log r_0}\right) + \frac{5}{\sqrt{2 r_0}}
\\ &+
\frac{1}{r_0} \left(\left(\frac{13}{4} \log e r_0 + 11.07\right) J_{r_0}
+ 13.66 \log er_0 + 37.55 
\right)\\
f_2(r_0,y) &= 3.36 \frac{((\log y)/2)^{1/6}}{y^{1/6}} \log \frac{r_1}{r_0},
\end{aligned}\end{equation}
where $F(t) = e^\gamma \log t + c_\gamma$, $c_\gamma = 1.025742$,
$y = x/\varkappa$ (as usual),
\begin{equation}\label{eq:waslight}
\begin{aligned}
I_{0,r_0,r_1,z} &= 
R_{z,2 r_0}^2
\cdot 
 \left(\frac{P_2(\log 2 r_0)}{\sqrt{r_0}} - \frac{P_2(\log 2 r_1)}{\sqrt{r_1}}\right)\\
&+ \frac{R_{z,2r_1}^2 - R_{z,2r_0}^2}{\log \frac{r_1}{r_0}}
 \left(\frac{P_2^-(\log 2 r_0)}{\sqrt{r_0}} - \frac{P_3(\log 2
     r_1)- (\log 2 r_0) P_2(\log 2 r_1)}{\sqrt{r_1}}\right) 
\\
J_{r} &= F(\log r) + \frac{e^\gamma}{\log r},\;\;\;\;
I_{1,r} = F(\log r) + \frac{2 e^{\gamma}}{\log r},\;\;\;\;\;
c_\varphi = \frac{C_{\varphi,2,\frac{\log y}{2}}/|\varphi|_1}{\log \frac{\log y}{2}}
\end{aligned}\end{equation}
and $C_{\varphi,2,K}$ is as in (\ref{eq:cecidad}). 

Let us recapitulate briefly. The term $f_2(r_0,y)$ in (\ref{eq:cymba})
comes from the term $3.36 x^{-1/16}$ in (\ref{eq:syryza}). The term $f_1(r_0,y)$
includes all other terms in (\ref{eq:syryza}), except for
$R_{x,2r} \log 2r \sqrt{\digamma(r)}/(\sqrt{2 r})$. The contribution of that
last term is (\ref{eq:caushwa}), divided by $\sqrt{2}$. That, in turn,
is at most (\ref{eq:fishfish}), divided by $\sqrt{2}$. The first integral
in (\ref{eq:fishfish}) was bounded in (\ref{eq:basmed}); the second integral
was bounded in (\ref{eq:hunivel}).

\chapter{Conclusion}

We now need to gather all results, using the smoothing functions
\[\eta_* = (\eta_2 \ast_M \varphi)(\varkappa t),
\]
where  $\varphi(t) = t^2 e^{-t^2/2}$, $\eta_2 = \eta_1\ast_M \eta_1$
and  $\eta_1 = 2\cdot I_{\lbrack -1/2,1/2\rbrack}$,
and
\[ \eta_+ = h_{200}(t) t e^{-t^2/2},\]
where 
\[\begin{aligned}
h_H(t) = \int_0^\infty h(t y^{-1}) F_H(y) \frac{dy}{y},&\\
h(t) = \begin{cases} t^2 (2-t)^3 e^{t-1/2} &\text{if $t\in \lbrack
    0,2\rbrack$,}\\ 0 &\text{otherwise,}\end{cases}
\;\;\;\;\;\;\;\;\;\;\;
&F_H(t) = \frac{\sin(H \log y)}{\pi \log y}.
\end{aligned}\]
We studied $\eta_*$ and $\eta_+$ in Part \ref{part:maj}. We saw
 $\eta_*$ in Thm.~\ref{thm:ostop} (which actually works for
general $\varphi:\lbrack 0,\infty)\to \lbrack 0,\infty)$, as its statement says). We will set $\kappa$ soon.

We fix a value for $r$, namely, $r = 150000$. Our results will
have to be valid for any $x\geq x_+$, where $x_+$ is fixed. We set
$x_+ = 4.9\cdot 10^{26}$, since we want a result valid for $N\geq 10^{27}$,
and, as was discussed in (\ref{subs:charme}), we will work with $x_+$
slightly smaller than $N/2$.

\section{The $\ell_2$ norm over the major arcs: explicit version}

We apply Lemma \ref{lem:drujal} with $\eta=\eta_+$ and $\eta_\circ$
as in (\ref{eq:cleo}). Let us first work out the
error terms defined in (\ref{eq:sreda}).
 Recall that $\delta_0=8$. By Thm.~\ref{thm:malpor},
\begin{equation}\label{eq:zakone1}
\begin{aligned}ET_{\eta_+,\delta_0 r/2} &=
\max_{|\delta|\leq \delta_0 r/2} |\err_{\eta,\chi_T}(\delta,x)|\\
&= 4.772\cdot 10^{-11} + \frac{251400}{\sqrt{x_+}}  
\leq 1.1405\cdot 10^{-8},
\end{aligned}\end{equation}
\begin{equation}\label{eq:zakone2}
\begin{aligned} &E_{\eta_+,r,\delta_0} =
\mathop{\mathop{\max_{\chi \mo q}}_{q\leq r\cdot \gcd(q,2)}}_{|\delta|\leq
  \gcd(q,2) \delta_0 r/2 q}
\sqrt{q^*} |\err_{\eta_+,\chi^*}(\delta,x)|\\ &\leq
1.3482 \cdot 10^{-14} \sqrt{300000} + \frac{1.617\cdot 10^{-10}}{\sqrt{2}} +
\frac{1}{\sqrt{x_+}} \left(499900 + 52 \sqrt{300000}\right)\\ &\leq
2.3992 \cdot 10^{-8},
\end{aligned}\end{equation}
where, in the latter case, we are using the fact that a stronger bound
for $q=1$ (namely, (\ref{eq:zakone1})) allows us to assume $q\geq 2$.

 We also need to bound
a few norms: by the estimates in \S \ref{subs:daysold}
and \S \ref{subs:byron}
(applied with $H=200$),
\begin{equation}\label{eq:sazar}\begin{aligned}
\left|\eta_+\right|_1 &\leq 1.062319,\;\;\;\;\;\;\;
\left|\eta_+\right|_2 \leq 0.800129 + \frac{274.8569}{200^{7/2}} \leq 
0.800132\\
\left|\eta_+\right|_\infty &\leq 
1+ 2.06440727 \cdot \frac{1 + \frac{4}{\pi} \log H}{H} \leq
1.079955.
\end{aligned}\end{equation}
By (\ref{eq:glenkin}) and (\ref{eq:zakone1}),
\[\begin{aligned}
|S_{\eta_+}(0,x)| &= \left|\widehat{\eta_+}(0)\cdot x + O^*\left(\err_{\eta_+,\chi_T}(0,x)\right)\cdot x\right|\\ &\leq (|\eta_+|_1 + ET_{\eta_+,\delta_0 r/2}) x
\leq 1.063 x.
\end{aligned}\]
This is far from optimal, but it will do, since all we wish to do with
this is to bound the tiny error term $K_{r,2}$ in (\ref{eq:sreda}): 
\[\begin{aligned}
K_{r,2} &= (1+ \sqrt{300000}) (\log x)^2 \cdot 1.079955\\ &\cdot 
(2\cdot 1.06232 + (1+ \sqrt{300000}) (\log x)^2 1.079955/x)\\
&\leq 1259.06 (\log x)^2 \leq 
9.71\cdot 10^{-21} x
\end{aligned}\]
for $x\geq x_+$. By (\ref{eq:zakone1}), we also have
\[5.19\delta_0 r \left(ET_{\eta_+,\frac{\delta_0 r}{2}}\cdot \left(|\eta_+|_1 + 
\frac{ET_{\eta_+,\frac{\delta_0 r}{2}}}{2}\right)\right)\leq 0.075272
\]
and
\[\delta_0 r (\log 2 e^2 r) 
\left(E_{\eta_+,r,\delta_0}^2
+ K_{r,2}/x\right)  \leq 1.00393\cdot 10^{-8}.\]

By (\ref{eq:impath}) and (\ref{eq:lamia}),
\begin{equation}\label{eq:lopez}
0.8001287\leq |\eta_\circ|_2 \leq 0.8001288\end{equation} and
\begin{equation}\label{eq:sanchez}
|\eta_+- \eta_\circ|_2 \leq \frac{274.856893}{H^{7/2}} \leq 2.42942 \cdot
10^{-6}.
\end{equation}

We bound $|\eta_\circ^{(3)}|_1$ using the fact that (as we can tell
by taking derivatives) $\eta_\circ^{(2)}(t)$ increases from $0$ at $t=0$
to a maximum within $\lbrack 0,1/2\rbrack$, and then decreases to 
$\eta_\circ^{(2)}(1) = -7$, only to increase to a maximum within
$\lbrack 3/2,2\rbrack$ (equal to the maximum attained
 within $\lbrack 0,1/2\rbrack$)
and then decrease to $0$ at $t=2$:
\begin{equation}\label{eq:halr}\begin{aligned}
|\eta_\circ^{(3)}|_1 &= 2 \max_{t\in \lbrack 0,1/2\rbrack}
\eta_\circ^{(2)}(t) - 2 \eta_\circ^{(2)}(1) + 2 \max_{t\in \lbrack 3/2,2\rbrack}
\eta_\circ^{(2)}(t)\\
&= 4 \max_{t\in \lbrack 0,1/2\rbrack}
\eta_{\circ}^{(2)}(t) + 14 \leq 4\cdot 4.6255653 + 14 \leq 32.5023,
\end{aligned}\end{equation}
where we compute the maximum by the bisection method with $30$ iterations
(using interval arithmetic, as always).

We evaluate explicitly
\[\mathop{\sum_{q\leq r}}_{\text{$q$ odd}} \frac{\mu^2(q)}{\phi(q)} =
 6.798779\dotsc,
\]
using, yet again, interval arithmetic. 

Looking at (\ref{eq:chetvyorg}) and 
(\ref{eq:mardi}), 
we conclude that
\[\begin{aligned}
L_{r,\delta_0} &\leq 2\cdot 6.798779\cdot 0.800132^2 \leq 8.70531
,\\
L_{r,\delta_0} &\geq 2\cdot 6.798779\cdot 0.8001287^2 
- ((\log r + 1.7) \cdot (3.888 \cdot 10^{-6}+5.91\cdot 10^{-12}))\\
&- \left(1.342\cdot 10^{-5}\right)\cdot 
\left(0.64787 + \frac{\log r}{4 r} + \frac{0.425}{r}\right) \geq 8.70517
.\end{aligned}\]

Lemma \ref{lem:drujal} thus gives us that
\begin{equation}\label{eq:celine}\begin{aligned}
\int_{\mathfrak{M}_{8,r_0}} \left|S_{\eta_+}(\alpha,x)\right|^2 d\alpha
&= (8.70524+O^*(0.00007)) x + O^*(0.075273) x\\ &= (8.7052+O^*(0.0754)) x
\leq 8.7806 x.
\end{aligned}\end{equation}

\section{The total major-arc contribution}

First of all, we must bound from below
\begin{equation}\label{eq:ausbeuter}
C_0 = \prod_{p|N} \left(1 - \frac{1}{(p-1)^2}\right) \cdot
\prod_{p\nmid N} \left(1 + \frac{1}{(p-1)^3}\right).\end{equation}
The only prime that we know does not divide $N$ is $2$. Thus, we use the
bound
\begin{equation}\label{eq:arnar}
C_0 \geq 2 \prod_{p>2} \left(1 - \frac{1}{(p-1)^2}\right)
\geq 1.3203236 .\end{equation}

The other main constant is $C_{\eta_\circ,\eta_*}$, 
which we defined in (\ref{eq:vulgo})
and already started to estimate in (\ref{eq:jaram}):
\begin{equation}\label{eq:karlmarx}
C_{\eta_\circ,\eta_*} =
|\eta_\circ|_2^2 \int_0^{\frac{N}{x}} \eta_*(\rho) d\rho + 
2.71 |\eta_\circ'|_2^2 \cdot O^*\left(
\int_0^{\frac{N}{x}} ((2-N/x)+\rho)^2 \eta_*(\rho) d\rho
\right)\end{equation}
provided that $N\geq 2 x$. 
Recall that $\eta_* = (\eta_2 \ast_M \varphi)(\varkappa t)$, where
$\varphi(t) = t^2 e^{-t^2/2}$. Therefore,
\[\begin{aligned}
\int_0^{N/x} \eta_*(\rho) d\rho &= 
\int_0^{N/x} (\eta_2\ast \varphi)(\varkappa \rho) d\rho = 
\int_{1/4}^1 \eta_2(w) \int_0^{N/x} \varphi\left(\frac{\varkappa
    \rho}{w}\right) d\rho \frac{dw}{w}\\
&= \frac{|\eta_2|_1 |\varphi|_1}{\varkappa} - \frac{1}{\varkappa}
\int_{1/4}^1 \eta_2(w) \int_{\varkappa N/x w}^{\infty} \varphi(\rho) d\rho dw.
\end{aligned}\]
By integration by parts and \cite[(7.1.13)]{MR0167642},
\[\int_y^\infty \varphi(\rho) d\rho = y e^{-y^2/2} + 
\sqrt{2} \int_{y/\sqrt{2}}^\infty e^{-t^2} dt < \left(y +
  \frac{1}{y}\right)
e^{-y^2/2}.\]
Hence
\[\int_{\varkappa N/x w}^{\infty} \varphi(\rho) d\rho
\leq \int_{2 \varkappa}^{\infty} \varphi(\rho) d\rho < 
\left(2\varkappa +
  \frac{1}{2 \varkappa}\right)
e^{-2 \varkappa^2}\]
and so, since $|\eta_2|_1=1$,
\begin{equation}\label{eq:sasa}
\begin{aligned}
\int_0^{N/x} \eta_*(\rho) d\rho &\geq
\frac{|\varphi|_1}{\varkappa} - \int_{1/4}^1 \eta_2(w) dw \cdot \left(2 +
  \frac{1}{2 \varkappa^2}\right) e^{-2 \varkappa^2}\\ &\geq
\frac{|\varphi|_1}{\varkappa} - \left(2 +
  \frac{1}{2 \varkappa^2}\right) e^{-2 \varkappa^2}.\end{aligned}\end{equation}

Let us now focus on the second integral in (\ref{eq:karlmarx}).
Write $N/x = 2 + c_1/\varkappa$. Then the integral equals
\[\begin{aligned}
\int_0^{2+c_1/\varkappa} &(- c_1/\varkappa+ \rho)^2 \eta_*(\rho) d\rho \leq
\frac{1}{\varkappa^3} \int_0^\infty (u-c_1)^2\; (\eta_2\ast_M \varphi)(u)\;
du\\ &=
\frac{1}{\varkappa^3} \int_{1/4}^1 \eta_2(w)
\int_0^\infty (v w - c_1)^2 \varphi(v) dv dw\\
&= \frac{1}{\varkappa^3} \int_{1/4}^1 \eta_2(w)
\left(3 \sqrt{\frac{\pi}{2}} w^2 - 2\cdot 2 c_1 w 
+ c_1^2 \sqrt{\frac{\pi}{2}}\right) dw \\
&= \frac{1}{\varkappa^3}
\left(\frac{49}{48}\sqrt{\frac{\pi}{2}} - \frac{9}{4} c_1 +
\sqrt{\frac{\pi}{2}} c_1^2\right).
\end{aligned}\]
It is thus best to choose $c_1 = (9/4)/\sqrt{2 \pi} = 0.89762\dotsc$.

We must now estimate $|\eta_\circ'|_2^2$. We could do this directly by
rigorous numerical integration, but we might as well do it the hard way
(which is actually rather easy).
By the definition (\ref{eq:cleo}) of $\eta_\circ$,
\begin{equation}\label{eq:ormal}
|\eta_\circ'(x+1)|^2 =  
\left(x^{14} - 18 x^{12} + 111 x^{10} - 284 x^8 +351 x^6 - 210 x^4 + 49 x^2\right)
e^{-x^2}\end{equation}
for $x\in \lbrack -1,1\rbrack$, and $\eta_\circ'(x+1) = 0$ for
$x\not\in \lbrack -1,1\rbrack$. Now, for any even integer $k>0$,
\[\int_{-1}^1 x^k e^{-x^2} dx = 2\int_0^1 x^k e^{-x^2} dx
= \gamma\left(\frac{k+1}{2},1\right),\]
where $\gamma(a,r) = \int_0^r e^{-t} t^{a-1} dt$
is the incomplete gamma function. (We substitute $t=x^2$ in the integral.)
By \cite[(6.5.16), (6.5.22)]{MR0167642},
$\gamma(a+1,1) = a \gamma(a,1) - 1/e$ for all $a>0$, and 
$\gamma(1/2,1) = \sqrt{\pi} \erf(1)$, where 
\[\erf(z) = \frac{2}{\sqrt{\pi}} \int_0^1 e^{-t^2} dt.\]
Thus, starting from (\ref{eq:ormal}), we see that
\begin{equation}\label{eq:melancho}\begin{aligned}
|\eta_\circ'|_2^2 &= 
\gamma\left(\frac{15}{2},1\right) - 18 \cdot
\gamma\left(\frac{13}{2},1\right) + 111\cdot
\gamma\left(\frac{11}{2},1\right) \\ &- 284\cdot \gamma\left(\frac{9}{2},1\right) 
+ 351 \cdot \gamma\left(\frac{7}{2},1\right) - 210\cdot \gamma\left(\frac{5}{2},1\right)
+ 49\cdot \gamma\left(\frac{3}{2},1\right)\\
&= \frac{9151}{128} \sqrt{\pi} \erf(1) - \frac{18101}{64 e}
= 2.7375292\dotsc .
\end{aligned}
 \end{equation}

We thus obtain
\[\begin{aligned}
2.71 |\eta_\circ'|_2^2 \cdot 
&\int_0^{\frac{N}{x}} ((2-N/x)+\rho)^2 \eta_*(\rho) d\rho \\ &\leq
7.4188 \cdot \frac{1}{\varkappa^3}
\left(\frac{49}{48} \sqrt{\frac{\pi}{2}} - 
\frac{(9/4)^2}{2 \sqrt{2 \pi}}\right)
\leq \frac{2.0002}{\varkappa^3}.\end{aligned}\]
We conclude that
\[C_{\eta_\circ,\eta_*} \geq \frac{1}{\varkappa} |\varphi|_1 |\eta_\circ|_2^2 
- |\eta_\circ|_2^2  \left(2 + \frac{1}{2 \varkappa^2}\right)
e^{-2 \varkappa^2} - \frac{2.0002}{\varkappa^3}.\]
Setting \[\varkappa = 49\] and using (\ref{eq:lopez}), we obtain
\begin{equation}\label{eq:barbar}
C_{\eta_\circ,\eta_*} \geq 
\frac{1}{\varkappa} (|\varphi|_1 |\eta_\circ|_2^2
- 0.000834).\end{equation}
Here it is useful to note that $|\varphi|_1 =  \sqrt{\frac{\pi}{2}}$, and so,
by (\ref{eq:lopez}), $|\varphi|_1 |\eta_\circ|_2^2 = 0.80237\dotsc$.

We have finally chosen $x$ in terms of $N$:
\begin{equation}\label{eq:warwar}
x = \frac{N}{2 + \frac{c_1}{\varkappa}}  = \frac{N}{2 + 
\frac{9/4}{\sqrt{2 \pi}} \frac{1}{49}}  = 0.495461\dotsc \cdot N.\end{equation}
Thus, we see that, since we are assuming $N\geq 10^{27}$,
we in fact have $x\geq 4.95461\dotsc \cdot 10^{26}$, and so, in
particular, 
\begin{equation}\label{eq:kalm}
x\geq 4.9\cdot 10^{26},\;\;\;\; \frac{x}{\varkappa} \geq 10^{25}.\end{equation}

Let us continue with our determination of the major-arcs total. 
We should compute the quantities in (\ref{eq:vulgato}). We already
have bounds for $E_{\eta_+,r,\delta_0}$, $A_{\eta_+}$ (see (\ref{eq:celine})),
$L_{\eta,r,\delta_0}$ and $K_{r,2}$. 
By Corollary \ref{cor:coprar},
we have
\begin{equation}\label{eq:fabienne}\begin{aligned}
E_{\eta_*,r,8} &\leq 
\mathop{\mathop{\max_{\chi \mo q}}_{q\leq r\cdot \gcd(q,2)}}_{|\delta|\leq
  \gcd(q,2) \delta_0 r/2 q}
\sqrt{q^*} |\err_{\eta_*,\chi^*}(\delta,x)|\\
&\leq \frac{1}{\varkappa} \left(
2.485\cdot 10^{-19} + \frac{1}{\sqrt{10^{25}}}\left(381500+76
\sqrt{300000}\right)\right)\\
&\leq \frac{1.33805 \cdot 10^{-8}}{\varkappa},
\end{aligned}\end{equation}
where the factor of $\varkappa$ comes from the scaling in 
$\eta_*(t) = (\eta_2 \ast_M \varphi)(\varkappa t)$ (which in effect
divides $x$ by $\varkappa$).
It remains only to bound the more harmless terms of type
$Z_{\eta,2}$ and $LS_\eta$.

Clearly, $Z_{\eta_+^2,2}\leq (1/x) \sum_n \Lambda(n) (\log n) \eta_+^2(n/x)$.
Now, by Prop.~\ref{prop:malheur},
\begin{equation}\label{eq:malavita}
\begin{aligned}
\sum_{n=1}^\infty &\Lambda(n) (\log n) \eta^2(n/x)\\ &= 
\left(0.640206 + O^*\left(2\cdot 10^{-6} + 
\frac{366.91}{\sqrt{x}}\right)\right) x \log x - 0.021095 x\\
&\leq (0.640206 + O^*(3\cdot 10^{-6})) x \log x - 0.021095 x.
\end{aligned}\end{equation}
Thus,
\begin{equation}\label{eq:aloet}
Z_{\eta_+^2,2}\leq 0.640209 \log x.\end{equation}
We will proceed a little more crudely for $Z_{\eta_*^2,2}$:
\begin{equation}\label{eq:bavette}\begin{aligned}
Z_{\eta^2_*,2} &= \frac{1}{x} \sum_n \Lambda^2(n) \eta_*^2(n/x) \leq
\frac{1}{x} \sum_n \Lambda(n) \eta_*(n/x) \cdot (\eta_*(n/x) \log n)\\
&\leq
(|\eta_*|_1 + |\err_{\eta_*,\chi_T}(0,x)|)
\cdot (|\eta_*(t) \cdot \log^+(\varkappa t)|_\infty + |\eta_*|_\infty \log
(x/\varkappa)),\end{aligned}\end{equation}
where $\log^+(t) := \max(0,\log t)$. 
It is easy to see that
\begin{equation}\label{eq:macadam}
|\eta_*|_\infty =  |\eta_2\ast_M \varphi|_\infty \leq 
\left|\frac{\eta_2(t)}{t}\right|_1 |\varphi|_\infty
\leq 4 (\log 2)^2 \cdot \frac{2}{e} \leq 1.414,
\end{equation}
and, since $\log^+$ is non-decreasing and $\eta_2$ is supported on
a subset of $\lbrack 0,1\rbrack$,
\[\begin{aligned}
|\eta_*(t) \cdot \log^+(\varkappa t)|_\infty &= |(\eta_2\ast_M \varphi) \cdot
\log^+|_\infty \leq |\eta_2 \ast_M (\varphi\cdot \log^+)|_\infty\\
&\leq \left|\frac{\eta_2(t)}{t}\right|_1 
 \cdot |\varphi \cdot \log^+|_\infty \leq 1.921813\cdot 0.381157 \leq
0.732513\end{aligned}\]
where we bound $|\varphi\cdot \log^+|_\infty$ by the bisection method with
$25$ iterations. We already know that 
\begin{equation}\label{eq:marldoro}
|\eta_*|_1 = \frac{|\eta_2|_1 |\varphi|_1}{\varkappa} = 
\frac{|\varphi|_1}{\varkappa} = \frac{\sqrt{\pi/2}}{\varkappa}.\end{equation}
By Cor.~\ref{cor:coprar},
\[|\err_{\eta_*,\chi_T}(0,x)| \leq
2.485 \cdot 10^{-19} + \frac{1}{\sqrt{10^{25}}} (381500+76)
\leq 1.20665 \cdot 10^{-7}.
\]
We conclude that 
\begin{equation}\label{eq:julie}
Z_{\eta_*^2,2} \leq (\sqrt{\pi/2}/49 + 1.20665\cdot 10^{-7}) (0.732513 +
1.414 \log(x/49))\leq 0.0362 \log x.\end{equation}

We have bounds for $|\eta_*|_\infty$ and $|\eta_+|_\infty$. We can also bound
\[|\eta_* \cdot t|_\infty = \frac{|(\eta_2 \ast_M \varphi)\cdot t|_\infty}{\kappa}
\leq  \frac{|\eta_2|_1 \cdot |\varphi \cdot t|_\infty}{\kappa}
 \leq \frac{3^{3/2} e^{-3/2}}{\kappa}.\]
We quote the estimate
\begin{equation}\label{eq:muthit}
|\eta_+\cdot t|_\infty = 1.064735+3.25312\cdot (1+(4/\pi) \log 200)/200 
\leq 1.19073\end{equation} from (\ref{eq:shchedrin}).

We can now bound $LS_\eta(x,r)$ for $\eta = \eta_*,\eta_+$:
\[\begin{aligned}
LS_{\eta}(x,r) &= \log r \cdot \max_{p\leq r} \sum_{\alpha\geq 1} \eta\left(
\frac{p^\alpha}{x}\right) \\ &\leq
(\log r) \cdot \max_{p\leq r} \left( 
\frac{\log x}{\log p} |\eta|_\infty + \mathop{\sum_{\alpha\geq 1}}_{p^\alpha\geq x} \frac{|\eta \cdot t|_\infty}{p^\alpha/x}\right)\\
&\leq (\log r) \cdot \max_{p\leq r} \left(\frac{\log x}{\log p} |\eta|_\infty +
\frac{|\eta\cdot t|_\infty}{1-1/p}\right) \\ &\leq
\frac{(\log r) (\log x)}{\log 2} |\eta|_\infty + 2 (\log r)
|\eta\cdot t|_\infty,
\end{aligned}\]
and so
\begin{equation}\label{eq:alisa}\begin{aligned}
LS_{\eta_*}&\leq 
\left(\frac{1.414}{\log 2} \log x + 2\cdot \frac{(3/e)^{3/2}}{49}\right) \log r 
\leq 24.32 \log x + 0.57,\\
\\
LS_{\eta_+}&\leq \left(\frac{1.07996}{\log 2} \log x + 2\cdot 1.19073
\right) \log r\leq 18.57\log x + 28.39,\end{aligned}\end{equation}
where we are using the bound on $|\eta_+|_\infty$ in (\ref{eq:sazar})

We can now start to put together all terms in (\ref{eq:opus111}). 
Let $\epsilon_0 = |\eta_+-\eta_\circ|_2/|\eta_\circ|_2$. Then,
by (\ref{eq:sanchez}),
\[\epsilon_0 |\eta_\circ|_2 =
|\eta_+ - \eta_\circ|_2 \leq 2.42942\cdot 10^{-6}. 
\]
Thus,
\[2.82643 |\eta_\circ|_2^2 (2+\epsilon_0)\cdot \epsilon_0
+ \frac{4.31004 |\eta_\circ|_2^2 +
0.0012 \frac{|\eta_\circ^{(3)}|_1^2}{\delta_0^5}}{r} \]
is at most
\[\begin{aligned}
&2.82643\cdot 2.42942\cdot 10^{-6} \cdot (2\cdot 0.80013 + 
2.42942\cdot 10^{-6}) \\ &+ 
\frac{4.3101\cdot 0.80013^2 + 0.0012 \cdot \frac{32.503^2}{8^5}}{150000}
\leq 2.9387\cdot 10^{-5}\end{aligned}\]
by (\ref{eq:lopez}), (\ref{eq:halr}), and (\ref{eq:marldoro}).

Since $\eta_* = (\eta_2 \ast_M \varphi)(\varkappa x)$ 
and $\eta_2$ is supported on $\lbrack 1/4,1\rbrack$,
\[\begin{aligned}
|\eta_*|_2^2 &= \frac{|\eta_2 \ast_M \varphi|_2^2}{\varkappa}
= \frac{1}{\varkappa} \int_0^\infty \left(\int_0^\infty \eta_2(t)
\varphi\left(\frac{w}{t}\right) \frac{dt}{t}\right)^2 dw\\
&\leq \frac{1}{\varkappa}
\int_0^\infty \left(1 - \frac{1}{4}\right)
 \int_0^\infty \eta_2^2(t)
\varphi^2\left(\frac{w}{t}\right) \frac{dt}{t^2} dw \\ &=  \frac{3}{4\varkappa}
\int_0^\infty \frac{\eta_2^2(t)}{t} \left(
\int_0^\infty 
\varphi^2\left(\frac{w}{t}\right) \frac{dw}{t}\right)
dt\\ &= \frac{3}{4 \varkappa}
|\eta_2(t)/\sqrt{t}|_2^2 \cdot |\varphi|_2^2
= \frac{3}{4 \varkappa}\cdot 
\frac{32}{3} (\log 2)^3 \cdot \frac{3}{8} \sqrt{\pi} 
\leq \frac{1.77082}{\varkappa},
\end{aligned}\]
where we go from the first to the second line by Cauchy-Schwarz.

Recalling the bounds on $E_{\eta_*,r,\delta_0}$ and $E_{\eta_+,r,\delta_0}$
we obtained in (\ref{eq:zakone2}) and (\ref{eq:fabienne}), we conclude that
the second line of (\ref{eq:opus111}) is at most $x^2$ times
\[\begin{aligned}
\frac{1.33805\cdot 10^{-8}}{\varkappa} &\cdot 8.7806 
 + 2.3922\cdot 10^{-8} \cdot 1.6812 \\&\cdot (\sqrt{8.7806} + 1.6812
\cdot 0.80014) \sqrt{\frac{1.77082}{\varkappa}}
&\leq \frac{1.7316 \cdot 10^{-6}}{\varkappa},\end{aligned}\]
where we are using the bound $A_{\eta_+}\leq 8.7806$ we obtained in
(\ref{eq:celine}). (We are also using the bounds on norms in (\ref{eq:sazar})
and the value $\varkappa = 49$.)

By the bounds (\ref{eq:aloet}),
(\ref{eq:julie}) and (\ref{eq:alisa}), we see that the
 third line of (\ref{eq:opus111}) is at most
\[\begin{aligned}
& 2\cdot (0.640209 \log x)
\cdot (24.32\log x + 0.57)\cdot x\\
&+ 4 \sqrt{0.640209 \log x \cdot 0.0362 \log x} (18.57 \log x + 28.39)
x \leq 43 (\log x)^2 x,
\end{aligned}\]
where we use the assumption $x\geq x_+= 4.9\cdot 10^{26}$ (though a much weaker
assumption would suffice). 

Using the assumption $x\geq x_+$ again, together with (\ref{eq:marldoro})
and the bounds we have just proven, 
we conclude that, for $r=150000$, the integral over the major arcs
\[\int_{\mathfrak{M}_{8,r}} S_{\eta_+}(\alpha,x)^2 S_{\eta_*}(\alpha,x) e(-N \alpha) d\alpha\]
is 
\begin{equation}\label{eq:margot}\begin{aligned}
&C_0 \cdot C_{\eta_0,\eta_*} x^2
+ O^*\left(2.9387\cdot 10^{-5} \cdot \frac{\sqrt{\pi/2}}{\varkappa} x^2 + 
\frac{1.7316\cdot 10^{-6}}{\varkappa} x^2 + 43 (\log x)^2 x\right)\\
&= 
C_0 \cdot C_{\eta_0,\eta_*} x^2 + O^*\left(\frac{3.85628\cdot 10^{-5}\cdot x^2}{\varkappa}\right)\\
&=
C_0 \cdot C_{\eta_0,\eta_*} x^2 + O^*(7.86996\cdot 10^{-7} x^2),
\end{aligned}\end{equation}
where $C_0$ and $C_{\eta_0,\eta_*}$ are as in (\ref{eq:vulgo}).
Notice that $C_0 C_{\eta_0,\eta_*} x^2$ is the expected asymptotic
for the integral over all of $\mathbb{R}/\mathbb{Z}$.

Moreover, by (\ref{eq:arnar}), (\ref{eq:barbar}) and (\ref{eq:lopez}),
as well as $|\varphi|_1 = \sqrt{\pi/2}$,
\[\begin{aligned}
C_0 \cdot C_{\eta_0,\eta_*} &\geq 1.3203236 \left(\frac{|\varphi|_1 |\eta_\circ|_2^2}{\varkappa} - \frac{0.000834}{\varkappa}\right)\\ &\geq
\frac{1.0594003}{\varkappa} - \frac{0.001102}{\varkappa} 
\geq \frac{1.058298}{49}.\end{aligned}\]
Hence
\begin{equation}\label{eq:juventud}
\int_{\mathfrak{M}_{8,r}} S_{\eta_+}(\alpha,x)^2 S_{\eta_*}(\alpha,x) e(-N \alpha) 
d\alpha
\geq \frac{1.058259}{\varkappa} x^2,\end{equation}
where, as usual, $\varkappa=49$. This is our total major-arc bound.

\section{The minor-arc total: explicit version}

We need to estimate the quantities $E$, $S$, $T$, $J$, $M$
in Theorem \ref{thm:ostop}. Let us start by bounding the constants
in (\ref{eq:malus}). The constants $C_{\eta_+,j}$, $j=0,1,2$, will appear
only in the minor term $E$, and so crude bounds on them will do.

By (\ref{eq:sazar}) and (\ref{eq:muthit}),
\[\sup_{r\geq t} \eta_+(r) \leq
\min\left(1.07996,\frac{1.19073}{t}\right)\]
for all $t\geq 0$.
Thus,
\[\begin{aligned}
C_{\eta_+,0} &= 0.7131 \int_0^\infty \frac{1}{\sqrt{t}} 
\left(\sup_{r\geq t} \eta_+(r)\right)^2 dt\\ &\leq
0.7131 \left(\int_0^1  \frac{1.07996^2}{\sqrt{t}} dt  + \int_1^\infty
\frac{1.19073^2}{t^{5/2}} dt\right)
\leq 2.33744.\end{aligned}\]
Similarly,
\[\begin{aligned}
C_{\eta_+,1} &= 0.7131 \int_1^\infty \frac{\log t}{\sqrt{t}} 
\left(\sup_{r\geq t} \eta_+(r)\right)^2 dt\\ &\leq
0.7131 \int_1^\infty \frac{1.19073^2 \log t}{t^{5/2}} dt \leq
0.44937.\end{aligned}\]
Immediately from (\ref{eq:sazar}),
\[C_{\eta_+,2} = 0.51942 |\eta_+|_\infty^2 \leq 0.60581.\]

We get
\begin{equation}\label{eq:dubistdie}\begin{aligned}
E &\leq \left((2.33744+0.60581) \log x + (2\cdot 2.33744 + 0.44937)\right) 
\cdot x^{1/2}\\
&\leq (2.94325 \log x + 5.12426) \cdot x^{1/2}
\leq 8.4029\cdot 10^{-12} \cdot x
,\end{aligned}\end{equation}
where $E$ is defined as in (\ref{eq:georgic}), and where
we are using the assumption $x\geq x_+ = 4.9\cdot 10^{26}$.
Using (\ref{eq:fabienne}) and (\ref{eq:marldoro}), we see that
\[S_{\eta_*}(0,x) = (|\eta_*|_1 + O^*(ET_{\eta_*,0})) x =
\left(\sqrt{\pi/2} + O^*(1.33805\cdot 10^{-8})\right) \frac{x}{\varkappa}
.\]
Hence
\begin{equation}\label{eq:hexelor}
S_{\eta_*}(0,x) \cdot E  \leq 1.05315 \cdot 10^{-11}\cdot
\frac{x^2}{\varkappa}
.\end{equation}

We can bound
\begin{equation}\label{eq:felipa}
S \leq \sum_n \Lambda(n) (\log n) \eta_+^2(n/x) \leq
0.640209 x \log x - 0.021095 x\end{equation}
by (\ref{eq:malavita}). Let us now estimate $T$.
Recall that $\varphi(t) = t^2 e^{-t^2/2}$. Since
\[\int_0^u \varphi(t) dt = \int_0^u t^2 e^{-t^2/2} dt \leq \int_0^u t^2 dt =
\frac{u^3}{3},\]
we can bound
\[C_{\varphi,3}\left(\frac{1}{2} \log \frac{x}{\varkappa}\right) 
= \frac{1.04488}{\sqrt{\pi/2}} \int_0^{\frac{2}{\log x/\varkappa}}
t^2 e^{-t^2/2} dt \leq \frac{0.2779}{((\log x/\varkappa)/2)^3}.\]

By (\ref{eq:celine}), we already know that
$J = (8.7052 + O^*(0.0754)) x$. Hence
\begin{equation}\label{eq:je}\begin{aligned}
(\sqrt{J}-\sqrt{E})^2 &=
(\sqrt{(8.7052 + O^*(0.0754)) x} - \sqrt{8.4029\cdot 10^{-12} \cdot x})^2\\
&\geq 8.6297 x , 
\end{aligned}\end{equation}
and so
\[\begin{aligned}
T &= C_{\varphi,3}\left(\frac{1}{2} \log \frac{x}{\varkappa}\right) 
\cdot (S - (\sqrt{J}-\sqrt{E})^2)\\
&\leq \frac{8\cdot 0.2779}{(\log x/\varkappa)^3} \cdot 
(0.640209 x \log x - 0.021095 x
- 8.6297 x) \\ &\leq 0.17792 \frac{8 x \log x}{(\log x/\varkappa)^3} - 2.40405 \frac{8 x}{(\log x/\varkappa)^3}
\\ &\leq 1.42336 \frac{x}{(\log x/\varkappa)^2} - 13.69293 \frac{x}{(\log x/\varkappa)^3}.
\end{aligned}\]
for $\varkappa=49$. Since $x/\varkappa \geq 10^{25}$, this implies that
\begin{equation}\label{eq:lamber}
T\leq 3.5776\cdot 10^{-4} \cdot x.
\end{equation}

It remains to estimate $M$. Let us first look at $g(r_0)$; here
$g=g_{x/\varkappa,\varphi}$, where $g_{y,\varphi}$ is defined as in 
(\ref{eq:basia}) and $\phi(t) = t^2 e^{-t^2/2}$, as usual. Write $y=x/\varkappa$.
 We must estimate the constant $C_{\varphi,2,K}$ defined
in (\ref{eq:midin}):
\[\begin{aligned}
C_{\varphi,2,K} &= - \int_{1/K}^1 \varphi(w) \log w\; dw \leq
- \int_{0}^1 \varphi(w) \log w\; dw \\ &\leq
-\int_0^1 w^2 e^{-w^2/2} \log w\; dw \leq 0.093426,\end{aligned}\]
where again we use VNODE-LP for rigorous numerical integration.
Since $|\varphi|_1= \sqrt{\pi/2}$ and $K = (\log y)/2$, this implies that
\begin{equation}\label{eq:dukas}
\frac{C_{\varphi,2,K}/|\varphi|_1}{\log K} \leq \frac{0.07455}{\log \frac{\log y}{2}}
\end{equation}
and so \begin{equation}\label{eq:calvino}
R_{y,K,\varphi,t} = \frac{0.07455}{\log \frac{\log y}{2}} R_{y/K,t} +
\left(1 - \frac{0.07455}{\log \frac{\log y}{2}}\right) R_{y,t}.\end{equation}

Let $t = 2 r_0 = 300000$; we recall that $K = (\log y)/2$. Recall from (\ref{eq:kalm}) that $y=x/\varkappa
\geq 10^{25}$; thus, $y/K\geq 3.47435\cdot 10^{23}$ and 
$\log((\log y)/2) \geq 3.35976$.
Going back to the definition of $R_{x,t}$ in (\ref{eq:veror}),
we see that
\begin{equation}\label{eq:mali1}
\begin{aligned}
R_{y,,2r_0} &\leq 0.27125 \log \left(1 + \frac{\log (8\cdot 150000)}{2\log \frac{9\cdot (10^{25})^{1/3}}{2.004\cdot 2\cdot 150000}}\right) + 0.41415
\leq 0.58341,
\end{aligned}\end{equation}
\begin{equation}\label{eq:mali2}\begin{aligned}
R_{y/K,2r_0} &\leq 0.27125 \log \left(1 + \frac{\log (8\cdot 150000)}{2\log \frac{9\cdot (3.47435\cdot 10^{23})^{1/3}}{2.004\cdot 2\cdot 150000}}\right) + 0.41415
\leq 0.60295,
\end{aligned}\end{equation}
and so
\[R_{y,K,\varphi,2 r_0} \leq \frac{0.07455}{3.35976} 0.60295 +
\left(1 - \frac{0.07455}{3.35976}\right) 0.58341 \leq 0.58385.\]

Using
\[\digamma(r) = e^{\gamma} \log \log r + \frac{2.50637}{\log \log r} \leq
5.42506,\]
we see from (\ref{eq:veror}) that
\[L_{2 r_0} = 5.42506\cdot \left(
\frac{13}{4} \log 300000 +7.82\right) + 13.66 \log 300000  + 37.55
\leq 474.608.\]
Going back to (\ref{eq:basia}), we sum up and obtain that
\[\begin{aligned} g(r_0) &= \frac{(0.58385\cdot \log 300000+0.5) \sqrt{5.42506} + 2.5}{
\sqrt{2\cdot 150000}} \\ &+ 
\frac{474.608}{150000} + 3.36 \left(\frac{\log y}{2 y}\right)^{1/6}\\ &\leq 0.041568.\end{aligned}\]
Using again the bound $x\geq 4.9\cdot 10^{26}$, we
obtain
\[\begin{aligned}&\frac{\log(150000+1)+c^+}{\log \sqrt{x} + c^-} \cdot S -
(\sqrt{J} - \sqrt{E})^2 \\ &\leq 
\frac{13.9716}{\frac{1}{2} \log x + 0.6394} \cdot (0.640209 x \log x - 
0.021095 x) 
- 8.6297 x \\ &\leq 
17.8895 x - \frac{11.7332 x}{\frac{1}{2} \log x + 0.6394} 
- 8.6297 x \\ &\leq (17.8895-8.6297) x \leq 9.2598 x,
\end{aligned}\]
where $c^+ = 2.0532$ and $c^- = 0.6394$.
Therefore,
\begin{equation}\label{eq:jbaldwin}\begin{aligned}
g(r_0)\cdot \left(\frac{\log(150000+1)+c^+}{\log \sqrt{x} + c^-} \cdot S -
(\sqrt{J} - \sqrt{E})^2\right) &\leq 0.041568\cdot 9.2598 x\\ &\leq	
0.38492 x.\end{aligned}\end{equation}
This is one of the main terms.

Let $r_1 = (3/8) y^{4/15}$, where, as usual, $y = x/\varkappa$ and
$\varkappa=49$. Then
\begin{equation}\label{eq:mali3}\begin{aligned}
R_{y,2r_1} &= 
0.27125 \log \left(1 + \frac{\log \left(8 \cdot
\frac{3}{8} y^{4/15}\right)}{2 \log \frac{9 y^{1/3}}{2.004\cdot 
\frac{3}{4} y^{4/15}}}\right) + 0.41415\\
&= 0.27125 \log \left(1 + \frac{\frac{4}{15} \log y + \log 3}{
2 \left(\frac{1}{3}-\frac{4}{15}\right) \log y + 2 \log \frac{9}{2.004\cdot
\frac{3}{4}}}\right)
+ 0.41415\\
&\leq 0.27125 \log\left(1+\frac{\frac{4}{15}}{2 \left(\frac{1}{3}-
\frac{4}{15}\right)}
\right) + 0.41415\leq 0.71215.
\end{aligned}\end{equation}
Similarly, for $K=(\log y)/2$ (as usual),
\begin{equation}\label{eq:kalda}\begin{aligned}
R_{y/K,2r_1} &= 
0.27125 \log \left(1 + \frac{\log \left(8 \cdot
\frac{3}{8} y^{4/15}\right)}{2 \log \frac{9 (y/K)^{1/3}}{2.004\cdot 
\frac{3}{4} y^{4/15}}}\right) + 0.41415\\
&= 0.27125 
\log \left(1 + \frac{\frac{4}{15} \log y + \log 3}{
\frac{2}{15} \log y - \frac{2}{3} \log \log y + 2 \log \frac{9 \cdot 2^{1/3}}{2.004\cdot \frac{3}{4}}}\right)
+0.41415\\ &=
 0.27125 
\log \left(3 + 
 \frac{\frac{4}{3} \log \log y - c}{
\frac{2}{15} \log y - \frac{2}{3} \log \log y + 2 \log 
\frac{12\cdot 2^{1/3}}{2.004}}
\right) + 0.41415,
\end{aligned}\end{equation}
where $c = 4 \log(12\cdot 2^{1/3}/2.004) - \log 3$.
Let
\[f(t) = \frac{\frac{4}{3} \log t - c
}{\frac{2}{15} t - \frac{2}{3} \log t + 2 \log \frac{12\cdot 2^{1/3}}{2.004}}.\]
The bisection method with $32$ iterations shows that
\begin{equation}\label{eq:golrof}
f(t) \leq 0.019562618\end{equation}
for $180\leq t\leq 30000$; since $f(t)<0$ for $0<t<180$
(by $(4/3) \log t - c <0$) and since, by $c>20/3$, we have 
$f(t)<(5/2) (\log t)/t$ as soon as $t> (\log t)^2$ (and so, in particular, for
$t>30000$), we see that (\ref{eq:golrof}) is valid for all $t>0$. Therefore,
\begin{equation}\label{eq:mali4}
R_{y/K,2 r_1} \leq 0.71392,\end{equation}
 and so, by (\ref{eq:calvino}), we conclude that
\[R_{y,K,\varphi,2 r_1} \leq \frac{0.07455}{3.35976} \cdot
 0.71392 +
\left(1 - \frac{0.07455}{3.35976}\right) \cdot 0.71215 \leq
0.71219.
\]

Since $r_1 = (3/8) y^{4/15}$ and $\digamma(r)$ is increasing for $r\geq 27$,
 we know that
\begin{equation}\label{eq:corple}\begin{aligned}
\digamma(r_1) &\leq \digamma(y^{4/15})
= e^\gamma \log \log y^{4/15} + \frac{2.50637}{\log \log y^{4/15}} \\
&= e^{\gamma} \log \log y + \frac{2.50637}{\log \log y - \log \frac{15}{4}} 
- e^{\gamma} \log \frac{15}{4}
\leq e^{\gamma} \log \log y - 1.43644
\end{aligned}\end{equation}
for $y\geq 10^{25}$. Hence,
(\ref{eq:veror}) gives us that
\[\begin{aligned}
L_{2 r_1} &\leq (e^{\gamma} \log \log y - 1.43644) \left(
\frac{13}{4} \log \frac{3}{4} y^{\frac{4}{15}} + 7.82\right) + 
13.66 \log \frac{3}{4} y^{\frac{4}{15}} + 37.55\\
&\leq
\frac{13}{15} e^\gamma \log y \log \log y + 
2.39776 \log y + 12.2628 \log \log y + 23.7304\\
&\leq  (2.13522 \log y + 18.118) \log \log y.\end{aligned}\]
Moreover, again by (\ref{eq:corple}),
\[\sqrt{\digamma(r_1)} \leq
\sqrt{e^{\gamma} \log \log y} - \frac{1.43644}{2 \sqrt{e^{\gamma} \log
\log y}}\] and so, by $y\geq 10^{25}$,
\[\begin{aligned}
&(0.71219 \log \frac{3}{4} y^{\frac{4}{15}}+0.5)
\sqrt{\digamma(r_1)}\\ &\leq
(0.18992 \log y + 0.29512) \left(
\sqrt{e^{\gamma} \log \log y} - \frac{1.43644}{2 \sqrt{e^{\gamma} \log
\log y}}\right)\\
&\leq 0.19505 \sqrt{e^{\gamma} \log \log y} - 
\frac{0.19505\cdot 1.43644\log y}{2 \sqrt{e^\gamma \log \log y}}\\
&\leq
0.26031 \log y \sqrt{\log \log y} - 
3.00147
.\end{aligned}\]
Therefore, by (\ref{eq:basia}),
\[\begin{aligned}g_{y,\varphi}(r_1) &\leq
\frac{0.26031 \log y
\sqrt{\log \log y} +2.5 - 3.00147}{\sqrt{\frac{3}{4} y^{\frac{4}{15}}}} \\ &+ 
\frac{(2.13522 \log y + 18.118) \log \log y}{\frac{3}{8} y^{\frac{4}{15}}}
+ \frac{3.36 ((\log y)/2)^{1/6}}{y^{1/6}}\\
&\leq
\frac{0.30059 \log y
\sqrt{\log \log y}}{y^{\frac{2}{15}}} + 
\frac{5.69392 \log y \log \log y}{y^{\frac{4}{15}}}\\
&- \frac{0.57904}{y^{\frac{2}{15}}}
 + \frac{48.3147 \log \log y}{y^{\frac{4}{15}}} + \frac{2.994 (\log y)^{1/6}}{y^{1/6}}\\
&\leq 
\frac{0.30059 \log y
\sqrt{\log \log y}}{y^{\frac{2}{15}}} + 
\frac{5.69392 \log y \log \log y}{y^{\frac{4}{15}}} + \frac{1.30151 (\log y)^{1/6}}{y^{1/6}}\\
&\leq \frac{0.30915 \log y \sqrt{\log \log y}}{y^{\frac{2}{15}}},
\end{aligned}\]
where
we use $y\geq 10^{25}$
and verify that the functions 
$t\mapsto 
 (\log t)^{1/6}/t^{1/6-2/15}$,
$t\mapsto \sqrt{\log \log t}/t^{4/15 - 2/15}$ and
$t\mapsto (\log \log t)/t^{4/15 - 2/15}$
are decreasing for $t\geq y$ (just by taking derivatives).

Since $\varkappa = 49$, one of the terms in (\ref{eq:gypo}) simplifies
easily:
\[\frac{7}{15} + \frac{-2.14938+ \frac{8}{15} \log \varkappa}{\log x + 2 c^-}
\leq \frac{7}{15}.\]
By (\ref{eq:felipa}) and $y = x/\varkappa = x/49$, we conclude that
\begin{equation}\label{eq:comrade}\begin{aligned}
\frac{7}{15} g(r_1) S &\leq
\frac{7}{15} \cdot \frac{0.30915 \log y \sqrt{\log \log y}}{y^{\frac{2}{15}}} 
 \cdot
(0.640209 \log x - 0.021095) x\\
&\leq \frac{0.14427 \log y \sqrt{\log \log y}}{y^{\frac{2}{15}}} 
(0.640209 \log y + 2.4705) x \leq 0.30517 x,
\end{aligned}\end{equation}
where we are using the fact that $y\mapsto (\log y)^2 \sqrt{\log \log y}/y^{2/15}$
is decreasing for $y\geq 10^{25}$ (because 
$y\mapsto (\log y)^{5/2}/y^{2/15}$ is decreasing
for $y\geq e^{75/4}$ and $10^{25}>e^{75/4}$).

It remains only to bound 
\[\frac{2 S}{\log x + 2 c^{-}} \int_{r_0}^{r_1} \frac{g(r)}{r} dr\]
in the expression (\ref{eq:gypo}) for $M$.
We will use the bound on the integral given in (\ref{eq:byrne}).
The easiest term to bound there is $f_1(r_0)$, defined in (\ref{eq:cymba}),
since it depends only on $r_0$: for $r_0 = 150000$,
\[f_1(r_0) = 0.0169073\dotsc.\]
It is also not hard to bound $f_2(r_0,x)$, also defined in (\ref{eq:cymba}):
\[\begin{aligned}f_2(r_0,y) &= 3.36
 \frac{((\log y)/2)^{1/6}}{x^{1/6}} \log \frac{
\frac{3}{8} y^{\frac{4}{15}}}{r_0}\\&\leq  
3.36 \frac{(\log y)^{1/6}}{(2 y)^{1/6}}
\left(\frac{4}{15} \log y + 0.05699 - \log r_0\right),\end{aligned}\]
where we recall again that $x = \varkappa y = 49 y$. Thus,
since $r_0 = 150000$ and $y\geq 10^{25}$,
\[f_2(r_0,y) \leq 0.001399.\]

Let us now look at the terms $I_{1,r}$, $c_\varphi$ in (\ref{eq:waslight}).
We already saw in (\ref{eq:dukas}) that
\[c_\varphi = \frac{C_{\varphi,2}/|\varphi|_1}{\log K} \leq \frac{0.07455}{\log 
\frac{\log y}{2}}\leq 0.02219.\]
Since $F(t) = e^{\gamma} \log t + c_\gamma$ with
$c_\gamma = 1.025742$,
\begin{equation}\label{eq:charlemagne}
I_{1,r_0} = F(\log r_0) + \frac{2 e^{\gamma}}{\log r_0} = 5.73826\dotsc
\end{equation}
It thus remains only to estimate $I_{0,r_0,r_1,z}$ for $z=y$ and $z=y/K$,
where $K = (\log y)/2$.

We will first give estimates for $y$ large.
Omitting negative terms from (\ref{eq:waslight}), we easily get 
 the following general bound, crude but useful enough:
\[
I_{0,r_0,r_1,z} \leq R_{z,2 r_0}^2 \cdot \frac{P_2(\log 2
  r_0)}{\sqrt{r_0}}
+ \frac{R_{z, 2r_1}^2 - 0.41415^2}{\log \frac{r_1}{r_0}}
\frac{P_2^-(\log 2 r_0)}{\sqrt{r_0}},\]
where $P_2(t) = t^2+4t+8$ and $P_2^-(t) = 2 t^2 + 16 t + 48$.
By (\ref{eq:mali3}) and (\ref{eq:mali4}), 
\[R_{y,2 r_1}\leq 0.71215,\;\;\;\;R_{y/K,2 r_1} \leq 0.71392\]
for $y\geq 10^{25}$. Assume now that $y\geq 10^{150}$. Then, since
$r_0 = 150000$,
\[\begin{aligned}
R_{y,r_0} &\leq 0.27125\log\left(1 + \frac{\log 4 r_0}{
2 \log \frac{9\cdot (10^{150})^{1/3}}{2.004 r_0}}\right) + 0.41415
\leq 0.43086,
\end{aligned}\]
and, similarly, $R_{y/K,r_0} \leq 0.43113$.  
Since
\[0.43086^2\cdot \frac{P_2(\log 2 r_0)}{\sqrt{r_0}} \leq
0.10426,
\;\;\;\;\;\;\;\;
0.43113^2\cdot \frac{P_2(\log 2 r_0)}{\sqrt{r_0}} \leq
0.10439,
\]
we obtain that
 \begin{equation}\label{eq:jijona}\begin{aligned}
(1-c_\varphi) &\sqrt{I_{0,r_0,r_1,y}} + c_\varphi \sqrt{I_{0,r_0,r_1,\frac{2 y}{\log y}}}\\
&\leq 0.97781\cdot \sqrt{0.10426 +
\frac{0.49214}{\frac{4}{15} \log y - \log 400000}}\\ &+ 0.02219
\sqrt{0.10439 + \frac{0.49584}{\frac{4}{15} \log y - \log 400000}}\leq
0.33239
\end{aligned}\end{equation}
for $y\geq 10^{150}$.

For $y$ between $10^{25}$ and $10^{150}$, we evaluate
the left side of (\ref{eq:jijona}) directly, using the definition
(\ref{eq:waslight}) of $I_{0,r_0,r_1,z}$ instead, as well as the bound
\[c_\varphi\leq \frac{0.07455}{\log \frac{\log y}{2}}\] from (\ref{eq:dukas}).
(It is clear from the second and third lines of (\ref{eq:basmed}) that 
$I_{0,r_0,r_1,z}$ is decreasing on $z$ for $r_0$, $r_1$ fixed, and so
the upper bound for $c_\varphi$ does give the worst case.)
The bisection method
(applied to the interval $\lbrack 25,150\rbrack$ with $30$ iterations,
including $30$ initial iterations)
gives us that
\begin{equation}\label{eq:hasmo}
(1-c_\varphi) \sqrt{I_{0,r_0,r_1,y}} + c_\varphi \sqrt{I_{0,r_0,r_1,\frac{2 y}{\log y}}}
\leq 0.4153461
\end{equation}
for $10^{25}\leq y\leq 10^{140}$. By (\ref{eq:jijona}),
(\ref{eq:hasmo}) is also true for $y>10^{150}$. Hence
\[f_0(r_0,y) \leq 
 0.4153461\cdot \sqrt{\frac{2}{\sqrt{r_0}} 
5.73827} \leq 0.071498.
\]

By (\ref{eq:byrne}), we conclude that
\[\int_{r_0}^{r_1} \frac{g(r)}{r} dr
\leq 0.071498 + 0.016908 + 0.001399 \leq 0.089805.\]
By (\ref{eq:felipa}),
\[\frac{2 S}{\log x + 2 c^-} \leq \frac{2 (0.640209 x \log x - 0.021095 x)}{
\log x + 2 c^-} \leq 2\cdot 0.640209 x = 1.280418 x,\]
where we recall that $c^-=0.6294>0$.
Hence
\begin{equation}\label{eq:casbah}
\frac{2 S}{\log x + 2 c^-} \int_{r_0}^{r_1} \frac{g(r)}{r} dr
\leq 0.114988 x.
\end{equation}

Putting (\ref{eq:jbaldwin}), (\ref{eq:comrade}) 
and (\ref{eq:casbah}) together, we conclude that the quantity $M$
defined in (\ref{eq:gypo}) is bounded by
\begin{equation}\label{eq:bustier}
M\leq 0.38492 x + 0.30517 x + 0.114988 x \leq 0.80508 x.
\end{equation}

Gathering the terms from (\ref{eq:hexelor}),
(\ref{eq:lamber})
and (\ref{eq:bustier}), we see that Theorem \ref{thm:ostop} states
that the minor-arc total
\[Z_{r_0} = \int_{(\mathbb{R}/\mathbb{Z})\setminus \mathfrak{M}_{8,r_0}}
|S_{\eta_*}(\alpha,x)| |S_{\eta_+}(\alpha,x)|^2 d\alpha\]
is bounded by
\begin{equation}\label{eq:rozoj}
\begin{aligned}Z_{r_0} &\leq
\left(\sqrt{\frac{|\varphi|_1 x}{\varkappa} (M+T)} + 
\sqrt{S_{\eta_*}(0,x)\cdot E}\right)^2\\
&\leq \left(\sqrt{|\varphi|_1 (0.80508 + 3.5776\cdot 10^{-4})}
\frac{x}{\sqrt{\varkappa}} + 
\sqrt{1.0532\cdot 10^{-11}} \frac{x}{\sqrt{\varkappa}}
\right)^2\\
&\leq 1.00948\frac{x^2}{\varkappa}
\end{aligned}\end{equation}
for $r_0=150000$, $x\geq 4.9\cdot 10^{26}$, where we use yet again the
fact that $|\varphi|_1 = \sqrt{\pi/2}$. This is our total minor-arc bound.

\section{Conclusion: proof of main theorem}

As we have known from the start,
\begin{equation}\label{eq:masd}\begin{aligned}
\sum_{n_1+n_2+n_3=N} &\Lambda(n_1) \Lambda(n_2) \Lambda(n_3) \eta_+(n_1)
\eta_+(n_2) \eta_*(n_3) \\ &=
\int_{\mathbb{R}/\mathbb{Z}} S_{\eta_+}(\alpha,x)^2 S_{\eta_*}(\alpha,x)
e(-N \alpha) d\alpha.\end{aligned}\end{equation}
We have just shown that, assuming $N\geq 10^{27}$, $N$ odd,
\[\begin{aligned}
\int_{\mathbb{R}/\mathbb{Z}} &S_{\eta_+}(\alpha,x)^2 S_{\eta_*}(\alpha,x)
e(-N \alpha) d\alpha \\ &= 
\int_{\mathfrak{M}_{8,r_0}} S_{\eta_+}(\alpha,x)^2 S_{\eta_*}(\alpha,x)
e(-N \alpha) d\alpha\\ &+ O^*\left(
\int_{(\mathbb{R}/\mathbb{Z})\setminus 
\mathfrak{M}_{8,r_0}} |S_{\eta_+}(\alpha,x)|^2 |S_{\eta_*}(\alpha,x)|
d\alpha\right) \\ &\geq 1.058259 \frac{x^2}{\varkappa}
+ O^*\left(1.00948 \frac{x^2}{\varkappa}\right) \geq
0.04877 \frac{x^2}{\varkappa}
\end{aligned}\]
for $r_0=150000$, where $x = N/(2+9/(196 \sqrt{2 \pi}))$, as in
(\ref{eq:warwar}).
 (We are using (\ref{eq:juventud}) and (\ref{eq:rozoj}).) Recall
that $\varkappa = 49$ and $\eta_*(t) = (\eta_2\ast_M \varphi)(\varkappa t)$,
where $\varphi(t) = t^2 e^{-t^2/2}$.

It only remains to show that the contribution of terms with $n_1$, $n_2$
or $n_3$ non-prime to the sum in (\ref{eq:masd}) is negligible.
(Let us take out $n_1$, $n_2$, $n_3$ equal to $2$ as well, since some
prefer to state the ternary Goldbach conjecture as follows:
every odd number $\geq 9$ is the sum of three {\em odd} primes.)
Clearly
\begin{equation}\label{eq:duke}\begin{aligned}
\mathop{\sum_{n_1+n_2+n_3=N}}_{\text{$n_1$, $n_2$ or $n_3$ even or non-prime}}
 \Lambda(n_1) &\Lambda(n_2) \Lambda(n_3) \eta_+(n_1)
\eta_+(n_2) \eta_*(n_3)\\ \leq
3 |\eta_+|_\infty^2 |\eta_*|_\infty
&\mathop{\sum_{n_1+n_2+n_3=N}}_{\text{$n_1$ even or non-prime}}
 \Lambda(n_1) \Lambda(n_2) \Lambda(n_3)\\
\leq 3 |\eta_+|_\infty^2 |\eta_*|_\infty\cdot 
&(\log N) 
\mathop{\sum_{\text{$n_1\leq N$ non-prime}}}_{\text{or $n_1=2$}}
\Lambda(n_1)
\sum_{n_2\leq N} \Lambda(n_2).
 \end{aligned}\end{equation}
By (\ref{eq:sazar}) and (\ref{eq:macadam}), 
$|\eta_+|_\infty\leq 1.079955$ and $|\eta_*|_\infty \leq 1.414$.
By \cite[Thms. 12 and 13]{MR0137689},
\[\begin{aligned}
\mathop{\sum_{\text{$n_1\leq N$ non-prime}}}_{\text{or $n_1=2$}}
\Lambda(n_1) &< 1.4262 \sqrt{N} + \log 2 < 1.4263 \sqrt{N},\\
\mathop{\sum_{\text{$n_1\leq N$ non-prime}}}_{\text{or $n_1=2$}}
\Lambda(n_1) \sum_{n_2\leq N}
\Lambda(n_2) &= 1.4263 \sqrt{N} \cdot 1.03883 N 
\leq 1.48169 N^{3/2}.
\end{aligned}\]
Hence, the sum on the first line of (\ref{eq:duke}) is at most
\[7.3306 N^{3/2} \log N.\]

Thus, for $N\geq 10^{27}$ odd, 
\[\begin{aligned}
\mathop{\sum_{n_1+n_2+n_3=N}}_{\text{$n_1$, $n_2$, $n_3$ odd primes}}
 &\Lambda(n_1) \Lambda(n_2) \Lambda(n_3) \eta_+(n_1)
\eta_+(n_2) \eta_*(n_3) \\&\geq
0.04877 \frac{x^2}{\varkappa}
- 7.3306 N^{3/2} \log N\\
&\geq 0.00024433 N^2 - 1.4412\cdot 10^{-11} \cdot N^2 \geq 0.0002443 N^2
\end{aligned}\]
by $\varkappa=49$ and (\ref{eq:warwar}).
Since $0.0002443 N^2>0$, this shows that every odd number $N\geq 10^{27}$
can be written as the sum of three odd primes.

Since the ternary Goldbach conjecture has already been checked for all
$N\leq 8.875\cdot 10^{30}$ \cite{MR3171101},
we conclude that
 every odd number $N>7$ can
be written as the sum of three odd primes, and every odd number $N>5$ can
be written as the sum of three primes. The main result is hereby proven:
the ternary Goldbach conjecture is true.

\appendix
\part{Appendices}

\chapter{Norms of smoothing functions}\label{app:norsmo}

Our aim here is to give bounds on the norms of some smoothing functions --
and, in particular, on several norms of a smoothing function $\eta_+:\lbrack 0,\infty)\to \mathbb{R}$ based on the Gaussian $\eta_\heartsuit(t) = e^{-t^2/2}$.

As before, we write
\begin{equation}\label{eq:clager}
h:t\mapsto \begin{cases}
t^2 (2-t)^3 e^{t-1/2} &\text{if $t\in \lbrack 0,2\rbrack$,}\\ 0 &\text{otherwise}\end{cases}
\end{equation}
We recall that we will work with an approximation $\eta_+$ to the function
$\eta_\circ:\lbrack 0,\infty)\to \mathbb{R}$ defined by
\begin{equation}\label{eq:cleoapp}
\eta_\circ(t) = h(t) \eta_\heartsuit(t) = \begin{cases}
t^3 (2-t)^3 e^{-(t-1)^2/2} &\text{for $t\in \lbrack 0,2\rbrack$,}\\
0 &\text{otherwise.}\end{cases}\end{equation}
The approximation $\eta_+$ is defined by 
\begin{equation}\label{eq:patra}\eta_+(t) = h_H(t) t e^{-t^2/2},\end{equation} 
where
\begin{equation}\label{eq:dirich}\begin{aligned}
F_H(t) &= \frac{\sin(H \log y)}{\pi \log y},\\
h_H(t) &= (h \ast_M F_H)(y) = \int_0^\infty h(t y^{-1}) F_H(y)
\frac{dy}{y}
\end{aligned}\end{equation}
and $H$ is a positive constant to be set later. By (\ref{eq:zorbag}),
$M h_H = M h \cdot M F_H$. Now $F_H$ is just a Dirichlet kernel
under a change of variables; using this, we get that, for $\tau$ real, 
\begin{equation}M F_H(i \tau) = \begin{cases}
1 &\text{if $|\tau|< H$,}\\
1/2 &\text{if $|\tau|=H$,}\\
0 &\text{if $|\tau|>H$.}
\end{cases}\end{equation}
Thus,
\begin{equation}\label{eq:karmor}
M h_H(i \tau) = \begin{cases}
M h(i \tau) &\text{if $|\tau|< H$,}\\
\frac{1}{2} M h(i \tau) &\text{if $|\tau|=H$,}\\
0 &\text{if $|\tau|>H$.}
\end{cases}\end{equation}

As it turns out, $h$, $\eta_\circ$ and $M h$ (and hence $M h_H$)
 are relatively easy to
work with, whereas we can already see that
$h_H$ and $\eta_+$ have more complicated definitions. 
Part of our work will consist in expressing norms of $h_H$ and $\eta_+$
in terms of norms of $h$, $\eta_\circ$ and $M h$.

\section{The decay of a Mellin transform}

Now, consider any
$\phi:\lbrack 0,\infty)\to \mathbb{C}$ that (a) has compact support
(or fast decay), (b)
 satisfies $\phi^{(k)}(t) t^{k-1} = O(1)$ for $t\to 0^+$ and
$0\leq k\leq 3$, and (c)
 is $C^2$ everywhere and quadruply differentiable outside a 
finite set of points.

By definition,
\[M\phi(s) = \int_0^\infty \phi(x) x^s \frac{dx}{x}.\]
Thus, by integration by parts, for
$\Re(s)>-1$ and $s\ne 0$,
\begin{equation}\label{eq:dotorwho}
\begin{aligned} &M\phi(s) = \int_0^\infty \phi(x) x^s \frac{dx}{x} = 
\lim_{t\to 0^+}
 \int_t^\infty \phi(x) x^s \frac{dx}{x} = 
- \lim_{t\to 0^+} \int_t^\infty \phi'(x) \frac{x^s}{s} dx \\ 
&= \lim_{t\to 0^+} \int_t^\infty \phi''(x)
\frac{x^{s+1}}{s (s+1)} dx = \lim_{t\to 0^+} 
- \int_t^\infty \phi^{(3)}(x) \frac{x^{s+2}}{s (s+1) (s+2)} dx \\ &= 
\lim_{t\to 0^+} \int_t^\infty \phi^{(4)}(x)
\frac{x^{s+3}}{s (s+1) (s+2) (s+3)} dx 
,\end{aligned}\end{equation}
where $\phi^{(4)}(x)$ is understood in the sense of distributions at the
finitely many points where it is not well-defined as a function.

Let $s = it$, $\phi = h$.
Let $C_k = \lim_{t\to 0^+} \int_t^\infty |h^{(k)}(x)| x^{k-1} dx$ for $0\leq k\leq 4$.
Then (\ref{eq:dotorwho}) gives us that
\begin{equation}\label{eq:kust}
Mh(i t) \leq \min\left(C_0, \frac{C_1}{|t|}, \frac{C_2}{|t| |t+i|},
\frac{C_3}{|t| |t+i| |t+2i|},\frac{C_4}{|t| |t+i| |t+2 i| |t+3i|}\right) .
\end{equation}
We must
 estimate the constants $C_j$, $0\leq j\leq 4$.

Clearly, $h(t) t^{-1} = O(1)$ as $t\to 0^+$, $h^{k}(t) = O(1)$ as $t\to
0^+$
for all $k\geq 1$, $h(2)=h'(2)=h''(2)=0$, and 
$h(x)$, $h'(x)$
and $h''(x)$ are all continuous. The function $h'''$ has 
a discontinuity at $t=2$. As we said, we understand $h^{(4)}$ in
the sense of distributions at $t=2$; for example,
$\lim_{\epsilon\to 0} \int_{2 -\epsilon}^{2+\epsilon} h^{(4)}(t) dt = 
\lim_{\epsilon\to 0} (h^{(3)}(2+\epsilon) - h^{(3)}(2-\epsilon))$.

Symbolic integration easily gives that
\begin{equation}\label{eq:nessu0}
C_0 = \int_0^2 t (2-t)^3 e^{t-1/2} dt = 92 e^{-1/2} - 12 e^{3/2}
= 2.02055184\dotsc\end{equation}
We will have to compute $C_{k}$, $1\leq k\leq 4$, with some care, due to the
absolute value involved in the definition.

The function $(x^2 (2-x)^3 e^{x-1/2})'= ((x^2 (2-x)^3)' + x^2 (2-x)^3) e^{x-1/2}$ 
has the same zeros as $H_1(x) = (x^2 (2-x)^3)' + x^2 (2-x)^3$, namely, 
$-4$, $0$, $1$ and $2$.
The sign of $H_1(x)$ (and hence
of $h'(x)$) is $+$ within
$(0,1)$ and $-$ within $(1,2)$. Hence
\begin{equation}\label{eq:nessu1}
\begin{aligned}C_{1} &= \int_0^\infty |h'(x)| dx = |h(1) - h(0)| + 
|h(2) - h(1)| =
 2 h(1) = 2 \sqrt{e}.
\end{aligned}\end{equation}

The situation with $(x^2 (2-x)^3 e^{x-1/2})''$ is similar: it has
zeros at the roots of $H_2(x)=0$, where $H_2(x) = H_1(x) + H_1'(x)$
(and, in general, $H_{k+1}(x) = H_k(x) + H_k'(x)$). This time, we will
prefer to find the roots numerically.
It is enough to find (candidates for)
the roots using any available tool\footnote{Routine \texttt{find\_root} in SAGE was used
here.}
and then check rigorously that the sign does change around the purported roots.
In this way, we check that $H_2(x)=0$ has two roots $\alpha_{2,1}$,
$\alpha_{2,2}$ in the interval
$(0,2)$, another root at $2$, and two more roots outside $\lbrack 0,2\rbrack$;
moreover,
\begin{equation}\label{eq:depard}\begin{aligned}
\alpha_{2,1} &= 0.48756597185712\dotsc,\\
\alpha_{2,2} &= 1.48777169309489\dotsc, \end{aligned}\end{equation}
where we verify the root using interval arithmetic.
The sign of $H_2(x)$ (and hence of $h''(x)$) is first $+$, then $-$, then $+$.
Write $\alpha_{2,0}=0$, $\alpha_{2,3}=2$. By integration by parts,
\begin{equation}\label{eq:nessu2}
\begin{aligned}C_{2} &= \int_0^\infty |h''(x)| x\; dx = 
\int_0^{\alpha_{2,1}} h''(x) x\; dx -
 \int_{\alpha_{2,1}}^{\alpha_{2,2}} h''(x) x\; dx + 
\int_{\alpha_{2,2}}^2 h''(x) x\; dx\\
&= \sum_{j=1}^3 (-1)^{j+1} \left( h'(x) x|_{\alpha_{2,j-1}}^{\alpha_{2,j}} - 
\int_{\alpha_{2,j-1}}^{\alpha_{2,j}} h'(x)\; dx\right)\\
&= 2 \sum_{j=1}^2 (-1)^{j+1} 
\left(h'(\alpha_{2,j}) \alpha_{2,j} - h(\alpha_{2,j})\right)
= 10.79195821037\dotsc
.\end{aligned}\end{equation}

To compute $C_{3}$, we proceed in the same way,
finding two roots of $H_3(x)=0$ (numerically) within the
interval $(0,2)$, viz.,
\[\begin{aligned}
\alpha_{3,1} &= 1.04294565694978\dotsc\\
\alpha_{3,2} &= 1.80999654602916\dotsc
\end{aligned}\]
The sign of $H_3(x)$ on the interval $\lbrack 0,2\rbrack$
is first $-$, then $+$, then $-$. Write $\alpha_{3,0}=0$, $\alpha_{3,3}=2$.
Proceeding as before -- with the only difference that the integration
by parts is iterated once now -- we obtain that
\begin{equation}
\begin{aligned}C_{3} &= \int_0^\infty |h'''(x)| x^2 dx
= \sum_{j=1}^3 (-1)^j \int_{\alpha_{3,j-1}}^{\alpha_{3,j}} h'''(x) x^2 dx\\
&= \sum_{j=1}^3 (-1)^j \left(h''(x) x^2 |_{\alpha_{3,j-1}}^{\alpha_{3,j}} -
\int_{\alpha_{3,j-1}}^{\alpha_{3,j}} h''(x) \cdot 2 x\right) dx\\
&= \sum_{j=1}^3 (-1)^j \left(h''(x) x^2 - h'(x) \cdot 2 x + 2 h(x)
\right)|_{\alpha_{3,j-1}}^{\alpha_{3,j}}\\
&=
2 \sum_{j=1}^2 (-1)^{j} (h''(\alpha_{3,j}) \alpha_{3,j}^2 - 
2 h'(\alpha_{3,j}) \alpha_{3,j} +
2 h(\alpha_{3,j}))\end{aligned}\end{equation}
and so interval arithmetic gives us
\begin{equation}\label{eq:nessu3}
C_{3} = 75.1295251672\dotsc
\end{equation}

The treatment of the integral in $C_{4}$ is very similar, at least
as first. There are two roots of $H_4(x)=0$ in the interval
$(0,2)$, namely,
\[\begin{aligned}
\alpha_{4,1} &= 0.45839599852663\dotsc\\
\alpha_{4,2} &= 1.54626346975533\dotsc\end{aligned}\]
The sign of $H_4(x)$ on the interval $\lbrack 0,2\rbrack$ is first 
$-$, $+$, then $-$. Using integration by parts as before, we obtain
\[\begin{aligned}
\int_{0^+}^{2^-} &\left| h^{(4)}(x) \right| x^3 dx 
\\ &= 
- \int_{0^+}^{\alpha_{4,1}} h^{(4)}(x) x^3 dx 
+ \int_{\alpha_{4,1}}^{\alpha_{4,2}} h^{(4)}(x) x^3 dx 
- \int_{\alpha_{4,1}}^{2^-} h^{(4)}(x) x^3 dx \\
&= 2 \sum_{j=1}^2 (-1)^j \left(h^{(3)}(\alpha_{4,j}) \alpha_{4,j}^3 -
3 h^{(2)}(\alpha_{4,j}) \alpha_{4,j}^2 + 6 h'(\alpha_{4,j}) \alpha_{4,j} -
6 h(\alpha_{4,j})\right)\\ &- \lim_{t\to 2^-} h^{(3)}(t) t^3
= 1152.69754862\dotsc,\end{aligned}\]
since $\lim_{t\to 0^+} h^{(k)}(t) t^k = 0$ for $0\leq k\leq 3$,
 $\lim_{t\to 2^-} h^{(k)}(t) = 0$ for $0\leq k\leq 2$ and
$\lim_{t\to 2^-} h^{(3)}(t) = - 24 e^{3/2}$.
Now
\[\int_{2^-}^\infty |h^{(4)}(x) x^3| dx = \lim_{\epsilon\to 0^+}
|h^{(3)}(2+\epsilon) - h^{(3)}(2-\epsilon)| \cdot 2^3 = 2^3\cdot 24 e^{3/2},\]
Hence
\begin{equation}\label{eq:nessu4}
C_{4} = \int_{0^+}^{2^-} \left| h^{(4)}(x) \right| x^3 dx 
+ 24 e^{3/2} \cdot 2^3 = 2013.18185012\dotsc
\end{equation}

We finish by remarking that can write down $M h$ explicitly:
\begin{equation}\label{eq:consti}Mh =
 - e^{-1/2} (-1)^{-s} 
(8 \gamma(s+2,-2) + 12 \gamma(s+3,-2) + 6 \gamma(s+4,-2) + \gamma(s+5,-2)),
\end{equation}
 where $\gamma(s,x)$ is the {\em (lower) 
incomplete Gamma function} 
\[\gamma(s,x) = \int_0^x e^{-t} t^{s-1} dt.\]
We will, however, find it easier to deal with $M h$ by means of the
bound (\ref{eq:kust}), in part because (\ref{eq:consti}) amounts to
an invitation to numerical instability.

For instance, it is easy to use (\ref{eq:kust}) to give a bound for
the $\ell_1$-norm of $Mh(i t)$. Since $C_4/C_3 > C_3/C_2 > C_2/C_1 > C_1/C_0$,
\[\begin{aligned}
&|M h(i t)|_1 = 2 \int_0^\infty Mh(it) dt \\ \leq
&2\left(C_0 \frac{C_1}{C_0} + C_1 \int_{C_1/C_0}^{C_2/C_1} \frac{dt}{t} +
C_2 \int_{C_2/C_1}^{C_3/C_2} \frac{dt}{t^2} + 
C_3 \int_{C_3/C_2}^{C_4/C_3} \frac{dt}{t^3} + 
C_4 \int_{C_4/C_3}^\infty \frac{dt}{t^4}\right)\\
= &2\left(C_1 + C_1 \log \frac{C_2 C_0}{C_1^2} + C_2 \left(\frac{C_1}{C_2} - 
\frac{C_2}{C_3}\right) + 
\frac{C_3}{2} \left(\frac{C_2^2}{C_3^2} - \frac{C_3^2}{C_4^2}\right)
+ \frac{C_4}{3} \cdot \frac{C_3^3}{C_4^3}\right),\end{aligned}\]
and so
\begin{equation}\label{eq:marpales}
|M h(i t)|_1\leq 16.1939176.\end{equation}
This bound is far from tight, but it will certainly be useful.

Similarly, $|(t +i) M h(it)|_1$ is at most two times
\[\begin{aligned}
&C_0\int_0^{\frac{C_1}{C_0}} |t+i| \; dt +
C_1 \int_{\frac{C_1}{C_0}}^{\frac{C_2}{C_1}} \left|1+ \frac{i}{t}\right| dt +
C_2 \int_{\frac{C_2}{C_1}}^{\frac{C_3}{C_2}} \frac{dt}{t} + 
C_3 \int_{\frac{C_3}{C_2}}^{\frac{C_4}{C_3}} \frac{dt}{t^2} + 
C_4 \int_{\frac{C_4}{C_3}}^\infty \frac{dt}{t^3}\\
&= \frac{C_0}{2} \left(\sqrt{\frac{C_1^4}{C_0^4} + \frac{C_1^2}{C_0^2}} + 
\sinh^{-1} \frac{C_1}{C_0}\right) + C_1
\left(\sqrt{t^2+1} + \log\left(\frac{\sqrt{t^2+1}-1}{t}\right)\right)|_{\frac{C_1}{C_0}}^{\frac{C_2}{C_1}}\\ &+
 C_2 \log \frac{C_3 C_1}{C_2^2} + C_3 \left(\frac{C_2}{C_3} - 
\frac{C_3}{C_4}\right) + 
\frac{C_4}{2} \frac{C_3^2}{C_4^2},\end{aligned}\]
and so
\begin{equation}\label{eq:marplat}
|(t+i) Mh(i t)|_1 \leq 27.8622803.
\end{equation}
\section{The difference $\eta_+-\eta_\circ$ in $\ell_2$ norm.}\label{subs:yosabi}
We wish to estimate the distance in $\ell_2$ norm between
$\eta_\circ$ and its approximation $\eta_+$. This will be an easy affair,
since, on the imaginary axis, the Mellin transform of $\eta_+$ is
just a truncation of the Mellin transform of $\eta_\circ$.

By (\ref{eq:cleoapp}) and (\ref{eq:patra}),
\begin{equation}\label{eq:estor}\begin{aligned}
|\eta_+-\eta_\circ|_2^2 &=
\int_0^\infty \left|h_H(t) t e^{-t^2/2} - h(t) t e^{-t^2/2}\right|^2 dt \\
&\leq
\left(\max_{t\geq 0} e^{-t^2} t^3 \right) \cdot
 \int_0^\infty |h_H(t)-h(t)|^2 \frac{dt}{t}.\end{aligned} \end{equation}
The maximum $\max_{t\geq 0} t^3 e^{-t^2}$ is $(3/2)^{3/2} e^{-3/2}$.
Since the Mellin transform is an isometry (i.e., (\ref{eq:victi}) holds),
\begin{equation}\label{eq:sev}\int_0^\infty |h_H(t)-h(t)|^2 \frac{dt}{t} = 
\frac{1}{2\pi} \int_{-\infty}^\infty |Mh_H(it)-Mh(it)|^2 dt =
\frac{1}{\pi} \int_{H}^\infty |Mh(it)|^2 dt.\end{equation}
 By (\ref{eq:kust}),
\begin{equation}\label{eq:shei}
\int_H^\infty
|Mh(it)|^2 dt \leq \int_H^\infty \frac{C_4^2}{t^8} dt
\leq \frac{C_4^2}{7 H^7}.
\end{equation}
Hence
\begin{equation}\label{eq:sevshei}
\int_0^\infty |h_H(t) - h(t)|^2 \frac{dt}{t} \leq \frac{C_4^2}{7 \pi H^7}.
\end{equation}
Using the bound (\ref{eq:nessu4}) for $C_4$, we conclude that 
\begin{equation}\label{eq:impath}
|\eta_+-\eta_\circ|_2 \leq \frac{C_4}{\sqrt{7 \pi}}
\left(\frac{3}{2e}\right)^{3/4} \cdot \frac{1}{H^{7/2}} \leq
\frac{274.856893}{H^{7/2}} .\end{equation}

It will also be useful to bound 
\[\left|\int_0^\infty (\eta_+(t) - \eta_\circ(t))^2 \log t\; dt\right|.\] 
This is at most
\[\left(\max_{t\geq 0} e^{-t^2} t^3 |\log t|\right)\cdot 
\int_0^\infty |h_H(t)- h(t)|^2 \frac{dt}{t}.
\]
Now
\[\begin{aligned}\max_{t\geq 0} e^{-t^2} t^3 |\log t| &= 
\max\left(\max_{t\in \lbrack 0,1\rbrack} e^{-t^2} t^3 (-\log t),
\max_{t\in \lbrack 1,5\rbrack} e^{-t^2} t^3 \log t\right)\\ &=
0.14882234545\dotsc 
\end{aligned}\]
where we find the maximum by the bisection method with $40$ iterations
(see \ref{sec:koloko}).
Hence, by (\ref{eq:sevshei}),
\begin{equation}\label{eq:lozhka}\begin{aligned}
\int_0^\infty (\eta_+(t)-\eta_\circ(t))^2 |\log t| dt &\leq
0.148822346 \frac{C_4^2}{7 \pi} \\ &\leq
\frac{27427.502}{H^7} \leq \left(\frac{165.61251}{H^{7/2}}\right)^2.
\end{aligned}\end{equation}

\section{Norms involving $\eta_+$}\label{subs:daysold}

Let us now bound some $\ell_1$- and $\ell_2$-norms involving $\eta_+$. Relatively crude bounds
will suffice in most cases. 

First, by (\ref{eq:impath}),
\begin{equation}\label{eq:mastodon}\begin{aligned}
|\eta_+|_2 &\leq |\eta_\circ|_2 + |\eta_+ - \eta_\circ|_2 \leq
0.800129 + \frac{274.8569}{H^{7/2}},\\
|\eta_+|_2 &\geq |\eta_\circ|_2 - |\eta_+ - \eta_\circ|_2 \geq
0.800128 - \frac{274.8569}{H^{7/2}},
\end{aligned}\end{equation}
where we obtain 
\begin{equation}\label{eq:lamia}
|\eta_\circ|_2 = \sqrt{0.640205997\dotsc} = 0.8001287\dotsc
\end{equation} by symbolic integration.

Let us now bound $|\eta_+\cdot \log|_2^2$.
By isometry and (\ref{eq:harva}),
\[|\eta_+\cdot \log|_2^2 
= \frac{1}{2 \pi i}
\int_{\frac{1}{2} -i \infty}^{\frac{1}{2} +i \infty}
|M(\eta_+\cdot \log)(s)|^2 ds = 
\frac{1}{2 \pi i}
\int_{\frac{1}{2} -i \infty}^{\frac{1}{2} +i \infty}
|(M \eta_+)'(s)|^2 ds .
\]
Now, $(M \eta_+)'(1/2+it)$ equals $1/2\pi$ times the additive convolution of 
$M h_H(it)$ and $(M \eta_\diamondsuit)'(1/2+i t)$, where $\eta_\diamondsuit(t) = 
t e^{-t^2/2}$. Hence, by Young's
inequality, \[|(M \eta_+)'(1/2+it)|_2\leq \frac{1}{2\pi} |M h_H(it)|_1 
|(M \eta_\diamondsuit)'(1/2+it)|_2.\] 

Again by isometry and (\ref{eq:harva}),
\[|(M \eta_\diamondsuit)'(1/2+it)|_2 = \sqrt{2\pi} |\eta_\diamondsuit\cdot \log|_2.
\]
Hence, by (\ref{eq:marpales}),
\[|\eta_+\cdot \log|_2 \leq
\frac{1}{2 \pi} |M h_H(it)|_1 |\eta_\diamondsuit \cdot \log|_2 \leq
2.5773421 \cdot |\eta_\diamondsuit \cdot \log |_2.\]
Since, by symbolic integration,
\begin{equation}\label{eq:sonamo}\begin{aligned}
|\eta_\diamondsuit\cdot \log|_2 &\leq 
\sqrt{\frac{\sqrt{\pi}}{32} \left(8 (\log 2)^2 + 2\gamma^2 + \pi^2 + 
8(\gamma-2) \log 2 - 8 \gamma\right)}\\
&\leq 0.3220301,
\end{aligned}\end{equation}
we get that
\begin{equation}\label{eq:pamiatka}
|\eta_+\cdot \log|_2 \leq 0.8299818.\end{equation}

Let us bound $|\eta_+(t) t^\sigma|_1$ 
for $\sigma \in (-2,\infty)$.
By Cauchy-Schwarz and
Plancherel,
\begin{equation}\label{eq:lesalpes}\begin{aligned}
&|\eta_+(t) t^\sigma|_1 
= \left|h_H(t) t^{1+\sigma} e^{-t^2/2}\right|_1 \leq
\left|t^{\sigma+3/2} e^{-t^2/2}\right|_2 |h_H(t)/\sqrt{t}|_2 \\ &=
\left|t^{\sigma+3/2} e^{-t^2/2}\right|_2 
\sqrt{\int_0^\infty |h_H(t)|^2 \frac{dt}{t}}
=
\left|t^{\sigma+3/2} e^{-t^2/2}\right|_2 \cdot
 \sqrt{\frac{1}{2\pi} \int_{- H}^{H}
 |Mh(ir)|^2 dr}\\&\leq 
\left|t^{\sigma+3/2} e^{-t^2/2}\right|_2 
\cdot \sqrt{\frac{1}{2\pi} \int_{- \infty}^{\infty}
 |Mh(ir)|^2 dr}
= 
\left|t^{\sigma+3/2} e^{-t^2/2}\right|_2 \cdot 
|h(t)/\sqrt{t}|_2.
\end{aligned}\end{equation}
Since
\[\begin{aligned}
\left|t^{\sigma+3/2} e^{-t^2/2}\right|_2 &= 
\sqrt{\int_0^\infty e^{-t^2} t^{2\sigma+3} dt}
= \sqrt{\frac{\Gamma(\sigma+2)}{2}},\\
|h(t)/\sqrt{t}|_2 &=
\sqrt{\frac{31989}{8 e} - \frac{585 e^3}{8}} \leq
1.5023459
,\end{aligned}\]
we conclude that \begin{equation}\label{eq:paytoplay}
|\eta_+(t) t^{\sigma}|_1 \leq 1.062319
\cdot \sqrt{\Gamma(\sigma+2)}
\end{equation}
for $\sigma>-2$.
\section{Norms involving $\eta_+'$}\label{subs:weeksold}

By one of the standard transformation rules (see (\ref{eq:harva})),
the Mellin transform of $\eta_+'$ equals $- (s-1)\cdot  M\eta_+(s-1)$.
Since the Mellin transform is an isometry in the sense of (\ref{eq:victi}),
\[
|\eta_+'|_2^2 = \frac{1}{2\pi i } \int_{\frac{1}{2} - i\infty }^{
\frac{1}{2} + i \infty} \left|M(\eta_+')(s)\right|^2 ds =
 \frac{1}{2\pi i} \int_{-\frac{1}{2} - i \infty}^{- \frac{1}{2} + i \infty} 
\left|s\cdot M \eta_+(s)\right|^2 ds. 
\]
Recall that $\eta_+(t) = h_H(t) \eta_\diamondsuit(t)$, where
$\eta_\diamondsuit(t) = t e^{-t^2/2}$. Thus,
by (\ref{eq:mouv}), the function 
$M\eta_+(-1/2+it)$ equals $1/2\pi$ times the (additive)
convolution of $M h_H(it)$ and $M \eta_{\diamondsuit}(-1/2+it)$. Therefore,
for $s=-1/2+it$,
\begin{equation}\label{eq:grabai}
\begin{aligned} |s| \left|M \eta_+(s) \right| &= 
\frac{|s|}{2 \pi} \int_{-H}^H Mh(ir) M\eta_{\diamondsuit}(s-ir) dr\\ &\leq
\frac{3}{2\pi} \int_{-H}^H |ir - 1| |Mh(ir)| \cdot |s- ir| |M\eta_{\heartsuit}(s-ir)| 
dr\\
&= \frac{3}{2 \pi} (f\ast g)(t),
\end{aligned}\end{equation}
where $f(t) = |it - 1| |Mh(i t)|$ and 
$g(t) = |-1/2+it| |M\eta_{\diamondsuit}(-1/2+it)|$. 
(Since $|(-1/2+i(t-r))+(1+ir)| = |1/2+it| = |s|$, 
either $|-1/2+i(t-r)|\geq |s|/3$ or
$|1+ir|\geq 2 |s|/3$; hence $|s-ir| 
|ir - 1| = |-1/2+i(t-r)| |1+ir| 
\geq |s|/3$.)
By Young's inequality (in a special case that follows from Cauchy-Schwarz),
$|f\ast g|_2 \leq |f|_1 |g|_2$. 
By (\ref{eq:marplat}),
\[\begin{aligned}
|f|_1 = 
|(r+i) Mh(ir)|_1 \leq 27.8622803.
\end{aligned}\]
Yet again by Plancherel,
\[\begin{aligned}
|g|_2^2 &= \int_{-\frac{1}{2} - i\infty}^{-\frac{1}{2} + i \infty}
|s|^2 |M \eta_{\diamondsuit}(s)|^2 ds \\ &=  
\int_{\frac{1}{2} - i\infty}^{\frac{1}{2} + i \infty}
|(M(\eta_{\diamondsuit}')) (s)|^2 ds = 2 \pi |\eta_{\diamondsuit}'|_2^2 = 
\frac{3 \pi^{\frac{3}{2}}}{4}.
\end{aligned}\]
Hence
\begin{equation}\label{eq:miran}\begin{aligned}
|\eta_+'|_2 &\leq \frac{1}{\sqrt{2\pi}} \cdot \frac{3}{2 \pi} |f\ast g|_2\\
 &\leq \frac{1}{\sqrt{2 \pi}} \frac{3}{2 \pi} \cdot 27.8622803
\sqrt{\frac{3 \pi^{\frac{3}{2}}}{4}} \leq 10.845789.\end{aligned}\end{equation}

Let us now bound $|\eta_+'(t) t^\sigma|_1$ for $\sigma\in (-1,\infty)$.
First of all,
\[\begin{aligned}
|\eta_+'(t) t^\sigma|_1
 &= \left|\left(h_H(t) t e^{-t^2/2}\right)'
t^\sigma\right|_1\\ &\leq 
\left|\left(h_H'(t) t e^{-t^2/2} +
h_H(t) (1-t^2) e^{-t^2/2}\right) \cdot t^\sigma\right|_1\\
&\leq \left|h_H'(t) t^{\sigma+1} e^{-t^2/2}\right|_1 +
|\eta_+(t) t^{\sigma-1}|_1 + |\eta_+(t) t^{\sigma+1}|_1
.\end{aligned}\]
We can bound the last two terms by (\ref{eq:paytoplay}).
Much as in (\ref{eq:lesalpes}), we note that
\[
\left|h_H'(t) t^{\sigma+1} e^{-t^2/2}\right|_1 \leq
\left|t^{\sigma+1/2} e^{-t^2/2}\right|_2 |h_H'(t) \sqrt{t}|_2,\]
and then see that
\[\begin{aligned}
&|h_H'(t) \sqrt{t}|_2
 =
\sqrt{\int_0^\infty |h_H'(t)|^2 t\; dt}
=
 \sqrt{\frac{1}{2\pi} \int_{- \infty}^{\infty}
 |M(h_H')(1 + ir)|^2 dr}\\&=
 \sqrt{\frac{1}{2\pi} \int_{- \infty}^{\infty}
 |(-ir) Mh_H(ir)|^2 dr}
=
 \sqrt{\frac{1}{2\pi} \int_{- H}^{H}
 |(-ir) Mh(ir)|^2 dr}\\ &= 
 \sqrt{\frac{1}{2\pi} \int_{- H}^{H}
 |M(h')(1 + ir)|^2 dr} \leq
 \sqrt{\frac{1}{2\pi} \int_{-\infty}^{\infty}
 |M(h')(1 + ir)|^2 dr} = |h'(t) \sqrt{t}|_2,\end{aligned}\]
where we use the first rule in (\ref{eq:harva}) twice.
Since
\[\left|t^{\sigma+1/2} e^{-t^2/2}\right|_2 =
\sqrt{\frac{\Gamma(\sigma+1)}{2}},
\;\;\;\;
|h'(t) \sqrt{t}|_2 = \sqrt{\frac{103983}{16 e} - \frac{1899 e^3}{16}}
= 2.6312226,
\]
we conclude that
\begin{equation}\label{eq:uzsu}\begin{aligned}
|\eta_+'(t) t^\sigma|_1 &\leq 1.062319\cdot (\sqrt{\Gamma(\sigma+1)}
 + \sqrt{\Gamma(\sigma+3)}) + 
\sqrt{\frac{\Gamma(\sigma+1)}{2}} \cdot 2.6312226\\
&\leq 2.922875 \sqrt{\Gamma(\sigma+1)} + 1.062319 \sqrt{\Gamma(\sigma+3)}
\end{aligned}\end{equation}
for $\sigma>-1$.
\section{The $\ell_\infty$-norm of $\eta_+$}\label{subs:byron}

Let us now get a bound for $|\eta_+|_\infty$. 
Recall that $\eta_+(t) = h_H(t) \eta_\diamondsuit(t)$, where 
$\eta_\diamondsuit(t) = t e^{-t^2/2}$.
Clearly
 \begin{equation}\label{eq:jorat}\begin{aligned}
|\eta_+|_\infty &= |h_H(t) \eta_\diamondsuit(t)|_\infty
\leq |\eta_\circ|_\infty + |(h(t)-h_H(t)) \eta_\diamondsuit(t)|_\infty\\
&\leq |\eta_\circ|_\infty + \left|\frac{h(t)-h_H(t)}{t}\right|_\infty 
|\eta_\diamondsuit(t) t|_\infty.
\end{aligned}\end{equation}
Taking derivatives, we easily see that
\[|\eta_\circ|_\infty = \eta_\circ(1) = 1, 
\;\;\;\;\;\;\;\;\;\;
|\eta_\diamondsuit(t) t|_\infty = 2/e. 
\]
It remains to bound $|(h(t)-h_H(t))/t|_\infty$. 
By (\ref{eq:dirich2}),
\begin{equation}\label{eq:narco}
h_H(t) = \int_{\frac{t}{2}}^\infty h(t y^{-1}) \frac{\sin(H \log y)}{\pi
\log y} \frac{dy}{y}
= \int_{- H \log \frac{2}{t}}^\infty h\left(\frac{t}{e^{w/H}}\right)
\frac{\sin w}{\pi w} dw.\end{equation}
The {\em sine integral} 
\[\Si(x) = \int_0^x \frac{\sin t}{t} dt\]
is defined for all $x$; it tends to $\pi/2$ as $x\to +\infty$ and
to $-\pi/2$ as $x\to -\infty$ (see \cite[(5.2.25)]{MR0167642}).  We apply integration by parts
to the second integral in (\ref{eq:narco}), and obtain
\[\begin{aligned}
h_H(t) - h(t) &= - \frac{1}{\pi} \int_{-H \log \frac{2}{t}}^\infty
\left(\frac{d}{dw} h\left(\frac{t}{e^{w/H}}\right)\right) \Si(w) dw - h(t)\\
&= -\frac{1}{\pi} \int_{0}^\infty
\left(\frac{d}{dw} h\left(\frac{t}{e^{w/H}}\right)\right) \left(
\Si(w) - \frac{\pi}{2}\right) dw \\ &-
\frac{1}{\pi} \int_{-H \log \frac{2}{t}}^0
\left(\frac{d}{dw} h\left(\frac{t}{e^{w/H}}\right)\right) \left(
\Si(w) + \frac{\pi}{2}\right) dw.
\end{aligned}\]
Now
\[\left|\frac{d}{dw} h\left(\frac{t}{e^{w/H}}\right)\right| = 
\frac{t e^{-w/H}}{H} \left|h'\left(\frac{t}{e^{w/H}}\right)\right| \leq
\frac{t |h'|_\infty}{H e^{w/H}}.
\]
Integration by parts easily yields the bounds $|\Si(x)-\pi/2|<2/x$ for
$x>0$ and $|\Si(x)+\pi/2|<2/|x|$ for $x<0$; we also know that $0\leq \Si(x)
\leq x < \pi/2$
for $x\in \lbrack 0,1\rbrack$ and 
$-\pi/2 < x \leq \Si(x)\leq 0$ for $x\in \lbrack -1,0\rbrack$. Hence
\[\begin{aligned}
|h_H(t)-h(t)| &\leq \frac{2 t |h'|_\infty}{\pi H} \left(
\int_0^1 \frac{\pi}{2} e^{-w/H} dw 
+ \int_1^\infty \frac{2 e^{-w/H}}{w} dw\right)\\
&= t |h'|_\infty \cdot \left(
(1 - e^{-1/H})
 + \frac{4}{\pi} \frac{E_1(1/H)}{H}\right),
\end{aligned}\]
where $E_1$ is the {\em exponential integral}
\[E_1(z) = \int_z^\infty \frac{e^{-t}}{t} dt.\] 
By \cite[(5.1.20)]{MR0167642},
\[0 < E_1(1/H) < \frac{\log(H+1)}{e^{1/H}},\]
and, since $\log(H+1) = \log H + \log(1+1/H) < \log H + 1/H <
(\log H) (1+1/H) < (\log H) e^{1/H}$ for $H\geq e$, we see that
this gives us that $E_1(1/H)<\log H$ (again for $H\geq e$, as is the case).
Hence
\begin{equation}\label{eq:havana}
\frac{|h_H(t)-h(t)|}{t} < |h'|_\infty \cdot
\left(1 - e^{-\frac{1}{H}} + \frac{4}{\pi} \frac{\log H}{H}\right) <
|h'|_\infty \cdot
\frac{1 + \frac{4}{\pi} \log H}{H},
\end{equation}
and so, by (\ref{eq:jorat}),
\[
|\eta_+|_\infty
 \leq 1 + \frac{2}{e} \left|\frac{h(t)-h_H(t)}{t}\right|_\infty 
< 1 + \frac{2}{e} |h'|_\infty \cdot \frac{1 + \frac{4}{\pi} \log H}{H}.\]
By (\ref{eq:depard}) and interval arithmetic, we determine that
\begin{equation}\label{eq:morno}
|h'|_\infty = |h'(\alpha_{2,2})| \leq 2.805820379671,
\end{equation}
where $\alpha_{2,2}$ is a root of $h''(x)=0$
as in (\ref{eq:depard}). We have proven
\begin{equation}\label{eq:malgache}
|\eta_+|_\infty < 1 + \frac{2}{e}  \cdot 2.80582038\cdot 
\frac{1 + \frac{4}{\pi} \log H}{H}
< 1 + 2.06440727\cdot  \frac{1 + \frac{4}{\pi} \log H}{H}.
\end{equation}

We will need three other bounds of this kind, namely, for $\eta_+(t) \log t$,
$\eta_+(t)/t$ and $\eta_+(t) t$.
We start as in (\ref{eq:jorat}):
\begin{equation}\label{eq:rasal}\begin{aligned}
|\eta_+ \log t|_\infty &\leq |\eta_\circ \log t|_\infty + 
|(h(t)-h_H(t)) \eta_\diamondsuit(t) \log t|_\infty\\ &\leq |\eta_\circ \log t|_\infty
+ |(h-h_H(t))/t|_\infty |\eta_\diamondsuit(t) t \log t|_\infty,\\
|\eta_+(t)/t|_\infty &\leq |\eta_\circ(t)/t|_\infty 
+ |(h-h_H(t))/t|_\infty |\eta_\diamondsuit(t)|_\infty\\
|\eta_+(t) t|_\infty &\leq |\eta_\circ(t) t|_\infty 
+ |(h-h_H(t))/t|_\infty |\eta_\diamondsuit(t) t^2|_\infty
.\end{aligned}\end{equation}
By the bisection method with $30$ iterations, implemented with interval
arithmetic,
\[
|\eta_\circ(t) \log t|_\infty \leq 0.279491, \;\;\;\;\;\;\;\;\;\;
|\eta_\diamondsuit(t) t \log t|_\infty \leq 0.3811561. 
\]
Hence, by (\ref{eq:havana}) and (\ref{eq:morno}),
\begin{equation}\label{eq:dalida}
|\eta_+ \log t|_\infty \leq 0.279491 + 1.069456
 \cdot \frac{1 + \frac{4}{\pi} \log H}{H}.\end{equation}
By the bisection method with $32$ iterations,
\[|\eta_\circ(t)/t|_\infty \leq 1.08754396
.\]
(We can also obtain this by solving $(\eta_\circ(t)/t)'=0$ symbolically.)
It is easy to show that $|\eta_\diamondsuit|_\infty = 1/\sqrt{e}$. 
Hence, again by (\ref{eq:havana})
and (\ref{eq:morno}),
\begin{equation}\label{eq:gobmark}
|\eta_+(t)/t|_\infty \leq 1.08754396 + 1.70181609 \cdot \frac{1 + \frac{4}{\pi}
\log H}{H} .\end{equation}
By the bisection method with $32$ iterations,
\[|\eta_\circ(t) t|_\infty \leq 1.06473476.\]
Taking derivatives, we see that
$|\eta_\diamondsuit(t) t^2|_\infty = 3^{3/2} e^{-3/2}$. Hence, yet again by (\ref{eq:havana}) and (\ref{eq:morno}),
\begin{equation}\label{eq:shchedrin}
\left|\eta_+(t) t\right|_\infty \leq  1.06473476 +
3.25312 \cdot \frac{1 + \frac{4}{\pi} \log H}{H} .
\end{equation}

\chapter{Norms of Fourier transforms}\label{sec:norms}

\section{The Fourier transform of $\eta_2''$}

Our aim here is to give upper bounds on $|\widehat{\eta_2''}|_\infty$,
where $\eta_2$ is as in (\ref{eq:eqeta}). We will do considerably
better than the trivial bound $|\widehat{\eta''}|_\infty \leq |\eta''|_1$.
\begin{lemma}\label{lem:wollust}
For every $t\in \mathbb{R}$,
\begin{equation}
|4 e(-t/4) - 4 e(-t/2) + e(-t)| \leq 7.87052. 
\end{equation}
\end{lemma}
We will describe an extremely simple, but rigorous,
 procedure to find the maximum. Since 
$|g(t)|^2$ is $C^2$ (in fact smooth), there are several more
efficient and equally rigourous algorithms -- for starters, the
bisection method with error bounded in terms of $|(|g|^2)''|_\infty$.
\begin{proof}
Let \begin{equation}\label{eq:ellib}
g(t) = 4 e(-t/4) - 4 e(-t/2) + e(-t).\end{equation} For $a\leq t\leq b$,
\begin{equation}\label{eq:maran}
g(t) = g(a) + \frac{t-a}{b-a} (g(b)-g(a)) + \frac{1}{8} (b-a)^2 \cdot
O^*(\max_{v\in \lbrack a,b\rbrack} |g''(v)|).\end{equation}
(This formula, in all likelihood well-known, is easy to derive. First,
we can assume without loss of generality that $a=0$, $b=1$ and $g(a)=g(b)=0$.
Dividing by $g$ by $g(t)$, we see that we can also assume that $g(t)$ is real
(and in fact $1$). We can also assume that $g$ is real-valued, in that 
it will be enough to prove (\ref{eq:maran}) for the
real-valued function $\Re g$, as this will give us the bound
$g(t) = \Re g(t) \leq (1/8) \max_v |(\Re g)''(v)| \leq \max_v |g''(v)|$ that
we wish for.
Lastly, we can assume (by symmetry) that $0\leq t\leq 1/2$, and that 
$g$ has a local maximum or minimum at $t$.
Writing $M = \max_{u\in \lbrack 0,1\rbrack} |g''(u)|$, we then have:
\[\begin{aligned}
g(t) &= \int_0^t g'(v) dv = \int_0^t \int_t^v g''(u) du dv =
O^*\left(\int_0^t \left|\int_t^v M du\right| dv\right)\\ &= O^*\left(
\int_0^t (v-t) M dv\right) = O^*\left(\frac{1}{2} t^2 M\right) = O^*\left(\frac{1}{8} M\right),\end{aligned}\]
as desired.)

We obtain immediately from (\ref{eq:maran}) that
\begin{equation}\label{eq:elek1}
\max_{t\in \lbrack a,b\rbrack} |g(t)| \leq \max(|g(a)|,|g(b)|) + 
\frac{1}{8} (b-a)^2 \cdot \max_{v\in \lbrack a,b\rbrack} |g''(v)| .\end{equation}

For any $v\in \mathbb{R}$,
\begin{equation}\label{eq:elek2}
|g''(v)| \leq \left(\frac{\pi}{2}\right)^2 \cdot 4 + \pi^2 \cdot 4 + (2\pi)^2
= 9 \pi^2 .
\end{equation}
Clearly $g(t)$ depends only on $t \mo 4 \pi$. Hence, by (\ref{eq:elek1})
and (\ref{eq:elek2}), to estimate \[\max_{t\in \mathbb{R}} |g(t)|\]
 with an error of at most $\epsilon$, 
it is enough to subdivide $\lbrack 0, 4\pi\rbrack$ into intervals of length
$\leq \sqrt{8\epsilon/9\pi^2}$ each. We set $\epsilon = 10^{-6}$ and compute. 
\end{proof}

\begin{lemma}\label{lem:camelo}
Let $\eta_2:\mathbb{R}^+\to \mathbb{R}$ be as in (\ref{eq:eqeta}).
Then
\begin{equation}\label{eq:elundot}
|\widehat{\eta_2''}|_\infty \leq 31.521.
\end{equation}
\end{lemma}
This should be compared with $|\eta_2''|_1 = 48$.
\begin{proof}
We can write
\begin{equation}\label{eq:petruchka}
\eta_2''(x) = 4 (4 \delta_{1/4}(x) - 4 \delta_{1/2}(x) + \delta_1(x)) + f(x),
\end{equation}
where $\delta_{x_0}$ is the point measure at $x_0$ of mass $1$ (Dirac
delta function) and
\[f(x) = \begin{cases} 0 &\text{if $x< 1/4$ or $x\geq 1$,}\\
-4 x^{-2} &\text{if $1/4 \leq x < 1/2$,}\\
4 x^{-2} &\text{if $1/2 \leq x < 1$.} \end{cases}\]
Thus $\widehat{\eta_2''}(t) = 4 g(t) + \widehat{f}(t)$, where $g$ is as in
(\ref{eq:ellib}). It is easy to see that $|f'|_1 = 2 \max_x f(x) - 
2 \min_x f(x) = 160$. Therefore,
\begin{equation}\label{eq:fedsan}
\left|\widehat{f}(t)\right| = \left|\widehat{f'}(t)/(2\pi i t)\right|
\leq \frac{|f'|_1}{2\pi |t|} = \frac{80}{\pi |t|} .
\end{equation}
Since $31.521-4\cdot 7.87052 = 0.03892$, we conclude that
(\ref{eq:elundot}) follows from Lemma \ref{lem:wollust}
and (\ref{eq:fedsan}) for $|t|\geq 655 > 80/(\pi\cdot 0.03892)$.

It remains to check the range $t\in (-655,655)$; since
$4 g(-t) + \widehat{f}(-t)$ is the complex conjugate of $4 g(t) + \widehat{f}(t)$,
 it suffices to consider $t$ non-negative. We use 
(\ref{eq:elek1}) (with $4 g+\widehat{f}$ instead of $g$) and obtain that, to 
estimate $\max_{t\in \mathbb{R}} |4 g+\widehat{f}(t)|$ with an error of at most
$\epsilon$, it is enough to subdivide $\lbrack 0,655)$
into intervals of length $\leq \sqrt{2 \epsilon/|(4 g+\widehat{f})''|_\infty}$ each
 and check $|4 g+\widehat{f}(t)|$ at the endpoints. 
Now, for every $t\in \mathbb{R}$,
\[\left|\left(\widehat{f}\right)''(t)\right| = \left| (-2 \pi i)^2
\widehat{x^2 f}(t)\right| = (2\pi)^2 \cdot O^*\left(|x^2 f|_1\right) =
12 \pi^2.\]
By this and (\ref{eq:elek2}), $|(4 g + \widehat{f})''|_\infty \leq 48 \pi^2$.
Thus, intervals of length $\delta_1$ give an error term of size at most
$24 \pi^2 \delta_1^2$. We choose $\delta_1 = 0.001$ and obtain an error
term less than $0.000237$ for this stage.

To evaluate $\widehat{f}(t)$ (and hence $4 g(t)+\widehat{f}(t)$) 
at a point, we integrate using Simpson's rule on subdivisions of the intervals
$\lbrack 1/4,1/2\rbrack$, $\lbrack 1/2,1\rbrack$ 
into $200\cdot \max(1,\lfloor \sqrt{|t|}\rfloor)$ sub-intervals 
each.\footnote{As usual, the code uses interval arithmetic (\S \ref{sec:koloko}).}
The largest value of $\widehat{f}(t)$ we find is $31.52065\dotsc$,
with an error term of at most $4.5\cdot 10^{-5}$.
\end{proof}

\section{Bounds involving a logarithmic factor}

Our aim now is to give upper bounds on $|\widehat{\eta_{(y)}''}|_\infty$,
where $\eta_{(y)}(t) = \log(y t) \eta_2(t)$ and $y\geq 4$.

\begin{lemma}
Let $\eta_2:\mathbb{R}^+\to \mathbb{R}$ be as in (\ref{eq:eqeta}).
Let $\eta_{(y)}(t) = \log(y t) \eta_2(t)$, where $y\geq 4$. Then
\begin{equation}
|\eta_{(y)}'|_1 < (\log y) |\eta'_2|_1.
\end{equation}
\end{lemma}
\begin{proof}
Recall that $\supp(\eta_2) = (1/4,1)$. For $t\in (1/4,1/2)$,
\[\eta_{(y)}'(t) = (4 \log(y t) \log 4 t)' = \frac{4 \log 4 t}{t} + \frac{4 \log y t}{t}
\geq \frac{8 \log 4 t}{t} > 0,\] whereas, for $t\in (1/2,1)$,
\[\eta_{(y)}'(t) = (- 4 \log(y t) \log t)' = -\frac{4 \log y t}{t}-\frac{4 \log t}{t}
= - \frac{4 \log y t^2}{t} < 0,
\] 
where we are using the fact that $y\geq 4$. Hence $\eta_{(y)}(t)$ is increasing on 
$(1/4,1/2)$ and decreasing on $(1/2,1)$; it is also continuous at $t=1/2$.
Hence $|\eta_{(y)}'|_1 = 2 |\eta_{(y)}(1/2)|$. We are done by
\[2 |\eta_{(y)}(1/2)| = 2 \log \frac{y}{2} \cdot \eta_2(1/2) = \log \frac{y}{2} \cdot
8 \log 2 < \log y \cdot 8 \log 2 = (\log y) |\eta'_2|_1.\]
\end{proof}

\begin{lemma}\label{lem:lujur}
Let $y\geq 4$. Let $g(t) = 
4 e(-t/4) - 4 e(-t/2) + e(-t)$ and $k(t) = 
2 e(-t/4)-e(-t/2)$. Then,
for every $t\in \mathbb{R}$,
\begin{equation}\label{eq:gotora}
|g(t) \cdot \log y  - k(t) \cdot 4 \log 2| \leq 7.87052 \log y. 
\end{equation}
\end{lemma}
\begin{proof}
By Lemma \ref{lem:wollust}, $|g(t)|\leq 7.87052$. Since $y\geq 4$, 
$k(t)\cdot (4 \log 2)/\log y \leq 6$. 
For any complex numbers $z_1$, $z_2$ with $|z_1|, |z_2|\leq \ell$,
 we can have $|z_1 - z_2|> \ell$ only if $|\arg(z_1/z_2)|> \pi/3$.
It is easy to check that, for all $t\in \lbrack -2,2\rbrack$,
\[\left|\arg\left(\frac{g(t) \cdot \log y}{4 \log 2 \cdot k(t)}\right)\right|
= \left|\arg\left(\frac{g(t)}{k(t)}\right)\right| < 0.7 < \frac{\pi}{3} .\]
(It is possible to bound maxima rigorously as in (\ref{eq:elek1}).)
Hence (\ref{eq:gotora}) holds.
\end{proof}

\begin{lemma}\label{lem:octet}
Let $\eta_2:\mathbb{R}^+\to \mathbb{R}$ be as in (\ref{eq:eqeta}).
Let $\eta_{(y)}(t) = (\log y t) \eta_2(t)$, where $y\geq 4$. Then
\begin{equation}\label{eq:yotoman}
|\widehat{\eta_{(y)}''}|_\infty < 31.521 \cdot \log y .
\end{equation}
\end{lemma}
\begin{proof}
Clearly
\[\begin{aligned}
\eta_{(y)}''(x) &= \eta''_2(x) (\log y) + \left((\log x) \eta''_2(x) + \frac{2}{x}
\eta'_2(x) - \frac{1}{x^2} \eta_2(x)\right)\\
&= \eta''_2(x) (\log y) + 4 (\log x) (4 \delta_{1/4}(x) - 4 \delta_{1/2}(x)
+ \delta_1(x)) +  h(x),
\end{aligned}\]
where \[h(x) = \begin{cases} 0 &\text{if $x<1/4$ or $x>1$,}\\
\frac{4}{x^2} (2- 2 \log 2 x)  &\text{if $1/4\leq x < 1/2$,}\\
\frac{4}{x^2} (-2 + 2 \log x) &\text{if $1/2 \leq x < 1$.}
\end{cases}\]
(Here we are using the expression 
(\ref{eq:petruchka}) for $\eta''_2(x)$.)
Hence \begin{equation}\label{eq:kokoso}
\widehat{\eta_{(y)}''}(t) = (4 g(t) + \widehat{f}(t)) (\log y) + 
(-16 \log 2 \cdot k(t) + \widehat{h}(t)),\end{equation}
where $k(t) = 2 e(-t/4) - e(-t/2)$. Just as in the proof of Lemma
\ref{lem:camelo},
\begin{equation}\label{eq:pasaremos}
|\widehat{f}(t)| \leq \frac{|f'|_1}{2\pi |t|} \leq 
\frac{80}{\pi |t|},\;\;\;\;\;
|\widehat{h}(t)| \leq \frac{160 (1 + \log 2)}{\pi |t|}.
\end{equation}
Again as before, this implies that (\ref{eq:yotoman}) holds for
\[|t| \geq \frac{1}{\pi\cdot 0.03892} \left(80 
+ \frac{160 (1 + \log 2)}{(\log 4)}\right) = 2252.51 .\]
Note also that it is enough to check (\ref{eq:yotoman}) for $t\geq 0$,
by symmetry. Our remaining task is to
 prove (\ref{eq:yotoman}) for $0\leq t\leq 2252.21$.

Let $I = \lbrack 0.3, 2252.21\rbrack \setminus \lbrack 3.25, 3.65\rbrack$.
For $t\in I$, we will have 
\begin{equation}\label{eq:hotor}
\arg\left(\frac{4 g(t) + \widehat{f}(t)}{-16 \log 2 \cdot k(t) + 
\widehat{h}(t)}\right)  \subset \left( -\frac{\pi}{3},
\frac{\pi}{3}\right) .
\end{equation}
(This is actually true for $0\leq t\leq 0.3$ as well, but we will use a
different strategy in that range in order to better control error terms.)
Consequently, by Lemma \ref{lem:camelo} and
$\log y \geq \log 4$, 
\[\begin{aligned}
|\widehat{\eta_{(y)}''}(t)| &< \max(|4 g(t)+\widehat{f}(t)| \cdot (\log y),
|16 \log 2 \cdot k(t) - \widehat{h}(t)|)\\ &<  \max(31.521 (\log y),
|48 \log 2 + 25|) = 31.521 \log y,
\end{aligned}\]
where we bound $\widehat{h}(t)$ by
(\ref{eq:pasaremos}) and by a numerical 
computation of the maximum of $|\widehat{h}(t)|$
for $0\leq t \leq 4$ as in the proof of Lemma
\ref{lem:camelo}. 

It remains to check (\ref{eq:hotor}). Here, as in the proof of Lemma 
\ref{lem:lujur}, the allowable error is relatively large (the expression
on the left of (\ref{eq:hotor}) is actually contained in $(-1,1)$
for $t\in I$).
We decide to evaluate the argument in (\ref{eq:hotor}) at all
$t\in 0.005\mathbb{Z} \cap I$,
computing $\widehat{f}(t)$ and $\widehat{h}(t)$ by numerical integration
(Simpson's rule) with a subdivision of $\lbrack -1/4,1\rbrack$ into $5000$
intervals. Proceeding as in the proof of Lemma \ref{lem:wollust}, we see
that the sampling induces an error of at most
\begin{equation}\label{eq:aros}
\frac{1}{2} 0.005^2 \max_{v\in I} ((4 |g''(v)| + |(\widehat{f})''(t)|)
\leq \frac{0.0001}{8} 48 \pi^2 < 0.00593\end{equation}
in the evaluation of $4 g(t)+\widehat{f}(t)$, and an error of at most
\begin{equation}\label{eq:oros}\begin{aligned}
\frac{1}{2} &0.005^2 \max_{v\in I} ((16 \log 2\cdot |k''(v)| + 
|(\widehat{h})''(t)|) \\ &\leq
\frac{0.0001}{8} (16 \log 2 \cdot 6 \pi^2 + 
24 \pi^2 \cdot (2-\log 2)) < 0.0121\end{aligned}\end{equation}
in the evaluation of $16 \log 2 \cdot |k''(v)| + |(\widehat{h})''(t)|$.

Running the numerical evaluation just described for $t\in I$, the estimates
for the left side of (\ref{eq:hotor}) at the sample points are at most
 $0.99134$ in absolute value; the absolute values of the estimates for
$4 g(t) + \widehat{f}(t)$ are all at least $2.7783$, 
and the absolute values of the estimates for
$|-16 \log 2 \cdot \log k(t) + \widehat{h}(t)|$ are all at least $2.1166$.
Numerical integration by Simpson's rule gives errors bounded by
$0.17575$ percent.
Hence
the absolute value of the left side of (\ref{eq:hotor}) is at most
\[\begin{aligned}
0.99134 &+ \arcsin \left(\frac{0.00593}{2.7783}
+0.0017575\right) + \arcsin\left(\frac{0.0121}{2.1166} 
 + 0.0017575\right)\\
&\leq 1.00271 < \frac{\pi}{3}\end{aligned}\]
for $t\in I$.

Lastly, for $t\in \lbrack 0,0.3\rbrack \cup \lbrack 3.25,3.65\rbrack$, 
a numerical computation (samples at 
$0.001\mathbb{Z}$; interpolation as in Lemma
\ref{lem:camelo};
integrals computed by Simpson's rule with a subdivision into $1000$ intervals)
gives
\[\max_{t\in \lbrack 0,0.3\rbrack \cup \lbrack 3.25,3.65\rbrack} \left(|(4g(t) + \widehat{f}(t))| + \frac{|-16 \log 2 \cdot k(t)
+ \widehat{h}(t)|}{\log 4}\right) < 29.08,\]
and so $\max_{t\in \lbrack 0,0.3\rbrack \cup \lbrack 3.25,3.65\rbrack} |\widehat{\eta_{(y)}''}|_\infty < 29.1 \log y < 31.521
\log y$.
\end{proof}

An easy integral gives us that the function $\log \cdot \eta_2$ satisfies 
\begin{equation}\label{eq:koasl}
|\log \cdot \eta_2|_1 = 2 - \log 4
\end{equation}
The following function will appear only in a lower-order term; thus,
an $\ell_1$ estimate will do.

\begin{lemma}\label{lem:marengo}
Let $\eta_2:\mathbb{R}^+\to \mathbb{R}$ be as in (\ref{eq:eqeta}).
Then
\begin{equation}\label{eq:carengo}
|(\log \cdot \eta_2)''|_1 = 96 \log 2.
\end{equation}
\end{lemma}
\begin{proof}
The function $\log \cdot \eta(t)$ is $0$ for $t\notin \lbrack 1/4,1\rbrack$,
is increasing and negative for $t\in (1/4,1/2)$ and is decreasing and positive
for $t\in (1/2,1)$. Hence
\[\begin{aligned}
|(\log \cdot \eta_2)''|_\infty &=
2 \left(  (\log \cdot \eta_2)'\left(\frac{1}{2}\right) - 
 (\log \cdot \eta_2)'\left(\frac{1}{4}\right)\right) \\ &= 
2 (16 \log 2 - (-32 \log 2)) = 96 \log 2.\end{aligned}\]
\end{proof}

\chapter{Sums involving $\Lambda$ and $\phi$}

\section{Sums over primes}

Here we treat some sums of the type $\sum_n \Lambda(n) \varphi(n)$, where
$\varphi$ has compact support. Since the sums are over all integers (not
just an arithmetic progression) and there is no phase $e(\alpha n)$
involved, the treatment is relatively straightforward.

The following is standard. 
\begin{lemma}[Explicit formula]\label{lem:expfor}
Let $\varphi:\lbrack 1,\infty)\to \mathbb{C}$ be continuous and piecewise $C^1$
with $\varphi'' \in \ell_1$; let it also be of compact support contained in
$\lbrack 1,\infty)$. 
Then
\begin{equation}\label{eq:zety}
\sum_n \Lambda(n) \varphi(n) 
 = \int_1^{\infty} \left(1 - \frac{1}{x (x^2-1)} \right) \varphi(x) dx 
- \sum_\rho (M\varphi)(\rho),
\end{equation}
where $\rho$ runs over the non-trivial zeros of $\zeta(s)$.
\end{lemma}
The non-trivial zeros of $\zeta(s)$ are, of course, those in the critical
strip
$0<\Re(s)< 1$.

\begin{remark}
Lemma \ref{lem:expfor} appears as exercise 5 in
\cite[\S 5.5]{MR2061214}; the condition there that $\varphi$ be smooth
can be relaxed, since already the weaker 
assumption that $\varphi''$ be in $L^1$
implies that the Mellin transform $(M\varphi)(\sigma + i t)$ 
decays quadratically on $t$ as $t\to \infty$, thereby guaranteeing 
that the sum $\sum_\rho (M\varphi)(\rho)$ converges absolutely.
\end{remark}

\begin{lemma}\label{lem:crepe}
Let $x\geq 10$. Let $\eta_2$ be as in (\ref{eq:eta2}). Assume that all
non-trivial zeros of $\zeta(s)$ with $|\Im(s)|\leq T_0$ 
lie on the critical line.

Then
\begin{equation}\label{eq:sucre}\begin{aligned}
\sum_n \Lambda(n) \eta_2\left(\frac{n}{x}\right) = x + 
O^*\left(0.135 x^{1/2} + \frac{9.7}{x^2}\right)
+ \frac{\log \frac{e T_0}{2\pi}}{T_0} \left( \frac{9/4}{2 \pi}
+ \frac{6.03}{T_0}\right)  x .
\end{aligned}\end{equation}

In particular, with $T_0 = 3.061\cdot 10^{10}$ 
in the assumption, we have,
for $x\geq 2000$,
\[\sum_n \Lambda(n) \eta_2\left(\frac{n}{x}\right) =
(1+O^*(\epsilon)) x + O^*(0.135 x^{1/2}),
\]
where $\epsilon = 2.73 \cdot 10^{-10}$.
\end{lemma}
The assumption that all non-trivial zeros up to $T_0 = 3.061\cdot
10^{10}$ lie on the critical line
was proven rigorously in \cite{Plattpi}; higher values of
$T_0$ have been reached elsewhere (\cite{Wed}, \cite{GD}).
\begin{proof}
By Lemma \ref{lem:expfor},
\[\sum_n \Lambda(n) \eta_2\left(\frac{n}{x}\right) =
\int_1^{\infty} 
 \eta_2\left(\frac{t}{x}\right) dt - 
\int_1^{\infty} \frac{\eta_2(t/x)}{t (t^2-1)} dt
- \sum_\rho (M \varphi)(\rho),\]
where $\varphi(u) = \eta_2(u/x)$ and
$\rho$ runs over all non-trivial zeros of $\zeta(s)$. Since $\eta_2$
is non-negative,
$\int_1^\infty \eta_2(t/x) dt = x |\eta_2|_1 = x$, while
\[\int_1^\infty \frac{\eta_2(t/x)}{t (t^2-1)} dt = O^*\left(
\int_{1/4}^1 \frac{\eta_2(t)}{t x^2 (t^2-1/100)} dt\right) =
O^*\left(\frac{9.61114}{x^2}\right).\]
By (\ref{eq:envy}),
\[\begin{aligned}
\sum_\rho (M\varphi)(\rho) = \sum_\rho M\eta_2(\rho) \cdot x^{\rho} &= 
\sum_\rho \left(\frac{1-2^{-\rho}}{\rho}\right)^2 x^{\rho} \\ &=
S_1(x) - 2 S_1(x/2) + S_1(x/4),\end{aligned}\]
where 
\begin{equation}\label{eq:gormo}
S_m(x) = \sum_\rho \frac{x^{\rho}}{\rho^{m+1}}.\end{equation}
Setting aside the contribution of all $\rho$ with $|\Im(\rho)|\leq T_0$ and
all $\rho$ with $|\Im(\rho)|>T_0$ and $\Re(s)\leq 1/2$, and using the
symmetry provided by the functional equation, we obtain
\[\begin{aligned}
|S_m(x)| &\leq x^{1/2}\cdot \sum_{\rho} \frac{1}{|\rho|^{m+1}} +
x \cdot \mathop{\mathop{\sum_{\rho}}_{|\Im(\rho)|>T_0}}_{|\Re(\rho)| > 1/2} 
\frac{1}{|\rho|^{m+1}}\\
&\leq x^{1/2}\cdot \sum_{\rho} \frac{1}{|\rho|^{m+1}} +
\frac{x}{2} \cdot \mathop{\sum_{\rho}}_{|\Im(\rho)|>T_0}
\frac{1}{|\rho|^{m+1}}.\end{aligned}\]
We bound the first sum by \cite[Lemma 17]{MR0003018} and the second sum
by \cite[Lemma 2]{MR1950435}. We obtain
\begin{equation}\label{eq:shim}|S_m(x)| 
\leq \left(\frac{1}{2 m \pi T_0^m} + \frac{2.68}{T_0^{m+1}}\right)
 x \log \frac{e T_0}{2 \pi} + \kappa_{m} x^{1/2},\end{equation}
where $\kappa_1 = 0.0463$, $\kappa_2=0.00167$ and $\kappa_3 = 0.0000744$.

Hence
\[\begin{aligned}
\left|\sum_\rho (M \eta)(\rho) \cdot x^\rho\right| &\leq
\left(\frac{1}{2 \pi T_0} + \frac{2.68}{T_0^2}\right) \frac{9 x}{4}
\log \frac{e T_0}{2 \pi} + \left(\frac{3}{2} + \sqrt{2}\right) \kappa_1 x^{1/2}.
\end{aligned}\]
For $T_0 = 3.061\cdot 10^{10}$ and $x\geq 2000$, we obtain
\[\sum_n \Lambda(n) \eta_2\left(\frac{n}{x}\right) =
(1+O^*(\epsilon)) x + O^*(0.135 x^{1/2}),
\]
where $\epsilon = 2.73 \cdot 10^{-10}$.
\end{proof}

\begin{corollary}\label{cor:austeria}
Let $\eta_2$ be as in (\ref{eq:eta2}). 
Assume that all non-trivial zeros of $\zeta(s)$ with $|\Im(s)|\leq T_0$,
$T_0 = 3.061\cdot 10^{10}$, lie on the critical line.
Then, for all $x\geq 1$,
\begin{equation}\label{eq:jotra}
\sum_n \Lambda(n) \eta_2\left(\frac{n}{x}\right) \leq 
\min\left((1+\epsilon)
x + 0.2 x^{1/2}, 1.04488 x\right),\end{equation}
where $\epsilon = 2.73 \cdot 10^{-10}$.
\end{corollary}
\begin{proof}
Immediate from Lemma \ref{lem:crepe} for $x\geq 2000$.
For $x<2000$, we use computation as follows. Since $|\eta_2'|_\infty = 16$ and 
$\sum_{x/4 \leq n\leq x} \Lambda(n)\leq x$ for all $x\geq 0$, computing
$\sum_{n\leq x} \Lambda(n) \eta_2(n/x)$ only for $x\in (1/1000) \mathbb{Z} \cap
\lbrack 0,2000\rbrack$ results in an inaccuracy of at most $(16\cdot 0.0005/
0.9995)x \leq 0.00801 x$. This resolves the matter at all points outside
$(205,207)$ (for the first estimate) or outside $(9.5,10.5)$ and
$(13.5,14.5)$ (for the second estimate). In those intervals,
the prime powers $n$ involved do not change (since whether $x/4 < n \leq x$
depends only on $n$ and $\lbrack x\rbrack$), and thus we can
find the maximum of the sum in (\ref{eq:jotra}) just
by taking derivatives. 
\end{proof}

\section{Sums involving $\phi$}\label{sec:sumphiq}
We need estimates for several sums involving $\phi(q)$ in the denominator.

The easiest are convergent sums, such as
$\sum_q \mu^2(q)/(\phi(q) q)$. We can express this as
$\prod_p (1 + 1/(p (p-1)))$. This is a convergent product,
 and the main task is to bound a tail: for $r$ an integer,
\begin{equation}\label{eq:maloso}\log \prod_{p>r} \left(1 + \frac{1}{p (p-1)}\right)
\leq \sum_{p>r} \frac{1}{p (p-1)} \leq \sum_{n>r} \frac{1}{n (n-1)} = 
\frac{1}{r}.\end{equation}
A quick computation\footnote{Using D. Platt's integer arithmetic package.} 
now suffices to give
\begin{equation}\label{eq:nagasa}
2.591461 \leq \sum_q \frac{\gcd(q,2) \mu^2(q)}{\phi(q) q} <
2.591463\end{equation} and so
\begin{equation}\label{eq:nagasa2}
1.295730 \leq \sum_{\text{$q$ odd}} \frac{\mu^2(q)}{\phi(q) q} <
1.295732,
\end{equation}
since the expression bounded in (\ref{eq:nagasa2}) is exactly half of
that bounded in (\ref{eq:nagasa}).

Again using (\ref{eq:maloso}), we get that
\begin{equation}\label{eq:massacre}
2.826419
 \leq \sum_q \frac{\mu^2(q)}{\phi(q)^2} < 2.826421.\end{equation}
In what follows, we will use values for convergent sums obtained in
much the same way -- an easy tail bound followed by a computation. 


By \cite[Lemma 3.4]{MR1375315},
\begin{equation}\label{eq:ramo}\begin{aligned}
\sum_{q\leq r} \frac{\mu^2(q)}{\phi(q)} 
&= \log r + c_E + O^*(7.284
r^{-1/3}),\\
\mathop{\sum_{q\leq r}}_{\text{$q$ odd}} \frac{\mu^2(q)}{\phi(q)} &= \frac{1}{2} \left(
\log r + c_E + \frac{\log 2}{2}\right) + O^*(4.899 r^{-1/3}),\\
\end{aligned}\end{equation}
where 
\[c_E = \gamma + \sum_p \frac{\log p}{p (p-1)} = 1.332582275+O^*(10^{-9}/3)\]
by \cite[(2.11)]{MR0137689}.
As we already said in (\ref{eq:charpy}), this, supplemented by a
computation for $r\leq 4\cdot 10^7$, gives
\[\log r + 1.312 \leq \sum_{q\leq r} \frac{\mu^2(q)}{\phi(q)} \leq
\log r + 1.354\]
for $r\geq 182$.
In the same way, we get that
\begin{equation}\label{eq:marmo}\frac{1}{2} \log r + 0.83 \leq 
\mathop{\sum_{q\leq r}}_{\text{$q$ odd}} \frac{\mu^2(q)}{\phi(q)} \leq
\frac{1}{2} \log r + 0.85\end{equation}
for $r\geq 195$.
(The numerical verification here goes up to $1.38\cdot 10^8$; for $r>
3.18\cdot 10^8$, use \ref{eq:marmo}.)

Clearly
\begin{equation}\label{eq:dsamo}
\mathop{\sum_{q\leq 2r}}_{\text{$q$ even}} \frac{\mu^2(q)}{\phi(q)} =
\mathop{\sum_{q\leq r}}_{\text{$q$ odd}} \frac{\mu^2(q)}{\phi(q)}.
\end{equation}

We wish to obtain bounds for the sums
\[\sum_{q\geq r} \frac{\mu^2(q)}{\phi(q)^2},\;\;\;\;
\mathop{\sum_{q\geq r}}_{\text{$q$ odd}} \frac{\mu^2(q)}{\phi(q)^2},\;\;\;\;
\mathop{\sum_{q\geq r}}_{\text{$q$ even}} \frac{\mu^2(q)}{\phi(q)^2},\]
where $N\in \mathbb{Z}^+$ and $r\geq 1$. To do this, it will be helpful
to express some of the quantities within these sums as
convolutions.\footnote{The
author would like to thank O. Ramar\'e for teaching him this technique.}
For $q$ squarefree and $j\geq 1$,
\begin{equation}\label{eq:merleau}
\frac{\mu^2(q) q^{j-1}}{\phi(q)^j} = \sum_{ab=q} \frac{f_j(b)}{a},
\end{equation}
where $f_j$ is the multiplicative function defined by
\[f_j(p) = \frac{p^j - (p-1)^j}{(p-1)^j p},\;\;\;\;\;
f_j(p^k) = 0\;\;\;\;\;\text{for $k\geq 2$.}\]
 
We will also find the following estimate useful.
\begin{lemma}\label{lem:sidio}
Let $j\geq 2$ be an integer and $A$ a positive real. 
Let $m\geq 1$ be an integer.
Then
\begin{equation}\label{eq:alaspe}
\mathop{\sum_{a\geq A}}_{(a,m)=1} \frac{\mu^2(a)}{a^j}\leq \frac{\zeta(j)/\zeta(2
  j)}{A^{j-1}} \cdot \prod_{p|m} \left(1 + \frac{1}{p^j}\right)^{-1}
.\end{equation}
\end{lemma}
It is useful to note that $\zeta(2)/\zeta(4) = 15/\pi^2 = 1.519817\dotsc$ and
$\zeta(3)/\zeta(6) = 1.181564\dotsc$.
\begin{proof}
The right
side of (\ref{eq:alaspe}) decreases as $A$ increases, while the left side depends only
on $\lceil A\rceil$. Hence,
it is enough to prove (\ref{eq:alaspe}) when $A$ is an integer.
 
For $A=1$, (\ref{eq:alaspe}) is an equality. Let
\[C = \frac{\zeta(j)}{\zeta(2
  j)} \cdot \prod_{p|m} \left(1 + \frac{1}{p^j}\right)^{-1}.\]
Let $A\geq 2$. Since
\[\mathop{\sum_{a\geq A}}_{(a,m)=1} \frac{\mu^2(a)}{a^j} = C -
\mathop{\sum_{a<A}}_{(a,m)=1} \frac{\mu^2(a)}{a^j}\]
and
\[\begin{aligned}
C &= \mathop{\sum_a}_{(a,m)=1} \frac{\mu^2(a)}{a^j} < \mathop{\sum_{a<A}}_{
(a,m)=1} \frac{\mu^2(a)}{a^j} + \frac{1}{A^j} 
+ \int_A^\infty \frac{1}{t^j} dt\\ 
&=  \mathop{\sum_{a<A}}_{(a,m)=1}
 \frac{\mu^2(a)}{a^j} + \frac{1}{A^j} + \frac{1}{(j-1) A^{j-1}},
\end{aligned}\]
we obtain
\[\begin{aligned}\mathop{\sum_{a\geq A}}_{(a,m)=1} 
&\frac{\mu^2(a)}{a^j} = \frac{1}{A^{j-1}} \cdot C
+ \frac{A^{j-1}-1}{A^{j-1}} \cdot C
- \mathop{\sum_{a<A}}_{(a,m)=1} \frac{\mu^2(a)}{a^j}\\
&< \frac{C}{A^{j-1}}
+ \frac{A^{j-1}-1}{A^{j-1}} \cdot \left(\frac{1}{A^j} + \frac{1}{
(j-1) A^{j-1}} \right)
- \frac{1}{A^{j-1}} \mathop{\sum_{a<A}}_{(a,m)=1} \frac{\mu^2(a)}{a^j}\\
&\leq \frac{C}{A^{j-1}} +
\frac{1}{A^{j-1}} \left(
\left(1 - \frac{1}{A^{j-1}}\right) \left(\frac{1}{A}
 + \frac{1}{j-1}\right) - 1\right)
.\end{aligned}\]
Since $(1 - 1/A) (1/A+1) < 1$ and $1/A + 1/(j-1)\leq 1$ for $j\geq 3$,
we obtain that 
\[\left(1 - \frac{1}{A^{j-1}}\right) \left(\frac{1}{A}
 + \frac{1}{j-1}\right) < 1\]
for all integers $j\geq 2$, and so the statement follows.
\end{proof}

We now obtain easily the estimates we want: by (\ref{eq:merleau}) and
Lemma \ref{lem:sidio} (with $j=2$ and $m=1$),
\begin{equation}\label{eq:gat1}\begin{aligned}
\sum_{q\geq r} \frac{\mu^2(q)}{\phi(q)^2} &= \sum_{q\geq r}
\sum_{ab = q} \frac{f_2(b)}{a} \frac{\mu^2(q)}{q} \leq \sum_{b\geq 1}
\frac{f_2(b)}{b}
\sum_{a\geq r/b} \frac{\mu^2(a)}{a^2}\\ &\leq \frac{\zeta(2)/\zeta(4)}{r}
\sum_{b\geq 1}  f_2(b) = \frac{\frac{15}{\pi^2}}{r} \prod_p
 \left(1+ \frac{2p-1}{(p-1)^2 p}\right) \leq \frac{6.7345}{r}.
\end{aligned}\end{equation}
Similarly, by (\ref{eq:merleau}) and
Lemma \ref{lem:sidio} (with $j=2$ and $m=2$),
\begin{equation}\label{eq:gat1o}\begin{aligned}
\mathop{\sum_{q\geq r}}_{\text{$q$ odd}}
 \frac{\mu^2(q)}{\phi(q)^2} &= \mathop{\sum_{b\geq 1}}_{\text{$b$ odd}}
\frac{f_2(b)}{b} \mathop{\sum_{a\geq r/b}}_{\text{$a$ odd}}
 \frac{\mu^2(a)}{a^2}
\leq \frac{\zeta(2)/\zeta(4)}{1+1/2^2} \frac{1}{r} \sum_{\text{$b$ odd}}
f_2(b)\\
&= \frac{12}{\pi^2} \frac{1}{r} \prod_{p>2} \left(1 + \frac{2p-1}{(p-1)^2
    p}\right)\leq \frac{2.15502}{r}
\end{aligned}\end{equation}
\begin{equation}\label{eq:gat1e}
\mathop{\sum_{q\geq r}}_{\text{$q$ even}} \frac{\mu^2(q)}{\phi(q)^2} =
\mathop{\sum_{q\geq r/2}}_{\text{$q$ odd}} \frac{\mu^2(q)}{\phi(q)^2} \leq
\frac{4.31004}{r}.\end{equation}

Lastly,
\begin{equation}\label{eq:gatosbuenos}\begin{aligned}
\mathop{\sum_{q\leq r}}_{\text{$q$ odd}} \frac{\mu^2(q) q}{\phi(q)} &= 
\mathop{\sum_{q\leq r}}_{\text{$q$ odd}} \mu^2(q)
\sum_{d|q} \frac{1}{\phi(d)} = \mathop{\sum_{d\leq r}}_{\text{$d$ odd}}
 \frac{1}{\phi(d)} \mathop{\mathop{\sum_{q\leq r}}_{d|q}}_{\text{$q$ odd}}
 \mu^2(q) \leq \mathop{\sum_{d\leq r}}_{\text{$d$ odd}} \frac{1}{2 \phi(d)}
\left(\frac{r}{d} + 1\right)\\ &\leq 
\frac{r}{2} \sum_{\text{$d$ odd}} \frac{1}{\phi(d) d}
+ \frac{1}{2} \mathop{\sum_{d\leq r}}_{\text{$d$ odd}} \frac{1}{\phi(d)}
\leq 0.64787 r + \frac{\log r}{4} + 0.425,\end{aligned}\end{equation}
where we are using (\ref{eq:nagasa2}) and (\ref{eq:marmo}).

\begin{center}
* * *
\end{center}

Since we are on the subject of $\phi(q)$, let us also prove a simple lemma
that we use at various points in the text to bound $q/\phi(q)$.

\begin{lemma}\label{lem:merkel}
For any $q\geq 1$ and any $r\geq \max(3,q)$,
\[\frac{q}{\phi(q)} < \digamma(r),\]
where
\begin{equation}\label{eq:locos}\begin{aligned}
\digamma(r) &= e^{\gamma} \log \log r + \frac{2.50637}{\log \log r}.
\end{aligned}\end{equation}
\end{lemma}
\begin{proof}
Since $\digamma(r)$ is increasing for $r\geq 27$, the statement follows
immediately for $q\geq 27$ by \cite[Thm. 15]{MR0137689}:
\[\frac{q}{\phi(q)} < \digamma(q) \leq \digamma(r).\]

For $q<27$, it is clear that $q/\phi(q) \leq 2\cdot 3/(1\cdot 2) = 3$.
By the arithmetic/geometric mean inequality, $\digamma(t) \geq
2 \sqrt{e^\gamma 2.50637} > 3$ for all $t>e$, and so the lemma holds for
$q<27$.
\end{proof}

\chapter{Checking small $n$ by checking zeros of $\zeta(s)$}\label{sec:appa}

In order to show that every odd number $n\leq N$ is
the sum of three primes, it is enough to show for some $M\leq N$ that
\begin{enumerate}
\item\label{it:oshp} every even integer $4\leq m\leq M$ can be written as the sum of two primes,
\item\label{it:gaps} the difference between any 
two consecutive primes $\leq N$ is at most $M-4$.
\end{enumerate}
(If we want to show that every
odd number $n\leq N$ is the sum of three {\em odd} primes, we just replace
$M-4$ by $M-6$ in (\ref{it:gaps}).)
The best known result of type (\ref{it:oshp}) is that of  Oliveira 
e Silva, Herzog and Pardi (\cite{OSHP}, $M= 4\cdot 10^{18}$).
As for (\ref{it:gaps}), it was proven in \cite{MR3171101} for
$M = 4\cdot 10^{18}$ and $N = 8.875694 \cdot 10^{30}$ by a direct
computation (valid even if we replace $M-4$ by $M-6$ in the statement of
(\ref{it:gaps})). 

Alternatively, one can establish results of type (\ref{it:gaps}) by
means of numerical verifications of the Riemann hypothesis up to a
certain height. This is a classical approach, followed in
\cite{MR0457373} and \cite{MR0457374}, and later in \cite{MR1950435};
we will use the version of (\ref{it:oshp})
 kindly provided by Ramar\'e in \cite{Rpc}.
We carry out this approach in full here, not because it is preferrable to \cite{MR3171101}
-- it is still based on computations, and it is slightly more indirect than
\cite{MR3171101}  --
but simply to show that one can establish what we need by a different
route.

A numerical verification of the Riemann hypothesis up to a certain
height consists simply in checking that all (non-trivial)
zeroes $z$ of the Riemann zeta function up
to a height $H$ (meaning: $\Im(z)\leq H$) lie on the critical line
$\Re(z)=1/2$. 

The height up to which the Riemann hypothesis has actually been fully
verified is not a matter on which there is unanimity. The strongest claim
in the literature is in \cite{GD}, which states that the first $10^{13}$
zeroes of the Riemann zeta function lie on the critical line $\Re(z)=1/2$.
This corresponds to checking the Riemann hypothesis up to height
$H = 2.44599\cdot 10^{12}$.
It is unclear whether this computation was or could be easily made
rigorous; as pointed out in \cite[p. 2398]{MR2684372}, it has not 
been replicated yet.

Before \cite{GD}, the strongest results were those of the ZetaGrid
distributed computing project led by S. Wedeniwski \cite{Wed}; the method
followed in it was more traditional, and should allow rigorous verification
involving interval arithmetic. Unfortunately, the results were never
formally published. The statement that the ZetaGrid project verified the
first $9\cdot 10^{11}$ zeroes (corresponding to $H = 2.419\cdot 10^{11}$)
is often quoted (e.g., \cite[p. 29]{MR2684771}); this is the point
to which the project had got by the time of Gourdon and Demichel's
announcement. Wedeniwski asserts
in private communication that the project verified the first $10^{12}$
zeroes, and that the computation was double-checked (by the same method).

The strongest claim  prior to ZetaGrid was that of van de Lune
($H = 3.293\cdot 10^9$, first $10^{10}$ zeroes; unpublished).
Recently, Platt \cite{Plattpi} checked the first $1.1\cdot 10^{11}$ 
zeroes ($H = 3.061\cdot 10^{10}$)
rigorously, following a method essentially based on that in
\cite{MR2293591}. Note that \cite{Plattpi} uses interval arithmetic,
which is highly desirable for floating-point computations.

\begin{prop}\label{prop:gosto}
Every odd integer $5\leq n\leq n_0$ is the sum of three primes, where
\[n_0 = \begin{cases} 5.90698\cdot 10^{29} &\text{if \cite{GD} is used
($H= 2.44\cdot 10^{12}$),}\\
6.15697 \cdot 10^{28} &\text{if 
ZetaGrid results are used ($H = 2.419\cdot 10^{11}$),} \\
1.23163
 \cdot 10^{27} &\text{if \cite{Plattpi} is used (
$H = 3.061\cdot 10^{10}$).}
\end{cases}
\]
\end{prop}
\begin{proof}
For $n\leq 4\cdot 10^{18}+3$, this is immediate from \cite{OSHP}. Let
$4\cdot 10^{18}+3 < n \leq n_0$.
We need to show that there is a prime $p$ in
$\lbrack n-4-(n-4)/\Delta,n-4\rbrack$, where $\Delta$ is large enough for
$(n-4)/\Delta \leq 4\cdot 10^{18}-4$ to hold. We will then have that 
$4\leq n-p \leq 4+(n-4)/\Delta \leq 4\cdot 10^{18}$. Since $n-p$ is even,
\cite{OSHP} will then imply that $n-p$ is the sum of two primes $p'$, $p''$,
and so
\[n = p + p' + p''.\]

Since $n-4>10^{11}$, the interval $\lbrack n-4-(n-4)/\Delta,n-4\rbrack$
with $\Delta=28314000$ must contain a prime \cite{MR1950435}.
This gives the solution for $(n-4)\leq 1.1325 \cdot 10^{26}$, since
then $(n-4)\leq 4\cdot 10^{18}-4$. Note $1.1325\cdot 10^{26}>e^{59}$.

From here onwards, we use the tables in \cite{Rpc} to find acceptable
values of $\Delta$. Since $n-4\geq e^{59}$, we can choose
\[\Delta = \begin{cases}
52211882224 &\text{if \cite{GD} is used (case (a)),}\\
13861486834 &\text{if ZetaGrid is used (case (b)),}\\
307779681 &\text{if \cite{Plattpi} is used (case (c)).} 
\end{cases}\]
This gives us $(n-4)/\Delta \leq 4\cdot 10^{18}-4$ for $n-4<e^{r_0}$, where
$r_0 = 67$ in case (a), $r_0 = 66$ in case (b) and $r_0 = 62$ in case (c).

If $n-4\geq e^{r_0}$, we can choose (again by \cite{Rpc})
\[\Delta = \begin{cases}
146869130682 &\text{in case (a),}\\
15392435100 &\text{in case (b),}\\
307908668 &\text{in case (c).}
\end{cases}\]
This is enough for $n-4<e^{68}$ in case (a), and without further conditions
for (b) or (c).

Finally, if $n-4\geq e^{68}$ and we are in case (a), \cite{Rpc}
assures us that the choice
\[\Delta = 147674531294\]
is valid; we verify as well that $(n_0-4)/\Delta \leq 4\cdot 10^{18}-4$.
\end{proof}

In other words, the rigorous results in \cite{Plattpi} are enough to show
the result for all odd $n\leq 10^{27}$. Of course, \cite{MR3171101} is also
more than enough, and gives stronger results than Prop.~\ref{prop:gosto}.

\backmatter
\bibliographystyle{alpha}
\bibliography{3prbook}
\printindex
\end{document}